\documentclass[12pt,letter]{amsart}
\usepackage{amssymb,latexsym, amsmath, amsxtra, bbm, bm}
\usepackage[dvips]{graphics}
\usepackage{xypic}
\usepackage{verbatim}
\usepackage[abs]{overpic}
\usepackage{hyperref}
\usepackage{mathtools}

\allowdisplaybreaks[3]

\theoremstyle{plain}
        \newtheorem{theorem}{Theorem}[section]
        \newtheorem*{theorem*}{Theorem}
        \newtheorem*{conj*}{Conjecture}
        \newtheorem{lemma}[theorem]{Lemma}
        \newtheorem{prop}[theorem]{Proposition}
        \newtheorem{cor}[theorem]{Corollary}

\theoremstyle{definition}
        \newtheorem{definition}[theorem]{Definition}
        \newtheorem{rem}[theorem]{Remark}
         \newtheorem{rems}[theorem]{Remarks}
         
         \newtheorem*{assumptions}{The Assumptions}

\theoremstyle{remark}
        \newtheorem*{remark}{Remark}

\numberwithin{equation}{section}
\numberwithin{theorem}{section}
\numberwithin{table}{section}
\numberwithin{figure}{section}

\providecommand{\defn}[1]{\emph{#1}}

\newcommand{\diam}  {\operatorname{diam}}

\newcommand{\inte}  {\operatorname{inte}}
\newcommand{\inter}  {\operatorname{int}}

\newcommand{\id} {\operatorname{id}}

\newcommand{\card} {\operatorname{card}}
\newcommand{\supp}{\operatorname{supp}}
\newcommand{\R}{\mathbb{R}}
\newcommand{\B}{\mathbb{B}}    
\newcommand{\C}{\mathbb{C}}      

\newcommand{\N}{\mathbb{N}}      
\newcommand{\Z}{\mathbb{Z}}      

\newcommand{\D}{\mathbb{D}}     
 
\renewcommand{\H}{\mathbb{H}}    
\providecommand{\abs}[1]{\lvert#1\rvert}
\providecommand{\Abs}[1]{\left|#1\right|}
\providecommand{\Absbig}[1]{\bigl|#1\bigr|}
\providecommand{\Absbigg}[1]{\biggl|#1\biggr|}
\providecommand{\AbsBig}[1]{\Bigl|#1\Bigr|}
\providecommand{\AbsBigg}[1]{\Biggl|#1\Biggr|}
\providecommand{\norm}[1]{\|#1\|}
\providecommand{\Norm}[1]{\left\|#1\right\|}
\providecommand{\Normbig}[1]{\bigl\|#1\bigr\|}
\providecommand{\Normbigg}[1]{\biggl\|#1\biggr\|}
\providecommand{\NormBig}[1]{\Bigl\|#1\Bigr\|}
\providecommand{\NormBigg}[1]{\Biggl\|#1\Biggr\|}
\renewcommand{\:}{\colon}

\def\({\left(}
\def\({\right)}
\def\[{\left[}
\def\]{\right]}

\renewcommand{\b}{\mathfrak{b}}
\newcommand{\w}{\mathfrak{w}}
\renewcommand{\c}{\mathfrak{c}}

\newcommand{\I}{\mathbf{i}}

\newcommand{\crit}{\operatorname{crit}}
\newcommand{\post}{\operatorname{post}}

\newcommand{\LIP}{\operatorname{LIP}}

\newcommand{\CC}{\mathcal{C}}

 \newcommand{\DD}{\mathbf{D}}

\newcommand{\X} {\mathbf{X}}

\newcommand{\E} {\mathbf{E}}

\newcommand{\V} {\mathbf{V}}
 
\newcommand{\W} {\mathbf{W}}

\newcommand{\CCC}{C}

\newcommand{\PPP}{\mathcal{P}}

\newcommand{\MMM}{\mathcal{M}}

\newcommand{\Holder}[1] {\CCC^{0,#1}}

\newcommand{\Hseminorm}[2] {\Abs{#2}_{#1}}

\newcommand{\Hseminormbig}[2] {\Absbig{#2}_{#1}}

\newcommand{\Hseminormbigg}[2] {\Absbigg{#2}_{#1}}

\newcommand{\HseminormBig}[2] {\AbsBig{#2}_{#1}}

\newcommand{\Hnorm}[3] {\Norm{#2}_{\Holder{#1}{#3}}}

\newcommand{\Hnormbig}[3] {\Normbig{#2}_{\Holder{#1}{#3}}}

\newcommand{\Hnormbigg}[3] {\Normbigg{#2}_{\Holder{#1}{#3}}}

\newcommand{\HnormBig}[3] {\NormBig{#2}_{\Holder{#1}{#3}}}

\newcommand{\NHnorm}[4] {\Norm{#3}^{[#2]}_{\Holder{#1}{#4}}}

\newcommand{\NHnormbig}[4] {\Normbig{#3}^{[#2]}_{\Holder{#1}{#4}}}

\newcommand{\NHnormBig}[4] {\NormBig{#3}^{[#2]}_{\Holder{#1}{#4}}}

\newcommand{\HseminormD}[2] {\Abs{#2}_{#1}}

\newcommand{\NHnormD}[3] {\Norm{#3}^{[#2]}_{#1}}

\newcommand{\vertiii}[1]{{\left\vert\kern-0.25ex\left\vert\kern-0.25ex\left\vert #1 
    \right\vert\kern-0.25ex\right\vert\kern-0.25ex\right\vert}}

\newcommand{\OpHnormD}[2] {\vertiii{#2}_{#1}}

\newcommand{\NOpHnormD}[3] {\vertiii{#3}^{[#2]}_{#1}}

\newcommand{\RR}{\mathcal{L}}

\newcommand{\RRR}{\mathbbm{L}}

\newcommand{\MM}{\mathbbm{M}}

\newcommand{\lcm}{\operatorname{lcm}}

\newcommand{\XX}{\mathbb{X}}

\newcommand{\wt}[1]{\widetilde{#1}}

\newcommand{\Arg}{\operatorname{Arg}}

\newcommand{\minus}{\scalebox{0.6}[0.6]{$-\!\!\;$}}

\newcommand{\circsmall}{\scalebox{0.6}[0.6]{$\circ$}}

\newcommand{\Orb}{\mathfrak{P}}

\newcommand{\ti}{\vartriangle}

\newcommand{\ee}{\boldsymbol{\shortparallel}}

\newcommand{\e}{\boldsymbol{\shortmid}}

\newcommand{\po}{\bullet}

\newcommand{\p}{\operatorname{p}}
\newcommand{\np}{\operatorname{np}}

\newcommand{\DS}{\mathfrak{D}}

\newcommand{\Inv}{\operatorname{Inv}}

\begin{document}
\title[Prime orbit theorems for expanding Thurston maps]{Prime orbit theorems for expanding Thurston maps}
\author{Zhiqiang~Li, Tianyi~Zheng}
\address{Institute for Mathematical Sciences, Stony Brook University, Stony Brook, NY 11794--3660}
\email{lizq@math.stonybrook.edu}
\address{Department of Mathematics, University of California, San Diego, San Diego, CA 92093--0112}
\email{tzheng2@ucsd.edu}



\begin{abstract}
We obtain an analogue of the prime number theorem for a class of branched covering maps on the $2$-sphere called expanding Thurston maps $f$, which are topological models of some rational maps without any smoothness or holomorphicity assumption. More precisely, by studying dynamical zeta functions and, more generally, dynamical Dirichlet series for $f$, we show that the number of primitive periodic orbits of $f$, ordered by a weight on each point induced by a non-constant (eventually) positive real-valued H\"{o}lder continuous function $\phi$ on $S^2$ satisfying some additional regularity conditions, is asymptotically the same as the well-known logarithmic integral, with an exponential error term. Such a result, known as a Prime Orbit Theorem, follows from our quantitative study of the holomorphic extension properties of the associated dynamical zeta functions and dynamical Dirichlet series. In particular, the above result applies to postcritically-finite rational maps whose Julia set is the whole Riemann sphere. Moreover, we prove that the regularity conditions needed here are generic; and for a Latt\`{e}s map $f$ and a continuously differentiable (eventually) positive function $\phi$, such a Prime Orbit Theorem holds if and only if $\phi$ is not co-homologous to a constant. 
\end{abstract}

\maketitle

\tableofcontents

\section{Introduction}

\subsection{History and motivations}

Counting is probably one of the very first mathematical activities that predates any written history of mankind. It remains at the core of virtually all fields of mathematics to count important objects in the field and study their statistical properties. 

One useful idea in such studies is to code the important objects in a function in the form of a polynomial or a series. Perhaps the most famous of such functions is the \emph{Riemann zeta function}
\begin{equation*}
\zeta_{\operatorname{Riemann}} (s) \coloneqq \sum\limits_{n=1}^{+\infty} \frac{1}{n^s}  
                            =  \prod\limits_{p \text{ prime}}  ( 1 - p^{-s}  )^{-1},    \qquad\qquad  \Re(s) > 1,
\end{equation*}
whose analytic properties were studied by B.~Riemann in the 19th century, even though the product formula was already known to L.~Euler in the 18th century. Analytic properties of the Riemann zeta function are closely related to the distribution of prime numbers. It is known that the assertion that $\zeta_{\operatorname{Riemann}}$ has a non-vanishing holomorphic extension on the line $\Re(s) = 1$ except for a simple pole at $s=1$ is equivalent to the famous Prime Number Theorem of Ch.~J.~de~la~Vall\'ee-Poussin and J.~Hadamard stating that the number $\pi(T)$ of primes no larger than $T>0$ satisfies 
\begin{equation*}
\pi(T) \sim \operatorname{Li} (T) ,  \qquad\qquad  \text{as } T \to +\infty,
\end{equation*}
where $\operatorname{Li}(y)$ is the well-known \defn{Eulerian logarithmic integral function}
\begin{equation}      \label{eqDefLogIntegral}
\operatorname{Li} (y) \coloneqq \int_2^y\! \frac{1}{\log u} \,\mathrm{d} u,  \qquad\qquad y>0.
\end{equation}
A more careful study of $\zeta_{\operatorname{Riemann}}$ reveals that a condition of H.~von~Koch from 1901 \cite{vK01} on the error term in the Prime Number Theorem, namely,
\begin{equation*}
\pi(T) = \operatorname{Li} (T)  + \operatorname{O} \bigl( \sqrt{T} \log T \bigr),  \qquad\qquad  \text{as } T \to +\infty,
\end{equation*}
is equivalent to the Riemann hypothesis (see also \cite[Section~5.1]{BCRW08}).

The idea of studying zeta functions was first introduced by A.~Selberg in 1956 from number theory into geometry, where (primitive) closed geodesics serve the role of prime numbers. He defined a zeta function
\begin{equation}    \label{eqDefZetaFnSelberg}
\zeta_{\operatorname{Selberg}} (s) \coloneqq   \prod\limits_{\gamma \in \Orb }   \prod\limits_{ n=0 }^{+\infty}   \bigl( 1 - e^{ -(s + n) l(\gamma) }  \bigr),    \qquad\qquad  \Re(s) > 1,
\end{equation}
where $\Orb$ denotes the set of primitive closed geodesics and $l(\gamma)$ is the length of the geodesic $\gamma$ \cite{Se56}.

H.~Huber established the first \emph{Prime Geodesic Theorem}, as an analogue of the Prime Number Theorem, for surfaces of constant negative curvature in 1961, where A.~Selberg's work \cite{Se56} was implicitly used.

\begin{theorem}[H.~Huber \cite{Hu61}]
Let $M$ be a compact surface of constant curvature $-1$, and by $\pi(T)$ we denote the number of primitive closed geodesics $\gamma$ of length $l(\gamma) \leq T$. 

Then there exists $\alpha \in (0,1)$ such that 
\begin{equation*}
\pi(T) = \operatorname{Li}  \bigl(  e^T  \bigr)  + \operatorname{O} \bigl( e^{\alpha T} \bigr),  \qquad\qquad  \text{as } T \to +\infty.
\end{equation*} 
\end{theorem}

The zeta functions were then introduced into dynamics by M.~Artin and B.~Mazur \cite{AM65} in 1965 for diffeomorphisms and by S.~Smale \cite{Sm67} in 1967 for Anosov flows, where (primitive) periodic orbits serve the role of prime numbers. S.~Smale used A.~Selberg's formulation in the context of geodesic flows on surfaces of constant negative curvature due to the direct correspondence between closed geodesics on the surface and periodic orbits of the geodesic flow. A related formulation of zeta functions for flows was later proposed and studied by D.~Ruelle \cite{Rue76a, Rue76b, Rue76c} in 1976, which behaves better under changes of time scale in the more general context of Axiom A systems. More precisely, for Anosov flows, the Ruelle zeta function is defined as
\begin{equation}   \label{eqDefZetaFnRuelle}
\zeta_{\operatorname{Ruelle}} (s) \coloneqq   \prod\limits_{\gamma \in \Orb }   \bigl( 1 - e^{ - s l(\gamma) }  \bigr)^{-1} ,    \qquad\qquad  \Re(s) > 1,
\end{equation}
where $\Orb$ denotes the set of primitive periodic orbits of the flow and $l(\gamma)$ is the length of the orbit $\gamma$. With this interpretation of $\Orb$ and $l(\gamma)$, we have
\begin{equation}   \label{eqRuelleSelberg}
\zeta_{\operatorname{Ruelle}} (s)  
= \frac{ \zeta_{\operatorname{Selberg}}(s+1)}{ \zeta_{\operatorname{Selberg}}(s) }
\end{equation}
when both sides are defined.

Extensive researches have been carried out in geometry and dynamics in establishing Prime Geodesic Theorems on various spaces and Prime Orbit Theorems for various flows and other dynamical systems. We recall but a few such results here and by no means claim to give a complete review of the literature.

We denote by $\pi(T)$ the number of primitive periodic orbits $\gamma$ of ``length'' (appropriately interpreted for the corresponding dynamical system) $l(\gamma) \leq T$. By a Prime Orbit Theorem without an error term, we mean the assertion that there exists a constant $h>0$ such that 
$\pi(T) \sim  \operatorname{Li}  \bigl(  e^{hT}  \bigr)$ as $T \to +\infty$. By a Prime Orbit Theorem with an exponential error term, we mean the assertion that there exist constants $h>0$ and $\delta\in(0,h)$ such that 
$\pi(T)=  \operatorname{Li}  \bigl(  e^{hT}  \bigr)  + \operatorname{O}\bigl(  e^{(h-\delta)T}  \bigr)$ as $T \to +\infty$. 

Generalizing the first order asymptotics of H.~Huber for geodesic flows over compact surfaces of constant negative curvature, G.~A.~Margulis established in his thesis in 1970 \cite{Mar04} (see also \cite{Mar69}) a Prime Orbit Theorem without an error term for the geodesic flows over compact Riemannian manifolds with variable negative curvature, and more generally, for weak-mixing Anosov flows preserving a smooth volume. Similar results were obtained by P.~Sarnak in his thesis in 1980 for non-compact surfaces of finite volume \cite{Sa80}.

For geodesic flows over convex-cocompact surfaces of constant negative curvature, a Prime Orbit Theorem without an error term was obtained conditionally by L.~Guillop\'e \cite{Gu86} and later unconditionally by S.~P.~Lalley \cite{La89}.

The exponential error terms in the Prime Orbit Theorems in the contexts above (except in H.~Huber's result) were out of reach until D.~Dolgopyat's seminal work on the exponential mixing of Anosov flows in his thesis \cite{Dol98}, where he developed an ingenuous approach to get new upper bounds on the norms of the complex Ruelle (transfer) operators on some appropriate function spaces. M.~Pollicott and R.~Sharp \cite{PoSh98} combined these bounds with techniques from number theory to get a Prime Orbit Theorem with an exponential error term for the geodesic flows over compact surfaces of variable negative curvature. For related works on closed geodesics satisfying some homological constraints, see R.~Phillips and ~P.~Sarnak \cite{PhSa87}, S.~P.~Lalley \cite{La89}, A.~Katsuda and T.~Sunada \cite{KS90}, M.~Pollicott \cite{Po91}, R.~Sharp \cite{Sh93}, M.~Babillot and F.~Ledrappier \cite{BabLe98}, M.~Pollicott and R.~Sharp \cite{PoSh98}, N.~Anantharaman \cite{An00a, An00b}, etc.

The elegant idea of M.~Pollicott and R.~Sharp in \cite{PoSh98} used in establishing the error term for their Prime Orbit Theorem, is summarized in a nutshell below:
\begin{enumerate}
\smallskip
\item[(1)] Obtain a quantitative bound for each term in the additive form of the Ruelle zeta function $\zeta_{\operatorname{Ruelle}}$ (compare with (\ref{eqDefZetaFn}) and (\ref{eqZetaFnOrbForm})) in terms of the operator norm of the Ruelle operator via an argument of D.~Ruelle \cite{Rue90} that matches the preimage points and periodic points of the symbolic dynamics induced by the Bowen--Ratner symbolic coding for the geodesic flows.

\smallskip
\item[(2)] By combining the bound above with D.~Dolgopyat's bound \cite{Dol98} on the norms of the Ruelle operator on some appropriate function spaces, derive a non-vanishing holomorphic extension to $\zeta_{\operatorname{Ruelle}}$ on a vertical strip $h-\epsilon \leq \Re(s) \leq h$, for some $\epsilon>0$, except for a simple pole at $s=h$, where $h\in\R$ is the smallest number such that the additive form of $\zeta_{\operatorname{Ruelle}} (s)$ converges on $\{ s\in\C \,|\, \Re(s)> h\}$, and additionally, obtain a quantitative bound of $\abs{ \zeta_{\operatorname{Ruelle}} }$ on this strip (compare with (\ref{eqZetaBound_SFT})).

\smallskip
\item[(3)] Establish the Prime Orbit Theorem with an exponential error term from the bound of $\abs{ \zeta_{\operatorname{Ruelle}} }$ above via standard arguments from analytic number theory.
\end{enumerate}

Variations and simplifications of this general strategy of M.~Pollicott and R.~Sharp, relying on the machinery of D.~Dolgopyat, have been adapted by many authors in various contexts, see for example, F.~Naud \cite{Na05}, L.~N.~Stoyanov \cite{St11}, P.~Giulietti, C.~Liverani, and M.~Pollicott \cite{GLP13}, H.~Oh and D.~Winter \cite{OW16, OW17}, D.~Winter \cite{Wi16}, etc. The importance of the analytic properties of various dynamical zeta functions in understanding the distribution of periodic orbits now becomes apparent. Not surprisingly, in view of the connection between \cite{Dol98} and \cite{PoSh98}, dynamical zeta functions are also closely related to the decay of correlations and resonances. As M.~Pollicott has put it, these are basically ``two sides of the same coin''. For related researches on the side of decay of correlations, see for example,  D.~Dolgopyat \cite{Dol98}, C.~Liverani \cite{Liv04}, A.~Avila, S.~Gou\"ezel, and J.~C.~Yoccoz \cite{AGY06}, L.~N.~Stoyanov \cite{St01, St11}, V.~Baladi and C.~Liverani \cite{BalLiv12}, V.~Baladi, M.~Demers, and C.~Liverani \cite{BDL18}, etc.

In the context of convex-cocompact surfaces $M$ of constant negative curvature, i.e., $M = \Gamma \backslash \H^2$ being the quotient of a \emph{classical Fuchsian Schottky group} $\Gamma$ (see \cite[Section~4.1]{Na05}) acting on the hyperbolic plane $\H^2$, F.~Naud \cite{Na05} established in 2005 a Prime Orbit Theorem with an exponential error term by producing some vertical strip in $\C$ on which the Selberg zeta function $\zeta_{\operatorname{Selberg}}$ (resp.\ the Ruelle zeta function $\zeta_{\operatorname{Ruelle}}$) has a non-vanishing holomorphic extension except a simple zero (resp.\ pole, see (\ref{eqRuelleSelberg})). For stronger results on the zero free strip and distribution of zeros in these contexts, see the recent works of J.~Bourgain, A.~Gamburd, and P.~Sarnak \cite{BGS11}, F.~Naud \cite{Na14}, H.~Oh and D.~Winter \cite{OW16}, S.~Dyatlov and J.~Zahl \cite{DZ16}, J.~Bourgain and S.~Dyatlov \cite{BD17}.

In the context of subgroups of the group of orientation preserving isometries of higher dimensional real hyperbolic space $\H^n$ and more general settings, T.~Roblin \cite{Ro03} proved a Prime Orbit Theorem without an error term for geometrically finite subgroups, G.~A.~Margulis, A.~Mohammadi, and H.~Oh \cite{MMO14} established an exponential error term for geometrically finite subgroups under additional conditions, and D.~Winter \cite{Wi16} showed a Prime Orbit Theorem with an exponential error term for convex-cocompact subgroups. A form of Prime Orbit Theorem without an error term for abelian covers of some hyperbolic manifolds was established by H.~Oh and W.~Pan \cite{OP18}.

In the same work \cite{Na05}, F.~Naud also established the first Prime Orbit Theorem with an exponential error term in complex dynamics, for a class of hyperbolic polynomials $z^2 + c$, $c\in(-\infty, -2)$. One key feature of this class of polynomials is that their Julia sets are Cantor sets. For an earlier work on dynamical zeta functions for a class of sub-hyperbolic quadratic polynomials, see V.~Baladi, Y.~Jiang, and H.~H.~Rugh \cite{BJR02}. For hyperbolic rational maps, S.~Waddington studied a variation of the Ruelle zeta function defined by strictly preperiodic points instead of periodic points (compare with (\ref{eqDefZetaFn}) and (\ref{eqZetaFnOrbForm})), and established a corresponding form of Prime Orbit Theorem without an error term in \cite{Wad97}.

The study of iterations of polynomials and rational maps, known as complex dynamics, dates back to the work of G.~K{\oe}nigs, E.~Schr\"oder, and others in the 19th century. This subject was developed into an active area of research, thanks to the remarkable works of S.~Latt\`{e}s, C.~Carath\'eodory, P.~Fatou, G.~Julia, P.~Koebe, L.~Ahlfors, L.~Bers, M.~Herman, A.~Douady, D.~P.~Sullivan, J.~H.~Hubbard, W.~P.~Thurston, J.-C.~Yoccoz, C.~McMullen, J.~Milnor, M.~Lyubich, M.~Shishikura, and many others.

In the early 1980s, D.~P.~Sullivan \cite{Su85, Su83} introduced a ``dictionary'', known as \emph{Sullivan's dictionary} nowadays, linking the theory of complex dynamics with another classical area of conformal dynamical systems, namely, geometric group theory, mainly concerning the study of Kleinian groups acting on the Riemann sphere. Many dynamical objects in both areas can be similarly defined and results similarly proven, yet essential and important differences remain.

The Prime Orbit Theorems with exponential error terms of F.~Naud in \cite{Na05} can be considered as another new correspondence in Sullivan's dictionary. Despite active researches on dynamical zeta functions and Prime Orbit Theorems in many areas of dynamical systems, especially the works of L.~N.~Stoyanov \cite{St11}, G.~A.~Margulis, A.~Mohammadi, and H.~Oh \cite{MMO14}, and D.~Winter \cite{Wi16} on the group side of Sullivan's dictionary, the authors are not aware of similar entries in complex dynamics since F.~Naud \cite{Na05}, until the recent work of H.~Oh and D.~Winter \cite{OW17}. At a suggeston of D.~P.~Sullivan regarding holonomies, H.~Oh and D.~Winter established a Prime Orbit Theorem (as well as the equidistribution of holonomies) with an exponential error term for hyperbolic rational maps in \cite{OW17}. A rational map is \emph{hyperbolic} if the closure of the union of forward orbits of critical points is disjoint from its Julia set. The Julia set of a hyperbolic rational map has zero area. A rational map is forward-expansive on some neighborhood of its Julia set if and only if it is hyperbolic. The novelty and emphasis of this paper somewhat differs from that of \cite{OW17}, see Subsection~\ref{subsctPlan} for more details.

In Sullivan's dictionary, Kleinian groups, i.e., discrete subgroups of M\"obius transformations on the Riemann sphere, correspond to rational maps, and convex-cocompact Kleinian groups correspond to rational maps that exhibit strong expansion properties such as hyperbolic rational maps, semi-hyperbolic rational maps, and postcritically-finite sub-hyperbolic rational maps. See insightful discussions on this part of the dictionary in \cite[Chapter~1]{BM17}, \cite[Chapter~1]{HP09}, and \cite[Section~1]{LM97}.

One important question in conformal dynamical systems is: ``\emph{What is special about conformal dynamical systems in a wider class of dynamical systems characterized by suitable metric-topological conditions?}''

W.~P.~Thurston gave an answer to this question in his celebrated combinatorial characterization theorem of \emph{postcritically-finite}  rational maps (i.e., the union of forward orbits of critical points is a finite set) on the Riemann sphere among a class of more general topological maps, known as Thurston maps nowadays \cite{DH93}. A \emph{Thurston map} is a (non-homeomorphic) branched covering map on the topological $2$-sphere $S^2$ whose finitely many critical points are all preperiodic (see Subsection~\ref{subsctThurstonMap} for a precise definition). Thurston's theorem asserts that a Thurston map is essentially a rational map if and only if there exists no so-called \emph{Thurston obstruction}, i.e., a collection of simple closed curves on $S^2$ subject to certain conditions \cite{DH93}. 

Under Sullivan's dictionary, the counterpart of Thurston's theorem in the geometric group theory is Cannon's Conjecture \cite{Ca94}. This conjecture predicts that an infinite, finitely presented Gromov hyperbolic group $G$ whose boundary at infinity $\partial_\infty G$ is a topological $2$-sphere is a Kleinian group. Gromov hyperbolic groups can be considered as metric-topological systems generalizing the conformal systems in the context, namely, convex-cocompact Kleinian groups. Inspired by Sullivan's dictionary and their interest in Cannon's Conjecture, M.~Bonk and D.~Meyer, along with others, studied a subclass of Thurston maps by imposing some additional condition of expansion. A new characterization theorem of rational maps from a metric space point of view is established in this context by M.~Bonk and D.~Meyer \cite{BM10, BM17}, and by P.~Ha\"issinsky and K.~M.~Pilgrim \cite{HP09}. Roughly speaking, we say that a Thurston map is \emph{expanding} if for any two points $x,y\in S^2$, their preimages under iterations of the map get closer and closer. For each expanding Thurston map, we can equip the $2$-sphere $S^2$ with a natural class of metrics, called \emph{visual metrics}. As the name suggests, these metrics are constructed in a similar fashion as the visual metrics on the boundary $\partial_\infty G$ of a Gromov hyperbolic group $G$. See Subsection~\ref{subsctThurstonMap} for a more detailed discussion on these notions.

\begin{theorem}[M.~Bonk \& D.~Meyer \cite{BM10, BM17}, P.~Ha\"issinsky \& K.~M.~Pilgrim \cite{HP09}]  \label{thmBM}
An expanding Thurston map is conjugate to a rational map if and only if the sphere $(S^2,d)$ equipped with a visual metric $d$ is quasisymmetrically equivalent to the Riemann sphere $\widehat\C$ equipped with the chordal metric.
\end{theorem}   

See \cite[Theorem~18.1~(ii)]{BM17} for a proof. For an equivalent formulation of Cannon's conjecture from a similar point of view, see \cite[Conjecture~5.2]{Bon06}. The definition of the chordal metric is recalled in Remark~\ref{rmChordalVisualQSEquiv} and the notion of quasisymmetric equivalence in Subsection~\ref{subsctLattes}.

We remark on the subtlety of the expansion property of expanding Thurston maps by pointing out that such maps are never forward-expansive due to the critical points. In fact, each expanding Thurston map without periodic critical points is \emph{asymptotically $h$-expansive}, but not \emph{$h$-expansive}; on the other hand, expanding Thurston maps with at least one periodic critical point are not even asymptotically $h$-expansive \cite{Li15}. Asymptotic $h$-expansiveness and $h$-expansiveness are two notions of weak expansion introduced by M.~Misiurewicz \cite{Mis73} and R.~Bowen \cite{Bow72}, respectively. Note that forward-expansiveness implies $h$-expansiveness, which in turn implies asymptotic $h$-expansiveness \cite{Mis76}.

Thanks to the fundamental works of W.~P.~Thurston, M.~Bonk, D.~Meyer, P.~Ha\"issinsky, and K.~M.~Pilgrim, the dynamics and geometry of expanding Thurston maps and similar topological branched covering maps has attracted considerable amount of interests, with motivations from both complex dynamics as well as Sullivan's dictionary. Under the dictionary, an expanding Thurston map corresponds to a Gromov hyperbolic group whose boundary at infinity is the topological $2$-sphere, and the special case of a rational expanding Thurston map (i.e., a postcritically-finite rational map whose Julia set is the whole Riemann sphere) corresponds to a convex-cocompact Kleinian group whose limit set is homeomorphic to a $2$-sphere (i.e., a cocompact lattice of $\rm{PSL}(2,\C)$) (see \cite[Chapter~1]{BM17}, \cite[Section~1]{Yi15}, and compare with \cite[Chapter~1]{HP09}).

Lastly, we want to remark that we have not been able to make connections to another successful approach to dynamical zeta functions dating back to the work of J.~Milnor and W.~P.~Thurston in 1988 on the \emph{kneading determinant} for real $1$-dimensional dynamics with critical points \cite{MT88}. The Milnor--Thurston kneading theory has been developed and used by many authors since then, see for example, V.~Baladi and D.~Ruelle \cite{BR96}, V.~Baladi, A.~Kitaev, D.~Ruelle, and S.~Semmes \cite{BKRS97}, M.~Baillif \cite{Bai04}, M.~Baillif and V.~Baladi \cite{BB05}, H.~H.~Rugh \cite{Rug16}, and V.~Baladi \cite[Chapter~3]{Bal18}.

\subsection{Main results}

The main goal of this paper is to establish a Prime Orbit Theorem with an exponential error term for expanding Thurston maps by a quantitative investigation on the holomorphic extension properties of the related dynamical zeta functions as well as more general \emph{dynamical Dirichlet series}. In the holomorphic context, as a special case, these results hold for postcritically-finite rational maps whose Julia set is the whole Riemann sphere.

To the best of the authors' knowledge, ours is the first Prime Orbit Theorem with an exponential error term in complex dynamics outside of hyperbolic rational maps, in constract to the abundance of remarkable results on the other side of Sullivan's dictionary mentioned above.\footnote{It has come to our attention that M.~Pollicott and M.~Urba\'{n}ski have recently completed a monograph \cite{PoU17} in which they established, among other things, asymptotic counting results without an error term for periodic points (as opposed to primitive periodic orbits considered in this paper) for a remarkable collection of hyperbolic and parabolic conformal dynamical systems, among them, hyperbolic and parabolic rational maps. Our emphasis is different and results disjoint from \cite{PoU17}.} We also want to emphasize that our setting is completely topological, without any holomorphicity or smoothness assumptions on the dynamical systems or the potentials, with metric and geometric structures arising naturally from the dynamics of our maps, while most if not all of the previous results of Prime Orbit Theorems were established for smooth dynamical systems.

Much of the difficulty in the study of the ergodic theory of complex dynamics comes from the singularities caused by critical points in the Julia set. In this sense, postcritically-finite maps are naturally the first class of rational maps to be considered after hyperbolic rational maps. We believe that the techniques and approaches we develop in this paper can be used in the investigations of dynamical zeta functions and Prime Orbit Theorems for more general rational maps and other (non-smooth) branched covering maps on topological spaces.

Before stating our main results, we first briefly recall dynamical zeta functions and define dynamical Dirichlet series in our context. See Subsection~\ref{subsctDynZetaFn} for a more detailed discussion.

Let $f\: S^2\rightarrow S^2$ be an expanding Thurston map and $\psi\in\CCC(S^2,\C)$ be a complex-valued continuous function on $S^2$. We denote by the formal infinite product
\begin{equation*}  
\zeta_{f,\,\minus\psi} (s) \coloneqq 
\exp \Biggl( \sum\limits_{n=1}^{+\infty} \frac{1}{n} \sum\limits_{ x = f^n(x) } e^{-s S_n \psi(x)} \Biggr), \qquad s\in\C,
\end{equation*}
the \defn{dynamical zeta function} for the map $f$ and the \emph{potential} $\psi$. Here we write $S_n \psi (x)  \coloneqq \sum_{j=0}^{n-1} \psi(f^j(x))$ as defined in (\ref{eqDefSnPt}). We remark that $\zeta_{f,\,\minus\psi}$ is the Ruelle zeta function for the suspension flow over $f$ with roof function $\psi$ if $\psi$ is positive. We define the \emph{dynamical Dirichlet series} associated to $f$ and $\psi$ as the formal infinite product
\begin{equation*}  
\DS_{f,\,\minus\psi,\, \deg_f} (s) \coloneqq \exp \Biggl( \sum\limits_{n=1}^{+\infty} \frac{1}{n} \sum\limits_{x = f^n(x)} e^{-s S_n \psi(x)} \deg_{f^n}(x) \Biggr), \qquad s\in\C.
\end{equation*}
Here $\deg_{f^n}$ is the \emph{local degree} of $f^n$ at $x\in S^2$ (see Definition~\ref{defBranchedCover}).

Note that if $f\: S^2 \rightarrow S^2$ is an expanding Thurston map, then so is $f^n$ for each $n\in\N$ (Remark~\ref{rmExpanding}). 

Recall that a function is holomorphic (resp.\ meromorphic) on a closed set if it is holomorphic (resp.\ meromorphic) on an open set containing this closed set.

\begin{theorem}[Holomorphic extensions of dynamical Dirichlet series and zeta functions for expanding Thurston maps]  \label{thmZetaAnalExt_InvC}
Let $f\: S^2 \rightarrow S^2$ be an expanding Thurston map, and $d$ be a visual metric on $S^2$ for $f$. Given $\alpha\in(0,1]$. Let $\phi \in \Holder{\alpha}(S^2,d)$ be an eventually positive real-valued H\"{o}lder continuous function that is not co-homologous to a constant in $\CCC( S^2 )$. Denote by $s_0$ the unique positive number with $P(f,-s_0 \phi) = 0$.  

Then there exists $N_f\in\N$ depending only on $f$ such that for each $n\in \N$ with $n\geq N_f$, the following statements hold for $F\coloneqq f^n$ and $\Phi\coloneqq \sum_{i=0}^{n-1} \phi \circ f^i$:

\begin{enumerate}
\smallskip
\item[(i)] Both the dynamical zeta function $\zeta_{F,\,\minus \Phi} (s)$ and the dynamical Dirichlet series $\DS_{F,\,\minus \Phi,\,\deg_F} (s)$ converge on the open half-plane $\{s\in\C \,|\, \Re(s) > s_0 \}$ and extend to non-vanishing holomorphic functions on the closed half-plane $\{s\in\C \,|\, \Re(s) \geq s_0\}$ except for the simple pole at $s=s_0$.

\smallskip
\item[(ii)] Assume in addition that $\phi$ satisfies the $\alpha$-strong non-integrability condition. Then there exists a constant $\epsilon_0 \in (0, s_0)$ such that both the dynamical zeta function $\zeta_{F,\,\minus \Phi} (s)$ and the dynamical Dirichlet series $\DS_{F,\,\minus \Phi,\,\deg_F} (s)$ converge on the open half-plane $\{s\in\C \,|\, \Re(s) > s_0 \}$ and extend to non-vanishing holomorphic functions on the closed half-plane $\{s\in\C \,|\, \Re(s) \geq s_0 - \epsilon_0 \}$ except for the simple pole at $s=s_0$. Moreover, for each $\epsilon >0$, there exist constants $C_\epsilon >0$, $a_\epsilon \in (0, \epsilon_0]$, and $b_\epsilon\geq  2 s_0+1$ such that
\begin{equation}   \label{eqZetaBound}
     \exp\left( - C_\epsilon \abs{\Im(s)}^{2+\epsilon} \right) 
\leq \abs{\zeta_{ F,\,\minus\Phi} (s) }
\leq \exp\left(   C_\epsilon \abs{\Im(s)}^{2+\epsilon} \right) 
\end{equation}
and
\begin{equation}   \label{eqWeightedZetaBound}
     \exp\left( - C_\epsilon \abs{\Im(s)}^{2+\epsilon} \right) 
\leq \Absbig{ \DS_{ F,\,\minus\Phi,\,\deg_F} (s) }
\leq \exp\left(   C_\epsilon \abs{\Im(s)}^{2+\epsilon} \right) 
\end{equation}
for all $s\in\C$ with $\abs{\Re(s) - s_0} < a_\epsilon$ and $\abs{\Im(s)}  \geq b_\epsilon$.
\end{enumerate}

\end{theorem}

Here by $P(f,\psi)$ we denote the topological pressure of $f$ with respect to a potential $\psi \in \CCC(S^2)$ (see Subsection~\ref{subsctThermodynFormalism}).  A real-valued continuous function $\phi \in \CCC(S^2)$ is \emph{co-homologous} to a constant  in $\CCC( S^2 )$ if there is a constant $K\in\R$ and a real-valued continuous function $\beta \in \CCC ( S^2 )$ with $\phi = K + \beta \circ f - \beta$. The function $\phi$ is \emph{eventually positive} if $\phi + \phi\circ f + \dots + \phi \circ f^n$ is strictly positive on $S^2$ for all sufficiently large $n\in\N$ (see Definition~\ref{defEventuallyPositive}). We postpone the discussion of the $\alpha$-strong non-integrability condition on $\phi$ until Theorem~\ref{thmSNIGeneric}.

\begin{rem}    \label{rmNf}
The integer $N_f$ can be chosen as the minimum of $N(f,\wt{\CC})$ from Lemma~\ref{lmCexistsL} over all Jordan curves $\wt{\CC}$ with $\post f \subseteq \wt{\CC} \subseteq S^2$, in which case $N_f = 1$ if there exists an Jordan curve $\CC\subseteq S^2$ satisfying $f(\CC)\subseteq \CC$, $\post f\subseteq \CC$, and no $1$-tile in $\X^1(f,\CC)$ joins opposite sides of $\CC$ (see Definition~\ref{defJoinOppositeSides}). The same number $N_f$ will be used in the statements of Theorem~\ref{thmLogDerivative}, Theorem~\ref{thmPrimeOrbitTheorem}, and Theorem~\ref{thmLattesPOT}. Here the set $\X^1(f,\CC)$ of $1$-tiles consists of closures of connected components of $S^2 \setminus f^{-1}(\CC)$ (see Subsection~\ref{subsctThurstonMap} for more detailed discussions). We also remark that most properties of expanding Thurston maps $f$ whose proofs rely on the Markov partitions can be established for $f$ after being verified first for $f^n$ for all $n\geq N_f$. However, some of the finer properties established for iterates of $f$ still remain open for the map $f$ itself, see for example, \cite{Mey13, Mey12}. That being said, we do expect all theorems in this subsection to hold for $f$ itself, but the verification may require new techniques.
\end{rem}

Below is a symbolic version of Theorem~\ref{thmZetaAnalExt_InvC}. For the notion of subshift of finite type, and the corresponding subshift of finite type $\bigl( \Sigma_{A_{\ti}}^+, \sigma_{A_{\ti}} \bigr)$ for an expanding Thurston map induced by a Jordan curve from Remark~\ref{rmNf}, see Proposition~\ref{propTileSFT} and the general discussions in Subsection~\ref{subsctSFT}.

\begin{theorem}[Holomorphic extensions of the symbolic zeta functions]  \label{thmZetaAnalExt_SFT}
Let $f\: S^2 \rightarrow S^2$ be an expanding Thurston map with an Jordan curve $\CC\subseteq S^2$ satisfying $f(\CC)\subseteq \CC$, $\post f\subseteq \CC$, and no $1$-tile in $\X^1(f,\CC)$ joins opposite sides of $\CC$. Let $d$ be a visual metric on $S^2$ for $f$. Given $\alpha\in(0,1]$. Let $\phi \in \Holder{\alpha}(S^2,d)$ be an eventually positive real-valued H\"{o}lder continuous function that is not co-homologous to a constant in $\CCC( S^2 )$. Denote by $s_0$ the unique positive number with $P(f,-s_0 \phi) = 0$. Let $\bigl(\Sigma_{A_{\ti}}^+,\sigma_{A_{\ti}}\bigr)$ be the one-sided subshift of finite type associated to $f$ and $\CC$ defined in Proposition~\ref{propTileSFT}, and let $\pi_{\ti}\: \Sigma_{A_{\ti}}^+\rightarrow S^2$ be the factor map as defined in (\ref{eqDefTileSFTFactorMap}).

Then the dynamical zeta function $\zeta_{\sigma_{A_{\ti}},\,\minus \phi\circsmall\pi_{\ti}} (s)$ converges on the open half-plane $\{s\in\C \,|\, \Re(s) > s_0 \}$, and the following statements hold:
\begin{enumerate}
\smallskip
\item[(i)] The function $\zeta_{\sigma_{A_{\ti}},\,\minus \phi\circsmall\pi_{\ti}} (s)$ extends to a non-vanishing holomorphic function on the closed half-plane $\{s\in\C \,|\, \Re(s) \geq s_0 \}$ except for the simple pole at $s=s_0$. 

\smallskip
\item[(ii)] Assume in addition that $\phi$ satisfies the $\alpha$-strong non-integrability condition. Then there exists a constant $\wt\epsilon_0 \in (0, s_0)$ such that $\zeta_{\sigma_{A_{\ti}},\,\minus \phi\circsmall\pi_{\ti}} (s)$  extends to a non-vanishing holomorphic function on the closed half-plane $\{s\in\C \,|\, \Re(s) \geq s_0 - \wt\epsilon_0 \}$ except for the simple pole at $s=s_0$. Moreover, for each $\epsilon >0$, there exist constants $\wt{C}_\epsilon >0$, $\wt{a}_\epsilon \in (0,s_0)$, and $\wt{b}_\epsilon \geq 2 s_0 + 1$ such that
\begin{equation}   \label{eqZetaBound_SFT}
     \exp\bigl( - \wt{C}_\epsilon \abs{\Im(s)}^{2+\epsilon} \bigr) 
\leq \Absbig{ \zeta_{\sigma_{A_{\ti}},\,\minus \phi\circsmall\pi_{\ti}} (s) }
\leq \exp\bigl(   \wt{C}_\epsilon \abs{\Im(s)}^{2+\epsilon} \bigr) 
\end{equation}
for all $s\in\C$ with $\abs{\Re(s) - s_0} < \wt{a}_\epsilon$ and $\abs{\Im(s)}  \geq \wt{b}_\epsilon$.
\end{enumerate}
\end{theorem}

Theorem~\ref{thmZetaAnalExt_InvC} leads to the following bound for the logarithmic derivative of the zeta function $\zeta_{F,\, \minus\Phi}$.

\begin{theorem}  \label{thmLogDerivative}
Let $f\: S^2 \rightarrow S^2$ be an expanding Thurston map, and $d$ be a visual metric on $S^2$ for $f$. Let $\phi \in \Holder{\alpha}(S^2,d)$ be an eventually positive real-valued H\"{o}lder continuous function with an exponent $\alpha\in (0,1]$ that satisfies the $\alpha$-strong non-integrability condition. Denote by $s_0$ the unique positive number with $P(f,-s_0 \phi) = 0$. 

Then there exists $N_f\in\N$ depending only on $f$ such that for each $n\in \N$ with $n\geq N_f$, the following statement holds for $F\coloneqq f^n$ and $\Phi\coloneqq \sum_{i=0}^{n-1} \phi \circ f^i$:

There exist constants $a \in (0, s_0)$, $b\geq 2 s_0 + 1$, and $D>0$ such that 
\begin{equation}   \label{eqLogDerivative}
          \Absbigg{  \frac{ \zeta'_{F,\, \minus\Phi}(s) }{ \zeta_{F,\, \minus\Phi}(s) } }
\leq  D \abs{\Im(s)}^{\frac12}
\end{equation}
for all $s\in\C$ with $\abs{\Re(s) - s_0} < a$ and $\abs{\Im(s)}  \geq b$.
\end{theorem}

Given an expanding Thurston map  $f\: S^2 \rightarrow S^2$ and a real-valued function $\psi\: S^2 \rightarrow \R$,  we define the weighted length $l_{f,\psi} (\tau) $ (induced by $f$ and $\psi$) of a primitive periodic orbit 
\begin{equation*}
\tau \coloneqq \{x, f(x), \cdots, f^{n-1}(x) \} \in \Orb(f)
\end{equation*}
as
\begin{equation}   \label{eqDefComplexLength}
l_{f,\psi} (\tau) \coloneqq \psi(x) + \psi(f(x)) + \cdots + \psi(f^{n-1}(x)).
\end{equation}
We denote by
\begin{equation}   \label{eqDefPiT}
\pi_{f,\psi}(T) \coloneqq \card \{ \tau \in \Orb(f)  \,|\, l_{f,\psi}( \tau )  \leq T \}, \qquad \text{for } T>0,
\end{equation}
the number of primitive periodic orbits with weighted length upto $T$. See (\ref{eqDefSetAllPeriodicOrbits}) for the precise definition of $\Orb(f)$.

The corresponding Prime Orbit Theorems follow from Theorem~\ref{thmZetaAnalExt_InvC} and Theorem~\ref{thmLogDerivative}.

\begin{theorem}[Prime Orbit Theorems for expanding Thurston maps]  \label{thmPrimeOrbitTheorem}
Let $f\: S^2 \rightarrow S^2$ be an expanding Thurston map, and $d$ be a visual metric on $S^2$ for $f$. Let $\phi \in \Holder{\alpha}(S^2,d)$ be an eventually positive real-valued H\"{o}lder continuous function with an exponent $\alpha\in (0,1]$. Denote by $s_0$ the unique positive number with $P(f,-s_0 \phi) = 0$. 

Then there exists $N_f\in\N$ depending only on $f$ such that for each $n\in \N$ with $n\geq N_f$, the following statements hold for $F\coloneqq f^n$ and $\Phi\coloneqq \sum_{i=0}^{n-1} \phi \circ f^i$:

\begin{enumerate}
\smallskip
\item[(i)] The asymptotic relation
\begin{equation*} 
\pi_{F,\Phi}(T) \sim \operatorname{Li}\bigl( e^{s_0 T} \bigr)     \qquad\qquad \text{as } T \to + \infty
\end{equation*}
holds if and only if $\phi$ is not co-homologous to a constant in $\CCC( S^2 )$, i.e., there are no constant $K\in\R$ and function $\beta \in \CCC ( S^2 )$ with $\phi = K + \beta \circ f - \beta$.

\smallskip
\item[(ii)] Assume that $\phi$ satisfies the $\alpha$-strong non-integrability condition. Then there exists a constant $\delta \in (0, s_0)$ such that
\begin{equation*} 
\pi_{F,\Phi}(T) = \operatorname{Li}\bigl( e^{s_0 T} \bigr)  + \operatorname{O} \bigl( e^{(s_0 - \delta)T} \bigr)      \qquad\qquad \text{as } T \to + \infty.
\end{equation*}
\end{enumerate}
Here $\operatorname{Li}(\cdot)$ is the Eulerian logarithmic integral function defined in (\ref{eqDefLogIntegral}).
\end{theorem}

Note that $\lim\limits_{y\to+\infty} \frac{ \operatorname{Li}(y) }{  \frac{y}{ \log y } }  = 1$, thus we also get $\pi_{F,\Phi}(T) \sim \frac{ \exp(s_0 T) }{ s_0 T}$ as $T\to + \infty$.

Once Theorem~\ref{thmZetaAnalExt_InvC} and Theorem~\ref{thmLogDerivative} are established, Theorem~\ref{thmPrimeOrbitTheorem} follows from standard number-theoretic arguments. More precisely, a proof of the backward implication in statement~(i) in Theorem~\ref{thmPrimeOrbitTheorem}, relying on statement~(i) in Theorem~\ref{thmZetaAnalExt_InvC} and the Ikehara--Wiener Tauberian Theorem (see \cite[Appendix~I]{PP90}), is verbatim the same as that of \cite[Theorem~6.9]{PP90} on pages 106--109 of \cite{PP90} (after defining $h\coloneqq s_0$, $\lambda(\tau) \coloneqq l_{F,\Phi}(\tau)$, $\pi\coloneqq \pi_{F,\Phi}$, and $\zeta \coloneqq \zeta_{F, \, \minus s_0 \Phi}$ in the notation of \cite{PP90}) with an additional observation that $\lim\limits_{y\to+\infty} \frac{ \operatorname{Li}(y) }{  \frac{y}{ \log y } }  = 1$. The forward implication in statement~(i) in Theorem~\ref{thmPrimeOrbitTheorem} follows immediately from Proposition~\ref{propSNI2NLI} and Theorem~\ref{thmNLI}. Similarly, a proof of statement~(ii) in Theorem~\ref{thmPrimeOrbitTheorem}, relying on Theorem~\ref{thmLogDerivative} and statement~(ii) in Theorem~\ref{thmZetaAnalExt_InvC}, is verbatim the same as that of \cite[Theorem~1]{PoSh98} presented in \cite[Section~3]{PoSh98}. We omit the these proofs here and direct the interested readers to the references cited above.

 \begin{rem}
We remark that Theorems~\ref{thmZetaAnalExt_InvC}, \ref{thmZetaAnalExt_SFT}, \ref{thmLogDerivative}, and \ref{thmPrimeOrbitTheorem} apply to general expanding Thurston maps including ones that are conjugate to rational maps and ones that are \emph{obstructed} (in the sense of W.~P.~Thurston's characterization theorem). In particular, these theorems apply to the special case when $S^2$ is the Riemann sphere $\widehat{\C}$ and $f\: \widehat{\C} \rightarrow \widehat{\C}$ is a rational expanding Thurston map, i.e., $f$ is a postcritically-finite rational map whose Julia set is the whole sphere $\widehat{\C}$ (or equivalently, $f$ is a postcritically-finite rational map without periodic critical points).
 \end{rem}
 
The following Prime Orbit Theorem for rational expanding Thurston maps follows immediately from Remark~\ref{rmChordalVisualQSEquiv} and statement~(i) in Theorem~\ref{thmPrimeOrbitTheorem}.

\begin{cor}  \label{corPrimeOrbitTheorem_rational}
Let $f\: \widehat{\C} \rightarrow \widehat{\C}$ be a postcritically-finite rational map without periodic critical points. Let $\sigma$ be the chordal metric Riemann sphere $\widehat{\C}$, and $\phi \in \Holder{\alpha} \bigl( \widehat{\C}, \sigma \bigr)$ be an eventually positive real-valued H\"{o}lder continuous function with an exponent $\alpha\in (0,1]$.

Then there exists a unique positive number $s_0>0$ with $P(f,-s_0 \phi) = 0$ and there exists $N_f\in\N$ depending only on $f$ such that for each $n\in \N$ with $n\geq N_f$, the following statement holds for $F\coloneqq f^n$ and $\Phi\coloneqq \sum_{i=0}^{n-1} \phi \circ f^i$:

The asymptotic relation
\begin{equation*} 
\pi_{F,\Phi}(T) \sim \operatorname{Li}\bigl( e^{s_0 T} \bigr)     \qquad\qquad \text{as } T \to + \infty.
\end{equation*} 
holds if and only if $\phi$ is not co-homologous to a constant in $\CCC \bigl( \widehat{\C} \bigr)$.
\end{cor}

\smallskip
 
The strong non-integrability condition on the potential $\phi$ essential in the theorems mentioned above is introduced in Subsection~\ref{subsctSNI}. The idea was first used by D.~Dolgopyat \cite{Dol98}. In the contexts of classical smooth dynamical systems on Riemannian manifolds with smooth potentials, the corresponding condition is often equivalent to a weaker condition, called \emph{non-local integrability condition}, introduced in our context in Section~\ref{sctNLI}. We can actually show that in the context of expanding Thurston maps, a potential is non-locally integrable if and only if it is not co-homologous to a constant (see Theorem~\ref{thmNLI} for more details). However, as we are endeavoring out of Riemannian settings into more general self-similar metric spaces in this paper, the equivalence between the strong non-integrability condition and the non-local integrability condition for smooth potentials is not expected except for Latt\`{e}s maps, for reasons discussed in Subsection~\ref{subsctLattes}. Nevertheless, generic potentials always satisfy the strong non-integrability condition, as stated more precisely in the following theorem.
 
\begin{theorem}[Genericity]     \label{thmSNIGeneric}
Let $f\: S^2 \rightarrow S^2$ be an expanding Thurston map, and $d$ be a visual metric on $S^2$ for $f$. Given $\alpha\in(0,1]$. We equipped the set $\Holder{\alpha}(S^2,d)$ of real-valued H\"{o}lder continuous functions with an exponent $\alpha$ with the topology induced by the H\"{o}lder norm $\Hnorm{\alpha}{\cdot}{(S^2,d)}$. Let $\mathcal{S}^\alpha$ be the subset of  $\Holder{\alpha}(S^2,d)$ consisting of functions satisfies the $\alpha$-strong non-integrability condition.

Then $\mathcal{S}^\alpha$ is open in $\Holder{\alpha}(S^2,d)$. Moreover, the following statements hold:
\begin{enumerate}
\smallskip
\item[(i)] $\mathcal{S}^\alpha$ is an open dense subset of $\Holder{\alpha}(S^2,d)$ if $\alpha \in (0,1)$.

\smallskip
\item[(ii)] $\mathcal{S}^1$ is an open dense subset of $\Holder{1}(S^2,d)$ if the expansion factor $\Lambda$ of $d$ is not equal to the combinatorial expansion factor $\Lambda_0(f)$ of $f$.
\end{enumerate}
\end{theorem}

The H\"{o}lder norm $\Hnorm{\alpha}{\cdot}{(S^2,d)}$ is recalled in Section~\ref{sctNotation}. The definition of the combinatorial expansion factor $\Lambda_0(f)$ of $f$ is given in (\ref{eqDefCombExpansionFactor}). See \cite[Chapter~16]{BM17} for a more detailed discussion on $\Lambda_0(f)$. In particular, the inequality $\Lambda \leq \Lambda_0(f)$ always holds for the expansion factor $\Lambda$ of any visual metric $d$ for an expanding Thurston map $f$. 

We note here that for each $\alpha\in(0,1]$, the set of real-valued H\"{o}lder continuous functions $\phi \in \Holder{\alpha}(S^2,d)$ that are eventually positive is an open subset of $\Holder{\alpha}(S^2,d)$ with respect to either the uniform norm or the H\"older norm.

For Latt\`{e}s maps $f\: \widehat{\C} \rightarrow \widehat{\C}$ (see Definition~\ref{defLattesMap}) and continuously differentiable potentials $\phi\: \widehat{\C} \rightarrow \R$ on the Riemann sphere, the equivalence between the strong non-integrability condition and the condition that $\phi$ is not co-homologous to a constant is established in Proposition~\ref{propLattes}, depending crucially on the properties of the canonical orbifold metric (see Remark~\ref{rmCanonicalOrbifoldMetric}).

\begin{rem}    \label{rmLattesCohomology}
Let $f\: \widehat{\C} \rightarrow \widehat{\C}$ be a Latt\`{e}s map on the Riemann sphere $\widehat{\C}$ equipped with the spherical metric. It follows immediately from Proposition~\ref{propLattes} that we can replace the conditions on $\phi$ (including the additional assumptions that $\phi$ satisfies the $\alpha$-strong non-integrability condition) in Theorems~\ref{thmZetaAnalExt_InvC}, \ref{thmZetaAnalExt_SFT}, \ref{thmLogDerivative}, and \ref{thmPrimeOrbitTheorem} by the following condition:
\begin{enumerate}
\smallskip
\item[] $\phi \: \widehat{\C} \rightarrow \R$ is an eventually positive, continuously differentiable, real-valued function that is not co-homologous to a constant in $\CCC\bigl( \widehat{\C} \bigr)$, 
\end{enumerate}
and these theorems still hold.
\end{rem}

In particular, we have the following characterization of the Prime Orbit Theorem in the context of Latt\`{e}s maps.

\begin{theorem}[Prime Orbit Theorem for Latt\`{e}s maps]   \label{thmLattesPOT}
Let $f\: \widehat{\C} \rightarrow \widehat{\C}$ be a Latt\`{e}s map on the Riemann sphere $\widehat{\C}$. Let $\phi \: \widehat{\C} \rightarrow \R$ be an eventually positive, continuously differentiable real-valued function on $\widehat{\C}$. Then there exists a unique positive number $s_0>0$ with $P(f,-s_0 \phi) = 0$ and there exists $N_f\in\N$ depending only on $f$ such that the following statements are equivalent:
\begin{enumerate}
\smallskip
\item[(i)] $\phi$ is not co-homologous to a constant in $\CCC\bigl(\widehat{\C} \bigr)$, i.e., there are no constant $K\in\R$ and function $\beta \in \CCC \bigl( \widehat{\C} \bigr)$ with $\phi = K + \beta \circ f - \beta$.

\smallskip
\item[(ii)] For each $n\in \N$ with $n\geq N_f$, we have 
\begin{equation*}  
\pi_{F,\Phi}(T) \sim \operatorname{Li}\bigl( e^{s_0 T} \bigr)   \qquad\qquad \text{as } T \to + \infty,
\end{equation*}
where  $F\coloneqq f^n$ and $\Phi\coloneqq \sum_{i=0}^{n-1} \phi \circ f^i$.

\smallskip
\item[(iii)] For each $n\in \N$ with $n\geq N_f$, there exists a constant $\delta \in (0, s_0)$ such that
\begin{equation*}  
\pi_{F,\Phi}(T) = \operatorname{Li}\bigl( e^{s_0 T} \bigr)  + \operatorname{O} \bigl( e^{(s_0 - \delta)T} \bigr)      \qquad\qquad \text{as } T \to + \infty,
\end{equation*}
where  $F\coloneqq f^n$ and $\Phi\coloneqq \sum_{i=0}^{n-1} \phi \circ f^i$.
\end{enumerate}  
Here $\operatorname{Li}(\cdot)$ is the Eulerian logarithmic integral function defined in (\ref{eqDefLogIntegral}).
\end{theorem}

\subsection{Plan of the paper}      \label{subsctPlan}
We adapt the general strategy of M.~Pollicott and R.~Sharp \cite{PoSh98} mentioned above combined with the machinery of D.~Dolgopyat \cite{Dol98} in our context. Apart from our new metric-topological setting which differs from the more classical smooth settings in the literature, there are other major obstacles to carrying out this plan. We describe such difficulties and comment on our tactics in overcoming them below before giving a summary of the contents of each section.

\begin{enumerate}
\smallskip
\item[(i)] One key ingredient shared both in the strategy of M.~Pollicott and R.~Sharp \cite{PoSh98} and in the original machinery of D.~Dolgopyat \cite{Dol98} is finite symbolic codings of the dynamics\footnote{Many efforts have been made to remove the requirement of finite symbolic coding in D.~Dolgopyat's machinery for applications in decay of correlations, notably the functional approaches of V.~Baladi, S.~Gou\"{e}zel, C.~Liverani, M.~Tsujii, and others.}.  There is no general finite symbolic coding for rational maps outside of the hyperbolic case. However, thanks to the works of J.~W.~Cannon, W.~J.~Floyd, and W.~R.~Parry \cite{CFP07} and M.~Bonk and D.~Meyer \cite{BM17}, we know that for each postcritically-finite rational map whose Julia set is the whole Riemann sphere , or more generally, each expanding Thurston map $f$, there exists $N\in\N$ such that for each $n>N$, the iteration $F\coloneqq f^n$ has some forward invariant Jordan curve $\CC$ on the sphere $S^2$ that induces finite Markov partitions for $F$. We actually need a slightly stronger result (see Lemma~\ref{lmCexistsL}) that was first established in \cite{Li18}. It is known that there exist expanding Thurston maps without such forward invariant Jordan curves (see \cite[Example~15.11]{BM17}).

\smallskip
\item[(ii)] The symbolic coding induced by the finite Markov partitions mentioned above is not finite-to-one for general expanding Thurston maps $F$, and A.~Manning's argument that are used in the literature for symbolic codings that are finite-to-one does not apply here to account for the periodic points on the boundaries of cylinder sets in the symbolic space induced by the Markov partitions (that we call \emph{tiles}). To overcome this obstacle, we introduce the dynamical Dirichlet series $\DS_{F,\,\minus\psi,\, \deg_F}$ and study its holomorphic extension properties from the symbolic coding instead, reducing the part of Theorem~\ref{thmZetaAnalExt_InvC} on $\DS_{F,\,\minus\psi,\, \deg_F}$ to Theorem~\ref{thmZetaAnalExt_SFT}.  It seems to be the first instance in the field where such general dynamical Dirichlet series other than $L$-functions are crucially used.

\smallskip
\item[(iii)] D.~Dolgopyat's machinery builds upon the existence of a spectral gap of the Ruelle operator acting on appropriate function spaces from the study of thermodynamical formalism for the corresponding dynamical systems. However, for a general rational map $f$ with at least one critical point in the Julia set---so $f$ is not a hyperbolic rational map---the Ruelle operator is not invariant on the set of $\alpha$-H\"{o}lder continuous functions on the Riemann sphere equipped with the spherical metric (see \cite[Remark~3.1]{DPU96}), even though it is known $\alpha$-H\"{o}lder continuous functions are mapped by all iterations of the Ruelle operator to $\alpha'$-H\"{o}lder continuous functions for some $\alpha'\in(0,\alpha)$ sufficiently small (see \cite[Section~3]{DPU96}). It is not clear how to incorporate this technical phenomenon with the intricate machinery of D.~Dolgopyat. In the context of expanding Thurston maps in this paper, we use a fruitful point of view introduced in the thesis of the first-named author (see \cite{Li18}) in developing the thermodynamical formalism. Namely, instead of spherical metric (or chordal metric), we use the more dynamically natural metrics on the sphere, i.e., the visual metrics. In this setting, the Ruelle operator (for an $\alpha$-H\"{o}lder continuous potential)
preserves the set of $\alpha$-H\"{o}lder continuous functions for $\alpha\in(0,1]$, and moreover, most of the classical results similar to those in the thermodynamical formalism for expansive systems with specification property hold including the existence of spectral gap. See \cite{Li18} for a detailed study of the thermodynamical formalism in this setting.

\smallskip
\item[(iv)] One still cannot use the Ruelle operator for an expanding Thurston map and a H\"{o}lder continuous potential introduced in \cite{Li18} in the proof of Theorem~\ref{thmZetaAnalExt_SFT} directly. One key problem is that in order to use D.~Ruelle's argument to bound the dynamical Dirichlet series in terms of the operator norm of the Ruelle operator by matching the preimage points in the Ruelle operator with periodic points in the dynamical Dirichlet series, one needs to apply the Ruelle operator to characteristic functions supported on the tiles. However, such characteristic functions are not (H\"{o}lder) continuous. Our strategy here is to ``split'' the Ruelle operator $\RR_\psi$ into ``pieces'' $\RR_{\psi, X^0_\c, E}^{(n)}$ corresponding to each tile (see Definition~\ref{defSplitRuelle}) and to piece them together to define a new operator $\RRR_\psi$ on the product space $\CCC\bigl( X^0_\b, \C \bigr) \times \CCC\bigl( X^0_\w, \C \bigr)$ that we call the \emph{split Ruelle operator} for the map $F$ and potential $\psi$ (see Definition~\ref{defSplitRuelleOnProductSpace}). Here $X^0_\b$ and $X^0_\w$ are the (closures) of the Jordan regions on the sphere induced by the invariant Jordan curve $\CC$. We then deduce various properties including spectral gap of $\RRR_\psi$ from, or in a similar way to, those of $\RR_\psi$ in Section~\ref{sctSplitRuelleOp}.

\smallskip
\item[(v)] We also need to relate the dynamical Dirichlet series $\DS_{F,\,\minus\psi,\, \deg_F}$ and the dynamical zeta function $\zeta_{F,\,\minus\psi}$. As mentioned above, A.~Manning's argument used in the literature for other dynamical systems does not apply in our setting since our symbolic coding is not finite-to-one in general. Instead, we establish the quotient of $\DS_{F,\,\minus\psi,\, \deg_F}$ and $\zeta_{F,\,\minus\psi}$ as a combination of a product and a quotient of the dynamical zeta functions for three symbolic dynamical systems on the boundaries of the tiles (see (\ref{eqDirichletSeriesIsCombZetaFns})) by a careful study of the combinatorics of tiles (see Subsection~\ref{subsctDynOnC_Combinatorics}) with ideas from \cite{Li16}. 

\smallskip
\item[(vi)] In order to deduce the holomorphic extension properties of $\zeta_{F,\,\minus\psi}$ from those of $\DS_{F,\,\minus\psi,\, \deg_F}$, we need to show that there are strictly less dynamics on the boundaries of tiles than on the tiles themselves measured by the topological pressures of corresponding symbolic dynamical systems, which is not a priori clear. A key construction (see (\ref{eqDefEm})) first introduced in \cite{Li15} is used in Subsection~\ref{subsctDynOnC_TopPressure} to establish such relations in Theorem~\ref{thmPressureOnC}.

\smallskip
\item[(vii)] The existence of critical points in the Julia set, and more seriously of periodic critical points for some expanding Thurston maps, also gives rise to obstacles in the proof of the equivalence of the non-local integrability condition (see Definition~\ref{defLI}) and the cohomology conditions on the potentials (see Theorem~\ref{thmNLI}). For example, an inverse branch of an expanding Thurston map with a periodic critical point may not have a fixed point. We nevertheless successfully establish such equivalence in Theorem~\ref{thmNLI} by a careful study of the universal orbifold covers, introduced by W.~P.~Thurston \cite{Th80} in 1970s for geometry of $3$-manifolds, in our context in Section~\ref{sctNLI}.

\smallskip
\item[(viii)] In the case of rational expanding Thurston maps $f\: \widehat{\C}  \rightarrow \widehat{\C}$ and continuously differentiable potentials $\phi \: \widehat{\C} \rightarrow \R$, one would hope to establish the Prime Orbit Theorem under the weaker assumption that $\phi$ is not co-homologous to a constant, or more precisely, to derive from this assumption the strong non-integrability condition. This is often done in the literature in smooth dynamical systems. However, the dynamically more nature metrics, i.e., the visual metrics, that we use to develop the thermodynamical formalism are not quite compatible with the smooth structure on $\widehat{\C}$. In fact, the chordal metric $\sigma$ (see Remark~\ref{rmChordalVisualQSEquiv} for the definition) is never a visual metric for $f$ (see \cite[Lemma~8.12]{BM17}), and $\bigl(\widehat{\C}, \sigma \bigr)$ is snowflake equivalent to $(S^2,d)$ for some visual metric $d$ if and only if $f$ is topologically conjugate to a Latt\`{e}s map (see \cite[Theorem~18.1~(iii)]{BM17} and Definition~\ref{defLattesMap} below). Moreover, the canonical orbifold metric $\omega_f$ (see Remark~\ref{rmCanonicalOrbifoldMetric}) is a visual metric if and only if $f$ is a Latt\`{e}s map. Nevertheless, we provide a positive answer in Theorem~\ref{thmLattesPOT} for Latt\`{e}s maps using properties of the canonical orbifold metric in this setting. The case of all rational expanding Thurston maps is still open.
\end{enumerate}

\smallskip

We will now give a brief description of the structure of this paper.

After fixing some notation in Section~\ref{sctNotation}, we give a review of the dynamical and geometric notions and basic facts needed in this paper in Section~\ref{sctPreliminaries}. We first discuss thermodynamical formalism in Subsection~\ref{subsctThermodynFormalism} very briefly. General branched covering maps between surfaces are reviewed in Subsection~\ref{subsctBranchedCoveringMaps} even though we are only concerned with the special case of Thurston maps on the topological $2$-sphere in the whole paper except Section~\ref{sctNLI}, . 

Subsection~\ref{subsctThurstonMap} focuses on the main objects of investigation of this paper, namely, expanding Thurston maps $f$. Many notions and results from M.~Bonk and D.~Meyer \cite{BM17} and the previous works of the first-named author \cite{Li16, Li15, Li18, Li17} that are crucially used in this paper are recorded for the convenience of the reader. We recall the notion of \emph{eventually positive potentials} at the end of this subsection and show that for such a potential $\phi$, the topological pressure $t \mapsto P(f, -t\phi)$ has unique zero $t= s_0$ in Corollary~\ref{corS0unique}.

In Subsection~\ref{subsctSFT}, we first recall and set notation for the symbolic dynamical systems known as subshifts of finite type. Some basic properties of the topological pressure and its relations with periodic points and preimage points in this context are recalled in Proposition~\ref{propSFT}, Lemma~\ref{lmUnifBddToOneFactorPressure}, and Lemma~\ref{lmCylinderIsTile}. We then construct a one-sided subshift of finite type $\bigl(\Sigma_{A_{\ti}}^+,\sigma_{A_{\ti}}\bigr)$ in Proposition~\ref{propTileSFT} encoding the dynamics on the tiles induced by an expanding Thurston map with some forward invariant Jordan curve $\CC\subseteq S^2$, and show that $f$ is a factor of $\bigl(\Sigma_{A_{\ti}}^+,\sigma_{A_{\ti}} \bigr)$ with the factor map $\pi_\ti \: \Sigma_{A_{\ti}}^+ \rightarrow S^2$ given by (\ref{eqDefTileSFTFactorMap}).

In Subsection~\ref{subsctDynZetaFn}, we recall in (\ref{eqDefZetaFn}) the notion of \emph{dynamical zeta function} $\zeta_{g,\,\minus\psi}$ for a continuous map $g$ on a compact metric space $(X,d)$ and a complex-valued continuous function $\psi$ on $X$. More generally, for an additional complex-valued function $w$ on $X$, we introduce the \emph{dynamical Dirichlet series} $\DS_{g,\,\minus\psi,\,w}$ in Definition~\ref{defDynDirichletSeries} as an analogue of Dirichlet series in analytic number theory, where $w$ corresponds to a strongly multiplicative arithmetic function. We establish the convergence of $\DS_{g,\,\minus\psi,\,w}$ and an analogue of the Euler product formula for $\DS_{g,\,\minus\psi,\,w}$ in Lemma~\ref{lmDynDirichletSeriesConv_general} under relevant assumptions in our context on periodic points. The corresponding results for $\zeta_{f,\,\minus\phi}$, $\DS_{f,\,\minus\phi,\,w}$, and $\zeta_{\sigma_{A_{\ti}},\,\minus\phi\circsmall\pi_{\ti}}$ are recorded in Proposition~\ref{propZetaFnConv_s0} for some expanding Thurston map $f$ and an eventually positive (see Definition~\ref{defEventuallyPositive}) real-valued H\"{o}lder continuous function $\phi$ on $S^2$.

In Section~\ref{sctAssumptions}, we state the assumptions on some of the objects in this paper, which we are going to repeatedly refer to later as \emph{the Assumptions}. Note that not all assumptions are assumed in all the statements in this paper. In fact, we have to gradually remove the dependence on some of the assumptions before proving our main results. This makes the paper more technically involved.

The goal of Section~\ref{sctDynOnC} is to establish a relation between the dynamical zeta functions $\zeta_{f,\,\minus\phi}$ and $\zeta_{\sigma_{A_{\ti}},\,\minus\phi\circsmall\pi_{\ti}}$ for an expanding Thurston map $f$ with some forward invariant Jordan curve $\CC \subseteq S^2$ and a H\"{o}lder continuous potential $\phi\: S^2 \rightarrow \R$ through a careful investigation on the dynamics induced by $f$ on the curve $\CC$. Consequently, we reduce Theorem~\ref{thmZetaAnalExt_InvC} on $f$ to Theorem~\ref{thmZetaAnalExt_SFT} on the corresponding one-sided subshift of finite type $\bigl(\Sigma_{A_{\ti}}^+,\sigma_{A_{\ti}} \bigr)$ defined from the tiles induced by $\CC$.

In Subsection~\ref{subsctDynOnC_Construction}, we construct two one-sided subshift of finite type $\bigl(\Sigma_{A_{\e}}^+,\sigma_{A_{\e}} \bigr)$ and $\bigl(\Sigma_{A_{\ee}}^+,\sigma_{A_{\ee}} \bigr)$ induced by $f$ on $\CC$. We study the combinatorics induced by $f$ and $\CC$ in Subsection~\ref{subsctDynOnC_Combinatorics} and establish in Theorem~\ref{thmNoPeriodPtsIdentity} a counting formula (\ref{eqNoPeriodPtsIdentity}) relating the ``multiplicity'' of an arbitrary point $y\in S^2$ in the sets of periodic points for the four related symbolic dynamical systems $\bigl(\Sigma_{A_{\ti}}^+,\sigma_{A_{\ti}} \bigr)$, $\bigl(\Sigma_{A_{\e}}^+,\sigma_{A_{\e}} \bigr)$, $\bigl(\Sigma_{A_{\ee}}^+,\sigma_{A_{\ee}} \bigr)$, and $(\V^0, f|_{\V^0} )$ where the set $\V^0$ coincides with the finite set $\post f$ of postcritical points of $f$ (i.e., union of forward orbits of critical values of $f$).

Before deducing Theorem~\ref{thmZetaAnalExt_InvC} from Theorem~\ref{thmZetaAnalExt_SFT} in Subsection~\ref{subsctDynOnC_Reduction} using the counting formula in Theorem~\ref{thmNoPeriodPtsIdentity}, we need to show in Subsection~\ref{subsctDynOnC_TopPressure} that topological pressures of the dynamical systems induced by $f$ and a real-valued H\"{o}lder continuous functions $\psi$  on $\CC$ are strictly smaller than the topological pressure $P(f,\psi)$ of $f$ and $\psi$ (see Theorem~\ref{thmPressureOnC}). Such a calculation uses a characterization of topological pressure (\ref{eqTopPressureDefPrePeriodPts}) in terms of refining sequences of finite open covers, and relies on a fine quantitative control over the rate of increase of the number of preimages of a point $q\in S^2$ under iterations of $f|_\CC$ that stay close along the orbit. In order to derive such a quantitative control in Proposition~\ref{propEmBound}, we use a key device $E_m(q_n,\,q_{n-1},\,\dots,\,q_1;\,q)$ (see (\ref{eqDefEm})) to keep track of such quantities that were first introduced and crucially used to establish some weak expansion property, namely, the \emph{asymptotic $h$-expansiveness} of expanding Thurston maps with no periodic critical points in \cite{Li15}. The method used here to bound $E_m(q_n,\,q_{n-1},\,\dots,\,q_1;\,q)$ is different from that in \cite{Li15}. In \cite{Li15}, the non-recurrence of critical points is the key. But here we allow periodic critical points, and the bound relies crucially on the topology of the ($1$-dimensional) Jordan curve instead. 

Section~\ref{sctNLI} is devoted to characterizations of a necessary condition, called \emph{non-local integrability condition}, on the potential $\phi\: S^2 \rightarrow \R$ for the Prime Orbit Theorems (Theorem~\ref{thmPrimeOrbitTheorem}). We recall the notion of  \emph{temporal distance} in Definition~\ref{defTemporalDist}, and define the non-local integrability condition on a potential $\phi$ for an expanding Thurston map $f\: S^2 \rightarrow S^2$. The characterizations are summarized in Theorem~\ref{thmNLI}. In particular, a real-valued H\"{o}lder continuous $\phi$ on $S^2$ is non-locally integrable if and only if $\phi$ is co-homologous to a constant in the set of real-valued continuous functions on $S^2$. We recall that a precise counting formula was obtained in \cite[Theorem~1.1]{Li16} in the case of potentials $\phi$ co-homologous to a constant $K\in \R$, i.e., $\phi= K + \tau\circ f - \tau$ for some continuous function $\tau \: S^2 \rightarrow \R$. We show that in this case the Prime Orbit Theorems do not hold in Corollary~\ref{corNecessary}, establishing the non-local integrability condition as a necessary condition for the Prime Orbit Theorems for expanding Thurston maps.

In order to establish Theorem~\ref{thmNLI}, we recall in Subsection~\ref{subsctNLI_Orbifold} the notion of orbifolds introduced in general by W.~P.~Thurston in 1970s in his study of geometry of $3$-manifolds (see \cite[Chapter~13]{Th80}). We follow the setup from \cite{BM17} for expanding Thurston maps. We recall the definitions of ramification functions, orbifolds, universal orbifold covering maps, deck transformations, and inverse branches, before showing in Proposition~\ref{propInvBranchFixPt} that, roughly speaking, each inverse branch on the universal orbifold cover has a unique attracting fixed point (possibly at infinity).

In Subsection~\ref{subsctNLI_Proof}, we first deduce in Lemma~\ref{lmLiftLI} a consequence of the local-integrability condition on $f$ and $\phi$ to a condition on the inverse branches of $f$ and the lifting of $\phi$ to the universal orbifold cover of $f$. Then the proof of Theorem~\ref{thmNLI} is given.

In Section~\ref{sctSplitRuelleOp}, we define appropriate variations of the Ruelle operator on the suitable function spaces in our context and establish some important inequalities that will be used later. More precisely, in Subsection~\ref{subsctSplitRuelleOp_Construction}, for an expanding Thurston map $f$ with a forward invariant Jordan curve $\CC\subseteq S^2$ and a complex-valued H\"{o}lder continuous function $\psi$, we ``split'' the Ruelle operator $\RR_{\psi} \: \CCC(S^2,\C) \rightarrow \CCC(S^2,\C)$ (see (\ref{eqDefRuelleOp})) into pieces $\RR_{\psi,X^0_\c,E}^{(n)} \: \CCC(E,\C) \rightarrow C\bigl(X^0_\c,\C\bigr)$ in Definition~\ref{defSplitRuelle}, for $\c\in\{\b,\w\}$, $n\in\N_0$, and a union $E$ of $n$-tiles in the cell decomposition $\X^n(f,\CC)$ of $S^2$ induced by $f$ and $\CC$. Such construction is crucial to the proof of Proposition~\ref{propTelescoping} where the image of characteristic functions supported on $n$-tiles under $\RR_{\psi,X^0_\c,E}^{(n)}$ are used to relate periodic points and preimage points of $f$. We then define the \emph{split Ruelle operators} $\RRR_\psi$ in Definition~\ref{defSplitRuelleOnProductSpace} on the product space $\CCC\bigl(X^0_\b,\C\bigr) \times \CCC\bigl(X^0_\w, \C\bigr)$ by piecing together $\RR_{\psi,X^0_{\c_1},X^0_{\c_2}}^{(1)}$, $\c_1,\c_2\in\{\b,\w\}$. The operator norm of $\RRR_\psi$ induced by the \emph{normalized H\"{o}lder norm} (see (\ref{eqDefNormalizedHolderNorm})) on the space $\Holder{\alpha}((X,d),\C)$ of complex-valued H\"{o}lder continuous functions is recorded in Definition~\ref{defOpHNormDiscont} with a equivalent characterization given in Lemma~\ref{lmSplitRuelleCoordinateFormula}.

Subsection~\ref{subsctSplitRuelleOp_BasicIneq} is devoted to establishing various inequalities, among them the \emph{basic inequalities} in Lemma~\ref{lmBasicIneq}, that are indispensable in the arguments in Section~\ref{sctDolgopyat}, adapted into our context from the machinery of D.~Dolgopyat.

In Subsection~\ref{subsctSplitRuelleOp_SpectralGap}, we verify the spectral gap for $\RRR_\psi$ that are essential in the proof of the basic inequalities in Lemma~\ref{lmBasicIneq}.

Section~\ref{sctPreDolgopyat} contains arguments to bound the dynamical zeta function $\zeta_{\sigma_{A_{\ti}},\,\minus\phi\circsmall\pi_{\ti}}$ with the bounds of the operator norm of $\RRR_{\minus s\phi}$, for an expanding Thurston map $f$ with some forward invariant Jordan curve $\CC$ and an eventually positive real-valued H\"{o}lder continuous potential $\phi$.

Subsection~\ref{subsctRuelleEstimate} contains the proof of Proposition~\ref{propTelescoping}, which provides a bound of the dynamical zeta function $\zeta_{\sigma_{A_{\ti}},\,\minus\phi\circsmall\pi_{\ti}}$ for the symbolic system $\bigl(\Sigma_{A_{\ti}}^+, \sigma_{A_{\ti}}\bigr)$ asscociated to $f$ in terms of the operator norms of $\RRR_{\minus s\phi}^n$, $n\in\N$ and $s\in\C$ in some vertical strip with $\abs{\Im(s)}$ large enough. The idea of the proof originated from D.~Ruelle \cite{Rue90}.

In Subsection~\ref{subsctOperatorNorm}, we establish in Theorem~\ref{thmOpHolderNorm} an exponential decay bound on the operator norm $\OpHnormD{\alpha}{\RRR_{\minus s\phi}^n}$ of $\RRR_{\minus s\phi}^n$, $n\in\N$, assuming the bound stated in Theorem~\ref{thmL2Shrinking}. Theorem~\ref{thmL2Shrinking} will be proved at the end of Subsection~\ref{subsctCancellation}.

Combining the bounds in Proposition~\ref{propTelescoping} and Theorem~\ref{thmOpHolderNorm}, we give a proof of Theorem~\ref{thmZetaAnalExt_SFT} in Subsection~\ref{subsctProofthmZetaAnalExt_SFT}.

In Subsection~\ref{subsctProofthmLogDerivative}, we deduce Theorem~\ref{thmLogDerivative} from Theorem~\ref{thmZetaAnalExt_InvC} following the ideas from \cite{PoSh98} using basic complex analysis.

In Section~\ref{sctDolgopyat}, we adapt the arguments of D.~Dolgopyat \cite{Dol98} in our metric-topological setting aiming to prove Theorem~\ref{thmL2Shrinking} at the end of this section, consequently establishing Theorem~\ref{thmZetaAnalExt_InvC}, Theorem~\ref{thmZetaAnalExt_SFT}, Theorem~\ref{thmLogDerivative}, and Theorem~\ref{thmPrimeOrbitTheorem}.

In Subsection~\ref{subsctSNI}, we first give a formulation of the \emph{$\alpha$-strong non-integrability condition}, $\alpha\in(0,1]$, for our setting (Definition~\ref{defStrongNonIntegrability}) and then show its independence on the choice of the Jordan curve $\CC$ in Lemma~\ref{lmSNIwoC}.

In Subsection~\ref{subsctDolgopyatOperator}, a consequence of the $\alpha$-strong non-integrability condition that we will use in the remaining  part of this section is formulated in Proposition~\ref{propSNI}. We remark that it is crucial for the arguments in Subsection~\ref{subsctCancellation} to have the same exponent $\alpha\in(0,1]$ in both the lower bound and the upper bound in (\ref{eqSNIBounds}). The definition of the Dolgopyat operator $\MM_{J,s,\phi}$ in our context is given in Definition~\ref{defDolgopyatOperator} after important constants in the construction are carefully chosen (see for example, (\ref{eqDef_m0}) through (\ref{eqDefmb})). The adjustment of such constants is one key difficulty in the intricate mechaniery of D.~Dolgopyat.

In Subsection~\ref{subsctCancellation}, we adapt the cancellation arguments of D.~Dolgopyat to establish the $l^2$-bound in Theorem~\ref{thmL2Shrinking}.

Section~\ref{sctExamples} is devoted to examples of potentials that satisfy the $\alpha$-strong non-integrability condition, and the investigation of the genericity of this condition.

Subsection~\ref{subsctLattes} focuses on the Latt\`{e}s maps (recalled in Definition~\ref{defLattesMap}). We show in Proposition~\ref{propLattes} that for a Latt\`{e}s map $f$ and a continuously differentiable real-valued potential $\phi \: \widehat{\C} \rightarrow \R$ the weaker condition of non-local integrability implies a stronger condition, namely, $1$-strong non-integrability for some visual metric $d$ for $f$. We include the proof of Theorem~\ref{thmLattesPOT} at the end of Subsection~\ref{subsctLattes}.

Similar results for continuously differentiable potentials seem to be too much to expect for general rational expanding Thurston maps, since the chordal metric $\sigma$ (see Remark~\ref{rmChordalVisualQSEquiv} for definition) is never a visual metric for $f$ (see \cite[Lemma~8.12]{BM17}), which prevents us to get the same exponent in both the lower bound and upper bound in (\ref{eqSNIBounds}) in Proposition~\ref{propSNI}. Nevertheless, we can still show that the $\alpha$-strong non-integrability condition is generic in the set $\Holder{\alpha}(S^2,d)$ of real-valued H\"{o}lder continuous functions with an exponent $\alpha$ equipped with the H\"{o}lder norm, i.e., there exists an open dense subset of functions satisfying such condition in $\Holder{\alpha}(S^2,d)$, provided $\alpha\in(0,1)$ or $\Lambda^\alpha < \Lambda_0(f)$ as stated in Theorem~\ref{thmSNIGeneric}. A constructive proof of Theorem~\ref{thmSNIGeneric} is given at the end of Subsection~\ref{subsctGeneric}, relying on Theorem~\ref{thmPerturbToStrongNonIntegrable} that gives a construction of a potential $\phi$ that satisfies the $\alpha$-strong non-integrability condition arbitrarily close to a given real-valued H\"{o}lder continuous potential $\psi \in \Holder{\alpha}(S^2,d)$.

\subsection*{Acknowledgments} 
The first-named author is grateful to the Institute for Computational and Experimental Research in Mathematics (ICERM) at Brown University for the hospitality during his stay from February to May 2016 where he learned about this area of research while participating the Semester Program ``Dimension and Dynamics'' as a research postdoctoral fellow. The authors want to  express their gratitude to Mark~Pollicott for his introductory discussion on dynamical zeta functions and Prime Orbit Theorems in ICERM and many conversations since then, and to Hee~Oh for her beautiful talks on her work and helpful comments, while the first-named author also wants to thank Jianyu~Chen and  Polina~Vytnova for interesting discussions on this area of research. Most of this work is done during the first-named author's stay at the Institute for Mathematical Sciences (IMS) at Stony Brook University as a postdoctoral fellow. He wants to thank IMS and his postdoctoral advisor Mikhail~Yu.~Lyubich for the great support and hospitality.

\section{Notation} \label{sctNotation}
Let $\C$ be the complex plane and $\widehat{\C}$ be the Riemann sphere. For each complex number $z\in \C$, we denote by $\Re(z)$ the real part of $z$, and by $\Im(z)$ the imaginary part of $z$. We denote by $\D$ the open unit disk $\D \coloneqq \{z\in\C \,|\, \abs{z}<1 \}$ on the complex plane $\C$. For each $a\in\R$, we denote by $\H_a$ the open (right) half-plane $\H_a\coloneqq \{ z\in\C \,|\, \Re(z) > a\}$ on $\C$, and by $\overline{\H}_a$ the closed (right) half-plane $\overline{\H}_a \coloneqq \{ z\in\C \,|\, \Re(z) \geq  a\}$.

We follow the convention that $\N \coloneqq \{1,2,3,\dots\}$, $\N_0 \coloneqq \{0\} \cup \N$, and $\widehat{\N} \coloneqq \N\cup \{+\infty\}$, with the order relations $<$, $\leq$, $>$, $\geq$ defined in the obvious way. For $x\in\R$, we define $\lfloor x\rfloor$ as the greatest integer $\leq x$, and $\lceil x \rceil$ the smallest integer $\geq x$. As usual, the symbol $\log$ denotes the logarithm to the base $e$, and $\log_c$ the logarithm to the base $c$ for $c>0$. The symbol $\I$ stands for the imaginary unit in the complex plane $\C$. For each $z\in\C\setminus \{0\}$, we denote by $\Arg(z)$ the principle argument of $z$, i.e., the unique real number in $(-\pi,\pi]$ with the property that $\abs{z} e^{\I \Arg(z)} = z$. The cardinality of a set $A$ is denoted by $\card{A}$. 

Given real-valued functions $u$, $v$, and $w$ on $(0,+\infty)$. We write $u(T) \sim v(T)$ as $T\to +\infty$ if $\lim_{T\to+\infty} \frac{u(T)}{v(T)} = 1$, and write $u(T) = v(T) + \operatorname{O} ( w(T) )$ as $T\to +\infty$ if $\limsup_{T\to+\infty} \Absbig{ \frac{u(T) - v(T) }{w(T)} } < +\infty$.

Let $g\: X\rightarrow Y$ be a map between two sets $X$ and $Y$. We denote the restriction of $g$ to a subset $Z$ of $X$ by $g|_Z$. 

Given a map $f\: X\rightarrow X$ on a set $X$. The inverse map of $f$ is denoted by $f^{-1}$. We write $f^n \coloneqq \underbrace{ f\circ  \cdots \circ f}_{n}$ for the $n$-th iterate of $f$, and $f^{-n} \coloneqq \left( f^n \right)^{-1}$, for $n\in\N$. We set $f^0 \coloneqq \id_X$, where the identity map $\id_X\: X\rightarrow X$ sends each $x\in X$ to $x$ itself. For each $n\in\N$, we denote by
\begin{equation} \label{eqDefSetPeriodicPts}
P_{n,f} \coloneqq \bigl\{ x\in X \,\big|\, f^n(x)=x, f^k(x)\neq x, k\in\{1,2,\dots,n-1\} \bigr\}
\end{equation}
the \defn{set of periodic points of $f$ with periodic $n$}, and by
\begin{equation}  \label{eqDefSetPeriodicOrbits} 
\Orb(n,f) \coloneqq \bigl\{ \bigl\{f^i(x)  \,\big|\, i\in \{0,1,\dots,n-1\} \bigr\} \,\big|\, x\in P_{n,f}  \bigr\}
\end{equation}
the \defn{set of primitive periodic orbits of $f$ with period $n$}. The set of all periodic orbits of $f$ is denoted by
\begin{equation} \label{eqDefSetAllPeriodicOrbits} 
\Orb(f) \coloneqq \bigcup\limits_{n=1}^{+\infty}  \Orb(n,f).
\end{equation}

Given a complex-valued function $\varphi\: X\rightarrow \C$, we write
\begin{equation}    \label{eqDefSnPt}
S_n \varphi (x)  = S_{n}^f \varphi (x)  \coloneqq \sum\limits_{j=0}^{n-1} \varphi(f^j(x)) 
\end{equation}
for $x\in X$ and $n\in\N_0$. The supscript $f$ is often omitted when the map $f$ is clear from the context. Note that when $n=0$, by definition we always have $S_0 \varphi = 0$.

Let $(X,d)$ be a metric space. For subsets $A,B\subseteq X$, we set $d(A,B) \coloneqq \inf \{d(x,y)\,|\, x\in A,\,y\in B\}$, and $d(A,x)=d(x,A) \coloneqq d(A,\{x\})$ for $x\in X$. For each subset $Y\subseteq X$, we denote the diameter of $Y$ by $\diam_d(Y) \coloneqq \sup\{d(x,y)\,|\,x,y\in Y\}$, the interior of $Y$ by $\inter Y$, and the characteristic function of $Y$ by $\mathbbm{1}_Y$, which maps each $x\in Y$ to $1\in\R$ and vanishes otherwise. We use the convention that $\mathbbm{1}=\mathbbm{1}_X$ when the space $X$ is clear from the context. For each $r>0$ and each $x\in X$, we denote the open (resp.\ closed) ball of radius $r$ centered at $x$ by $B_d(x, r)$ (resp.\ $\overline{B_d}(x,r)$). 

We set $\CCC(X)$ (resp.\ $B(X)$) to be the space of continuous (resp.\ bounded Borel) functions from $X$ to $\R$, $\MMM(X)$ the set of finite signed Borel measures, and $\PPP(X)$ the set of Borel probability measures on $X$. We denote by $\CCC(X,\C)$ (resp.\ $B(X,\C)$) the space of continuous (resp.\ bounded Borel) functions from $X$ to $\C$. Obviously $\CCC(X) \subseteq \CCC(X,\C)$ and $B(X)\subseteq B(X,\C)$. We will adopt the convention that unless specifically referring to $\C$, we only consider real-valued functions. If we do not specify otherwise, we equip $\CCC(X)$ and $\CCC(X,\C)$ with the uniform norm $\Norm{\cdot}_{\CCC^0(X)}$. For $g\in\CCC(X)$ we set $\MMM(X,g)$ to be the set of $g$-invariant Borel probability measures on $X$. 

The space of real-valued (resp.\ complex-valued) H\"{o}lder continuous functions with an exponent $\alpha\in (0,1]$ on a compact metric space $(X,d)$ is denoted by $\Holder{\alpha}(X,d)$ (resp.\ $\Holder{\alpha}((X,d),\C)$). For each $\psi\in\Holder{\alpha}((X,d),\C)$, we denote
\begin{equation}   \label{eqDef|.|alpha}
\Hseminorm{\alpha,\,(X,d)}{\psi} \coloneqq \sup \bigg\{\frac{\abs{\psi(x)- \psi(y)}}{ d(x,y)^\alpha} \,\bigg|\, x,y\in X, \,x\neq y \bigg\},
\end{equation}
and for $b\in\R\setminus\{0\}$, the \emph{normalized H\"{o}lder norm} of $\psi$ is defined as
\begin{equation}  \label{eqDefNormalizedHolderNorm}
\NHnorm{\alpha}{b}{\psi}{(X,d)} \coloneqq \frac{1}{\abs{b}} \Hseminorm{\alpha,\,(X,d)}{\psi}  + \Norm{\psi}_{\CCC^0(X)},
\end{equation}
and for each bounded linear operator $L\: \Holder{\alpha}((X,d),\C) \rightarrow \Holder{\alpha}((X,d),\C)$, the \emph{normalized} operator norm of $L$ is
\begin{equation}   \label{eqDefNormalizedOpHNorm}
\NHnorm{\alpha}{b}{L}{(X,d)} \coloneqq \sup\Bigg\{ \frac{\NHnorm{\alpha}{b}{L(v)}{(X,d)}}{\NHnorm{\alpha}{b}{v}{(X,d)}}  \,\Bigg|\,  v\in\Holder{\alpha}((X,d),\C) \text{ with } v\not\equiv 0  \Bigg\},
\end{equation}
while the standard H\"{o}lder norm of $\psi$ and operator norm of $L$ are denoted by $\Hnorm{\alpha}{\psi}{(X,d)}$ and $\Hnorm{\alpha}{L}{(X,d)}$, respectively, as usual, i.e.,
\begin{equation}    \label{eqDefHolderNorm}
\Hnorm{\alpha}{\psi}{(X,d)} \coloneqq \NHnorm{\alpha}{1}{\psi}{(X,d)} = \NHnorm{\alpha}{-1}{\psi}{(X,d)} 
= \Hseminorm{\alpha,\,(X,d)}{\psi}  + \Norm{\psi}_{\CCC^0(X)},
\end{equation}
and respectively,
\begin{equation}    \label{eqDefOpHNorm}
\Hnorm{\alpha}{L}{(X,d)} \coloneqq \NHnorm{\alpha}{1}{L}{(X,d)} = \NHnorm{\alpha}{-1}{L}{(X,d)} .
\end{equation}

For a Lipschitz map $g\: (X,d)\rightarrow (X,d)$ on a metric space $(X,d)$, we denote the Lipschitz constant by
\begin{equation}   \label{eqDefLipConst}
\LIP_d(g) \coloneqq \sup \biggl\{ \frac{d(g(x),g(y))}{d(x,y)} \,\bigg|\, x,y\in X \text{ with } x\neq y\biggr\}.
\end{equation}

\section{Preliminaries}  \label{sctPreliminaries}

\subsection{Thermodynamical formalism}  \label{subsctThermodynFormalism}

We first review some basic concepts from dynamical systems. We refer the readers to \cite[Chapter~3]{PrU10}, \cite[Chapter~9]{Wal82} or \cite[Chapter~20]{KH95} for more detailed studies of these concepts.

Let $(X,d)$ be a compact metric space and $g\:X\rightarrow X$ a continuous map. For $n\in\N$ and $x,y\in X$,
\begin{equation*}
d^n_g(x,y) \coloneqq \operatorname{max}\bigl\{  d \bigl(g^k(x),g^k(y)\bigr)  \big| k\in\{0,1,\dots,n-1\} \!\bigr\}
\end{equation*}
defines a new metric on $X$. A set $F\subseteq X$ is \defn{$(n,\epsilon)$-separated}, for some $n\in\N$ and $\epsilon>0$, if for each pair of distinct points $x,y\in F$, we have $d^n_g(x,y)\geq \epsilon$. For $\epsilon > 0$ and $n\in\N$, let $F_n(\epsilon)$ be a maximal (in the sense of inclusion) $(n,\epsilon)$-separated set in $X$.

For each real-valued continuous function $\phi \in\CCC(X)$, the following limits exist and are equal, and we denote these limits by $P(g,\phi)$ (see for example, \cite[Theorem~3.3.2]{PrU10}):
\begin{align}  \label{defTopPressure}
P(g,\phi)  \coloneqq &  \lim \limits_{\epsilon\to 0} \limsup\limits_{n\to+\infty} \frac{1}{n} \log  \sum\limits_{x\in F_n(\epsilon)} \exp(S_n \phi(x)) \notag \\
                          =        &  \lim \limits_{\epsilon\to 0} \liminf\limits_{n\to+\infty} \frac{1}{n} \log  \sum\limits_{x\in F_n(\epsilon)} \exp(S_n \phi(x)), 
\end{align}
where $S_n \phi (x) = \sum\limits_{j=0}^{n-1} \phi\left(g^j(x)\right)$ is defined in (\ref{eqDefSnPt}). We call $P(g,\phi)$ the \defn{topological pressure} of $g$ with respect to the \emph{potential} $\phi$. The quantity $h_{\operatorname{top}}(g) \coloneqq P(g,0)$ is called the \defn{topological entropy} of $g$. Note that $P(g,\phi)$ is independent of $d$ as long as the topology on $X$ defined by $d$ remains the same (see \cite[Section~3.2]{PrU10}).

A \defn{cover} of $X$ is a collection $\xi=\{A_j \,|\, j\in J\}$ of subsets of $X$ with the property that $\bigcup\xi = X$, where $J$ is an index set. The cover $\xi$ is an \defn{open cover} if $A_j$ is an open set for each $j\in J$. The cover $\xi$ is \defn{finite} if the index set $J$ is a finite set.

Let $\xi=\{A_j\,|\,j\in J\}$ and $\eta=\{B_k\,|\,k\in K\}$ be two covers of $X$, where $J$ and $K$ are the corresponding index sets. We say $\xi$ is a \defn{refinement} of $\eta$ if for each $A_j\in\xi$, there exists $B_k\in\eta$ such that $A_j\subseteq B_k$. The \defn{common refinement} $\xi \vee \eta$ of $\xi$ and $\eta$ defined as
\begin{equation*}
\xi \vee \eta \coloneqq \{A_j\cap B_k \,|\, j\in J,\, k\in K\}
\end{equation*}
is also a cover. Note that if $\xi$ and $\eta$ are both open covers (resp., measurable partitions), then $\xi \vee \eta$ is also an open cover (resp., a measurable partition). Define $g^{-1}(\xi) \coloneqq \{g^{-1}(A_j) \,|\,j\in J\}$, and denote for $n\in\N$,
\begin{equation*}
\xi^n_g \coloneqq \bigvee\limits_{j=0}^{n-1} g^{-j}(\xi) = \xi\vee g^{-1}(\xi)\vee\cdots\vee g^{-(n-1)}(\xi),
\end{equation*}
and let $\xi^\infty_g$ be the smallest $\sigma$-algebra containing $\bigcup\limits_{n=1}^{+\infty}\xi^n_g$.

A \defn{measurable partition} $\xi$ of $X$ is a cover $\xi=\{A_j\,|\,j\in J\}$ of $X$ consisting of countably many mutually disjoint Borel sets $A_j$, $j\in J$, where $J$ is a countable index set. 
%
The \defn{entropy} of a measurable partition $\xi$ is
\begin{equation*}
H_{\mu}(\xi) \coloneqq -\sum\limits_{j\in J} \mu(A_j)\log\left(\mu (A_j)\right),
\end{equation*}
where $0\log 0$ is defined to be 0. One can show (see \cite[Chapter 4]{Wal82}) that if $H_{\mu}(\xi)<+\infty$, then the following limit exists:
\begin{equation*}
h_{\mu}(g,\xi) \coloneqq \lim\limits_{n\to+\infty} \frac{1}{n} H_{\mu}(\xi^n_g) \in[0,+\infty).
\end{equation*}

The \defn{measure-theoretic entropy} of $g$ for $\mu$ is given by
\begin{equation}   \label{eqDefMeasThEntropy}
h_{\mu}(g) \coloneqq \sup\{h_{\mu}(g,\xi)\,|\, \xi   \text{ is a measurable partition of } X  \text{ with } H_{\mu}(\xi)<+\infty\}.   
\end{equation}
For each real-valued continuous function $\phi\in\CCC(X)$, the \defn{measure-theoretic pressure} $P_\mu(g,\phi)$ of $g$ for the measure $\mu$ and the potential $\phi$ is
\begin{equation}  \label{eqDefMeasTheoPressure}
P_\mu(g,\phi) \coloneqq  h_\mu (g) + \int \! \phi \,\mathrm{d}\mu.
\end{equation}

By the Variational Principle (see for example, \cite[Theorem~3.4.1]{PrU10}), we have that for each $\phi\in\CCC(X)$,
\begin{equation}  \label{eqVPPressure}
P(g,\phi)=\sup\{P_\mu(g,\phi)\,|\,\mu\in \MMM(X,g)\}.
\end{equation}
In particular, when $\phi$ is the constant function $0$,
\begin{equation}  \label{eqVPEntropy}
h_{\operatorname{top}}(g)=\sup\{h_{\mu}(g)\,|\,\mu\in \MMM(X,g)\}.
\end{equation}
A measure $\mu$ that attains the supremum in (\ref{eqVPPressure}) is called an \defn{equilibrium state} for the map $g$ and the potential $\phi$. A measure $\mu$ that attains the supremum in (\ref{eqVPEntropy}) is called a \defn{measure of maximal entropy} of $g$.

Given a continuous map $g\: X\rightarrow X$ on a compact metric space $(X,d)$ and a real-valued continuous potential $\varphi \in \CCC(X)$. By \cite[Theorem~3.3.2]{PrU10}, our definition of the topological pressure $P(g,\varphi)$ in (\ref{defTopPressure}) coincides with the definition presented in \cite[Section~3.2]{PrU10}. More precisely, combining (3.2.3), Lemma~3.2.1, Definition~3.2.3, and Lemma~3.2.4 from \cite{PrU10}, the topological pressure $P(g,\varphi)$ of $g$ with respect to $\varphi$ is also given by
\begin{equation}   \label{eqEquivDefByCoverForTopPressure}
 P(g,\varphi) 
            =  \lim\limits_{m\to +\infty}  \lim\limits_{n\to +\infty}  \frac{1}{n} \log 
                \inf\Biggl\{ \sum\limits_{V\in\mathcal{V}}  \exp \Bigl( \sup\limits_{x\in V}   S_n\varphi(x)   \Bigr)
                             \,\Bigg|\,  \mathcal{V} \subseteq \bigvee\limits_{i=0}^{n} g^{-i}(\xi_m),\,\bigcup \mathcal{V} = X  \Biggr\}    ,
\end{equation}
where $\{\xi_m\}_{m\in\N_0}$ is an arbitrary sequence of finite open covers of $X$ with $\lim\limits_{m\to+\infty} \max \{ \diam_d(U) \,|\, U\in \xi_m \} = 0$.

\smallskip

One of the main tools in the study of the existence, uniqueness, and other properties of equilibrium states is the \emph{Ruelle operator}. We will postpone the discussion of the Ruelle operators of expanding Thurston maps to Subsection~\ref{subsctThurstonMap}.

\subsection{Branched covering maps bewteen surfaces} \label{subsctBranchedCoveringMaps}
This paper is devoted to the discussion of expanding Thurston maps, which are branched covering maps on $S^2$ with certain expansion properties. We will discuss such branched covering maps in detail in Subsection~\ref{subsctThurstonMap}. However, since we are going to use lifting properties of branched covering maps and universal orbifold covers in Section~\ref{sctNLI}, we need to discuss briefly branched covering maps between surfaces in general here. For more detailed discussions on the concepts and results in this subsection, see \cite[Appendices~A.6]{BM17} and references therein. For a study of branched covering maps between more general topological spaces, see P.~Ha\"{\i}ssinsky and K.~M.~Pilgrim \cite{HP09}. 

This subsection is only used in Section~\ref{sctNLI}. Relevant concepts and results adapted to the special case of branched covering maps on $S^2$ will be reiterated in Subsection~\ref{subsctThurstonMap}. The readers may safely skip this subsection in the first reading.

In this paper, a \defn{surface} is a connected and oriented $2$-dimensional topological manifold.

\begin{definition}[Branched covering maps between surfaces]   \label{defBranchedCover}
Let $X$ and $Y$ be (connected and oriented) surfaces, and $f\: X\rightarrow Y$ be a continuous map. Then $f$ is a \defn{branched covering map} (between $X$ and $Y$) if for each point $q\in Y$ there exists an open set $V \subseteq Y$ with $q\in V$ and there exists a collection $\{U_i\}_{i\in I}$ of open sets $U_i\subseteq X$ for some index set $I\neq \emptyset$ such that the following conditions are satisfied:
\begin{enumerate}
\smallskip
\item[(i)] $f^{-1}(V)$ is a disjoint union $f^{-1}(V) = \bigcup_{i\in I} U_i$,

\smallskip
\item[(ii)] $U_i$ contains precisely one point $p_i \in f^{-1}(q)$ for each $i\in I$, and

\smallskip
\item[(iii)] for each $i\in I$, there exists $d_i \in\N$, and orientation-preserving homeomorphisms $\varphi_i \: U_i\rightarrow\D$ and $\psi_i\: V\rightarrow\D$ with $\varphi_i(p_i) = 0$ and $\psi_i(q) = 0$ such that 
\begin{equation}   \label{eqBranchCoverMapLocalPowerMap}
\bigl( \psi_i \circ f \circ \varphi_i^{-1} \bigr) (z) = z^{d_i}
\end{equation}
for all $z\in\D$.
\end{enumerate}

The positive integer $d_i$ is called the \defn{local degree} of $f$ at $p \coloneqq p_i$, denoted by $\deg_f(p)$.
\end{definition}

\begin{rem}
We say that the set $V\subseteq Y$ is \defn{evenly covered} by $f$ if conditions~(i) and (ii) above are both satisfied.
\end{rem}

Note that in Definition~\ref{defBranchedCover}, we do not require $X$ and $Y$ to be compact. In fact we will need to use the incompact case in Section~\ref{sctNLI}.
 
Note that the local degree $\deg_f(p_i)=d_i$ in Definition~\ref{defBranchedCover} is uniquely determined by $p\coloneqq p_i$. If $q'\in V$ is a point close to, but distinct from, $q=f(p)$, then $\deg_f(p)$ is equal to the number of distinct preimages of $q'$ under $f$ close to $p$. In particular, near $p$ (but not at $p$) the map $f$ is $d$-to-$1$, where $d=\deg_f(p)$.

Every branched covering map $f\: X\rightarrow Y$ is surjective, \defn{open} (i.e., images of open sets are open), and \defn{discrete} (i.e., the preimage set $f^{-1}(q)$ of every point $q\in Y$ has no limit points in $X$). Every covering map is also a branched covering map.

A \defn{critical point} of a branched covering map $f \: X\rightarrow Y$ is a point $p\in X$ with $\deg_f(p) \geq 2$. We denote the set of critical points of $f$ by $\crit f$. A \defn{critical value} is a point $q\in Y$ such that $f^{-1}(q)$ contains a critical point of $f$. The set of critical points of $f$ is \defn{discrete} in $X$ (i.e., it has no limit points in $X$), and the set of critical values of $f$ is discrete in $Y$. The map $f$ is an orientation-preserving local homeomorphism near each point $p\in X \setminus \crit f$.

Branched covering maps between surfaces behave well under compositions. We record the facts from Lemma~A.16 and Lemma~A.17 in \cite{BM17} in the following lemma.

\begin{lemma}[Compositions of branched covering maps]   \label{lmBranchedCoverCompositionBM}
Let $X$, $Y$, and $Z$ be (connected and oriented) surfaces, and $f\: X\rightarrow Z$, $g\: Y\rightarrow Z$, and $h\: X\rightarrow Y$ be continuous maps such that $f=g\circ h$.
\begin{enumerate}
\smallskip
\item[(i)] If $g$ and $h$ are branched covering maps, and $Y$ and $Z$ are compact, then $f$ is also a branched covering map, and moreover, for each $x\in X$, we have
\begin{equation*}
\deg_f(x) = \deg_g(h(x)) \cdot \deg_h(x).
\end{equation*}

\smallskip
\item[(ii)] If $f$ and $g$ are branched covering maps, then $h$ is a branched covering map. Similarly, if $f$ and $h$ are branched covering maps, then $g$ is a branched covering map.
\end{enumerate}

If, in addition, $X$, $Y$, and $Z$ are Riemann surfaces and the two branched covering maps in the hypotheses of \textup{(i)} or \textup{(ii)} are holomorphic, then the third map is also holomorphic.
\end{lemma}

Let $\pi\: X \rightarrow Y$ be a branched covering map, $Z$ a topological space, and $f\: Z\rightarrow Y$ be a continuous map. A continuous map $g\: Z\rightarrow X$ is called a \defn{lift} of $f$ (by $\pi$) if $\pi\circ g = f$, i.e., the following diagram commutes:
\begin{equation*}
\xymatrix{                                       & X \ar[d]^{\pi} \\ 
            Z   \ar[ur]^{g}    \ar[r]_{f}        & Y.}
\end{equation*}

We record Lemma~A.18 in \cite{BM17} below.

\begin{lemma}[Lifting paths by branched covering maps]  \label{lmLiftPathBM}
Let $X$ and $Y$ be (connected and oriented) surfaces, $\pi\: X\rightarrow Y$ be a branched covering map, $\gamma \: [0,1]\rightarrow Y$ be a path in $Y$, and $x_0\in \pi^{-1}(\gamma(0))$. Then there exists a path $\lambda\: [0,1]\rightarrow X$ with $\lambda(0)=x_0$ and $\pi\circ \lambda = \gamma$.
\end{lemma}

Branched covering maps are closely related to covering maps. The following lemma recorded from \cite[Lemma~A11]{BM17} makes such a connection explicit.

\begin{lemma}   \label{lmBranchCoverToCover}
Let $X$ and $Y$ be (connected and oriented) surfaces, and $f\: X\rightarrow Y$ be a branched covering map. Suppose $P\subseteq Y$ is a set with $f(\crit f) \subseteq P$ that is discrete in $Y$. Then $f\: X\setminus f^{-1}(P) \rightarrow Y \setminus P$ is a covering map.
\end{lemma}

We will use the lifting properties of covering maps in Section~\ref{sctNLI}. The proof and the terminology of the following lemma formulated as Lemma~A.6 in \cite{BM17} can be found in \cite[Section~1.3, Proposition~1.34, and Proposition~1.33]{Ha02} (see also \cite[Section~1.4 and Theorem~4.17]{Fo81}).

\begin{lemma}[Lifting by covering maps]   \label{lmLiftCoveringMap}
Let $X$ and $Y$ be (connected and oriented) surfaces, $\pi\: X\rightarrow Y$ be a covering map, and $Z$ be a path-connected and locally path-connected topological space.
\begin{enumerate}
\smallskip
\item[(i)] Suppose $g_1, g_2\: Z\rightarrow X$ are two continuous maps such that $\pi\circ g_1 = \pi\circ g_2$. If there exists $z_0\in Z$ with $g_1(z_0) = g_2(z_0)$, then $g_1=g_2$.

\smallskip
\item[(ii)] Suppose $Z$ is simply connected, $f\: Z\rightarrow Y$ is a continuous map, and $z_0\in Z$ and $x_0\in X$ are points such that $f(z_0)=\pi(x_0)$. Then there exists a continuous map $g\: Z\rightarrow X$ such that $g(z_0) = x_0$ and $f = \pi \circ g$.
\end{enumerate}
\end{lemma}

\subsection{Thurston maps} \label{subsctThurstonMap}
In this subsection, we go over some key concepts and results on Thurston maps, and expanding Thurston maps in particular. For a more thorough treatment of the subject, we refer to \cite{BM17}.

Let $S^2$ denote an oriented topological $2$-sphere. A continuous map $f\:S^2\rightarrow S^2$ is called a \defn{branched covering map} on $S^2$ if for each point $x\in S^2$, there exists a positive integer $d\in \N$, open neighborhoods $U$ of $x$ and $V$ of $y=f(x)$, open neighborhoods $U'$ and $V'$ of $0$ in $\widehat{\C}$, and orientation-preserving homeomorphisms $\varphi\:U\rightarrow U'$ and $\eta\:V\rightarrow V'$ such that $\varphi(x)=0$, $\eta(y)=0$, and
\begin{equation*}
(\eta\circ f\circ\varphi^{-1})(z)=z^d
\end{equation*}
for each $z\in U'$. The positive integer $d$ above is called the \defn{local degree} of $f$ at $x$ and is denoted by $\deg_f (x)$. 

Note that the definition of branched covering maps on $S^2$ mentioned above is compatible with Definition~\ref{defBranchedCover}, see the discussion succeeding Lemma~A.10 in \cite{BM17} for more details.

The \defn{degree} of $f$ is
\begin{equation}   \label{eqDeg=SumLocalDegree}
\deg f=\sum\limits_{x\in f^{-1}(y)} \deg_f (x)
\end{equation}
for $y\in S^2$ and is independent of $y$. If $f\:S^2\rightarrow S^2$ and $g\:S^2\rightarrow S^2$ are two branched covering maps on $S^2$, then so is $f\circ g$, and
\begin{equation} \label{eqLocalDegreeProduct}
 \deg_{f\circsmall g}(x) = \deg_g(x)\deg_f(g(x)), \qquad \text{for each } x\in S^2,
\end{equation}   
and moreover, 
\begin{equation}  \label{eqDegreeProduct}
\deg(f\circ g) =  (\deg f)( \deg g).
\end{equation}

A point $x\in S^2$ is a \defn{critical point} of $f$ if $\deg_f(x) \geq 2$. The set of critical points of $f$ is denoted by $\crit f$. A point $y\in S^2$ is a \defn{postcritical point} of $f$ if $y = f^n(x)$ for some $x\in\crit f$ and $n\in\N$. The set of postcritical points of $f$ is denoted by $\post f$. Note that $\post f=\post f^n$ for all $n\in\N$.

\begin{definition} [Thurston maps] \label{defThurstonMap}
A Thurston map is a branched covering map $f\:S^2\rightarrow S^2$ on $S^2$ with $\deg f\geq 2$ and $\card(\post f)<+\infty$.
\end{definition}

We now recall the notation for cell decompositions of $S^2$ used in \cite{BM17} and \cite{Li17}. A \defn{cell of dimension $n$} in $S^2$, $n \in \{1,2\}$, is a subset $c\subseteq S^2$ that is homeomorphic to the closed unit ball $\overline{\B^n}$ in $\R^n$. We define the \defn{boundary of $c$}, denoted by $\partial c$, to be the set of points corresponding to $\partial\B^n$ under such a homeomorphism between $c$ and $\overline{\B^n}$. The \defn{interior of $c$} is defined to be $\inte (c) = c \setminus \partial c$. For each point $x\in S^2$, the set $\{x\}$ is considered a \defn{cell of dimension $0$} in $S^2$. For a cell $c$ of dimension $0$, we adopt the convention that $\partial c=\emptyset$ and $\inte (c) =c$. 

We record the following three definitions from \cite{BM17}.

\begin{definition}[Cell decompositions]\label{defcelldecomp}
Let $\DD$ be a collection of cells in $S^2$.  We say that $\DD$ is a \defn{cell decomposition of $S^2$} if the following conditions are satisfied:

\begin{itemize}

\smallskip
\item[(i)]
the union of all cells in $\DD$ is equal to $S^2$,

\smallskip
\item[(ii)] if $c\in \DD$, then $\partial c$ is a union of cells in $\DD$,

\smallskip
\item[(iii)] for $c_1,c_2 \in \DD$ with $c_1 \neq c_2$, we have $\inte (c_1) \cap \inte (c_2)= \emptyset$,  

\smallskip
\item[(iv)] every point in $S^2$ has a neighborhood that meets only finitely many cells in $\DD$.

\end{itemize}
\end{definition}

\begin{definition}[Refinements]\label{defrefine}
Let $\DD'$ and $\DD$ be two cell decompositions of $S^2$. We
say that $\DD'$ is a \defn{refinement} of $\DD$ if the following conditions are satisfied:
\begin{itemize}

\smallskip
\item[(i)] every cell $c\in \DD$ is the union of all cells $c'\in \DD'$ with $c'\subseteq c$.

\smallskip
\item[(ii)] for every cell $c'\in \DD'$ there exits a cell $c\in \DD$ with $c'\subseteq c$,

\end{itemize}
\end{definition}

\begin{definition}[Cellular maps and cellular Markov partitions]\label{defcellular}
Let $\DD'$ and $\DD$ be two cell decompositions of  $S^2$. We say that a continuous map $f \: S^2 \rightarrow S^2$ is \defn{cellular} for  $(\DD', \DD)$ if for every cell $c\in \DD'$, the restriction $f|_c$ of $f$ to $c$ is a homeomorphism of $c$ onto a cell in $\DD$. We say that $(\DD',\DD)$ is a \defn{cellular Markov partition} for $f$ if $f$ is cellular for $(\DD',\DD)$ and $\DD'$ is a refinement of $\DD$.
\end{definition}

Let $f\:S^2 \rightarrow S^2$ be a Thurston map, and $\CC\subseteq S^2$ be a Jordan curve containing $\post f$. Then the pair $f$ and $\CC$ induces natural cell decompositions $\DD^n(f,\CC)$ of $S^2$, for $n\in\N_0$, in the following way:

By the Jordan curve theorem, the set $S^2\setminus\CC$ has two connected components. We call the closure of one of them the \defn{white $0$-tile} for $(f,\CC)$, denoted by $X^0_\w$, and the closure of the other the \defn{black $0$-tile} for $(f,\CC)$, denoted by $X^0_\b$. The set of \defn{$0$-tiles} is $\X^0(f,\CC) \coloneqq \bigl\{ X_\b^0,X_\w^0 \bigr\}$. The set of \defn{$0$-vertices} is $\V^0(f,\CC) \coloneqq \post f$. We set $\overline\V^0(f,\CC) \coloneqq \{ \{x\} \,|\, x\in \V^0(f,\CC) \}$. The set of \defn{$0$-edges} $\E^0(f,\CC)$ is the set of the closures of the connected components of $\CC \setminus  \post f$. Then we get a cell decomposition 
\begin{equation*}
\DD^0(f,\CC) \coloneqq \X^0(f,\CC) \cup \E^0(f,\CC) \cup \overline\V^0(f,\CC)
\end{equation*}
of $S^2$ consisting of \emph{cells of level $0$}, or \defn{$0$-cells}.

We can recursively define unique cell decompositions $\DD^n(f,\CC)$, $n\in\N$, consisting of \defn{$n$-cells} such that $f$ is cellular for $(\DD^{n+1}(f,\CC),\DD^n(f,\CC))$. We refer to \cite[Lemma~5.12]{BM17} for more details. We denote by $\X^n(f,\CC)$ the set of $n$-cells of dimension 2, called \defn{$n$-tiles}; by $\E^n(f,\CC)$ the set of $n$-cells of dimension 1, called \defn{$n$-edges}; by $\overline\V^n(f,\CC)$ the set of $n$-cells of dimension 0; and by $\V^n(f,\CC)$ the set $\big\{x\,\big|\, \{x\}\in \overline\V^n(f,\CC)\big\}$, called the set of \defn{$n$-vertices}. The \defn{$k$-skeleton}, for $k\in\{0,1,2\}$, of $\DD^n(f,\CC)$ is the union of all $n$-cells of dimension $k$ in this cell decomposition. 

We record Proposition~5.16 of \cite{BM17} here in order to summarize properties of the cell decompositions $\DD^n(f,\CC)$ defined above.

\begin{prop}[M.~Bonk \& D.~Meyer \cite{BM17}] \label{propCellDecomp}
Let $k,n\in \N_0$, let   $f\: S^2\rightarrow S^2$ be a Thurston map,  $\CC\subseteq S^2$ be a Jordan curve with $\post f \subseteq \CC$, and   $m=\card(\post f)$. 
 
\smallskip
\begin{itemize}

\smallskip
\item[(i)] The map  $f^k$ is cellular for $\bigl( \DD^{n+k}(f,\CC), \DD^n(f,\CC) \bigr)$. In particular, if  $c$ is any $(n+k)$-cell, then $f^k(c)$ is an $n$-cell, and $f^k|_c$ is a homeomorphism of $c$ onto $f^k(c)$.

\smallskip
\item[(ii)]  Let  $c$ be  an $n$-cell.  Then $f^{-k}(c)$ is equal to the union of all 
$(n+k)$-cells $c'$ with $f^k(c')=c$.

\smallskip
\item[(iii)] The $1$-skeleton of $\DD^n(f,\CC)$ is  equal to  $f^{-n}(\CC)$. The $0$-skeleton of $\DD^n(f,\CC)$ is the set $\V^n(f,\CC)=f^{-n}(\post f )$, and we have $\V^n(f,\CC) \subseteq \V^{n+k}(f,\CC)$. 

\smallskip
\item[(iv)] $\card(\X^n(f,\CC))=2(\deg f)^n$,  $\card(\E^n(f,\CC))=m(\deg f)^n$,  and $\card (\V^n(f,\CC)) \leq m (\deg f)^n$.

\smallskip
\item[(v)] The $n$-edges are precisely the closures of the connected components of $f^{-n}(\CC)\setminus f^{-n}(\post f )$. The $n$-tiles are precisely the closures of the connected components of $S^2\setminus f^{-n}(\CC)$.

\smallskip
\item[(vi)] Every $n$-tile  is an $m$-gon, i.e., the number of $n$-edges and the number of $n$-vertices contained in its boundary are equal to $m$.  

\smallskip
\item[(vii)] Let $F\coloneqq f^k$ be an iterate of $f$ with $k \in \N$. Then $\DD^n(F,\CC) = \DD^{nk}(f,\CC)$.
\end{itemize}
\end{prop}

We also note that for each $n$-edge $e\in\E^n(f,\CC)$, $n\in\N_0$, there exist exactly two $n$-tiles $X,X'\in\X^n(f,\CC)$ such that $X\cap X' = e$.

For $n\in \N_0$, we define the \defn{set of black $n$-tiles} as
\begin{equation*}
\X_\b^n(f,\CC) \coloneqq \bigl\{X\in\X^n(f,\CC) \, \big|\,  f^n(X)=X_\b^0 \bigr\},
\end{equation*}
and the \defn{set of white $n$-tiles} as
\begin{equation*}
\X_\w^n(f,\CC) \coloneqq \bigl\{X\in\X^n(f,\CC) \, \big|\, f^n(X)=X_\w^0 \bigr\}.
\end{equation*}
It follows immediately from Proposition~\ref{propCellDecomp} that
\begin{equation}   \label{eqCardBlackNTiles}
\card ( \X_\b^n(f,\CC) ) = \card  (\X_\w^n(f,\CC) ) = (\deg f)^n 
\end{equation}
for each $n\in\N_0$.

From now on, if the map $f$ and the Jordan curve $\CC$ are clear from the context, we will sometimes omit $(f,\CC)$ in the notation above.

If we fix the cell decomposition $\DD^n(f,\CC)$, $n\in\N_0$, we can define for each $v\in \V^n$ the \defn{$n$-flower of $v$} as
\begin{equation}   \label{defFlower}
W^n(v) \coloneqq \bigcup  \{\inte (c) \,|\, c\in \DD^n,\, v\in c \}.
\end{equation}
Note that flowers are open (in the standard topology on $S^2$). Let $\overline{W}^n(v)$ be the closure of $W^n(v)$. We define the \defn{set of all $n$-flowers} by
\begin{equation}   \label{defSetNFlower}
\W^n \coloneqq \{W^n(v) \,|\, v\in\V^n\}.
\end{equation}
\begin{rem}  \label{rmFlower}
For $n\in\N_0$ and $v\in\V^n$, we have 
\begin{equation*}
\overline{W}^n(v)=X_1\cup X_2\cup \cdots \cup X_m,
\end{equation*}
where $m \coloneqq 2\deg_{f^n}(v)$, and $X_1, X_2, \dots X_m$ are all the $n$-tiles that contain $v$ as a vertex (see \cite[Lemma~5.28]{BM17}). Moreover, each flower is mapped under $f$ to another flower in such a way that is similar to the map $z\mapsto z^k$ on the complex plane. More precisely, for $n\in\N_0$ and $v\in \V^{n+1}$, there exist orientation preserving homeomorphisms $\varphi\: W^{n+1}(v) \rightarrow \D$ and $\eta\: W^{n}(f(v)) \rightarrow \D$ such that $\D$ is the unit disk on $\C$, $\varphi(v)=0$, $\eta(f(v))=0$, and 
\begin{equation*}
(\eta\circ f \circ \varphi^{-1}) (z) = z^k
\end{equation*}
for all $z\in D$, where $k \coloneqq \deg_f(v)$. Let $\overline{W}^{n+1}(v)= X_1\cup X_2\cup \cdots \cup X_m$ and $\overline{W}^n(f(v))= X'_1\cup X'_2\cup \cdots \cup X'_{m'}$, where $X_1, X_2, \dots X_m$ are all the $(n+1)$-tiles that contain $v$ as a vertex, listed counterclockwise, and $X'_1, X'_2, \dots X'_{m'}$ are all the $n$-tiles that contain $f(v)$ as a vertex, listed counterclockwise, and $f(X_1)=X'_1$. Then $m= m'k$, and $f(X_i)=X'_j$ if $i\equiv j \pmod{k}$, where $k=\deg_f(v)$. (See also Case~3 of the proof of Lemma~5.24 in \cite{BM17} for more details.) If particular, both $W^n(v)$ and $\overline{W}^n(v)$ are simply connected.
\end{rem}

We denote, for each $x\in S^2$ and $n\in\Z$,
\begin{equation}  \label{defU^n}
U^n(x) \coloneqq \bigcup \{Y^n\in \X^n \,|\,    \text{there exists } X^n\in\X^n  
                                        \text{ with } x\in X^n, \, X^n\cap Y^n \neq \emptyset  \}  
\end{equation}
if $n\geq 0$, and set $U^n(x) \coloneqq S^2$ otherwise. 

We can now give a definition of expanding Thurston maps.

\begin{definition} [Expansion] \label{defExpanding}
A Thurston map $f\:S^2\rightarrow S^2$ is called \defn{expanding} if there exists a metric $d$ on $S^2$ that induces the standard topology on $S^2$ and a Jordan curve $\CC\subseteq S^2$ containing $\post f$ such that 
\begin{equation*}
\lim\limits_{n\to+\infty}\max \{\diam_d(X) \,|\, X\in \X^n(f,\CC)\}=0.
\end{equation*}
\end{definition}

By \cite[Lemma~6.1]{BM17}, $\card(\post f) \geq 3$ for each expanding Thurston map $f$.

\begin{rems}  \label{rmExpanding}
It is clear from Proposition~\ref{propCellDecomp}~(vii) and Definition~\ref{defExpanding} that if $f$ is an expanding Thurston map, so is $f^n$ for each $n\in\N$. We observe that being expanding is a topological property of a Thurston map and independent of the choice of the metric $d$ that generates the standard topology on $S^2$. By Lemma~6.2 in \cite{BM17}, it is also independent of the choice of the Jordan curve $\CC$ containing $\post f$. More precisely, if $f$ is an expanding Thurston map, then
\begin{equation*}
\lim\limits_{n\to+\infty}\max \!\big\{ \! \diam_{\wt{d}}(X) \,\big|\, X\in \X^n\bigl(f,\wt\CC \hspace{0.5mm}\bigr)\hspace{-0.3mm} \big\}\hspace{-0.3mm}=0,
\end{equation*}
for each metric $\wt{d}$ that generates the standard topology on $S^2$ and each Jordan curve $\wt\CC\subseteq S^2$ that contains $\post f$.
\end{rems}

P.~Ha\"{\i}ssinsky and K.~M.~Pilgrim developed a notion of expansion in a more general context for finite branched coverings between topological spaces (see \cite[Section~2.1 and Section~2.2]{HP09}). This applies to Thurston maps and their notion of expansion is equivalent to our notion defined above in the context of Thurston maps (see \cite[Proposition~6.4]{BM17}). Such concepts of expansion are natural analogues, in the contexts of finite branched coverings and Thurston maps, to some of the more classical versions, such as expansive homeomorphisms and forward-expansive continuous maps between compact metric spaces (see for example, \cite[Definition~3.2.11]{KH95}), and distance-expanding maps between compact metric spaces (see for example, \cite[Chapter~4]{PrU10}). Our notion of expansion is not equivalent to any such classical notion in the context of Thurston maps. One topological obstruction comes from the presence of critical points for (non-homeomorphic) branched covering maps on $S^2$. In fact, as mentioned in the introduction, there are subtle connections between our notion of expansion and some classical notions of weak expansion. More precisely, one can prove that an expanding Thurston map is asymptotically $h$-expansive if and only if it has no periodic points. Moreover, such a map is never $h$-expansive. See \cite{Li15} for details.

For an expanding Thurston map $f$, we can fix a particular metric $d$ on $S^2$ called a \emph{visual metric for $f$} . For the existence and properties of such metrics, see \cite[Chapter~8]{BM17}. For a visual metric $d$ for $f$, there exists a unique constant $\Lambda > 1$ called the \emph{expansion factor} of $d$ (see \cite[Chapter~8]{BM17} for more details). One major advantage of a visual metric $d$ is that in $(S^2,d)$ we have good quantitative control over the sizes of the cells in the cell decompositions discussed above. We summarize several results of this type (\cite[Proposition~8.4, Lemma~8.10, Lemma~8.11]{BM17}) in the lemma below.

\begin{lemma}[M.~Bonk \& D.~Meyer \cite{BM17}]   \label{lmCellBoundsBM}
Let $f\:S^2 \rightarrow S^2$ be an expanding Thurston map, and $\CC \subseteq S^2$ be a Jordan curve containing $\post f$. Let $d$ be a visual metric on $S^2$ for $f$ with expansion factor $\Lambda>1$. Then there exist constants $C\geq 1$, $C'\geq 1$, $K\geq 1$, and $n_0\in\N_0$ with the following properties:
\begin{enumerate}
\smallskip
\item[(i)] $d(\sigma,\tau) \geq C^{-1} \Lambda^{-n}$ whenever $\sigma$ and $\tau$ are disjoint $n$-cells for $n\in \N_0$.

\smallskip
\item[(ii)] $C^{-1} \Lambda^{-n} \leq \diam_d(\tau) \leq C\Lambda^{-n}$ for all $n$-edges and all $n$-tiles $\tau$ for $n\in\N_0$.

\smallskip
\item[(iii)] $B_d(x,K^{-1} \Lambda^{-n} ) \subseteq U^n(x) \subseteq B_d(x, K\Lambda^{-n})$ for $x\in S^2$ and $n\in\N_0$.

\smallskip
\item[(iv)] $U^{n+n_0} (x)\subseteq B_d(x,r) \subseteq U^{n-n_0}(x)$ where $n= \lceil -\log r / \log \Lambda \rceil$ for $r>0$ and $x\in S^2$.

\smallskip
\item[(v)] For every $n$-tile $X^n\in\X^n(f,\CC)$, $n\in\N_0$, there exists a point $p\in X^n$ such that $B_d(p,C^{-1}\Lambda^{-n}) \subseteq X^n \subseteq B_d(p,C\Lambda^{-n})$.
\end{enumerate}

Conversely, if $\wt{d}$ is a metric on $S^2$ satisfying conditions \textnormal{(i)} and \textnormal{(ii)} for some constant $C\geq 1$, then $\wt{d}$ is a visual metric with expansion factor $\Lambda>1$.
\end{lemma}

Recall that $U^n(x)$ is defined in (\ref{defU^n}).

In addition, we will need the fact that a visual metric $d$ induces the standard topology on $S^2$ (\cite[Proposition~8.3]{BM17}) and the fact that the metric space $(S^2,d)$ is linearly locally connected (\cite[Proposition~18.5]{BM17}). A metric space $(X,d)$ is \defn{linearly locally connected} if there exists a constant $L\geq 1$ such that the following conditions are satisfied:
\begin{enumerate}
\smallskip

\item  For all $z\in X$, $r > 0$, and $x,y\in B_d(z,r)$ with $x\neq y$, there exists a continuum $E\subseteq X$ with $x,y\in E$ and $E\subseteq B_d(z,rL)$.

\smallskip

\item For all $z\in X$, $r > 0$, and $x,y\in X \setminus B_d(z,r)$ with $x\neq y$, there exists a continuum $E\subseteq X$ with $x,y\in E$ and $E\subseteq X \setminus B_d(z,r/L)$.
\end{enumerate}
We call such a constant $L \geq 1$ a \defn{linear local connectivity constant of $d$}.

\begin{rem}   \label{rmChordalVisualQSEquiv}
If $f\: \widehat\C \rightarrow \widehat\C$ is a rational expanding Thurston map, then a visual metric is quasisymmetrically equivalent to the chordal metric on the Riemann sphere $\widehat\C$ (see Theorem~\ref{thmBM}). Here the chordal metric $\sigma$ on $\widehat\C$ is given by
$
\sigma(z,w) =\frac{2\abs{z-w}}{\sqrt{1+\abs{z}^2} \sqrt{1+\abs{w}^2}}
$
for $z,w\in\C$, and $\sigma(\infty,z)=\sigma(z,\infty)= \frac{2}{\sqrt{1+\abs{z}^2}}$ for $z\in \C$. We also note that quasisymmetric embeddings of bounded connected metric spaces are H\"{o}lder continuous (see \cite[Section~11.1 and Corollary~11.5]{He01}). Accordingly, the classes of H\"{o}lder continuous functions on $\widehat\C$ equipped with the chordal metric and on $S^2=\widehat\C$ equipped with any visual metric for $f$ are the same (upto a change of the H\"{o}lder exponent).
\end{rem}

A Jordan curve $\CC\subseteq S^2$ is \defn{$f$-invariant} if $f(\CC)\subseteq \CC$. We are interested in $f$-invariant Jordan curves that contain $\post f$, since for such a Jordan curve $\CC$, we get a cellular Markov partition $(\DD^1(f,\CC),\DD^0(f,\CC))$ for $f$. According to Example~15.11 in \cite{BM17}, such $f$-invariant Jordan curves containing $\post{f}$ need not exist. However, M.~Bonk and D.~Meyer \cite[Theorem~15.1]{BM17} proved that there exists an $f^n$-invariant Jordan curve $\CC$ containing $\post{f}$ for each sufficiently large $n$ depending on $f$. A slightly stronger version of this result was proved in \cite[Lemma~3.11]{Li16} and we record it below.

\begin{lemma}[M.~Bonk \& D.~Meyer \cite{BM17}, Z.~Li \cite{Li16}]  \label{lmCexistsL}
Let $f\:S^2\rightarrow S^2$ be an expanding Thurston map, and $\wt{\CC}\subseteq S^2$ be a Jordan curve with $\post f\subseteq \wt{\CC}$. Then there exists an integer $N(f,\wt{\CC}) \in \N$ such that for each $n\geq N(f,\wt{\CC})$ there exists an $f^n$-invariant Jordan curve $\CC$ isotopic to $\wt{\CC}$ rel.\ $\post f$ such that no $n$-tile in $\X^n(f,\CC)$ joins opposite sides of $\CC$.
\end{lemma}

The phrase ``joining opposite sides'' has a specific meaning in our context. 

\begin{definition}[Joining opposite sides]  \label{defJoinOppositeSides} 
Fix a Thurston map $f$ with $\card(\post f) \geq 3$ and an $f$-invariant Jordan curve $\CC$ containing $\post f$.  A set $K\subseteq S^2$ \defn{joins opposite sides} of $\CC$ if $K$ meets two disjoint $0$-edges when $\card( \post f)\geq 4$, or $K$ meets  all  three $0$-edges when $\card(\post f)=3$. 
 \end{definition}
 
Note that $\card (\post f) \geq 3$ for each expanding Thurston map $f$ \cite[Lemma~6.1]{BM17}.

We now summarize some basic properties of expanding Thurston maps in the following theorem.

\begin{theorem}[Z.~Li \cite{Li18}, \cite{Li16}]   \label{thmETMBasicProperties}
Let $f\:S^2 \rightarrow S^2$ be an expanding Thurston map, and $d$ be a visual metric on $S^2$ for $f$ with expansion factor $\Lambda>1$. Then the following statements are satisfied:
\begin{enumerate}
\smallskip
\item[(i)] The map $f$ is Lipschitz with respect to $d$.

\smallskip
\item[(ii)] The map $f$ has $1+ \deg f$ fixed points, counted with weight given by the local degree of the map at each fixed point. In particular, $\card \sum\limits_{x\in P_{1,f^n}} \deg_{f^n}(x) = 1 + \deg f^n$.
\end{enumerate}
\end{theorem}

Theorem~\ref{thmETMBasicProperties}~(i) was shown in \cite[Lemma~3.12]{Li18}. Theorem~\ref{thmETMBasicProperties}~(ii) follows from \cite[Theorem~1.1]{Li16} and Remark~\ref{rmExpanding}.

We record the following two lemmas from \cite{Li16} (see \cite[Lemma~4.1 and Lemma~4.2]{Li16}) which give us almost precise information on the locations of the periodic points of an expanding Thurston map.

\begin{lemma}[Z.~Li \cite{Li16}]  \label{lmAtLeast1}
Let $f$ be an expanding Thurston map with an $f$-invariant Jordan curve  $\CC$ containing $\post f$. If $X\in \X^1_{\w\w}(f,\CC) \cup \X^1_{\b\b}(f,\CC)$ is a white $1$-tile contained in the while $0$-tile $X^0_\w$ or a black $1$-tile contained in the black $0$-tile $X^0_\b$, then $X$ contains at least one fixed point of $f$. If $X\in \X^1_{\w\b}(f,\CC) \cup \X^1_{\b\w}(f,\CC)$ is a white $1$-tile contained in the black $0$-tile $X^0_\b$ or a black $1$-tile contained in the white $0$-tile $X^0_\w$, then $\inte (X)$ contains no fixed points of $f$.
\end{lemma}

Recall that cells in the cell decompositions are by definition closed sets, and the set of $0$-tiles $\X^0(f,\CC)$ consists of the white $0$-tile $X^0_\w$ and the black $0$-tile $X^0_\b$.

\begin{lemma}[Z.~Li \cite{Li16}]   \label{lmAtMost1}
Let $f$ be an expanding Thurston map with an $f$-invariant Jordan curve  $\CC$ containing $\post f$ such that no $1$-tile in $\X^1(f,\CC)$ joins opposite sides of $\CC$. Then for every $n\in\N$, each $n$-tile $X^n \in\X^n(f,\CC)$ contains at most one fixed point of $f^n$.
\end{lemma}

The following lemma proved in \cite[Lemma~3.13]{Li18} generalizes \cite[Lemma~15.25]{BM17}.

\begin{lemma}[M.~Bonk \& D.~Meyer \cite{BM17}, Z.~Li \cite{Li18}]   \label{lmMetricDistortion}
Let $f\:S^2 \rightarrow S^2$ be an expanding Thurston map, and $\CC \subseteq S^2$ be a Jordan curve that satisfies $\post f \subseteq \CC$ and $f^{n_\CC}(\CC)\subseteq\CC$ for some $n_\CC\in\N$. Let $d$ be a visual metric on $S^2$ for $f$ with expansion factor $\Lambda>1$. Then there exists a constant $C_0 > 1$, depending only on $f$, $d$, $\CC$, and $n_\CC$, with the following property:

If $k,n\in\N_0$, $X^{n+k}\in\X^{n+k}(f,\CC)$, and $x,y\in X^{n+k}$, then 
\begin{equation}   \label{eqMetricDistortion}
\frac{1}{C_0} d(x,y) \leq \frac{d(f^n(x),f^n(y))}{\Lambda^n}  \leq C_0 d(x,y).
\end{equation}
\end{lemma}

We summarize the existence, uniqueness, and some basic properties of equilibrium states for expandning Thurston maps in the following theorem.

\begin{theorem}[Z.~Li \cite{Li18}]   \label{thmEquilibriumState}
Let $f\:S^2 \rightarrow S^2$ be an expanding Thurston map and $d$ a visual metric on $S^2$ for $f$. Let $\phi,\gamma\in \Holder{\alpha}(S^2,d)$ be real-valued H\"{o}lder continuous functions with an exponent $\alpha\in(0,1]$. Then the following statements are satisfied:
\begin{enumerate}
\smallskip
\item[(i)] There exists a unique equilibrium state $\mu_\phi$ for the map $f$ and the potential $\phi$.

\smallskip
\item[(ii)] For each $t\in\R$, we have
$
\frac{\mathrm{d}}{\mathrm{d}t} P(f,\phi + t\gamma) = \int \!\gamma \,\mathrm{d}\mu_{\phi + t\gamma}.
$

\smallskip
\item[(iii)] If $\CC \subseteq S^2$ is a Jordan curve containing $\post f$ with the property that $f^{n_\CC}(\CC)\subseteq \CC$ for some $n_\CC\in\N$
, then
\begin{equation*}
\mu_\phi \biggl( \, \bigcup\limits_{i=0}^{+\infty}  f^{-i}(\CC)  \biggr)  = 0.
\end{equation*}
\end{enumerate}
\end{theorem}

Theorem~\ref{thmEquilibriumState}~(i) is part of \cite[Theorem~1.1]{Li18}. Theorem~\ref{thmEquilibriumState}~(ii) follows immediately from \cite[Theorem~6.13]{Li18} and the uniqueness of equilibrium states in Theorem~\ref{thmEquilibriumState}~(i). Theorem~\ref{thmEquilibriumState}~(iii) was established in \cite[Proposition~7.1]{Li18}.

The following two distortion lemmas serve as cornerstones in the developement of thermodynamical formalism for expanding Thurston maps in \cite{Li18} (see \cite[Lemma~5.1 and Lemma~5.2]{Li18}).

\begin{lemma}[Z.~Li \cite{Li18}]    \label{lmSnPhiBound}
Let $f\:S^2 \rightarrow S^2$ be an expanding Thurston map and $\CC \subseteq S^2$ be a Jordan curve containing $\post f$ with the property that $f^{n_\CC}(\CC)\subseteq \CC$ for some $n_\CC\in\N$. Let $d$ be a visual metric on $S^2$ for $f$ with expansion factor $\Lambda>1$. Let $\phi\in \Holder{\alpha}(S^2,d)$ be a real-valued H\"{o}lder continuous function with an exponent $\alpha\in(0,1]$. Then there exists a constant $C_1=C_1(f,\CC,d,\phi,\alpha)$ depending only on $f$, $\CC$, $d$, $\phi$, and $\alpha$ such that
\begin{equation}  \label{eqSnPhiBound}
\Abs{S_n\phi(x)-S_n\phi(y)}  \leq C_1 d(f^n(x),f^n(y))^\alpha,
\end{equation}
for $n,m\in\N_0$ with $n\leq m $, $X^m\in\X^m(f,\CC)$, and $x,y\in X^m$. Quantitatively, we choose 
\begin{equation}   \label{eqC1Expression}
C_1 \coloneqq \frac{\Hseminorm{\alpha,\, (S^2,d)}{\phi} C_0}{1-\Lambda^{-\alpha}},
\end{equation}
where $C_0 > 1$ is a constant depending only on $f$, $\CC$, and $d$ from Lemma~\ref{lmMetricDistortion}.
\end{lemma}

\begin{lemma}[Z.~Li \cite{Li18}]    \label{lmSigmaExpSnPhiBound}
Let $f\:S^2 \rightarrow S^2$ be an expanding Thurston map and $\CC \subseteq S^2$ be a Jordan curve containing $\post f$ with the property that $f^{n_\CC}(\CC)\subseteq \CC$ for some $n_\CC\in\N$. Let $d$ be a visual metric on $S^2$ for $f$ with expansion factor $\Lambda>1$. Let $\phi\in \Holder{\alpha}(S^2,d)$ be a real-valued H\"{o}lder continuous function with an exponent $\alpha\in(0,1]$. Then there exists $C_2=C_2(f,\CC,d,\phi,\alpha)  \geq 1$ depending only on $f$, $\CC$, $d$, $\phi$, and $\alpha$ such that for each $x,y\in S^2$, and each $n\in\N_0$, we have
\begin{equation}   \label{eqSigmaExpSnPhiBound}
\frac{\sum\limits_{x'\in f^{-n}(x)}  \deg_{f^n}(x')  \exp (S_n\phi(x'))}{\sum\limits_{y'\in f^{-n}(y)}  \deg_{f^n}(y')  \exp (S_n\phi(y'))} \leq \exp\left(4C_1 Ld(x,y)^\alpha\right) \leq C_2,
\end{equation}
where $C_1$ is the constant from Lemma~\ref{lmSnPhiBound}. Quantitatively, we choose 
\begin{equation}  \label{eqC2Bound}
C_2 \coloneqq \exp\left(4C_1 L \left(\diam_d(S^2)\right)^\alpha \right)  = \exp\biggl(4 \frac{\Hseminorm{\alpha,\, (S^2,d)}{\phi} C_0}{1-\Lambda^{-\alpha}} L \left(\diam_d (S^2)\right)^\alpha \biggr),
\end{equation}
where $C_0 > 1$ is a constant depending only on $f$, $\CC$, and $d$ from Lemma~\ref{lmMetricDistortion}.
\end{lemma}

Recall that the main tool used in \cite{Li18} to develop the thermodynamical formalism for expanding Thurston maps is the Ruelle operator. We will need a complex version of the Ruelle operator in this paper discussed in \cite{Li17}. We summarize relevant definitions and facts about the Ruelle operator below and refer the readers to \cite[Chapter~3.3]{Li17} for a detailed discussion.

Let $f\: S^2\rightarrow S^2$ be an expanding Thurston map and $\phi\in \CCC(S^2,\C)$ be a complex-valued continuous function.  The \defn{Ruelle operator} $\RR_\phi$ (associated to $f$ and $\phi$) acting on $\CCC(S^2,\C)$ is defined as the following
\begin{equation}   \label{eqDefRuelleOp}
\RR_\phi(u)(x)= \sum\limits_{y\in f^{-1}(x)}  \deg_f(y) u(y) \exp(\phi(y)),
\end{equation}
for each $u\in \CCC(S^2,\C)$. Note that $\RR_\phi$ is a well-defined and continuous operator on $\CCC(S^2,\C)$. The Ruelle operator $\RR_\phi \: \CCC(S^2,\C) \rightarrow \CCC(S^2,\C)$ has a natural extension to the space of complex-valued bounded Borel functions $B(S^2,\C)$ (equipped with the uniform norm) given by (\ref{eqDefRuelleOp}) for each $u\in B(S^2,\C)$. 

We observe that if $\phi\in\CCC(S^2)$ is real-valued, then $\RR_\phi ( \CCC(S^2)) \subseteq \CCC(S^2)$ and $\RR_\phi ( B(S^2)) \subseteq B(S^2)$. The adjoint operator $\RR_\phi^*\: \CCC^*(S^2)\rightarrow \CCC^*(S^2)$ of $\RR_\phi$ acts on the dual space $\CCC^*(S^2)$ of the Banach space $\CCC(S^2)$. We identify $\CCC^*(S^2)$ with the space $\MMM(S^2)$ of finite signed Borel measures on $S^2$ by the Riesz representation theorem.

When $\phi\in\CCC(S^2)$ is real-valued, we denote 
\begin{equation}   \label{eqDefPhi-}
\overline\phi \coloneqq \phi - P(f,\phi).
\end{equation}

We record the following three technical results on the Ruelle operators in our context.

\begin{lemma}[Z.~Li \cite{Li18}]    \label{lmR1properties}
Let $f\:S^2 \rightarrow S^2$ be an expanding Thurston map and $\CC \subseteq S^2$ be a Jordan curve containing $\post f$ with the property that $f^{n_\CC}(\CC)\subseteq \CC$ for some $n_\CC\in\N$. Let $d$ be a visual metric on $S^2$ for $f$ with expansion factor $\Lambda>1$. Let $\phi\in \Holder{\alpha}(S^2,d)$ be a real-valued H\"{o}lder continuous function with an exponent $\alpha\in(0,1]$. Then there exists a constant $C_3=C_3(f, \CC, d, \phi, \alpha)$ depending only on $f$, $\CC$, $d$, $\phi$, and $\alpha$ such that for each $x,y\in S^2$ and each $n\in \N_0$ the following equations are satisfied
\begin{equation}  \label{eqR1Quot}
\frac{\RR_{\overline{\phi}}^n(\mathbbm{1})(x)}{\RR_{\overline{\phi}}^n(\mathbbm{1})(y)} \leq \exp\left(4C_1 Ld(x,y)^\alpha\right) \leq C_2,
\end{equation}
\begin{equation}   \label{eqR1Bound}
\frac{1}{C_2} \leq \RR_{\overline{\phi}}^n(\mathbbm{1})(x)  \leq C_2,
\end{equation}
\begin{equation}    \label{eqR1Diff}
 \Abs{\RR_{\overline{\phi}}^n(\mathbbm{1})(x) - \RR_{\overline{\phi}}^n(\mathbbm{1})(y) }  
\leq  C_2 \left( \exp\left(4C_1 Ld(x,y)^\alpha\right) - 1 \right)  \leq C_3 d(x,y)^\alpha,   
\end{equation}
where $C_1, C_2$ are constants in Lemma~\ref{lmSnPhiBound} and Lemma~\ref{lmSigmaExpSnPhiBound} depending only on $f$, $\CC$, $d$, $\phi$, and $\alpha$.
\end{lemma}

Lemma~\ref{lmR1properties} was proved in \cite[Lemma~5.15]{Li18}. The next theorem is part of \cite[Theorem~5.16]{Li18}.

\begin{theorem}[Z.~Li \cite{Li18}]   \label{thmMuExist}
Let $f\:S^2 \rightarrow S^2$ be an expanding Thurston map and $\CC \subseteq S^2$ be a Jordan curve containing $\post f$ with the property that $f^{n_\CC}(\CC)\subseteq \CC$ for some $n_\CC\in\N$. Let $d$ be a visual metric on $S^2$ for $f$ with expansion factor $\Lambda>1$. Let $\phi\in \Holder{\alpha}(S^2,d)$ be a real-valued H\"{o}lder continuous function with an exponent $\alpha\in(0,1]$. Then the sequence $\Big\{\frac{1}{n}\sum\limits_{j=0}^{n-1} \RR_{\overline{\phi}}^j(\mathbbm{1})\Big\}_{n\in\N}$ converges uniformly to a function $u_\phi \in \Holder{\alpha}(S^2,d)$, which satisfies
\begin{equation}    \label{eqRu=u}
\RR_{\overline{\phi}}(u_\phi)  = u_\phi,
\end{equation}
and
\begin{equation}    \label{eqU_phiBounds}
\frac{1}{C_2} \leq u_\phi(x) \leq C_2, \qquad  \text{for each } x\in S^2,
\end{equation}
where $C_2\geq 1$ is a constant from Lemma~\ref{lmSigmaExpSnPhiBound}.
\end{theorem}

Let $f\:S^2 \rightarrow S^2$ be an expanding Thurston map and $d$ be a visual metric on $S^2$ for $f$ with expansion factor $\Lambda>1$. Let $\phi\in \Holder{\alpha}(S^2,d)$ be a real-valued H\"{o}lder continuous function with an exponent $\alpha\in(0,1]$. Then we denote
\begin{equation}    \label{eqDefPhiWidetilde}
\wt{\phi} \coloneqq \phi - P(f,\phi) + \log u_\phi - \log ( u_\phi \circ f ),
\end{equation}
where $u_\phi$ is the continuous function given by Theorem~\ref{thmMuExist}.

The next theorem follows immediately from \cite[Theorem~6.8 and Corollary~6.10]{Li18}.

\begin{theorem}[Z.~Li \cite{Li18}]   \label{thmRR^nConv} 
Let $f\:S^2 \rightarrow S^2$ be an expanding Thurston map. Let $d$ be a visual metric on $S^2$ for $f$ with expansion factor $\Lambda>1$. Let $b \in (0,+\infty)$ be a constant and $h\:[0,+\infty)\rightarrow [0,+\infty)$ be an abstract modulus of continuity. Let $H$ be a bounded subset of $\Holder{\alpha}(S^2,d)$ for some $\alpha \in (0,1]$. Then for each $v\in \CCC_h^b(S^2,d)$ and each $\phi\in H$, we have
\begin{equation}   \label{eqRR^nConv}
\lim\limits_{n\to+\infty} \Norm{\RR_{\overline\phi}^n (v) - u_\phi \int\! v\,\mathrm{d}m_\phi}_{\CCC^0(S^2)} = 0.
\end{equation} 
Moreover, the convergence in (\ref{eqRR^nConv}) is uniform in $v\in\CCC_h^b(S^2,d)$ and $\phi\in H$. Here we denote by $m_\phi$ the unique eigenmeasure of $\RR^*_\phi$, the function $u_\phi$ as defined in Theorem~\ref{thmMuExist}, and $\overline\phi = \phi - P(f,\phi)$.
\end{theorem}

A measure $\mu\in\PPP(S^2)$ is an \defn{eigenmeasure} of $\RR_\phi^*$ if $\RR_\phi^*\mu = c\mu$ for some $c\in\R$. See \cite[Corollary~6.10]{Li18} for the uniqueness of the measure $m_\phi$.

We only need the following two corollaries of Theorem~\ref{thmRR^nConv} in this paper.

\begin{cor}     \label{corUphiCountinuous}
Let $f\:S^2 \rightarrow S^2$ be an expanding Thurston map and $d$ be a visual metric on $S^2$ for $f$ with expansion factor $\Lambda>1$. Let $\phi\in \Holder{\alpha}(S^2,d)$ be a real-valued H\"{o}lder continuous function with an exponent $\alpha\in(0,1]$. We define a map $\tau \: \R\rightarrow \Holder{\alpha}(S^2,d)$ by setting $\tau(t) = u_{t\phi}$. Then $\tau$ is continuous with respect to the uniform norm $\norm{\cdot}_{\CCC^0(S^2)}$ on $\Holder{\alpha}(S^2,d)$.
\end{cor}

\begin{proof}
Fix an arbitrary bounded open interval $I \subseteq \R$. For each $n\in\N$, define $T_n \: I \rightarrow \CCC(S^2,d)$ by $T_n(t) \coloneqq \RR^n_{\overline{t\phi}}  ( \mathbbm{1}_{S^2} )$ for $t \in I$. Since $\overline{t\phi}  = t\phi - P(f, t\phi)$, by (\ref{eqDefRuelleOp}) and the continuity of the topological pressure (see for example, \cite[Theorem~3.6.1]{PrU10}), we know that $T_n$ is a continuous function with respect to the uniform norm $\norm{\cdot}_{\CCC^0(S^2)}$ on  $\CCC(S^2,d)$. Applying Theorem~\ref{thmRR^nConv} with $v\coloneqq \mathbbm{1}_{S^2}$ and $H\coloneqq \{ t\phi \,|\, t\in I \}$, we get that $T_n(t)$ converges to $\tau|_I (t)$ in the uniform norm on  $\CCC(S^2,d)$ uniformly in $t\in I$ as $n\to +\infty$. Hence $\tau(t)$ is continuous on $I$. Recall $u_{t\phi} \in \Holder{\alpha}(S^2,d)$ (see Theorem~\ref{thmMuExist}). Therefore $\tau(t)$ is continuous in $t\in\R$ with respect to the uniform norm on $\Holder{\alpha}(S^2,d)$.
\end{proof}

\begin{cor}  \label{corRR^nConvToTopPressureUniform}
Let $f\:S^2 \rightarrow S^2$ be an expanding Thurston map. Let $d$ be a visual metric on $S^2$ for $f$ with expansion factor $\Lambda>1$. Let $H$ be a bounded subset of $\Holder{\alpha}(S^2,d)$ for some $\alpha \in (0,1]$. Then for each $\phi\in H$ and $x\in S^2$, we have
\begin{equation}   \label{eqRR^nConvToTopPressure}
P(f,\phi)= \lim\limits_{n\to +\infty} \frac{1}{n} \log \RR_{\phi}^n(\mathbbm{1})(x) 
         = \lim\limits_{n\to +\infty} \frac{1}{n} \log \sum\limits_{y\in f^{-n}(x)} \deg_{f^n}(y) \exp (S_n\phi(y)).
\end{equation} 
Moreover, the convergence in (\ref{eqRR^nConvToTopPressure}) is uniform in $\phi\in H$ and $x\in S^2$.
\end{cor}

\begin{proof}
The second equality in (\ref{eqRR^nConvToTopPressure}) follows from (\ref{eqDefRuelleOp}) and (\ref{eqLocalDegreeProduct}). Substitute $u \coloneqq \mathbbm{1}$ into Theorem~\ref{thmRR^nConv}, we get that 
$
\RR^n_{\overline\phi}(\mathbbm{1})(x)= e^{-nP(f,\phi)} \RR^n_{\phi}(\mathbbm{1})(x) 
$
converges to $u_\phi(x)$ uniformly in $\phi\in H$ and $x\in S^2$ as $n\to +\infty$. By Theorem~\ref{thmMuExist}, the function $u_\phi$ is continuous and $\frac{1}{C_2} \leq u_\phi(x)\leq C_2$ for the same constant $C_2>1$ from Lemma~\ref{lmSigmaExpSnPhiBound}. Note that $C_2$ depends only on $f$, $\CC$, $d$, $\phi$, and $\alpha$, but it is bounded on $H$ by (\ref{eqC2Bound}). Therefore $\frac{1}{n} \log \RR_{\overline\phi}^n(\mathbbm{1})(x) =-P(f,\phi)+ \frac{1}{n} \log \RR_{\phi}^n(\mathbbm{1})(x) $ converges to $0$ uniformly in $\phi\in H$ and $x\in S^2$ as $n\to+\infty$. 
\end{proof}

Another characterization of the topological pressure in our context analogous to (\ref{eqRR^nConvToTopPressure}) but in terms of periodic points was obtained in \cite[Proposition~6.8]{Li15}. We record it below.

\begin{prop}[Z.~Li \cite{Li15}]    \label{propTopPressureDefPeriodicPts}
Let $f\:S^2 \rightarrow S^2$ be an expanding Thurston map and $d$ be a visual metric on $S^2$ for $f$ with expansion factor $\Lambda>1$. Let $\phi\in \Holder{\alpha}(S^2,d)$ be a real-valued H\"{o}lder continuous function with an exponent $\alpha\in(0,1]$. Fix an arbitrary sequence of functions $\{ w_n \: S^2\rightarrow \R\}_{n\in\N}$ satisfying $w_n(y) \in [1,\deg_{f^n}(y)]$ for each $n\in\N$ and each $y\in S^2$. Then
\begin{equation}  \label{eqTopPressureDefPrePeriodPts}
P(f,\phi)  = \lim\limits_{n\to +\infty} \frac{1}{n} \log \sum\limits_{y\in P_{1,f^n}} w_n(y) \exp (S_n\phi(y)).
\end{equation}  
\end{prop}

\smallskip

The potentials that satisfy the following property are of special interest in the considerations of Prime Orbit Theorems (compare Definition~\ref{defEventuallyPositive}) and in the analytic study of dynamical zeta functions (compare Corollary~\ref{corS0unique}).

\begin{definition}[Eventually positive functions]  \label{defEventuallyPositive}
Let $g\: X\rightarrow X$ be a map on a set $X$, and $\varphi\:X\rightarrow\C$ be a complex-valued functions on $X$. Then $\varphi$ is \defn{eventually positive} if there exists $N\in\N$ such that $S_n\varphi(x)>0$ for each $x\in X$ and each $n\in\N$ with $n\geq N$.
\end{definition}

\begin{lemma}  \label{lmSnPhiHolder}
Let $f\: S^2\rightarrow S^2$ be an expanding Thurston map and $d$ be a visual metric on $S^2$ for $f$. If $\psi \in \Holder{\alpha}((S^2,d),\C)$ is a complex-valued H\"{o}lder continuous function with an exponent $\alpha\in(0,1]$, then $S_n\psi$ also satisfies $S_n\psi \in \Holder{\alpha}((S^2,d),\C)$ for each $n\in\N$.
\end{lemma}
\begin{proof}
Since $f$ is Lipschitz with respect to $d$ by Theorem~\ref{thmETMBasicProperties}~(i), so is $f^i$ for each $i\in\N$. Then $\psi\circ f^i \in \Holder{\alpha}((S^2,d),\C)$ for each $i\in\N$. Thus by (\ref{eqDefSnPt}), $S_n\psi  \in \Holder{\alpha}((S^2,d),\C)$.
\end{proof}

Theorem~\ref{thmEquilibriumState}~(ii) leads to the following corollary that we frequently use, often implicitly, throughout this paper.

\begin{cor}   \label{corS0unique}
Let $f\:S^2 \rightarrow S^2$ be an expanding Thurston map, and $d$ be a visual metric on $S^2$ for $f$. Let $\phi \in \Holder{\alpha}(S^2,d)$ be an eventually positive real-valued H\"{o}lder continuous function with an exponent $\alpha\in(0,1]$. Then the function $t\mapsto P(f,-t\phi)$, $t\in\R$, is strictly decreasing and there exists a unique number $s_0 \in\R$ such that $P(f,-s_0\phi)=0$. Moreover, $s_0>0$.
\end{cor}

\begin{proof}
By Definition~\ref{defEventuallyPositive}, we can choose $m\in\N$ such that $\Phi \coloneqq S_m \phi$ is strictly positive. Denote $A \coloneqq \inf \{ \Phi(x) \,|\, x\in S^2 \} > 0$. Then by Theorem~\ref{thmEquilibriumState}~(ii) and the fact that the equilibrium state $\mu_{\minus t\phi}$ for $f$ and $-t\phi$ is an $f$-invariant probability measure (see Theorem~\ref{thmEquilibriumState}~(i) and Subsection~\ref{subsctThermodynFormalism}), we have that for each $t\in\R$,
\begin{equation}   \label{eqPfcorS0unique}
    \frac{\mathrm{d}}{\mathrm{d}t} P(f,- t\phi) 
= - \int \!\phi \,\mathrm{d}\mu_{\minus t\phi} 
= - \frac{1}{m} \int \! S_m \phi \,\mathrm{d}\mu_{\minus t\phi} 
< 0.
\end{equation}   
By Corollary~\ref{corRR^nConvToTopPressureUniform}, (\ref{eqDeg=SumLocalDegree}), and (\ref{eqLocalDegreeProduct}), for each $t\in\R$ sufficiently large, we have
\begin{align*}
P(f, - t \phi) & =   \lim\limits_{n\to+\infty} \frac{1}{mn} \log \sum\limits_{y\in f^{mn}(x)} \deg_{f^{mn}}(y)
                                \exp \biggl( -t \sum\limits_{i=0}^{n-1} \bigl(\Phi\circ f^{mi} \bigr) (y) \biggr) \\
               & \leq \lim\limits_{n\to+\infty} \frac{1}{mn} \log  ((\deg f)^{mn} \exp  ( -t n A ) ) \\
               & =    \frac{1}{m} \log ( (\deg f )^m \exp ( -t A ) )
               <0.
\end{align*}
Since the topological entropy $h_{\operatorname{top}} (f) = P(f,0) = \log (\deg f)>0$ (see \cite[Corollary~17.2]{BM17}), the corollary follows immediately from (\ref{eqPfcorS0unique}) and the fact that $P(f,\cdot) \: \CCC(S^2)\rightarrow \R$ is continuous (see for example, \cite[Theorem~3.6.1]{PrU10}).
\end{proof}

\subsection{Subshifts of finite type}   \label{subsctSFT}

We give a brief review on the dynamics of one-sided subshifts of finite type in this subsection. We refer the readers to \cite{Ki98} for a beautiful introduction to symbolic dynamics. For a discussion on results on subshifts of finite type in our context, see \cite{PP90, Bal00}.

Let $S$ be a finite nonempty set, and $A \: S\times S \rightarrow \{0,1\}$ be a matrix whose entries are either $0$ or $1$. For $n\in\N_0$, we denote by $A^n$ the usual matrix product of $n$ copies of $A$. We denote the \defn{set of admissible sequences defined by $A$} by
\begin{equation*}
\Sigma_A^+ = \{ \{x_i\}_{i\in\N_0} \,|\, x_i \in S, \, A(x_i,x_{i+1})=1, \, \text{for each } i\in\N_0\}.
\end{equation*}
Given $\theta\in(0,1)$, we equip the set $\Sigma_A^+$ with a metric $d_\theta$ given by $d_\theta(\{x_i\}_{i\in\N_0},\{y_i\}_{i\in\N_0})=\theta^N$ for $\{x_i\}_{i\in\N_0} \neq \{y_i\}_{i\in\N_0}$, where $N$ is the smallest integer with $x_N \neq y_N$. The metric space $\left(\Sigma_A^+,d_\theta \right)$ is equipped with the topology induced from the product topology, and is therefore compact.

The \defn{left-shift operator} $\sigma_A \: \Sigma_A^+ \rightarrow \Sigma_A^+$ (defined by $A$) is given by
\begin{equation*}
\sigma_A ( \{x_i\}_{i\in\N_0} ) = \{x_{i+1}\}_{i\in\N_0}  \qquad \text{for } \{x_i\}_{i\in\N_0} \in \Sigma_A^+.
\end{equation*}

The pair $\left(\Sigma_A^+, \sigma_A\right)$ is called the \defn{one-sided subshift of finite type} defined by $A$. The set $S$ is called the \defn{set of states} and the matrix $A\: S\times S \rightarrow \{0,1\}$ is called the \defn{transition matrix}.

We say that a one-sided subshift of finite type $\bigl(\Sigma_A^+, \sigma_A \bigr)$ is \defn{topologically mixing} if there exists $N\in\N$ such that $A^n(x,y)>0$ for each $n\geq N$ and each pair of $x,y\in S$.

Let $X$ and $Y$ be topological spaces, and $f\:X\rightarrow X$ and $g\:Y\rightarrow Y$ be continuous maps. We say that the topological dynamical system $(X,f)$ is a \defn{factor} of the topological dynamical system $(Y,g)$ if there is a surjective continuous map $\pi\:Y\rightarrow X$ such that $\pi\circ g=f\circ\pi$. We call the map $\pi\: Y\rightarrow X$ a \emph{factor map}. We get the following commutative diagram:
\begin{equation*}
\xymatrix{  Y  \ar[d]_\pi \ar[r]^{g  }  & Y   \ar[d]^\pi \\ 
            X  \ar[r]_{f}               & X.}
\end{equation*}
It follows immediately that $\pi\circ g^n = f^n \circ \pi$ for each $n\in\N$.

We collect some basic facts about subshifts of finite type.

\begin{prop}  \label{propSFT}
Given a finite set of states $S$ and a transition matrix $A \: S\times S \rightarrow \{0,1\}$. Let $\left( \Sigma_A^+ , \sigma_A \right)$ be the one-sided subshift of finite type defined by $A$, and $\phi\in\Holder{1}\bigl(\Sigma_A^+, d_\theta\bigr)$ be a real-valued Lipschitz continuous function with $\theta \in (0,1)$. Then the following statements are satisfied:
\begin{enumerate}
\smallskip
\item[(i)] $\card P_{1,\sigma_A^n} \leq (\card S)^n$ for all $n\in\N$.

\smallskip
\item[(ii)] $P(\sigma_A,\phi) \geq \limsup\limits_{n\to +\infty} \frac{1}{n} \log  \sum\limits_{\underline{x}\in P_{1,\sigma_A^n}}  \exp (S_n \phi(\underline{x}))$. 

\smallskip
\item[(iii)] $P(\sigma_A,\phi) = \lim\limits_{n\to +\infty} \sup\limits_{\underline{x}\in\Sigma_A^+} 
                 \frac{1}{n} \log \sum\limits_{\underline{y}\in \sigma_A^{-n}(\underline{x})}  \exp (S_n \phi(\underline{y}))$.
\end{enumerate}

If, in addition, $(\Sigma_A^+,\sigma_A)$ is topologically mixing, then
\begin{enumerate}
\smallskip
\item[(iv)] $P(\sigma_A,\phi) = \lim\limits_{n\to +\infty} \frac{1}{n} \log \sum\limits_{\underline{y}\in \sigma_A^{-n}(\underline{x})}  \exp (S_n \phi(\underline{y}))$ for each $\underline{x}\in \Sigma_A^+$.
\end{enumerate}
\end{prop}

\begin{proof}
(i) Fix $n\in\N$. The inequality follows trivially from the observation that each $\{x_i\}_{i\in\N_0} \in P_{1,\sigma_A^n}$ is uniquely determined by the first $n$ entries in the sequence $\{x_i\}_{i\in\N_0}$.

\smallskip

(ii) Fix $n\in\N$. Since each $\{x_i\}_{i\in\N_0} \in P_{1,\sigma_A^n}$ is uniquely determined by the first $n$ entries in the sequence $\{x_i\}_{i\in\N_0}$,
it is clear that each pair of distinct $\{x_i\}_{i\in\N_0}, \{x'_i\}_{i\in\N_0}\in P_{1,\sigma_A^n}$ are $(n,1)$-separated (see Subsection~\ref{subsctThermodynFormalism}). The inequality now follows from (\ref{defTopPressure}) and the observation that an $(n,1)$-separated set is also $(n,\epsilon)$-separated for all $\epsilon\in(0,1)$.

\smallskip

(iii) As remarked in \cite[Remark~1.3]{Bal00}, the proof of statement~(iii) follows from \cite[Lemma~4.5]{Rue89} and the beginning of the proof of Theorem~3.1 in \cite{Rue89}.

\smallskip

(iv) One can find a proof of this well-known fact in \cite[Proposition~4.4.3]{PrU10} (see the first page of Chapter~4 in \cite{PrU10} for relevant definitions).
\end{proof}

\begin{lemma}   \label{lmUnifBddToOneFactorPressure}
For each $i\in\{1,2\}$, given a finite set of states $S_i$ and a transition matrix $A_i \: S_i \times S_i \rightarrow \{0,1\}$, we denote by $\bigl( \Sigma_{A_i}^+ , \sigma_{A_i} \bigr)$ the one-sided subshift of finite type defined by $A_i$. Let $\phi\in\Holder{1}\bigl(\Sigma_{A_2}^+, d_\theta\bigr)$ be a real-valued Lipschitz continuous function on $\Sigma_{A_2}^+$ with $\theta \in (0,1)$. Suppose that there exists a uniformly bounded-to-one H\"{o}lder continuous factor map $\pi\: \Sigma_{A_1}^+ \rightarrow \Sigma_{A_2}^+$, i.e., $\pi$ is a H\"{o}lder continuous surjective map with $\sigma_{A_2} \circ \pi = \pi \circ \sigma_{A_1}$ and $\sup \bigl\{ \card \bigl(\pi^{-1}(\underline{x}) \bigr) \,\big|\, \underline{x}\in \Sigma_{A_2}^+ \bigr\} < +\infty$. Then
\begin{equation*}
P(\sigma_{A_1}, \phi\circ \pi ) = P(\sigma_{A_2}, \phi).
\end{equation*}
\end{lemma}

\begin{proof}
We observe that since $\bigl(\Sigma_{A_2}^+, \sigma_{A_2}\bigr)$ is a factor of $\bigl(\Sigma_{A_1}^+, \sigma_{A_1}\bigr)$ with the factor map $\pi$, it follows from \cite[Lemma~3.2.8]{PrU10} that $P  ( \sigma_{A_1},  \phi \circ \pi  ) \geq P  ( \sigma_{A_2},  \phi  )$. It remains to show $P  ( \sigma_{A_1},  \phi \circ \pi  ) \leq P  ( \sigma_{A_2},  \phi  )$.

Denote $M \coloneqq \sup \bigl\{ \card \bigl(\pi^{-1}(\underline{x}) \bigr) \,\big|\, \underline{x}\in \Sigma_{A_2}^+ \bigr\}$.

Note that $\phi \circ \pi \in\Holder{1}\bigl(\Sigma_{A_1}^+, d_{\theta'}\bigr)$ for some $\theta' \in (0,1)$. By Lemma~\ref{propSFT}~(iii), for each $\epsilon>0$, we can choose a sequence $\{\underline{x}^n\}_{n\in\N_0}$ in $\Sigma_{A_1}^+$ such that
\begin{equation}   \label{eqPflmUnifBddToOneFactorPressure}
     \liminf\limits_{n\to +\infty}  \frac{1}{n} \log  \sum\limits_{\underline{y}\in \sigma_{A_1}^{-n}(\underline{x}^n)}
          \exp \biggl( \sum\limits_{i=0}^{n-1} \bigl( \phi \circ \pi \circ \sigma_{A_1}^i \bigr)(\underline{y}) \biggr)
\geq P(\sigma_{A_1}, \phi \circ \pi) - \epsilon.
\end{equation}
Observe that for all $\underline{x}, \underline{y} \in \Sigma_{A_1}^+$, and $n\in \N_0$, if $\sigma_{A_1}^n(\underline{y})=\underline{x}$, then 
$\pi(\underline{x}) = \bigl( \pi\circ \sigma_{A_1}^n \bigr) (\underline{y}) =  \bigl( \sigma_{A_2}^n \circ \pi \bigr) (\underline{y})$. Thus by (\ref{eqPflmUnifBddToOneFactorPressure}) and Proposition~\ref{propSFT}~(iii),
\begin{align*}
        P(\sigma_{A_1}, \phi \circ \pi) - \epsilon
\leq  & \liminf\limits_{n\to +\infty}  \frac{1}{n} \log  \sum\limits_{\underline{y}\in \sigma_{A_1}^{-n}(\underline{x}^n)}
          \exp \biggl( \sum\limits_{i=0}^{n-1} \bigl( \phi \circ \sigma_{A_2}^i \bigr)( \pi(\underline{y} ) ) \biggr) \\
\leq  & \liminf\limits_{n\to +\infty}  \frac{1}{n} \log  \sum\limits_{\underline{y}\in \pi^{-1} \left( \sigma_{A_2}^{-n}(\pi(\underline{x}^n)) \right)}
          \exp \biggl( \sum\limits_{i=0}^{n-1} \bigl( \phi \circ \sigma_{A_2}^i \bigr)( \pi(\underline{y} ) ) \biggr) \\
\leq  & \limsup\limits_{n\to +\infty}  \frac{1}{n} \log \Biggl( M \sum\limits_{\underline{z}\in  \sigma_{A_2}^{-n}(\pi(\underline{x}^n))  }
          \exp \biggl( \sum\limits_{i=0}^{n-1} \bigl( \phi \circ \sigma_{A_2}^i \bigr)(  \underline{z}  ) \biggr)  \Biggr) \\
\leq  & P(\sigma_{A_2}, \phi ) .
\end{align*}
Since $\epsilon>0$ is arbitrary, we get $P(\sigma_{A_1}, \phi \circ \pi) \leq P(\sigma_{A_2}, \phi )$. The proof is complete.
\end{proof}

\smallskip

We will now consider a one-sided subshift of finite type associated to an expanding Thurston map and an invariant Jordan curve on $S^2$ containing $\post f$. The construction of other related symbolic systems will be postponed to Section~\ref{sctDynOnC}. We will need the following technical lemma in the construction of these symbolic systems. 

\begin{lemma}   \label{lmCylinderIsTile}
Let $f\: S^2 \rightarrow S^2$ be an expanding Thurston map with a Jordan curve $\CC\subseteq S^2$ satisfying $f(\CC)\subseteq \CC$ and $\post f\subseteq \CC$. Let $\{X_i\}_{i\in\N_0}$ be a sequence of $1$-tiles in $\X^1(f,\CC)$ satisfying $f(X_i)\supseteq X_{i+1}$ for all $i\in\N_0$. Let $\{e_j\}_{j\in\N_0}$ be a sequence of $1$-edges in $\E^1(f,\CC)$ satisfying $f(e_j)\supseteq e_{j+1}$ for all $j\in\N_0$. Then for each $n\in\N$, we have
\begin{align}  
\bigl( (f|_{X_0})^{-1} \circ (f|_{X_1})^{-1} \circ \cdots \circ (f|_{X_{n-2}})^{-1} \bigr) (X_{n-1}) =  &\bigcap\limits_{i=0}^{n-1}  f^{-i} (X_i) \in \X^n(f,\CC), \label{eqCylinderIsTile}\\
\bigl( (f|_{e_0})^{-1} \circ (f|_{e_1})^{-1} \circ \cdots \circ (f|_{e_{n-2}})^{-1} \bigr) (e_{n-1}) =&\bigcap\limits_{j=0}^{n-1}  f^{-j} (e_j) \in \E^n(f,\CC). \label{eqCylinderIsEdge}
\end{align}
Moreover, both $\bigcap_{i\in\N_0} f^{-i}(X_i)$ and $\bigcap_{j\in\N_0} f^{-j}(e_j)$ are singleton sets.
\end{lemma}

\begin{proof}
Let $d$ be a visual metric on $S^2$ for $f$.

We call a sequence $\{c_i\}_{i\in\N_0}$ of subsets of $S^2$ admissible if $f(c_i)\supseteq c_{i+1}$ for all $i\in\N_0$.
  
We are going to prove (\ref{eqCylinderIsTile}) by induction. The proof of the case of edges in (\ref{eqCylinderIsEdge}) is verbatim the same.

For $n=1$, (\ref{eqCylinderIsTile}) holds trivially for each admissible sequence of $1$-tiles $\{X_i\}_{i\in\N_0}$ in $\X^1$.

Assume that (\ref{eqCylinderIsTile}) holds for each admissible sequence of $1$-tiles $\{X_i\}_{i\in\N_0}$ in $\X^1$ and for $n=m$ for some $m\in\N$. We fix such a sequence $\{X_i\}_{i\in\N_0}$. Then $\{X_{i+1}\}_{i\in\N_0}$ is also admissible. By the induction hypothesis, we denote
\begin{equation*}
X^m \coloneqq \bigl( (f|_{X_1})^{-1} \circ (f|_{X_2})^{-1} \circ \cdots \circ (f|_{X_{m-1}})^{-1} \bigr) (X_m) = \bigcap\limits_{i=0}^{m-1} f^{-i}(X_{i+1}) \in \X^m. 
\end{equation*}
Since $f(X_0) \supseteq X_1$ and $X^m\subseteq X_1$, we get from Proposition~\ref{propCellDecomp}~(i) and (ii) that $f$ is injective on $X_1$, and thus $\bigcap\limits_{i=0}^{m}  f^{-i} (X_i) = X_0 \cap f^{-1}(X^m) \in \X^{m+1}$, and $(f|_{X_0})^{-1} (X^m) = X_0 \cap f^{-1}(X^m) \in \X^{m+1}$.

The induction is complete. We have established (\ref{eqCylinderIsTile}).

Note that $\bigcap\limits_{i=0}^{n-1}  f^{-i} (X_i)  \supseteq \bigcap\limits_{i=0}^{n}  f^{-i} (X_i) \in \X^{n+1}$ for each $n\in\N$. By Lemma~\ref{lmCellBoundsBM}~(ii), $\bigcap_{i\in\N_0} f^{-i}(X_i)$ is the intersection of a nested sequence of closed sets with radii convergent to zero, thus it contains exactly one point in $S^2$. Similarly, $\card \bigcap_{j\in\N_0} f^{-i}(e_j) = 1$.
\end{proof}

\begin{prop}   \label{propTileSFT}
Let $f\: S^2 \rightarrow S^2$ be an expanding Thurston map with a Jordan curve $\CC\subseteq S^2$ satisfying $f(\CC)\subseteq \CC$ and $\post f\subseteq \CC$. Let $d$ be a visual metric on $S^2$ for $f$ with expansion factor $\Lambda>1$. Fix $\theta\in(0,1)$. We set $S_{\ti} \coloneqq \X^1(f,\CC)$, and define a transition matrix $A_{\ti}\: S_{\ti}\times S_{\ti} \rightarrow \{0,1\}$ by
\begin{equation*}
A_{\ti}(X,X') = \begin{cases} 1 & \text{if } f(X)\supseteq X', \\ 0  & \text{otherwise}  \end{cases}
\end{equation*}
for $X,X'\in \X^1(f,\CC)$. Then $f$ is a factor of the one-sided subshift of finite type $\bigl(\Sigma_{A_{\ti}}^+, \sigma_{A_{\ti}}\bigr)$ defined by the transition matrix $A_{\ti}$. More precisely, the following diagram commutes
\begin{equation*}
\xymatrix{ \Sigma_{A_{\ti}}^+ \ar[d]_{\pi_{\ti}} \ar[r]^{\sigma_{A_{\ti}}}  & \Sigma_{A_{\ti}}^+ \ar[d]^{\pi_{\ti}} \\ 
            S^2                                  \ar[r]_{f}                 & S^2,}
\end{equation*}
where the factor map $\pi_{\ti}\: \Sigma_{A_{\ti}}^+ \rightarrow S^2$ is a surjective H\"{o}lder continuous map defined by 
\begin{equation}   \label{eqDefTileSFTFactorMap}
\pi_{\ti} \left( \{X_i\}_{i\in\N_0} \right)= x,  \text{ where } \{x\} = \bigcap\limits_{i \in \N_0} f^{-i} (X_i).
\end{equation}
Here $\Sigma_{A_{\ti}}^+$ is equipped with the metric $d_\theta$ defined in Subsection~\ref{subsctSFT}, and $S^2$ is equipped with the visual metric $d$.

Moreover, $\bigl(\Sigma_{A_{\ti}}^+, \sigma_{A_{\ti}} \bigr)$ is topologically mixing and $\pi_{\ti}$ is injective on $\pi_{\ti}^{-1} \Bigl( S^2 \setminus \bigcup\limits_{i\in\N_0} f^{-i}(\CC) \Bigr)$.
\end{prop}

\begin{rem}
We can show that if $f$ has no periodic critical points, then $\pi$ is uniformly bounded-to-one (i.e., there exists $N\in\N_0$ depending only on $f$ such that $\card \left(\pi_{\ti}^{-1}(x)\right) \leq N$ for each $x\in S^2$); if $f$ has at least one periodic critical point, then $\pi_{\ti}$ is uncountable-to-one on a dense set. We will not use this fact in this paper.
\end{rem}

\begin{proof}
We denote by $\{X_i\}_{i\in\N_0} \in \Sigma_{A_{\ti}}^+$ an arbitrary admissible sequence.

Since $f(X_i)\supseteq X_{i+1}$ for each $i\in\N_0$, by Lemma~\ref{lmCylinderIsTile}, the map $\pi_{\ti}$ is well-defined.

Note that for each $m\in \N_0$ and each $\{X'_i\}_{i\in\N_0} \in \Sigma_{A_{\ti}}^+$ with $X_{m+1} \neq X'_{m+1}$ and $X_j=X'_j$ for each integer $j\in[0,m]$, we have $\{ \pi_{\ti}  (\{X_i\}_{i\in\N_0}  ),  \pi_{\ti} ( \{ X'_i \}_{i\in\N_0} ) \} \subseteq \bigcap\limits_{i=0}^{m} f^{-i}(X_i) \in \X^{m+1}$ by Lemma~\ref{lmCylinderIsTile}. Thus it follows from Lemma~\ref{lmCellBoundsBM}~(ii) that $\pi_{\ti}$ is H\"{o}lder continuous.

To see that $\pi_{\ti}$ is surjective, we observe that for each $x\in S^2$, we can find a sequence $\bigl\{ X^j(x) \bigr\}_{j\in\N}$ of tiles such that $X^j(x) \in \X^j$, $x\in X^j(x)$, and $X^j(x) \supseteq X^{j+1}(x)$ for each $j\in\N$. Then it is clear that $\bigl\{  f^i\bigl( X^{i+1}(x) \bigr) \bigr\}_{i\in\N_0} \in \Sigma_{A_{\ti}}^+$ and $\pi_{\ti} \Bigl(  \bigl\{ f^i\bigl( X^{i+1}(x) \bigr) \bigr\}_{i\in\N_0} \Bigr) = x$.

To check that $\pi_{\ti} \circ \sigma_{A_{\ti}} = f \circ \pi_{\ti}$, it suffices to observe that
\begin{align*}
            \{ (f\circ \pi_{\ti}) (\{X_i\}_{i\in\N_0}) \}
  =       & f \biggl(  \bigcap\limits_{j\in\N_0} f^{-j} (X_j) \biggr)
\subseteq  \bigcap\limits_{j\in\N} f^{-(j-1)} (X_j) \\
  =       & \bigcap\limits_{i\in\N_0} f^{-i} (X_{i+1})  
  =         \{(\pi_{\ti} \circ \sigma_{A_{\ti}} ) (\{X_i\}_{i\in\N_0})  \}.
\end{align*}

To show that $\pi_{\ti}$ is injective on $\pi_{\ti}^{-1}(S^2\setminus \E)$, where we denote $\E \coloneqq \bigcup\limits_{i\in\N_0} f^{-i}(\CC)$, we fix another arbitrary $\{Y_i\}_{i\in\N_0} \in \Sigma_{A_{\ti}}^+$ with $\{X_i\}_{i\in\N_0} \neq \{Y_i\}_{i\in\N_0}$. Suppose that $x = \pi_{\ti} (\{X_i\}_{i\in\N_0}) = \pi_{\ti} (\{Y_i\}_{i\in\N_0}) \notin \E$. Choose $n\in\N_0$ with $X_n\neq Y_n$. Then by Lemma~\ref{lmCylinderIsTile}, $x\in \bigcap\limits_{i=0}^n f^{-i} (X_i) \in \X^{n+1}$ and $x\in \bigcap\limits_{i=0}^n f^{-i} (Y_i) \in \X^{n+1}$. Thus $f^n(x) \in X_n \cap Y_n \subseteq f^{-1}(\CC)$ by Proposition~\ref{propCellDecomp}~(v). This is a contradiction to the assumption that $x\notin \E$. 

We finally demonstrate that $\bigl( \Sigma_{A_{\ti}}^+, \sigma_{A_{\ti}} \bigr)$ is topologically mixing. By Lemma~\ref{lmCellBoundsBM}~(iii), there exists a number $M\in\N$ such that for each $m\geq M$, there exist white $m$-tiles $X^m_\w, Y^m_\w \in \X^m_\w$ and black $m$-tiles $X^m_\b, Y^m_\b \in \X^m_\b$ satisfying $X^m_\w \cup X^m_\b \subseteq X^0_\b$ and $Y^m_\w \cup Y^m_\b \subseteq X^0_\w$, where $X^0_\b$ and $X^0_\w$ are the black $0$-tile and the white $0$-tile, respectively. Thus for all $X,X'\in\X^1$, and all $n\geq M+1$, we have $f^n(X) = f^{n-1}(f(X)) \supseteq X^0_\b \cup X^0_\w = S^2 \supseteq X'$.
\end{proof}

\subsection{Dynamical zeta functions and Dirichlet series}   \label{subsctDynZetaFn}
Let $g\: X\rightarrow X$ be a continuous map on a compact metric space $(X,d)$. Let $\psi\in\CCC(X,\C)$ be a complex-valued continuous function on $X$. We write
\begin{equation}  \label{eqDefZn}
Z_{g,\,\minus\psi}^{(n)} (s) \coloneqq  \sum\limits_{x\in P_{1,g^n}} e^{-s S_n \psi(x)}  , \qquad n\in\N \text{ and }  s\in\C.
\end{equation}
Recall that $P_{1,g^n}$ defined in (\ref{eqDefSetPeriodicPts}) is the set of fixed points of $g^n$, and $S_n\psi$ is defined in (\ref{eqDefSnPt}). We denote by the formal infinite product
\begin{equation}  \label{eqDefZetaFn}
\zeta_{g,\,\minus\psi} (s) \coloneqq \exp \Biggl( \sum\limits_{n=1}^{+\infty} \frac{Z_{g,\,\minus\psi}^{(n)} (s)}{n}  \Biggr) 
 = \exp \Biggl( \sum\limits_{n=1}^{+\infty} \frac{1}{n} \sum\limits_{x\in P_{1,g^n}} e^{-s S_n \psi(x)} \Biggr), \qquad s\in\C
\end{equation}
the \defn{dynamical zeta function} for the map $g$ and the potential $\psi$. 

More generally, we can define dynamical Dirichlet series as analogues of Dirichlet series in analytic number theory.

\begin{definition}   \label{defDynDirichletSeries}
Let $g\: X\rightarrow X$ be a continuous map on a compact metric space $(X,d)$. Let $\psi\in\CCC(X,\C)$ be a complex-valued continuous function on $X$, and $w\: X\rightarrow \C$ be a complex-valued function on $X$. We denote by the formal infinite product
\begin{equation}  \label{eqDefDynDirichletSeries}
\DS_{g,\,\minus\psi,\,w} (s) \coloneqq \exp \Biggl( \sum\limits_{n=1}^{+\infty} \frac{1}{n} \sum\limits_{x\in P_{1,g^n}} e^{-s S_n \psi(x)} \prod\limits_{i=0}^{n-1} w \bigl( g^i(x) \bigr) \Biggr), \qquad s\in\C
\end{equation}
the \defn{dynamical Dirichlet series} with coefficient $w$ for the map $g$ and the potential $\psi$.

\end{definition}

\begin{rem}
Dynamical zeta functions are special cases of dynamical Dirichlet series, more precisely, $\zeta_{g,\,\minus\psi}  =  \DS_{g,\,\minus\psi,\,\mathbbm{1}_X}$. Dynamical Dirichlet series defined above can be considered as analogues of Dirichlet series equipped with a strongly multiplicative arithmetic function in analytic number theory. We can define more general dynamical Dirichlet series by replacing $w$ by $w_n$, where $w_n\: X \rightarrow \C$ is a complex-valued function on $X$ for each $n\in\N$. We will not need such generality in this paper.
\end{rem}

\begin{lemma}   \label{lmDynDirichletSeriesConv_general}
Let $g\:X\rightarrow X$ be a continuous map on a compact metric space $(X,d)$. Let $\varphi\in \CCC(X)$ be a real-valued continuous function on $X$, and $w\: X\rightarrow \C$ be a complex-valued function on $X$. Given $a\in\R$. Suppose that the following conditions are satisfied:
\begin{enumerate}
\smallskip
\item[(i)] $\card  P_{1,g^n}  < +\infty$ for all $n\in\N$.

\smallskip
\item[(ii)] $\limsup\limits_{n\to+\infty} \frac{1}{n} \log  \sum\limits_{x\in P_{1,g^n}} \exp ( - a S_n \varphi(x)) \prod\limits_{i=0}^{n-1} \Absbig{w \bigl( g^i(x) \bigr)} < 0$.
\end{enumerate}

Then for each $s\in\C$ with $\Re(s)=a$, the dynamical Dirichlet series $\DS_{g,\,\minus\varphi,\,w}(s)$ as an infinite product converges uniformly and absolutely, and
\begin{equation}   \label{eqZetaFnOrbForm}
\DS_{g,\,\minus\varphi,\,w} (s) = \prod\limits_{\tau\in \Orb(g)} \biggl( 1- e^{ - s l_\varphi (\tau)} \prod\limits_{x\in\tau} w(x) \biggr)^{-1},
\end{equation} 
where $l_\varphi (\tau) \coloneqq \sum\limits_{x\in\tau} \varphi(x)$.

If, in addition, we assume that $\varphi$ is eventually positive, then $\DS_{g,\,\minus\varphi,\,w}(s)$ converges uniformly and absolutely to a non-vanishing bounded holomorphic function on the closed half-plane $\overline{\H}_a = \{s\in\C \,|\, \Re(s) \geq a \}$, and (\ref{eqZetaFnOrbForm}) holds on $\overline{\H}_a$.
\end{lemma}

Recall $\Orb(g)$ denotes the set of all primitive periodic orbits of $g$ (see (\ref{eqDefSetAllPeriodicOrbits})). We recall that an infinite product of the form $\exp \sum a_i$, $a_i\in\C$, converges uniformly (resp.\ absolutely) if $\sum a_i$ converges uniformly (resp.\ absolutely).

\begin{rem}   \label{rmDynDirichletSeriesZetaFn}
It is often possible to verify condition~(ii) by showing $P(g, - a\varphi) <0$ and $P(g, - a\varphi) \geq \limsup\limits_{n\to+\infty} \frac{1}{n} \log  \sum\limits_{x\in P_{1,g^n}} \exp ( - a S_n \varphi(x)) \prod\limits_{i=0}^{n-1} \Absbig{w \bigl( g^i(x) \bigr)}$. This is how we are going to use Lemma~\ref{lmDynDirichletSeriesConv_general} in this paper. In particular, if $\card X < +\infty$, then it follows immediately from (\ref{defTopPressure}) that
\begin{equation*}
    P(g, - a\varphi) 
= \lim\limits_{n\to+\infty} \frac{1}{n} \log  \sum\limits_{x\in X} \exp ( - a S_n \varphi(x))  
\geq \limsup\limits_{n\to+\infty} \frac{1}{n} \log  \sum\limits_{x\in P_{1,g^n}} \exp ( - a S_n \varphi(x)) .
\end{equation*}
\end{rem}

\begin{proof}
Fix $s\in\C$ with $\Re(s)=a$.

By condition~(ii), we can choose constants $N\in\N$ and $\beta\in (0,1)$ such that 
\begin{equation*}
\sum\limits_{x\in P_{1,g^n}} \exp( - a S_n\varphi(x)) \prod\limits_{i=0}^{n-1} \Absbig{w \bigl( g^i(x) \bigr)}  \leq \beta^n
\end{equation*}
for each integer $n\geq N$. Thus
\begin{align*}
&    \sum\limits_{n=N}^{+\infty} \frac{1}{n} \sum\limits_{x\in P_{1,g^n}}  \Absbigg{\exp( - s S_n\varphi(x))  \prod\limits_{i=0}^{n-1}w \bigl(g^i(x)\bigr)}\\
&\qquad =  \sum\limits_{n=N}^{+\infty} \frac{1}{n} \sum\limits_{x\in P_{1,g^n}}       \exp( - a S_n\varphi(x)) 
                                   \prod\limits_{i=0}^{n-1} \Absbig{w \bigl( g^i(x) \bigr)}  
 \leq  \sum\limits_{n=N}^{+\infty}  \beta^n.
\end{align*}
Combining the above inequalities with condition~(i), we can conclude that $\DS_{g,\,\minus\varphi,\,w} (s)$ converges absolutely. Moreover,
\begin{align*}
      \DS_{g,\,\minus\varphi,\,w} (s)
 = &  \exp \Biggl( \sum\limits_{n=1}^{+\infty} \frac{1}{n} \sum\limits_{x\in P_{1,g^n}} e^{ - s S_n \varphi(x)} 
                                                           \prod\limits_{i=0}^{n-1} w \bigl( g^i(x) \bigr)   \Biggr)  \\ 
 = &  \exp \Biggl( \sum\limits_{m=1}^{+\infty} \frac{1}{m} \sum\limits_{x\in P_{m,g}} \sum\limits_{k=1}^{+\infty} \frac{1}{k} e^{-sk S_m \varphi(x)}
                                                           \prod\limits_{i=0}^{kn-1} w \bigl( g^i(x) \bigr)  \Biggr) \\
 = &  \exp \Biggl( \sum\limits_{\tau\in\Orb(g)} \sum\limits_{k=1}^{+\infty} \frac{1}{k} e^{ - sk l_\varphi(\tau)} 
                                                           \prod\limits_{y\in\tau}  w^k  (y) \Biggr) \\
 = &  \exp \Biggl(-\sum\limits_{\tau\in\Orb(g)} \log \biggl( 1- e^{ - s l_\varphi(\tau)} \prod\limits_{y\in\tau} w(y)  \biggr) \Biggr) \\
 = &  \prod\limits_{\tau\in \Orb(g)} \biggl( 1- e^{ - s l_\varphi (\tau)} \prod\limits_{y\in\tau} w(y) \biggr)^{-1}.
\end{align*}

Now we assume, in addition, that $\varphi$ is eventually positive. Then it is clear from the definition that $S_n\varphi(x) > 0$ for all $n\in\N$ and $x\in P_{1,g^n}$. For each $z \in \overline{\H}_a$ and each $m\in\N$,
\begin{equation*}
  \sum\limits_{n=m}^{+\infty} \frac{1}{n} \sum\limits_{x\in P_{1,g^n}}  \Absbigg{\exp( - z S_n\varphi(x))  \prod\limits_{i=0}^{n-1}w \bigl(g^i(x)\bigr)}  
  \leq \sum\limits_{n=m}^{+\infty} \frac{1}{n} \sum\limits_{x\in P_{1,g^n}}       \exp( - a S_n\varphi(x))
                                                                                         \prod\limits_{i=0}^{n-1} \Absbig{w \bigl( g^i(x) \bigr)}.
\end{equation*}
Hence $\DS_{g,\,\minus\varphi,\,w}(z)$ converges uniformly and absolutely to a non-vanishing bounded holomorphic function on the closed half-plane $\overline{\H}_a$.

Finally, to verify (\ref{eqZetaFnOrbForm}) for $z\in \overline{\H}_a$ when $\varphi$ is eventually positive, it suffices to apply (\ref{eqZetaFnOrbForm}) to $a\coloneqq \Re(z)$ with the observation that
\begin{align*}
&            \limsup\limits_{n\to+\infty} \frac{1}{n} \log  \sum\limits_{x\in P_{1,g^n}} \exp ( - \Re(z) S_n \varphi(x)) 
                                                                                         \prod\limits_{i=0}^{n-1} \Absbig{w \bigl( g^i(x) \bigr)}  \\
&\qquad \leq \limsup\limits_{n\to+\infty} \frac{1}{n} \log  \sum\limits_{x\in P_{1,g^n}} \exp ( - a      S_n \varphi(x))
                                                                                         \prod\limits_{i=0}^{n-1} \Absbig{w \bigl( g^i(x) \bigr)}
 < 0,
\end{align*}
i.e., condition~(ii) holds with $a = \Re(z)$.
\end{proof}

\smallskip

We now consider dynamical zeta functions and Dirichlet series associated to expanding Thurston maps.

\begin{prop}   \label{propZetaFnConv_s0}
Let $f\: S^2 \rightarrow S^2$ be an expanding Thurston map with a Jordan curve $\CC\subseteq S^2$ satisfying $f(\CC) \subseteq \CC$ and $\post f \subseteq \CC$. Let $d$ be a visual metric on $S^2$ for $f$ with expansion factor $\Lambda>1$. Let $\phi \in \Holder{\alpha}(S^2,d)$ be an eventually positive real-valued H\"{o}lder continuous function with an exponent $\alpha\in (0,1]$. Denote by $s_0$ the unique positive number with $P(f,-s_0 \phi) = 0$. Let $\bigl(\Sigma_{A_{\ti}}^+,\sigma_{A_{\ti}}\bigr)$ be the one-sided subshift of finite type associated to $f$ and $\CC$ defined in Proposition~\ref{propTileSFT}, and let $\pi_{\ti}\: \Sigma_{A_{\ti}}^+\rightarrow S^2$ be the factor map defined in (\ref{eqDefTileSFTFactorMap}). Denote by $\deg_f(\cdot)$ the local degree of $f$. Then the following statements are satisfied:
\begin{enumerate}
\smallskip
\item[(i)] $P(\sigma_{A_{\ti}}, \varphi\circ\pi_{\ti}) = P(f,\varphi)$ for each $\varphi \in \Holder{\alpha}(S^2,d)$ . In particular, for an arbitrary number $t\in\R$, we have $P(\sigma_{A_{\ti}}, - t \phi \circ \pi_{\ti}) = 0$ if and only if $t=s_0$.

\smallskip
\item[(ii)] All three infinite products $\zeta_{f,\,\minus \phi}$, $\zeta_{\sigma_{A_{\ti}},\,\minus\phi\circsmall\pi_{\ti}}$, and $\DS_{f,\,\minus\phi,\,\deg_f}$ converge uniformly and absolutely to respective non-vanishing holomorphic functions on each closed half-plane $\overline{\H}_a=\{ s\in\C \,|\, \Re(s) \geq a \}$, where $a\in\R$ satisfies $a>s_0$.

\smallskip
\item[(iii)] For all $s\in \C$ with $\Re(s)>s_0$, we have
\begin{align}
\zeta_{f,\,\minus\phi} (s)           & = \prod\limits_{\tau\in \Orb(f)} \biggl( 1- \exp\biggl( - s \sum\limits_{y\in\tau} \phi(y) \biggr) \biggr)^{-1},  \label{eqZetaFnOrbitForm_ThurstonMap}\\
\DS_{f,\,\minus\phi,\,\deg_f}(s)  & = \prod\limits_{\tau\in \Orb(f)} \biggl( 1- \exp\biggl( - s \sum\limits_{y\in\tau} \phi(y) \biggr) \prod\limits_{z\in\tau} \deg_f(z) \biggr)^{-1}, \label{eqZetaFnOrbitForm_ThurstonMapDegree}\\
\zeta_{\sigma_{A_{\ti}},\,\minus\phi\circsmall\pi_{\ti}} (s)    & = \prod\limits_{\tau\in \Orb(\sigma_{A_{\ti}}) } \biggl( 1- \exp\biggl( - s \sum\limits_{y\in\tau} \phi\circ\pi_{\ti}(y) \biggr) \biggr)^{-1}  \label{eqZetaFnOrbitForm_Subshift}.
\end{align}
\end{enumerate}
\end{prop}

\begin{proof}
We first claim that for each $\varphi \in \Holder{\alpha}(S^2,d)$, $P(\sigma_{A_{\ti}}, \varphi\circ\pi_{\ti}) = P(f,\varphi)$. Statement~(i) follows from this claim and Corollary~\ref{corS0unique} immediately.

Indeed, by Theorem~\ref{thmEquilibriumState}~(iii), we can choose $x\in S^2 \setminus \bigcup\limits_{i\in\N_0} f^{-i}(\CC)$. By Proposition~\ref{propTileSFT}, the map $\pi_{\ti}$ is H\"{o}lder continuous on $\Sigma_{A_{\ti}}^+$ and injective on $\pi_{\ti}^{-1}(B)$, where $B= \{x\} \cup \bigcup\limits_{i\in\N} f^{-i}(x)$. So we can consider $\pi_{\ti}^{-1}$ as a function from $B$ to $\pi_{\ti}^{-1}(B)$ in the calculation below. By Corollary~\ref{corRR^nConvToTopPressureUniform}, Proposition~\ref{propSFT}~(iv), and the fact that $\bigl( \Sigma_{A_{\ti}}^+, \sigma_{A_{\ti}} \bigr)$ is topologically mixing (see Proposition~\ref{propTileSFT}),
\begin{align*}
    P(\sigma_{A_{\ti}} , \varphi\circ\pi_{\ti}) 
= & \lim\limits_{n\to+\infty} \frac{1}{n} \log \sum\limits_{\underline{y}\in \sigma_{A_{\ti}}^{-n}\left(\pi_{\ti}^{-1}(x)\right)} \exp \bigl(S_n^{\sigma_{A_{\ti}}}  (\varphi\circ\pi_{\ti})(\underline{y})\bigr) \\
= & \lim\limits_{n\to+\infty} \frac{1}{n} \log \sum\limits_{\underline{y}\in  \pi_{\ti}^{-1}( f^{-n} (x))  }      \exp \bigl(S_n^{\sigma_{A_{\ti}}}  (\varphi\circ\pi_{\ti}) (\underline{y})\bigr) \\
= & \lim\limits_{n\to+\infty} \frac{1}{n} \log \sum\limits_{z\in    f^{-n} (x)   }              \exp \bigl(S_n^f  \varphi           (z) \bigr)  
=   P(f,\varphi).
\end{align*}
The claim is now established.

\smallskip
Next, we observe that by Corollary~\ref{corS0unique}, for each $a>s_0$,
\begin{equation}   \label{eqPfpropZetaFnConv_s0}
 P(\sigma_{A_{\ti}}, -a\phi\circ \pi_{\ti})  = P(f,-a\phi) < 0.
\end{equation}

By Theorem~\ref{thmETMBasicProperties}~(ii), Proposition~\ref{propTopPressureDefPeriodicPts}, and Proposition~\ref{propSFT}~(i) and (ii), 
we can apply Lemma~\ref{lmDynDirichletSeriesConv_general} and Remark~\ref{rmDynDirichletSeriesZetaFn} to establish statements~(ii) and (iii). 
\end{proof}

\section{The Assumptions}      \label{sctAssumptions}
We state below the hypotheses under which we will develop our theory in most parts of this paper. We will repeatedly refer to such assumptions in the later sections. We emphasize again that not all assumptions are assumed in all the statements in this paper.

\begin{assumptions}
\quad

\begin{enumerate}

\smallskip

\item $f\:S^2 \rightarrow S^2$ is an expanding Thurston map.

\smallskip

\item $\CC\subseteq S^2$ is a Jordan curve containing $\post f$ with the property that there exists $n_\CC\in\N$ such that $f^{n_\CC} (\CC)\subseteq \CC$ and $f^m(\CC)\nsubseteq \CC$ for each $m\in\{1,2,\dots,n_\CC-1\}$.

\smallskip

\item $d$ is a visual metric on $S^2$ for $f$ with expansion factor $\Lambda>1$ and a linear local connectivity constant $L\geq 1$.

\smallskip

\item $\alpha\in(0,1]$.

\smallskip

\item $\psi\in \Holder{\alpha}((S^2,d),\C)$ is a complex-valued H\"{o}lder continuous function with an exponent $\alpha$.

\smallskip

\item $\phi\in \Holder{\alpha}(S^2,d)$ is an eventually positive real-valued H\"{o}lder continuous function with an exponent $\alpha$, and $s_0\in\R$ is the unique positive real number satisfying $P(f, -s_0\phi)=0$.

\smallskip

\item $\mu_\phi$ is the unique equilibrium state for the  map $f$ and the potential $\phi$.

\end{enumerate}

\end{assumptions}

Note that the uniqueness of $s_0$ in (5) is not always true, but it is guaranteed, for example, if $\phi$ is strictly positive or if $\phi$ is eventually positive (see Corollary~\ref{corS0unique}). For a pair of $f$ in (1) and $\phi$ in (6), we will say that a quantity depends on $f$ and $\phi$ if it depends on $s_0$.

Observe that by Lemma~\ref{lmCexistsL}, for each $f$ in (1), there exists at least one Jordan curve $\CC$ that satisfies (2). Since for a fixed $f$, the number $n_\CC$ is uniquely determined by $\CC$ in (2), in the remaining part of the paper we will say that a quantity depends on $\CC$ even if it also depends on $n_\CC$.

Recall that the expansion factor $\Lambda$ of a visual metric $d$ on $S^2$ for $f$ is uniquely determined by $d$ and $f$. We will say that a quantity depends on $f$ and $d$ if it depends on $\Lambda$.

Note that even though the value of $L$ is not uniquely determined by the metric $d$, in the remainder of this paper, for each visual metric $d$ on $S^2$ for $f$, we will fix a choice of linear local connectivity constant $L$. We will say that a quantity depends on the visual metric $d$ without mentioning the dependence on $L$, even though if we had not fixed a choice of $L$, it would have depended on $L$ as well.

In the discussion below, depending on the conditions we will need, we will sometimes say ``Let $f$, $\CC$, $d$, $\psi$, $\alpha$ satisfy the Assumptions.'', and sometimes say ``Let $f$ and $d$ satisfy the Assumptions.'', etc.

\section{Dynamics on the invariant Jordan curve}  \label{sctDynOnC}

The main goal in Sections~\ref{sctSplitRuelleOp} thorough \ref{sctDolgopyat} is to establish Theorem~\ref{thmZetaAnalExt_SFT}, namely, a holomorphic extention of the dynamical zeta function with quantitative bounds for the one-sided subshift of finite type $\sigma_{A_{\ti}} \: \Sigma_{A_{\ti}}^+ \rightarrow \Sigma_{A_{\ti}}^+$ associated to some expanding Thurston map $f$ with some invariant Jordan curve $\CC\subseteq S^2$. One hopes to derive from Theorem~\ref{thmZetaAnalExt_SFT} similar results for the dynamical zeta function for $f$ itself (stated in Theorem~\ref{thmZetaAnalExt_InvC}). However, there is no one-to-one correspondance between the periodic points of $\sigma_{A_{\ti}}$ and those of $f$ through the factor map $\pi_{\ti} \: \Sigma_{A_{\ti}}^+ \rightarrow S^2$. A relation between the two dynamical zeta functions $\zeta_{f,\,\minus\phi}$ and $\zeta_{\sigma_{A_{\ti}},\,\minus\phi\circsmall\pi_{\ti}}$ can nevertheless be established through a careful investigation on the dynamics induced by $f$ on the Jordan curve $\CC$.

\subsection{Constructions}   \label{subsctDynOnC_Construction}

Let $f\: S^2\rightarrow S^2$ be an expanding Thurston map with a Jordan curve $\CC\subseteq S^2$ satisfying $f(\CC)\subseteq \CC$ and $\post f \subseteq \CC$.

Let $\bigl( \Sigma_{A_{\ti}}^+,\sigma_{A_{\ti}} \bigr)$ be the one-sided subshift of finite type associated to $f$ and $\CC$ defined in Proposition~\ref{propTileSFT}, and let $\pi_{\ti}\: \Sigma_{A_{\ti}}^+\rightarrow S^2$ be the factor map defined in (\ref{eqDefTileSFTFactorMap}).

We first construct two more one-sided subshifts of finite type that are related to the dynamics induced by $f$ on $\CC$.

Define the set of states $S_{\e} \coloneqq \bigl\{ e \in \E^1(f,\CC)   \,\big|\, e \subseteq \CC \bigr\}$, and the transition matrix $A_{\e}\: S_{\e} \times S_{\e}\rightarrow \{0,1\}$ by
\begin{equation}    \label{eqDefA|}
A_{\e} \left(  e_1, e_2 \right)  
= \begin{cases} 1 & \text{if } f\left(e_1\right)\supseteq e_2, \\ 0  & \text{otherwise}  \end{cases}
\end{equation}
for $e_1, e_2 \in S_{\e}$.

Define the set of states $S_{\ee} \coloneqq \bigl\{ (e,\c) \in \E^1(f,\CC) \times \{\b,\w\} \,\big|\, e \subseteq \CC \bigr\}$. For each $(e,\c)\in S_{\ee}$, we denote by $X^1(e,\c) \in \X^1(f,\CC)$ the unique $1$-tile satisfying 
\begin{equation}   \label{eqDefXec}
e \subseteq X^1(e,\c) \subseteq X^0_\c. 
\end{equation}
The existence and uniqueness of $X^1(f,\c)$ defined by (\ref{eqDefXec}) follows immediately from Proposition~\ref{propCellDecomp}~(iii), (v), and (vi) and the assumptions that $f(\CC)\subseteq \CC$ and $e\subseteq\CC$. We define the transition matrix $A_{\ee}\: S_{\ee} \times S_{\ee}\rightarrow \{0,1\}$ by
\begin{equation}   \label{eqDefA||}
A_{\ee} \left( \left(e_1,\c_1 \right), \left(e_2,\c_2\right)\right)  
= \begin{cases} 1 & \text{if } f\left(e_1\right)\supseteq e_2 \text{ and } f\left(X^1\left(e_1,\c_1\right) \right) \supseteq X^1(e_2, \c_2), \\ 0  & \text{otherwise}  \end{cases}
\end{equation}
for $\left(e_1,\c_1 \right), \left(e_2,\c_2\right) \in S_{\ee}$.

We will consider the one-sided subshift of finite type $\bigl( \Sigma_{A_{\e}}^+, \sigma_{A_{\e}} \bigr)$ defined by the transition matrix $A_{\e}$ and $\bigl( \Sigma_{A_{\ee}}^+, \sigma_{A_{\ee}} \bigr)$ defined by the transition matrix $A_{\ee}$, where  
\begin{align*}
\Sigma_{A_{\e}}^+ & =  \{  \{   e_i  \}_{i\in\N_0}  \,|\,  
             e_i \in S_{\e},\, A_{\e} (  e_i, e_{i+1} ) = 1,\, \text{for each } i\in\N_0  \},  \\
\Sigma_{A_{\ee}}^+& =  \{  \{  (e_i,\c_i )  \}_{i\in\N_0}  \,|\,  
            (e_i,\c_i)\in S_{\ee},\, A_{\ee} ( (e_i,\c_i ), (e_{i+1},\c_{i+1})) = 1,\, \text{for each } i\in\N_0  \},
\end{align*}
$\sigma_{A_{\e}}$ and $\sigma_{A_{\ee}}$ are the left-shift operators on $\Sigma_{A_{\e}}^+$ and $\Sigma_{A_{\ee}}^+$, respectively (see Subsection~\ref{subsctSFT}).

See Figure~\ref{figDynOnC23} for the sets of states $S_{\ee}$ and $S_{\e}$ associated to an expanding Thurston map $f$ and an invariant Jordan curve $\CC$ whose cell decomposition $\DD^1(f,\CC)$ of $S^2$ is sketched in Figure~\ref{figDynOnC1}. Note that $S_{\ti} = \X^1(f,\CC)$. In this example, $f$ has three postcritical points $A$, $B$, and $C$.

\begin{figure}
    \centering
    \begin{overpic}
    [width=6cm, 
    tics=20]{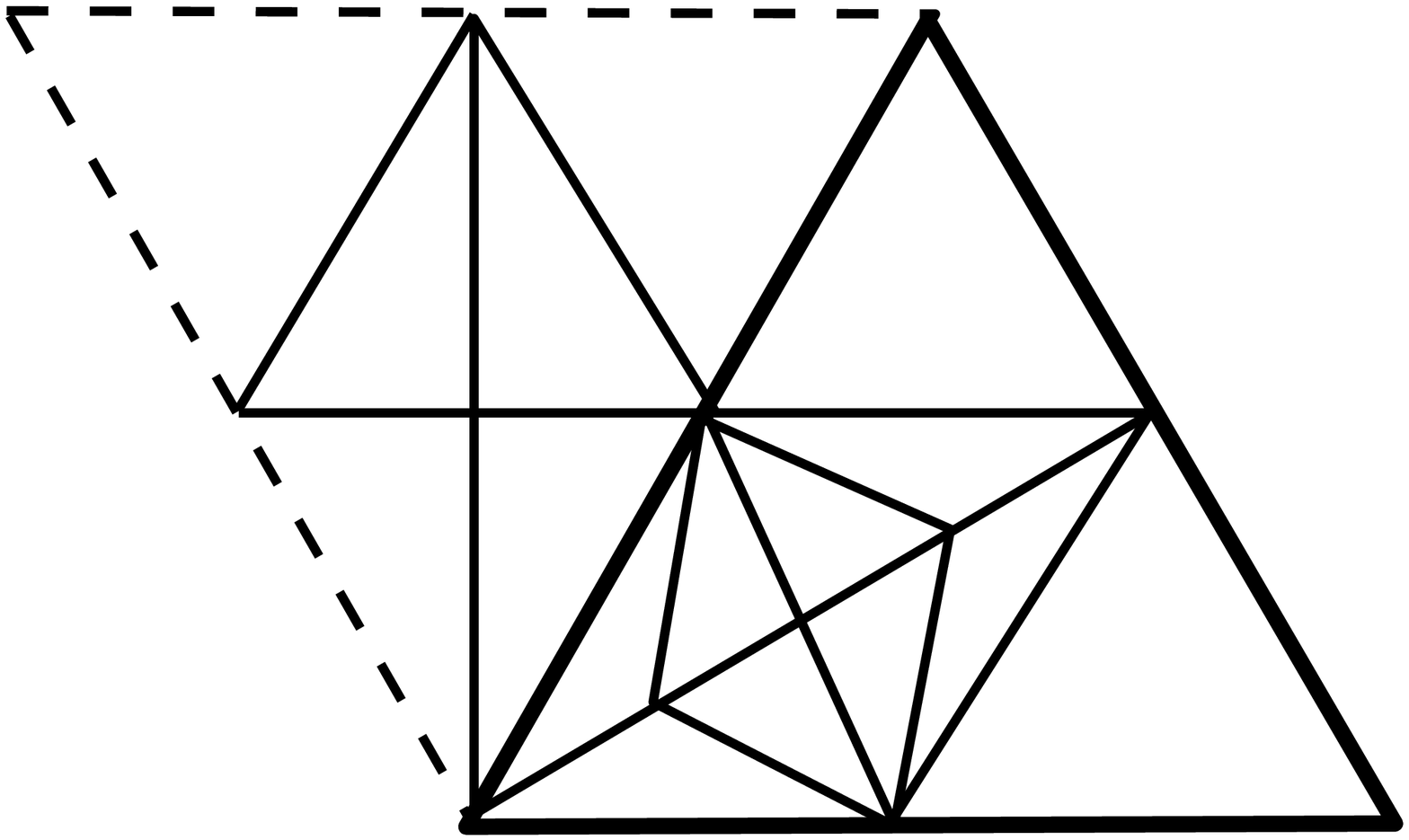}
    \put(109,101){$A$}
    \put(46,-8){$B$}
    \put(171,-8){$C$}
    \put(-9,100){$C$}
    \put(79,-10){$e^1_1$}
    \put(127,-10){$e^1_2$}
    \put(157,29){$e^1_3$}
    \put(130,78){$e^1_4$}
    \put(88,78){$e^1_5$}
    \put(64,36){$e^1_6$}
    \end{overpic}
    \caption{The cell decomposition $\DD^1(f,\CC)$. $S_{\ti} = \X^1(f,\CC)$.}
    \label{figDynOnC1}

    \centering
    \begin{overpic}
    [width=11cm, 
    tics=20]{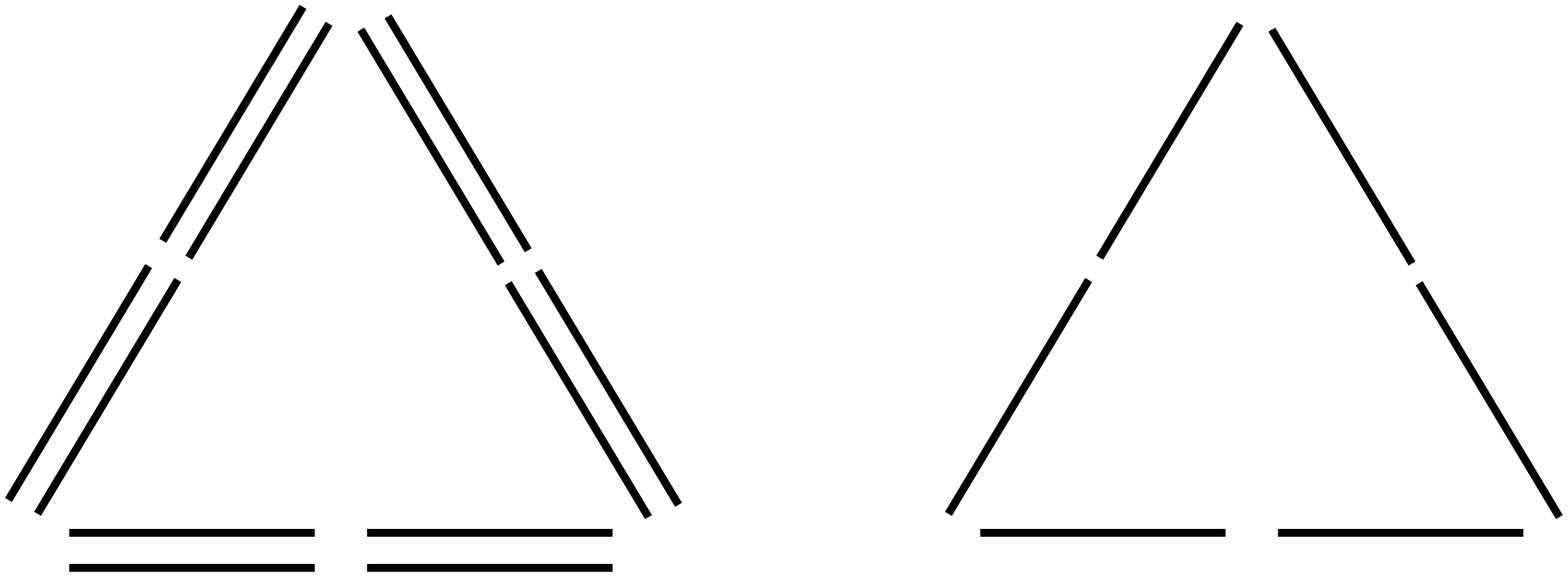}
    \put(24,-11){$(e^1_1,\b)$}
    \put(23, 14){$(e^1_1,\w)$}
    \put(83,-11){$(e^1_2,\b)$}
    \put(82, 14){$(e^1_2,\w)$}
    \put(125,36){$(e^1_3,\b)$}
    \put(77, 36){$(e^1_3,\w)$}
    \put(92,91){$(e^1_4,\b)$}
    \put(15,91){$(e^1_5,\b)$}
    \put(-18,36){$(e^1_6,\b)$}
    
    \put(215,-5){$e^1_1$}
    \put(273,-5){$e^1_2$}
    \put(299, 36){$e^1_3$}
    \put(268,90){$e^1_4$}
    \put(222,90){$e^1_5$}
    \put(190,36){$e^1_6$}
    \end{overpic}
    \caption{Elements in $S_{\ee}$ (left) and elements in $S_{\e}$ (right).}
    \label{figDynOnC23}

\end{figure}

\begin{prop}   \label{propSFTs_C}
Let $f$, $\CC$, $d$ satisfy the Assumptions. We assume in addition that $f(\CC)\subseteq\CC$. Let $\bigl( \Sigma_{A_{\ti}}^+,\sigma_{A_{\ti}} \bigr)$ be the one-sided subshift of finite type associated to $f$ and $\CC$ defined in Proposition~\ref{propTileSFT}, and let $\pi_{\ti}\: \Sigma_{A_{\ti}}^+\rightarrow S^2$ be the factor map defined in (\ref{eqDefTileSFTFactorMap}). Fix $\theta \in(0,1)$. Recall the one-sided subshifts of finite type $\bigl( \Sigma_{A_{\e}}^+, \sigma_{A_{\e}} \bigr)$ and $\bigl( \Sigma_{A_{\ee}}^+, \sigma_{A_{\ee}} \bigr)$, with the spaces $\Sigma_{A_{\e}}^+$ and $\Sigma_{A_{\ee}}^+$ equipped with the corresponding metrics $d_\theta$ constructed in Subsection~\ref{subsctSFT}. We write $\V(f,\CC) \coloneqq \bigcup\limits_{i\in\N_0} \V^i(f,\CC)$. 

Then the following statements are satisfied:

\begin{enumerate}
\smallskip
\item[(i)] $\bigl( \Sigma_{A_{\e}}^+, \sigma_{A_{\e}} \bigr)$ is a factor of $\bigl( \Sigma_{A_{\ee}}^+, \sigma_{A_{\ee}} \bigr)$ with a Lipschitz continuous factor map $\pi_{\ee} \:  \Sigma_{A_{\ee}}^+ \rightarrow  \Sigma_{A_{\e}}^+$ defined by
\begin{equation}  \label{eqDefPi||}
 \pi_{\ee} ( \{ (e_i,\c_i) \}_{i\in\N_0} ) = \{ e_i \}_{i\in\N_0}
\end{equation}
for $\{ (e_i,\c_i) \}_{i\in\N_0} \in \Sigma_{A_{\ee}}^+$. Moreover, for each $\{  e_i  \}_{i\in\N_0} \in \Sigma_{A_{\e}}^+$, we have 
\begin{equation*}
\card \left(\pi_{\ee}^{-1} ( \{  e_i  \}_{i\in\N_0} )\right) = 2.
\end{equation*}

\smallskip
\item[(ii)] $( \CC, f|_\CC)$ is a factor of $\bigl( \Sigma_{A_{\e}}^+, \sigma_{A_{\e}} \bigr)$ with a H\"{o}lder continuous factor map $\pi_{\e} \:  \Sigma_{A_{\e}}^+ \rightarrow  \CC$ defined by
\begin{equation}  \label{eqDefPi|}
 \pi_{\e} ( \{  e_i  \}_{i\in\N_0} ) = x, \text{ where } \{x\} = \bigcap\limits_{i \in \N_0} f^{-i}(e_i)
\end{equation}
for $\{  e_i  \}_{i\in\N_0} \in \Sigma_{A_{\e}}^+$. Moreover, for each $x\in\CC$, we have
\begin{equation}  \label{eqPi|Card}
\card \left(\pi_{\e}^{-1} (x)\right) = \begin{cases} 1 & \text{if } x\in \CC \setminus \V(f,\CC), \\ 2  & \text{if } x\in \CC\cap\V(f,\CC). \end{cases}
\end{equation}

\end{enumerate}
\end{prop}

Thus we have the following commutative diagram:
\begin{equation*}
\xymatrix{ \Sigma_{A_{\ee}}^+ \ar[d]_{\sigma_{A_{\ee}}} \ar[r]^{\pi_{\ee}}  & \Sigma_{A_{\e}}^+ \ar[d]^{\sigma_{A_{\e}}} \ar[r]^{\pi_{\e}}  & \CC \ar[d]^{f|_\CC}\\ 
           \Sigma_{A_{\ee}}^+                           \ar[r]_{\pi_{\ee}}  & \Sigma_{A_{\e}}^+                          \ar[r]_{\pi_{\e}}  & \CC.}
\end{equation*}

\begin{proof}
(i) It follows immediately from (\ref{eqDefPi||}), (\ref{eqDefA||}), (\ref{eqDefA|}), and the definitions of $\Sigma_{A_{\ee}}^+$ and $\Sigma_{A_{\e}}^+$ that $\pi_{\ee} \bigl( \Sigma_{A_{\ee}}^+ \bigr) \subseteq \Sigma_{A_{\e}}^+$. By (\ref{eqDefPi||}), it is clear that $\pi_{\ee}$ is Lipschitz continuous.

Next, we show that $\card \bigl(\pi_{\ee}^{-1}  (  \{ e_i \}_{i\in\N_0} ) \bigr) = 2$ for each $\{ e_i \}_{i\in\N_0} \in \Sigma_{A_{\e}}^+$. The fact that $\pi_{\ee}$ is surjective then follows for free.

Fix arbitrary $\c\in\{\b,\w\}$ and $\{ e_i \}_{i\in\N_0} \in \Sigma_{A_{\e}}^+$.

We recursively construct $\c_i \in \{\b,\w\}$ for each $i\in\N_0$ such that $\c_0 = \c$ and  $\{ (e_i,\c_i) \}_{i\in\N_0} \in \Sigma_{A_{\ee}}^+$, and prove that such a sequence $\{ \c_i \}_{i\in\N_0}$ is unique. Let $\c_0 \coloneqq \c$. Assume that for some $k\in\N_0$, $\c_j$ is determined and is unique for all integer $j\in\N_0$ with $j\leq k$, in the sense that any other choice of $\c_j$ for any $j\in\N_0$ with $j\leq k$ would result in $\{ (e_i,\c_i) \}_{i\in\N_0} \notin \Sigma_{A_{\ee}}^+$ regardless of choices of $\c_j$ for $j>k$. Recall $X^1( e_k, \c_k )$ defined in (\ref{eqDefXec}). Since $f(e_k)\supseteq e_{k+1}$ and $f\bigl(X^1( e_k, \c_k )\bigr)$ is the black $0$-tile $X^0_\b$ or the white $0$-tile $X^0_\w$ by Proposition~\ref{propCellDecomp}~(i), we will have to choose $\c_{k+1} \coloneqq \b$ in the former case and $\c_{k+1} \coloneqq \w$ in the latter case due to (\ref{eqDefA||}). Hence $\{ (e_i,\c_i) \}_{i\in\N_0} \in \pi_{\ee}^{-1}  (  \{  e_i \}_{i\in\N_0}  )$ is uniquely determined by $\{  e_i \}_{i\in\N_0}$ and $\c\in\{\b,\w\}$. This proves $\card \bigl(\pi_{\ee}^{-1}  (  \{ e_i \}_{i\in\N_0} ) \bigr) = 2$ for each $\{ e_i \}_{i\in\N_0} \in \Sigma_{A_{\e}}^+$.

Finally, it follows immediately from (\ref{eqDefPi||}) that $\pi_{\ee} \circ \sigma_{A_{\ee}} = \sigma_{A_{\e}} \circ \pi_{\ee}$.

\smallskip

(ii) Fix an arbitrary $\{e_i\}_{i\in\N_0} \in \Sigma_{A_{\e}}^+$.

Since $f(e_i)\supseteq e_{i+1}$ for each $i\in\N_0$, by Lemma~\ref{lmCylinderIsTile}, the map $\pi_{\e}$ is well-defined.

Note that for each $m\in \N_0$ and each $\{e'_i\}_{i\in\N_0} \in \Sigma_{A_{\e}}^+$ with $e_{m+1} \neq e'_{m+1}$ and $e_j=e'_j$ for each integer $j\in[0,m]$, we have $\{ \pi_{\e}  (\{e_i\}_{i\in\N_0}  ),  \pi_{\e}  ( \{ e'_i \}_{i\in\N_0}  ) \} \subseteq \bigcap\limits_{i=0}^{m} f^{-i}(e_i) \in \E^{m+1}$ by Lemma~\ref{lmCylinderIsTile}. Thus it follows from Lemma~\ref{lmCellBoundsBM}~(ii) that $\pi_{\e}$ is H\"{o}lder continuous.

To see that $\pi_{\e}$ is surjective, we observe that for each $x\in \CC$, we can find a sequence $\bigl\{ e^j(x) \bigr\}_{j\in\N}$ of edges such that $e^j(x) \in \E^j$, $e^j(x) \subseteq \CC$, $x\in e^j(x)$, and $e^j(x) \supseteq e^{j+1}(x)$ for each $j\in\N$. Then it is clear from Proposition~\ref{propCellDecomp}~(i) that $\bigl\{  f^i \bigl( e^{i+1}(x) \bigr) \bigr\}_{i\in\N_0} \in \Sigma_{A_{\e}}^+$ and $\pi_{\e} \Bigl(  \bigl\{ f^i \bigl( e^{i+1}(x) \bigr) \bigr\}_{i\in\N_0} \Bigr) = x$.

Next, to check that $\pi_{\e} \circ \sigma_{A_{\e}} = f \circ \pi_{\e}$, it suffices to observe that 
\begin{align*}
            \{ (f\circ \pi_{\e}) (\{e_i\}_{i\in\N_0}) \}
  =       & f \biggl(  \bigcap\limits_{j\in\N_0} f^{-j} (e_j) \biggr)
\subseteq        \bigcap\limits_{j\in\N} f^{-(j-1)} (e_j) \\
  =      &   \bigcap\limits_{i\in\N_0} f^{-i} (e_{i+1})  
  =         \{(\pi_{\e} \circ \sigma_{A_{\e}} ) (\{e_i\}_{i\in\N_0})  \}.
\end{align*}

Finally, we are going to establish (\ref{eqPi|Card}). Fix an arbitrary point $x\in\CC$.

\smallskip

\emph{Case~1.} $x\in \CC\setminus \V(f,\CC)$.

\smallskip
We argue by contradiction and assume that there exist distinct $\{e_i\}_{i\in\N_0}, \{e'_i\}_{i\in\N_0} \in \pi_{\e}^{-1}(x)$. Choose $m\in\N_0$ to be the smallest non-negative integer with $e_m\neq e'_m$. Then by Lemma~\ref{lmCylinderIsTile}, $x\in \bigcap\limits_{i=0}^m f^{-i} (e_i) \in \E^{m+1}$ and $x\in \bigcap\limits_{i=0}^m f^{-i} (e'_i) \in \E^{m+1}$. Thus $f^m(x) \in e_m \cap e'_m \subseteq \V^1$ by Proposition~\ref{propCellDecomp}~(v). This is a contradiction to the assumption that $x\in \CC\setminus \V(f,\CC)$. Hence $\card \bigl( \pi_{\e}^{-1} (x) \bigr) = 1$.

\smallskip

\emph{Case~2.} $x\in \CC\cap \V(f,\CC)$.

\smallskip

Denote $n \coloneqq \min \bigl\{ i\in\N \,\big|\, x\in\V^i \bigr\} \in \N$.

For each $j\in\N$ with $j<n$, we define $e^j_1,e^j_2 \in \E^j$ to be the unique $j$-edge with $x\in e^j_1=e^j_2 \subseteq \CC$. For each $i\in\N$ with $i\geq n$, we choose the unique pair $e^i_i,e^i_2 \in \E^i$ of $i$-edges satisfying (1) $e^i_1 \cup e^i_2 \subseteq \CC$, (2) $e^i_1 \cap e^i_2 = \{x\}$, and (3) if $i \geq 2$, then $e^i_1 \subseteq e^{i-1}_1$ and $e^i_2 \subseteq e^{i-1}_2$. Then it is clear from Proposition~\ref{propCellDecomp}~(i) that for each $k\in\{1,2\}$, $\bigl\{ f^i \bigl(e^{i+1}_k\bigr) \bigr\}_{i\in\N_0} \in \Sigma_{A_{\e}}^+$ and $\pi_{\e} \Bigl(  \bigl\{ f^i \bigl(e^{i+1}_k\bigr) \bigr\}_{i\in\N_0} \Bigr) = x$.

Note that if $n=1$, then $e^1_1 \neq e^1_2$. If $n\geq 2$, then $f^i(x) \notin \V^1$ and thus $f^i(x) \notin \crit f|_\CC$, for each $i\in\{0,1,\dots,n-2\}$. So $f^{n-1}$ is injective on a neighborhood of $x$ and consequently by Proposition~\ref{propCellDecomp}~(i), $f^{n-1} \bigl(e^n_1\bigr) \neq f^{n-1} \bigl(e^n_2\bigr)$. Hence $\card \bigl( \pi_{\e}^{-1} (x) \bigr) \geq 2$.

On the other hand, for each $\{ e_i \}_{i\in\N_0} \in \pi_{\e}^{-1} (x)$ and each $j\in\N_0$, (1) if $j<n-1$, then $e_j = f^j \bigl(e^{j+1}_1 \bigr) = f^j \bigl(e^{j+1}_2 \bigr)$ (since $e^{j+1}_1=e^{j+1}_2$) and $f^j(x)\in \inte \bigl( f^j  \bigl(e^{j+1}_1 \bigr) \bigr)$ (by the definition of $n$); (2) if $j=n-1$, then either $e_j= f^j  \bigl(e^{j+1}_1 \bigr)$ or $e_j= f^j  \bigl(e^{j+1}_2 \bigr)$ since $f^{j}(x) \in f^j  \bigl(e^{j+1}_1 \bigr) \cap f^j  \bigl(e^{j+1}_2 \bigr)$; (3) if $j\geq n$, then  $f^j(x) \in \V^0$ and consequently by Proposition~\ref{propCellDecomp}~(i) and (v), there exists exactly one $1$-edge $e_j\in\E^1$ such that $f^j(x) \in e_j$ and $f(e_{j-1}) \supseteq e_j$. Hence $\card \bigl( \pi_{\e}^{-1} (x) \bigr) \leq 2$. 

The identity (\ref{eqPi|Card}) is now established.
\end{proof}

\subsection{Combinatorics}   \label{subsctDynOnC_Combinatorics}

\begin{lemma}   \label{lmPeriodicPtsLocationCases}
Let $f$ and $\CC$ satisfy the Assumptions. We assume in addition that $f(\CC)\subseteq\CC$. For each $n\in\N$ and each $x\in S^2$ with $f^n(x)=x$, exactly one of the following statements holds:
\begin{enumerate}
\smallskip
\item[(i)] $x\in\inte (X^n)$ for some $n$-tile $X^n\in\X^n(f,\CC)$, where $X^n$ is either a black $n$-tile contained in the black $0$-tile $X^0_\b$ or a white $n$-tile contained in the white $0$-tile $X^0_\w$. Moreover, $x\notin \bigcup\limits_{i\in\N_0} \left( \bigcup \E^i(f,\CC) \right)$.

\smallskip
\item[(ii)] $x\in \inte (e^n)$ for some $n$-edge $e^n\in\E^n(f,\CC)$ satisfying $e^n \subseteq \CC$. Moreover, $x\notin \bigcup\limits_{i\in\N_0}  \V^i(f,\CC)$.

\smallskip
\item[(iii)] $x\in\post f$.
\end{enumerate} 
\end{lemma}

\begin{proof}
Fix $x\in S^2$ and $n\in\N$ with $f^n(x)=x$. It is easy to see that at most one of Cases~(i), (ii) and (iii) holds. By Proposition~\ref{propCellDecomp}~(iii) and (v), it is clear that exactly one of the following cases holds:
\begin{enumerate}
\smallskip
\item[(1)] $x\in\inte (X^n)$ for some $n$-tile $X^n\in\X^n$.

\smallskip
\item[(2)] $x\in \inte (e^n)$ for some $n$-edge $e^n\in\E^n$.

\smallskip
\item[(3)] $x\in \V^n$.
\end{enumerate}

Assume that case~(1) holds. We argue by contradiction and assume that there exist $j\in\N_0$ and $e\in\E^j$ such that $x\in e$. Then for $k \coloneqq \bigl\lceil\frac{j+1}{n}\bigr\rceil \in\N$, $x=f^{kn}(x)\in f^{kn}(e)\subseteq \CC$, contradicting with $x\in\inte(X^n)$. So $x\notin \bigcup\limits_{i\in\N_0} \left( \bigcup \E^i \right)$. By Lemma~\ref{lmAtLeast1}, the rest of statement~(i) holds. Hence statement~(i) holds in case~(1).

Assume that case~(2) holds. By Proposition~\ref{propCellDecomp}~(i), $x=f^n(x) \in \inte (e^0) \subseteq\CC$ where $e^0=f^n(e^n)\in \E^0$. Since $f(\CC)\subseteq \CC$, $\DD^n$ is a refinement of $\DD^0$ (see Definition~\ref{defrefine}). So we can choose an arbitrary $n$-edge $e^n_*\in\E^n$ contained in $e^0$ with $x\in e^n_*$. Since $x\notin \V^n$, we have $x\in \inte (e^n_*)$. By Definition~\ref{defcelldecomp}, $e^n = e^n_* \subseteq e^0\subseteq \CC$. To verify that $x\notin \bigcup\limits_{i\in\N_0}  \V^i$, we argue by contradiction and assume that there exists $j\in\N_0$ such that $x\in\V^j$. Then for $k \coloneqq \bigl\lceil\frac{j+1}{n}\bigr\rceil \in\N$, $x=f^{kn}(x)\in \V^0$, contradicting with $x\in\inte(e^n)$. Thus $x\notin \bigcup\limits_{i\in\N_0}  \V^i$. Hence statement~(ii) holds in case~(2).

Assume that case~(3) holds. By Proposition~\ref{propCellDecomp}~(i), $x = f^n(x) \subseteq \V^0 = \post f$. Hence statement~(iii) holds in case~(3).
\end{proof}

Let $f$ be an expanding Thurston map with an $f$-invariant Jordan curve $\CC$ containing $\post f$. We orient $\CC$ in such a way that the white $0$-tile lies on the left of $\CC$. Let $p\in \CC$ be a fixed point of $f$. We say that $f|_\CC$ \defn{preserves the orientation at $p$} (resp.\ \defn{reverses the orientation at $p$}) if there exists an open arc $l\subseteq\CC$ with $p\in l$ such that $f$ maps $l$ homeomorphically to $f(l)$ and $f|_\CC$ preserves (resp.\ reverses) the orientation on $l$. Note that it may happen that $f|_\CC$ neither preserves nor reverses the orientation at $p$, because $f|_\CC$ need not be a local homeomorphism near $p$, where it may behave like a ``folding map''.

\begin{theorem}   \label{thmNoPeriodPtsIdentity}
Let $f$ and $\CC$ satisfy the Assumptions. We assume in addition that $f(\CC)\subseteq\CC$. Let $\bigl(\Sigma_{A_{\ti}}^+,\sigma_{A_{\ti}}\bigr)$ be the one-sided subshift of finite type associated to $f$ and $\CC$ defined in Proposition~\ref{propTileSFT}, and let $\pi_{\ti}\: \Sigma_{A_{\ti}}^+\rightarrow S^2$ be the factor map defined in (\ref{eqDefTileSFTFactorMap}). Recall the one-sided subshifts of finite type $\bigl( \Sigma_{A_{\e}}^+, \sigma_{A_{\e}} \bigr)$ and $\bigl( \Sigma_{A_{\ee}}^+, \sigma_{A_{\ee}} \bigr)$ constructed in Subsection~\ref{subsctDynOnC_Construction}, and the factor maps $\pi_{\e}\: \Sigma_{A_{\e}}^+\rightarrow S^2$, $\pi_{\ee}\: \Sigma_{A_{\ee}}^+\rightarrow \Sigma_{A_{\e}}^+$ defined in Proposition~\ref{propSFTs_C}. We denote by $\left(\V^0, f|_{\V^0}\right)$ the dynamical system on $\V^0=\V^0(f,\CC) = \post f$ induced by $f|_{\V^0} \: \V^0\rightarrow \V^0$.

For each $y\in S^2$ and each $i\in\N$, we write
\begin{align*}
 M_{\po} (y,i)  \coloneqq & \card \bigl( P_{1, (f|_{\V^0})^i }       \cap                                  \{y\} \bigr),  \\
 M_{\e}  (y,i)  \coloneqq & \card \Bigl( P_{1, \sigma_{A_{\e}}^i }   \cap \pi_{\e}^{-1}                   (y)\Bigr),  \\
 M_{\ee} (y,i)  \coloneqq & \card \Bigl( P_{1, \sigma_{A_{\ee}}^i }  \cap (\pi_{\e} \circ \pi_{\ee})^{-1} (y)\Bigr),  \\
 M_{\ti} (y,i)  \coloneqq & \card \Bigl( P_{1, \sigma_{A_{\ti}}^i }  \cap \pi_{\ti}^{-1}                  (y)\Bigr). 
\end{align*}

Then for each $n\in\N$ and each $x\in P_{1,f^n}$, we have
\begin{equation}   \label{eqNoPeriodPtsIdentity}
M_{\ti} (x,n)  - M_{\ee} (x,n) + M_{\e} (x,n) + M_{\po} (x,n) = \deg_{f^n} (x).
\end{equation}
\end{theorem}

\begin{proof}
Fix an arbitrary integer $n\in\N$ and an arbitrary fixed point $x\in P_{1,f^n}$ of $f^n$.

We establish (\ref{eqNoPeriodPtsIdentity}) by verifying it in each of the three cases of Lemma~\ref{lmPeriodicPtsLocationCases} depending on the location of $x$.

\smallskip

\emph{Case~(i) of Lemma~\ref{lmPeriodicPtsLocationCases}}: $x\in \inte(X^n)$ for some $n$-tile $X^n \in \X^n$, where $X^n$ is either a black $n$-tile contained in the black $0$-tile $X^0_\b$ or a white $n$-tile contained in the white $0$-tile $X^0_\w$. Moreover, $x\notin \bigcup\limits_{i\in\N_0} \bigl( \bigcup \E^i \bigr) = \bigcup\limits_{i\in\N_0} f^{-i} (\CC)$ (see Proposition~\ref{propCellDecomp}~(iii)). 

Thus by Proposition~\ref{propTileSFT}, $\card \bigl(\pi_{\ti}^{-1}(x) \bigr) = 1$. For each $i\in\N_0$, we denote by $X^i(x)\in\X^i$ the unique $i$-tile containing $x$. Fix an arbitrary integer $j\in \N_0$. Then $f^j\bigl(X^{j+1}(x)\bigr) \in \X^1$ (see Propostion~\ref{propCellDecomp}~(i)) and $X^{j+1}(x)\subseteq X^j(x)$. Thus $f \bigl( f^j\bigl(X^{j+1}(x)\bigr) \bigr) \supseteq f^{j+1}\bigl(X^{j+2}(x)\bigr)$. It follows from Lemma~\ref{lmCylinderIsTile} and (\ref{eqDefTileSFTFactorMap}) that $\pi_{\ti}^{-1} (x) = \Bigl\{  \bigl\{ f^i\bigl(X^{i+1}(x)\bigr) \bigr\}_{i\in\N_0} \Bigr\} \subseteq \Sigma_{A_{\ti}}^+$. Observe that $f^j\bigl(X^{j+1}(x)\bigr)$ is the unique $1$-tile containing $f^j(x)$, and that $f^{j+n}\bigl(X^{j+n+1}(x)\bigr)$ is the unique $1$-tile containing $f^{j+n}(x)$. Since $f^n(x)=x$, we can conclude from Definition~\ref{defcelldecomp} that $f^{j+n}\bigl(X^{j+n+1}(x)\bigr) = f^j\bigl(X^{j+1}(x)\bigr)$. Hence $\bigl\{ f^i \bigl( X^{i+1}(x) \bigr) \bigr\}_{i\in\N_0} \in P_{1,\sigma_{A_{\ti}}^n}$ and $M_{\ti} = 1$. On the other hand, since $x\notin\CC$, we have $M_{\ee} (x,n) = M_{\e} (x,n) = M_{\po} (x,n) = 0$ by Proposition~\ref{propSFTs_C}. Since $x\in \inte (X^n)$, we have $\deg_{f^n} (x) = 1$. This establishes the identity (\ref{eqNoPeriodPtsIdentity}) in Case~(i) of Lemma~\ref{lmPeriodicPtsLocationCases}.

\smallskip

\emph{Case~(ii) of Lemma~\ref{lmPeriodicPtsLocationCases}}: $x \in \inte(e^n)$ for some $n$-edge $e^n\in\E^n$ with $e^n \subseteq \CC$. Moreover, $x\notin \bigcup\limits_{i\in\N_0} \V^i$. So $\deg_{f^n}(x) = 1$ and $M_{\po} (x,n) = 0$.

We will establish (\ref{eqNoPeriodPtsIdentity}) in this case by proving the following two claims.

\smallskip

\emph{Claim~1.} $M_{\e}(x,n) = 1$.

\smallskip

Since $\card \bigl( \pi_{\e}^{-1} (x) \bigr) = 1$ by Proposition~\ref{propSFTs_C}~(ii), it suffices to show that $\sigma_{A_{\e}}^n \bigl( \pi_{\e}^{-1} (x) \bigr) =  \pi_{\e}^{-1} (x)$. For each $y\in \CC \setminus \bigcup\limits_{i\in\N_0} \V^i$ and $i\in\N_0$, we denote by $e^i(y) \in \E^i$ to be the unique $i$-edge containing $y$. Fix an arbitrary integer $j\in \N_0$. Then $f^j\bigl(e^{j+1}(x)\bigr) \in\E^1$ (see Proposition~\ref{propCellDecomp}~(i)) and $e^{j+1}(x) \subseteq e^j(x)$. Thus $f \bigl( f^j \bigl( e^{j+1}(x) \bigr) \bigr) \supseteq f^{j+1} \bigl( e^{j+2}(x) \bigr)$. It follows from Lemma~\ref{lmCylinderIsTile} and (\ref{eqDefPi|}) that $\pi_{\e}^{-1} (x) = \Bigl\{ \bigl\{ f^i \bigl( e^{i+1}(x) \bigr) \bigr\}_{i\in\N_0} \Bigr\} \subseteq \Sigma_{A_{\e}}^+$. Observe that $f^j\bigl(e^{j+1}(x)\bigr)$ is the unique $1$-edge containing $f^j(x)$, and that $f^{j+n}\bigl(e^{j+n+1}(x)\bigr)$ is the unique $1$-edge containing $f^{j+n}(x)$. Since $f^n(x)=x$, we can conclude from Definition~\ref{defcelldecomp} that $f^{j+n}\bigl(e^{j+n+1}(x)\bigr) = f^j\bigl(e^{j+1}(x)\bigr)$. Hence $\bigl\{ f^i \bigl( e^{i+1}(x) \bigr) \bigr\}_{i\in\N_0} \in P_{1,\sigma_{A_{\e}}^n}$ and $M_{\e}(x,n) = 1$, proving Claim~1.

\smallskip

\emph{Claim~2.} $M_{\ee}(x,n) = M_{\ti}(x,n)$.

\smallskip

We prove this claim by constructing a bijection $h\: \pi_{\ti}^{-1} (x) \rightarrow (\pi_{\e}\circ\pi_{\ee})^{-1} (x)$ explicitly and show that $h(\underline{z})\in P_{1,\sigma_{A_{\ee}}^n}$ if and only if $\underline{z}\in P_{1,\sigma_{A_{\ti}}^n}$.

For each $y\in \CC \setminus \bigcup\limits_{i\in\N_0} \V^i$, each $\c\in\{\b,\w\}$, and each $i\in\N_0$, we denote by $X^{\c,i}(y) \in \X^i$ the unique $i$-tile satisfying $y\in X^{\c,i}(y)$ and $X^{\c,i}(y) \subseteq X^0_\c$. Here $X^0_\b$ (resp.\ $X^0_\w$) is the unique black (resp.\ white) $0$-tile. Recall that as defined above, $e^i(x) \in \E^i$ is the unique $i$-edge containing $x$, for $i\in\N_0$. Then for each $\c\in\{\b,\w\}$ and each $i\in\N_0$, we have $e^i(x)\subseteq X^{\c,i}(x)$ (see Definition~\ref{defcelldecomp}), $f^i\bigl( X^{\c,i+1}(x) \bigr) \in \X^1$ (see Proposition~\ref{propCellDecomp}~(i)), and $X^{\c,i+1}(x) \subseteq X^{\c,i}(x)$. Thus 
\begin{equation}   \label{eqPfthmNoPeriodPtsIdentity_X}
f \bigl( f^i \bigl( X^{\c,i+1}(x) \bigr) \bigr) \supseteq f^{i+1} \bigl( X^{\c,i+2}(x) \bigr).
\end{equation}
It follows from Lemma~\ref{lmCylinderIsTile} and (\ref{eqDefTileSFTFactorMap}) that $\bigl\{ f^i \bigl( X^{\c,i+1}(x) \bigr) \bigr\}_{i\in\N_0} \in \Sigma_{A_{\ti}}^+$ and 
\begin{equation*} 
\pi_{\ti} \Bigl( \bigl\{ f^i \bigl( X^{\c,i+1}(x) \bigr) \bigr\}_{i\in\N_0} \Bigr) = x.
\end{equation*}

Next, we show that $\card \bigl( \pi_{\ti}^{-1}(x) \bigr) = 2$. We argue by contradiction and assume that $\card \bigl( \pi_{\ti}^{-1}(x) \bigr) \geq 3$. We choose $\{X_i\}_{i\in\N_0} \in \pi_{\ti}^{-1}(x)$ different from $\bigl\{ f^i \bigl( X^{\b,i+1}(x) \bigr) \bigr\}_{i\in\N_0}$ and $\bigl\{ f^i \bigl( X^{\w,i+1}(x) \bigr) \bigr\}_{i\in\N_0}$. Since $x\in \CC \setminus \bigcup\limits_{i\in\N_0} \V^i$, for each $j\in\N_0$, there exist exactly two $1$-tiles containing $f^j(x)$, namely, $X^{\b,1}\bigl(f^j(x)\bigr)$ and $X^{\w,1}\bigl(f^j(x)\bigr)$. Since $f^j (x) \in X_j$ for each $j\in\N_0$ (see (\ref{eqDefTileSFTFactorMap})), we get that there exists an integer $k\in\N_0$ and distinct $\c_1,\c_2\in\{\b,\w\}$ such that $X_k = f^k\bigl( X^{\c_1,k+1}(x) \bigr)$ and $X_{k+1} = f^{k+1}\bigl( X^{\c_2,k+2}(x) \bigr)$. Since $X^{\c_2,k+2}(x) \subseteq X^{\c_2,k+1}(x)$, $X^{\c_2,k+2}(x) \nsubseteq X^{\c_1,k+1}(x)$, and $f^{k+1}$ is injective on $\inte \bigl(  X^{\c_1,k+1}(x) \bigr) \cup \inte \bigl(  X^{\c_2,k+1}(x) \bigr)$, we get 
\begin{equation*}
f(X_k) = f^{k+1} \bigl(  X^{\c_1,k+1}(x) \bigr) \nsupseteq f^{k+1} \bigl(  X^{\c_2,k+2}(x) \bigr) = X_{k+1}.
\end{equation*}
This is a contradiction. Hence $\card \bigl( \pi_{\ti}^{-1}(x) \bigr) = 2$.

We define $h\: \pi_{\ti}^{-1}(x) \rightarrow (\pi_{\e} \circ \pi_{\ee})^{-1}(x)$ by
\begin{equation}   \label{eqPfthmNoPeriodPtsIdentity_Def_h}
   h \Bigl( \bigl\{ f^i \bigl( X^{\c,i+1}(x) \bigr) \bigr\}_{i\in\N_0} \Bigr) 
= \bigl\{ \bigl( f^i \bigl( e^{i+1} (x) \bigr), \c_i (\c)   \bigr)  \bigr\}_{i\in\N_0},   \qquad \c\in\{\b,\w\},
\end{equation}
where $\c_i (\c) \in \{\b,\w\}$ is the unique element in $\{\b,\w\}$ with the property that
\begin{equation}   \label{eqPfthmNoPeriodPtsIdentity_h2}
f^i \bigl(  X^{\c,i+1} (x) \bigr)  \subseteq X_{\c_i (\c)}^0
\end{equation}
for $i\in\N_0$. 

We first verify that $\bigl\{ \bigl( f^i \bigl( e^{i+1} (x) \bigr), \c_i (\c)   \bigr)  \bigr\}_{i\in\N_0} \in \Sigma_{A_{\ee}}^+$ for each $\c\in\{\b,\w\}$. Fix arbitrary $\c\in\{\b,\w\}$ and $j\in\N_0$. Since $e^{j+2}(x) \subseteq e^{j+1}(x) \subseteq \CC$, we get $f^j\bigl( e^{j+1}(x) \bigr) \subseteq \CC$ and 
\begin{equation}   \label{eqPfthmNoPeriodPtsIdentity_h1}
f \bigl( f^j \bigl( e^{j+1}(x) \bigr) \bigr) \supseteq f^{j+1} \bigl( e^{j+2}(x) \bigr).
\end{equation}
Recall that $X^1 \bigl(  f^j \bigl( e^{j+1} (x) \bigr), \c_j (\c) \bigr) \in \X^1$ denotes the unique $1$-tile satisfying 
\begin{equation*}
f^j \bigl( e^{j+1} (x) \bigr)    \subseteq     X^1 \bigl(  f^j \bigl( e^{j+1} (x) \bigr), \c_j (\c) \bigr)    \subseteq     X_{\c_j (\c)}^0
\end{equation*}
(see Proposition~\ref{propCellDecomp}~(iii), (v), and (vi) for its existence and uniqueness). Then by (\ref{eqPfthmNoPeriodPtsIdentity_h2}) and the fact that $e^{j+1}(x) \subseteq X^{\c,j+1} (x)$ (see Definition~\ref{defcelldecomp}), we get
\begin{equation}   \label{eqPfthmNoPeriodPtsIdentity_X=X}
   X^1 \bigl(  f^j \bigl( e^{j+1} (x) \bigr), \c_j (\c) \bigr)
=  f^j \bigl( X^{\c,j+1} (x) \bigr).
\end{equation}
Then by (\ref{eqDefA||}), (\ref{eqPfthmNoPeriodPtsIdentity_X=X}), (\ref{eqPfthmNoPeriodPtsIdentity_X}), and (\ref{eqPfthmNoPeriodPtsIdentity_h1}), $\bigl\{ \bigl( f^i \bigl( e^{i+1} (x) \bigr), \c_i (\c)   \bigr)  \bigr\}_{i\in\N_0} \in \Sigma_{A_{\ee}}^+$.

Note that since $X^{\c,1} (x) \subseteq X^0_\c$, we get $\c_1 (\c) = \c$ for $\c\in\{\b,\w\}$ from (\ref{eqPfthmNoPeriodPtsIdentity_h2}). Thus
\begin{equation*}
     \bigl\{ \bigl( f^i \bigl( e^{i+1} (x) \bigr), \c_i (\b)   \bigr)  \bigr\}_{i\in\N_0}
\neq \bigl\{ \bigl( f^i \bigl( e^{i+1} (x) \bigr), \c_i (\w)   \bigr)  \bigr\}_{i\in\N_0},
\end{equation*}
i.e., $h$ is injective. By Proposition~\ref{propSFTs_C}~(i) and (ii), $\card \bigl( (\pi_{\e} \circ \pi_{\ee} )^{-1} ( x ) \bigr) = 2$. Thus $h$ is a bijection.

It suffices now to show that for each $\underline{z} \in \pi_{\ti}^{-1}(x)$,  $h(\underline{z}) \in P_{1,\sigma_{A_{\ee}}^n }$ if and only if $\underline{z} \in P_{1,\sigma_{A_{\ti}}^n }$. Note that $e^i(x) \subseteq \CC$ for all $i\in\N_0$. Fix arbitrary $\c\in\{\b,\w\}$ and $i\in\N$. Note that since $f^i(x) \in f^i \bigl( e^{i+1}(x) \bigr)$, $f^i(x) = f^{i+n}(x) \in f^{i+n} \bigl( e^{i+n+1}(x) \bigr)$, and $f^i(x) \notin \bigcup\limits_{i\in\N_0} \V^i$, we have
\begin{equation}  \label{eqPfthmNoPeriodPtsIdentity_Periodic}
\bigl( f^i \bigl( e^{i+1} (x) \bigr), \c_i (\c) \bigr) = \bigl( f^{i+n} \bigl( e^{i+n+1} (x) \bigr), \c_{i+n} (\c) \bigr)
\end{equation}
if and only if $X^1 \bigl(  f^i \bigl( e^{i+1} (x) \bigr), \c_i (\c) \bigr) = X^1  \bigl( f^{i+n} \bigl( e^{i+n+1} (x) \bigr), \c_{i+n} (\c) \bigr)$. Thus by (\ref{eqPfthmNoPeriodPtsIdentity_X=X}), we get that (\ref{eqPfthmNoPeriodPtsIdentity_Periodic}) holds if and only if $f^i \bigl( X^{\c,i+1} (x) \bigr) = f^{i+n} \bigl( X^{\c,i+n+1} (x) \bigr)$. Hence by (\ref{eqPfthmNoPeriodPtsIdentity_Def_h}), $h(\underline{z}) \in P_{1,\sigma_{A_{\ee}}^n }$ if and only if $\underline{z} \in P_{1,\sigma_{A_{\ti}}^n }$, for each $\underline{z} \in \pi_{\ti}^{-1}(x)$.

Claim~2 is now established. Therefore (\ref{eqNoPeriodPtsIdentity}) holds in Case~(ii) of Lemma~\ref{lmPeriodicPtsLocationCases}.

\smallskip

\emph{Case~(iii) of Lemma~\ref{lmPeriodicPtsLocationCases}}: $x\in\post f$.

We will establish (\ref{eqNoPeriodPtsIdentity}) in this case by verifying it in each of the following subcases.

\begin{figure}
    \centering
    \begin{overpic}
    [width=6cm, 
    tics=20]{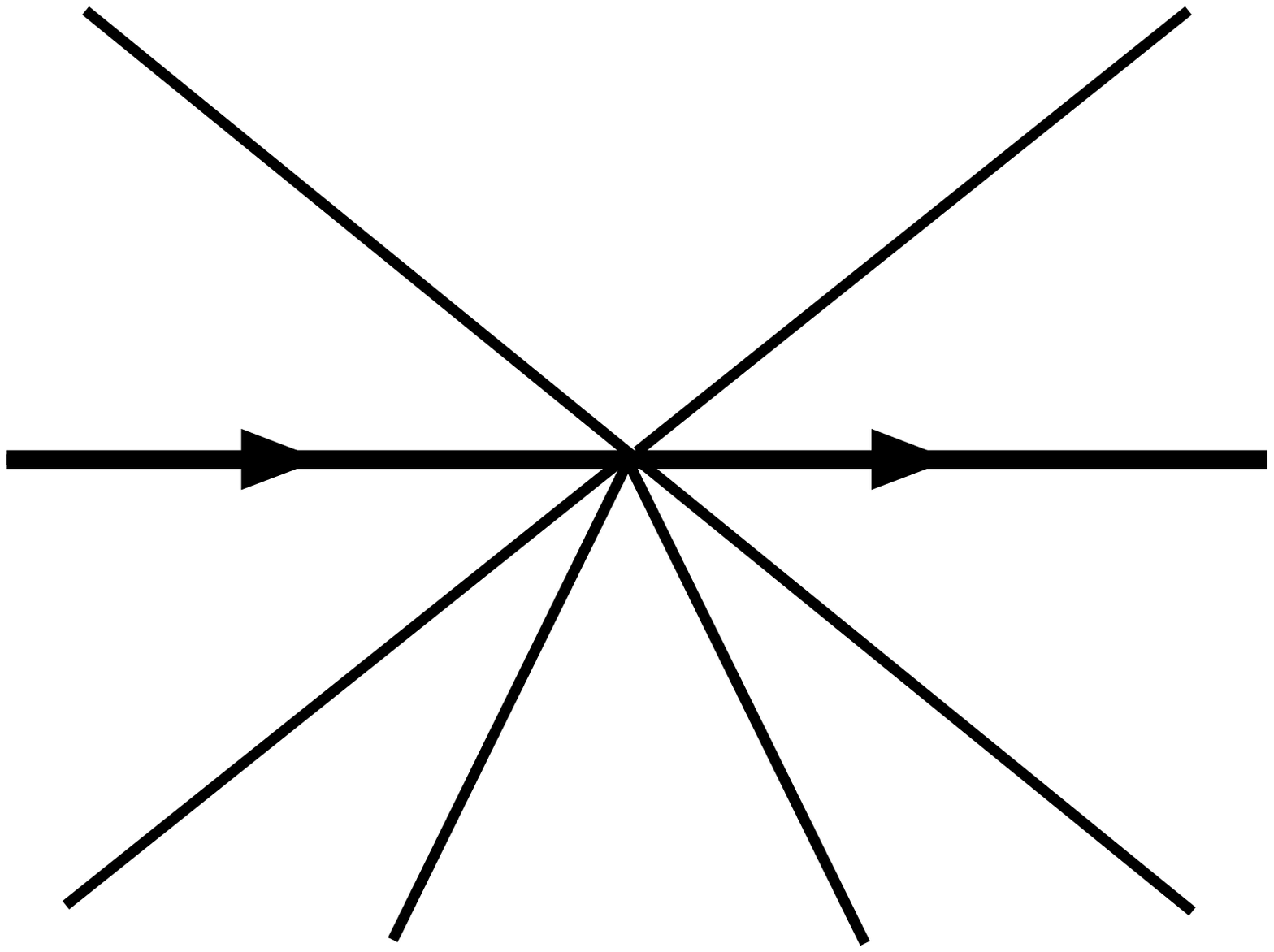}
    \put(75,100){$\b\w$}
    \put(25,80){$\w\w$}
    \put(130,80){$\w\w$}
    \put(25,33){$\b\b$}
    \put(40,20){$\w\b$}
    \put(79,12){$\b\b$}
    \put(112,20){$\w\b$}
    \put(135,33){$\b\b$}
    \put(20,57){$e_1$}
    \put(140,57){$e_2$}
    \put(80,50){$x$}
    \put(173,60){$\CC$}
    \put(163,80){$X^0_\w$}
    \put(163,40){$X^0_\b$}
    \end{overpic}
    \caption{Subcase (2)(a) where $f^n(e_1)\supseteq e_1$ and $f^n(e_2)\supseteq e_2$. $k=2$, $l=3$.}
    \label{figPlota}

    \centering
    \begin{overpic}
    [width=6cm, 
    tics=20]{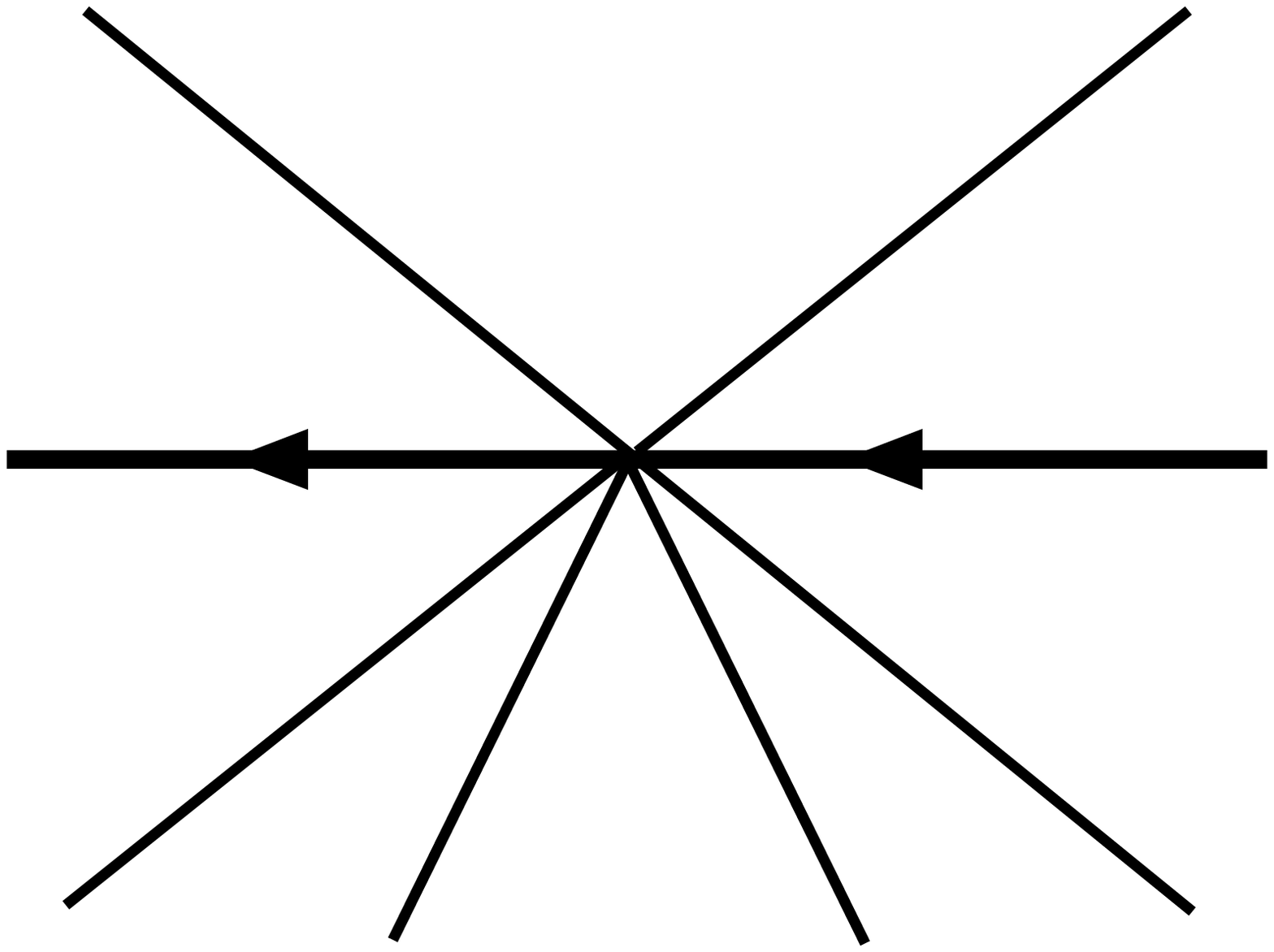}
    \put(75,100){$\w\w$}
    \put(25,80){$\b\w$}
    \put(130,80){$\b\w$}
    \put(25,33){$\w\b$}
    \put(40,20){$\b\b$}
    \put(79,12){$\w\b$}
    \put(112,20){$\b\b$}
    \put(135,33){$\w\b$}
    \put(20,57){$e_1$}
    \put(140,57){$e_2$}
    \put(80,50){$x$}
    \put(173,60){$\CC$}
    \put(163,80){$X^0_\w$}
    \put(163,40){$X^0_\b$}
    \end{overpic}
    \caption{Subcase (2)(b) where $f^n(e_1)\supseteq e_2$ and $f^n(e_2)\supseteq e_1$. $k=2$, $l=3$.}
    \label{figPlotb}

    \centering
    \begin{overpic}
    [width=6cm, 
    tics=20]{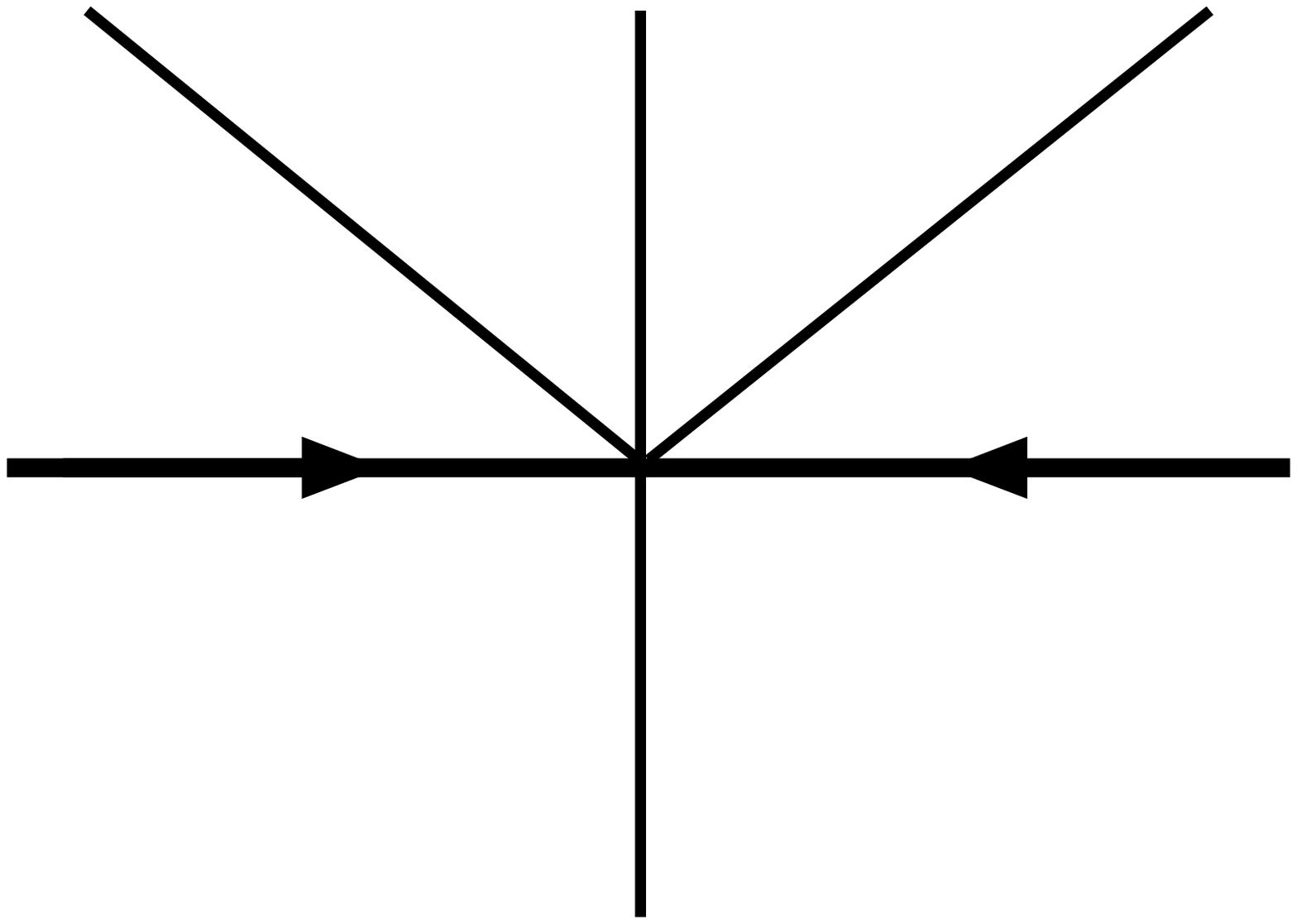}
    \put(55,100){$\b\w$}
    \put(95,100){$\w\w$}
    \put(25,70){$\w\w$}
    \put(130,70){$\b\w$}
    \put(37,22){$\b\b$}
    \put(117,22){$\w\b$}
    \put(25,48){$e_1$}
    \put(140,48){$e_2$}
    \put(75,50){$x$}
    \put(173,57){$\CC$}
    \put(163,77){$X^0_\w$}
    \put(163,37){$X^0_\b$}
    \end{overpic}
    \caption{Subcase (2)(c) where $f^n(e_1)=f^n(e_2)\supseteq e_1$. $k=2$, $l=1$.}
    \label{figPlotc}

    \centering
    \begin{overpic}
    [width=6cm, 
    tics=20]{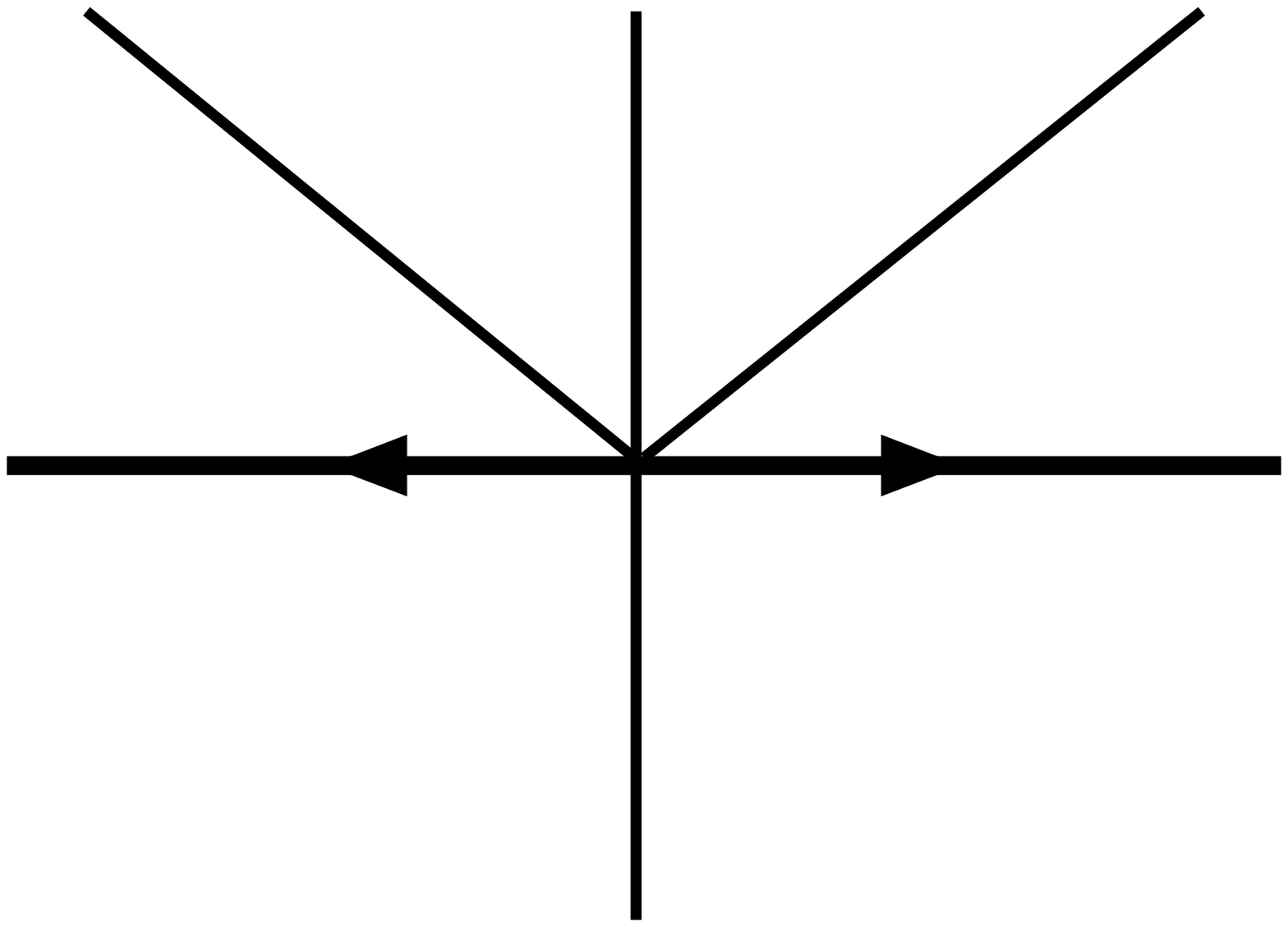}
    \put(55,100){$\w\w$}
    \put(95,100){$\b\w$}
    \put(25,70){$\b\w$}
    \put(130,70){$\w\w$}
    \put(37,22){$\w\b$}
    \put(117,22){$\b\b$}
    \put(25,48){$e_1$}
    \put(140,48){$e_2$}
    \put(75,50){$x$}
    \put(173,57){$\CC$}
    \put(163,77){$X^0_\w$}
    \put(163,37){$X^0_\b$}
    \end{overpic}
    \caption{Subcase (2)(d) where $f^n(e_1)=f^n(e_2)\supseteq e_2$. $k=2$, $l=1$.}
    \label{figPlotd}
\end{figure}

\begin{enumerate}

\smallskip
\item[(1)] If $x \notin \crit f^n$, then $(f^n)|_\CC$ either preserves or reverses the orientation at $x$ and the point $x$ is contained in exactly one white $n$-tile $X^n_\w$ and one black $n$-tile $X^n_\b$.
\begin{enumerate}

\smallskip
\item[(a)] If $(f^n)|_\CC$ preserves the orientation at $x$, then $X^n_\w \subseteq X^0_\w$ and $X^n_\b \subseteq X^0_\b$.

In this subcase, $M_{\ti} = 2$, $M_{\ee} = 4$, $M_{\e} = 2$, $M_{\po} = 1$, and $\deg_{f^n}(x) = 1$.

\smallskip
\item[(b)] If $(f^n)|_\CC$ reverses the orientation at $x$, then $X^n_\w \subseteq X^0_\b$ and $X^n_\b \subseteq X^0_\w$.

In this subcase, $M_{\ti} = 0$, $M_{\ee} = 0$, $M_{\e} = 0$, $M_{\po} = 1$, and $\deg_{f^n}(x) = 1$.
\end{enumerate}

\smallskip
\item[(2)] If $x \in \crit f^n$, then $x=f^n(x)\in\post f$ and so there are two distinct $n$-edges $e_1,e_2\subseteq\CC$ such that $\{x\}=e_1\cap e_2$. We refer to Figures \ref{figPlota} to \ref{figPlotd}.
\begin{enumerate}

\smallskip
\item[(a)] If $e_1\subseteq f^n(e_1)$ and $e_2\subseteq f^n(e_2)$, then $x$ is contained in exactly $k$ white and $k-1$ black $n$-tiles that are contained in the white $0$-tile, as well as in exactly $l-1$ white and $l$ black $n$-tiles that are contained in the black $0$-tile, for some $k,l\in\N$ with $k+l-1=\deg_{f^n}(x)$. Note that in this case $(f^n)|_\CC$ preserves the orientation at $x$.

In this subcase, $M_{\ti} = k+l$, $M_{\ee} = 4$, $M_{\e} = 2$, $M_{\po} = 1$, and $\deg_{f^n}(x) = k+l-1$.

\smallskip
\item[(b)] If $e_2\subseteq f^n(e_1)$ and $e_1\subseteq f^n(e_2)$, then $x$ is contained in exactly $k-1$ white and $k$ black $n$-tiles that are contained in the white $0$-tile, as well as in exactly $l$ white and $l-1$ black $n$-tiles that are contained in the black $0$-tile, for some $k,l\in\N$ with $k+l-1=\deg_{f^n}(x)$. Note that in this case $(f^n)|_\CC$ reverses the orientation at $x$.

In this subcase, $M_{\ti} = k+l-2$, $M_{\ee} = 0$, $M_{\e} = 0$, $M_{\po} = 1$, and $\deg_{f^n}(x) = k+l-1$.

\smallskip
\item[(c)] If $e_1\subseteq f^n(e_1)= f^n(e_2)$, then $x$ is contained in exactly $k$ white and $k$ black $n$-tiles that are contained in the white $0$-tile, as well as in exactly $l$ white and $l$ black $n$-tiles that are contained in the black $0$-tile, for some $k,l\in\N$ with $k+l=\deg_{f^n}(x)$. Note that in this case $(f^n)|_\CC$ neither preserves nor reverses the orientation at $x$.

In this subcase, $M_{\ti} = k+l$, $M_{\ee} = 2$, $M_{\e} = 1$, $M_{\po} = 1$, and $\deg_{f^n}(x) = k+l$.

\smallskip
\item[(d)] If $e_2\subseteq f^n(e_1)= f^n(e_2)$, then $x$ is contained in exactly $k$ white and $k$ black $n$-tiles that are contained in the white $0$-tile, as well as in exactly $l$ white and $l$ black $n$-tiles that are contained in the black $0$-tile, for some $k,l\in\N$ with $k+l=\deg_{f^n}(x)$. Note that in this case $(f^n)|_\CC$ neither preserves nor reverses the orientation at $x$.

In this subcase, $M_{\ti} = k+l$, $M_{\ee} = 2$, $M_{\e} = 1$, $M_{\po} = 1$, and $\deg_{f^n}(x) = k+l$.
\end{enumerate}
\end{enumerate}

This finishes the verification of (\ref{eqNoPeriodPtsIdentity}) in Case~(iii) of Lemma~\ref{lmPeriodicPtsLocationCases}.

\smallskip 

The proof of the theorem is now complete.
\end{proof}

Since all periodic points of $\bigl( \Sigma_{A_{\e}}^+, \sigma_{A_{\e}} \bigr)$ and $\bigl( \Sigma_{A_{\ee}}^+, \sigma_{A_{\ee}} \bigr)$ are mapped to periodic points of $f$ by the corresponding factor maps, we can write the dynamical Dirichlet series $\DS_{f,\,\minus\phi,\,\deg_f} (s)$ formally as a combination of products and quotient of the dynamical zeta functions for $\bigl( \Sigma_{A_{\ti}}^+, \sigma_{A_{\ti}} \bigr)$, $\bigl( \Sigma_{A_{\ee}}^+, \sigma_{A_{\ee}} \bigr)$, $\bigl( \Sigma_{A_{\e}}^+, \sigma_{A_{\e}} \bigr)$, and $\left(\V^0, f|_{\V^0}\right)$. In order to deduce Theorem~\ref{thmZetaAnalExt_InvC} from Theorem~\ref{thmZetaAnalExt_SFT}, we will need to verify that the zeta functions for the last three systems converge on an open half-plane on $\C$ containing $\{ s\in\C \,|\, \Re(s) \geq s_0 \}$.

\subsection{Calculation of topological pressure}  \label{subsctDynOnC_TopPressure}

Let $f\:S^2\rightarrow S^2$ be an expanding Thurston map with a Jordan curve $\CC\subseteq S^2$ satisfying $f(\CC)\subseteq \CC$ and $\post f\subseteq \CC$. We define for $m\in \N_0$ and $p \in\CC \cap \V^m(f,\CC)$,
\begin{equation}   \label{eqDefEdgePair}
\ae^m(p) \coloneqq \inte (e_1) \cup \{p\} \cup \inte (e_2) \quad \qquad  \text{and} \quad  \qquad
\overline\ae^m(p) \coloneqq e_1  \cup e_2,
\end{equation} 
where $e_1,e_2\in\E^m(f,\CC)$ are the unique pair of $m$-edges with $e_1\cup e_2 \subseteq \CC$ and $e_1\cap e_2 = \{p\}$. We denote for $m\in \N_0$, $n\in\N$, $q\in\CC$, and $q_j\in\CC \cap \V^m(f,\CC)$ for $j\in\{1,2,\dots,n\}$,
\begin{align}  \label{eqDefEm}
 &          E_m(q_n,\,q_{n-1},\,\dots,\,q_1;\,q)  \notag \\
 & \qquad = \bigl\{ x\in (f|_\CC)^{-n} (q)  \,\big|\,  (f|_\CC)^i(x) \in \overline\ae^m(q_{n-i}), \, i\in\{0,1,\dots,n-1\}  \bigr\} \\
 & \qquad = (f|_\CC)^{-n} (q)  \cap \biggl( \bigcap\limits_{i=0}^{n-1}  (f|_\CC)^{-i} (\overline\ae^m (q_{n-i}) ) \biggr)  \notag \\
 & \qquad \subseteq \CC \cap \V^{m+n}(f,\CC) \notag.
\end{align}

\begin{lemma}   \label{lmEmInductive}
Let $f$ and $\CC$ satisfy the Assumptions. We assume in addition that $f(\CC)\subseteq \CC$. Then
\begin{equation*}
\bigcup\limits_{x\in E_m(p_n,\,p_{n-1},\,\dots,\,p_1;\,p_0)} E_m(p_{n+1};\,x) =  E_m(p_{n+1},\,p_n,\,\dots,\,p_1;\,p_0).
\end{equation*}
for $m\in\N_0$, $n\in\N$, and $p_i \in\CC\cap \V^m(f,\CC)$ for $i\in\{1,2,\dots,n+1\}$. Here $E_m$ is defined in (\ref{eqDefEm}).
\end{lemma}

\begin{proof}
By (\ref{eqDefEm}), we get
\begin{align*}
&            \bigcup\limits_{x\in E_m(p_n,\,p_{n-1},\,\dots,\,p_1;\,p_0)} E_m(p_{n+1};\,x)     \\
&\qquad =    \biggl\{  y\in (f|_\CC)^{-1} (x) \,\bigg|\,  y\in \overline{\ae}^m(p_{n+1}), \, 
                           x\in (f|_\CC)^{-n}(p_0) \cap \biggl(  \bigcap\limits_{i=0}^{n-1} (f|_\CC)^{-i} (\overline{\ae}^m(p_{n-i}) ) \biggr)\biggr\}\\
&\qquad =    \biggl\{  y\in (f|_\CC)^{-n-1} (p_0) \,\bigg|\,  y\in \overline{\ae}^m(p_{n+1}), \, 
                           f(y) \in  \bigcap\limits_{i=0}^{n-1} (f|_\CC)^{-i} (\overline{\ae}^m(p_{n-i}) ) \biggr\}\\
&\qquad =    E_m(p_{n+1},\,p_n,\,\dots,\,p_1;\,p_0).                       
\end{align*}
The lemma is now established.
\end{proof}

\begin{lemma}   \label{lmEdgeInPair}
Let $f$ and $\CC$ satisfy the Assumptions. Fix $m,n\in\N_0$ with $m\leq n$. If $f(\CC)\subseteq\CC$, then the following statements hold:
\begin{enumerate}
\smallskip
\item[(i)] For each $n$-edge $e^n\in\E^n(f,\CC)$ and each $m$-edge $e^m\in\E^m(f,\CC)$, if $e^m \cap \inte(e^n) \neq \emptyset$, then $e^n\subseteq e^m$.

\smallskip
\item[(ii)] For each $n$-vertex $v\in\CC\cap \V^n(f,\CC)$ and each $m$-vertex $w\in\CC\cap\V^m(f,\CC)$ on the curve $\CC$, if $v\notin \overline{\ae}^m(w)$, then $\overline{\ae}^m(w) \cap \ae^n(v) = \emptyset$.

\smallskip
\item[(iii)] Assume that $m\geq 1$ and no $1$-tile in $\X^1(f,\CC)$ joins opposite sides of $\CC$. For each pair $v_0,v_1 \in \CC\cap\V^m(f,\CC)$ of $m$-vertices on $\CC$, we denote by $e^1_i,e^2_i \in \E^m(f,\CC)$ the unique pair of $m$-edges with $e^1_i \cup e^2_i = \overline{\ae}^m(v_i)$ and $e^1_i \cap e^2_i = \{v_i\}$, for each $i\in\{0,1\}$. Then $f$ is injective on $e^j_1$ for each $j\in\{1,2\}$, and exactly one of the following cases is satisfied:

\begin{enumerate}
\smallskip
\item[(1)] $f( \ae^m(v_1) ) \cap \ae^m (v_0) = \emptyset$. In this case, 
$\card\{ x\in \overline{\ae}^m(v_1) \,|\, f(x) \in \overline{\ae}^m(v_0)\} \leq 2$.

\smallskip
\item[(2)] There exist $j,k\in\{1,2\}$ such that 
\begin{itemize}
\smallskip
\item $f\bigl(e^j_1 \bigr) \supseteq e^k_0$, 

\smallskip
\item $f\bigl(e^j_1\bigr) \cap e^{k'}_0 \setminus \{v_0\} = \emptyset$ for $k'\in \{1,2\}\setminus \{k\}$, 

\smallskip
\item $f\bigl(e^{j'}_1\bigr) \cap \overline{\ae}^m(v_0) = \emptyset$ for $j'\in \{1,2\}\setminus\{j\}$, and

\smallskip
\item $v_1\notin\crit f|_\CC$.
\end{itemize}

\smallskip
\item[(3)] There exists $j\in\{1,2\}$ such that
\begin{itemize}
\smallskip
\item $f\bigl(e^j_1 \bigr) \supseteq \overline{\ae}^m(v_0) = e^1_0 \cup e^2_0$, 

\smallskip
\item $f\bigl(e^{j'}_1 \setminus \{v_1\} \bigr) \cap \overline{\ae}^m(v_0) = \emptyset$ for $j'\in \{1,2\}\setminus \{j\}$,

\smallskip
\item $v_1\notin\crit f|_\CC$, and $f(v_1)\neq v_0$.
\end{itemize}

\smallskip
\item[(4)] There exists $k\in\{1,2\}$ such that
\begin{itemize}
\smallskip
\item $f\bigl(e^1_1 \bigr) \supseteq e^k_0$, $f\bigl(e^2_1 \bigr) \supseteq e^{k'}_0$, 

\smallskip
\item $f\bigl(e^1_1 \setminus \{v_1\} \bigr) \cap e^{k'}_0 = \emptyset$, $f\bigl(e^2_1 \setminus \{v_1\} \bigr) \cap e^k_0 = \emptyset$,

\smallskip
\item $v_1\notin\crit f|_\CC$, and $f(v_1)=v_0$,
\end{itemize}
\smallskip
where $k'\in \{1,2\} \setminus \{k\}$.

\smallskip
\item[(5)] There exists $k\in\{1,2\}$ such that 
\begin{itemize}
\smallskip
\item $f\bigl(e^1_1 \bigr) \cap f\bigl(e^2_1 \bigr) \supseteq e^k_0$, 

\smallskip
\item $f\bigl(e^1_1\bigr) \cap e^{k'}_0  \setminus \{v_0\} = \emptyset$ for $k'\in \{1,2\}\setminus \{k\}$,

\smallskip
\item $f(e^1_1)=f(e^2_1)$, and $v_1\in\crit f|_\CC$.
\end{itemize}

\smallskip
\item[(6)] For each $j\in\{1,2\}$, $f\bigl(e^j_1\bigr) \supseteq \overline{\ae}^m(v_0)$. In this case, we have $f(e^1_1)=f(e^2_1)$ and $v_1\in\crit f|_\CC$.
\end{enumerate}
 
\end{enumerate}
\end{lemma}

We say that a point $x\in\CC$ is a \defn{critical point of $f|_\CC$}, denoted by $x\in\crit f|_\CC$, if there is no neighborhood $U\subseteq \CC$ of $x$ on which $f$ is injective. Clearly, $\crit f|_\CC \subseteq \crit f$. Our proof below relies crucially on the fact that $\CC$ is a Jordan curve.

\begin{proof}
(i) Fix arbitrary $e^n\in\E^n$ and $e^m\in\E^m$. Since $f(\CC)\subseteq \CC$, $e^m = \bigcup \{e\in\E^n \,|\, e\subseteq e^m\}$. If $e^m \cap \inte(e^n) \neq \emptyset$, then there exists $e\in\E^n$ with $e\subseteq e^m$ and $\inte(e) \cap \inte(e^n) \neq \emptyset$. Then $e=e^n$ (see Definition~\ref{defcelldecomp}). Hence $e^n\subseteq e^m$.

\smallskip

(ii) Fix $v\in\CC\cap\V^n$ and $w\in\CC\cap\V^m$ with $v\notin \overline{\ae}^m(w)$. Suppose $\overline{\ae}^m(w) \cap \ae^n(v) \neq \emptyset$. Then there exist $e^n\in\E^n$ and $e^m\in\E^m$ such that $e^n\subseteq \overline{\ae}^n(v)$, $e^m\subseteq \overline{\ae}^m(w)$, and $\inte(e^n)\cap e^m \neq\emptyset$. Then by Lemma~\ref{lmEdgeInPair}~(i), $e^n\subseteq e^m$. Hence $v\in e^n \subseteq \overline{\ae}^m(w)$, a contradiction.

\smallskip

(iii) Fix arbitrary $v_0,v_1\in\CC\cap\V^m$. Recall $m\geq 1$.

By Proposition~\ref{propCellDecomp}~(i), the map $f$ is injective on $e^j_1$ for each $j\in\{1,2\}$.

Recall that $\card (\post f) \geq 3$ (see \cite[Lemma~6.1]{BM17}).

We denote $I \coloneqq \bigl\{ (j,k) \in \{1,2\}\times\{1,2\} \,\big|\, f\bigl( e^j_1 \bigr) \supseteq e^k_0 \bigr\}$. Then $\card I \in \{0,1,2,3,4\}$. 

We will establish Lemma~\ref{lmEdgeInPair}~(iii) by proving the following statements:

\begin{enumerate}
\smallskip
\item[(a)] $\card I \neq 3$.

\smallskip
\item[(b)] Case~(1), (2), or (6) holds if and only if $\card I = 0$, $1$, or $4$, respectively.

\smallskip
\item[(c)] Case~(3) holds if and only if $\card I = 2$, $v_1 \notin \crit f|_\CC$, and $f(v_1) \neq v_0$.

\smallskip
\item[(d)] Case~(4) holds if and only if $\card I = 2$, $v_1 \notin \crit f|_\CC$, and $f(v_1) = v_0$.

\smallskip
\item[(e)] Case~(5) holds if and only if $\card I = 2$, $v_1 \in \crit f|_\CC$.
\end{enumerate}

\smallskip

(a) Suppose $\card I = 3$. Without loss of generality, we assume that $f\bigl( e^1_1 \bigr) \supseteq e^1_0 \cup e^2_0$, $f\bigl( e^2_1 \bigr) \supseteq e^1_0$, and $f\bigl( e^2_1 \bigr) \nsubseteq e^2_0$. Since $f\bigl( e^1_1 \bigr), f\bigl( e^2_1 \bigr) \in \E^{m-1}$ (see Proposition~\ref{propCellDecomp}~(i)), $f\bigl( e^1_1 \bigr) \cap f\bigl( e^2_1 \bigr) \supseteq e^1_0$, we get $f\bigl( e^1_1 \bigr) \cap \inte \bigl( f\bigl( e^2_1 \bigr) \bigr) \neq \emptyset$, thus by Lemma~\ref{lmEdgeInPair}~(i), $f\bigl( e^1_1 \bigr) = f\bigl( e^2_1 \bigr)$. Hence $\card I = 4$. This is a contradiction.

\smallskip

(b) If Case~(1) holds, then clearly $\card I = 0$. Conversely, we assume that $\card I = 0$. Fix arbitrary $j,k\in\{1,2\}$. Since $f\bigl( e^j_1 \bigr) \in \E^{m-1}$ and $f$ is injective on $e^j_1$ (see Proposition~\ref{propCellDecomp}~(i)), by Lemma~\ref{lmEdgeInPair}~(i), $f\bigl( \inte \bigl(  e^j_1 \bigr) \bigr) \cap   e^k_0  =  f\bigl( e^j_1 \bigr) \cap \inte \bigl(  e^k_0 \bigr) = \emptyset$. Observe that it follows from $\card I = 0$ and Lemma~\ref{lmEdgeInPair}~(i) that $f(v_1)\neq v_0$. Thus $f( \ae^m(v_1) ) \cap \ae^m (v_0) = \emptyset$. In order to show $\card\{ x\in \overline{\ae}^m(v_1) \,|\, f(x) \in \overline{\ae}^m(v_0)\} \leq 2$, it suffices to prove $\card \bigl\{ x\in e^j_1 \setminus \inte\bigl(e^j_1\bigr) \,\big|\, f(x)\in \overline{\ae}^m(v_0) \bigr\} \leq 1$. Suppose not, then since $f\bigl( \inte \bigl(  e^j_1 \bigr) \bigr) \cap   \overline{\ae}^m(v_0) = \emptyset$, $f$ is injective on $e^j_1$, and $\CC$ is a Jordan curve, we get $f\bigl(  e^j_1 \bigr) \cup \overline{\ae}^m(v_0) = \CC$. This contradicts the fact that $\card (\post f) \geq 3$ and the condition that no $1$-tile in $\X^1$ joins opposite sides of $\CC$. Therefore, Case~(1) holds.

\smallskip

If Case~(2) holds, then clearly $\card I = 1$. Conversely, we assume that $\card I = 1$. Without loss of generality, we assume that $f\bigl( e^1_1 \bigr) \supseteq e^1_0$. We observe that $v_1\notin \crit f|_\CC$. For otherwise, $f\bigl( e^1_1 \bigr) = f\bigl( e^2_1 \bigr) \in\E^{m-1}$ (see Proposition~\ref{propCellDecomp}~(i)), thus $\card I \neq 1$, which is a contradiction. 

To show $f\bigl( e^1_1 \bigr) \cap e^2_0 \setminus \{v_0\} = \emptyset$, we argue by contradiction and assume that $f\bigl( e^1_1 \bigr) \cap e^2_0 \setminus \{v_0\} \neq \emptyset$. Since $e^2_0 \nsubseteq f\bigl( e^1_1 \bigr) \in \E^{m-1}$ (see Proposition~\ref{propCellDecomp}~(i)), by Lemma~\ref{lmEdgeInPair}~(i), $f\bigl( e^1_1 \bigr) \cap \inte \bigl(  e^2_0 \bigr) = \emptyset$. Note that $v_0\in e_0^1 \subseteq f \bigl( e_1^1 \bigr)$. Since $f\bigl( e^1_1 \bigr)$ is connected and $\CC$ is a Jordan curve, we get $f\bigl( e^1_1 \bigr) \cup e^2_0 = \CC$. This contradicts the fact that $\card (\post f) \geq 3$.

Next, we verify that $f(v_1) \notin e^1_0$ as follows. We argue by contradiction and assume that $f(v_1)\in e^1_0$. Since $f(v_1) \in \V^{m-1}$ (see Proposition~\ref{propCellDecomp}~(i)), we get $f(v_1) \in e^1_0 \setminus \inte\bigl(e^1_0\bigr)$. Since $\card I = 1$, it is clear that $f(v_1) \neq v_0$. Thus $f(v_1) \in e^1_0 \setminus \bigl( \inte\bigl(e^1_0 \bigr) \cup \{v_0\} \bigr)$. Since $v_1\notin \crit f|_\CC$ and no $1$-tile in $\X^1$ joins opposite sides of $\CC$, we get from Proposition~\ref{propCellDecomp}~(i) that either $f\bigl(e^1_1 \bigr) \supseteq e^1_0 \cup e^2_0$ or $f\bigl(e^2_1 \bigr) \supseteq e^1_0 \cup e^2_0$. This contradicts the assumption that $\card I = 1$. Hence $f(v_1) \notin e^1_0$.

Finally we show that $f\bigl(e^2_1\bigr) \cap \overline{\ae}^m(v_0) = \emptyset$. To see this, we argue by contradiction and assume that $f\bigl(e^2_1\bigr) \cap \overline{\ae}^m(v_0) \neq \emptyset$. Since $\CC$ is a Jordan curve, $v_1\notin \crit f|_\CC$, $f(v_1) \notin e^1_0$, $f\bigl( e^1_1 \bigr) \supseteq e^1_0$, and $f \bigl( e_1^1 \bigr) \nsupseteq e_0^2$, we get that $f\bigl( e^1_1 \bigr) \cup f\bigl( e^2_1 \bigr) \cup e^2_0 = \CC$. This contradicts the fact that $\card (\post f) \geq 3$ and the condition that no $1$-tile in $\X^1$ joins opposite sides of $\CC$.

Therefore, Case~(2) holds.

\smallskip

If Case~(6) holds, then clearly $\card I = 4$. Conversely, we assume that $\card = 4$. Then $f\bigl( e^j_1\bigr) \supseteq e^1_0 \cup e^2_0 = \overline{\ae}^m(v_0)$ for each $j\in\{1,2\}$. We show that $v_1\in\crit f|_\CC$ as follows. We argue by contradiction and assume that $v_1\notin\crit f|_\CC$. Then Since $\CC$ is a Jordan curve and $f\bigl( e^1_1\bigr) \cap f\bigl( e^2_1\bigr) \supseteq e^1_0$, we get that $f\bigl( e^1_1\bigr) \cup f\bigl( e^2_1\bigr) = \CC$. This contradicts the fact that $\card (\post f) \geq 3$. Hence $v_1\in\crit f|_\CC$, and consequently $f\bigl( e^1_1\bigr) = f\bigl( e^2_1\bigr) \in \E^{m-1}$ by Proposition~\ref{propCellDecomp}~(i). Therefore, Case~(6) holds.

\smallskip

(c) If Case~(3) holds, then clearly $\card I = 2$, $v_1 \notin \crit f|_\CC$, and $f(v_1) \neq v_0$. Conversely, we assume that $\card I = 2$, $v_1 \notin \crit f|_\CC$, and $f(v_1) \neq v_0$. Suppose that $f \bigl( e_1^1 \bigr) \supseteq e_0^k$ and $f \bigl( e_1^2 \bigr) \supseteq e_0^{k'}$ for some $k,k'\in\{1,2\}$ with $k\neq k'$, then since $v_1 \notin  \crit f|_\CC$ and $f(v_1) \neq v_0$, we get from Proposition~\ref{propCellDecomp}~(i) that $f \bigl( e_1^1 \bigr) \cup f \bigl( e_1^2 \bigr) = \CC$. This contradicts the fact that $\card (\post f) \geq 3$. Thus, without loss of generality, we can assume that $f\bigl(e^1_1\bigr) \supseteq e^1_0 \cup e^2_0$. In order to show that Case~(3) holds, it suffices now to verify that $f\bigl( e^2_1 \setminus \{v_1\} \bigr) \cap \overline{\ae}^m(v_0) = \emptyset$. We argue by contradiction and assume that $f\bigl( e^2_1 \setminus \{v_1\} \bigr) \cap \overline{\ae}^m(v_0) \neq \emptyset$. Then $f\bigl( e^2_1 \setminus \{v_1\} \bigr) \cap f\bigl( e^1_1  \bigr) \neq \emptyset$. Since $\CC$ is a Jordan curve and $v_1\notin\crit f|_\CC$, we get from Proposition~\ref{propCellDecomp}~(i) that $f\bigl( e^2_1   \bigr) \cup f\bigl( e^1_1  \bigr) = \CC$. This contradicts the fact that $\card (\post f) \geq 3$. Therefore, Case~(3) holds.

\smallskip

(d) If Case~(4) holds, then clearly $\card I = 2$, $v_1 \notin \crit f|_\CC$, and $f(v_1) = v_0$. Conversely, we assume that $\card I = 2$, $v_1 \notin \crit f|_\CC$, and $f(v_1) = v_0$. Without loss of generality, we assume that $f\bigl(e^1_1\bigr) \supseteq e^1_0$ and $f\bigl(e^2_1\bigr) \supseteq e^2_0$. In order to show that Case~(4) holds, by symmetry, it suffices to show that $f\bigl( e^1_1 \setminus \{v_1\} \bigr) \cap e^2_0 = \emptyset$. We argue by contradiction and assume that $f\bigl( e^1_1 \setminus \{v_1\} \bigr) \cap e^2_0 \neq \emptyset$. Then $f\bigl( e^1_1 \setminus \{v_1\} \bigr) \cap f\bigl( e^2_1  \bigr) \neq \emptyset$. Since $\CC$ is a Jordan curve and $v_1\notin\crit f|_\CC$, we get from Proposition~\ref{propCellDecomp}~(i) that $f\bigl( e^2_1   \bigr) \cup f\bigl( e^1_1  \bigr) = \CC$. This contradicts the fact that $\card (\post f) \geq 3$. Therefore, Case~(4) holds.

\smallskip

(e) If Case~(5) holds, then clearly $\card I = 2$ and $v_1 \in \crit f|_\CC$. Conversely, we assume that $\card I = 2$ and $v_1 \in \crit f|_\CC$. Since $v_1 \in \crit f|_\CC$, $f\bigl(e^1_1\bigr) = f\bigl(e^2_1\bigr) \in \E^{m-1}$ by Proposition~\ref{propCellDecomp}~(i). Without loss of generality, we assume that $f\bigl(e^1_1\bigr) \cap f\bigl(e^2_1\bigr) \supseteq e^1_0$. In order to show that Case~(5) holds, it suffices now to show that $f\bigl( e^1_1  \bigr) \cap e^2_0 \setminus \{v_0\}= \emptyset$. We argue by contradiction and assume that $f\bigl( e^1_1 \bigr) \cap e^2_0 \setminus \{v_0\} \neq \emptyset$.  Since $e^2_0 \nsubseteq f\bigl( e^1_1 \bigr) \in \E^{m-1}$ (see Proposition~\ref{propCellDecomp}~(i)), by Lemma~\ref{lmEdgeInPair}~(i), $f\bigl( e^1_1 \bigr) \cap \inte \bigl(  e^2_0 \bigr) = \emptyset$. Thus $e^2_0\setminus \inte\bigl(e^2_0\bigr) \subseteq f\bigl( e^1_1 \bigr)$ as we already know $v_0\in e^1_0 \subseteq f\bigl( e^1_1 \bigr)$. Since $f\bigl( e^1_1 \bigr)$ is connected and $\CC$ is a Jordan curve, we get $f\bigl( e^1_1 \bigr) \cup e^2_0 = \CC$. This contradicts the fact that $\card (\post f) \geq 3$. Therefore, Case~(5) holds.
\end{proof}

\begin{lemma}   \label{lmCoverByEdgePair}
Let $f$ and $\CC$ satisfy the Assumptions. We assume in addition that $f(\CC)\subseteq \CC$. Then
\begin{equation*}
\bigcap\limits_{i=0}^n (f|_\CC)^{-i} (\ae^m(p_{n-i}))  \subseteq \bigcup\limits_{x\in E_m(p_n,\,p_{n-1},\,\dots,\,p_1;\,p_0)}  \ae^{m+n} (x),
\end{equation*}
for all $m\in\N_0$, $n\in\N$, and $p_i \in \CC \cap \V^m(f,\CC)$ for each $i\in \{0,1,\dots,n\}$. Here $\ae^m$ is defined in (\ref{eqDefEdgePair}) and $E_m$ in (\ref{eqDefEm}).
\end{lemma}

\begin{proof}
We fix $m\in\N_0$ and an arbitrary sequence $\{p_i\}_{i\in\N_0}$ in $\CC\cap\V^m$. We prove the lemma by induction on $n\in\N$.

For $n=1$, we get
\begin{align*}
               \ae^m(p_1) \cap (f|_\CC)^{-1} (\ae^m(p_0)) 
\subseteq     & \bigcup \bigl\{ \ae^{m+1}(x) \,\big|\, x\in (f|_\CC)^{-1} (p_0),\, x\in \overline{\ae}^m(p_1)  \bigr\}  \\
   =          & \bigcup\limits_{x\in E_m(p_1;\,p_0)} \ae^{m+1}(x)
\end{align*}
by (\ref{eqDefEm}), Proposition~\ref{propCellDecomp}~(ii), and the fact that $\ae^{m+1}(x) \cap \ae^m(p_1) = \emptyset$ if both $x\in\CC \cap \V^{m+1}$ and $x\notin \overline{\ae}^m(p_1)$ are satisfied (see Lemma~\ref{lmEdgeInPair}~(ii)).

We now assume that the lemma holds for $n=l$ for some integer $l\in\N$. Then by the induction hypothesis, we have
\begin{align*}
&             \bigcap\limits_{i=0}^{l+1}  (f|_\CC)^{-i} (\ae^m(p_{l+1-i})) \\
&\qquad =          \ae^m(p_{l+1}) \cap (f|_\CC)^{-1} \biggl( \bigcap\limits_{i=1}^{l+1}  (f|_\CC)^{-(i-1)} (\ae^m(p_{l+1-i})) \biggr) \\
&\qquad\subseteq   \ae^m(p_{l+1}) \cap (f|_\CC)^{-1} \biggl( \bigcup\limits_{x\in E_m(p_l,\,p_{l-1},\,\dots,\,p_1;\,p_0)}  \ae^{m+l}(x) \biggr) \\
&\qquad =          \bigcup\limits_{x\in E_m(p_l,\,p_{l-1},\,\dots,\,p_1;\,p_0)}  \bigl(  \ae^m(p_{l+1}) \cap (f|_\CC)^{-1} \bigl(\ae^{m+l}(x)\bigr)\bigr) \\
&\qquad\subseteq   \bigcup\limits_{x\in E_m(p_l,\,p_{l-1},\,\dots,\,p_1;\,p_0)}  
                           \Bigl(  \bigcup \bigl\{ \ae^{m+l+1}(y) \,\big|\, y\in(f|_\CC)^{-1} (x),\, y\in \overline{\ae}^m(p_{l+1}) \bigr\} \Bigr)  \\
&\qquad =          \bigcup\limits_{x\in E_m(p_l,\,p_{l-1},\,\dots,\,p_1;\,p_0)}  \bigcup\limits_{y\in E_m(p_{l+1};\,x)} \ae^{m+l+1}(y),
\end{align*}
where the last two lines are due to (\ref{eqDefEm}), Proposition~\ref{propCellDecomp}~(ii), and the fact that $\ae^{m+l+1}(y) \cap \ae^m(p_{l+1}) = \emptyset$ if both $y\in\CC \cap \V^{m+l+1}$ and $y\notin \overline{\ae}^m(p_{l+1})$ are satisfied (see Lemma~\ref{lmEdgeInPair}~(ii)).

By Lemma~\ref{lmEmInductive}, we get
\begin{equation*}
\bigcap\limits_{i=0}^{l+1} (f|_\CC)^{-i} (\ae^m(p_{l+1-i}))  \subseteq \bigcup\limits_{x\in E_m(p_{l+1},\,p_l,\,\dots,\,p_1;\,p_0)}  \ae^{m+l+1} (x).
\end{equation*}

The induction step is now complete.
\end{proof}

\begin{prop}    \label{propEmBound}
Let $f$ and $\CC$ satisfy the Assumptions. We assume in addition that $f(\CC)\subseteq \CC$ and no $1$-tile in $\X^1(f,\CC)$ joins opposite sides of $\CC$. Then
\begin{equation}   \label{eqEmBound}
\card (  E_m(p_n,\,p_{n-1},\,\dots,\,p_1;\,p_0)  )  \leq m 2^{\frac{n}{m}}
\end{equation}
for all $m, n\in\N$ with $m\geq 14$, and $p_i \in \CC \cap \V^m(f,\CC)$ for each $i\in \{0,1,\dots,n\}$. Here $E_m$ is defined in (\ref{eqDefEm}).
\end{prop}

\begin{proof}
We fix $m\in\N$ with $m\geq 14$, and fix an arbitrary sequence $\{p_i\}_{i\in\N_0}$ of $m$-vertices in $\CC\cap\V^m$.

For each $n\in\N$, we write $E_{m,n} \coloneqq E_m (p_n,\,p_{n-1},\,\dots,\,p_1;\,p_0)$. Note that for each $n\in\N$, by (\ref{eqDefEm}),
\begin{equation}  \label{eqPfpropEmBound_EmnInAEpn}
E_{m,n} = E_m (p_n,\,p_{n-1},\,\dots,\,p_1;\,p_0) \subseteq \overline{\ae}^m(p_n),
\end{equation}
where $\overline{\ae}^m$ is defined in (\ref{eqDefEdgePair}). We denote by $e_{n,1}, e_{n,2}\in\E^m$ the unique pair of $m$-edges with $e_{n,1} \cup e_{n,2} = \overline{\ae}^m(p_n)$ and $e_{n,1} \cap e_{n,2} = \{p_n\}$. For each $i\in\{1,2\}$, we define
\begin{equation} \label{eqPfpropEmBound_Lni}
L_{n,i} \coloneqq \begin{cases} \card (E_{m,n} \cap e_{n,i}) & \text{if } E_{m,n} \cap e_{n,i} \setminus \{p_n\} \neq \emptyset, \\ 0  & \text{otherwise}.  \end{cases}
\end{equation}

\smallskip

We first observe that by (\ref{eqDefEm}), (\ref{eqPfpropEmBound_EmnInAEpn}), Lemma~\ref{lmEmInductive}, and the fact that $f$ is injective on each $m$-edge (see Proposition~\ref{propCellDecomp}~(i)), we get that
\begin{equation}   \label{eqPfpropEmBound_EmDouble}
\card E_{m,1} \leq 2  \qquad \text{ and } \qquad  \card E_{m,n+1} \leq 2 \card E_{m,n},
\end{equation}
for each $n\in\N$.

Next, we need to establish two claims.

\smallskip

\emph{Claim~1.} For each $n\in\N$, if $L_{n,1}L_{n,2}\neq 0$ then $L_{n,1}=L_{n,2}$.

\smallskip

We will establish Claim~1 by induction on $n\in\N$.

For $n=1$, we apply Lemma~\ref{lmEdgeInPair}~(iii) with $v_0=p_0$ and $v_1=p_1$. By (\ref{eqDefEm}), it is easy to verify that in Cases~(1) through (4) discussed in Lemma~\ref{lmEdgeInPair}~(iii), we have $L_{1,1}L_{1,2} = 0$, and in Cases~(5) and (6), we have $L_{1,1}=L_{1,2}$.

We now assume that Claim~1 holds for $n=l$ for some integer $l\in\N$. We apply Lemma~\ref{lmEdgeInPair}~(iii) with $v_0=p_l$ and $v_1=p_{l+1}$. Then by (\ref{eqPfpropEmBound_EmnInAEpn}) and Lemma~\ref{lmEmInductive} with $n=l$, it is easy to verify that in Case~(1) discussed in Lemma~\ref{lmEdgeInPair}~(iii), we have either $L_{l+1,1}L_{l+1,2} = 0$ or $L_{l+1,1}=L_{l+1,2}=1$; in Cases~(2) and (3), we have $L_{l+1,1}L_{l+1,2} = 0$; in Case~(4), we have $L_{l+1,1}=L_{l,k}$ and $L_{l+1,2}=L_{l,k'}$, where $k,k'\in\{1,2\}$ satisfy $f(e_{l+1,1})\supseteq e_{l,k}$, $f(e_{l+1,2})\supseteq e_{l,k'}$, and $k\neq k'$; and in Cases~(5) and (6), we have $L_{l+1,1}=L_{l+1,2}$.

The induction step is now complete. Claim~1 follows.

\smallskip

\emph{Claim~2.} For each $n\in\N$ with $\card E_{m,n} \geq 4$, the following statements hold:
\begin{enumerate}
\smallskip
\item[(i)] If $\card E_{m,n+1} < \card E_{m,n}$, then 
\begin{equation*}
\card E_{m,n+1} \leq \Bigl\lceil \frac{1}{2} \card E_{m,n} \Bigr\rceil.
\end{equation*}

\smallskip
\item[(ii)] If $\card E_{m,n+1} = \card E_{m,n}$ and $p_{n+1} \in \crit f|_\CC$, then
\begin{equation*}
\card (e_{n+1,1} \cap E_{m,n+1}) = \card (e_{n+1,2} \cap E_{m,n+1}).
\end{equation*}

\smallskip
\item[(iii)] If $\card E_{m,n+1} > \card E_{m,n}$, then
\begin{enumerate}
\smallskip
\item[(a)] $\card (e_{n+1,1} \cap E_{m,n+1}) = \card (e_{n+1,2} \cap E_{m,n+1})$,

\smallskip
\item[(b)] $p_{n+1} \in \crit f|_\CC$, and

\smallskip
\item[(c)] $E_{m,n} \subseteq e^{m-1} \in \E^{m-1}$, where $e^{m-1} \coloneqq f(e_{n+1,1}) = f(e_{n+1,2})$.
\end{enumerate}
\end{enumerate}

\smallskip
To prove Claim~2, we first note that by (\ref{eqPfpropEmBound_EmDouble}), $\card E_{m,1} \leq 2$, so it suffices to consider $n\geq 2$. We fix an integer $n\geq 2$ with $\card E_{m,n} \geq 4$. If such $n$ does not exist, then Claim~2 holds trivially.

We will verify statements~(i) through (iii) according to the cases discussed in Lemma~\ref{lmEdgeInPair}~(iii) with $v_0 = p_n$, $v_1 = p_{n+1}$, $e^i_0=e_{n,i}$, and $e^i_1=e_{n+1,i}$ for each $i\in\{1,2\}$.

\smallskip

\emph{Case~(1).} It is easy to see that $\card E_{m,n+1} \leq 2 \leq \bigl\lceil \frac{1}{2} \card E_{m,n} \bigr\rceil$.

\smallskip

\emph{Case~(2).} We have $p_{n+1}\notin \crit f|_\CC$. Without loss of generality, we assume that $f(e_{n+1,1}) \supseteq e_{n,1}$, $f(e_{n+1,1}) \cap e_{n,2} \setminus \{p_n\} = \emptyset$, and $f(e_{n+1,2}) \cap \overline{\ae}^m(p_n) = \emptyset$. Since $f$ is injective on each $m$-edge (see Proposition~\ref{propCellDecomp}~(i)), $E_{m,n} \subseteq \overline{\ae}^m(p_{n})$ (see (\ref{eqDefEm})), and either $L_{n,1}  L_{n,2} = 0$ or $L_{n,1} = L_{n,2}$ by Claim~1, it is easy to verify from Lemma~\ref{lmEmInductive} that either $\card E_{m,n+1} = \card E_{m,n}$ or $\card E_{m,n+1} \leq \bigl\lceil \frac{1}{2} \card E_{m,n} \bigr\rceil$.

\smallskip

\emph{Case~(3).} We have $p_{n+1}\notin \crit f|_\CC$. Without loss of generality, we assume that $f(e_{n+1,1}) \supseteq \overline{\ae}^m (p_n)$ and $f(e_{n+1,2} \setminus \{ p_{n+1} \} ) \cap \overline{\ae}^m(p_n) = \emptyset$. Since $f$ is injective on each $m$-edge (see Proposition~\ref{propCellDecomp}~(i)), $E_{m,n} \subseteq \overline{\ae}^m(p_{n})$ (see (\ref{eqDefEm})), and either $L_{n,1}  L_{n,2} = 0$ or $L_{n,1} = L_{n,2}$ by Claim~1, it is easy to verify from Lemma~\ref{lmEmInductive} that $\card E_{m,n+1} = \card E_{m,n}$.

\smallskip

\emph{Case~(4).} We have $p_{n+1} \notin \crit f|_\CC$. By Proposition~\ref{propCellDecomp}~(i), $f$ maps $\overline{\ae}^m(p_{n+1})$ bijectively onto $f( \overline{\ae}^m(p_{n+1}) )$. Since $f( \overline{\ae}^m(p_{n+1}) ) \supseteq \overline{\ae}^m(p_{n})$ and $E_{m,n} \subseteq \overline{\ae}^m(p_{n})$ (see (\ref{eqDefEm})), we get $\card E_{m,n+1} = \card E_{m,n}$ by Lemma~\ref{lmEmInductive}.

\smallskip

\emph{Case~(5).} We have $p_{n+1}\in \crit f|_\CC  \subseteq \crit f$. Without loss of generality, we assume that $f(e_{n+1,j}) \supseteq e_{n,1}$ and $f(e_{n+1,j}) \cap e_{n,2} \setminus \{p_n\} = \emptyset$ for each $j\in\{1,2\}$. Since $E_{m,n} \subseteq \overline{\ae}^m(p_{n})$ (see (\ref{eqDefEm})) and either $L_{n,1}  L_{n,2} = 0$ or $L_{n,1} = L_{n,2}$ by Claim~1, it is easy to verify from Lemma~\ref{lmEmInductive} that either $\card E_{m,n+1} \leq 2 \leq \bigl\lceil \frac{1}{2} \card E_{m,n} \bigr\rceil$ or $\card E_{m,n+1} \geq \card E_{m,n}$. Note that in either case, we have $\card (e_{n+1,1} \cap E_{m,n+1}) = \card (e_{n+1,2} \cap E_{m,n+1})$. Moreover, if $\card E_{m,n+1} > \card E_{m,n}$, then it follows that $L_{n,2} = 0$, and thus $E_{m,n} \subseteq e_{n,1} \subseteq f(e_{n+1,1}) = f(e_{n+1,2}) \in \E^{m-1}$ (see Proposition~\ref{propCellDecomp}~(i)). 

\smallskip

\emph{Case~(6).} We have $p_{n+1}\in \crit f|_\CC  \subseteq \crit f$. Since $f$ is injective on each $m$-edge (see Proposition~\ref{propCellDecomp}~(i)) and $E_{m,n} \subseteq \overline{\ae}^m(p_{n})$ (see (\ref{eqDefEm})), it is easy to verify that $\card E_{m,n+1} > \card E_{m,n}$ and $E_{m,n} \subseteq \overline{\ae}^m(p_n)  \subseteq f(e_{n+1,1}) = f(e_{n+1,2}) \in \E^{m-1}$ (see Proposition~\ref{propCellDecomp}~(i)).

\smallskip

Claim~2 now follows.

\smallskip

Finally, we will establish (\ref{eqEmBound}) by induction on $n\in\N$. Recall that we assume $m\geq 14$.

For $n=1$, by (\ref{eqPfpropEmBound_EmDouble}), $\card E_{m,1} \leq 2 < m 2^{\frac{n}{m}}$. 

We now assume that (\ref{eqEmBound}) holds for all $n\leq l$ for some integer $l\in\N$. If $\card E_{m,l} < 8$, then (\ref{eqEmBound}) holds for $n=l+1$ by (\ref{eqPfpropEmBound_EmDouble}) and the assumption that $m\geq 14$. So we can assume that $\card E_{m,l} \geq 8$. Moreover, if $\card E_{m,l+1} \leq \card E_{m,l}$, then (\ref{eqEmBound}) holds for $n=l+1$ trivially from the induction hypothesis. Thus we can also assume that $\card E_{m,l+1} > \card E_{m,l}$.

Since $\card E_{m,1} \leq 2$ (see (\ref{eqPfpropEmBound_EmDouble})) and $\card E_{m,l} \geq 8$, we can define a number
\begin{equation}  \label{eqPfpropEmBound_Defk}
k \coloneqq \max \{ i\in\N \,|\, i<l, \, \card E_{m,i} \neq \card E_{m,l} \} \leq l-1.
\end{equation}
Note that $l\geq 3$ and $\card E_{m,k} \geq \frac{1}{2} \card E_{m,k+1} = \frac{1}{2} \card E_{m,l} \geq 4$ by (\ref{eqPfpropEmBound_EmDouble}).

We will establish (\ref{eqEmBound}) for $n=l+1$ by considering the following two cases:

\smallskip

\emph{Case I.} $\card E_{m,k} > \card E_{m,l}=\card E_{m,k+1}$. Then by (\ref{eqPfpropEmBound_EmDouble}) and Claim~2(i), we have
\begin{align*}
 \card E_{m,l+1} & \leq 2 \card E_{m,l}
                    =   2 \card E_{m,k+1}
                   \leq 2 \Bigl\lceil \frac{1}{2} \card E_{m,k} \Bigr\rceil \\
                 & \leq 1 + \card E_{m,k} 
                   \leq 1 + m 2^{\frac{k}{m}}
                   \leq m 2^{\frac{k+2}{m}}
                   \leq m 2^{\frac{l+1}{m}},
\end{align*}
where the second-to-last inequality follows from the fact that the function $h(x) \coloneqq x\bigl(2^{\frac{2}{x}} - 1 \bigr)$, $x>1$, satisfies $\lim\limits_{x\to+\infty} h(x) = \log 4 >1$, $\lim\limits_{x\to+\infty} \frac{\mathrm{d}}{\mathrm{d} x} h(x) = 0$, and $\frac{\mathrm{d}^2}{\mathrm{d} x^2} h(x) > 0$ for $x>1$.

\smallskip

\emph{Case II.} $\card E_{m,k} < \card E_{m,l}=\card E_{m,k+1}$. Then by Claim~2(iii), we have $p_{k+1} \in \crit f|_\CC$. We define
\begin{equation}  \label{eqPfpropEmBound_Defk'}
k' \coloneqq \max \{ i\in \N \,|\, i\leq l, \, p_i\in\crit f|_\CC \}  \in [k+1, l].
\end{equation}
Note that by (\ref{eqPfpropEmBound_Defk}), (\ref{eqPfpropEmBound_Defk'}), and (\ref{eqPfpropEmBound_EmDouble}), we get $\card E_{m,k'-1} \geq \frac{1}{2} \card E_{m,l} \geq 4$. By Claim~2(ii) and (iii), regardless of whether $k'=k+1$ or not, we have
\begin{equation}  \label{eqPfpropEmBound_k'TwoSides}
\card (e_{k',1} \cap E_{m,k'} ) = \card (e_{k',2} \cap E_{m,k'} ) \geq 2.
\end{equation}
By Claim~2(iii), we have $p_{l+1} \in \crit f|_\CC \subseteq \V^1$, $E_{m,l} \subseteq e^{m-1} \in \E^{m-1}$ where $e^{m-1} \coloneqq f(e_{l+1,1}) = f(e_{l+1,2})$. Note that $f(p_{l+1}) \in \V^0 \cap e^{m-1}$. We now show that $k'< l$. We argue by contradiction and assume that $k'\geq l$. By (\ref{eqPfpropEmBound_Defk'}), $k'=l$. Then $p_l = p_{k'} \in \crit f|_\CC \subseteq \V^1$ and $p_l = p_{k'} \in \inte \bigl( e^{m-1} \bigr)$ (see (\ref{eqPfpropEmBound_k'TwoSides})). This contradicts the fact that no $(m-1)$-edge can contain a $1$-vertex in its interior. Thus $k'<l$.

We now show that
\begin{equation}  \label{eqPfpropEmBound_jump}
l-k'\geq m-1.
\end{equation}
Fix an arbitrary integer $i\in[k'+1,l]$. Since $\card E_{m,i} = \card E_{m,i-1} \in [8,+\infty)$ (see (\ref{eqPfpropEmBound_Defk'}) and (\ref{eqPfpropEmBound_Defk})), $f(E_{m,i}) \subseteq E_{m,i-1}$ (see (\ref{eqDefEm})), and $f^{l-k'}$ is injective on $e^{m-1} \supseteq E_{m,l}$ (see Proposition~\ref{propCellDecomp}~(i)), we get $f(E_{m,i}) = E_{m,i-1}$. By Proposition~\ref{propCellDecomp}~(i), $f^{l-k'}\bigl( e^{m-1} \bigr) \in \E^{m-1-l+k'}$. Since $f^{l-k'}\bigl( e^{m-1} \bigr) \supseteq f^{l-k'}(E_{m,l}) = E_{m,k'}$ and $f^{l-k'}$ is injective and continuous on $e^{m-1}$ (see Proposition~\ref{propCellDecomp}~(i)), it follows from (\ref{eqPfpropEmBound_k'TwoSides}) that $p_{k'} \in \inte \bigl( f^{l-k'}\bigl( e^{m-1} \bigr) \bigr)$. Since $p_{k'} \in \crit f|_\CC \subseteq \V^1$, we get $m-1-l+k' \leq 0$, proving (\ref{eqPfpropEmBound_jump}).

Hence by (\ref{eqPfpropEmBound_EmDouble}), (\ref{eqPfpropEmBound_Defk}), (\ref{eqPfpropEmBound_Defk'}), (\ref{eqPfpropEmBound_jump}), and the induction hypothesis, we get 
\begin{equation*}
\card E_{m,l+1} \leq 2 \card E_{m,l} = 2 \card E_{m,k+1} \leq 2 m 2^{\frac{k+1}{m}}  \leq 2 m 2^{\frac{k'}{m}} \leq m 2^{\frac{l+1}{m}}.
\end{equation*}
The induction step is now complete, establishing Proposition~\ref{propEmBound}.
\end{proof}

\begin{theorem}  \label{thmPressureOnC}
Let $f$, $\CC$, $d$, $\alpha$ satisfy the Assumptions. We assume in addition that $f(\CC)\subseteq \CC$. Let $\varphi \in \Holder{\alpha}(S^2,d)$ be a real-valued H\"{o}lder continuous function with an exponent $\alpha$. Recall the one-sided subshifts of finite type $\bigl( \Sigma_{A_{\e}}^+, \sigma_{A_{\e}} \bigr)$ and $\bigl( \Sigma_{A_{\ee}}^+, \sigma_{A_{\ee}} \bigr)$ constructed in Subsection~\ref{subsctDynOnC_Construction}. We denote by $\left(\V^0, f|_{\V^0}\right)$ the dynamical system on $\V^0=\V^0(f,\CC) = \post f$ induced by $f|_{\V^0} \: \V^0\rightarrow \V^0$. Then the following relations between the topological pressure of these systems hold:
\begin{equation*}
     P  (   \sigma_{A_{\ee}}, \varphi \circ \pi_{\e} \circ \pi_{\ee}  )
 =   P  (   \sigma_{A_{\e}},  \varphi \circ \pi_{\e}                  )
 =   P  (   f|_\CC,           \varphi|_\CC                            )
<  P(f,\varphi)
>    P  (   f|_{\V^0} ,       \varphi|_{\V^0}                         ) .
\end{equation*}
\end{theorem}

\begin{proof}
The identity $P  (   \sigma_{A_{\ee}}, \varphi \circ \pi_{\e} \circ \pi_{\ee}  )  =  P  (   \sigma_{A_{\e}},  \varphi \circ \pi_{\e}  )$ follows directly from Lemma~\ref{lmUnifBddToOneFactorPressure} and H\"{o}lder continuity of $\pi_{\e}$ and $\pi_{\ee}$ (see Proposition~\ref{propSFTs_C}).

The strict inequalities $P (f|_\CC, \varphi|_\CC) < P(f,\varphi)$ and $P ( f|_{\V^0}), \varphi|_{\V^0} ) < P(f,\varphi)$ follow from the uniqueness of the equilibrium state $\mu_\phi$ for the map $f$ and the potential $\varphi$ (see Theorem~\ref{thmEquilibriumState}~(i)), the fact that $\mu_\phi (\CC) = 0$ (see Theorem~\ref{thmEquilibriumState}~(iii)), and the Variational Principle (\ref{eqVPPressure}) (see for example, \cite[Theorem~3.4.1]{PrU10}).

We observe that since $(\CC,f|_\CC)$ is a factor of $\bigl(\Sigma_{A_{\e}}^+, \sigma_{A_{\e}}\bigr)$ with a factor map $\pi_{\e} \: \Sigma_{A_{\e}}^+ \rightarrow \CC$ (Proposition~\ref{propSFTs_C}~(ii)), it follows from \cite[Lemma~3.2.8]{PrU10} that $P  ( \sigma_{A_{\e}},  \varphi \circ \pi_{\e} ) \geq P  (  f|_\CC,  \varphi|_\CC )$. It remains to show $P  ( \sigma_{A_{\e}},  \varphi \circ \pi_{\e} ) \leq P  (  f|_\CC,  \varphi|_\CC )$. 

By Lemma~\ref{lmCellBoundsBM}~(ii) and Proposition~\ref{propCellDecomp}~(vii), no $1$-tile in $\X^1(f^n,\CC)$ joins opposite sides of $\CC$ for all sufficiently large $n\in\N$. Note that for all $m\in\N$, $P\bigl( f^m|_\CC, S^f_m \varphi|_\CC \bigr) = m P ( f|_\CC,  \varphi|_\CC  )$ and $P  \bigl( \sigma_{A_{\e}}^m,  S^{\sigma_{A_{\e}}}_m ( \varphi \circ \pi_{\e}) \bigr)  = mP  ( \sigma_{A_{\e}},  \varphi \circ \pi_{\e} )$ (see for example, \cite[Theorem~9.8]{Wal82}). It is clear that, without loss of generality, we can assume that no $1$-tile in $\X^1(f,\CC)$ joins opposite sides of $\CC$.

We define a sequence of finite open covers $\{\eta_i\}_{i\in\N_0}$ of $\CC$ by 
\begin{equation*}
\eta_i \coloneqq \bigl\{  \ae^i(v)  \,\big|\, v\in \CC \cap \V^i  \bigr\}
\end{equation*}
for $i\in\N_0$. We note that since we are considering the metric space $(\CC,d)$, $\ae^i(v)$ is indeed an open set for each $i\in\N_0$ and each $v \in \CC \cap \V^i$. By Lemma~\ref{lmCellBoundsBM}~(ii), 
\begin{equation*}
\lim\limits_{i\to+\infty} \max \{ \diam_d(V) \,|\, V\in\eta_i \} = 0.
\end{equation*}

Fix arbitrary integers $l,m,n\in\N$ with $l\geq m\geq 14$. Choose $U\in \bigvee\limits_{i=0}^n (f|_\CC)^{-i} (\eta_m)$ arbitrarily, say
\begin{equation*}
U = \bigcap\limits_{i=0}^n (f|_\CC)^{-i} (\ae^m(p_{n-i})),
\end{equation*}
where $p_0,p_1,\dots,p_n \in \CC\cap\V^m$. By Lemma~\ref{lmCoverByEdgePair},
\begin{equation*}
U \subseteq \bigcup\limits_{x\in E_m(p_n,\,p_{n-1},\,\dots,\,p_1;\,p_0)}  \ae^{m+n} (x),
\end{equation*}
where $E_m$ is defined in (\ref{eqDefEm}). It follows immediately from (\ref{eqDeg=SumLocalDegree}) and Proposition~\ref{propCellDecomp}~(i) and (v) that
\begin{equation*}
\card \bigl\{ e^{l+n} \in \E^{l+n} \,\big|\, e^{l+n} \subseteq e \bigr\}  \leq (\deg f)^{l-m} \card(\post f)
\end{equation*}
for each $(m+n)$-edge $e\in\E^{m+n}$. Thus we can construct a collection $\mathcal{E}^{l+n}(U) \subseteq \E^{l+n}$ of $(l+n)$-edges such that $U\subseteq \bigcup \mathcal{E}^{l+n}(U)$ by setting
\begin{equation*}
\mathcal{E}^{l+n}(U)  \coloneqq \bigcup\limits_{x\in E_m(p_n,\,p_{n-1},\,\dots,\,p_1;\,p_0)}   \bigl\{ e^{l+n} \in \E^{l+n} 
                                                                                   \,\big|\, e^{l+n} \subseteq \overline{\ae}^{m+n}(x)  \bigr\}
\end{equation*}
Then 
\begin{align}   \label{eqPfthmPressureOnC_EdgeNoBound}
       \card \bigl(  \mathcal{E}^{l+n}(U) \bigr) 
& \leq 2 (\deg f)^{l-m} \card(\post f) \card (E_m(p_n,\,p_{n-1},\,\dots,\,p_1;\,p_0))  \notag \\
& \leq 2 m 2^{\frac{n}{m}} (\deg f)^{l-m} \card(\post f),
\end{align}
where the last inequality follows from Proposition~\ref{propEmBound}.

Hence by (\ref{eqEquivDefByCoverForTopPressure}), Lemma~\ref{lmSnPhiBound}, and (\ref{eqPfthmPressureOnC_EdgeNoBound}), we get
\begin{align*}
&            P(f|_\CC, \varphi|_\CC)  \\
&  \qquad  = \lim\limits_{m\to+\infty}   \lim\limits_{l\to+\infty}   \lim\limits_{n\to+\infty}  \frac{1}{n} \log
                 \inf_\xi \Biggl\{ \sum\limits_{U\in\xi} \exp  \bigl( \sup  \bigl\{S_n^f\varphi(x) \,\big|\, x\in U \bigr\}  \bigr)    \Biggr\}     \\
&\qquad\geq  \lim\limits_{m\to+\infty}   \lim\limits_{l\to+\infty}   \lim\limits_{n\to+\infty}  \frac{1}{n} \log
                 \inf_\xi \Biggl\{ \sum\limits_{U\in\xi} \sum\limits_{e\in\mathcal{E}^{l+n}(U)}
                   \frac{  \exp  \bigl( \sup \bigl\{S_n^f\varphi(x) \,\big|\, x\in e\cap U \bigr\}  \bigr)}
                        {  \card  (  \mathcal{E}^{l+n}(U)   ) }  \Biggr\}  \\        
&\qquad\geq  \lim\limits_{m\to+\infty}   \lim\limits_{l\to+\infty}   \lim\limits_{n\to+\infty}  \frac{1}{n} \log
                 \inf_\xi \Biggl\{ \sum\limits_{U\in\xi} \sum\limits_{e\in\mathcal{E}^{l+n}(U)}
                   \frac{  \exp  \bigl( \sup    \bigl\{S_n^f\varphi(x) \,\big|\, x\in e \bigr\}  \bigr)}
                        { 2 m 2^{\frac{n}{m}} (\deg f)^{l-m}   D } \Biggr\}     \\ 
&\qquad\geq  \lim\limits_{m\to+\infty}   \lim\limits_{l\to+\infty}   \lim\limits_{n\to+\infty}  \frac{1}{n} \log
                   \sum\limits_{\substack{e\in\E^{l+n} \\ e\subseteq \CC  }}
                   \frac{  \exp  \bigl( \sup    \bigl\{S_n^f\varphi(x) \,\big|\, x\in e \bigr\}  \bigr)}
                        { 2 m 2^{\frac{n}{m}} (\deg f)^{l-m}   D }     \\  
&\qquad  =   \lim\limits_{m\to+\infty}   \lim\limits_{l\to+\infty}   \lim\limits_{n\to+\infty}  \frac{1}{n} 
              \Biggl( \log   \sum\limits_{\substack{e\in\E^{l+n} \\ e\subseteq \CC  }}
                              \exp \Bigl( \sup_{x\in e}   S_n^f\varphi(x)   \Bigr)
                      - \log \bigl( 2 m 2^{\frac{n}{m}} (\deg f)^{l-m}   D \bigr)   \Biggr)\\
&\qquad  =   \lim\limits_{l\to+\infty}   \lim\limits_{n\to+\infty}  \frac{1}{n} \log
                \sum\limits_{\substack{e\in\E^{l+n} \\ e\subseteq \CC  }}
                              \exp  \bigl( \sup \bigl\{ S_n^f\varphi(x) \,\big|\, x\in e  \bigr\}  \bigr)  ,                   
\end{align*}
where the constant $D>1$ is defined to be $D \coloneqq \card(\post f)  \exp (C_1 (\diam_d(S^2) )^\alpha  )$ with $C_1 = C_1(f,\CC,d,\varphi,\alpha)>0$ depending only on $f$, $\CC$, $d$, $\varphi$, and $\alpha$ defined in (\ref{eqC1Expression}), and the infima are taken over all finite open subcovers $\xi$ of $\bigvee\limits_{i=0}^n (f|_\CC)^{-i}(\eta_m)$, i.e., $\xi \in \Bigl\{ \varsigma \,\Big|\, \varsigma \subseteq \bigvee\limits_{i=0}^n (f|_\CC)^{-i}(\eta_m), \, \bigcup \varsigma = \CC \Bigr\}$. The second inequality follows from Lemma~\ref{lmSnPhiBound}, and the last inequality follows from the fact that 
\begin{equation*}
\CC\supseteq \bigcup \Bigl( \bigcup\limits_{U\in\xi}  \mathcal{E}^{l+n}(U) \Bigr) =  \bigcup\limits_{U\in\xi} \Bigl( \bigcup  \mathcal{E}^{l+n}(U)\Bigr) \supseteq  \bigcup \xi = \CC
\end{equation*}
and thus $\bigcup\limits_{U\in\xi}  \mathcal{E}^{l+n}(U) = \bigl\{ e\in\E^{l+n} \,\big|\, e\subseteq \CC \}$.

Finally, we will show that 
\begin{equation*}
 P  ( \sigma_{A_{\e}},  \varphi \circ \pi_{\e} ) 
 =  \lim\limits_{l\to+\infty}   \lim\limits_{n\to+\infty}  \frac{1}{n} \log
                \sum\limits_{\substack{e\in\E^{l+n} \\ e\subseteq \CC  }}
                              \exp  \bigl( \sup \bigl\{ S_n^f\varphi(x) \,\big|\, x\in e \bigr\}  \bigr).
\end{equation*}
We denote by $C_n(e_0,e_1,\dots,e_n)$ the $n$-cylinder set 
\begin{equation*}
C_n(e_0,e_1,\dots,e_n) \coloneqq \bigl\{ \{e'_i\}_{i\in\N_0} \in \Sigma_{A_{\e}}^+ \,\big|\, e'_i=e_i \text{ for all } i\in\N_0 \text{ with } i\leq n  \bigr\}
\end{equation*}  
in $\Sigma_{A_{\e}}^+$ containing $\{e_i\}_{i\in\N_0}$, for each $n\in \N_0$ and each $\{e_i\}_{i\in\N_0} \in \Sigma_{A_{\e}}^+$. For each $n\in\N_0$, we denote by $\mathfrak{C}_n$ the set all $n$-cylinder sets in $\Sigma_{A_{\e}}^+$, i.e., $\mathfrak{C}_n = \bigl\{ C_n(e_0,e_1,\dots,e_n) \,\big|\, \{e_i\}_{i\in\N_0} \in \Sigma_{A_{\e}}^+ \bigr\}$. Then it is easy to verify that for all $n,l\in\N_0$, $\mathfrak{C}_n$ is a finite open cover of $\Sigma_{A_{\e}}^+$, and $\bigvee\limits_{i=0}^n \sigma_{A_{\e}}^{-i} (\mathfrak{C}_l)  = \mathfrak{C}_{l+n}$. Hence by (\ref{eqEquivDefByCoverForTopPressure}), Lemma~\ref{lmCylinderIsTile}, and Proposition~\ref{propSFTs_C}~(i), we have
\begin{align*}
&          P  ( \sigma_{A_{\e}},  \varphi \circ \pi_{\e} ) \\
&\qquad =  \lim\limits_{l\to +\infty}  \lim\limits_{n\to +\infty}  \frac{1}{n} \log 
               \inf\Biggl\{ \sum\limits_{V\in\mathcal{V}}  \exp 
                      \biggl( \sup\limits_{\underline{z}\in V} S_n^{\sigma_{A_{\e}}} (\varphi \circ \pi_{\e})(\underline{z})   \biggr)
                             \,\Bigg|\,  \mathcal{V} \subseteq \bigvee\limits_{i=0}^{n} \sigma_{A_{\e}}^{-i}(\mathfrak{C}_l),\,
                                                        \bigcup \mathcal{V} = \Sigma_{A_{\e}}^+ \Biggr\}  \\
&\qquad =  \lim\limits_{l\to +\infty}  \lim\limits_{n\to +\infty}  \frac{1}{n} \log 
              \sum\limits_{V\in\mathfrak{C}_{l+n}}  \exp 
                      \Bigl( \sup  \Bigl\{ S_n^{\sigma_{A_{\e}}} (\varphi \circ \pi_{\e})(\underline{z}) \,\Big|\, \underline{z}\in V \Bigr\}   \Bigr)  \\
&\qquad =  \lim\limits_{l\to+\infty}   \lim\limits_{n\to+\infty}  \frac{1}{n} \log
                \sum\limits_{\substack{e\in\E^{l+n+1} \\ e\subseteq \CC  }}
                              \exp  \bigl( \sup \bigl\{ S_n^f \varphi(x) \,\big|\, x\in e \bigr\}  \bigr) .                               
\end{align*}
Hence, $P(f|_\CC, \varphi|_\CC) \geq P  ( \sigma_{A_{\e}},  \varphi \circ \pi_{\e} )$. The proof is therefore complete.
\end{proof}

\subsection{Deduction of Theorem~\ref{thmZetaAnalExt_InvC} from Theorem~\ref{thmZetaAnalExt_SFT}}  \label{subsctDynOnC_Reduction}

In this subsection, we give a proof of Theorem~\ref{thmZetaAnalExt_InvC} assuming Theorem~\ref{thmZetaAnalExt_SFT}.

\begin{proof}[Proof of Theorem~\ref{thmZetaAnalExt_InvC}]
We choose $N_f\in\N$ as in Remark~\ref{rmNf}. Note that $P \bigl( f^i, - s_0 S_i^f \phi \bigr) = i P(f, - s_0 \phi) = 0$ for each $i\in\N$ (see for example, \cite[Theorem~9.8]{Wal82}). We observe that by Lemma~\ref{lmCexistsL}, it suffices to prove the case $n=N_f = 1$. In this case, $F=f$, $\Phi=\phi$, and there exists an Jordan curve $\CC\subseteq S^2$ satisfying $f(\CC)\subseteq \CC$, $\post f\subseteq \CC$, and no $1$-tile in $\X^1(f,\CC)$ joins opposite sides of $\CC$.

In this proof, we write $l_\phi(\tau) \coloneqq \sum\limits_{y\in\tau} \phi(y)$ and $\deg_f(\tau) \coloneqq \prod\limits_{y\in\tau} \deg_f(y)$ for each primitive periodic orbit $\tau\in\Orb(f)$ and each $y\in\tau$. 

\smallskip

\emph{Claim~1.} The dynamical Dirichlet series $\DS_{f,\,\minus\phi,\,\deg_f}$ converges on the open half-plane $H_{s_0} = \{ s\in\C \,|\, \Re(s) > s_0 \}$ and extends to a non-vanishing holomorphic function on the closed half-plane $\overline{\H}_{s_0} = \{ s\in\C \,|\, \Re(s) \geq s_0 \}$ except for the simple pole at $s=s_0$. 

\smallskip

We first observe that by the continuity of the topological pressure (see for example, \cite[Theorem~3.6.1]{PrU10}) and Theorem~\ref{thmPressureOnC}, there exists a real number $\epsilon'_0 \in (0, \min \{\wt{\epsilon}_0, \, s_0 \} )$ such that $P (f|_{\V^0}, -(s_0 - \epsilon'_0) \phi|_{\V^0} ) < 0 $ and
\begin{equation*}
     P  (   \sigma_{A_{\ee}}, -(s_0 - \epsilon'_0) \varphi \circ \pi_{\e} \circ \pi_{\ee}  )
 =   P  (   \sigma_{A_{\e}},  -(s_0 - \epsilon'_0) \varphi \circ \pi_{\e}                  )
 <   0.
\end{equation*}
Here $\wt{\epsilon}_0 > 0$ is a constant from Theorem~\ref{thmZetaAnalExt_SFT} depending only on $f$, $\CC$, $d$, and $\phi$.

By Lemma~\ref{lmDynDirichletSeriesConv_general}, Remark~\ref{rmDynDirichletSeriesZetaFn}, Proposition~\ref{propSFT}~(ii), and the fact that $\phi$ is eventually positive, each of the zeta functions $\zeta_{ f|_{\V^0}, \, \minus \phi|_{\V^0} }$, $\zeta_{ \sigma_{A_{\e}}, \, \minus \phi \circsmall \pi_{\e} }$, and $\zeta_{ \sigma_{A_{\ee}}, \, \minus \phi \circsmall \pi_{\e} \circsmall \pi_{\ee} }$ converges uniformly and absolutely to a non-vanishing bounded holomorphic function on the closed half-plane $\overline{\H}_{s_0 - \epsilon'_0} = \{ s\in\C \,|\, \Re(s) \geq s_0 - \epsilon'_0 \}$.

On the other hand, for each $n\in\N$, we have $P_{1, (f|_{\V^0})^n} \subseteq P_{1, f^n}$, and by Proposition~\ref{propSFTs_C},
\begin{equation*}
             (\pi_{\e} \circ \pi_{\ee} )  \bigl(   P_{1, \sigma_{A_{\ee}}^n } \bigr)
\subseteq     \pi_{\e}                    \bigl(  P_{1, \sigma_{A_{\e }}^n } \bigr)
\subseteq                                         P_{1, (f|_\CC)^n }
\subseteq                                         P_{1, f^n }.
\end{equation*} 
Thus by (\ref{eqDefDynDirichletSeries}), (\ref{eqDefZetaFn}), and Theorem~\ref{thmNoPeriodPtsIdentity}, we get that for each $s\in \H_{s_0}$,
\begin{equation}    \label{eqDirichletSeriesIsCombZetaFns}
   \DS_{f,\,\minus\phi,\,\deg_f} (s) 
=  \zeta_{ \sigma_{A_{\ti}}, \, \minus \phi \circsmall \pi_{\ti} } (s) 
   \frac{   \zeta_{ \sigma_{A_{\e}}, \, \minus \phi \circsmall \pi_{\e} } (s)   \zeta_{ f|_{\V^0}, \, \minus \phi|_{\V^0} } (s)   }
        {   \zeta_{ \sigma_{A_{\ee}}, \, \minus \phi \circsmall \pi_{\e} \circsmall \pi_{\ee} } (s)  }.
\end{equation}

Claim~1 now follows from statement~(i) in Theorem~\ref{thmZetaAnalExt_SFT} and the discussion above.

\smallskip

Next, we observe that by (\ref{eqZetaFnOrbitForm_ThurstonMap}) and  (\ref{eqZetaFnOrbitForm_ThurstonMapDegree}) in Proposition~\ref{propZetaFnConv_s0},
\begin{equation}   \label{eqPfthmZetaAnalExt_InvC_zeta}
   \zeta_{f,\,\minus\phi} (s) \prod\limits_{\tau\in\Orb^>(f|_{\V^0})}  \Bigl(  1- e^{ - s l_\phi(\tau)  } \Bigr)   
=   \DS_{f,\,\minus\phi,\,\deg_f} (s)  \prod\limits_{\tau\in\Orb^>(f|_{\V^0})}  \Bigl(  1- \deg_f(\tau) e^{ - s l_\phi(\tau)  } \Bigr)                  
\end{equation}
for all $s\in\C$ with $\Re(s)> s_0$, where 
\begin{equation}  \label{eqPfthmZetaAnalExt_InvC_Orb>}
\Orb^>(f|_{\V^0}) \coloneqq \bigl\{\tau\in \Orb(f|_{\V^0}) \,\big|\,  \deg_f(\tau) >1 \bigr\}
\end{equation}
is a finite set since $\V^0=\post f$ is a finite set.

We denote, for each $\tau\in\Orb^>(f|_{\V^0})$,
\begin{equation*}
\beta_\tau \coloneqq \deg_f(\tau) e^{-s_0 l_\phi(\tau)}.
\end{equation*}

Fix an arbitrary $\tau \in \Orb^>(f|_{\V^0})$. We show now that $1-\beta_\tau \geq 0$. We argue by contradiction and assume that $\beta_\tau >1$. Let $k \coloneqq \card \tau$, and fix an arbitrary $y\in\tau$. Then $y\in P_{1,f^{km}}$ for each $m\in\N$. Thus by Proposition~\ref{propTopPressureDefPeriodicPts},
\begin{align*}
0 =     P(f,-s_0\phi)
 \geq & \lim\limits_{m\to+\infty} \frac{1}{km} \log \bigl(  \deg_{f^{km}} (y)  \exp \bigl(-s_0 S^f_{km} \phi (y) \bigr) \bigr)  \\
  =   & \lim\limits_{m\to+\infty} \frac{1}{km} \log \bigl( \beta_\tau^m \bigr)
  =     \frac{\log \beta_\tau }{k}
  >     0.
\end{align*}
This is a contradiction, proving $1-\beta_\tau \geq 0$ for each $\tau\in\Orb^>(f|_{\V^0})$.

\smallskip

\emph{Claim~2.} We have $1-\beta_\tau > 0$ for each $\tau\in\Orb^>(f|_{\V^0})$.

\smallskip

We argue by contradiction and assume that there exists $\eta \in \Orb^>(f|_{\V^0})$ with $1-\beta_\eta = 0$. We define a function $w\: S^2\rightarrow \C$ by
\begin{equation*}
w(x) \coloneqq \begin{cases} \deg_f(x) & \text{if } x\in S^2 \setminus \eta, \\ 0  & \text{otherwise}. \end{cases}
\end{equation*}

Fix an arbitrary real number $a>s_0$. By (\ref{eqLocalDegreeProduct}), Proposition~\ref{propTopPressureDefPeriodicPts}, and Corollary~\ref{corS0unique}, for each $n\in\N$,
\begin{align*}
&             \limsup\limits_{n\to+\infty} \frac{1}{n} \log \sum\limits_{y\in P_{1,f^n}} \exp(-a S_n\phi(y)) \prod\limits_{i=0}^{n-1} w\bigl(f^i(y)\bigr) \\
&\qquad  \leq \limsup\limits_{n\to+\infty} \frac{1}{n} \log \sum\limits_{y\in P_{1,f^n}} \deg_{f^n}(y) \exp(-a S_n\phi(y))    \\
&\qquad  = P(f, -a\phi) <0.
\end{align*}
Hence by Lemma~\ref{lmDynDirichletSeriesConv_general} and Theorem~\ref{thmETMBasicProperties}~(ii), $\DS_{f,\,\minus\phi,\,w}(s)$ converges uniformly and absolutely on the closed half-plane $\overline{\H}_a$, and 
\begin{align}  \label{eqPfthmZetaAnalExt_InvC_DynDSeriesIdenw}
\DS_{f,\,\minus\phi,\,w}(s) & = \prod\limits_{\tau\in\Orb(f)\setminus\{\eta\}}  \Bigl(1 - \deg_f(\tau) e^{-s l_\phi(\tau)   } \Bigr)^{-1}  \\
                            & = \DS_{f,\,\minus\phi,\,\deg_f}(s) \Bigl(1 - \deg_f(\eta) e^{-s l_\phi(\eta)  } \Bigr)  \notag
\end{align}
for $s\in \overline{\H}_a$. Note that by our assumption that $1-\beta_\tau = 0$, we know that $1 - \deg_f(\eta) e^{-s l_\phi(\eta)  }$ is an entire function with simple zeros at $s=s_0 +\I j h_0 $, $j\in\Z$, where $h_0 \coloneqq \frac{2\pi}{l_\phi(\eta)}$. Note that $l_\phi(\eta) > 0$ since $\phi$ is eventually positive (see Definition~\ref{defEventuallyPositive}). Since $\DS_{f,\,\minus\phi,\,\deg_f}$ has a non-vanishing holomorphic extension to $\overline{\H}_{s_0}$ except a simple pole at $s=s_0$, we get from (\ref{eqPfthmZetaAnalExt_InvC_DynDSeriesIdenw}) that $\DS_{f,\,\minus\phi,\,w}$ has a holomorphic extension to $\overline{\H}_{s_0}$ with $\DS_{f,\,\minus\phi,\,w}(s_0) \neq 0$ and $\DS_{f,\,\minus\phi,\,w}(s_0 + \I jh_0) = 0$ for each $j\in\Z \setminus \{0\}$.

On the other hand, for each $s\in \overline{\H}_a$,
\begin{equation} \label{eqPfthmZetaAnalExt_InvC_DynDSeriesIdenwBound}
      \AbsBigg{ \sum\limits_{n=1}^{+\infty}  \frac{1}{n} \sum\limits_{x\in P_{1,f^n}} e^{-s S_n\phi(x)} \prod\limits_{i=0}^{n-1} w\bigl(f^i(x)\bigr) }
\leq  \sum\limits_{n=1}^{+\infty}  \frac{1}{n} \sum\limits_{x\in P_{1,f^n}} e^{-\Re(s) S_n\phi(x)} \prod\limits_{i=0}^{n-1} w\bigl(f^i(x)\bigr).
\end{equation}
Since $a>s_0$ is arbitrary, it follows from (\ref{eqDefDynDirichletSeries}), (\ref{eqPfthmZetaAnalExt_InvC_DynDSeriesIdenwBound}), and $\DS_{f,\,\minus\phi,\,w}(s_0) \neq 0$ that
\begin{equation*}
\limsup_{a\to s_0^+}   \AbsBigg{ \sum\limits_{n=1}^{+\infty}  \frac{1}{n} \sum\limits_{x\in P_{1,f^n}} e^{- (a + \I b) S_n\phi(x)} \prod\limits_{i=0}^{n-1} w\bigl(f^i(x)\bigr) }  < +\infty
\end{equation*}
for each $b\in \R$. By (\ref{eqDefDynDirichletSeries}), this is a contradiction to the fact that $\DS_{f,\,\minus\phi,\,w}$ has a holomorphic extension to $\overline{\H}_{s_0}$ with $\DS_{f,\,\minus\phi,\,w}(s_0 + \I jh_0) = 0$ for each $j\in\Z \setminus \{0\}$. Claim~2 is now established.

\smallskip

Hence $\prod\limits_{\tau \in \Orb^>(f|_{\V^0})}  \frac{  1 - \deg_f(\tau) \exp(-s l_\phi(\tau))  }{  1 -  \exp(-s l_\phi(\tau))   }$ is uniformly bounded away from $0$ and $+\infty$ on the closed half-plane $\overline{\H}_{s_0-\epsilon_0}$ for some $\epsilon_0 \in ( 0, \epsilon'_0 )$.

Statement~(i) in Theorem~\ref{thmZetaAnalExt_InvC} now follows from Claim~1 and (\ref{eqPfthmZetaAnalExt_InvC_zeta}).

\smallskip

To verify statement~(ii) in Theorem~\ref{thmZetaAnalExt_InvC}, we assume that $\phi$ satisfies the $\alpha$-strong non-integrability condition. By statement~(ii) in Theorem~\ref{thmZetaAnalExt_SFT} and the proof of Claim~1, $\DS_{f,\,\minus\phi,\,\deg_f}$ extends to a non-vanishing holomorphic function on $\overline{\H}_{s_0 - \epsilon'_0 }$ except for the simple pole at $s=s_0$. Moreover, for each $\epsilon >0$, there exists a constant $C'_\epsilon >0$ such that
\begin{equation*}  
     \exp\left( - C'_\epsilon \abs{\Im(s)}^{2+\epsilon} \right) 
\leq \Absbig{ \DS_{ f,\,\minus\phi,\,\deg_f} (s) }
\leq \exp\left(   C'_\epsilon \abs{\Im(s)}^{2+\epsilon} \right) 
\end{equation*}
for all $s\in\C$ with $\abs{\Re(s) - s_0} < \epsilon'_0$ and $\abs{\Im(s)}  \geq b_\epsilon$, where $b_\epsilon \coloneqq \wt{b}_\epsilon>0$ is a constant from Theorem~\ref{thmZetaAnalExt_SFT} depending only on $f$, $\CC$, $d$, $\phi$, and $\epsilon$.

Therefore, statement~(ii) in Theorem~\ref{thmZetaAnalExt_InvC} holds for  $a_\epsilon \coloneqq \min \{ \epsilon_0, \, \wt{a}_\epsilon \}>0$, $b_\epsilon = \wt{b}_\epsilon>0$, and some constant $C_{\epsilon} > C'_{\epsilon}>0$ depending only on $f$, $\CC$, $d$, $\phi$, and $\epsilon$. 
\end{proof}

\section{Non-local integrability}  \label{sctNLI}
This section is devoted to characterizations of a necessary condition, called \emph{non-local integrability condition}, on the potential $\phi\: S^2 \rightarrow \R$ for the Prime Orbit Theorems for expanding Thurston maps. The characterizations are summarized in Theorem~\ref{thmNLI}. In particular, a real-valued H\"{o}lder continuous $\phi$ on $S^2$ is non-locally integrable if and only if $\phi$ is co-homologous to a constant in the set of real-valued continuous functions on $S^2$. As we will eventually show, such a condition is actually equivalent in our context to the Prime Orbit Theorem without an error term (see Theorem~\ref{thmPrimeOrbitTheorem}). In our proof of Theorem~\ref{thmNLI}, we use the notion of orbifolds introduced in general by W.~P.~Thurston in 1970s in his study of geometry of $3$-manifolds (see \cite[Chapter~13]{Th80}).

\subsection{Definition and charactizations}   \label{subsctNLI_Def}

Let $f\: S^2\rightarrow S^2$ be an expanding Thurston map, $d$ be a visual metric on $S^2$ for $f$, and $\CC\subseteq S^2$ be a Jordan curve satisfying $f(\CC)\subseteq \CC$ and $\post f\subseteq \CC$. Recall the one-sided subshift of finite type $\bigl(\Sigma_{A_{\ti}}^+,\sigma_{A_{\ti}} \bigr)$ associated to $f$ and $\CC$ defined in Proposition~\ref{propTileSFT}. In this section, we write $\Sigma_{f,\, \CC}^+ \coloneqq \Sigma_{A_{\ti}}^+$ and $\sigma \coloneqq \sigma_{A_{\ti}}$, i.e.,
\begin{equation}    \label{eqDefSigma+}
\Sigma_{f,\,\CC}^+ = \bigl\{\{X_i\}_{i\in\N_0} \,\big|\, X_i\in \X^1(f,\CC) \text{ and } f(X_i) \supseteq X_{i+1},\text{ for } i\in\N_0 \bigr\},
\end{equation}
and $\sigma$ is the left-shift operator defined by $\sigma(\{X_i\}_{i\in\N_0} ) = \{X_{i+1}\}_{i\in\N_0}$ for $\{X_i\}_{i\in\N_0} \in \Sigma_{f,\,\CC}^+$.

Similarly, we define
\begin{equation}   \label{eqDefSigma-}
\Sigma_{f,\,\CC}^- \coloneqq \bigl\{\{X_{\minus i}\}_{i\in\N_0} \,\big|\, X_{\minus i}\in \X^1(f,\CC) \text{ and } f \bigl(X_{\minus (i+1)}\bigr) \supseteq X_{\minus i},\text{ for } i\in\N_0 \bigr\}.
\end{equation}
For each $X\in\X^1(f,\CC)$, since $f$ is injective on $X$ (see Proposition~\ref{propCellDecomp}~(i)), we denote the inverse branch of $f$ restricted on $X$ by $f_X^{-1}\: f(X) \rightarrow X$, i.e., $f_X^{-1} \coloneqq (f|_X)^{-1}$.

Let $\psi\in \Holder{\alpha}((S^2,d),\C)$ be a complex-valued H\"{o}lder continuous function with an exponent $\alpha\in (0,1]$. For each $\xi=\{ \xi_{\minus i} \}_{i\in\N_0} \in  \Sigma_{f,\,\CC}^-$, we define the function
\begin{equation}   \label{eqDelta}
\Delta^{f,\,\CC}_{\psi,\,\xi} (x,y) \coloneqq \sum\limits_{i=0}^{+\infty} \bigl( \bigl(\psi \circ f^{-1}_{\xi_{\minus i}} \circ \cdots \circ f^{-1}_{\xi_{0}}\bigr) (x) -  \bigl( \psi \circ f^{-1}_{\xi_{\minus i}} \circ \cdots \circ f^{-1}_{\xi_{0}}\bigr) (y)   \bigr)
\end{equation}
for $(x,y)\in \bigcup\limits_{\substack{X\in\X^1(f,\CC) \\ X\subseteq f(\xi_0)}}X \times X$.  

We will see in the following lemma that the series in (\ref{eqDelta}) converges.

\begin{lemma}  \label{lmDeltaHolder}
Let $f$, $\CC$, $d$, $\psi$, $\alpha$ satisfy the Assumptions. We assume in addition that $f(\CC)\subseteq \CC$. Let $\xi=\{ \xi_{\minus i} \}_{i\in\N_0} \in  \Sigma_{f,\,\CC}^-$. Then for each $X\in\X^1(f,\CC)$ with $X\subseteq f(\xi_0)$ and each triple of $x,y,z\in X$, we get that $\Delta^{f,\,\CC}_{\psi,\,\xi} (x,y)$ as a series defined in (\ref{eqDelta}) converges absolutely and uniformly in $x,y\in X$, and moreover, the identity 
\begin{equation}   \label{eqDeltaDifference}
\Delta^{f,\,\CC}_{\psi,\,\xi} (x,y)   = \Delta^{f,\,\CC}_{\psi,\,\xi} (z,y) - \Delta^{f,\,\CC}_{\psi,\,\xi} (z,x)
\end{equation}
holds with
\begin{equation*}
\AbsBig{ \Delta^{f,\,\CC}_{\psi,\,\xi} (x,y) }   \leq C_1 d(x,y)^\alpha,
\end{equation*} 
where $C_1=C_1(f,\CC,d,\psi,\alpha)$ is a constant depending on $f$, $\CC$, $d$, $\psi$, and $\alpha$ from Lemma~\ref{lmSnPhiBound}.
\end{lemma}

\begin{proof}
We fix $X\in\X^1(f,\CC)$ with $X\subseteq f(\xi_0)$. By Proposition~\ref{propCellDecomp}~(i) and Lemma~\ref{lmCellBoundsBM}~(ii), for each $i\in\N_0$,
\begin{align} \label{eqPflmDeltaHolder}
&         \Absbig{    \bigl( \psi \circ f^{-1}_{\xi_{\minus i}} \circ \cdots \circ f^{-1}_{\xi_{0}}\bigr) (x) 
                   -  \bigl( \psi \circ f^{-1}_{\xi_{\minus i}} \circ \cdots \circ f^{-1}_{\xi_{0}}\bigr) (y)   } \notag \\
&\qquad   \leq \Hseminorm{\alpha,\, (S^2,d)}{\psi} \diam_d\bigl( \bigl(  f^{-1}_{\xi_{\minus i}} \circ \cdots \circ f^{-1}_{\xi_{0}} \bigr) (X) \bigr)\\
&\qquad   \leq \Hseminorm{\alpha,\, (S^2,d)}{\psi}  C^\alpha \Lambda^{-i\alpha-\alpha}  \notag  
\end{align}
for $x,y\in X$, where $C\geq 1$ is a constant from Lemma~\ref{lmCellBoundsBM}. Thus the series on the right-hand side of (\ref{eqDelta}) converges absolutely. Hence by (\ref{eqDelta}),  $\Delta^{f,\,\CC}_{\psi,\,\xi} (x,y) =  \Delta^{f,\,\CC}_{\psi,\,\xi} (z,y) - \Delta^{f,\,\CC}_{\psi,\,\xi} (z,x)$. Moreover, by Proposition~\ref{propCellDecomp}~(i), Lemma~\ref{lmSnPhiBound}, and (\ref{eqPflmDeltaHolder}), for each pair of $x,y\in X$, and each $j\in\N$,
\begin{align*}
\AbsBig{ \Delta^{f,\,\CC}_{\psi,\,\xi} (x,y)}   
=     &   \Absbigg{\sum\limits_{i=0}^{+\infty} \bigl(   \bigl( \psi \circ f^{-1}_{\xi_{\minus i}} \circ \cdots \circ f^{-1}_{\xi_{0}}\bigr) (x) 
                                                                          -  \bigl( \psi \circ f^{-1}_{\xi_{\minus i}} \circ \cdots \circ f^{-1}_{\xi_{0}}\bigr) (y) \bigr) }\\
\leq& \Absbigg{\sum\limits_{i=0}^{j-1}     \bigl(   \bigl( \psi \circ f^{-1}_{\xi_{\minus i}} \circ \cdots \circ f^{-1}_{\xi_{0}}\bigr) (x) 
                                                        -  \bigl( \psi \circ f^{-1}_{\xi_{\minus i}} \circ \cdots \circ f^{-1}_{\xi_{0}}\bigr) (y) \bigr) }\\
       &   + \sum\limits_{i=j}^{+\infty} \Hseminorm{\alpha,\, (S^2,d)}{\psi} C^\alpha \Lambda^{-i\alpha-\alpha} \\
\leq& C_1 d(x,y)^\alpha + \Hseminorm{\alpha,\, (S^2,d)}{\psi} C^\alpha (1-\Lambda^\alpha)^{-1}  \Lambda^{-j\alpha-\alpha}.
\end{align*}
We complete our proof by taking $j$ to infinity.
\end{proof}

\begin{definition}[Temporal distance]  \label{defTemporalDist}
Let $f$, $\CC$, $d$, $\psi$, $\alpha$ satisfy the Assumptions. We assume in addition that $f(\CC)\subseteq \CC$. For $\xi=\{ \xi_{\minus i} \}_{i\in\N_0} \in  \Sigma_{f,\,\CC}^-$ and $\eta=\{ \eta_{\minus i} \}_{i\in\N_0} \in  \Sigma_{f,\,\CC}^-$ with $f(\xi_0) = f(\eta_0)$, we define the \defn{temporal distance} $\psi^{f,\,\CC}_{\xi,\,\eta}$ as 
\begin{equation}  \label{eqDefTemporalDist}
\psi^{f,\,\CC}_{\xi,\,\eta}(x,y) \coloneqq \Delta^{f,\,\CC}_{\psi,\,\xi} (x,y) - \Delta^{f,\,\CC}_{\psi,\,\eta} (x,y)
\end{equation}
for 
\begin{equation*}
(x,y)\in \bigcup\limits_{\substack{X\in\X^1(f,\CC) \\ X\subseteq f(\xi_0)}}X \times X.
\end{equation*}
\end{definition}

Recall that $f^n$ is an expanding Thurston map with $\post f^n = \post f$ for each expanding Thurston map $f\: S^2\rightarrow S^2$ and each $n\in\N$ (see Remark~\ref{rmExpanding}).

\begin{definition}[Local integrability]   \label{defLI}
Let $f\: S^2\rightarrow S^2$ be an expanding Thurston map and $d$ a visual metric on $S^2$ for $f$. A complex-valued H\"{o}lder continuous function $\psi \in \Holder{\alpha}((S^2,d),\C)$ is \defn{locally integrable} (with respect to $f$ and $d$) if for each natural number $n\in\N$, and each Jordan curve $\CC\subseteq S^2$ satisfying $f^n(\CC)\subseteq \CC$ and $\post f \subseteq \CC$, we have 
\begin{equation}  \label{eqDefLI}
\bigl(S_n^f \psi\bigr)^{f^n,\,\CC}_{\xi,\,\eta}(x,y)=0
\end{equation}
for all $\xi=\{ \xi_{\minus i} \}_{i\in\N_0} \in  \Sigma_{f^n,\,\CC}^-$ and $\eta=\{ \eta_{\minus i} \}_{i\in\N_0} \in  \Sigma_{f^n,\,\CC}^-$ satisfying $f^n(\xi_0) = f^n(\eta_0)$, and all
\begin{equation*}
(x,y)\in \bigcup\limits_{\substack{X\in\X^1(f^n,\CC) \\ X\subseteq f^n(\xi_0)}}X \times X.
\end{equation*}

The function $\psi$ is \defn{non-locally integrable} if it is not locally integrable.
\end{definition}

The main theorem of this section is the following characterization of the local integrability condition.

\begin{theorem}[Characterization of the local integrability condition]  \label{thmNLI}
Let $f\: S^2\rightarrow S^2$ be an expanding Thurston map and $d$ a visual metric on $S^2$ for $f$. Let $\psi \in \Holder{\alpha}((S^2,d),\C)$ be a complex-valued H\"{o}lder continuous function with an exponent $\alpha\in (0,1]$. Then the following statements are equivalent.

\begin{enumerate}
\smallskip
\item[(i)] The function $\psi$ is locally integrable (in the sense of Definition~\ref{defLI}).

\smallskip
\item[(ii)] There exists $n\in\N$ and a Jordan curve $\CC\subseteq S^2$ with $f^n(\CC)\subseteq \CC$ and $\post f\subseteq \CC$ such that
\begin{equation*}
\bigl(S_n^f \psi\bigr)^{f^n,\,\CC}_{\xi,\,\eta}(x,y) = 0
\end{equation*}
for all $\xi=\{ \xi_{\minus i} \}_{i\in\N_0} \in  \Sigma_{f^n,\,\CC}^-$ and $\eta=\{ \eta_{\minus i} \}_{i\in\N_0} \in  \Sigma_{f^n,\,\CC}^-$ with $f^n(\xi_0) = f^n(\eta_0)$, and all $(x,y)\in \bigcup\limits_{\substack{X\in\X^1(f^n,\CC) \\ X\subseteq f^n(\xi_0)}}X \times X$.

\smallskip
\item[(iii)] The function $\psi$ is co-homologous to a constant in $\CCC(S^2,\C)$, i.e.,
\begin{equation*}
\psi= K+ \beta \circ f - \beta
\end{equation*}
for some $K\in\C$ and $\beta \in \CCC(S^2,\C)$.

\smallskip
\item[(iv)] The function $\psi$ is co-homologous to a constant in $\Holder{\alpha}((S^2,d),\C)$, i.e.,
\begin{equation*}
\psi= K+ \tau\circ f -  \tau
\end{equation*}
for some $K\in\C$ and $\tau\in \Holder{\alpha}((S^2,d),\C)$.

\smallskip
\item[(v)] There exists $n\in\N$ and a Jordan curve $\CC \subseteq S^2$ with $f^n(\CC) \subseteq \CC$ and $\post f \subseteq \CC$ such that the following statement holds for $F\coloneqq f^n$, $\Psi\coloneqq S_n^f \psi$, the one-sided subshift of finite type $\bigl( \Sigma_{A_{\ti}}^+, \sigma_{A_{\ti}} \bigr)$ associated to $F$ and $\CC$ defined in Proposition~\ref{propTileSFT}, and the factor map $\pi_\ti \: \Sigma_{A_{\ti}}^+ \rightarrow S^2$ defined in (\ref{eqDefTileSFTFactorMap}):

\smallskip

The function $\Psi \circ \pi_\ti$ is co-homologous to a constant multiple of an integer-valued continuous function in $\CCC\bigl( \Sigma_{A_{\ti}}^+ , \C \bigr)$, i.e., 
\begin{equation*}
\Psi \circ \pi_\ti = K M + \varpi \circ \sigma_{A_{\ti}} - \varpi
\end{equation*}
for some $K\in\C$, $M \in \CCC\bigl( \Sigma_{A_{\ti}}^+ , \Z \bigr)$, and $\varpi \in \CCC\bigl( \Sigma_{A_{\ti}}^+ , \C \bigr)$.
\end{enumerate}

\smallskip

If, in addition, $\psi$ is real-valued, then the above statements are equivalent to
\begin{enumerate}
\smallskip
\item[(vi)] The equilibrium state $\mu_\psi$ for $f$ and $\psi$ is equal to the measure of maximal entropy $\mu_0$ of $f$.
\end{enumerate}
\end{theorem}

We will prove Theorem~\ref{thmNLI} in Subsection~\ref{subsctNLI_Proof} after we introduce orbifolds associated to Thurston maps in the next subsection.

It follows from Theorem~\ref{thmNLI} and Theorem~\ref{thmETMBasicProperties}~(ii) that the non-local integrability condition is a necessary condition for the Prime Orbit Theorem.

\begin{cor}  \label{corNecessary}
Let $f$, $d$, $\phi$, $s_0$ satisfy the Assumptions. Let $\pi_{f,\phi}(T)$, for $T>0$, denote the number of primitive periodic orbits $\tau\in \Orb(f)$ with weighted length $l_\phi(\tau)$ no larger than $T$ (as defined in (\ref{eqDefPiT}) and (\ref{eqDefComplexLength})). Then $\pi_{f,\phi}(T) \sim \operatorname{Li} \bigl( e^{s_0 T} \bigr)$ as $T\to+\infty$ implies that $\phi$ is non-locally integrable (in the sense of Definition~\ref{defLI}).
\end{cor}

\begin{proof}
We assume that $\pi_{f,\phi}(T) \sim \operatorname{Li} \bigl( e^{s_0 T} \bigr)$ as $T\to+\infty$. Thus
\begin{equation}   \label{eqPfcorNecessary_alt}
\pi_{f,\phi}(T) \sim \frac{ e^{s_0 T} }{s_0 T} \qquad\qquad \text{as }  T\to+\infty.
\end{equation}

We argue by contradiction and assume that $\phi$ is locally integrable. Then by Theorem~\ref{thmNLI},
\begin{equation}  \label{eqPfcorNecessary_cohom}
\phi = K + \beta \circ f - \beta
\end{equation}
for some $K\in\C$ and $\beta \in \CCC((S^2,d),\C)$. By taking the real part on both sides of (\ref{eqPfcorNecessary_cohom}) and due to the fact that $\phi$ is eventually positive (see Definition~\ref{defEventuallyPositive}), we can assume without loss of generality that $K\in\R$ with $K>0$. By (\ref{eqPfcorNecessary_cohom}), $S_n\phi (x) = nK$ for each $n\in\N$ and each $x\in P_{1,f^n}$.

We note that by Proposition~\ref{propTopPressureDefPeriodicPts} and Theorem~\ref{thmETMBasicProperties}~(ii),
\begin{align*}
          P \biggl( f, - \frac{\log(\deg f)}{K} \phi \biggr) 
=   & \lim\limits_{n\to+\infty} \frac{1}{n} \log \sum\limits_{y\in P_{1,f^n}} \deg_{f^n} (y) \exp \biggl(  - \frac{\log(\deg f)}{K} nK \biggr)  \\
=   & \lim\limits_{n\to+\infty} \frac{1}{n} \log  \biggl(   \frac{ (\deg f)^n + 1 }{ (\deg f)^n } \biggr)  
=      0.
\end{align*}
Thus by Corollary~\ref{corS0unique}, $s_0 = \frac{\log(\deg f)}{K}$.

Denote a finite set $M\coloneqq \post f \cap \bigcup_{i=1}^{+\infty} P_{1,f^i}$ consisting periodic postcritical points. Observe that we can choose positive integers $N_1 \leq N_2 \leq \cdots \leq N_m \leq \cdots$ such that for each $m\in\N$ and each integer $n\geq N_m$,
\begin{equation}   \label{eqPfcorNecessary_M}
\card M < \frac{1}{m} (\deg f)^n.
\end{equation}
Then by \cite[Lemma~5.11]{Li16}, there exist constants $C>0$ and $\epsilon\in(0,1]$ depending only on $f$ such that for each $m\in\N$,
\begin{equation}   \label{eqPfcorNecessary_MdegSum}
\frac{1}{ (\deg f)^n }   \sum\limits_{x\in M} \deg_{f^n}(x) \leq C m^{-\epsilon} \qquad\qquad \text{for } n\geq N_m.
\end{equation}

Note that if $x\in S^2 \setminus M$ is periodic, 
\begin{equation}   \label{eqPfcorNecessary_LocalDeg1}
\deg_{f^i}(x) = 1 \qquad \text{for all } i\in\N.
\end{equation} 

\smallskip

\emph{Case~1.} $\deg f \geq 3$. Fix arbitrary $m\in\N$ and $n\geq N_m$. Then
\begin{equation}  \label{eqPfcorNecessary_Case1}
\frac{ e^{ s_0 n K} }{ s_0 n K }    =  \frac{ (\deg f)^n }{ n \log ( \deg f ) }  \leq  \frac{ (\deg f)^n }{ n } \cdot \frac{1}{\log 3}.
\end{equation}
On the other hand, by Theorem~\ref{thmETMBasicProperties}~(ii), (\ref{eqPfcorNecessary_LocalDeg1}), and (\ref{eqPfcorNecessary_MdegSum}),
\begin{align*}
\pi_{f,\phi}(nK)
\geq &\sum\limits_{i \mid n} \card \Orb(i,f) 
=     \sum\limits_{i \mid n} \frac{ \card  P_{i,f}     }{i}  
\geq  \sum\limits_{i \mid n} \frac{ \card  P_{i,f}     }{n}  \\
\geq& \frac{ \card (P_{1,f^n}   \setminus M )  }{n}
\geq \frac{ (1-Cm^{-\epsilon})  (\deg f)^n}{n}.
\end{align*}
Combining the above with (\ref{eqPfcorNecessary_Case1}) we get
\begin{equation*}
         \limsup\limits_{n\to+\infty} \frac{\pi_{f,\phi}( nK )}{ \frac{\exp (s_0 n K)}{s_0 n K} }
\geq \limsup\limits_{m\to+\infty} \limsup\limits_{n\to+\infty}  ( 1-Cm^{-\epsilon} ) \log 3 >1.
\end{equation*}
This contradicts (\ref{eqPfcorNecessary_alt}).

\smallskip

\emph{Case~2.} $\deg f = 2$. Fix arbitrary $m\in\N$ and $n> N_m$. Since $n$ and $n-1$ are coprime, by Theorem~\ref{thmETMBasicProperties}~(ii), (\ref{eqPfcorNecessary_LocalDeg1}) and (\ref{eqPfcorNecessary_MdegSum}),
\begin{align*}
              \pi_{f,\phi}(nK)
\geq &  \sum\limits_{i \mid n} \card \Orb(i,f)   +   \sum\limits_{j \mid n-1} \card \Orb(j,f)  - \card \Orb(1,f)  \\
\geq &  \sum\limits_{i \mid n} \frac{ \card  P_{i,f}     }{i}  +   \sum\limits_{j \mid n-1} \frac{ \card  P_{j,f}     }{j}  -  (\deg f +1) \\
\geq &  \frac{ \card (P_{1,f^n}   \setminus M )  }{n}  +  \frac{ \card (P_{1,f^{n-1}}   \setminus M )  }{n-1}    -   3  \\
\geq &  \frac{ (1-Cm^{-\epsilon})  2^n + 1}{n}   +  \frac{ (1-Cm^{-\epsilon})  2^{n-1} + 1}{n-1}  - 3 \\
\geq & \frac{ (1-Cm^{-\epsilon})  2^n }{n}  \cdot \frac{3}{2}  - 3.
\end{align*}
So
\begin{equation*}
             \limsup\limits_{n\to+\infty} \frac{\pi_{f,\phi}( nK )}{ \frac{e^{s_0 n K}}{s_0 n K} } 
\geq     \limsup\limits_{m\to+\infty} \limsup\limits_{n\to+\infty}  \biggl( \frac{ (1-Cm^{-\epsilon})  2^n }{n}  \cdot \frac{3}{2}  - 3 \biggr) \frac{ n \log 2}{ 2^n }
  =         \frac{3}{2} \log 2>1.
\end{equation*}
This contradicts (\ref{eqPfcorNecessary_alt}).
\end{proof}

\subsection{Orbifolds and universal orbifold covers}   \label{subsctNLI_Orbifold}

In order to establish Theorem~\ref{thmNLI}, we need to consider orbifolds associated to Thurston maps. An orbifold is a space that is locally represented as a quotient of a model space by a group action (see \cite[Chapter~13]{Th80}). For the purpose of this work, we restrict ourselves to orbifolds on $S^2$. In this context, only cyclic groups can occur, so a simpler defintion (than that of W.~P.~Thurston) will be used. We follow closely the setup from \cite{BM17}.

An \defn{orbifold} is a pair $\mathcal{O} = (S, \alpha)$, where $S$ is a surface and $\alpha\: S \rightarrow \widehat{\N} = \N \cup \{+\infty\}$ is a map such that the set of points $p\in S$ with $\alpha(p) \neq 1$ is a discrete set in $S$, i.e., it has no limit points in $S$. We call such a function $\alpha$ a \defn{ramification function} on $S$. The set $\supp(\alpha) \coloneqq \{p\in S \,|\, \alpha(p)\geq 2 \}$ is the \defn{support} of $\alpha$. We will only consider orbifolds with $S=S^2$, an oriented $2$-sphere, in this paper.

The \defn{Euler characteristic} of an orbifold $\mathcal{O}= (S^2,\alpha)$ is defined as
\begin{equation*}
\chi(\mathcal{O}) \coloneqq 2 - \sum_{x\in S^2} \biggl(1- \frac{1}{\alpha(x)} \biggr),
\end{equation*}
where we use the convention $\frac{1}{+\infty} = 0$, and the terms in the summation are nonzero on a finite set of points. The orbifold $\mathcal{O}$ is \defn{parabolic} if $\chi(\mathcal{O}) = 0$ and \defn{hyperbolic} if $\chi(\mathcal{O}) < 0$.

Every Thurston map $f$ has an associated orbifold $\mathcal{O}_f = (S^2,\alpha_f)$, which plays an important role in this section.

\begin{definition}  \label{defRamificationFn}
Let $f\: S^2\rightarrow S^2$ be a Thurston map. The \defn{ramification function} of $f$ is the map $\alpha_f\: S^2\rightarrow\widehat{\N}$ defined as
\begin{equation} \label{eqDefRamificationFn}
\alpha_f(x) \coloneqq \lcm \bigl\{\deg_{f^n}(y) \,\big|\, y\in S^2, n\in \N, \text{ and } f^n(y)=x \bigr\}
\end{equation}
for $x\in S^2$.
\end{definition}
Here $\widehat{\N} = \N\cup \{+\infty\}$ with the order relations $<$, $\leq$, $>$, $\geq$ extended in the obvious way, and $\lcm$ denotes the least common multiple on $\widehat{\N}$ defined by $\lcm(A)=+\infty$ if $A\subseteq \widehat{\N}$ is not a bounded set of natural numbers, and otherwise $\lcm(A)$ is calculated in the usually way. 

Note that different Thurston maps can share the same ramification function, in particular, we have the following fact from \cite[Proposition~2.16]{BM17}.

\begin{prop}   \label{propSameRamificationFn}
Let $f\: S^2\rightarrow S^2$ be a Thurston map. Then $\alpha_f = \alpha_{f^n}$ for each $n\in \N$.
\end{prop}

\begin{definition}[Orbifolds associated to Thurston maps]  \label{defOrbifoldThurstonMaps}
Let $f\: S^2\rightarrow S^2$ be a Thurston map. The \defn{orbifold associated to $f$} is a pair $\mathcal{O}_f \coloneqq (S^2,\alpha_f)$, where $S^2$ is an oriented $2$-sphere and $\alpha_f\: S^2\rightarrow \widehat{\N}$ is the ramification function of $f$.
\end{definition}

Orbifolds associated to Thurston maps are either parabolic or hyperbolic (see \cite[Proposition~2.12]{BM17}).

\smallskip

For an orbifold $\mathcal{O} = (S^2,\alpha)$, we set
\begin{equation}  \label{eqDefS0}
S_0^2 \coloneqq S^2 \setminus \bigl\{x\in S^2 \,\big|\, \alpha(x) = +\infty \bigr\}.
\end{equation}

We record the following facts from \cite{BM17}, whose proofs can be found in \cite{BM17} and references therein.

\begin{theorem}  \label{thmUniOrbCoverBM}
Let $\mathcal{O}=(S^2,\alpha)$ be an orbifold that is parabolic or hyperbolic. Then the following statements are satisfied:
\begin{enumerate}
\smallskip
\item[(i)] There exists a simply connected surface $\XX$ and a branched covering map $\Theta\: \XX\rightarrow S_0^2$ such that
\begin{equation*} 
 \deg_\Theta(x) = \alpha(\Theta(x))
\end{equation*}
for each $x\in \XX$.

\smallskip
\item[(ii)] The branched covering map $\Theta$ in $\operatorname{(i)}$ is unique. More precisely, if $\wt \XX$ is a simply connected surface and $\wt\Theta \: \wt \XX \rightarrow S_0^2$ satisfies $\deg_{\wt\Theta}(y) = \alpha\bigl(\wt\Theta(x)\bigr)$ for each $y\in \wt \XX$, then for all points $x_0\in \XX$ and $\wt{x}_0 \in \wt \XX$ with $\Theta(x_0)=\wt\Theta(\wt{x}_0)$ there exists orientation-preserving homeomorphism $A\: \XX\rightarrow \wt \XX$ with $A(x_0)= \wt{x}_0$ and $\Theta=\wt\Theta\circ A$. Moreover, if $\alpha(\Theta(x_0)) = 1$, then $A$ is unique.
%
\end{enumerate}
\end{theorem}
See Theorem~A.26 and Corollary~A.29 in \cite{BM17}. 

\begin{definition}[Universal orbifold covering maps]  \label{defUniOrbCover}
Let $\mathcal{O}=(S^2,\alpha)$ be an orbifold that is parabolic or hyperbolic. The map $\Theta\: \XX \rightarrow S_0^2$ from Theorem~\ref{thmUniOrbCoverBM} is called the \defn{universal orbifold covering map} of $\mathcal{O}$.
\end{definition}

We now discuss the deck transformations of the universal orbifold covering map. 

\begin{definition}[Deck transformations]  \label{defDeckTransf}
Let $\mathcal{O}=(S^2,\alpha)$ be an orbifold that is parabolic or hyperbolic, and $\Theta\: \XX \rightarrow S_0^2$ be the universal orbifold covering map of $\mathcal{O}$. A homeomorphism $\wt\sigma\: \XX\rightarrow\XX$ is called a \defn{deck transformation} of $\Theta$ if $\Theta\circ \wt\sigma = \Theta$. The group of deck transformations with composition as the group operation, denoted by $\pi_1(\mathcal{O})$, is called the \defn{fundamental group} of the orbifold $\mathcal{O}$. 
\end{definition}

Note that deck transformations are orientation-preserving. We record the following proposition from \cite[Proposition~A.31]{BM17}. 

\begin{prop}  \label{propDeckTransf}
Let $\mathcal{O}=(S^2,\alpha)$ be an orbifold that is parabolic or hyperbolic, and $\Theta\: \XX \rightarrow S_0^2$ be the universal orbifold covering map of $\mathcal{O}$. Then for all $u,v\in \XX$, the following statements are equivalent:
\begin{enumerate}
\smallskip
\item[(i)] There exists a deck transformation $\wt\sigma\in \pi_1(\mathcal{O})$ with $v=\wt\sigma(u)$.

\smallskip
\item[(ii)] $\Theta(u)=\Theta(v)$.
\end{enumerate}
%
\end{prop}

\smallskip

We now focus on the orbifold $\mathcal{O}_f = (S^2,\alpha_f)$ associated to a Thurston map $f\: S^2 \rightarrow S^2$.

One of the advantages of introducing orbifolds is the ability to lift branches of inverse map $f^{-1}$ by the universal orbifold covering map.

\begin{lemma} \label{lmLiftInverse}
Let $f\: S^2\rightarrow S^2$ be a Thurston map, $\mathcal{O}_f=(S^2,\alpha_f)$ be the orbifold associated to $f$, and $\Theta\: \XX \rightarrow S_0^2$ be the universal orbifold covering map of $\mathcal{O}_f$. Given $u_0,v_0\in \XX$ with $(f\circ \Theta)(v_0) = \Theta(u_0)$.

Then there exists a branched covering map $\wt{g} \: \XX\rightarrow\XX$ with $\wt{g}(u_0)=v_0$ and 
\begin{equation*}
f\circ \Theta \circ \wt{g} = \Theta.
\end{equation*}
If $u_0\notin \crit \Theta$, then the map $\wt{g}$ is unique.
%
\end{lemma}

See Lemma~A.32 in \cite{BM17} for a proof of Lemma~\ref{lmLiftInverse}.

\begin{definition}  \label{defInverseBranch}
Let $f\: S^2\rightarrow S^2$ be a Thurston map, $\mathcal{O}_f=(S^2,\alpha_f)$ be the orbifold associated to $f$, and $\Theta\: \XX \rightarrow S_0^2$ be the universal orbifold covering map of $\mathcal{O}_f$.  A branched covering map $\wt{g}\:\XX\rightarrow\XX$ is called an \defn{inverse branch of $f$ on $\XX$} if $f\circ \Theta\circ\wt{g}=\Theta$.

We denote the set of inverse branches of $f$ on $\XX$ by $\Inv(f)$.
\end{definition}

Note that by the definition of branched covering maps, $\wt{g} \: \XX\rightarrow \XX$ is surjective for each $\wt{g}\in \Inv(f)$.

\begin{lemma}   \label{lmHomotopyCurveOnX}
Let $f$ and $\CC$ satisfy the Assumptions. Let $\mathcal{O}_f=(S^2,\alpha_f)$ be the orbifold associated to $f$, and $\Theta\: \XX \rightarrow S_0^2$ be the universal orbifold covering map of $\mathcal{O}_f$. Then there exists $N\in\N$ such that for each $n\in\N$ with $n\geq N$ and each continuous path $\wt\gamma\: [0,1]\rightarrow \XX\setminus \Theta^{-1}(\post f)$, there exists a continuous path $\gamma\: [0,1]\rightarrow \XX\setminus \Theta^{-1}(\post f)$ with the following properties:
\begin{enumerate}
\smallskip
\item[(i)] $\gamma$ is homotopic to $\wt\gamma$ relative to $\{0,1\}$ in $\XX\setminus \Theta^{-1}(\post f)$.

\smallskip
\item[(ii)] There exists a number $k\in\N$, a strictly increasing sequence of numbers $0\eqqcolon a_0<a_1<\cdots<a_{k-1}<a_k\coloneqq 1$, and a sequence $\{X^n_i\}_{i\in \{1,2,\dots,k\}}$ of $n$-tiles in $\X^n(f,\CC)$ such that for each $i\in\{1,2,\dots,k\}$,
\begin{equation*}
(\Theta \circ \gamma)((a_{i-1},a_i)) \subseteq \inte(X^n_i).
\end{equation*}
\end{enumerate} 
\end{lemma}

Let $Z$ and $X$ be two topological spaces and $Y\subseteq Z$ be a subset of $Z$. A continuous function $f\: Z\rightarrow X$ is \defn{homotopic to} a continuous function $g\: Z\rightarrow X$ \defn{relative to} $Y$ (in $X$) if there exists a continuous function $H\: Z\times[0,1] \rightarrow X$ such that for each $z\in Z$, each $y\in Y$, and each $t\in[0,1]$, $H(z,0)= f(z)$, $H(z,1)=g(z)$, and $H(y,t)=f(y)=g(y)$.

\begin{remark}
We can choose $N$ to be the smallest number satisfying no $n$-tile joins opposite sides of $\CC$ for all $n\geq N$.
\end{remark}

\begin{proof}
Since $\post f$ is a finite set, by Lemma~\ref{lmCellBoundsBM}~(ii), we can choose $N\in\N$ large enough such that for each $n\in\N$ with $n\geq N$ and each $n$-tile $X^n\in\X^n$ such that 
\begin{equation}   \label{eqPflmHomotopyCurveOnX_AtMost1}
\card(X^n\cap \post f) \leq 1.
\end{equation}

Fix $n\geq N$ and a continuous path $\wt\gamma\:[0,1]\rightarrow \XX\setminus \Theta^{-1}(\post f)$.

We first claim that for each $x\in[0,1]$, there exists an $n$-vertex $v^n_x\in \V^n\setminus \post f$ and an open interval $I_x\subseteq \R$ such that $x\in I_x$ and $(\Theta\circ\wt\gamma)(I_x) \subseteq W^n(v^n_x) \subseteq S^2_0$.

We establish the claim by explicit construction in the following three cases:
\begin{enumerate}
\smallskip
\item[(1)] Assume $(\Theta\circ\wt\gamma)(x) \in \V^n$. Then we let $v^n_x\coloneqq (\Theta\circ\wt\gamma)(x)$. Since $(\Theta\circ\wt\gamma)(x)$ is contained in the open set $W^n(v^n_x)$, we can choose an open interval $I_x\subseteq \R$ containing $x$ with $(\Theta\circ\wt\gamma)(I_x) \subseteq W^n(v^n_x) \subseteq S^2_0$.

\smallskip
\item[(2)] Assume $(\Theta\circ\wt\gamma)(x) \in \inte(e^n)$ for some $n$-edge $e^n\in\E^n$. Since $\card(e^n\cap \post f) \leq 1$ by (\ref{eqPflmHomotopyCurveOnX_AtMost1}), we can choose $v^n_x \in e^n \cap \V^n \setminus \post f \subseteq S^2_0$. Then $(\Theta\circ\wt\gamma)(x) \in \inte(e^n) \subseteq W^n(v^n_x) \subseteq S^2_0$. Thus we can choose an open interval $I_x\subseteq \R$ containing $x$ with $(\Theta\circ\wt\gamma)(I_x) \subseteq W^n(v^n_x) \subseteq S^2_0$.

\smallskip
\item[(3)] Assume $(\Theta\circ\wt\gamma)(x) \in \inte(X^n)$ for some $n$-edge $X^n\in\X^n$. By (\ref{eqPflmHomotopyCurveOnX_AtMost1}), we can choose $v^n_x \in X^n \cap \V^n \setminus \post f \subseteq S^2_0$. Then $(\Theta\circ\wt\gamma)(x) \subseteq \inte(X^n) \subseteq W^n(v^n_x) \subseteq S^2_0$. Thus we can choose an open interval $I_x\subseteq \R$ containing $x$ with $(\Theta\circ\wt\gamma)(I_x) \subseteq W^n(v^n_x) \subseteq S^2_0$. 
\end{enumerate}

The claim is now established.

\smallskip

Since $[0,1]$ is compact, we can choose finitely many numbers $0\eqqcolon x_0 < x_1 < \cdots < x_{m'-1} < x_m' \coloneqq 1$ for some $m'\in\N$ such that $\bigcup\limits_{i=1}^m I_{x_i} \supseteq [0,1]$. Then it is clear that we can choose $m\leq m'$ and $0\eqqcolon b_0 < b_1 < \cdots < b_{m-1} < b_m \coloneqq 1$ such that for each $i\in\{1,2,\dots,m\}$, $[b_{i-1},b_i]\subseteq I_{x_{j(i)}}$ for some $j(i) \in \{1,2,\dots, m'\}$.

Fix an arbitrary $i\in\{1,2,\dots,m\}$. From the discussion above, we have
\begin{equation*}   \label{eqPflmHomotopyCurveOnX_InFlower}
(\Theta\circ\wt\gamma) ([b_{i-1},b_i]) \subseteq (\Theta\circ\wt\gamma) \bigl( I_{x_{j(i)}} \bigr)  \subseteq W^n\bigl( v^n_{x_{j(i)}} \bigr)  \subseteq S^2_0.
\end{equation*}

It follows from Remark~\ref{rmFlower} that we can choose a continuous path $\gamma_i\: [b_{i-1},b_i]\rightarrow W^n\bigl(v^n_{x_{j(i)}}\bigr)$ such that $\gamma_i$ is injective, $\gamma_i(b_{i-1}) = (\Theta \circ \wt\gamma) (b_{i-1})$, $\gamma_i(b_i) = (\Theta \circ \wt\gamma) (b_i)$, and that for each $n$-tile $X^n\in\X^n$ with $X^n\subseteq \overline{W}^n\bigl(v^n_{x_{j(i)}}\bigr)$, $\gamma_i^{-1}( \inte(X^n))$ is connected and $\card \bigl( \gamma_i^{-1}( \partial X^n ) \bigr) \leq 2$. See Figure~\ref{figHomotopy}. Since $W^n\bigl(v^n_{x_{j(i)}}\bigr)$ is simply connected (see Remark~\ref{rmFlower}), $(\Theta\circ\wt\gamma)|_{[b_{i-1},b_i]}$ is homotopic to $\gamma_i$ relative to $\{b_{i-1},b_i\}$ in $W^n\bigl(v^n_{x_{j(i)}}\bigr)$. It follows from Definition~\ref{defRamificationFn}, Definition~\ref{defUniOrbCover}, Lemma~\ref{lmBranchCoverToCover}, and Lemma~\ref{lmLiftCoveringMap} that there exists a unique continuous path $\wt\gamma_i\: [b_{i-1},b_i] \rightarrow \XX\setminus\Theta^{-1}(\post f)$ such that $\Theta\circ \wt\gamma_i = \gamma_i$ and $\wt\gamma_i$ is homotopic to $\wt\gamma|_{[b_{i-1},b_i]}$ relative to $\{b_{i-1},b_i\}$ in $\XX\setminus\Theta^{-1}(\post f)$.

\begin{figure}
    \centering
    \begin{overpic}
    [width=6cm, 
    tics=20]{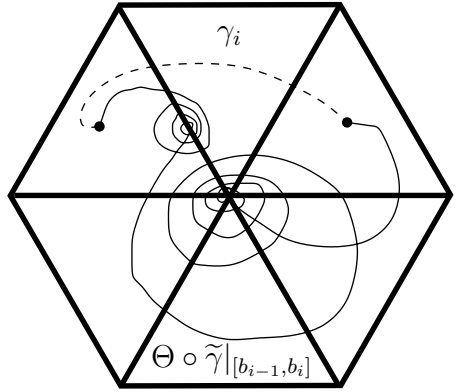}
    \put(80,133){$\gamma_i$}
    \put(55,9){$\Theta \circ \wt\gamma|_{[b_{i-1},b_i]}$}
    \end{overpic}
    \caption{Homotopic curves in $W^n\bigl(v^n_{x_{j(i)}}\bigr)$.}
    \label{figHomotopy}
\end{figure}

We define $\gamma\:[0,1]\rightarrow \XX\setminus\Theta^{-1}(\post f)$ by setting $\gamma|_{[b_{i-1},b_i]} = \wt\gamma_i$. Then it is clear that $\gamma$ is continuous and homotopic to $\wt\gamma$ relative to $\{0,1\}$ in $\XX\setminus\Theta^{-1}(\post f)$. It also follows immediately from our construction that Property~(ii) is satisfied.
\end{proof}

\begin{cor} \label{corGoodPath}
Let $f$ and $\CC$ satisfy the Assumptions. Let $\mathcal{O}_f=(S^2,\alpha_f)$ be the orbifold associated to $f$, and $\Theta\: \XX \rightarrow S_0^2$ be the universal orbifold covering map of $\mathcal{O}_f$. For each pair of points $x,y\in\XX$, there exists a continuous path $\wt\gamma\:[0,1] \rightarrow \XX$, numbers $k,n\in\N$, a strictly increasing sequence of numbers $0\eqqcolon a_0<a_1<\cdots<a_{k-1}<a_k\coloneqq 1$, and a sequence $\{X^n_i\}_{i\in \{1,2,\dots,k\}}$ of $n$-tiles in $\X^n(f,\CC)$ such that $\wt\gamma(0)=x$, $\wt\gamma(1)=y$, and
\begin{equation}   \label{eqGoodPathInTileInte}
(\Theta \circ \wt\gamma)((a_{i-1},a_i)) \subseteq \inte(X^n_i)
\end{equation}
for each $i\in\{1,2,\dots,k\}$. 

Moreover, if $\{ \wt{g}_j \}_{j\in\N}$ is a sequence in $\Inv(f)$ of inverse branches of $f$ on $\XX$, (i.e., $f\circ \Theta \circ \wt{g}_j = \Theta$ for each $j\in\N$,) then for each $m\in\N$, there exists a sequence  $\{X^{n+m}_i\}_{i\in \{1,2,\dots,k\}}$ of $(n+m)$-tiles in $\X^{n+m} (f,\CC)$ such that
\begin{equation}    \label{eqGoodPathInTileInteSmaller}
(\Theta \circ \wt{g}_m \circ \cdots \circ \wt{g}_1 \circ \wt\gamma )((a_{i-1},a_i)) \subseteq \inte ( X^{n+m}_i )
\end{equation}
for each $i\in\{1,2,\dots,k\}$. If $d$ is a visual metric on $S^2$ for $f$ with expansion factor $\Lambda >1$, then
\begin{equation}   \label{eqGoodPathLimitLength0}
 \diam_d  (  ( \Theta  \circ \wt{g}_m \circ \cdots \circ \wt{g}_1 \circ \wt\gamma  ) ([0,1])  )  \leq kC \Lambda^{-(n+m)}
\end{equation}
for $m\in\N$, where $C\geq 1$ is a constant from Lemma~\ref{lmCellBoundsBM} depending only on $f$, $\CC$, and $d$.
\end{cor}

\begin{proof}
Let $\CC \subseteq S^2$ be a Jordan curve on $S^2$ with $\post f \subseteq \CC$. Fix an arbitrary number $n\geq N$, where $N\in\N$ is a constant depending only on $f$ and $\CC$ from Lemma~\ref{lmHomotopyCurveOnX}.

Choose $n$-tiles $X^n_1,X^n_{\minus 1} \in \X^n(f,\CC)$ with $\Theta(x)\in X^n_1$ and $\Theta(y)\in X^n_{\minus 1}$. Since $n$-tiles are cells of dimension $2$ as discussed in Subsection~\ref{subsctThurstonMap}, we can choose continuous paths $\gamma_x \: \bigl[0,\frac14\bigr] \rightarrow S^2_0$ and $\gamma_y \: \bigl[\frac34, 1\bigr] \rightarrow S^2_0$ with $\gamma_x(0) = \Theta(x)$, $\gamma_y(1) = \Theta(y)$, $\gamma_x\bigl( \bigl(0, \frac14 \bigr] \bigr) \subseteq \inte (X^n_1)$, and $\gamma_y\bigl( \bigl(\frac34, 1 \bigr] \bigr) \subseteq \inte (X^n_{\minus 1})$. Since $\Theta$ is a branched covering map (see Theorem~\ref{thmUniOrbCoverBM}), by Lemma~\ref{lmLiftPathBM} we can lift $\gamma_x$ (resp.\ $\gamma_y$) to $\wt\gamma_x \: \bigl[ 0, \frac14 \bigr] \rightarrow \XX$ (resp.\ $\wt\gamma_y \: \bigl[ \frac34,  1 \bigr] \rightarrow \XX$) such that $\wt\gamma_x(0) = x$ and $\Theta \circ \wt\gamma_x = \gamma_x$ (resp.\ $\wt\gamma_y(1) = y$ and $\Theta \circ \wt\gamma_y = \gamma_y$).

Since $u\coloneqq \wt\gamma_x \bigl( \frac14 \bigr) \in \Theta^{-1} (\inte(X^n_1))$ and $v\coloneqq \wt\gamma_y \bigl( \frac34 \bigr) \in \Theta^{-1} (\inte(X^n_{\minus 1}))$, we have $\{u,v\} \subseteq \XX \setminus \Theta^{-1}(\post f)$. Since $\post f$ is a finite set and $\Theta$ is discrete, we can choose a continuous path $\widehat\gamma\: \bigl[\frac14, \frac34\bigr] \rightarrow \XX \setminus \Theta^{-1}(\post f)$ with $\widehat\gamma \bigl(\frac14\bigr) = u$ and $\widehat\gamma \bigl(\frac34\bigr) = v$. By Lemma~\ref{lmHomotopyCurveOnX}, there exists a number $k\in\N$, a continuous path $\gamma \: \bigl[ \frac14, \frac34 \bigr] \rightarrow \XX \setminus \Theta^{-1}(\post f)$, a sequence of numbers $\frac14 \eqqcolon a_1 < a_2 < \cdots < a_{k-2} < a_{k-1} \coloneqq \frac34$, and a sequence $\{X^n_i\}_{i\in\{2,3,\dots, k-1\}}$ of $n$-tiles in $\X^n(f,\CC)$ such that $\gamma \bigl( \frac14 \bigr) = u$, $\gamma \bigl( \frac34 \bigr) = v$, and
\begin{equation*}
(\Theta \circ \gamma) ( (a_{i-1},a_i) ) \subseteq \inte(X^n_i)
\end{equation*}
for each $i\in \{2,3,\dots, k-1\}$.

We define a continuous path $\wt\gamma \: [0,1] \rightarrow \XX$ by
\begin{equation*}
\wt\gamma(t) \coloneqq \begin{cases} 
\wt\gamma_x(t) & \text{if } t\in \bigl[ 0, \frac14 \bigr), \\ 
   \gamma (t)  & \text{if } t\in \bigl[ \frac14, \frac34 \bigr], \\
\wt\gamma_y(t) & \text{if } t\in \bigl( \frac34, 1 \bigr].   \end{cases}
\end{equation*}
Let $X^n_k \coloneqq X^n_{\minus 1}$, $a_0 \coloneqq 0$, and $a_k\coloneqq 1$. By our construction, we have $\wt\gamma(0) = x$, $\wt\gamma(1) = y$, and
\begin{equation*}
(\Theta \circ \wt\gamma) ( (a_{i-1},a_i) ) \subseteq \inte(X^n_i)
\end{equation*}
for each $i\in \{1,2,\dots, k\}$, establishing (\ref{eqGoodPathInTileInte}).

\smallskip

Fix a  sequence $\{ \wt{g}_j \}_{j\in\N}$ of inverse branches of $f$ on $\XX$ in $\Inv(f)$. Fix arbitrary integers $m\in\N$ and  $i\in \{1,2,\dots, k\}$. Denote $I_i \coloneqq (a_{i-1}, a_i)$.

By (\ref{eqGoodPathInTileInte}), each connected component of $f^{-m}((\Theta \circ \wt\gamma)(I_i))$ is contained in some connected component of $f^{-m}(\inte(X^n_i))$. Since both $( \Theta \circ \wt{g}_m \circ \cdots \circ \wt{g}_1 \circ \wt \gamma ) (I_i)$ and $f^m ( ( \Theta \circ \wt{g}_m \circ \cdots \circ \wt{g}_1\circ \wt\gamma ) (I_i) ) = (\Theta \circ \wt\gamma)(I_i)$ are connected, by Proposition~\ref{propCellDecomp}~(i), (ii), and (v), there exists an $(n+m)$-tile $X^{n+m}_i \in \X^{n+m}(f,\CC)$ such that 
\begin{equation*}  
( \Theta \circ \wt{g}_m \circ \cdots \circ \wt{g}_1 \circ \wt\gamma ) (I_i) \subseteq \inte ( X^{n+m}_i ). 
\end{equation*}
Since $m\in\N$ and  $i\in \{1,2,\dots, k\}$ are arbitrary, (\ref{eqGoodPathInTileInteSmaller}) is established. Finally, it is clear that (\ref{eqGoodPathLimitLength0}) follows immediately from (\ref{eqGoodPathInTileInteSmaller}) and Lemma~\ref{lmCellBoundsBM}~(ii).
\end{proof}

If we assume that $f$ is expanding, then roughly speaking, each inverse branch on the universal orbifold cover has a unique attracting fixed point (possibly at infinity). The precise statement is formulated in the following proposition.

\begin{prop}  \label{propInvBranchFixPt}
Let $f$, $d$, $\Lambda$ satisfy the Assumptions. Let $\mathcal{O}_f=(S^2,\alpha_f)$ be the orbifold associated to $f$, and $\Theta\: \XX \rightarrow S_0^2$ be the universal orbifold covering map of $\mathcal{O}_f$. Given a branched covering map $\wt{g}\: \XX\rightarrow \XX$ satisfying $f\circ \Theta \circ \wt{g} = \Theta$. Then the map $\wt{g}$ has at most one fixed point. Moreover, 
\begin{enumerate}
\smallskip
\item[(i)] if $w\in \XX$ is a fixed point of $\wt{g}$, then $\lim\limits_{i\to +\infty} \wt{g}^i(u) = w$ for all $u\in\XX$;

\smallskip
\item[(ii)] if $\wt{g}$ has no fixed point in $\XX$, then $f$ has a fixed critical point $z\in S^2$ such that $\lim\limits_{i\to +\infty} \Theta\bigl(\wt{g}^i(u)\bigr) = z$ for all $u\in \XX$.
\end{enumerate}
\end{prop}

\begin{proof}
Fix an arbitrary Jordan curve $\CC\subseteq S^2$ on $S^2$ with $\post f\subseteq\CC$.

We observe that it  follows immediately from statement~(i) that $\wt{g}$ has at most one fixed point.

\smallskip

(i) We assume that $w\in \XX$ is a fixed point of $\wt{g}$. We argue by contradiction and assume that $\wt{g}^i(u)$ does not converge to $w$  as $i\to+\infty$ for some $u\in\XX$. By Corollary~\ref{corGoodPath} (with $\wt{g}_j\coloneqq \wt{g}$ for each $j\in\N$), we choose a continuous path $\wt\gamma \: [0,1]\rightarrow \XX$ with $\wt\gamma(0)=u$, $\wt\gamma(1)=w$, and
\begin{equation}   \label{eqPfpropInvBranchFixPt_CurveDiamLimit}
\lim\limits_{i\to+\infty} \diam_d \bigl( \bigl(\Theta\circ\wt{g}^i\circ\wt\gamma \bigr) ([0,1]) \bigr) = 0.
\end{equation}
Denote $q\coloneqq  \Theta(w)$. Since $\Theta$ is a branched covering map (see Theorem~\ref{thmUniOrbCoverBM}), we can choose open sets $V\subseteq S^2_0$, $U_i\subseteq\XX$, and homeomorphisms $\varphi_i\: U_i \rightarrow \D$, $\psi_i\: V\rightarrow \D$, for $i\in I$, as in Definition~\ref{defBranchedCover} (with $X\coloneqq \XX$, $Y \coloneqq S^2_0$, and $f \coloneqq \Theta$). We choose $i_0 \in I$ such that $w \in U_{i_0}$. Then by our assumption there exists $r \in (0,1)$ and a strictly increasing sequence $\{k_j\}_{j\in\N}$ of positive integers such that 
\begin{equation*}
\wt{g}^{k_j} (u) \notin \varphi_{i_0}^{-1} ( \{z\in\C \,|\, \abs{z} < r \} )
\end{equation*}
for each $j\in\N$.

For each $j\in\N$, since $\bigl( \wt{g}^{k_j} \circ \wt\gamma \bigr) ([0,1])$ is a path on $\XX$ connecting $\wt{g}^{k_j}(u)$ and $\wt{g}^{k_j}(w)=w$, we have
\begin{equation*}
\bigl( \wt{g}^{k_j} \circ \wt\gamma \bigr) ([0,1])    \cap \varphi_{i_0}^{-1} ( \{z\in\C \,|\, \abs{z} = r \} )  \neq \emptyset.
\end{equation*}
Combining the above with (\ref{eqBranchCoverMapLocalPowerMap}) in Definition~\ref{defBranchedCover}, we get
\begin{equation*}
\diam_{\rho}   \bigl(  \bigl( \psi_{i_0} \circ \Theta \circ \wt{g}^{k_j} \circ \wt\gamma \bigr) ([0,1]) \bigr) 
\geq  \rho \bigl(0, \bigl( \psi_{i_0} \circ \Theta \circ \varphi_{i_0}^{-1} \bigr) ( \{z\in\C \,|\, \abs{z} = r \} ) \bigr) =  r^{d_{i_0}} >0
\end{equation*}
for $j\in\N$, where $d_{i_0} \coloneqq \deg_\Theta (w)$ as in Definition~\ref{defBranchedCover} and $\rho$ is the Euclidean metric on $\C$. This immediately leads to a contradiction with the fact that $\wt\gamma$ satisfies (\ref{eqPfpropInvBranchFixPt_CurveDiamLimit}), proving statement~(i).

\smallskip
(ii) We assume that $\wt{g}$ has no fixed point in $\XX$. Fix an arbitrary point $v\in\XX$. Let $x \coloneqq u$ and $y \coloneqq \wt{g}(u)$.

By Corollary~\ref{corGoodPath} (with $\wt{g}_j \coloneqq \wt{g}$ for each $j\in\N$), there exists a continuous path $\wt\gamma\:[0,1] \rightarrow \XX$, numbers $k,n\in\N$, a strictly increasing sequence of numbers $0\eqqcolon a_0<a_1<\cdots<a_{k-1}<a_k\coloneqq 1$, and for each $m\in\N_0$ there exists a sequence $\{X^{n+m}_i\}_{i\in \{0,1,\dots,k\}}$ of $(n+m)$-tiles in $\X^{(n+m)}$ such that $\wt\gamma(0)=x$, $\wt\gamma(1)=y$, and
\begin{equation}   \label{eqPfpropInvBranchFixPt_InTileInte}
(\Theta \circ \wt{g}^m \circ \wt\gamma)((a_{i-1},a_i)) \subseteq \inte(X^{n+m}_i)
\end{equation}
for each $i\in\{1,2,\dots,k\}$ and each $m\in\N_0$.  Moreover, for each $m\in\N_0$,
\begin{equation}   \label{eqPfpropInvBranchFixPt_GoodPathLimitLength0}
 \diam_d  ( ( \Theta  \circ \wt{g}^m \circ \wt\gamma  ) ([0,1]) ) \leq kC \Lambda^{-(n+m)}.
\end{equation}
where $C\geq 1$ is a constant from Lemma~\ref{lmCellBoundsBM} depending only on $f$, $\CC$, and $d$.

Since $\wt\gamma(0)=x$ and $\wt\gamma(1)=\wt{g}(x)$, by (\ref{eqPfpropInvBranchFixPt_GoodPathLimitLength0}), for each $m\in\N$ we have  
\begin{equation}  \label{eqPfpropInvBranchFixPt_gmxIterateDistance}
 d  \bigl( \Theta  ( \wt{g}^m(x)  ), \Theta \bigl( \wt{g}^{m+1}(x) \bigr) \bigr)  \leq kC \Lambda^{-(n+m)}.
\end{equation}
Since $S^2$ is compact, we get $\lim\limits_{m\to+\infty}  \Theta  ( \wt{g}^m(x)  ) = z$ for some $z\in S^2$. Since $\bigl( f \circ \Theta \circ \wt{g}^{m+1} \bigr) = \Theta \circ \wt{g}^m$ for each $m\in\N$, we have $f(z) = z$. To see that $z$ is independent of $x$, we choose arbitrary points $x'\in\XX$ and $z'\in S^2$ with $\lim\limits_{m\to+\infty} \Theta  ( \wt{g}^im(x')) = z'$. Then by the same argument as above, we get $f(z')=z'$. Applying Corollary~\ref{corGoodPath} (with $y\coloneqq x'$), we get $z=z'$.

It suffices to show $z\in \crit f$ now. We observe that it follows from (\ref{eqDefS0}), (\ref{eqDefRamificationFn}), and $f(z)=z$ that it suffices to prove $z\notin S^2_0$. We argue by contradiction and assume that $z\in S^2_0$. Since $\Theta$ is a branched covering map (see Theorem~\ref{thmUniOrbCoverBM}), we can choose open sets $V\subseteq S^2_0$ and $U_i\subseteq \XX$, $i\in I$, as in Definition~\ref{defBranchedCover} (with $X \coloneqq \XX$, $Y \coloneqq S^2_0$, $q \coloneqq z$, and $f \coloneqq \Theta$). By Lemma~\ref{lmCellBoundsBM}~(ii) and the fact that flowers are open sets (see Remark~\ref{rmFlower}), it is clear that there exist numbers $l,L\in\N$, an $l$-vertex $v^l\in \V^l$, and an $L$-vertex $v^L\in\V^L$ such that $l<L$ and
\begin{equation} \label{eqPfpropInvBranchFixPt_zInVInV}
z \in W^L \bigl( v^L \bigr) \subseteq \overline{W}^L \bigl( v^L \bigr) \subseteq W^l \bigl( v^l \bigr) \subseteq V.
\end{equation}  
By (\ref{eqPfpropInvBranchFixPt_gmxIterateDistance}) there exists $N\in\N$ large enough so that for each $m\in\N$ with $m\geq N$, we have
\begin{equation*}
\Theta  ( \wt{g}^m(x)  ) \subseteq W^L \bigl( v^L \bigr) \subseteq V.
\end{equation*}
Since $\wt{g}$ has no fixed points in $\XX$, $\wt{g}^m(x)$ does not converge to any point in $\Theta^{-1}(z)$ as $m\to+\infty$, for otherwise, suppose $\lim\limits_{m\to+\infty} \wt{g}^m(x) \eqqcolon p \in \XX$, then $\wt{g}(p)=p$, a contradiction. Hence there exists a strictly increasing sequence $\{m_j\}_{j\in\N}$ of positive integers such that $\wt{g}^{m_j}(x)$ and $\wt{g}^{m_j + 1}(x)$ are contained in different connected components of $\Theta^{-1}(V)$. Since $\wt{g}^{m_j} ( \wt\gamma(0) ) = \wt{g}^{m_j} (x)$, $\wt{g}^{m_j} ( \wt\gamma(1) ) = \wt{g}^{m_j + 1} (x)$, and the set $\wt{g}^{m_j} (\wt\gamma ( [0,1] ))$ is connected, we get from (\ref{eqPfpropInvBranchFixPt_zInVInV} that
\begin{equation*}
 ( \Theta \circ \wt{g}^{m_j} \circ \wt\gamma  ) ( [0,1] )  \cap \partial W^L \bigl( v^L \bigr)     \neq  \emptyset \neq ( \Theta \circ \wt{g}^{m_j} \circ \wt\gamma  ) ( [0,1] )  \cap \partial W^l \bigl( v^l \bigr)    .
\end{equation*}
This contradicts with (\ref{eqPfpropInvBranchFixPt_GoodPathLimitLength0}). Therefore $z\in S^2_0$ and $z\notin \crit f$.
\end{proof}

\subsection{Proof of the characterization Theorem~\ref{thmNLI}}     \label{subsctNLI_Proof}

We first lift the local integrability condition by the universal orbifold covering map.

\begin{lemma}   \label{lmLiftLI}
Let $f$,  $\CC$, $d$, $\psi$ satisfy the Assumptions. We assume in addition that $f(\CC)\subseteq \CC$. Let $\mathcal{O}_f=(S^2, \alpha_f)$ be the orbifold associated to $f$, and $\Theta\: \XX\rightarrow S^2_0$ the universal orbifold covering map of $\mathcal{O}_f$. Assume that 
\begin{equation}  \label{eqLiftLI_LIAssumption}
\psi^{f,\,\CC}_{\xi,\,\eta}(x,y) = 0
\end{equation}
for all $\xi \coloneqq \{ \xi_{\minus i} \}_{i\in\N_0} \in  \Sigma_{f,\,\CC}^-$ and $\eta \coloneqq \{ \eta_{\minus i} \}_{i\in\N_0} \in  \Sigma_{f,\,\CC}^-$ with $f(\xi_0) = f(\eta_0)$, and all $(x,y)\in \bigcup\limits_{\substack{X\in\X^1(f,\CC) \\ X\subseteq f(\xi_0)}}X \times X$.

Then for each pair of sequences $\{ \wt{g}_i \}_{i\in\N}$ and $\{ \wt{h}_i \}_{i\in\N}$ of inverse branches of $f$ on $\XX$, (i.e., $f\circ\Theta\circ\wt{g}_i = \Theta$ and $f\circ\Theta\circ\wt{h}_i = \Theta$ for $i\in\N$,) we have
\begin{align}  \label{eqLiftLI}
 &\sum\limits_{i=1}^{+\infty} \Bigl(  \Bigl(   \wt{\psi} \circ \wt{g}_i \circ\cdots\circ \wt{g}_1 \Bigr)(u)  
                                    - \Bigl(   \wt{\psi} \circ \wt{g}_i \circ\cdots\circ \wt{g}_1 \Bigr)(v) \Bigr) \\ 
=& \sum\limits_{i=1}^{+\infty}\Bigl(  \Bigl(   \wt{\psi} \circ \wt{h}_i \circ\cdots\circ \wt{h}_1 \Bigr)(u) 
                                    - \Bigl(   \wt{\psi} \circ \wt{h}_i \circ\cdots\circ \wt{h}_1 \Bigr)(v) \Bigr)   \notag
\end{align}
for $u,v\in\XX$, where $\wt{\psi} \coloneqq \psi\circ \Theta$.
\end{lemma}

\begin{proof}
Fix sequences $\{\wt{g}_i\}_{i\in\N}$ and $\{\wt{h}_i\}_{i\in\N}$ of inverse branches of $f$ on $\XX$. Fix arbitrary points $u,v\in\XX$.

By Corollary~\ref{corGoodPath}, there exists a continuous path $\wt\gamma \: [0,1] \rightarrow \XX$, integers $k,n\in\N$, a strictly increasing sequence of numbers $0 \eqqcolon a_0 < a_1 < \cdots < a_{k-1} < a_k \coloneqq 1$, and a sequence $\{X^n_i\}_{i\in\{1,2,\dots, k\}}$ of $n$-tiles in $\X^n(f,\CC)$ such that $\wt\gamma(0) = u$, $\wt\gamma(1) = v$, and 
\begin{equation} \label{eqPflmLiftLI_InTileInte}
( \Theta \circ \wt\gamma ) ( (a_{i-1},a_i) ) \subseteq \inte (X^n_i).
\end{equation}
Moreover, for each $m\in\N$, there exist two sequences  $\{X^{n+m}_i\}_{i\in \{1,2,\dots,k\}}$ and $\{Y^{n+m}_i\}_{i\in \{1,2,\dots,k\}}$ of $(n+m)$-tiles in $\X^{n+m} (f,\CC)$ such that
\begin{equation}    \label{eqPflmLiftLI_InTileInteSmallerX}
( \Theta \circ \wt{g}_m \circ \cdots \circ \wt{g}_1 \circ \wt\gamma ) ( (a_{i-1}, a_i) )   \subseteq \inte ( X^{n+m}_i )
\end{equation}
and 
\begin{equation}    \label{eqPflmLiftLI_InTileInteSmallerY}
\bigl( \Theta \circ \wt{h}_m \circ \cdots \circ \wt{h}_1 \circ \wt\gamma \bigr) ( (a_{i-1}, a_i) )   \subseteq \inte ( Y^{n+m}_i )
\end{equation}
for each $i\in\{1,2,\dots,k\}$.

We denote $u_0 \coloneqq \wt{\gamma} (a_0) =  u$, $u_i \coloneqq \wt\gamma (a_i)$, and $I_i \coloneqq (a_{i-1},a_i)$ for $i\in\{1, 2, \dots, k\}$.

Observe that it suffices to show that (\ref{eqLiftLI}) holds with $u$ and $v$ replaced by $u_{i-1}$ and $u_i$, respectively, for each $i\in\{1,2,\dots,k\}$.

Fix an arbitrary integer $i\in\{1,2,\dots,k\}$.

For each $j\in\N_0$,  we denote by $\xi_{\minus j}$ the unique $1$-tile in $\X^1$ containing $X^{n+j+1}_i$, and denote by $\eta_{\minus j}$ the unique $1$-tile in $\X^1$ containing $Y^{n+j+1}_i$.

We will show that $\xi \coloneqq \{\xi_{\minus j}\}_{j\in\N_0}$ and $\eta \coloneqq \{ \eta_{\minus j} \}_{j\in\N_0}$ satisfy the following properties:
\begin{enumerate}
\smallskip
\item[(1)]  $\xi,\eta \in  \Sigma_{f,\,\CC}^-$.

\smallskip
\item[(2)] $f(\xi_0) = f(\eta_0) \eqqcolon X^0 \supseteq X^n_i \supseteq (\Theta \circ \wt\gamma) (I_i)$.

\smallskip
\item[(3)] $\bigl( \Theta \circ \wt{g}_{j+1} \circ \cdots \circ \wt{g}_1 \circ \wt\gamma \bigr) (I_i)
             \subseteq \bigl( f_{\xi_{\minus j}}^{-1} \circ  \cdots \circ f_{\xi_0}^{-1} \bigr)  (X^0)$ and $\bigl( \Theta \circ \wt{h}_{j+1} \circ \cdots \circ \wt{h}_1 \circ \wt\gamma \bigr) (I_i)
             \subseteq \bigl( f_{\eta_{\minus j}}^{-1} \circ  \cdots \circ f_{\eta_0}^{-1} \bigr)  (X^0)$ for each $j\in\N_0$.
\end{enumerate}

\smallskip

(1) Fix an arbitrary integer $m\in\N_0$. We note that by (\ref{eqPflmLiftLI_InTileInteSmallerX}),
\begin{align}   \label{eqPflmLiftLI_Prop1}
&        f \bigl(\xi_{\minus(m+1)} \bigr) \cap \inte(\xi_{\minus m}) \notag\\
&\qquad \supseteq f \bigl( X^{n+m+2}_i \bigr) \cap \inte \bigl( X^{n+m+1}_i \bigr)   \\
&\qquad \supseteq ( f \circ \Theta \circ \wt{g}_{m+2} \circ \cdots \circ \wt{g}_1 \circ \wt\gamma ) (I_i) 
               \cap ( \Theta \circ \wt{g}_{m+1}           \circ \cdots \circ \wt{g}_1 \circ \wt\gamma ) (I_i)  \notag\\
&\qquad    =              ( \Theta \circ \wt{g}_{m+1}            \circ \cdots \circ \wt{g}_1 \circ \wt\gamma ) (I_i) 
   \neq \emptyset. \notag                    
\end{align}
Since $f \bigl( \xi_{\minus(m+1)} \bigr) \in \X^0$ (see Proposition~\ref{propCellDecomp}~(i)), we get from (\ref{eqPflmLiftLI_Prop1}) that  $f \bigl( \xi_{\minus(m+1)} \bigr) \supseteq \xi_{\minus m}$. Since $m\in\N_0$ is arbitrary, we get $\xi \in \Sigma_{f,\,\CC}^-$. Similarly, we have $\eta \in \Sigma_{f,\,\CC}^-$.

\smallskip

(2) We note that by (\ref{eqPflmLiftLI_InTileInteSmallerX}), (\ref{eqPflmLiftLI_InTileInteSmallerY}), and (\ref{eqPflmLiftLI_InTileInte}),
\begin{align}   \label{eqPflmLiftLI_Prop2}
&f ( \inte (\xi_0) ) \cap f (\inte (\eta_0) )   \cap  \inte (X^n_i)  \notag\\
&\qquad \supseteq ( f \circ \Theta \circ \wt{g}_1 \circ \wt\gamma ) (I_i)  \cap \bigl( f \circ \Theta \circ \wt{h}_1 \circ \wt\gamma \bigr) (I_i)   \cap  \inte (X^n_i)  \\
&\qquad    =     (  \Theta \circ \wt\gamma ) (I_i)  
  \neq      \emptyset.   \notag
\end{align}
It follows from (\ref{eqPflmLiftLI_Prop2}) and Proposition~\ref{propCellDecomp}~(i) that $f(\xi_0) = f(\eta_0) \eqqcolon X^0 \supseteq  X^n_i \supseteq (\Theta \circ \wt\gamma) (I_i)$. This verifies Property~(2).

\smallskip

(3) We will establish the first relation in Property~(3) since the proof of the second one is the same. We note that by (\ref{eqPflmLiftLI_InTileInteSmallerX}), it suffices to show that
\begin{equation}   \label{eqPflmLiftLI_Prop3Equiv}
X^{n+j+1}_i \subseteq  \bigl( f_{\xi_{\minus j}}^{-1} \circ  \cdots \circ f_{\xi_0}^{-1} \bigr)  (X^0)
\end{equation}
for each $j\in\N_0$.

We prove (\ref{eqPflmLiftLI_Prop3Equiv}) by induction on $j\in\N_0$.

For $j=0$, we have $f_{\xi_0}^{-1} (X^0) = \xi_0 \supseteq X^{n+1}_i$ by Property~(2) and our construction above.

We now assume that (\ref{eqPflmLiftLI_Prop3Equiv}) holds for some $j\in\N_0$. Then by the induction hypothesis, (\ref{eqPflmLiftLI_InTileInteSmallerX}), and the fact that $f$ is injective on $\xi_{\minus (j+1)} \supseteq X^{n+j+2}_i$ (see Proposition~\ref{propCellDecomp}~(i)), we get
\begin{align*}
&                  \bigl( f_{\xi_{\minus (j+1)}}^{-1} \circ  \cdots \circ f_{\xi_0}^{-1} \bigr)  (X^0)  \cap   \inte \bigl( X^{n+j+2}_i \bigr)  \\
&\qquad \supseteq  f_{\xi_{\minus (j+1)}}^{-1} \bigl( X^{n+j+1}_i \bigr)   \cap   \inte \bigl( X^{n+j+2}_i \bigr)  
           =       f_{\xi_{\minus (j+1)}}^{-1} \bigl( X^{n+j+1}_i    \cap  f \bigl(  \inte \bigl( X^{n+j+2}_i \bigr) \bigr) \bigr) \\
&\qquad \supseteq  f_{\xi_{\minus (j+1)}}^{-1} \bigl( \bigl( \Theta \circ \wt{g}_{j+1} \circ \cdots \circ \wt{g}_1 \circ \wt\gamma \bigr) (I_i)
                                         \cap \bigl( f \circ \Theta \circ \wt{g}_{j+2} \circ \cdots \circ \wt{g}_1 \circ \wt\gamma \bigr) (I_i)  \bigr) \\
&\qquad    =       f_{\xi_{\minus (j+1)}}^{-1} \bigl( \bigl( \Theta \circ \wt{g}_{j+1} \circ \cdots \circ \wt{g}_1 \circ \wt\gamma \bigr) (I_i)     \bigr). 
\end{align*}
The set on the right-hand side of the last line above is nonempty, since by (\ref{eqPflmLiftLI_InTileInteSmallerX}) and our construction,
\begin{align*}
        \bigl( \Theta \circ \wt{g}_{j+1} \circ \cdots \circ \wt{g}_1 \circ \wt\gamma \bigr) (I_i) 
   =  & \bigl( f \circ \Theta \circ \wt{g}_{j+2} \circ \cdots \circ \wt{g}_1 \circ \wt\gamma \bigr) (I_i) \\
\subseteq &  f \bigl( X^{n+j+2}_i \bigr)
\subseteq    f \bigl( \xi_{\minus (j+1)} \bigr).
\end{align*}
On the other hand, since $\xi \in \Sigma_{f,\,\CC}^-$, it follows immediately from Lemma~\ref{lmCylinderIsTile} that
\begin{equation}  \label{eqPflmLiftLI_ImageIsTile}
\bigl( f_{\xi_{\minus (j+1)}}^{-1} \circ  \cdots \circ f_{\xi_0}^{-1} \bigr)  (X^0)  \in \X^{j+2}.
\end{equation}
Hence $X^{n+j+2}_i \subseteq \bigl( f_{\xi_{\minus (j+1)}}^{-1} \circ  \cdots \circ f_{\xi_0}^{-1} \bigr)  (X^0)$.

The induction step is now complete, establishing Property~(3).

\smallskip

Finally, by Property~(3), for each $m\in\N_0$ and each $w\in\{ u_{i-1}, u_i \}$ we have
\begin{equation*}
 ( \Theta \circ \wt{g}_{m+1} \circ \cdots \circ \wt{g}_1 ) (w)
             \subseteq \bigl( f_{\xi_{\minus m}}^{-1} \circ  \cdots \circ f_{\xi_0}^{-1} \bigr)  (X^0).
\end{equation*}
Since $f^{m+1} ( ( \Theta \circ \wt{g}_{m+1} \circ \cdots \circ \wt{g}_1 ) (w) ) = \Theta (w)$ and $f^{m+1}$ is injective on $\bigl( f_{\xi_{\minus m}}^{-1} \circ  \cdots \circ f_{\xi_0}^{-1} \bigr)  (X^0)$ (by (\ref{eqPflmLiftLI_ImageIsTile}) and Proposition~\ref{propCellDecomp}~(i)) with $\bigl( f^{m+1} \circ f_{\xi_{\minus m}}^{-1} \circ  \cdots \circ f_{\xi_0}^{-1} \bigr)  (x) = x$ for each $x\in X^0$, we get
\begin{equation*}
 ( \Theta \circ \wt{g}_{m+1} \circ \cdots \circ \wt{g}_1 ) (w)   =  \bigl( f_{\xi_{\minus m}}^{-1} \circ  \cdots \circ f_{\xi_0}^{-1} \bigr)  ( \Theta(w) ). 
\end{equation*}
Hence
\begin{align*}
&                       \sum\limits_{j=0}^{+\infty}    \Bigl(  \Bigl(   \wt{\psi} \circ \wt{g}_{j+1} \circ\cdots\circ \wt{g}_1 \Bigr)(u_{i-1})  
                                                                                      - \Bigl(   \wt{\psi} \circ \wt{g}_{j+1} \circ\cdots\circ \wt{g}_1 \Bigr)(u_i) \Bigr) \\ 
&\qquad  =      \sum\limits_{j=0}^{+\infty}    \Bigl(  \Bigl(   \psi \circ f_{\xi_{\minus j}}^{-1} \circ \cdots \circ f_{\xi_{0}}^{-1} \Bigr)  (   \Theta (u_{i-1})  )  
                                                                                     - \Bigl(   \psi \circ f_{\xi_{\minus j}}^{-1} \circ \cdots \circ f_{\xi_{0}}^{-1} \Bigr)  (   \Theta (u_i)  )   \Bigr)  \\
&\qquad  =     \Delta^{f,\,\CC}_{\psi,\,\xi} (  \Theta(u_{i-1}),   \Theta(u_i)  )                                                                                     
\end{align*}
where $\Delta^{f,\,\CC}_{\psi,\,\xi} $ is defined in (\ref{eqDelta}).

Similarly, the right-hand side of (\ref{eqLiftLI}) with $u$ and $v$ replaced by $u_{i-1}$ and $u_i$, respectively, is equal to $\Delta^{f,\,\CC}_{\psi,\,\eta} (  \Theta(u_{i-1}),   \Theta(u_i)  )$.

Since $\{ \Theta(u_{i-1}), \Theta(u_i) \} \subseteq X^n_i  \subseteq f ( \xi_0 )  = f ( \eta_0 )$ by Property~(2), we get from (\ref{eqLiftLI_LIAssumption}) and Definition~\ref{defTemporalDist} that
\begin{equation*}
0 = \psi^{f,\,\CC}_{\xi,\,\eta}( \Theta(u_{i-1}),   \Theta(u_i) ) = \Delta^{f,\,\CC}_{\psi,\,\xi} ( \Theta(u_{i-1}),   \Theta(u_i) ) - \Delta^{f,\,\CC}_{\psi,\,\eta} ( \Theta(u_{i-1}),   \Theta(u_i) ).
\end{equation*}
Therefore  (\ref{eqLiftLI}) holds with $u$ and $v$ replaced by $u_{i-1}$ and $u_i$, respectively. This establishes the lemma.
\end{proof}

\begin{proof}[Proof of Theorem~\ref{thmNLI}]
In the case when $\psi\in\Holder{\alpha}(S^2,d)$ is real-valued, the implication (vi)$\implies$(iii) follows immediately from Theorem~5.45 in \cite{Li17}. Conversely, statement~(iii) implies that $S_n\psi(x)=nK=S_n(K\mathbbm{1}_{S^2})(x)$ for all $x\in S^2$ and $n\in\N$ satisfying $f^n(x)=x$. So $K\in\R$. By Proposition~5.52 in \cite{Li17}, the function $\beta$ in statement~(iii) can be assumed to be real-valued. Then (vi) follows from Theorem~5.45 in \cite{Li17}.

\smallskip 
We now focus on the general case when $\psi\in\Holder{\alpha}((S^2,d),\C)$ is complex-valued. The implication (i)$\implies$(ii) is trivial.

\smallskip

To verify the implication (iii)$\implies$(iv), we note that statement~(iii) implies that $S_n\psi(x)=nK=S_n(K\mathbbm{1}_{S^2})(x)$ for all $x\in S^2$ and $n\in\N$ satisfying $f^n(x)=x$. Now statement~(iv) follows from Proposition~5.52 in \cite{Li17}.

\smallskip

To verify the implication (iv)$\implies$(i), we fix a Jordan curve $\CC\subseteq S^2$ satisfying $\post f\subseteq\CC$ and $f^n(\CC)\subseteq \CC$ for some $n\in\N$ (see Lemma~\ref{lmCexistsL}). We assume that statement~(iv) holds. Denote $F \coloneqq f^n$ and $\Psi \coloneqq S_n^f\psi=nK+\tau\circ F - \tau \in \Holder{\alpha} ((S^2,d),\C)$ (by Lemma~\ref{lmSnPhiHolder}). Fix any $\xi=\{ \xi_{\minus i} \}_{i\in\N_0} \in  \Sigma_{F,\,\CC}^-$ and $\eta=\{ \eta_{\minus i} \}_{i\in\N_0} \in  \Sigma_{F,\,\CC}^-$ with $F(\xi_0) = F(\eta_0)$. By (\ref{eqDelta}), we get that for all $(x,y)\in \bigcup\limits_{\substack{X\in\X^1(F,\CC) \\ X\subseteq F(\xi_0)}}X \times X$,
\begin{align*}
    \Delta^{F,\,\CC}_{\Psi,\,\xi}(x,y) 
= & \lim_{i\to+\infty}  \Bigl(  \tau(x)  -   \tau\bigl(F^{-1}_{\xi_{\minus i}} \circ\cdots\circ F^{-1}_{\xi_0}(x)\bigr)   
                                                - \tau(y) +   \tau\bigl(F^{-1}_{\xi_{\minus i}} \circ\cdots\circ F^{-1}_{\xi_0}(y)\bigr)  \Bigr) \\
= & \tau(x)-\tau(y).
\end{align*}
The second equality here follows from the H\"{o}lder continuity of $\tau$ and Lemma~\ref{lmCellBoundsBM}~(ii). Similarly, we have $\Delta^{F,\,\CC}_{\Psi,\,\eta}(x,y)= \tau(x)-\tau(y)$. Therefore by Definition~\ref{defTemporalDist}, $\Psi_{\xi,\,\eta}^{F,\,\CC}(x,y)=0$ for all $(x,y)\in \bigcup\limits_{\substack{X\in\X^1(F,\CC) \\ X\subseteq F(\xi_0)}}X \times X$, establishing (i).

\smallskip

To verify the implication (iv)$\implies$(v), we define $M\coloneqq \mathbbm{1}_{\Sigma_{A_{\ti}}^+}$ and $\varpi \coloneqq \tau \circ \pi_{\ti}$. Then (v) follows immediately from Proposition~\ref{propTileSFT}.

\smallskip

To verify the implication (v)$\implies$(ii), we argue by contradiction and assume that (ii) does not hold but (v) holds. Fix $n\in\N$, $\CC\subseteq S^2$, $K\in\C$, $M\in\CCC \bigl( \Sigma_{A_{\ti}}^+ , \Z \bigr)$, and $\varpi\in\CCC \bigl( \Sigma_{A_{\ti}}^+ , \C \bigr)$ as in (v). Recall $F\coloneqq f^n$, $\Psi \coloneqq S_n^f \psi$, and
\begin{equation}   \label{eqPfthmNLI_DefPhiPi_PsiCohomology}
\Psi \circ \pi_{\ti} = KM + \varpi \circ \sigma_{A_{\ti}} - \varpi.
\end{equation}
Since $ \Sigma_{A_{\ti}}^+$ is compact, we know that $\card \bigl( M \bigl(  \Sigma_{A_{\ti}}^+ \bigr) \bigr)$ is finite. Thus considering that the topology on $\Sigma_{A_{\ti}}^+$ is induced from the product topology, we can choose $m\in\N$ such that $M( \underline{u} ) = M ( \underline{v} )$ for all $\underline{u} = \{u_i\}_{i\in\N_0}\in  \Sigma_{A_{\ti}}^+$ and $\underline{v} = \{v_i\}_{i\in\N_0} \in  \Sigma_{A_{\ti}}^+$ with $u_i=v_i$ for each $i\in\{0,1,\dots, m\}$. Fix $\xi=\{ \xi_{\minus i} \}_{i\in\N_0} \in  \Sigma_{F,\,\CC}^-$, $\eta=\{ \eta_{\minus i} \}_{i\in\N_0} \in  \Sigma_{F,\,\CC}^-$, $X \in \X^1(F,\CC)$, and $x,y\in X$ with $X \subseteq F(\xi_0) = F(\eta_0)$ and
\begin{equation*}
\Psi_{\xi,\,\eta}^{F,\,\CC}(x,y) \neq 0
\end{equation*}
as in (ii). Since $D\coloneqq S^2 \setminus \bigcup_{i\in\N_0} F^{-i} (\CC)$ is dense in $S^2$, by (\ref{eqDeltaDifference}) in Lemma~\ref{lmDeltaHolder} and Definition~\ref{defTemporalDist}, we can assume without loss of generality that there exists $X^m \in \X^m(F,\CC)$ with $x,y\in X^m \setminus D \subseteq X$. Thus by Proposition~\ref{propTileSFT}, $\pi_{\ti}$ is injective on $\pi_{\ti}^{-1} (P)$ where $P\coloneqq \bigcup_{i\in\N_0} F^{-i} (\{x,y\})$. With abuse of notation, we denote by $\pi_{\ti}^{-1} \: P \rightarrow \pi_{\ti}^{-1}(P)$ the inverse of $\pi_{\ti}$ on $P$. Then by (\ref{eqDelta}), Proposition~\ref{propTileSFT}, and (\ref{eqPfthmNLI_DefPhiPi_PsiCohomology}),
\begin{align*}
        \Delta^{F,\,\CC}_{\Psi,\,\xi}(x,y) 
= &  \sum\limits_{i=0}^{+\infty} \Bigl( 
                     \bigl( ( \Psi \circ \pi_{\ti} )  \circ \bigl( \pi_{\ti}^{-1} \circ F^{-1}_{\xi_{\minus i}} \circ \pi_{\ti} \bigr) \circ \cdots \circ \bigl( \pi_{\ti}^{-1} \circ F^{-1}_{\xi_0}  \circ \pi_{\ti} \bigr) \bigr) \bigl( \pi_{\ti}^{-1} (x) \bigr) \\
   & \quad - \bigl( ( \Psi \circ \pi_{\ti} )  \circ \bigl( \pi_{\ti}^{-1} \circ F^{-1}_{\xi_{\minus i}} \circ \pi_{\ti} \bigr) \circ \cdots \circ \bigl( \pi_{\ti}^{-1} \circ F^{-1}_{\xi_0}  \circ \pi_{\ti} \bigr) \bigr) \bigl( \pi_{\ti}^{-1} (y) \bigr)
                                      \Bigr) \\
= & \lim_{i\to+\infty}  \Bigl( \varpi \bigl( \pi_{\ti}^{-1}(x) \bigr)  -\varpi \bigl( \pi_{\ti}^{-1}(y) \bigr) \\
    &\qquad\quad                -\bigl( \varpi \circ \bigl( \pi_{\ti}^{-1} \circ F^{-1}_{\xi_{\minus i}} \circ \pi_{\ti} \bigr) \circ \cdots \circ \bigl( \pi_{\ti}^{-1} \circ F^{-1}_{\xi_0}  \circ \pi_{\ti} \bigr) \bigr) \bigl( \pi_{\ti}^{-1} (x) \bigr) \\
    &\qquad\quad               +\bigl( \varpi \circ \bigl( \pi_{\ti}^{-1} \circ F^{-1}_{\xi_{\minus i}} \circ \pi_{\ti} \bigr) \circ \cdots \circ \bigl( \pi_{\ti}^{-1} \circ F^{-1}_{\xi_0}  \circ \pi_{\ti} \bigr) \bigr) \bigl( \pi_{\ti}^{-1} (y) \bigr)  \Bigr) \\
= & \varpi \bigl( \pi_{\ti}^{-1}(x) \bigr)  -\varpi \bigl( \pi_{\ti}^{-1}(y) \bigr).
\end{align*}
The last identity follows from the uniform continuity of $\varpi$. Similarly, $ \Delta^{F,\,\CC}_{\Psi,\,\eta}(x,y) = \varpi \bigl( \pi_{\ti}^{-1}(x) \bigr)  -\varpi \bigl( \pi_{\ti}^{-1}(y) \bigr) $. Thus by Definition~\ref{defTemporalDist}, $\Psi_{\xi,\,\eta}^{F,\,\CC}(x,y)=0$, a contradiction. The implication (v)$\implies$(ii) is now established.

\smallskip

It remains to show the implication (ii)$\implies$(iii).

We assume that statement~(ii) holds. Fix $n\in\N$ and a Jordan curve $\CC\subseteq S^2$ with $\post f \subseteq \CC$ and $f^n(\CC)\subseteq \CC$. We denote $F \coloneqq f^n$, and $\Psi \coloneqq S_n^f \psi\in\Holder{\alpha}((S^2,d),\C)$ (see Lemma~\ref{lmSnPhiHolder}).

Let $\mathcal{O}_F=(S^2,\alpha_F)$ be the orbifold associated to $F$. By Proposition~\ref{propSameRamificationFn}, we have $\mathcal{O}_F=\mathcal{O}_f$. Let $\Theta\:\XX\rightarrow S^2_0$ be the universal orbifold covering map of $\mathcal{O}_F=\mathcal{O}_f$, which depends only on $f$, and in particular, is independent of $n$.

For each branched covering map $\wt{h}\in \Inv(F)$, i.e., $\wt{h}\: \XX\rightarrow\XX$ satisfying $F\circ \Theta \circ \wt{h} = \Theta$, we define a function $\wt{\beta}_{\wt{h}} \: \XX\rightarrow\C$ by
\begin{equation}  \label{eqDef_b_h}
\wt{\beta}_{\wt{h}} (u)   \coloneqq   \sum\limits_{i=1}^{+\infty}  \bigl( (\Psi\circ\Theta)\bigl(\wt{h}^i (u)\bigr)  -  \Psi_{\wt{h}}   \bigr)
\end{equation}
for $u\in\XX$, where
\begin{equation}  \label{eqPfthmNLI_DefPhi_wt_h}
\Psi_{\wt{h}} \coloneqq \lim\limits_{j\to+\infty} (\Psi\circ\Theta)\bigl(\wt{h}^j (u)\bigr)
\end{equation}
converges and the limit in (\ref{eqPfthmNLI_DefPhi_wt_h}) is independent of $u\in\XX$ by Proposition~\ref{propInvBranchFixPt}. 

Fix an arbitrary point $u\in\XX$. We will show that the series in (\ref{eqDef_b_h}) converges absolutely. Note that it follows immediately from (\ref{eqGoodPathLimitLength0}) in Corollary~\ref{corGoodPath} (applied to $x\coloneqq u$ and $y\coloneqq \wt{h}(u)$) that there exists an integer $k\in\N$ such that for each $i\in\N$,
\begin{align*}
&                        \AbsBig{  (\Psi\circ\Theta)\bigl(\wt{h}^i (u)\bigr) - \lim\limits_{j\to+\infty} (\Psi\circ\Theta)\bigl(\wt{h}^{i+j} (u)\bigr)  } \\
&\qquad  \leq  \sum\limits_{j=1}^{+\infty}  \AbsBig{  (\Psi\circ\Theta) \bigl(\wt{h}^{i+j-1} (u)\bigr) - (\Psi\circ\Theta) \bigl(\wt{h}^{i+j} (u)\bigr)  }  \\
&\qquad \leq \sum\limits_{j=1}^{+\infty} k^\alpha C^\alpha \Lambda^{-(i+j-1)\alpha} \Hnorm{\alpha}{\Psi}{(S^2,d)}   \\
&\qquad \leq \frac{ k C}{ 1-\Lambda^{-\alpha}}      \Lambda^{-i\alpha}  \Hnorm{\alpha}{\Psi}{(S^2,d)},
\end{align*} 
where $C\geq 1$ is a constant from Lemma~\ref{lmCellBoundsBM} depending only on $f$, $\CC$, and $d$. Thus the series in (\ref{eqDef_b_h}) converges absolutely, and $\wt{\beta}_{\wt{h}}$ is well-defined and continuous on $\XX$.

Hence, for arbitrary $l\in\N_0$ and $u\in\XX$, it follows from (\ref{eqGoodPathLimitLength0}) in Corollary~\ref{corGoodPath} (applied to $x\coloneqq u$ and $y\coloneqq \wt{h}(u)$) that
\begin{align*}
&                      \Absbigg{ \wt{\beta}_{\wt{h}} (u) - \sum\limits_{i=1}^{+\infty}  \bigl( (\Psi\circ\Theta)\bigl(\wt{h}^i (u)\bigr)- (\Psi\circ\Theta)\bigl(\wt{h}^{i} \bigl(\wt{h}^{l}(u)\bigr)\bigr)  \bigr) } \\
&\qquad  \leq \sum\limits_{i=1}^{+\infty} \AbsBig{ (\Psi\circ\Theta)\bigl(\wt{h}^{i+l}  (u) \bigr)   -   \lim\limits_{j\to+\infty} (\Psi\circ\Theta)\bigl(\wt{h}^{i+j} (u)\bigr)  \bigr) } \\
&\qquad  \leq \sum\limits_{i=1}^{+\infty}   \sum\limits_{j=l+1}^{+\infty}    \AbsBig{  (\Psi\circ\Theta) \bigl(\wt{h}^{i+j-1} (u)\bigr) - (\Psi\circ\Theta) \bigl(\wt{h}^{i+j} (u)\bigr)  }  \\
&\qquad  \leq \sum\limits_{i=1}^{+\infty}   \sum\limits_{j=l+1}^{+\infty}    k^\alpha C^\alpha \Lambda^{-(i+j-1)\alpha} \Hnorm{\alpha}{\Psi}{(S^2,d)}   \\
&\qquad  \leq \frac{k C}{ (1-\Lambda^{-\alpha})^2 }      \Lambda^{-l\alpha}  \Hnorm{\alpha}{\Psi}{(S^2,d)}.
\end{align*}
Hence for each $u\in\XX$,
\begin{equation}  \label{eqSwitchLimitSum}
\wt{\beta}_{\wt{h}} (u)  = \lim\limits_{j\to+\infty} \sum\limits_{i=1}^{+\infty}  \bigl( (\Psi\circ\Theta)\bigl(\wt{h}^i (u)\bigr)- (\Psi\circ\Theta)\bigl(\wt{h}^{i} \bigl(\wt{h}^{j}(u)\bigr)\bigr)  \bigr).
\end{equation}

We now fix an inverse branch $\wt{g}\in\Inv(F)$ of $F$ on $\XX$ and consider the map $\wt{\beta}_{\wt{g}}$. Note that by the absolute convergence of the series in (\ref{eqDef_b_h}), for each $u\in \XX$,
\begin{equation}  \label{eqPreCohomology}
\wt{\beta}_{\wt{g}}(u) - \wt{\beta}_{\wt{g}}  ( \wt{g}(u)  ) = \Psi (\Theta  (\wt{g}(u)  )  ) - \Psi_{\wt{g}} .
\end{equation}

We claim that $\wt{\beta}_{\wt{g}}(u) = \wt{\beta}_{\wt{g}}(\wt{\sigma}(u))$ for each $u\in\XX$ and each deck transformation $\wt{\sigma}\in \pi_1(\mathcal{O}_f)$.

By Proposition~\ref{propDeckTransf} and the fact that $\mathcal{O}_f$ is either parabolic or hyperbolic (see \cite[Proposition~2.12]{BM17}), the claim is equivalent to 
\begin{equation}  \label{eqBeta_hInvFiberTheta}
\wt{\beta}_{\wt{g}}(u)= \wt{\beta}_{\wt{g}}(v), \qquad\qquad \text{for all }y\in S^2, u,v\in\Theta^{-1}(y).
\end{equation}

We assume that the claim holds for now and postpone its proof to the end of this discussion. Then by (\ref{eqBeta_hInvFiberTheta}), the function $\beta \: S^2_0 \rightarrow \C$, defined by
\begin{equation}  \label{eqDefBeta}
\beta(y) \coloneqq \wt{\beta}_{\wt{g}} (v)
\end{equation}
with $v\in \Theta^{-1}(y)$ independent of the choice of $v$, for $y\in S^2_0$, is well-defined.

For each $x\in S^2_0$, by the surjectivity of $\wt{g}$, we can choose $u_0\in\XX$ such that $\Theta (\wt{g}(u_0) ) = x$. Note that $\Theta(u_0)=( F\circ\Theta\circ\wt{g} ) (u_0) = F(x)$. So $\wt{g}(u_0) \in \Theta^{-1}(x)$ and $u_0\in\Theta^{-1}(F(x))$. Then by (\ref{eqPreCohomology}), 
\begin{equation*}
\beta(F(x)) - \beta(x)  
= \wt{\beta}_{\wt{g}} (u_0) - \wt{\beta}_{\wt{g}}  (\wt{g}(u_0) ) 
= \Psi (\Theta  (\wt{g}(u_0)  )  ) - \Psi_{\wt{g}}
= \Psi(x) -  \Psi_{\wt{g}}.
\end{equation*}


We will show that $\beta \in \Holder{\alpha} \bigl( \bigl( X^0_\c \setminus \post F, d \bigr), \C \bigr)$ for each $\c\in\{\b,\w\}$. Here $X^0_\b$ (resp.\ $X^0_\w$) is the black (resp.\ white) $0$-tile in $\X^0(F,\CC)$.

Fix arbitrary $\c\in\{\b,\w\}$ and $x_0,y_0 \in X^0_\c \setminus \post F$. Let $\gamma \: [0,1] \rightarrow X^0_\c \setminus \post F$ be an arbitrary continuous path with $\gamma(0) = x_0$, $\gamma(1) = y_0$, and $\gamma((0,1)) \subseteq \inte \bigl(X^0_\c \bigr)$.

By Lemma~\ref{lmLiftPathBM}, we can lift $\gamma$ to $\wt\gamma \: [0,1] \rightarrow \XX$ such that $\Theta \circ \wt\gamma = \gamma$. Denote $u \coloneqq \wt\gamma(0)$, $v \coloneqq \wt\gamma(1)$, and $I \coloneqq (0,1)$. Thus
\begin{equation}   \label{eqPfthmNLI_InTileInte}
(\Theta \circ \wt\gamma) (I) \subseteq \inte \bigl( X^0_\c \bigr).
\end{equation}

Fix an arbitrary integers $m\in\N$.

By (\ref{eqPfthmNLI_InTileInte}), each connected component of $F^{-m}((\Theta \circ \wt\gamma)(I))$ is contained in some connected component of $F^{-m}\bigl(\inte \bigl(X^0_\c\bigr) \bigr)$. Since both $\bigl( \Theta \circ \wt{g}^m \circ \wt \gamma \bigr) (I)$ and $F^m \bigl( \bigl( \Theta \circ \wt{g}^m \circ \wt\gamma \bigr) (I) \bigr) = (\Theta \circ \wt\gamma)(I)$ are connected, by Proposition~\ref{propCellDecomp}~(v), there exists an $m$-tile $X^m \in \X^m(F,\CC)$ such that 
\begin{equation*}  
\bigl( \Theta \circ \wt{g}^m \circ \wt\gamma \bigr) (I)   \subseteq   \inte ( X^m ). 
\end{equation*}
We denote $x_m \coloneqq (\Theta \circ \wt{g}^m ) (u)$ and $y_m \coloneqq (\Theta \circ \wt{g}^m ) (v)$. Then $F^m(x_m) =  ( F^m \circ \Theta \circ \wt{g}^m ) (u) = x_0$, $F^m(y_m) =  ( F^m \circ \Theta \circ \wt{g}^m ) (v) = y_0$, and $x_m, y_m \in X^m$. Hence by (\ref{eqDefBeta}), the absolute convergence of the series defining $\wt{\beta}_{\wt{g}}$ in (\ref{eqDef_b_h}), and Lemma~\ref{lmSnPhiBound},
\begin{align*}
            \abs{ \beta(x_0) - \beta(y_0)} 
&=       \AbsBig{ \wt\beta_{\wt{g}}(u) -\wt\beta_{\wt{g}}(v) }  \\
&=       \lim\limits_{m\to+\infty}  \Absbigg{  \sum\limits_{i=1}^m \bigl( \bigl( \Psi \circ \Theta \circ \wt{g}^i \bigr) (u) - \bigl( \Psi \circ \Theta \circ \wt{g}^i \bigr) (v)  \bigr) }  \\
&=       \lim\limits_{m\to+\infty}  \Absbigg{  \sum\limits_{j=0}^{m-1} \bigl( \bigl( \Psi \circ F^j \circ \Theta \circ \wt{g}^m \bigr) (u) - \bigl( \Psi \circ F^j \circ \Theta \circ \wt{g}^m \bigr) (v)  \bigr) }  \\
&=       \lim\limits_{m\to+\infty} \Absbig{ S^F_m \Psi (x_m) - S^F_m \Psi (y_m) }  \\
&\leq  \limsup\limits_{m\to+\infty} C_1 d ( F^m (x_m), F^m (y_m) )^\alpha \\
&= C_1 d(x_0, y_0)^\alpha,
\end{align*}
where $C_1 = C_1(F,\CC,d,\Psi,\alpha)>0$ is a constant depending only on $F$, $\CC$, $d$, $\Psi$, and $\alpha$ from Lemma~\ref{lmSnPhiBound}.

Hence $\beta \in \Holder{\alpha} \bigl( \bigl( X^0_\c \setminus \post F, d \bigr), \C \bigr)$.

We can now extend $\beta$ continuously to $X^0_\c$, denoted by $\beta_\c$. Since $\c\in\{\b,\w\}$ is arbitrary, $\bigl( X^0_\b \setminus \post F \bigr) \cap \bigl( X^0_\b \setminus \post F \bigr) = \CC \setminus \post F$, and $\post F$ is a finite set, we get $\beta_\b|_\CC = \beta_\w|_\CC$. Thus $\beta$ can be extended to $S^2$, and this shows that $\Psi=S_n^f\psi$ is co-homologous to a constant $K_1$ in $\CCC(S^2,\C)$. Therefore, by Lemma~5.53 in \cite{Li17},
\begin{equation*}
\psi = K - \beta + \beta\circ f
\end{equation*}
for some constant $K\in\C$, establishing statement~(iii).

Thus it suffices to prove the claim. The verification of the claim occupies the remaining part of this proof.

We denote $\wt\Psi \coloneqq \Psi \circ \Theta$.

Let $\wt\sigma \in \pi_1(\mathcal{O}_f)$ be a deck transformation on $X$. Then by Definition~\ref{defInverseBranch} and Definition~\ref{defDeckTransf} it is clear that $\wt\sigma \circ \wt{g} \in \Inv(F)$. Thus by the absolute convergence of the series defining $\wt{\beta}_{\wt{g}}$ in (\ref{eqDef_b_h}), for each $u\in\XX$,
\begin{align}    \label{eqPfthmNLI_PfClaim1}
& \wt{\beta}_{\wt{g}} ( (\wt\sigma \circ \wt{g})(u) )  - \wt{\beta}_{\wt{g}} (  \wt{g}(u) ) \notag\\
&\qquad =  \bigl(\wt{\beta}_{\wt{g}} ( (\wt\sigma \circ \wt{g})(u) ) - \wt{\beta}_{\wt{g}}  (u) \bigr)
    -\bigl(\wt{\beta}_{\wt{g}} ( \wt{g}(u) ) - \wt{\beta}_{\wt{g}}  (u) \bigr)  \\
&\qquad =  \sum\limits_{i=1}^{+\infty} \Bigl( \wt\Psi \bigl(\wt{g}^i ((\wt\sigma \circ \wt{g}) (u) ) \bigr)   -  \wt\Psi \bigl(\wt{g}^i (u) \bigr) \Bigr)
                                                                 + \Bigl( \wt\Psi(\wt{g}(u)) - \Psi_{\wt{g}} \Bigr).  \notag
\end{align}
Fix arbitrary $i_0,j_0\in\N$. It follows immediately from (\ref{eqGoodPathLimitLength0}) in Corollary~\ref{corGoodPath} that all three series
\begin{align*}
&   \sum\limits_{i=1}^{+\infty} \AbsBig{ \wt\Psi \bigl(\wt{g}^i ((\wt\sigma \circ \wt{g}) (u) ) \bigr) 
                                           -\wt\Psi \bigl( \wt{g}^{i} \bigl( (\wt\sigma \circ \wt{g})^{j_0} (u) \bigr) \bigr)   },   \\
&  \sum\limits_{i=1}^{+\infty}  \AbsBig{ \wt\Psi \bigl(\wt{g}^i (u) \bigr) 
                                          -\wt\Psi \bigl( \wt{g}^{i} \bigl( (\wt\sigma \circ \wt{g})^{j_0} (u) \bigr) \bigr)   } ,   \text{ and}\\                                       
&   \sum\limits_{j=1}^{+\infty} \AbsBig{ \wt\Psi \bigl(\wt{g}^{i_0} ((\wt\sigma \circ \wt{g})^j (u) ) \bigr) 
                                           -\wt\Psi \bigl( \wt{g}^{i_0} \bigl( (\wt\sigma \circ \wt{g})^{j+1} (u) \bigr) \bigr)   }
\end{align*}
are majorized by convergent geometric series. Hence the right-hand side of (\ref{eqPfthmNLI_PfClaim1}) is equal to
\begin{align*}
&\qquad =   \sum\limits_{i=1}^{+\infty} \Bigl(\wt\Psi \bigl(\wt{g}^i ((\wt\sigma \circ \wt{g}) (u) ) \bigr) 
                                           -\wt\Psi \bigl( \wt{g}^{i} \bigl( (\wt\sigma \circ \wt{g})^{j_0} (u) \bigr) \bigr)   \Bigr)   \\
&\qquad\quad      - \sum\limits_{i=1}^{+\infty} \Bigl(\wt\Psi \bigl(\wt{g}^i (u) \bigr) 
                                          -\wt\Psi \bigl( \wt{g}^{i} \bigl( (\wt\sigma \circ \wt{g})^{j_0} (u) \bigr) \bigr)   \Bigr)
   + \Bigl( \wt\Psi(\wt{g}(u)) -  \Psi_{\wt{g}} \Bigr) \\
&\qquad =   \lim\limits_{j\to+\infty} \sum\limits_{i=1}^{+\infty} \Bigl(\wt\Psi \bigl(\wt{g}^i ((\wt\sigma \circ \wt{g}) (u) ) \bigr) 
                                           -\wt\Psi \bigl( \wt{g}^{i} \bigl( (\wt\sigma \circ \wt{g})^j (u) \bigr) \bigr)   \Bigr)   \\
&\qquad\quad      - \lim\limits_{j\to+\infty} \sum\limits_{i=1}^{+\infty} \Bigl(\wt\Psi \bigl(\wt{g}^i (u) \bigr) 
                                          -\wt\Psi \bigl( \wt{g}^{i} \bigl( (\wt\sigma \circ \wt{g})^j (u) \bigr) \bigr)   \Bigr)
   + \Bigl( \wt\Psi(\wt{g}(u)) -  \Psi_{\wt{g}} \Bigr).
\end{align*}
Then by Lemma~\ref{lmLiftLI}, (\ref{eqSwitchLimitSum}), and (\ref{eqPreCohomology}), the right-hand side of the above equation is equal to
\begin{align*}
 &     \lim\limits_{j\to+\infty} \sum\limits_{i=1}^{+\infty} \Bigl(\wt\Psi \bigl((\wt\sigma\circ\wt{g})^i ((\wt\sigma \circ \wt{g}) (u) ) \bigr) 
                                           -\wt\Psi \bigl( (\wt\sigma\circ\wt{g})^{i} \bigl( (\wt\sigma \circ \wt{g})^j (u) \bigr) \bigr)   \Bigr)   \\
&\quad     - \lim\limits_{j\to+\infty} \sum\limits_{i=1}^{+\infty} \Bigl(\wt\Psi \bigl((\wt\sigma\circ\wt{g})^i (u) \bigr)  
                                          -\wt\Psi \bigl( (\wt\sigma\circ\wt{g})^{i} \bigl( (\wt\sigma \circ \wt{g})^j (u) \bigr) \bigr)   \Bigr) 
                           + \Bigl( \wt\Psi(\wt{g}(u)) - \Psi_{\wt{g}} \Bigr)\\
&\qquad=    \wt{\beta}_{ \wt{\sigma} \circsmall \wt{g} }  ( ( \wt{\sigma} \circ \wt{g}  ) (u) )  -   \wt{\beta}_{ \wt{\sigma} \circsmall \wt{g} }  ( u ) +  \wt\Psi(\wt{g}(u)) - \Psi_{\wt{g}} \\
&\qquad=   - \wt\Psi((\wt\sigma \circ \wt{g})(u)) + \Psi_{ \wt{\sigma} \circsmall \wt{g} }    +  \wt\Psi(\wt{g}(u)) -  \Psi_{\wt{g}} \\
&\qquad=    \Psi_{ \wt{\sigma} \circsmall \wt{g} }     -  \Psi_{\wt{g}} .
\end{align*}
The last equality follows from $\wt{\Psi}\circ\wt{\sigma}= \Psi\circ(\Theta\circ\wt{\sigma}) = \Psi\circ\Theta = \wt{\Psi}$. Since $\wt{g} \: \XX \rightarrow \XX$ is surjective, we can conclude that for each $v\in\XX$ and each $\wt{\sigma} \in \pi_1( \mathcal{O}_f )$,
\begin{equation}  \label{eqPfthmNLI_PfClaim2}
\wt{\beta}_{\wt{g}} ( \wt{\sigma} (v) ) - \wt{\beta}_{\wt{g}} ( v ) =  \Psi_{ \wt{\sigma} \circsmall \wt{g} }     -  \Psi_{\wt{g}} .
\end{equation}

The claim follows once we show that $ \Psi_{ \wt{\sigma} \circsmall \wt{g} }    =  \Psi_{\wt{g}}$ for each $\wt{\sigma} \in \pi_1( \mathcal{O}_f )$. We argue by contradiction and assume that $\Psi_{ \wt{\sigma} \circsmall \wt{g} }    -  \Psi_{\wt{g}} \neq 0$ for some $\wt{\sigma} \in \pi_1( \mathcal{O}_f )$. Then by (\ref{eqPfthmNLI_PfClaim2}), for each $k\in \N$,
\begin{equation*}
        \Psi_{ \wt{\sigma}^k \circsmall \wt{g} }    -  \Psi_{\wt{g}}  
    = \wt{\beta}_{\wt{g}} \bigl( \wt{\sigma}^k (v) \bigr) - \wt{\beta}_{\wt{g}} ( v ) 
    = \sum\limits_{i=0}^{k-1}  \Bigl(   \wt{\beta}_{\wt{g}} \bigl( \wt{\sigma} \bigl( \bigl( \wt{\sigma}^i (v) \bigr) \bigr) \bigr) -  \wt{\beta}_{\wt{g}} \bigl( \wt{\sigma}^i (v) \bigr)   \Bigr)
    = k \bigl( \Psi_{ \wt{\sigma} \circsmall \wt{g} }    -  \Psi_{\wt{g}}   \bigr).
\end{equation*}
However, by (\ref{eqPfthmNLI_DefPhi_wt_h}), Proposition~\ref{propInvBranchFixPt}, and Theorem~\ref{thmETMBasicProperties}~(ii), 
\begin{align*}
\card \bigl\{  \Psi_{ \wt{\sigma}^k \circsmall \wt{g} }  \,\big|\, k\in\N \bigr\} 
\leq &\card \bigl\{  \Psi_{ \wt{h} }  \,\big|\, \wt{h}\in \Inv(F) \bigr\} \\
\leq &\card \{ \Psi(x) \,|\, x\in S^2, \, F(x)=x \}
< +\infty. 
\end{align*}
This is a contradiction.

The claim is now proved, establishing the implication  (ii)$\implies$(iii).
\end{proof}

\section{Ruelle operators and split Ruelle operators}    \label{sctSplitRuelleOp}

In this section, we define appropriate variations of the Ruelle operator on the suitable function spaces in our context and establish some important inequalities that will be used later. More precisely, in Subsection~\ref{subsctSplitRuelleOp_Construction}, for an expanding Thurston map $f$ with a forward invariant Jordan curve $\CC\subseteq S^2$ and a complex-valued H\"{o}lder continuous function $\psi$, we ``split'' the Ruelle operator $\RR_{\psi} \: \CCC(S^2,\C) \rightarrow \CCC(S^2,\C)$ into pieces $\RR_{\psi,X^0_\c,E}^{(n)} \: \CCC(E,\C) \rightarrow C\bigl(X^0_\c,\C\bigr)$, for $\c\in\{\b,\w\}$, $n\in\N_0$, and a union $E$ of $n$-tiles in the cell decomposition $\X^n(f,\CC)$ of $S^2$ induced by $f$ and $\CC$. Such construction is crucial to the proof of Proposition~\ref{propTelescoping} where the image of characteristic functions supported on $n$-tiles under $\RR_{\psi,X^0_\c,E}^{(n)}$ are used to relate periodic points and preimage points of $f$. We then define the \emph{split Ruelle operators} $\RRR_\psi$ on the product space $\CCC\bigl(X^0_\b,\C\bigr) \times \CCC\bigl(X^0_\w, \C\bigr)$ by piecing together $\RR_{\psi,X^0_{\c_1},X^0_{\c_2}}^{(1)}$, $\c_1,\c_2\in\{\b,\w\}$. Subsection~\ref{subsctSplitRuelleOp_BasicIneq} is devoted to establishing various inequalities, among them the \emph{basic inequalities} in Lemma~\ref{lmBasicIneq}, that are indispensable in the arguments in Section~\ref{sctDolgopyat}. In Subsection~\ref{subsctSplitRuelleOp_SpectralGap}, we verify the spectral gap for $\RRR_\psi$ that are essential in the proof of the basic inequalities.

\subsection{Construction}   \label{subsctSplitRuelleOp_Construction}

\begin{lemma}   \label{lmExpToLiniear}
Let $f$, $\CC$, $d$, $\Lambda$, $\alpha$ satisfy the Assumptions. Fix a constant $T>0$. Then for all $n\in\N$, $X^n\in\X^n(f,\CC)$, $x,x'\in X^n$, and $\psi\in\Holder{\alpha}((S^2,d),\C)$ with $\Hseminorm{\alpha,\, (S^2,d)}{\Re(\psi)} \leq T$, we have
\begin{equation}  \label{eqExpToLinear}
\abs{1-\exp(S_n\psi(x)-S_n\psi(x'))} \leq C_{10} \Hseminorm{\alpha,\, (S^2,d)}{\psi} d(f^n(x),f^n(x'))^\alpha,
\end{equation}
where the constant
\begin{equation}     \label{eqDefC10}
C_{10} = C_{10}(f,\CC,d,\alpha,T) \coloneqq \frac{2 C_0 }{1-\Lambda^{-\alpha}} \exp\biggl(\frac{C_0 T }{1-\Lambda^{-\alpha}}\bigl(\diam_d(S^2)\bigr)^\alpha\biggr)  >1
\end{equation}
depends only on $f$, $\CC$, $d$, $\alpha$, and $T$. Here $C_0>1$ is a constant from Lemma~\ref{lmMetricDistortion} depending only on $f$, $\CC$, and $d$.
\end{lemma}

\begin{proof}
Fix $T>0$, $n\in\N$, $X^n\in\X^n(f,\CC)$, $x,x'\in X^n$, and $\psi\in\Holder{\alpha}((S^2,d),\C)$ with $\Hseminorm{\alpha,\, (S^2,d)}{\Re(\psi)} \leq T$. By Lemma~\ref{lmSnPhiBound}, for each $\phi\in\Holder{\alpha}(S^2,d)$,
\begin{equation}   \label{eqPflmExpToLinear1}
\abs{S_n \phi(x)-S_n \phi(x')}  \leq \frac{C_0\Hseminorm{\alpha,\, (S^2,d)}{\phi} }{1-\Lambda^{-\alpha}}  d(f^n(x),f^n(x') )^\alpha.
\end{equation}

Then by (\ref{eqPflmExpToLinear1}) and the fact that $\abs{1-e^y} \leq  \abs{y}e^{\abs{y}}$ and $\abs{1-e^{\I y}} \leq \abs{y}$ for $y\in\R$, we get
\begin{align*}
&             \abs{1-\exp(S_n\psi(x)-S_n\psi(x'))}\\
&\qquad\leq   \abs{1-\exp(S_n \Re(\psi)(x)-S_n\Re(\psi)(x'))} \\
&\qquad\quad +\exp(S_n \Re(\psi)(x)-S_n\Re(\psi)(x'))  \abs{1-\exp(\I S_n \Im(\psi)(x) - \I S_n\Im(\psi)(x'))} \\
&\qquad\leq   \frac{C_0\Hseminorm{\alpha,\, (S^2,d)}{\Re(\psi)} }{1-\Lambda^{-\alpha}}  d(f^n(x),f^n(x'))^\alpha  \exp\biggl(\frac{ C_0 T}{1-\Lambda^{-\alpha}}\bigl(\diam_d(S^2)\bigr)^\alpha\biggr)  \\
&\qquad\quad +\exp\biggl(\frac{C_0 T }{1-\Lambda^{-\alpha}}\bigl(\diam_d(S^2)\bigr)^\alpha\biggr) \frac{C_0\Hseminorm{\alpha,\, (S^2,d)}{\Im(\psi)} }{1-\Lambda^{-\alpha}}  d(f^n(x),f^n(x'))^\alpha\\
&\qquad\leq   C_{10}  \Hseminorm{\alpha,\, (S^2,d)}{\psi} d(f^n(x),f^n(x'))^\alpha.
\end{align*}
Here the constant $C_{10} =C_{10} (f, \CC, d, \alpha, T )$ is defined in (\ref{eqDefC10}).
\end{proof}

Fix an expanding Thurston map $f\: S^2 \rightarrow S^2 $ with a Jordan curve $\CC\subseteq S^2$ satisfying $\post f\subseteq \CC$. Let $d$ be a visual metric for $f$ on $S^2$, and $\psi\in\Holder{\alpha}((S^2,d),\C)$ a complex-valued H\"{o}lder continuous function.

Let $n\in\N$, $\c \in\{\b,\w\}$, and $x \in\inte (X^0_\c)$, where $X^0_\b$ (resp.\ $X^0_\w$) is the black (resp.\ white) $0$-tile. If $E\subseteq S^2$ is a union of $n$-tiles in $\X^n(f,\CC)$, $u\in\CCC((E,d),\C)$ a complex-valued continuous function defined on $E$, and if we define a function $v\in B(S^2,\C)$ by
\begin{equation}   \label{eqExtendToS2}
v(y) =  \begin{cases} u(y) & \text{if } y\in E, \\ 0  & \text{otherwise}, \end{cases}
\end{equation}
then
\begin{equation}  \label{eqDefLcInte}
\RR_\psi^n (v) (x) = \sum\limits_{\substack{X^n\in\X^n_\c\\ X^n\subseteq E}} u\bigl((f^n|_{X^n})^{-1}(x)\bigr) \exp \bigl(S_n\psi\bigl((f^n|_{X^n})^{-1}(x)\bigr)\bigr).
\end{equation}

Note that by default, a summation over an empty set is equal to $0$. We will always use this convention in this paper. Inspired by (\ref{eqDefLcInte}), we give the following definition.

\begin{definition}   \label{defSplitRuelle}
Let $f\: S^2 \rightarrow S^2 $ be an expanding Thurston map with a Jordan curve $\CC\subseteq S^2$ satisfying $\post f\subseteq \CC$. Let $d$ be a visual metric for $f$ on $S^2$, and $\psi\in\Holder{\alpha}((S^2,d),\C)$ a complex-valued H\"{o}lder continuous function. Let $n\in\N_0$, and $E\subseteq S^2$ be a union of $n$-tiles in $\X^n(f,\CC)$. We define a map $\RR_{\psi,X^0_\c,E}^{(n)} \: \CCC(E,\C) \rightarrow \CCC\bigl(X^0_\c,\C\bigr)$, for each $\c\in\{\b,\w\}$, by
\begin{equation}  \label{eqDefLc}
\RR_{\psi,X^0_\c,E}^{(n)} (u) (y) = \sum\limits_{\substack{X^n\in\X^n_\c\\ X^n\subseteq E}} u\bigl((f^n|_{X^n})^{-1}(y)\bigr) \exp \bigl(S_n\psi\bigl((f^n|_{X^n})^{-1}(y)\bigr)\bigr),
\end{equation}
for each complex-valued continuous function $u\in\CCC(E,\C)$ defined on $E$, and each point $y \in X^0_\c$. 
\end{definition}

Note that $\RR_{\psi,X^0_\c,E}^{(0)} (u) = \begin{cases} u & \text{if } X^0_\c \subseteq E \\ 0  & \text{otherwise}  \end{cases}$, for $\c\in\{\b,\w\}$, whenever the expression on the left-hand side of the equation makes sense.

Note that as seen in (\ref{eqDefLcInte}), for each $\c\in\{\b,\w\}$, we have
\begin{equation}    \label{eqLbwInteMatchLbw}
\bigl(\RR_{\psi,X^0_\c,E}^{(n)} (u)\bigr)\big|_{\inte (X^0_\c)} = \bigl( \RR_{\psi}^n (v)\bigr) \big|_{\inte (X^0_\c)},
\end{equation}
where $v\in B(S^2,\C)$ is defined in (\ref{eqExtendToS2}).

\begin{lemma}  \label{lmLDiscontProperties}
Let $f$, $\CC$, $d$, $\psi$, $\alpha$ satisfy the Assumptions. Fix numbers $n, m\in\N_0$ and a union $E\subseteq S^2$ of $n$-tiles in $\X^n(f,\CC)$, i.e.,
\begin{equation*}
E= \bigcup \{  X^n  \in \X^n(f,\CC)  \,|\,   X^n\subseteq E\}.
\end{equation*}  
Then for each $X^0\in \X^0(f,\CC)$ and each complex-valued continuous function $u\in \CCC(E,\C)$, we have
\begin{equation}    \label{eqLDiscontProperties_Continuity}
\RR_{\psi, X^0, E}^{(n)} (u) \in \CCC(X^0,\C),
\end{equation}   
and
\begin{equation}  \label{eqLDiscontProperties_Split}
\RR_{\psi,X^0,E}^{(n+m)} (u)  = \sum\limits_{Y^0\in\X^0(f,\CC)}   \RR_{\psi,X^0,Y^0}^{(m)} \Bigl( \RR_{\psi,Y^0,E}^{(n)} (u)  \Bigr). 
\end{equation}

If, in addition, we assume that $u\in\Holder{\alpha}((E,d),\C)$ is H\"{o}lder continuous, then
\begin{equation}    \label{eqLDiscontProperties_Holder}
\RR_{\psi, X^0, E}^{(n)} (u) \in \Holder{\alpha}((X^0,d),\C).
\end{equation}
\end{lemma}

\begin{rem}    \label{rmSplitRuelleCoincide}
Note that in the above context $\RR_\psi^n(v) \in B(S^2,\C)$ may not be continuous on $S^2$ if $E\neq S^2$, where $v$ is defined in (\ref{eqExtendToS2}) extending $u$ to $S^2$. If $E=S^2$, then it follows immediately from (\ref{eqLbwInteMatchLbw}) and (\ref{eqLDiscontProperties_Continuity}) that for each $X^0\in\X^0$,
\begin{equation*}
  \RR_{\psi,X^0,E}^{(n)} (u)   = \bigl( \RR_{\psi}^n (u) \bigr) \big|_{X^0}.
\end{equation*}
Hence by (\ref{eqLDiscontProperties_Holder}) and the linear local connectivity of $(S^2,d)$, it can be shown that
\begin{equation*}
\RR_\psi^n  ( \Holder{\alpha}((S^2,d),\C)  ) \subseteq \Holder{\alpha}((S^2,d),\C).
\end{equation*} 
We will not use the last fact in this paper (except in Remark~\ref{rmHolderOpNormCompare}).
\end{rem}

\begin{proof}
Without loss of generality, we assume in this proof that $X^0=X^0_\b$ is the black $0$-tile.

The case of Lemma~\ref{lmLDiscontProperties} when either $m=0$ or $n=0$ follows immediately from Definition~\ref{defSplitRuelle}. Thus without loss of generality, we can assume $m,n\in\N$.

The continuity of $\RR_{\psi, X^0_\b, E}^{(n)} (u)$ follows trivially from (\ref{eqDefLc}) and Proposition~\ref{propCellDecomp}~(i).

To establish (\ref{eqLDiscontProperties_Split}), we first fix an arbitrary point $x\in\inte (X^0_\b)$, and define a function $v\in B(S^2,\C)$ by
\begin{equation*}
v(z)  \coloneqq  \begin{cases} u(z) & \text{if } z\in E, \\ 0  & \text{otherwise}. \end{cases}
\end{equation*}
Similarly, we define functions $w_{Y^0}\in B(S^2,d)$, for $Y^0\in \X^0$, by
\begin{equation*}
w_{Y^0}(y)  \coloneqq  \begin{cases} \RR_{\psi,Y^0,E}^{(n)}(u)(y) & \text{if } y\in Y^0, \\ 0  & \text{otherwise}. \end{cases}
\end{equation*}
Then by (\ref{eqLbwInteMatchLbw}) and the fact that $f^{-m}(x) \cap \CC \neq \emptyset$, we get
\begin{align*}
     \sum\limits_{Y^0\in\X^0}   \RR_{\psi,X^0,Y^0}^{(m)} \Bigl( \RR_{\psi,Y^0,E}^{(n)} (u)  \Bigr) (x)  
= &  \sum\limits_{Y^0\in\X^0}    \RR_\psi^m (w_{Y^0}) (x)   
=    \sum\limits_{Y^0\in\X^0}  \sum\limits_{\substack{ y\in f^{-m}(x)\\y\in Y^0} } e^{S_m\psi(y)}    w_{Y^0}  (y)\\
= &  \sum\limits_{Y^0\in\X^0}  \sum\limits_{\substack{ y\in f^{-m}(x)\\y\in Y^0} } e^{S_m\psi(y)}   \bigl( \RR_\psi^n(v) \bigr) (y) \\
= &  \sum\limits_{Y^0\in\X^0}  \sum\limits_{\substack{ y\in f^{-m}(x)\\y\in Y^0} } e^{S_m\psi(y)}  \sum\limits_{\substack{ z\in f^{-n}(y)\\z\in E}}  e^{S_n\psi(z)}   u(z)\\
= &  \sum\limits_{  y\in f^{-m}(x)  }  \sum\limits_{\substack{ z\in f^{-n}(y)\\z\in E}} e^{S_m\psi(y) + S_n\psi(z)}   u(z)  \\
= &  \sum\limits_{\substack{ z\in f^{-(n+m)}(x)\\z\in E}} e^{S_{n+m}\psi(z)}   u(z) 
=    \RR_\psi^{n+m}(v)(x)\\
= &  \RR_{\psi,X^0_\b,E}^{(n+m)} (u)(x).
\end{align*}
The identity (\ref{eqLDiscontProperties_Split}) is now established by the continuity of two sides of the equation above.

\smallskip

Finally, to prove (\ref{eqLDiscontProperties_Holder}), we first fix two distinct points $x,x'\in X^0_\b$. Denote $y_{X^n} \coloneqq (f^n|_{X^n})^{-1}(x)$ and $y'_{X^n} \coloneqq (f^n|_{X^n})^{-1}(x')$ for each $X^n\in\X^n_\b$.

By Lemma~\ref{lmMetricDistortion}, Lemma~\ref{lmSnPhiBound}, and Lemma~\ref{lmExpToLiniear}, we have
\begin{align*}
     & \frac{1}{d(x,x')^\alpha} \AbsBig{   \RR_{\psi,X^0_\b,E}^{(n)}(u) (x) -  \RR_{\psi,X^0_\b,E}^{(n)}(u) (x')  } \\
\leq & \frac{1}{d(x,x')^\alpha} \sum\limits_{\substack{X^n\in\X^n_\b \\ X^n\subseteq E} }  \Absbig{  e^{S_n\psi(y_{X^n})} u(y_{X^n}) - e^{S_n\psi(y'_{X^n})} u(y'_{X^n}) } \\
\leq & \frac{1}{d(x,x')^\alpha} \sum\limits_{\substack{X^n\in\X^n_\b \\ X^n\subseteq E} }   \Bigl(   \Absbig{e^{S_n\psi(y_{X^n})}}  \abs{ u(y_{X^n}) - u(y'_{X^n})}
        +   \Absbig{e^{S_n\psi(y_{X^n})}  -  e^{S_n\psi(y'_{X^n})}}    \abs{  u(y'_{X^n})} \Bigr) \\
\leq & \frac{1}{d(x,x')^\alpha} \sum\limits_{\substack{X^n\in\X^n_\b \\ X^n\subseteq E} }   e^{S_n \Re(\psi)(y_{X^n})} \Hseminorm{\alpha,\, (E,d)}{u} d(y_{X^n},y'_{X^n})^\alpha \\
     & + \frac{1}{d(x,x')^\alpha} \sum\limits_{\substack{X^n\in\X^n_\b \\ X^n\subseteq E} } \Absbig{1- e^{ S_n\psi(y_{X^n}) - S_n\psi(y'_{X^n}) } }  e^{S_n \Re(\psi)(y'_{X^n})}  \abs{  u(y'_{X^n})}  \\
\leq & \Hseminorm{\alpha,\, (E,d)}{u} \frac{ C_0^\alpha  }{ \Lambda^{\alpha n}  } \sum\limits_{X^n\in\X^n} e^{S_n\Re(\psi)(y_{X^n})} 
      + C_{10} \Hseminorm{\alpha,\, (S^2,d)}{\psi}   \sum\limits_{\substack{X^n\in\X^n_\b \\ X^n\subseteq E} }  e^{S_n\Re(\psi)(y'_{X^n})} \abs{u(y'_{X^n})} \\
\leq &  \frac{ C_0  }{ \Lambda^{\alpha n}  }  \Hseminorm{\alpha,\, (E,d)}{u}  \Normbig{\RR_{\Re(\psi)}^n (\mathbbm{1}_{S^2}) }_{\CCC^0(S^2)}  
      + C_{10} \Hseminorm{\alpha,\, (S^2,d)}{\psi}  \NormBig{\RR_{\Re(\psi),X^0_\b,E}^{(n)} (\abs{u}) }_{\CCC^0(S^2)}  ,      
\end{align*}
where $C_0 > 1$ is a constant depending only on $f$, $\CC$, and $d$ from Lemma~\ref{lmMetricDistortion}, and $C_{10}>1$ is a constant depending only on $f$, $\CC$, $d$, $\alpha$, and $\psi$ from Lemma~\ref{lmExpToLiniear}. Therefore (\ref{eqLDiscontProperties_Holder}) holds.
\end{proof}

\begin{definition}[Split Ruelle operators]   \label{defSplitRuelleOnProductSpace}
Let $f\: S^2 \rightarrow S^2$ be an expanding Thurston map with a Jordan curve $\CC\subseteq S^2$ satisfying $f(\CC)\subseteq\CC$ and $\post f\subseteq \CC$. Let $d$ be a visual metric for $f$ on $S^2$, and $\psi\in\Holder{\alpha}((S^2,d),\C)$ a complex-valued H\"{o}lder continuous function with an exponent $\alpha\in(0,1]$. Let $X^0_\b,X^0_\w\in\X^0(f,\CC)$ be the black $0$-tile and the while $0$-tile, respectively. The \defn{split Ruelle operator} 
\begin{equation*}
\RRR_\psi \: \CCC\bigl(X^0_\b,\C\bigr) \times \CCC\bigl(X^0_\w, \C\bigr) \rightarrow \CCC \bigl(X^0_\b, \C\bigr) \times \CCC \bigl(X^0_\w,\C\bigr)
\end{equation*}
on the product space $\CCC \bigl(X^0_\b, \C\bigr) \times \CCC \bigl(X^0_\w, \C\bigr)$ is given by
\begin{equation*}    \label{eqDefSplitRuelleOnProductSpace}
\RRR_\psi(u_\b,u_\w) = \Bigl(   \RR_{\psi,X^0_\b,X^0_\b}^{(1)} (u_\b) + \RR_{\psi,X^0_\b,X^0_\w}^{(1)} (u_\w)   ,
                              \RR_{\psi,X^0_\w,X^0_\b}^{(1)} (u_\b) + \RR_{\psi,X^0_\w,X^0_\w}^{(1)} (u_\w)           \Bigr)
\end{equation*}
for $u_\b \in \CCC\bigl(X^0_\b, \C\bigr)$ and $u_\w \in \CCC \bigl(X^0_\w, \C\bigr)$.
\end{definition}

Note that by (\ref{eqLDiscontProperties_Continuity}) in Lemma~\ref{lmLDiscontProperties}, the operator $\RRR_\psi$ is well-defined. Moreover, by (\ref{eqLDiscontProperties_Holder}) in Lemma~\ref{lmLDiscontProperties}, we have
\begin{equation}    \label{eqSplitRuelleRestrictHolder}
 \RRR_\psi \bigl(  \Holder{\alpha}\bigl(\bigl(X^0_\b,d\bigr),\C\bigr) \times \Holder{\alpha}\bigl(\bigl(X^0_\w,d\bigr),\C\bigr) \bigr) 
  \subseteq        \Holder{\alpha}\bigl(\bigl(X^0_\b,d\bigr),\C\bigr) \times \Holder{\alpha}\bigl(\bigl(X^0_\w,d\bigr),\C\bigr).
\end{equation}

Note that it follows immediately from Definition~\ref{defSplitRuelle} that $\RRR_\psi^0$ is the identity map, and $\RRR_\psi$ is a linear operator on the Banach $\Holder{\alpha}\bigl(\bigl(X^0_\b,d\bigr),\C\bigr) \times \Holder{\alpha}\bigl(\bigl(X^0_\w,d\bigr),\C\bigr)$ equipped with a norm given by $\norm{(u_\b,u_\w)} \coloneqq \max\Bigl\{  \NHnorm{\alpha}{b}{ u_\b }{(X^0_\b,d)}, \, \NHnorm{\alpha}{b}{ u_\w }{(X^0_\w,d)} \Bigr\}$, for each $b\in\R \setminus \{0\}$. See (\ref{eqDefNormalizedHolderNorm}) for the definition of the normalized H\"{o}lder norm $\NHnorm{\alpha}{b}{u}{(E,d)}$.

For each $\c\in\{\b,\w\}$, we define the projection $\pi_\c \: \CCC(X^0_\b,\C) \times \CCC(X^0_\w,\C) \rightarrow \CCC(X^0_\c,\C)$ by
\begin{equation}   \label{eqDefProjections}
\pi_\c(u_\b,u_\w) = u_\c, \qquad \mbox{ for } (u_\b,u_\w)\in \CCC(X^0_\b,\C) \times \CCC(X^0_\w,\C).
\end{equation}

\begin{definition}   \label{defOpHNormDiscont}
Let $f\: S^2 \rightarrow S^2$ be an expanding Thurston map with a Jordan curve $\CC\subseteq S^2$ satisfying $f(\CC)\subseteq\CC$ and $\post f\subseteq \CC$. Let $d$ be a visual metric for $f$ on $S^2$, and $\psi\in\Holder{\alpha}((S^2,d),\C)$ a complex-valued H\"{o}lder continuous function with an exponent $\alpha\in(0,1]$. For all $n\in\N_0$ and $b\in \R\setminus\{0\}$, we write the operator norm
\begin{align}  \label{eqDefNormalizedOpHNormSplitRuelle}
 \NOpHnormD{\alpha}{b}{\RRR_\psi^n}   
\coloneqq  & \sup \bigg\{     \NHnormbig{\alpha}{b}{  \pi_\c \bigl( \RRR_\psi^n (u_\b, u_\w)\bigr) }{(X^0_\c,d)}  \notag  \\
                   & \qquad \quad \quad \,\bigg|\, \begin{array}{ll}  \c\in\{\b,\w\}, \, u_\b\in\Holder{\alpha} ((X^0_\b,d),\C), \, u_\w\in\Holder{\alpha}((X^0_\w,d),\C) \\ 
                     \text{with } \NHnorm{\alpha}{b}{ u_\b}{(X^0_\b,d)} \leq 1 \text{ and } \NHnorm{\alpha}{b}{ u_\w}{(X^0_\w,d)} \leq 1\end{array} \bigg\}.   
\end{align}
For typographical reasons, we write
\begin{equation}  \label{eqDefOpHNormSplitRuelle}
\OpHnormD{\alpha}{\RRR_\psi^n}  \coloneqq  \NOpHnormD{\alpha}{1}{\RRR_\psi^n}   =  \NOpHnormD{\alpha}{-1}{\RRR_\psi^n} .
\end{equation}
\end{definition}

\begin{lemma}   \label{lmSplitRuelleCoordinateFormula}
Let $f$, $\CC$, $d$, $\alpha$, $\psi$ satisfy the Assumptions. We assume in addition that $f(\CC) \subseteq \CC$. Let $X^0_\b,X^0_\w\in\X^0(f,\CC)$ be the black $0$-tile and the while $0$-tile, respectively. 

Then for all $n\in\N_0$, $u_\b \in \CCC \bigl(X^0_\b, \C\bigr)$, and $u_\w \in \CCC \bigl(X^0_\w, \C\bigr)$, 
\begin{equation}    \label{eqSplitRuelleCoordinateFormula}
\RRR_\psi^n(u_\b,u_\w) = \Bigl(   \RR_{\psi,X^0_\b,X^0_\b}^{(n)} (u_\b) + \RR_{\psi,X^0_\b,X^0_\w}^{(n)} (u_\w)   ,
                                \RR_{\psi,X^0_\w,X^0_\b}^{(n)} (u_\b) + \RR_{\psi,X^0_\w,X^0_\w}^{(n)} (u_\w)           \Bigr).
\end{equation}
Consequently,
\begin{align}   \label{eqSplitRuelleOpNormQuot}
    \NOpHnormD{\alpha}{b}{\RRR_\psi^n}  
 =  \sup \Bigg\{ & \frac{   \NHnormbig{\alpha}{b}{   \RR_{\psi, X^0, X^0_\b}^{(n)} (u_\b)  +  \RR_{\psi, X^0, X^0_\w}^{(n)} (u_\w)    }{(X^0,d)}      }{  \max\bigl\{     \NHnorm{\alpha}{b}{u_\b}{(X^0_\b,d)}, \, \NHnorm{\alpha}{b}{u_\w}{(X^0_\w,d)}  \bigr\}   } \\
                &\qquad \,\Bigg|\, \begin{array}{ll}  u_\b\in\Holder{\alpha}((X^0_\b,d),\C), \, u_\w\in\Holder{\alpha}((X^0_\w,d),\C) \\ \text{with } \norm{u_\b}_{\CCC^0(X^0_\b)} \norm{u_\w}_{\CCC^0(X^0_\w)} \neq 0, \, X^0\in \X^0(f,\CC)  \end{array} \Bigg\}.  \notag
\end{align}
\end{lemma}

\begin{proof}
We prove (\ref{eqSplitRuelleCoordinateFormula}) by induction. The case when $n=0$ and the case when $n=1$ both hold by definition. Assume now (\ref{eqSplitRuelleCoordinateFormula}) holds when $n=m$ for some $m\in\N$. Then by (\ref{eqLDiscontProperties_Split}) in Lemma~\ref{lmLDiscontProperties}, for each $\c\in\{\b,\w\}$, we have
\begin{align*}
&                    \pi_\c \bigl( \RRR_\psi^{m+1}(u_\b,u_\w)  \bigr) \\
&\qquad  =   \pi_\c \Bigl(  \RRR_\psi \Bigl(   \RR_{\psi,X^0_\b,X^0_\b}^{(m)} (u_\b) + \RR_{\psi,X^0_\b,X^0_\w}^{(m)} (u_\w)   ,
                                                                             \RR_{\psi,X^0_\w,X^0_\b}^{(m)} (u_\b) + \RR_{\psi,X^0_\w,X^0_\w}^{(m)} (u_\w)           \Bigr)    \Bigr)\\
&\qquad  =       \sum\limits_{X^0\in\X^0}  \RR_{\psi,X^0_\c,X^0}^{(1)} \Bigl(   \RR_{\psi,X^0,X^0_\b}^{(m)} (u_\b) + \RR_{\psi,X^0,X^0_\w}^{(m)} (u_\w) \Bigr) \\
&\qquad  =      \sum\limits_{X^0\in\X^0}  \RR_{\psi,X^0_\c,X^0}^{(1)}  \Bigl(\RR_{\psi,X^0,X^0_\b}^{(m)} (u_\b) \Bigr)
                      +  \sum\limits_{X^0\in\X^0}  \RR_{\psi,X^0_\c,X^0}^{(1)}  \Bigl(\RR_{\psi,X^0,X^0_\w}^{(m)} (u_\w) \Bigr)  \\  
&\qquad  =        \RR_{\psi,X^0_\c,X^0_\b}^{(m+1)} (u_\b) + \RR_{\psi,X^0_\c,X^0_\w}^{(m+1)} (u_\w) ,                               
\end{align*}
for $u_\b \in \CCC \bigl(X^0_\b, \C\bigr)$ and $u_\w \in \CCC \bigl(X^0_\w, \C\bigr)$. This completes the inductive step, establishing (\ref{eqSplitRuelleCoordinateFormula}).

The identity (\ref{eqSplitRuelleOpNormQuot}) follows immediately from Definition~\ref{defOpHNormDiscont} and the identity (\ref{eqSplitRuelleCoordinateFormula}).
\end{proof}

\begin{rem}  \label{rmHolderOpNormCompare}
One can show that $\NOpHnormD{\alpha}{b}{\RRR_\psi^n} \geq \NHnormbig{\alpha}{b}{\RR_\psi^n}{(S^2,d)}$, $n\in\N_0$, where $\NHnormbig{\alpha}{b}{\RR_\psi^n}{(S^2,d)}$ is the normalized operator norm of $\RR_\psi^n \: \Holder{\alpha}((S^2,d),\C) \rightarrow \Holder{\alpha}((S^2,d),\C)$ defined in (\ref{eqDefNormalizedOpHNorm}). We will not use this fact in this paper.
\end{rem}

\subsection{Basic inequalities}   \label{subsctSplitRuelleOp_BasicIneq}

Let $f\: S^2 \rightarrow S^2$ be an expanding Thurston map, and $d$ be a visual metric on $S^2$ for $f$ with expansion factor $\Lambda>1$. Let $\psi\in \Holder{\alpha}((S^2,d),\C)$ be a complex-valued H\"{o}lder continuous function with an exponent $\alpha\in(0,1]$. We define
\begin{equation}   \label{eqDefWideTildePsiComplex}
\wt{\psi} \coloneqq \wt{\Re(\psi)} +  \I \Im(\psi) = \psi - P(f,\Re(\psi)) + \log u_{\Re(\psi)} - \log \bigl(u_{\Re(\psi)} \circ f \bigr),
\end{equation}
where $u_{\Re(\psi)}$ is the continuous function given by Theorem~\ref{thmMuExist} with $\phi \coloneqq \Re(\psi)$. Then for each $u\in \CCC(S^2,\C)$ and each $x\in S^2$, we have
\begin{align}   \label{eqRuelleTilde_NoTilde}
     \RR_{\wt{\psi}} (u)(x)
& =  \sum\limits_{y\in f^{-1}(x)} \deg_f(y)u(y) e^{ \psi(y) - P(f,\Re(\psi)) + \log u_{\Re(\psi)}(y) - \log  ( u_{\Re(\psi)} (f(y)) ) }   \notag \\
& =  \frac{\exp(-P(f,\Re(\psi))}{u_{\Re(\psi)} (x) }  \sum\limits_{y\in f^{-1}(x)} \deg_f(y)u(y) u_{\Re(\psi)}(y) \exp ( \psi(y))\\
& =  \frac{\exp(-P(f,\Re(\psi))}{u_{\Re(\psi)} (x) }  \RR_{ \psi } \bigl(u_{\Re(\psi)} u \bigr)(x). \notag
\end{align}
If, in addition, given a Jordan curve $\CC\subseteq S^2$ with $\post f\subseteq\CC$, then for each $n\in\N_0$, each union $E$ of $n$-tiles in $\X^n(f,\CC)$, each continuous function $v\in\CCC(E,\C)$, each $0$-tile $X^0\in\X^0(f,\CC)$, and each $z\in X^0$,
\begin{align}   \label{eqSplitRuelleTilde_NoTilde}
   \RR_{\wt{\psi},X^0,E}^{(n)} (v)(z)  
=&   \sum\limits_{\substack{X^n\in\X^n(f,\CC)\\ X^n\subseteq E}} \Bigl( v  e^{ S_n  ( \psi - P(f,\Re(\psi)) + \log u_{\Re(\psi)}  - \log  ( u_{\Re(\psi)}\circ f )  )} \Bigr) \bigl(  (f^n|_{X^n} )^{-1}(z) \bigr)  \notag \\
=&   \frac{\exp(-n P(f,\Re(\psi))}{u_{\Re(\psi)} (z) }  \sum\limits_{\substack{X^n\in\X^n(f,\CC)\\ X^n\subseteq E}} \bigl(v  u_{\Re(\psi)}  \exp (S_n \psi )\bigr)   \bigl(  (f^n|_{X^n} )^{-1}(z) \bigr) \\
=&   \frac{\exp(-n P(f,\Re(\psi))}{u_{\Re(\psi)} (z) }  \RR_{ \psi,X^0,E }^{(n)} \bigl(u_{\Re(\psi)} v \bigr)(z). \notag
\end{align}

\begin{definition}  \label{defCone}
Let $f\: S^2 \rightarrow S^2$ be an expanding Thurston map, and $d$ be a visual metric on $S^2$ for $f$ with expansion factor $\Lambda>1$. Fix a constant $\alpha\in(0,1]$.

For each subset $E\subseteq S^2$ and each constant $B\in \R$ with $B>0$, we define the \defn{$B$-cone inside $\Holder{\alpha}(E,d)$} as 
\begin{equation}  \label{eqDefCone}
  K_B(E,d)
 =   \bigl\{ u\in \Holder{\alpha} (E,d)  \,\big|\, u(x)> 0, \,\abs{u(x)-u(y)} \leq B (u(x)+u(y)) d(x,y)^\alpha \text{ for } x,y\in E\bigr\}.  
\end{equation}
\end{definition}

It is important to define the $B$-cones inside $\Holder{\alpha}(E,d)$ in the form above in order to establish the following lemma, which will be used in the proof of Proposition~\ref{propDolgopyatOperator}.

\begin{lemma}   \label{lmConePower}
Let $(X,d)$ be a metric space and $\alpha\in (0,1]$. Then for each $B>0$ and each $u\in K_B(X,d)$, we have $u^2 \in K_{2B}(X,d)$.
\end{lemma}

\begin{proof}
Fix arbitrary $B>0$ and $u\in K_B(X,d)$. For any $x,y\in X$,
\begin{align*}
      \Absbig{u^2(x)-u^2(y)} = & \abs{u(x)+u(y)}\abs{u(x)-u(y)} \leq B\abs{u(x)+u(y)}^2 d(x,y)^\alpha \\
                       \leq &  2B \bigl(u^2(x)+u^2(y)\bigr) d(x,y)^\alpha.
\end{align*}
Therefore $u^2\in K_{2B}(X,d)$.
\end{proof}

\begin{lemma}   \label{lmRtildeNorm=1}
Let $f$, $d$, $\alpha$, $\psi$ satisfy the Assumptions. Let $\phi\in \Holder{\alpha}(S^2,d)$ be a real-valued H\"{o}lder continuous function with an exponent $\alpha$. Then the operator norm of $\RR_{\wt\phi}$ acting on $\CCC(S^2)$ is given by $\Normbig{\RR_{\wt\phi}}_{C^0(S^2)}=1$. In addition, $\RR_{\wt\phi}(\mathbbm{1}_{S^2})=\mathbbm{1}_{S^2}$.

Moreover, given a Jordan curve $\CC\subseteq S^2$ satisfying $\post f\subseteq \CC$. Assume in addition that $f(\CC)\subseteq\CC$. Then for all $n\in\N_0$, $\c\in\{\b,\w\}$, $u_\b\in\CCC(X^0_\b,\C)$, and $u_\w\in\CCC(X^0_\w,\C)$, we have
\begin{equation}   \label{eqRDiscontTildeSupDecreasing}
\Norm{  \RR_{\wt{\psi}, X^0_\c, X^0_\b}^{(n)}  (u_\b)  }_{C^0(X^0_\c)} \leq \norm{u_\b}_{C^0(X^0_\b)}, \qquad  \Norm{  \RR_{\wt{\psi}, X^0_\c, X^0_\w} ^{(n)} (u_\w)  }_{C^0(X^0_\c)} \leq \norm{u_\w}_{C^0(X^0_\w)},
\end{equation}
and
\begin{equation}   \label{eqSplitRuelleTildeSupDecreasing}
      \Norm{  \RR_{\wt{\psi}, X^0_\c, X^0_\b}^{(n)}  (u_\b) + \RR_{\wt{\psi}, X^0_\c, X^0_\w} ^{(n)} (u_\w)  }_{C^0(X^0_\c)}  
 \leq \max\bigl\{\norm{u_\b}_{C^0(X^0_\b)}, \, \norm{u_\w}_{C^0(X^0_\w)}\bigr\}.
\end{equation}
\end{lemma}

\begin{proof}
The fact that $\Normbig{\RR_{\wt\phi}}_{C^0(S^2)}=1$ and $\RR_{\wt\phi}(\mathbbm{1}_{S^2})=\mathbbm{1}_{S^2}$ is established in \cite[Lemma~5.25]{Li17}.

To prove (\ref{eqSplitRuelleTildeSupDecreasing}), we first fix arbitrary $n\in\N_0$, $\c\in\{\b,\w\}$, $u_\b\in\CCC(X^0_\b)$, and $u_\w\in\CCC(X^0_\w)$. Denote $M \coloneqq \max\bigl\{\norm{u_\b}_{C^0(X^0_\b)},\,\norm{u_\w}_{C^0(X^0_\w)}\bigr\}$. Then by Definition~\ref{defSplitRuelle}, (\ref{eqDefWideTildePsiComplex}), and the fact that $\RR_{\wt{\Re(\psi)}}(\mathbbm{1}_{S^2})=\mathbbm{1}_{S^2}$, for each $y\in\inte(X^0_\c)$,
\begin{align*}
              \Norm{  \RR_{\wt{\psi}, X^0_c, X^0_\b}^{(n)}  (u_\b) + \RR_{\wt{\psi}, X^0_\c, X^0_\w} ^{(n)} (u_\w)  }_{C^0(X^0_\c)}  
  \leq & M  \sum\limits_{X^n\in\X^n_\c} \Absbig{  \exp \bigl(S_n \wt{\psi} \bigl((f^n|_{X^n})^{-1}(y)\bigr)\bigr)     }  \\
  =      & M \RR_{\wt{\Re(\psi)}}^n (\mathbbm{1}_{S^2})  (y)
  =          M.
\end{align*}
This establishes (\ref{eqSplitRuelleTildeSupDecreasing}). Finally, (\ref{eqRDiscontTildeSupDecreasing}) follows immediately from (\ref{eqSplitRuelleTildeSupDecreasing}) and Definition~\ref{defSplitRuelle} by setting one of the functions $u_\b$ and $u_\w$ to be $0$.
\end{proof}

\begin{lemma}   \label{lmBound_aPhi}
Let $f$, $\CC$, $d$, $L$, $\alpha$, $\Lambda$ satisfy the Assumptions. Then there exist constants $C_{13}>1$ and $C_{14}>0$ depending only on $f$, $\CC$, $d$, and $\alpha$ such that the following is satisfied:

\smallskip 

For all $K,M,T,a\in\R$ with $K>0$, $M>0$, $T>0$, and $\abs{a}\leq T$, and all real-valued H\"{o}lder continuous function $\phi\in\Holder{\alpha}(S^2,d)$ with $\Hseminorm{\alpha,\, (S^2,d)}{\phi}  \leq K$ and $\Norm{\phi}_{\CCC^0(S^2)} \leq M$, we have
\begin{equation}\label{eqBound_aPhi_Sup}
\Normbig{\wt{a\phi}}_{\CCC^0(S^2)}  \leq C_{13} (K+M)T + \abs{\log(\deg f)},
\end{equation}
\begin{equation}    \label{eqBound_aPhi_Hseminorm}
\Hseminormbig{\alpha,\, (S^2,d)}{\wt{a\phi}}  \leq C_{13} KT   e^{C_{14}KT},
\end{equation}
\begin{equation}   \label{eqBound_Uaphi_Hnorm}
\Hnorm{\alpha}{u_{a\phi}}{(S^2,d)} \leq \biggl( 4\frac{TK C_0}{1-\Lambda^{-\alpha}}  L +1 \biggr) e^{2C_{15}},
\end{equation}
\begin{equation}  \label{eqBound_Uaphi_UpperLower}
\exp(-C_{15}) \leq u_{a\phi}(x) \leq \exp(C_{15})
\end{equation}
for $x\in S^2$, where the constant $C_0>1$ depending only on $f$, $d$, and $\CC$ is from Lemma~\ref{lmMetricDistortion}, and the constant
\begin{equation}     \label{eqConst_lmBound_aPhi1}
C_{15} = C_{15}(f,\CC,d,\alpha,T,K) \coloneqq 4 \frac{TKC_0}{1-\Lambda^{-\alpha}} L \bigl(\diam_d(S^2)\bigr)^\alpha  >  0
\end{equation}
depends only on $f$, $\CC$, $d$, $\alpha$, $T$, and $K$.
\end{lemma}

\begin{proof}
Fix $K$, $M$, $T$, $a$, $\phi$ satisfying the conditions in this lemma.

Recall
\begin{equation}  \label{eqPflmBound_TildeDef}
\wt{a\phi} = a\phi - P(f,a\phi) + \log u_{a\phi} - \log(u_{a\phi} \circ f),
\end{equation}
where the function $u_{a\phi}$ is defined as $u_{\phi}$ in Theorem~\ref{thmMuExist}.

By (\ref{eqU_phiBounds}) in Theorem~\ref{thmMuExist} and (\ref{eqC2Bound}) in Lemma~\ref{lmSigmaExpSnPhiBound}, we immediately get (\ref{eqBound_Uaphi_UpperLower}).

By Corollary~\ref{corRR^nConvToTopPressureUniform}, (\ref{eqDeg=SumLocalDegree}), and (\ref{eqLocalDegreeProduct}), for each $x\in S^2$,
\begin{align*}
P(f,a\phi) & = \lim\limits_{n\to+\infty} \frac{1}{n} \log \sum\limits_{y\in f^{-n}(x)} \deg_{f^n}(y)\exp(aS_n\phi(y))  \\
          & \leq  \lim\limits_{n\to+\infty} \frac{1}{n} \log \sum\limits_{y\in f^{-n}(x)} \deg_{f^n}(y)\exp(nTM)  \\
          & = TM + \lim\limits_{n\to+\infty}\frac{1}{n}  \log \sum\limits_{y\in f^{-n}(x)} \deg_{f^n}(y) 
            = TM + \log(\deg f).
\end{align*}
Similarly, $P(f,a\phi)\geq -TM + \log(\deg f)$. So $\abs{P(f,a\phi)}  \leq TM + \abs{\log(\deg f)}$.

Thus by combining the above with (\ref{eqBound_Uaphi_UpperLower}) and (\ref{eqConst_lmBound_aPhi1}), we get 
\begin{equation}  \label{eqPflmBound_aPhi3}
\Normbig{\wt{a\phi}}_{\CCC^0(S^2)} \leq TM +  TM + \abs{\log(\deg f)} + 2C_{15} = C_{16} T(K+M) + \abs{\log(\deg f)},
\end{equation}
where $C_{16} \coloneqq 2 + 8 \frac{C_0}{1-\Lambda^{-\alpha}} L \bigl(\diam_d(S^2)\bigr)^\alpha$ is a constant depending only on $f$, $\CC$, $d$, and $\alpha$.

By (\ref{eqPflmBound_TildeDef}), Theorem~\ref{thmETMBasicProperties}~(i), and (\ref{eqBound_Uaphi_UpperLower}), and the fact that $\abs{\log t_1 - \log t_2} \leq \frac{\abs{t_1-t_2}}{\min\{t_1, \, t_2\}}$ for all $t_1,t_2>0$, we get 
\begin{align}  \label{eqPflmBound_aPhi2}
\Hseminormbig{\alpha,\, (S^2,d)}{\wt{a\phi}}  
\leq & \Hseminorm{\alpha,\, (S^2,d)}{a\phi}  + \Hseminorm{\alpha,\, (S^2,d)}{\log u_{a\phi}}  +\Hseminorm{\alpha,\, (S^2,d)}{\log(u_{a\phi}\circ f)} \notag \\
\leq & TK + e^{C_{15}} (1+ \LIP_d(f)  ) \Hseminorm{\alpha,\, (S^2,d)}{ u_{a\phi}}   .                                            
\end{align}
Here $\LIP_d(f)$ denotes the Lipschitz constant of $f$ with respect to the visual metric $d$ (see (\ref{eqDefLipConst})). 

By Theorem~\ref{thmMuExist}, (\ref{eqR1Diff}) in Lemma~\ref{lmR1properties}, (\ref{eqC2Bound}) in Lemma~\ref{lmSigmaExpSnPhiBound}, (\ref{eqConst_lmBound_aPhi1}), and the fact that $\abs{1-e^{-t}} \leq t$ for $t>0$, we get 
\begin{align*}
        \abs{u_{a\phi}(x) - u_{a\phi}(y) }  
=    & \AbsBigg{ \lim\limits_{n\to+\infty} \frac{1}{n} \sum\limits_{j=0}^{n-1}  \Bigl( \RR_{\overline{a\phi}}^j \bigl(\mathbbm{1}_{S^2}\bigr)(x) - \RR_{\overline{a\phi}}^j \bigl(\mathbbm{1}_{S^2}\bigr)(y)  \Bigr)}\\
\leq & \limsup\limits_{n\to+\infty} \frac{1}{n} \sum\limits_{j=0}^{n-1} \AbsBig{ \RR_{\overline{a\phi}}^j \bigl(\mathbbm{1}_{S^2}\bigr)(x) - \RR_{\overline{a\phi}}^j \bigl(\mathbbm{1}_{S^2}\bigr)(y)  } \\
\leq & e^{2C_{15}} \Bigl( 1 - \exp\Bigl( - 4 \frac{TKC_0}{1-\Lambda^{-\alpha}} L d(x,y)^\alpha   \Bigr)  \Bigr)  
\leq  e^{2C_{15}}  \frac{4TKC_0}{1-\Lambda^{-\alpha}} L d(x,y)^\alpha,
\end{align*}
for all $x,y\in S^2$. So 
\begin{equation}    \label{eqBound_Uaphi_Hseminorm}
\Hseminorm{\alpha,\, (S^2,d)}{u_{a\phi}} \leq 4\frac{TKC_0}{1-\Lambda^{-\alpha}} L e^{2C_{15}} . 
\end{equation}

Thus by (\ref{eqPflmBound_aPhi2}), (\ref{eqBound_Uaphi_Hseminorm}), and (\ref{eqConst_lmBound_aPhi1}), we get
\begin{equation*}
\Hseminormbig{\alpha,\, (S^2,d)}{\wt{a\phi}} \leq TK  C_{13} e^{C_{14}TK} ,
\end{equation*}
where the constants
\begin{equation}
C_{13}  \coloneqq \max\biggl\{ C_{16}, \, 1+ \bigl(1+ \LIP_d(f) \bigr) \frac{4C_0}{1-\Lambda^{-\alpha}} L\biggr\}
\end{equation}
and 
\begin{equation}
C_{14} \coloneqq 12 \frac{C_0}{1-\Lambda^{-\alpha}} L \bigl(\diam_d(S^2)\bigr)^\alpha
\end{equation}
depend only on $f$, $\CC$, $d$, and $\alpha$. Since $C_{13}\geq C_{16}$, (\ref{eqBound_aPhi_Sup}) follows from (\ref{eqPflmBound_aPhi3}). 

Finally, (\ref{eqBound_Uaphi_Hnorm}) follows from (\ref{eqBound_Uaphi_UpperLower}) and (\ref{eqBound_Uaphi_Hseminorm}).
\end{proof}

\begin{lemma}[Basic inequalities]   \label{lmBasicIneq}
Let $f$, $\CC$, $d$, $\alpha$, $\phi$, $s_0$ satisfy the Assumptions. Then there exists a constant $A_0=A_0\bigl(f,\CC,d,\Hseminorm{\alpha,\, (S^2,d)}{\phi},\alpha\bigr)\geq 2C_0>2$ depending only on $f$, $\CC$, $d$, $\Hseminorm{\alpha,\, (S^2,d)}{\phi}$, and $\alpha$ such that $A_0$ increases as $\Hseminorm{\alpha,\, (S^2,d)}{\phi}$ increases, and that for all $X^0\in \X^0(f,\CC)$, $x,x'\in X^0$, $n\in\N$, union $E\subseteq S^2$ of $n$-tiles in $\X^n(f,\CC)$, $B\in\R$ with $B>0$, and $a,b\in\R$ with $\abs{a}\leq 2 s_0$ and $\abs{b} \in\{0\}\cup [1,+\infty)$, the following statements are satisfied:

\begin{enumerate}
\smallskip
\item[(i)] For each $u\in K_B(E,d)$, we have
\begin{equation}   \label{eqBasicIneqR}
\frac{\Abs{\RR_{\wt{a\phi},X^0,E}^{(n)} (u)(x) - \RR_{\wt{a\phi},X^0,E}^{(n)} (u)(x')}}
     {  \RR_{\wt{a\phi},X^0,E}^{(n)} (u)(x)  + \RR_{\wt{a\phi},X^0,E}^{(n)} (u)(x') }  
\leq A_0\biggl( \frac{B}{\Lambda^{\alpha n}} + \frac{\Hseminormbig{\alpha,\, (E,d)}{\wt{a\phi}}}{1-\Lambda^{-\alpha}}  \biggr) d(x,x')^\alpha.
\end{equation}

\smallskip
\item[(ii)] Denote $s \coloneqq a+\I b$. Let $v\in \Holder{\alpha}((E,d),\C)$ be a complex-valued H\"{o}lder continuous function, then we have
\begin{align}  \label{eqBasicIneqN1}
     & \AbsBig{\RR_{\wt{s\phi}, X^0, E}^{(n)} (v)(x) - \RR_{\wt{s\phi}, X^0, E}^{(n)} (v)(x')} \notag  \\
\leq &  \biggl( C_0  \frac{ \Hseminorm{\alpha,\, (E,d)}{v} }{\Lambda^{\alpha n}}  + A_0 \max\{1, \, \abs{b}\} \RR_{\wt{a\phi}, X^0, E}^{(n)} (\abs{v})(x)    \biggr) d(x,x')^\alpha,
\end{align}
where $C_0>1$ is a constant from Lemma~\ref{lmMetricDistortion} depending only on $f$, $d$, and $\CC$.

If, in addition, there exists a non-negative real-valued H\"{o}lder continuous function $h\in\Holder{\alpha}(E,d)$ such that 
\begin{equation*}
\abs{v(y)-v(y')} \leq B (h(y)+h(y')) d(y,y')^\alpha
\end{equation*}
when $y,y'\in E$, then
\begin{align}  \label{eqBasicIneqC}
     & \AbsBig{\RR_{\wt{s\phi}, X^0, E}^{(n)} (v)(x) - \RR_{\wt{s\phi}, X^0, E}^{(n)} (v)(x')}   \\
\leq &  A_0 \biggl(  \frac{B}{\Lambda^{\alpha n}}  \Bigl( \RR_{\wt{a\phi}, X^0, E}^{(n)} (h)(x) + \RR_{\wt{a\phi}, X^0, E}^{(n)} (h)(x')  \Bigr) + \max\{1, \, \abs{b}\} \RR_{\wt{a\phi}, X^0, E}^{(n)} (\abs{v})(x)    \biggr) d(x,x')^\alpha.\notag
\end{align}
\end{enumerate}
\end{lemma}

\begin{proof}
Without loss of generality, we assume that $X^0=X^0_\b$ is the black $0$-tile. Fix $n$, $E$, $B$, $a$, and $b$ as in the statement of Lemma~\ref{lmBasicIneq}.

(i) Note that by Lemma~\ref{lmBound_aPhi},
\begin{equation}  \label{eqT0Bound} 
\sup\bigl\{ \Hseminormbig{\alpha,\, (S^2,d)}{\wt{\tau\phi}} \,\big|\, \tau \in\R, \abs{ \tau }\leq 2 s_0  \bigr\} \leq T_0,
\end{equation}
where the constant
\begin{equation}  \label{eqDefT0}
T_0 = T_0\bigl(f,\CC,d,\Hseminorm{\alpha,\, (S^2,d)}{\phi},\alpha\bigr) \coloneqq (2 s_0 + 1) C_{13} \Hseminorm{\alpha,\, (S^2,d)}{\phi} \exp\bigl(2 s_0C_{14} \Hseminorm{\alpha,\, (S^2,d)}{\phi}\bigr)>0
\end{equation}
depends only on $f$, $\CC$, $d$, $\Hseminorm{\alpha,\, (S^2,d)}{\phi}$, and $\alpha$. Here $C_{13}>1$ and $C_{14}>0$  are constants from Lemma~\ref{lmBound_aPhi} depending only on $f$, $\CC$, $d$, and $\alpha$.

Fix $u\in K_B(E,d)$ and $x,x'\in X^0_\b$. For each $X^n\in\X^n_\b$, denote $y_{X^n}  \coloneqq (f^n|_{X^n})^{-1}(x)$ and $y'_{X^n} \coloneqq (f^n|_{X^n})^{-1}(x')$.

Then by (\ref{eqDefCone}),
\begin{align*}
     & \frac{\AbsBig{\RR_{\wt{a\phi},X^0_\b,E}^{(n)} (u)(x) - \RR_{\wt{a\phi},X^0_\b,E}^{(n)} (u)(x')}}
            {   \RR_{\wt{a\phi},X^0_\b,E}^{(n)} (u)(x)  +  \RR_{\wt{a\phi},X^0_\b,E}^{(n)} (u)(x')  }  \\
\leq & \frac{1} { \RR_{\wt{a\phi},X^0_\b,E}^{(n)} (u)(x) + \RR_{\wt{a\phi},X^0_\b,E}^{(n)} (u)(x')  }  
       \sum\limits_{\substack{X^n\in\X^n_\b \\ X^n\subseteq E}}  \Abs{u(y_{X^n}) e^{S_n\wt{a\phi}(y_{X^n})} - u(y'_{X^n}) e^{S_n\wt{a\phi}(y'_{X^n})} } \\
\leq & \frac{1} { \RR_{\wt{a\phi},X^0_\b,E}^{(n)} (u)(x) + \RR_{\wt{a\phi},X^0_\b,E}^{(n)} (u)(x')  }  
       \sum\limits_{\substack{X^n\in\X^n_\b \\ X^n\subseteq E}} \Abs{u(y_{X^n}) -  u(y'_{X^n}) }  e^{S_n\wt{a\phi}(y'_{X^n})} \\
     & + \frac{1} { \RR_{\wt{a\phi},X^0_\b,E}^{(n)} (u)(x) + \RR_{\wt{a\phi},X^0_\b,E}^{(n)} (u)(x')  }  
         \sum\limits_{\substack{X^n\in\X^n_\b \\ X^n\subseteq E}} u(y_{X^n})  \Abs{e^{S_n\wt{a\phi}(y_{X^n})}  - e^{S_n\wt{a\phi}(y'_{X^n})}  }   \\
\leq & \frac{ \sum\limits_{\substack{X^n\in\X^n_\b \\ X^n\subseteq E}} 
            B \Bigl( u(y_{X^n}) e^{S_n\wt{a\phi}(y_{X^n})}    e^{\Absbig{S_n\wt{a\phi}(y'_{X^n}) - S_n\wt{a\phi}(y_{X^n})}}
                    +u(y'_{X^n})e^{S_n\wt{a\phi}(y'_{X^n})}  \Bigr)    d(y_{X^n}, y'_{X^n})^\alpha}
            { \RR_{\wt{a\phi},X^0_\b,E}^{(n)} (u)(x) + \RR_{\wt{a\phi},X^0_\b,E}^{(n)} (u)(x')  }   \\
     & + \frac{1} { \RR_{\wt{a\phi},X^0_\b,E}^{(n)} (u)(x) + \RR_{\wt{a\phi},X^0_\b,E}^{(n)} (u)(x')  } 
         \sum\limits_{\substack{X^n\in\X^n_\b \\ X^n\subseteq E}}  
             u(y_{X^n}) \AbsBig{1- e^{S_n\wt{a\phi}(y'_{X^n}) - S_n\wt{a\phi}(y_{X^n})} }   e^{S_n\wt{a\phi}(y_{X^n})}  .
\end{align*}

By Lemma~\ref{lmSnPhiBound}, Lemma~\ref{lmMetricDistortion}, Lemma~\ref{lmExpToLiniear}, (\ref{eqT0Bound}), and (\ref{eqDefT0}), the right-hand side of the last inequality is
\begin{align*}
\leq  & B\exp \Biggl( \frac{\Hseminormbig{\alpha,\, (S^2,d)}{\wt{a\phi}}  C_0 \bigl(\diam_d(S^2)\bigr)^\alpha}{1-\Lambda^{-\alpha}}  \Biggr) \frac{d(x,x')^\alpha C_0^\alpha}{\Lambda^{\alpha n}}      + C_{10} \Hseminormbig{\alpha,\, (S^2,d)}{\wt{a\phi}}  d(x,x')^\alpha \\
\leq  & A_1 \Biggl( \frac{B}{\Lambda^{\alpha n}}  + \frac{ \Hseminormbig{\alpha,\, (S^2,d)}{\wt{a\phi}} }{1-\Lambda^{-\alpha}} \Biggr) d(x,x')^\alpha,
\end{align*}
where 
\begin{equation}   \label{eqPflmBasicIneq_DefC10}   
C_{10}= C_{10} (f,\CC,d,\alpha,T_0) = \frac{2 C_0 }{1-\Lambda^{-\alpha}} \exp\biggl(\frac{C_0 T_0 }{1-\Lambda^{-\alpha}}\bigl(\diam_d(S^2)\bigr)^\alpha\biggr)
\end{equation}
is a constant from Lemma~\ref{lmExpToLiniear}, and
\begin{equation}    \label{eqDefA1}   
A_1 \coloneqq (1-\Lambda^{-\alpha}) C_{10}(f,\CC,d,\alpha,T_0).
\end{equation}
Both of these constants only depend on $f$, $\CC$, $d$, $\Hseminorm{\alpha,\, (S^2,d)}{\phi}$ and $\alpha$.

Define a constant
\begin{equation}   \label{eqDefA0}
A_0= A_0\bigl(f,\CC,d,\Hseminorm{\alpha,\, (S^2,d)}{\phi},\alpha\bigr) \coloneqq \frac{(1+2T_0) A_1}{1-\Lambda^{-\alpha}}    = (1+2T_0)C_{10}\bigl(f,\CC,d,\alpha,T_0\bigr)>2
\end{equation}
depending only on $f$, $\CC$, $d$, $\Hseminorm{\alpha,\, (S^2,d)}{\phi}$, and $\alpha$. By (\ref{eqDefA0}), (\ref{eqDefT0}), and (\ref{eqPflmBasicIneq_DefC10}), we see that $A_0$ increases as $\Hseminorm{\alpha,\, (S^2,d)}{\phi}$ increases. Now (\ref{eqBasicIneqR}) follows from the fact that $A_0\geq A_1$.

\smallskip

(ii) Fix $x,x'\in X^0_\b$. For each $X^n\in\X^n_\b$, denote $y_{X^n} \coloneqq (f^n|_{X^n})^{-1}(x)$ and $y'_{X^n} \coloneqq (f^n|_{X^n})^{-1}(x')$.

Note that by (\ref{eqDefPhiWidetilde}) and (\ref{eqT0Bound}), we have
\begin{equation}  \label{eqPflmBasicIneqC1}
\Hseminormbig{\alpha,\, (S^2,d)}{\wt{s\phi}}  \leq  \Hseminormbig{\alpha,\, (S^2,d)}{\wt{a\phi}} + \Hseminorm{\alpha,\, (S^2,d)}{b\phi} \leq T_0+ \abs{b} \Hseminorm{\alpha,\, (S^2,d)}{\phi} \leq 2T_0 \max\{1, \, \abs{b}\},
\end{equation}
since $T_0 \geq \Hseminorm{\alpha,\, (S^2,d)}{\phi}$ by (\ref{eqDefT0}) and the fact that $C_{13}>1$ from Lemma~\ref{lmBound_aPhi}.

Note that
\begin{align}    \label{eqPflmBasicIneqC2}
     &       \AbsBig{\RR_{\wt{s\phi},X^0_\b,E}^{(n)} (v)(x) - \RR_{\wt{s\phi},X^0_\b,E}^{(n)} (v)(x')}   \notag\\
&\qquad \leq    \sum\limits_{\substack{X^n\in\X^n_\b \\ X^n\subseteq E}} 
             \Abs{ v(y_{X^n})  e^{S_n\wt{s\phi}(y_{X^n})}- v(y'_{X^n}) e^{S_n\wt{s\phi}(y'_{X^n})} }  \\
&\qquad \leq   \sum\limits_{\substack{X^n\in\X^n_\b \\ X^n\subseteq E}}   
             \Bigl(  \Abs{v(y_{X^n}) -  v(y'_{X^n}) }  \AbsBig{ e^{S_n\wt{s\phi}(y'_{X^n})}}  
                   + \abs{v(y_{X^n})} \AbsBig{e^{S_n\wt{s\phi}(y_{X^n})}  - e^{S_n\wt{s\phi}(y'_{X^n})}  } \Bigr). \notag 
\end{align}

We bound the two terms in the last summation above separately.

By Lemma~\ref{lmSnPhiBound}, Lemma~\ref{lmExpToLiniear}, (\ref{eqDefA1}), and (\ref{eqPflmBasicIneqC1}), 
\begin{align}   \label{eqPflmBasicIneqC3}
&            \sum\limits_{\substack{X^n\in\X^n_\b \\ X^n\subseteq E}} 
                     \abs{v(y_{X^n})} \Abs{e^{S_n\wt{s\phi}(y_{X^n})}  - e^{S_n\wt{s\phi}(y'_{X^n})}  }    \notag \\
&\qquad \leq  \sum\limits_{\substack{X^n\in\X^n_\b \\ X^n\subseteq E}} 
                     \abs{v(y_{X^n})}  \AbsBig{1- e^{S_n\wt{s\phi}(y'_{X^n}) - S_n\wt{s\phi}(y_{X^n})} }  
                      e^{S_n\wt{a\phi}(y_{X^n})}    \notag \\
&\qquad \leq C_{10}(f,\CC,d,\alpha,T_0) \Hseminormbig{\alpha,\, (S^2,d)}{\wt{s\phi}}  d(x,x')^\alpha  \RR_{\wt{a\phi},X^0_\b,E}^{(n)} (\abs{v})(x)  \\
&\qquad \leq  A_1  \frac{   2T_0 \max\{1, \, \abs{b}\} \RR_{\wt{a\phi},X^0_\b,E}^{(n)} (\abs{v})(x)   }{   1-\Lambda^{-\alpha} }  d(x,x')^\alpha  \notag\\
&\qquad =     A_0  \max\{1, \, \abs{b}\} \RR_{\wt{a\phi}, X^0_\b, E}^{(n)} (\abs{v})(x)   d(x,x')^\alpha, \notag
\end{align}
where the last inequality follows from (\ref{eqDefA0}).

By (\ref{eqDefWideTildePsiComplex}), Lemma~\ref{lmMetricDistortion}, and (\ref{eqRDiscontTildeSupDecreasing}) in Lemma~\ref{lmRtildeNorm=1},
\begin{align}   \label{eqPflmBasicIneqC4}
         \sum\limits_{\substack{X^n\in\X^n_\b \\ X^n\subseteq E}}    
            \Abs{v(y_{X^n}) -  v(y'_{X^n}) }  \AbsBig{ e^{S_n\wt{s\phi}(y'_{X^n})}}  
\leq &   \sum\limits_{\substack{X^n\in\X^n_\b \\ X^n\subseteq E}} 
             \HseminormD{\alpha,\, (E,d)}{v}    d(y_{X^n}, y'_{X^n})^\alpha  e^{S_n\wt{a\phi}(y'_{X^n})}  \notag   \\
\leq &     \HseminormD{\alpha,\, (E,d)}{v}    \frac{d(x,x')^\alpha C_0^\alpha}{\Lambda^{\alpha n}}   \sum\limits_{X^n\in\X^n_\b} e^{S_n\wt{a\phi}(y'_{X^n})} \\
\leq &  C_0 \frac{  \HseminormD{\alpha,\, (E,d)}{v}  }{   \Lambda^{\alpha n}   }   d(x,x')^\alpha.  \notag 
\end{align}

Thus (\ref{eqBasicIneqN1}) follows from (\ref{eqPflmBasicIneqC2}), (\ref{eqPflmBasicIneqC3}) and (\ref{eqPflmBasicIneqC4}).

If, in addition, there exists a non-negative real-valued H\"{o}lder continuous function $h\in\Holder{\alpha}(E,d)$ such that 
\begin{equation*}
\abs{v(y)-v(y')} \leq B (h(y)+h(y')) d(y,y')^\alpha
\end{equation*}
when $y,y'\in E$, then by Lemma~\ref{lmSnPhiBound}, Lemma~\ref{lmMetricDistortion}, (\ref{eqT0Bound}), (\ref{eqDefA1}), and (\ref{eqPflmBasicIneq_DefC10}),
\begin{align}   \label{eqPflmBasicIneqC5}
&             \sum\limits_{\substack{X^n\in\X^n_\b \\ X^n\subseteq E}}   
                  \Abs{v(y_{X^n}) -  v(y'_{X^n}) }   \AbsBig{ e^{S_n\wt{s\phi}(y'_{X^n})}}  \\
&\qquad \leq  \sum\limits_{\substack{X^n\in\X^n_\b \\ X^n\subseteq E}} 
                 B \Bigl( h(y_{X^n})  e^{S_n\wt{a\phi}(y_{X^n})}   e^{\Absbig{S_n\wt{a\phi}(y'_{X^n}) - S_n\wt{a\phi}(y_{X^n})}} 
                         +h(y'_{X^n}) e^{S_n\wt{a\phi}(y'_{X^n})}  \Bigr)  d(y_{X^n}, y'_{X^n})^\alpha \notag \\
&\qquad \leq  B\exp \Biggl( \frac{ \Hseminormbig{\alpha,\, (S^2,d)}{\wt{a\phi}}  C_0 \bigl(\diam_d(S^2)\bigr)^\alpha  }
                                 {  1-\Lambda^{-\alpha}  }  \Biggr) \frac{  d(x,x')^\alpha C_0^\alpha  }{  \Lambda^{\alpha n}  }   
                     \Bigl( \RR_{\wt{a\phi},X^0_\b,E}^{(n)} (h)(x) + \RR_{\wt{a\phi},X^0_\b,E}^{(n)} (h)(x')  \Bigr)  \notag\\
&\qquad \leq   A_1  \frac{B}{\Lambda^{\alpha n}}  \Bigl( \RR_{\wt{a\phi},X^0_\b,E}^{(n)} (h)(x) + \RR_{\wt{a\phi},X^0_\b,E}^{(n)} (h)(x')  \Bigr) d(x,x')^\alpha.   \notag
\end{align}
Therefore, (\ref{eqBasicIneqC}) follows from (\ref{eqPflmBasicIneqC2}), (\ref{eqPflmBasicIneqC3}), (\ref{eqPflmBasicIneqC5}), and the fact that $A_0\geq A_1$ from (\ref{eqDefA0}).
\end{proof}

\subsection{Spectral gap}   \label{subsctSplitRuelleOp_SpectralGap}

Let $(X,d)$ be a metric space. A function $h\: [0,+\infty)\rightarrow [0,+\infty)$ is an \defn{abstract modulus of continuity} if it is continuous at $0$, non-decreasing, and $h(0)=0$. Given any constant $\tau \in [0,+\infty]$, and any abstract modulus of continuity $g$, we define the subclass $\CCC_g^\tau((X, d),\C)$ of $\CCC(X,\C)$ as
\begin{equation*}
\CCC_g^\tau((X,d),\C)=\bigl\{u\in\CCC(X,\C) \,\big|\,  \Norm{u}_{\CCC^0(X)}\leq b \text{ and for }x,y\in X, \, \Abs{u(x)-u(y)}\leq g(d(x,y))  \bigr\}.
\end{equation*}
We denote $\CCC_g^\tau(X,d) \coloneqq \CCC_g^\tau((X,d),\C) \cap \CCC(X)$.

Assume now that $(X,d)$ is compact. Then by the Arzel\`a-Ascoli Theorem, each $\CCC_g^\tau((X,d),\C)$ (resp.\ $\CCC_g^\tau(X,d)$) is precompact in $\CCC(X,\C)$ (resp.\ $\CCC(X)$) equipped with the uniform norm. It is easy to see that each $\CCC_g^\tau((X,d),\C)$ (resp.\ $\CCC_g^\tau(X,d)$) is actually compact. On the other hand, for $u\in\CCC(X,\C)$, we can define an abstract modulus of continuity by
\begin{equation}    \label{eqAbsModContForU}
g(t) \coloneqq \sup \{\abs{u(x)-u(y)} \,|\, x,y \in X, d(x,y) \leq t\}
\end{equation}
for $t\in [0,+\infty)$, so that $u\in \CCC_g^\iota ((X,d),\C)$, where $\iota \coloneqq \Norm{u}_\infty$.

The following lemma is  easy to check (see also \cite[Lemma~5.24]{Li17}).

\begin{lemma}  \label{lmChbChbSubsetChB}
Let $(X,d)$ be a metric space. For each pair of constants $\tau_1,\tau_2 \geq 0$ and each pair of abstract moduli of continuity $g_1,g_2$, we have
\begin{equation}
\bigl\{u_1 u_2 \,\big|\, u_1 \in \CCC_{g_1}^{\tau_1} ((X,d),\C), u_2 \in \CCC_{g_2}^{\tau_2} ((X,d),\C) \bigr\}  \subseteq \CCC_{\tau_1 g_2 + \tau_2 g_1}^{\tau_1 \tau_2} ((X,d),\C),
\end{equation}
and for each $c>0$,
\begin{equation}
\Big\{ \frac{1}{u} \,\Big|\,  u\in \CCC_{g_1}^{\tau_1} ((X,d),\C), u(x) \geq c \text{ for each } x\in X\Big\}  \subseteq \CCC_{c^{-2} g_1}^{c^{-1}}((X,d),\C).
\end{equation}
\end{lemma}

The following corollary follows immediately from Lemma~\ref{lmChbChbSubsetChB}. We leave the proof to the readers.

\begin{cor}    \label{corProductInverseHolderNorm}
Let $(X,d)$ be a metric space, and $\alpha \in (0,1]$ a constant. Then for all H\"{o}lder continuous functions $u,v \in \Holder{\alpha}((X,d),\C)$, we have $u, v \in \Holder{\alpha}((X,d),\C)$ with
\begin{equation*}
 \Hnorm{\alpha}{u v}{(X,d)} \leq \Hnorm{\alpha}{u }{(X,d)}   \Hnorm{\alpha}{v}{(X,d)},
\end{equation*}
and if, in addition, $\abs{u(x)} \geq c$, for each $x\in X$, for some constant $c>0$, then $\frac{1}{u} \in \Holder{\alpha}((X,d),\C)$ with
\begin{equation*}
\Hnormbigg{\alpha}{\frac{1}{u}}{(X,d)} \leq \frac{1}{c} + \frac{1}{c^2} \Hnorm{\alpha}{u}{(X,d)}.
\end{equation*}
\end{cor}

\begin{lemma}   \label{lmRDiscontTildeDual}
Let $f$, $\CC$, $d$, $\alpha$ satisfy the Assumptions. Assume in addition that $f(\CC)\subseteq\CC$. Let $\phi\in \Holder{\alpha}(S^2,d)$ be a real-valued H\"{o}lder continuous function with an exponent $\alpha$, and $\mu_\phi$ denote the unique equilibrium state for $f$ and $\phi$. Given $X^0\in\X^0(f,\CC)$ and $u\in\CCC(X^0)$. Then
\begin{equation*}
    \int_{X^0}\! u\,\mathrm{d}\mu_\phi 
=   \sum\limits_{ \c\in\{\b,\w\} }\int_{X^0_\c} \!  \RR_{\wt{\phi}, X^0_\c, X^0}^{(n)} (u) \,\mathrm{d}\mu_\phi.
\end{equation*}
for each $n\in\N$.
\end{lemma}

\begin{proof}
We denote a function $v\in B(S^2)$ by
\begin{equation*} 
v(x)  \coloneqq  \begin{cases} u(x) & \text{if } x\in \inte  (X^0), \\ 0  & \text{otherwise}. \end{cases}
\end{equation*}
We choose a pointwise increasing sequence of continuous non-negative functions $\tau_i\in \CCC(S^2)$, $i\in\N$, such that  $\lim\limits_{i\to+\infty} \tau_i(x) = \mathbbm{1}_{\inte (X^0)}$ for all $x\in S^2$. Then $\{v\tau_i\}_{i\in\N}$ is a bounded sequence of continuous functions on $S^2$, convergent pointwise to $v$. 

Fix $n\in\N$. Since $\mu_\phi (\CC) = 0$ by \cite[Proposition~5.39]{Li17}, then by (\ref{eqLbwInteMatchLbw}) and the Dominated Convergence Theorem, we get
\begin{align*}
      \sum\limits_{ \c\in\{\b,\w\} } \int_{X^0_\c} \!  \RR_{\wt{\phi}, X^0_\c, X^0}^{(n)} (u) \,\mathrm{d}\mu_\phi   
= &\sum\limits_{ \c\in\{\b,\w\} } \int_{X^0_\c} \! \sum\limits_{\substack{X^n\in \X^n_\c \\ X^n\subseteq X^0}  } \Bigl(e^{S_n\wt{\phi}} u\Bigr) \bigl((f^n|_{X^n})^{-1} (x)\bigr) \,\mathrm{d}\mu_\phi(x) \\ 
= &\sum\limits_{ \c\in\{\b,\w\} } \lim\limits_{i\to+\infty}   \int_{X^0_\c} \! \sum\limits_{\substack{X^n\in \X^n_\c \\ X^n\subseteq X^0}  } \Bigl(e^{S_n\wt{\phi}} v\tau_i \Bigr) \bigl((f^n|_{X^n})^{-1} (x)\bigr) \,\mathrm{d}\mu_\phi(x) \\
= &\sum\limits_{ \c\in\{\b,\w\} } \lim\limits_{i\to+\infty}   \int_{X^0_\c} \! \RR_{\wt{\phi}}^n (v \tau_i)(x) \,\mathrm{d}\mu_\phi(x)    \\
= & \lim\limits_{i\to+\infty}   \int_{S^2} \! \RR_{\wt{\phi}}^n (v \tau_i)  \,\mathrm{d}\mu_\phi 
=  \lim\limits_{i\to+\infty}   \int_{S^2} \! v \tau_i \,\mathrm{d} \bigl(\RR^*_{\wt{\phi}}\bigr)^n (\mu_\phi)  \\
= & \lim\limits_{i\to+\infty}   \int_{S^2} \! v \tau_i \,\mathrm{d}   \mu_\phi   
=     \int_{S^2} \! v   \,\mathrm{d}   \mu_\phi  
=     \int_{X^0} \! u   \,\mathrm{d}   \mu_\phi .
\end{align*}
The proof is now complete.
\end{proof}

\begin{lemma}    \label{lmLSplitTildeSupBound}
Let $f$, $\CC$, $d$ satisfy the Assumptions. Assume in addition that $f(\CC)\subseteq\CC$. Given an abstract modulus of continuity $g$. Then for each $\alpha \in (0,1]$, $K\in(0,+\infty)$, and $\delta_1\in(0,+\infty)$, there exist constants $\delta_2\in (0,+\infty)$ and $N\in\N$ with the following property:

For all $\c\in\{\b,\w\}$, $u_\b\in \CCC_g^{+\infty}(X^0_\b,d)$, $u_\w\in \CCC_g^{+\infty}(X^0_\w,d)$, and $\phi\in\Holder{\alpha}(S^2,d)$, if $\Hnorm{\alpha}{\phi}{(S^2,d)} \leq K$, $\max\bigl\{  \norm{u_\b}_{\CCC^0(X^0_\b)} , \, \norm{u_\w}_{\CCC^0(X^0_\w)}  \bigr\} \geq \delta_1$, and $\int_{X^0_\b}\! u_\b\,\mathrm{d}\mu_\phi + \int_{X^0_\w}\! u_\w\,\mathrm{d}\mu_\phi  = 0$ where $\mu_\phi$ denotes the unique equilibrium state for $f$ and $\phi$, then
\begin{equation*}
\NormBig{\RR_{\wt{\phi},X^0_\c,X^0_\b}^{(N)}(u_\b)  + \RR_{\wt{\phi},X^0_\c,X^0_\w}^{(N)}(u_\w)  }_{\CCC^0(X^0)} \leq \max\bigl\{  \norm{u_\b}_{\CCC^0(X^0_\b)} , \, \norm{u_\w}_{\CCC^0(X^0_\w)}  \bigr\} - \delta_2.
\end{equation*}
\end{lemma}

\begin{proof}
Fix arbitrary constants $\alpha \in (0,1]$, $K\in (0,+\infty)$, and $\delta_1 \in (0,+\infty)$. Choose $\epsilon > 0$ small enough such that $g(\epsilon) < \frac{\delta_1}{2}$. Fix $\c\in\{\b,\w\}$. Let $n_0\in\N$ be the smallest number such that $f^{n_0} \bigl( \inte\bigl( X^0_\b \bigr) \bigr)= S^2 = f^{n_0}  \bigl( \inte\bigl( X^0_\w \bigr) \bigr)$.

By Lemma~\ref{lmCellBoundsBM}~(iv), there exists a number $N\in \N$ depending only on $f$, $\CC$, $d$, $g$, and $\delta_1$ such that $N\geq 2n_0$ and for each $z\in S^2$, we have $U^{N-n_0}(z) \subseteq B_d(z,\epsilon)$ (see (\ref{defU^n})).

Fix an arbitrary $\phi\in \Holder{\alpha}(S^2,d)$ with $\Hnorm{\alpha}{\phi}{(S^2,d)} \leq K$, and arbitrary functions $u_\b\in \CCC_g^{+\infty}(X^0_\b,d)$ and $u_\w\in \CCC_g^{+\infty}(X^0_\w,d)$ with $\max\bigl\{  \norm{u_\b}_{\CCC^0(X^0_\b)} , \, \norm{u_\w}_{\CCC^0(X^0_\w)}  \bigr\} \geq \delta_1$ and $\int_{X^0_\b}\! u_\b\,\mathrm{d}\mu_\phi + \int_{X^0_\w}\! u_\w\,\mathrm{d}\mu_\phi  = 0$. Without loss of generality, we assume that $\int_{X^0_\b}\! u_\b\,\mathrm{d}\mu_\phi \leq 0$ and $\int_{X^0_\w}\! u_\w\,\mathrm{d}\mu_\phi \geq 0$. So we can choose points $y_1 \in X^0_\b$ and $y_2\in X^0_\w$ in such a way that $u_\b(y_1) \leq 0$ and $u_\w(y_2)\geq 0$.

We denote
\begin{equation*}
M   \coloneqq   \max\bigl\{  \norm{u_\b}_{\CCC^0(X^0_\b)} , \, \norm{u_\w}_{\CCC^0(X^0_\w)}  \bigr\}.
\end{equation*}

We fix a point $x\in X^0_\c$. Since $f^N \bigl( U^{N-n_0}(y_1) \cap X^0_\b \bigr) = S^2$, there exists $y\in f^{-N}(x)\cap X^0_\b$ such that $y\in U^{N-n_0}(y_1) \subseteq B_d(y_1,\epsilon)$. Since $M\geq \delta_1$,
\begin{equation*}
u_\b(y)\leq  u_\b(y_1) + g(\epsilon) < \frac{\delta_1}{2} \leq M - \frac{\delta_1}{2}.
\end{equation*}
Choose $X^N_y \in \X^N_\b$ such that $y\in X^N_y \subseteq X^0_\b$. Denote $w_{X^N} \coloneqq (f^N|_{X^N})^{-1}(x)$ for each $X^N\in \X^N_\c$. So by Lemma~\ref{lmRtildeNorm=1}, we have
\begin{align*}
&            \RR_{\wt{\phi},X^0_\c,X^0_\b}^{(N)} (u_\b) (x)  +  \RR_{\wt{\phi},X^0_\c,X^0_\w}^{(N)} (u_\w) (x)      \\
&\qquad =    u_\b(y) e^{S_N\wt{\phi}(y)} 
              + \sum\limits_{\substack{X^N\in \X^N_\c \setminus\{X^N_y\} \\ X^N\subseteq X^0_\b}}  u_\b(w_{X^N}) e^{S_N\wt{\phi}(w_{X^N})}
              + \sum\limits_{\substack{X^N\in \X^N_\c                   \\ X^N\subseteq X^0_\w}}  u_\w(w_{X^N}) e^{S_N\wt{\phi}(w_{X^N})}      \\
&\qquad\leq \biggl( M - \frac{\delta_1}{2} \biggr) \exp\bigl(S_N\wt{\phi}(y)\bigr)  
              +  M  \sum\limits_{ X^N\in \X^N_\c\setminus\{X^N_y\}}  \exp\bigl(S_N\wt{\phi}(w_{X^N})\bigr) \\
&\qquad =    M  \sum\limits_{ X^N\in \X^N_\c  }  \exp\bigl(S_N\wt{\phi}(w_{X^N})\bigr)  -  \frac{\delta_1}{2}   \exp\bigl(S_N\wt{\phi}(y)\bigr) \\
&\qquad =    M  -  \frac{\delta_1}{2}   \exp\bigl(S_N\wt{\phi}(y)\bigr) .
\end{align*}
Similarly, there exists $z\in f^{-N}(x)\cap X^0_\w$ such that $z\in U^{N-n_0}(y_2) \subseteq B_d(y_2, \epsilon)$ and
\begin{equation*}
\RR_{\wt{\phi},X^0_\c,X^0_\b}^{(N)} (u_\b) (x)  +  \RR_{\wt{\phi},X^0_\c,X^0_\w}^{(N)} (u_\w) (x) \geq - M  +  \frac{\delta_1}{2}   \exp\bigl(S_N\wt{\phi}(z)\bigr) .
\end{equation*}
Hence we get
\begin{equation*}
       \NormBig{ \RR_{\wt{\phi},X^0_\c,X^0_\b}^{(N)} (u_\b)  +  \RR_{\wt{\phi},X^0_\c,X^0_\w}^{(N)} (u_\w)  }_{ \CCC^0(X^0_\b) }  
\leq  M -  \frac{\delta_1}{2} \inf \bigl\{  \exp\bigl(S_N\wt{\phi}(w)\bigr)  \,\big|\, w\in S^2  \bigr\}.
\end{equation*}

By (\ref{eqBound_aPhi_Sup}) in Lemma~\ref{lmBound_aPhi} with $T\coloneqq 1$, the definition of $M$ above, and (\ref{eqDefHolderNorm}), we have
\begin{equation*}
     \NormBig{\RR_{\wt{\phi},X^0_\c,X^0_\b}^{(N)}(u_\b)  + \RR_{\wt{\phi},X^0_\c,X^0_\w}^{(N)}(u_\w)  }_{\CCC^0(X_b^0)}
 \leq \max\Bigl\{  \norm{u_\b}_{\CCC^0(X^0_\b)} , \, \norm{u_\w}_{\CCC^0(X^0_\w)}  \Bigr\} - \delta_2
\end{equation*}
with
\begin{equation*}
\delta_2   \coloneqq    \frac{\delta_1}{2} \exp ( - N (C_{13} K + \abs{\log(\deg f)})),
\end{equation*}
where $C_{13}$ is a constant from Lemma~\ref{lmBound_aPhi} depending only on $f$, $\CC$, $d$, and $\alpha$. Therefore the constant $\delta_2$ depends only on $f$, $\CC$, $d$, $\alpha$, $g$, $K$, and $\delta_1$.
\end{proof}

\begin{theorem}    \label{thmLDiscontTildeUniformConv}
Let $f\: S^2 \rightarrow S^2$ be an expanding Thurston map with a Jordan curve $\CC\subseteq S^2$ satisfying $f(\CC)\subseteq\CC$ and $\post f\subseteq \CC$. Let $d$ be a visual metric on $S^2$ for $f$ with expansion factor $\Lambda>1$, $\alpha\in(0,1]$ be a constant, and $X^0\in \X^0(f,\CC)$ be a $0$-tile. Let $H$, $H_\b$, and $H_\w$ be bounded subsets of $\Holder{\alpha}(S^2,d)$, $\Holder{\alpha} \bigl( X^0_\b,d \bigr)$, and $\Holder{\alpha} \bigl(X^0_\w,d \bigr)$, respectively. Then for each $\phi\in H$, each pair of functions $u_\b\in H_\b$ and $u_\w\in H_\w$, we have
\begin{equation}  \label{eqLDiscontTildeUniformConv}
\lim\limits_{n\to+\infty}  \NormBig{ \RR_{\wt{\phi}, X^0, X^0_\b}^{(n)}  (\overline{u}_\b )    +       \RR_{\wt{\phi}, X^0, X^0_\w}^{(n)} (\overline{u}_\w)}_{\CCC^0(X^0)} = 0,
\end{equation}
where the pair of functions $\overline{u}_\b \in \Holder{\alpha} \bigl(X^0_\b,d \bigr)$ and $\overline{u}_\w \in \Holder{\alpha} \bigl(X^0_\w,d \bigr)$ are given by
\begin{equation*}
\overline{u}_\b  \coloneqq  u_\b - \int_{X^0_\b}\! u_\b\,\mathrm{d}\mu_\phi - \int_{X^0_\w}\! u_\w\,\mathrm{d}\mu_\phi   \mbox{ and } 
\overline{u}_\w \coloneqq  u_\w - \int_{X^0_\b}\! u_\b\,\mathrm{d}\mu_\phi - \int_{X^0_\w}\! u_\w\,\mathrm{d}\mu_\phi
\end{equation*}
with $\mu_\phi$ denoting the unique equilibrium state for $f$ and $\phi$.

Moreover, the convergence in (\ref{eqLDiscontTildeUniformConv}) is uniform in $\phi\in H$, $u_\b\in H_\b$, and $u_\w\in H_\w$.
\end{theorem}

\begin{proof}
Without loss of generality, we assume that $H\neq\emptyset$, $H_\b\neq\emptyset$, and $H_\w\neq\emptyset$. Define constants $K  \coloneqq  \sup \bigl\{ \Hnorm{\alpha}{\phi}{(S^2,d)} \,\big|\, \phi\in H\bigr\} \in [0,+\infty)$ and $K_\c  \coloneqq  \sup \bigl\{ \Hnorm{\alpha}{u_\c}{(X^0_\b,d)} \,\big|\, u_\c\in H_\c\bigr\}  \in [0,+\infty)$ for $\c\in\{\b,\w\}$. Define for each $n\in\N_0$,
\begin{equation*}
a_n  \coloneqq  \sup \biggl\{  \NormBig{ \RR_{\wt{\phi}, X^0_1, X^0_\b}^{(n)}  (\overline{u}_\b )    +       \RR_{\wt{\phi}, X^0_1, X^0_\w}^{(n)} (\overline{u}_\w)}_{\CCC^0(X^0_1)}  \,\bigg|\,  X^0_1\in\X^0, \phi\in H, u_\b\in H_\b, u_\w\in H_\w \biggr\}.
\end{equation*}
Note that by Definition~\ref{defSplitRuelle}, $a_0\leq 2K_\b+2K_\w <+\infty$.

By (\ref{eqLDiscontProperties_Split}) in Lemma~\ref{lmLDiscontProperties} and (\ref{eqSplitRuelleTildeSupDecreasing}) in Lemma~\ref{lmRtildeNorm=1}, for each $n\in\N_0$, each $\phi\in H$, each $X^0_1\in\X^0$, each $v_\b\in\CCC(X^0_\b)$, and each $v_\w\in\CCC(X^0_\w)$, we have
\begin{align*}
&         \NormBig{  \RR_{\wt{\phi},X^0_1,X^0_\b}^{(n+1)}  (v_\b) +  \RR_{\wt{\phi},X^0_1,X^0_\w}^{(n+1)}  (v_\w)   }_{ \CCC^0(X^0_1,d)}  \\
&\qquad = \Normbigg{   \sum\limits_{Y^0\in\X^0} \  \RR_{\wt{\phi},X^0_1,Y^0}^{(1)}  \Bigl(  \RR_{\wt{\phi},Y^0,X^0_\b}^{(n)}  (v_\b) + \RR_{\wt{\phi},Y^0,X^0_\w}^{(n)}  (v_\w)\Bigr)   }_{ \CCC^0(X^0_1,d)}    \\
&\qquad\leq \max\biggl\{   \NormBig{   \RR_{\wt{\phi},Y^0,X^0_\b}^{(n)}  (v_\b) + \RR_{\wt{\phi},Y^0,X^0_\w}^{(n)}  (v_\w)   }_{ \CCC^0(Y^0,d)}  \,\bigg|\, Y^0\in \X^0    \biggr\}.
\end{align*}
So $\{a_n\}_{n\in\N_0}$ is a non-increasing sequence of non-negative real numbers. 

Suppose now that $\lim\limits_{n\to+\infty} a_n   \eqqcolon   a_* >0$. By Lemma~\ref{lmRtildeNorm=1}, (\ref{eqBasicIneqN1}) in Lemma~\ref{lmBasicIneq} with $a \coloneqq 1$ and $b  \coloneqq 0$, (\ref{eqDefA0}), (\ref{eqDefT0}), and (\ref{eqDefC10}), we get that $\RR_{\wt{\phi}, X^0_\c, X^0_\b}^{(n)} (\overline{u}_\b)  + \RR_{\wt{\phi}, X^0_\c, X^0_\w}^{(n)} (\overline{u}_\w)  \in \CCC^{2(K_\b+K_\w)}_g(X^0_\c,d)$, for each $\c\in\{\b,\w\}$ and each pair of $u_\b\in H_\b$ and $u_\w\in H_\w$, with an abstract modulus of continuity $g$ given by $g(t)  \coloneqq  2 (C_0(K_\b+K_\w)+2(K_\b+K_\w)A)t^\alpha$, $t\in[0,+\infty)$, where the constant $A>1$ is given by
\begin{equation*}
A  \coloneqq (1+2T) \frac{ 2C_0 }{1-\Lambda^{-\alpha}} \exp\biggl( \frac{C_0 T}{1-\Lambda^{-\alpha}} \bigl( \diam_d(S^2)\bigr)^\alpha   \biggr),
\end{equation*}
and $T  \coloneqq  (2 s_0 + 1)  C_{13} K \exp ( 2 s_0 C_{14} K)$. Here the constant $C_0 > 1$ depending only on $f$, $d$, and $\CC$ comes from Lemma~\ref{lmMetricDistortion}, and $C_{13}>1$, $C_{14}>0$ are constants from Lemma~\ref{lmBound_aPhi} depending only on $f$, $\CC$, $d$, and $\alpha$. So $g$ and $A$ both depend only on $f$, $\CC$, $d$, $\alpha$, $H$, $H_\b$, and $H_\w$. By Lemma~\ref{lmRDiscontTildeDual},
\begin{equation*}
\sum\limits_{Y^0 \in \X^0}  \int_{Y^0} \!  \Bigl(  \RR_{\wt{\phi},Y^0,X^0_\b}^{(n)}  (\overline{u}_\b) + \RR_{\wt{\phi},Y^0,X^0_\w}^{(n)}  (\overline{u}_\w)  \Bigr) \,\mathrm{d}\mu_\phi 
=  \int_{X^0_\b} \!  \overline{u}_\b \,\mathrm{d}\mu_\phi + \int_{X^0_\w} \!  \overline{u}_\w \,\mathrm{d}\mu_\phi = 0.
\end{equation*}
By  (\ref{eqLDiscontProperties_Split}) in Lemma~\ref{lmLDiscontProperties}, (\ref{eqSplitRuelleTildeSupDecreasing}) in Lemma~\ref{lmRtildeNorm=1}, and applying Lemma~\ref{lmLSplitTildeSupBound} with $f$, $\CC$, $d$, $g$, $\alpha$, $K$, and $\delta_1 \coloneqq  \frac{a_*}{2}>0$, we find constants $N\in\N$ and $\delta_2>0$ such that
\begin{align*}
&          \Normbigg{  \RR_{\wt{\phi}, X^0_1, X^0_\b}^{(N+n)} (\overline{u}_\b)  + \RR_{\wt{\phi}, X^0_1, X^0_\w}^{(N+n)} (\overline{u}_\w)      }_{\CCC^0(X^0_1)}   \\
&\qquad = \NormBigg{ \sum\limits_{Y^0_1\in\X^0} \RR_{\wt{\phi}, X^0_1, Y^0_1}^{(N)} \Bigl(  \RR_{\wt{\phi}, Y^0_1, X^0_\b}^{(n)} (\overline{u}_\b ) +   \RR_{\wt{\phi}, Y^0_1, X^0_\w}^{(n)} (\overline{u}_\w) \Bigr)       }_{\CCC^0(X^0_1)} \\
&\qquad\leq \max \biggl\{  \NormBig{  \RR_{\wt{\phi}, Y^0_1, X^0_\b}^{(n)} (\overline{u}_\b ) +   \RR_{\wt{\phi}, Y^0_1, X^0_\w}^{(n)} (\overline{u}_\w)  }_{\CCC^0(Y^0_1)} \,\bigg|\, Y^0_1\in\X^0       \biggr\}   - \delta_2\\
&\qquad\leq  a_n - \delta_2,
\end{align*}
for each $n\in\N_0$, each $X^0_1\in\X^0$, each $\phi\in H$, and each pair $u_\b\in H_\b$ and $u_\w\in H_\w$ satisfying
\begin{equation}   \label{eqPfthmLDiscontTildeUniformConv_Geq}
\max \biggl\{  \NormBig{  \RR_{\wt{\phi}, Y^0_1, X^0_\b}^{(n)} (\overline{u}_\b ) +   \RR_{\wt{\phi}, Y^0_1, X^0_\w}^{(n)} (\overline{u}_\w)  }_{\CCC^0(Y^0_1)} \,\bigg|\, Y^0_1\in\X^0       \biggr\}    \geq \frac{a_*}{2}.
\end{equation}
Since $\lim\limits_{n\to+\infty} a_n = a_*$, we can fix $m\geq 1$ large enough so that $a_m\leq a_* +\frac{\delta_2}{2}$. Then for each $X^0_1\in\X^0$, each $\phi\in H$, and each pair $u_\b\in H_\b$ and $u_\w\in H_\w$ satisfying (\ref{eqPfthmLDiscontTildeUniformConv_Geq}) with $n  \coloneqq  m$, we have 
\begin{equation*}
 \NormBig{  \RR_{\wt{\phi}, X^0_1, X^0_\b}^{(N+m)} (\overline{u}_\b)  + \RR_{\wt{\phi}, X^0_1, X^0_\w}^{(N+m)} (\overline{u}_\w)      }_{\CCC^0(X^0_1)} 
\leq a_m - \delta_2 \leq a_* - \frac{\delta_2}{2}.
\end{equation*}
On the other hand, by (\ref{eqSplitRuelleTildeSupDecreasing}) in Lemma~\ref{lmRtildeNorm=1}, we get that for each $\phi \in H$ and each pair $u_\b\in H_\b$ and $u_\w\in H_\w$ with 
\begin{equation*}   \label{eqPfthmLDiscontTildeUniformConv_Less}
\max \biggl\{  \NormBig{  \RR_{\wt{\phi}, Y^0_1, X^0_\b}^{(m)} (\overline{u}_\b ) +   \RR_{\wt{\phi}, Y^0_1, X^0_\w}^{(m)} (\overline{u}_\w)  }_{\CCC^0(Y^0_1)} \,\bigg|\, Y^0_1\in\X^0       \biggr\}  < \frac{a_*}{2},
\end{equation*}
the following holds for each $X^0_1\in\X^0$:
\begin{equation*}
 \NormBig{  \RR_{\wt{\phi}, X^0_1, X^0_\b}^{(N+m)} (\overline{u}_\b)  + \RR_{\wt{\phi}, X^0_1, X^0_\w}^{(N+m)} (\overline{u}_\w)      }_{\CCC^0(X^0_1)} 
< \frac{a_*}{2}.
\end{equation*}
Thus $a_{N+m} \leq \max \bigl\{ a_*- \frac{\delta_2}{2}, \, \frac{a_*}{2} \bigr\}  < a_*$, contradicting the fact that $\{a_n\}_{n\in\N_0}$ is a non-increasing sequence and the assumption that $\lim\limits_{n\to+\infty}  a_n = a_*>0$. This proves the uniform convergence in (\ref{eqLDiscontTildeUniformConv}).
\end{proof}

\begin{theorem}    \label{thmSpectralGap}
Let $f\: S^2 \rightarrow S^2$ be an expanding Thurston map with a Jordan curve $\CC\subseteq S^2$ satisfying $f(\CC)\subseteq\CC$ and $\post f\subseteq \CC$. Let $d$ be a visual metric on $S^2$ for $f$ with expansion factor $\Lambda>1$. Let $\alpha\in(0,1]$ be a constant and $H$ be a bounded subsets of $\Holder{\alpha}(S^2,d)$. Then there exists a constant $\rho_1  \in(0,1)$ depending on $f$, $\CC$, $d$, $\alpha$, and $H$ such that the following property is satisfied:

\smallskip

For each $\phi\in H$, each $n\in\N_0$, each $0$-tile $X^0\in \X^0(f,\CC)$, each pair of real-valued H\"{o}lder continuous functions $u_\b\in \Holder{\alpha} \bigl(X^0_\b,d \bigr)$ and $u_\w\in \Holder{\alpha} \bigl(X^0_\w,d \bigr)$, we have
\begin{equation}  \label{eqSpectralGap}
 \NormBig{ \RR_{\wt{\phi}, X^0, X^0_\b}^{(n)}  (\overline{u}_\b )    +       \RR_{\wt{\phi}, X^0, X^0_\w}^{(n)} (\overline{u}_\w)}_{\CCC^0(X^0)}
 \leq 6 \rho_1^n \max\Bigl\{    \Hnorm{\alpha}{u_\b}{(X^0_\b,d)}, \,  \Hnorm{\alpha}{u_\w}{(X^0_\w,d)}  \Bigr\},
\end{equation}
where the pair of functions $\overline{u}_\b \in \Holder{\alpha} \bigl(X^0_\b,d \bigr)$ and $\overline{u}_\w \in \Holder{\alpha} \bigl(X^0_\w,d \bigr)$ are given by
\begin{equation*}
\overline{u}_\b   \coloneqq   u_\b - \int_{X^0_\b}\! u_\b\,\mathrm{d}\mu_\phi - \int_{X^0_\w}\! u_\w\,\mathrm{d}\mu_\phi   \mbox{ and } 
\overline{u}_\w   \coloneqq  u_\w - \int_{X^0_\b}\! u_\b\,\mathrm{d}\mu_\phi - \int_{X^0_\w}\! u_\w\,\mathrm{d}\mu_\phi
\end{equation*}
with $\mu_\phi$ denoting the unique equilibrium state for $f$ and $\phi$. In particular,
\begin{align}   \label{eqSpectralGapAll}
      & \NormBig{  \RR_{\wt{\phi}, X^0, X^0_\b}^{(n)} (u_\b)  + \RR_{\wt{\phi}, X^0, X^0_\w}^{(n)} (u_\w)}_{\CCC^0(X^0)} \notag   \\
\leq  & \Absbigg{ \int_{X^0_\b}\! u_\b\,\mathrm{d}\mu_\phi + \int_{X^0_\w}\! u_\w\,\mathrm{d}\mu_\phi  } + 6 \rho_1^n \max\Bigl\{    \Hnorm{\alpha}{u_\b}{(X^0_\b,d)},  \, \Hnorm{\alpha}{u_\w}{(X^0_\w,d)}  \Bigr\}.
\end{align}
\end{theorem}

\begin{proof}
Without loss of generality, we assume that $H\neq \emptyset$. Define a constant 
\begin{equation}   \label{eqPfthmSpectralGap_K}
K   \coloneqq   \sup\bigl\{  \Hnorm{\alpha}{\phi}{(S^2,d)}  \,\big|\, \phi\in H   \bigr\} \in [0,+\infty).
\end{equation}
Denote, for $\c\in\{\b,\w\}$,
\begin{equation*}
H_\c    \coloneqq    \Bigl\{  v_\c\in \Holder{\alpha} \bigl(X^0_\c,d \bigr) \,\Big|\, \Hnorm{\alpha}{v_\c}{(X^0_\c,d)} \leq 3      \Bigr\}.
\end{equation*}

The equation (\ref{eqSpectralGapAll}) follows immediately from the equation (\ref{eqSpectralGap}), the triangle inequality, and the fact that $\RR_{\wt{\phi},X^0,X^0_\b}^{(1)} \bigl(\mathbbm{1}_{X^0_\b}\bigr) + \RR_{\wt{\phi},X^0,X^0_\w}^{(1)} \bigl(\mathbbm{1}_{X^0_\w}\bigr) = \mathbbm{1}_{X^0}$ by (\ref{eqDefLc}) and Lemma~\ref{lmRtildeNorm=1}. So it suffices to establish (\ref{eqSpectralGap}).

We first consider the special case when $u_\b \in H_\b$ and $u_\w in H_\w$.

By (\ref{eqBasicIneqN1}) in Lemma~\ref{lmBasicIneq} with $s\coloneqq 1$, (\ref{eqRDiscontTildeSupDecreasing}) in Lemma~\ref{lmRtildeNorm=1}, and (\ref{eqPfthmSpectralGap_K}), we get that for each $j\in \N$, each $Y^0\in\X^0$, each $\phi\in H$, and each pair of functions $u_\b\in H_\b$ and $u_\w\in H_\w$,
\begin{align}   \label{eqPfthmSpectralGap_HolderBound}
&              \Hseminormbigg{\alpha,\, (Y^0,d)}{   \sum\limits_{\c\in\{\b,\w\}}    \RR_{\wt{\phi},Y^0,X^0_\c}^{(j)}( \overline{u}_\c )   }   \notag  \\
&\qquad  \leq  \frac{C_0}{\Lambda^{\alpha j}}    \sum\limits_{\c\in\{\b,\w\}}  \Hseminorm{\alpha,\, (X^0_\c,d)}{ \overline{u}_\c }  
                          +  A_0    \sum\limits_{\c\in\{\b,\w\}}   \NormBig{ \RR_{\wt{\phi}, Y^0, X^0_\c}^{(j)}  (\abs{\overline{u}_\c })}_{\CCC^0(X^0_\c)}\\
&\qquad \leq  \frac{6C_0}{\Lambda^{\alpha j}}  + A_0    \sum\limits_{\c\in\{\b,\w\}}    \Norm{\overline{u}_\c}_{\CCC^0(X^0_\c)}   
                 \leq    C_{17},  \notag
\end{align}
where the constant $C_{17}$ is given by $C_{17} \coloneqq 6 C_0 + 12 A_0$, the constant $A_0 \coloneqq A_0 \bigl(f,\CC,d,K, \alpha  \bigr)>2$ defined in (\ref{eqDefA0}) from Lemma~\ref{lmBasicIneq} depends only on $f$, $\CC$, $d$, $H$, and $\alpha$,  and the constant $C_0>1$ from Lemma~\ref{lmMetricDistortion} depends only on $f$, $\CC$, and $d$. Thus $C_{17}>1$ depends only on $f$, $\CC$, $d$, and $H$.

So by (\ref{eqLDiscontProperties_Split}) in Lemma~\ref{lmLDiscontProperties}, (\ref{eqBasicIneqN1}) in Lemma~\ref{lmBasicIneq} with $s \coloneqq 1$, (\ref{eqPfthmSpectralGap_HolderBound}), and (\ref{eqRDiscontTildeSupDecreasing}) in Lemma~\ref{lmRtildeNorm=1}, we get that for all $k\in\N$,
\begin{align}   \label{eqPfthmSpectralGap_JK}
&            \HseminormBig{\alpha,\, (Y^0,d)}{    \RR_{\wt{\phi},Y^0,X^0_\b}^{(k+j)}( \overline{u}_\b )   +  \RR_{\wt{\phi},Y^0,X^0_\w}^{(k+j)}( \overline{u}_\w ) }  \notag\\
&\qquad\leq  \sum\limits_{X^0_1\in\X^0}   \HseminormBig{ \alpha,\, (Y^0,d)}{     \RR_{\wt{\phi},Y^0,X^0_1}^{(k)} \Bigl( \RR_{\wt{\phi},X^0_1,X^0_\b}^{(j)}( \overline{u}_\b )  +     \RR_{\wt{\phi},X^0_1,X^0_\w}^{(j)}( \overline{u}_\w ) \Bigr)  }  \notag\\
&\qquad\leq  \sum\limits_{X^0_1\in\X^0}       \frac{C_0}{\Lambda^{\alpha k}}   \HseminormBig{ \alpha,\, (X^0_1,d)}{    \RR_{\wt{\phi},X^0_1,X^0_\b}^{(j)}( \overline{u}_\b )  +     \RR_{\wt{\phi},X^0_1,X^0_\w}^{(j)}( \overline{u}_\w )   }     \\
&\qquad\quad  + \sum\limits_{X^0_1\in\X^0}    A_0        \RR_{\wt{\phi},Y^0,X^0_1}^{(k)} \Bigl( \AbsBig{ \RR_{\wt{\phi},X^0_1,X^0_\b}^{(j)}( \overline{u}_\b )  +     \RR_{\wt{\phi},X^0_1,X^0_\w}^{(j)}( \overline{u}_\w ) }   \Bigr) \notag \\
&\qquad\leq  \frac{2C_0 C_{17}}{\Lambda^{\alpha k}}   +  A_0  \sum\limits_{X^0_1\in\X^0}  \NormBig{ \RR_{\wt{\phi},X^0_1,X^0_\b}^{(j)}( \overline{u}_\b )  +     \RR_{\wt{\phi},X^0_1,X^0_\w}^{(j)}( \overline{u}_\w )}_{\CCC^0(X^0_1)}.  \notag
\end{align}

By Theorem~\ref{thmLDiscontTildeUniformConv}, we can choose $N_0\in \N$ with the property that for each $j\in \N$ with $j\geq N_0$, each $Y^0\in\X^0$, each $\phi\in H$, and each pair of functions $u_\b\in H_\b$ and $u_\w\in H_\w$, we have
\begin{equation}   \label{eqPfthmSpectralGap_N0_1}
\frac{2C_0 C_{17}}{\Lambda^{\alpha j}} \leq \frac18,
\end{equation}
and
\begin{equation}    \label{eqPfthmSpectralGap_N0_2}
 ( 1+ A_0 )    \NormBig{ \RR_{\wt{\phi},Y^0 ,X^0_\b}^{(j)}( \overline{u}_\b )  +     \RR_{\wt{\phi},Y^0,X^0_\w}^{(j)}( \overline{u}_\w )}_{\CCC^0(Y^0)}    \leq \frac18.
\end{equation}
We fix $N_0\in\N$ to be the smallest integer with this property. So $N_0$ depends only on $f$, $\CC$, $d$, $\alpha$, and $H$.

For each $m\in \N$, each $\c\in\{\b,\w\}$, each $\phi\in H$, and each pair of functions $u_\b\in H_\b$ and $u_\w\in H_\w$, we denote
\begin{equation}    \label{eqPfthmSpectralGap_v_mc}
v_{m,\c}   \coloneqq    \RR_{\wt{\phi}, X^0_\c,X^0_\b}^{(2N_0 m)} (\overline{u}_\b)  + \RR_{\wt{\phi}, X^0_\c,X^0_\w}^{(2N_0 m)} (\overline{u}_\w).
\end{equation}
Then by (\ref{eqPfthmSpectralGap_JK}), (\ref{eqPfthmSpectralGap_N0_1}), and (\ref{eqPfthmSpectralGap_N0_2}), the function $v_{m,\c} \in \Holder{\alpha} \bigl(X^0_\c,d \bigr)$ satisfies
\begin{equation*}
\Hnorm{\alpha}{v_{m,\c}}{(X^0_\c,d)}   \leq \frac38.
\end{equation*}
So $2 v_{m,\c} \in H_\c$. We also note that by Lemma~\ref{lmRDiscontTildeDual},
\begin{align*}  
        \sum\limits_{\c\in\{\b,\w\}}  \int_{X^0_\c} \! v_{m,\c} \,\mathrm{d}\mu_\phi 
 = & \sum\limits_{\c\in\{\b,\w\}} \biggl( \int_{X^0_\c} \! \RR_{\wt{\phi},X^0_\c,X^0_\b}^{(2N_0 m)} (\overline{u}_\b)   \,\mathrm{d}\mu_\phi 
                                     +\int_{X^0_\c} \! \RR_{\wt{\phi},X^0_\c,X^0_\w}^{(2N_0 m)} (\overline{u}_\w)   \,\mathrm{d}\mu_\phi \biggr)  \\
 = & \int_{X^0_\b}  \! \overline{u}_\b \,\mathrm{d}\mu_\phi    +     \int_{X^0_\w}  \! \overline{u}_\w \,\mathrm{d}\mu_\phi 
 =   0.                                    
\end{align*}

\smallskip

Next, we prove by induction that for each $m\in\N$, each $\phi\in H$, and each pair of functions $u_\b\in H_\b$ and $u_\w\in H_\w$, we have
\begin{equation}   \label{eqPfthmSpectralGap_Induction}
\max\Bigl\{   \Hnorm{\alpha}{v_{m,\b}}{(X^0_\b,d)} ,  \, \Hnorm{\alpha}{v_{m,\w}}{(X^0_\w,d)}   \Bigr\}  \leq 3 \Bigl(\frac12\Bigr)^m.
\end{equation}

We have already shown that (\ref{eqPfthmSpectralGap_Induction}) holds for $m=1$. 

Assume that (\ref{eqPfthmSpectralGap_Induction}) holds for $m=j$ for some $j\in\N$, then $2^j v_{j,\b} \in H_\b$ and $2^j v_{j,\w} \in H_\w$. By (\ref{eqLDiscontProperties_Split}) in Lemma~\ref{lmLDiscontProperties}, for each $\c\in\{\b,\w\}$, we have
\begin{equation*}
2^j v_{j+1,\c}  = \RR_{\wt{\phi}, X^0_\c, X^0_\b}^{(2N_0)} \bigl(2^j v_{j,\b}\bigr)  +  \RR_{\wt{\phi}, X^0_\c, X^0_\w}^{(2N_0)} \bigl(2^j v_{j,\w}\bigr).
\end{equation*}
Thus $\Hnormbig{\alpha}{   2^j v_{j+1,\c}}{(X^0_\c,d)}  \leq \frac38 < \frac12$. So  $\Hnormbig{\alpha}{  v_{j+1,\c}}{(X^0_\c,d)} \leq     \bigl(\frac12\bigr)^{j+1} < 3 \bigl(\frac12\bigr)^{j+1}$.

The induction is now complete.

\smallskip

Then by (\ref{eqLDiscontProperties_Split}) in Lemma~\ref{lmLDiscontProperties}, (\ref{eqSplitRuelleTildeSupDecreasing}) in Lemma~\ref{lmRtildeNorm=1}, (\ref{eqPfthmSpectralGap_v_mc}), and (\ref{eqPfthmSpectralGap_Induction}), we get that for each $j\in \N$, each $m\in\N_0$, each $X^0\in \X^0$, each $\phi\in H$, and each pair of functions $u_\b\in H_\b$ and $u_\w\in H_\w$, the following holds:
\begin{align*}
&           \NormBig{  \RR_{\wt{\phi}, X^0, X^0_\b}^{(j+ 2N_0 m)} (\overline{u}_\b) +  \RR_{\wt{\phi}, X^0, X^0_\w}^{(j+ 2N_0 m)} (\overline{u}_\w)  }_{\CCC^0(X^0)}      \\
&\qquad =   \NormBigg{ \sum\limits_{\c\in\{\b,\w\}} \RR_{\wt{\phi}, X^0, X^0_\c}^{(j)} \Bigl(\RR_{\wt{\phi},X^0_\c, X^0_\b}^{(2N_0 m)} (\overline{u}_\b) +    \RR_{\wt{\phi}, X^0_\c, X^0_\w}^{(2N_0 m)} (\overline{u}_\w) \Bigr)  }_{\CCC^0(X^0)}    \\
&\qquad\leq  \max \biggl\{   \NormBig{    \RR_{\wt{\phi},X^0_\c, X^0_\b}^{(2N_0 m)} (\overline{u}_\b) +    \RR_{\wt{\phi}, X^0_\c, X^0_\w}^{(2N_0 m)} (\overline{u}_\w)    }_{\CCC^0(X^0_\c)}   \,\bigg|\,  \c\in\{\b,\w\}   \biggr\}   \\
&\qquad\leq  3   \Bigl(\frac12\Bigr)^m.
\end{align*}

Hence for each $n\in\N_0$,
\begin{equation}    \label{eqPfthmSpectralGap_Half}
 \NormBig{  \RR_{\wt{\phi}, X^0, X^0_\b}^{(n)} (\overline{u}_\b) +  \RR_{\wt{\phi}, X^0, X^0_\w}^{(n)} (\overline{u}_\w)  }_{\CCC^0(X^0)}
\leq 3  \Bigl(\frac12\Bigr)^{ \lfloor \frac{n}{2N_0} \rfloor}
\leq 6 \rho_1^n,
\end{equation}
where the constant 
\begin{equation*}
\rho_1 \coloneqq 2^{-\frac{1}{2N_0}}
\end{equation*}
depends only on $f$, $\CC$, $d$, $\alpha$, and $H$.

Finally, we consider the general case. For each pair of functions $w_\b \in \Holder{\alpha} \bigl(X^0_\b,d \bigr)$ and $w_\w \in \Holder{\alpha} \bigl(X^0_\w,d \bigr)$, we denote
\begin{equation*}
M  \coloneqq \max\Bigl\{    \Hnorm{\alpha}{w_\b}{(X^0_\b,d)}, \, \Hnorm{\alpha}{w_\w}{(X^0_\w,d)}  \Bigr\},
\end{equation*}
\begin{equation*}
\overline{w}_\b \coloneqq w_\b - \int_{X^0_\b}\! w_\b\,\mathrm{d}\mu_\phi - \int_{X^0_\w}\! w_\w\,\mathrm{d}\mu_\phi, \qquad 
\overline{w}_\w \coloneqq w_\w - \int_{X^0_\b}\! w_\b\,\mathrm{d}\mu_\phi - \int_{X^0_\w}\! w_\w\,\mathrm{d}\mu_\phi.
\end{equation*}
Let $u_\b \coloneqq \frac{1}{M}w_\b$ and $u_\w \coloneqq \frac{1}{M}w_\w$. Then clearly $u_\b\in H_\b$, $u_\w\in H_\w$, $\overline{u}_\b = \frac{1}{M}\overline{w}_\b$, and $\overline{u}_\w = \frac{1}{M}\overline{w}_\w$. Therefore, by (\ref{eqPfthmSpectralGap_Half}), for each $n\in\N_0$, each $\phi\in H$, each $X^0\in \X^0$,
\begin{equation*}
 \NormBig{ \RR_{\wt{\phi}, X^0, X^0_\b}^{(n)}  \Bigl(\frac{\overline{w}_\b}{M}  \Bigr)    +       \RR_{\wt{\phi}, X^0, X^0_\w}^{(n)} \Bigl(\frac{\overline{w}_\w}{M} \Bigr)}_{\CCC^0(X^0)}
 \leq 6 \rho_1^n,
\end{equation*}
i.e.,
\begin{equation*}
 \NormBig{ \RR_{\wt{\phi}, X^0, X^0_\b}^{(n)}  (\overline{w}_\b )    +       \RR_{\wt{\phi}, X^0, X^0_\w}^{(n)} (\overline{w}_\w)}_{\CCC^0(X^0)}
 \leq 6 \rho_1^n M.
\end{equation*}
This completes the proof.
\end{proof}

\begin{rem}   \label{rmSpectralGap}
For $\phi\in \Holder{\alpha}(S^2, d)$, the existence of the spectral gap for the split Ruelle operator $\RRR_{\wt{\phi}}$  on 
$
\Holder{\alpha}\bigl(X^0_\b,d\bigr) \times \Holder{\alpha}\bigl(X^0_\w,d\bigr)
$
follows immediately from (\ref{eqSplitRuelleCoordinateFormula}) in Lemma~\ref{lmSplitRuelleCoordinateFormula}, Theorem~\ref{thmSpectralGap}, and Lemma~\ref{lmBasicIneq}~(ii).
\end{rem}

Finally, we establish the following lemma that will be used in Section~\ref{sctPreDolgopyat}.

\begin{lemma}    \label{lmSplitRuelleUnifBoundNormalizedNorm}
Let $f$, $\CC$, $d$, $\alpha$, $\phi$, $s_0$ satisfy the Assumptions. Assume in addition $f(\CC)\subseteq\CC$. Then for all  $n\in\N$ and $s\in\C$ satisfying $\abs{\Re(s)} \leq 2 s_0$ and $\abs{\Im(s)}\geq 1$, we have
\begin{equation}     \label{eqSplitRuelleUnifBoundNormalizedNorm}
\NOpHnormD{\alpha}{\Im(s)}{\RRR_{\wt{s\phi}}^n}  \leq 4 A_0,
\end{equation}
and more generally,
\begin{equation}   \label{eqSplitRuelleUnifBoundNormalizedNorm_m}
\NHnormBig{\alpha}{\Im(s)}{ \Bigl( \RR_{\wt{s\phi}, X^0,X^0_\b}^{(n)} (u_\b) + \RR_{\wt{s\phi}, X^0,X^0_\w}^{(n)} (u_\w)  \Bigr)^m }{(X^0,d)}  \leq  (3m+1)A_0
\end{equation}
for each $m\in\N$, each $X^0\in\X^0(f,\CC)$, and each pair of complex-valued H\"{o}lder continuous functions $u_\b \in \Holder{\alpha} \bigl( \bigl(X^0_\b, d \bigr),\C \bigr)$ and $u_\w \in \Holder{\alpha} \bigl( \bigl(X^0_\w, d \bigr),\C \bigr)$ satisfying 
\begin{equation*}
\NHnorm{\alpha}{\Im(s)}{u_\b}{(X^0_\b,d)} \leq 1 \quad \mbox{ and } \quad \NHnorm{\alpha}{\Im(s)}{u_\w}{(X^0_\w,d)}  \leq 1.
\end{equation*}
Here $A_0=A_0\bigl(f,\CC,d,\Hseminorm{\alpha,\, (S^2,d)}{\phi},\alpha\bigr)\geq 2C_0>2$ is a constant from Lemma~\ref{lmBasicIneq} depending only on $f$, $\CC$, $d$, $\Hseminorm{\alpha,\, (S^2,d)}{\phi}$, and $\alpha$, and $C_0 > 1$ is a constant depending only on $f$, $\CC$, and $d$ from Lemma~\ref{lmMetricDistortion}.
\end{lemma}

\begin{proof}
Fix $n,m\in\N$, $\c\in\{\b,\w\}$, and $s=a+\I b$ with $a,b\in\R$ satisfying $\abs{a}\leq 2 s_0$ and $\abs{b}\geq 1$.

Choose an arbitrary pair of functions $u_\b \in \Holder{\alpha} \bigl( \bigl(X^0_\b, d \bigr),\C \bigr)$ and $u_\w \in \Holder{\alpha} \bigl( \bigl(X^0_\w, d \bigr),\C \bigr)$ satisfying 
\begin{equation*}
\NHnorm{\alpha}{b}{u_\b}{(X^0_\b,d)} \leq 1 \quad \mbox{ and } \quad \NHnorm{\alpha}{b}{u_\w}{(X^0_\w,d)} \leq 1. 
\end{equation*}
We denote $M \coloneqq \NormBig{ \RR_{\wt{s\phi}, X^0_\c,X^0_\b}^{(n)} (u_\b) + \RR_{\wt{s\phi}, X^0_\c,X^0_\w}^{(n)} (u_\w) }_{\CCC^0(X^0_\c)}$. By (\ref{eqSplitRuelleTildeSupDecreasing}) in Lemma~\ref{lmRtildeNorm=1}, we have $M\leq 1$. 

We then observe that for each H\"{o}lder continuous function $v\in \Holder{\alpha} \bigl( \bigl(X,d_0 \bigr),\C \bigr)$ on a compact metric space $(X,d_0)$, we have 
$\Hseminorm{\alpha,\, (X,d_0)}{v^m} \leq m \norm{v}_{\CCC^0(X)}^{m-1}\Hseminorm{\alpha,\, (X,d_0)}{v^m}$. Thus we get from (\ref{eqBasicIneqN1}) in Lemma~\ref{lmBasicIneq}, (\ref{eqRDiscontTildeSupDecreasing}) in Lemma~\ref{lmRtildeNorm=1}, and the above observation that
\begin{align*}
&            \NHnormBig{\alpha}{b}{ \Bigl( \RR_{\wt{s\phi}, X^0_\c,X^0_\b}^{(n)} (u_\b) + \RR_{\wt{s\phi}, X^0_\c,X^0_\w}^{(n)} (u_\w) \Bigr)^m }{(X^0_\c,d)}  \\
&\qquad=     M^m + \frac{1}{\abs{b}} \HseminormBig{\alpha,\, (X^0_\c,d)}{ \Bigl( \RR_{\wt{s\phi}, X^0_\c,X^0_\b}^{(n)} (u_\b) + \RR_{\wt{s\phi}, X^0_\c,X^0_\w}^{(n)} (u_\w) \Bigr)^m}  \\  
&\qquad\leq  1 + \frac{mM^{m-1}}{\abs{b}} \HseminormBig{\alpha,\, (X^0_\c,d)}{   \RR_{\wt{s\phi}, X^0_\c,X^0_\b}^{(n)} (u_\b) + \RR_{\wt{s\phi}, X^0_\c,X^0_\w}^{(n)} (u_\w)  } \\      
&\qquad\leq  1 +   m   \frac{C_0}{\Lambda^{\alpha n}} \Bigl(  \NHnorm{\alpha}{b}{u_\b}{(X^0_\b,d)}  +   \NHnorm{\alpha}{b}{u_\w}{(X^0_\w,d)} \Bigr) \\
&\qquad\quad   + m A_0\biggl( \NormBig{\RR_{\wt{a\phi},X^0_\c,X^0_\b}^{(n)} (\abs{u_\b})  }_{\CCC^0(X^0_\c)}                   +\NormBig{\RR_{\wt{a\phi},X^0_\c,X^0_\w}^{(n)} (\abs{u_\w})  }_{\CCC^0(X^0_\c)} \biggr)    \\
&\qquad\leq  1 +   2mC_0 +  m A_0 \bigl( \norm{u_\b}_{\CCC^0(X^0_\b)}  +   \norm{u_\w}_{\CCC^0(X^0_\w)}  \bigr) \\
&\qquad\leq  1+ 2mC_0 + 2mA_0 
                \leq  (3m+1)A_0,
\end{align*}
where $C_0 > 1$ is a constant depending only on $f$, $\CC$, and $d$ from Lemma~\ref{lmMetricDistortion}, and the last inequality follows from the fact that $A_0\geq 2C_0>2$ (see Lemma~\ref{lmBasicIneq}).

The inequality (\ref{eqSplitRuelleUnifBoundNormalizedNorm_m}) is now established, and (\ref{eqSplitRuelleUnifBoundNormalizedNorm}) follows from (\ref{eqSplitRuelleOpNormQuot}) in Lemma~\ref{lmSplitRuelleCoordinateFormula} and (\ref{eqSplitRuelleUnifBoundNormalizedNorm_m}).
\end{proof}

\section{Bound the zeta function with the operator norm}   \label{sctPreDolgopyat}

In this section, we bound the dynamical zeta function $\zeta_{\sigma_{A_{\ti}},\,\minus\phi\circsmall\pi_{\ti}}$ using some bounds of the operator norm of $\RRR_{\minus s\phi}$, for an expanding Thurston map $f$ with some forward invariant Jordan curve $\CC$ and an eventually positive real-valued H\"{o}lder continuous potential $\phi$.

Subsection~\ref{subsctRuelleEstimate} focuses on Proposition~\ref{propTelescoping}, which provides a bound of the dynamical zeta function $\zeta_{\sigma_{A_{\ti}},\,\minus\phi\circsmall\pi_{\ti}}$ for the symbolic system $\bigl(\Sigma_{A_{\ti}}^+, \sigma_{A_{\ti}}\bigr)$ asscociated to $f$ in terms of the operator norms of $\RRR_{\minus s\phi}^n$, $n\in\N$ and $s\in\C$ in some vertical strip with $\abs{\Im(s)}$ large enough. The idea of the proof originated from D.~Ruelle \cite{Rue90}. In Subsection~\ref{subsctOperatorNorm}, we establish in Theorem~\ref{thmOpHolderNorm} an exponential decay bound on the operator norm $\OpHnormD{\alpha}{\RRR_{\minus s\phi}^n}$ of $\RRR_{\minus s\phi}^n$, $n\in\N$, assuming the bound stated in Theorem~\ref{thmL2Shrinking}. Theorem~\ref{thmL2Shrinking} will be proved at the end of Subsection~\ref{subsctCancellation}. Combining the bounds in Proposition~\ref{propTelescoping} and Theorem~\ref{thmOpHolderNorm}, we give a proof of Theorem~\ref{thmZetaAnalExt_SFT} in Subsection~\ref{subsctProofthmZetaAnalExt_SFT}. Finally in Subsection~\ref{subsctProofthmLogDerivative}, we deduce Theorem~\ref{thmLogDerivative} from Theorem~\ref{thmZetaAnalExt_InvC} following the ideas from \cite{PoSh98} using basic complex analysis.

\subsection{Ruelle's estimate}     \label{subsctRuelleEstimate}

\begin{prop}  \label{propTelescoping}
Let $f$, $\CC$, $d$, $\Lambda$, $\alpha$, $\phi$, $s_0$ satisfy the Assumptions. We assume in addition that $f(\CC) \subseteq \CC$ and no $1$-tile in $\X^1(f,\CC)$ joins opposite sides of $\CC$. Let $\bigl( \Sigma_{A_{\ti}}^+,\sigma_{A_{\ti}} \bigr)$ be the one-sided subshift of finite type associated to $f$ and $\CC$ defined in Proposition~\ref{propTileSFT}, and let $\pi_{\ti}\: \Sigma_{A_{\ti}}^+\rightarrow S^2$ be defined in (\ref{eqDefTileSFTFactorMap}). 

Then for each $\delta>0$ there exists a constant $D_\delta>0$ such that for all integers $n\geq 2$ and $k\in\N$, we have
\begin{align} \label{eqSumHnormSplitRuelleOn1}
         &  \sum\limits_{X^k\in\X^k(f,\CC)} \max\limits_{X^0\in\X^0(f,\CC)} \Hnormbig{\alpha}{\RR_{\minus s\phi,X^0,X^k}^{(k)} (\mathbbm{1}_{X^k})}{(X^0,d)}    \\
 \leq &  D_\delta \abs{ \Im(s) } \Lambda^{-\alpha}  \exp(k(\delta + P(f,-\Re(s)\phi)))   \notag
\end{align}
and
\begin{align} \label{eqTelescoping}
     &  \Absbigg{ Z_{\sigma_{A_{\ti}},\,\minus \phi \circsmall \pi_{\ti}}^{(n)} (s) - \sum\limits_{X^0\in\X^0(f,\CC)}\sum\limits_{\substack{X^1\in\X^1(f,\CC)\\X^1\subseteq X^0}} \RR_{\minus s\phi,X^0,X^1}^{(n)}(\mathbbm{1}_{X^1})(x_{X^1})   }  \\
\leq & D_\delta \abs{\Im(s)} \sum\limits_{m=2}^{n} \OpHnormD{\alpha}{\RRR_{\minus s\phi}^{n-m}} \Bigl(\frac{1}{\Lambda^\alpha} \exp(\delta + P(f,-\Re(s)\phi))  \Bigr)^m     \notag
\end{align}
for any choice of a point $x_{X^1} \in \inte (X^1)$ for each $X^1\in\X^1(f,\CC)$, and for all $s\in\C$ with $\abs{\Im(s)} \geq 2 s_0 +1$ and $\abs{\Re(s)-s_0} \leq  s_0 $, where $Z_{\sigma_{A_{\ti}},\,\minus \phi \circsmall \pi_{\ti}}^{(n)} (s)$ is defined in (\ref{eqDefZn}).
\end{prop}

\begin{proof}
Fix the integer $n\geq 2$.

We first choose $x_{X^n} \in X^n$ for each $n$-tile $X^n \in \X^n$ in the following way. If $X^n\subseteq f^n(X^n)$, then let $x_{X^n}$ be the unique point in $X^n\cap P_{1,f^n}$ (see Lemma~\ref{lmAtLeast1} and Lemma~\ref{lmAtMost1}); otherwise $X^n$ must be a black $n$-tile contained in the white $0$-tile, or a white $n$-tile contained in the black $0$-tile, in which case we choose an arbitrary point $x_{X^n} \in \inte (X^n)$. Next, for each $i\in\N_0$ with $i\leq n-1$, and each $X^i\in\X^i$, we fix an arbitrary point $x_{X^i}\in \inte (X^i)$.

By (\ref{eqDefLc}) and our construction, we get that for all $s\in\C$, $X^0\in\X^0$, and $X^n\in\X^n$ with $X^n\subseteq X^0$,
\begin{equation}   \label{eqPfpropTelescoping_1}
\RR_{\minus s\phi,X^0,X^n}^{(n)} (\mathbbm{1}_{X^n})(x_{X^n}) =  \begin{cases} \exp(-sS_n\phi(x_{X^n})) & \text{if } X^n \subseteq f^n (X^n), \\ 0  & \text{otherwise}. \end{cases}
\end{equation}

It is easy to check that by (\ref{eqPfpropTelescoping_1}), the function $Z_{\sigma_{A_{\ti}},\,\minus \phi \circsmall \pi_{\ti}}^{(n)} (s)$ defined in (\ref{eqDefZn}) satisfies
\begin{equation}   \label{eqPfpropTelescoping_ZnRewrite}
Z_{\sigma_{A_{\ti}},\,\minus \phi \circsmall \pi_{\ti}}^{(n)} (s) = \sum\limits_{X^0\in\X^0} \sum\limits_{\substack{X^n\in\X^n\\X^n \subseteq X^0}}  \RR_{\minus s\phi,X^0,X^n}^{(n)} (\mathbbm{1}_{X^n}) (x_{X^n}).
\end{equation}

Thus by the triangle inequality, we get
\begin{align}   \label{eqPfpropTelescoping_SplitTileSum}
 &  \Absbigg{  Z_{\sigma_{A_{\ti}},\,\minus \phi \circsmall \pi_{\ti}}^{(n)} (s) - \sum\limits_{X^0\in\X^0} \sum\limits_{\substack{X^1\in\X^1\\X^1 \subseteq X^0}} \RR_{\minus s\phi,X^0,X^1}^{(n)} (\mathbbm{1}_{X^1}) (x_{X^1}) }\\
\leq & \sum\limits_{m=2}^{n}  \sum\limits_{X^0\in\X^0} \Absbigg{ \sum\limits_{\substack{X^{m-1}\in\X^{m-1}\\X^{m-1} \subseteq X^0}} \RR_{\minus s\phi,X^0,X^{m-1}}^{(n)} (\mathbbm{1}_{X^{m-1}}) (x_{X^{m-1}})  
  -  \sum\limits_{\substack{X^m\in\X^m\\X^m \subseteq X^0}} \RR_{\minus s\phi,X^0,X^m}^{(n)} (\mathbbm{1}_{X^m}) (x_{X^m})  } \notag \\
\leq & \sum\limits_{m=2}^{n}  \sum\limits_{X^0\in\X^0}  \sum\limits_{\substack{X^{m-1}\in\X^{m-1}\\X^{m-1} \subseteq X^0}} 
    \Absbigg{  \RR_{\minus s\phi,X^0,X^{m-1}}^{(n)} (\mathbbm{1}_{X^{m-1}}) (x_{X^{m-1}}) 
 -  \sum\limits_{\substack{X^m\in\X^m\\X^m \subseteq X^{m-1}}} \RR_{\minus s\phi,X^0,X^m}^{(n)} (\mathbbm{1}_{X^m}) (x_{X^m}) }.  \notag
\end{align}

Note that for all $s\in\C$, $2\leq m\leq n$, $\c\in\{\b,\w\}$, and $X^{m-1} \in \X^{m-1}$ with $X^{m-1}\subseteq X^0_\c$, by (\ref{eqDefLc}),
\begin{align}
    \RR_{\minus s\phi,X_\c^0,X^{m-1}}^{(n)} (\mathbbm{1}_{X^{m-1}}) (x_{X^{m-1}})  
= & \sum\limits_{\substack{X^n\in\X^n_\c\\X^n \subseteq X^{m-1}}} \exp\left( - sS_n\phi \left( (f^n|_{X^n})^{-1}(x_{X^{m-1}}) \right)\right)   \notag \\
= & \sum\limits_{\substack{X^m\in\X^m\\X^m \subseteq X^{m-1}}} \sum\limits_{\substack{X^n\in\X^n_\c\\X^n \subseteq X^m}}  
                                                                                         \exp\left( -sS_n\phi \left( (f^n|_{X^n})^{-1}(x_{X^{m-1}}) \right)\right)  \label{eqPfpropTelescoping_SplitTile} \\
= & \sum\limits_{\substack{X^m\in\X^m\\X^m \subseteq X^{m-1}}} \RR_{\minus s\phi,X^0_\c,X^m}^{(n)} (\mathbbm{1}_{X^m}) (x_{X^{m-1}}).  \notag
\end{align}

Hence by (\ref{eqPfpropTelescoping_SplitTileSum}), (\ref{eqPfpropTelescoping_SplitTile}), and (\ref{eqLDiscontProperties_Holder}), we get
\begin{align*}
&                     \Absbigg{  Z_{\sigma_{A_{\ti}},\,\minus \phi \circsmall \pi_{\ti}}^{(n)} (s) 
                                         - \sum\limits_{X^0\in\X^0} \sum\limits_{\substack{X^1\in\X^1\\X^1 \subseteq X^0}} \RR_{\minus s\phi,X^0,X^1}^{(n)} (\mathbbm{1}_{X^1}) (x_{X^1}) }\\
&\qquad \leq \sum\limits_{m=2}^{n}  \sum\limits_{X^0\in\X^0}  \sum\limits_{\substack{X^{m-1}\in\X^{m-1}\\X^{m-1} \subseteq X^0}}    \sum\limits_{\substack{X^m\in\X^m\\X^m \subseteq X^{m-1}}}   
                                       \Absbig{ \RR_{\minus s\phi,X^0,X^m}^{(n)} (\mathbbm{1}_{X^m}) (x_{X^{m-1}})    -  \RR_{\minus s\phi,X^0,X^m}^{(n)} (\mathbbm{1}_{X^m}) (x_{X^m}) }  \\
&\qquad \leq  \sum\limits_{m=2}^{n}  \sum\limits_{X^0\in\X^0}  \sum\limits_{\substack{X^{m-1}\in\X^{m-1}\\X^{m-1} \subseteq X^0}}    \sum\limits_{\substack{X^m\in\X^m\\X^m \subseteq X^{m-1}}}   
                                      \Hnormbig{\alpha}{\RR_{\minus s\phi,X^0,X^m}^{(n)} (\mathbbm{1}_{X^m})}{(X^0,d)}  d(x_{X^{m-1}},x_{X^m})^\alpha.
\end{align*}

Note that by (\ref{eqLDiscontProperties_Holder}),
\begin{equation*}
\RR_{\minus s\phi,X^0,X^m}^{(m)} (\mathbbm{1}_{X^m}) \in \Holder{\alpha} ((X^0,d),\C)
\end{equation*}
for $s\in\C$, $m\in\N$, $X^0\in\X^0$, $X^m\in\X^m$, and that by Lemma~\ref{lmCellBoundsBM}~(ii),
\begin{equation*}
d(x_{X^{m-1}},x_{X^m}) \leq \diam_d(X^{m-1}) \leq C\Lambda^{-m+1},
\end{equation*}
where $C\geq 1$ is a constant from Lemma~\ref{lmCellBoundsBM} depending only on $f$, $\CC$, and $d$. So by (\ref{eqLDiscontProperties_Split}) in Lemma~\ref{lmLDiscontProperties} and (\ref{eqSplitRuelleOpNormQuot}) in Lemma~\ref{lmSplitRuelleCoordinateFormula},
\begin{align}   \label{eqPfpropTelescoping_BoundZn-L}
&       \Absbigg{ Z_{\sigma_{A_{\ti}},\,\minus \phi \circsmall \pi_{\ti}}^{(n)} (s) 
                           - \sum\limits_{X^0\in\X^0} \sum\limits_{\substack{X^1\in\X^1\\X^1 \subseteq X^0}} \RR_{\minus s\phi,X^0,X^1}^{(n)} (\mathbbm{1}_{X^1}) (x_{X^1}) } \\
&\qquad  \leq   \sum\limits_{m=2}^{n}   \sum\limits_{ X^m\in\X^m }    \OpHnormD{\alpha}{\RRR_{\minus s\phi}^{n-m}}      
                               \Bigl(\max\limits_{Y^0\in\X^0} \Hnormbig{\alpha}{\RR_{\minus s\phi,Y^0,X^m}^{(m)} (\mathbbm{1}_{X^m})}{(Y^0,d)}  \Bigr)       C^\alpha\Lambda^{\alpha(1-m)}  .   \notag
\end{align}

\smallskip

We now give an upper bound for $\sum\limits_{X^m\in\X^m}  \max\limits_{Y^0\in\X^0} \Hnormbig{\alpha}{\RR_{\minus s\phi,Y^0,X^m}^{(m)} (\mathbbm{1}_{X^m})}{(Y^0,d)} $.

Fix an arbitrary point $y\in\CC\setminus \post f$.

Consider arbitrary $s\in\C$ with $\abs{\Re(s)-s_0}\leq s_0$, $m\in\N$, $X^m_\b\in\X^m_\b$, $X^m_\w\in\X^m_\w$, $x_\b,x'_\b \in X^0_\b$, $x_\w,x'_\w\in X^0_\w$, and $\c,\c'\in\{\b,\w\}$ with $\c\neq\c'$. By (\ref{eqDefLc}), Lemma~\ref{lmMetricDistortion}, Lemma~\ref{lmCellBoundsBM}~(ii), we have
\begin{equation}   \label{eqPfpropTelescoping_LDiscontSupNormBW}
\RR_{\minus s\phi,X^0_\c,X^m_{\c'}}^{(m)} \bigl( \mathbbm{1}_{X^m_{\c'}} \bigr)  (x_\c) = 0,
\end{equation}
and 
\begin{align}   \label{eqPfpropTelescoping_LDiscontSupNormBB}
&           \Abs{\RR_{\minus s\phi,X^0_\c,X^m_\c}^{(m)}  (\mathbbm{1}_{X^m_\c} ) (x_\c)} \notag \\
&\qquad =   \Abs{\exp\bigl( -sS_m\phi \bigl(  (f^m|_{X^m_\c} )^{-1}(x_\c)\bigr)  \bigr)}  \notag \\
&\qquad =   \exp\bigl( -\Re(s)S_m\phi \bigl(  (f^m|_{X^m_\c} )^{-1}(y)\bigr)  \bigr)   \frac{  \exp  ( -\Re(s)S_m\phi   (  (f^m|_{X^m_\c} )^{-1}(x_\c)   ) )} {\exp ( -\Re(s) S_m\phi  (  (f^m|_{X^m_\c} )^{-1}(y) ) )  }  \\
&\qquad\leq \exp\bigl( -\Re(s)S_m\phi \bigl(  (f^m|_{X^m_\c} )^{-1}(y)\bigr)  \bigr)   \exp\left( \Re(s) C_1 \left( \diam_d\left(X^0_\c\right)\right)^\alpha  \right)  \notag\\
&\qquad\leq \exp\bigl( -\Re(s)S_m\phi \bigl(  (f^m|_{X^m_\c} )^{-1}(y)\bigr)  \bigr)   \exp\left(\Re(s)C^\alpha C_1\right),   \notag
\end{align}
where $C_1>0$ is a constant from Lemma~\ref{lmSnPhiBound} depending only on $f$, $\CC$, $d$, $\phi$, and $\alpha$.

Hence by (\ref{eqPfpropTelescoping_LDiscontSupNormBW}) and (\ref{eqPfpropTelescoping_LDiscontSupNormBB}), we get 
\begin{equation}   \label{eqPfpropTelescoping_LDiscontSupNorm}
      \Norm{ \RR_{\minus s\phi,X^0,X^m }^{(m)} (\mathbbm{1}_{X^m }) }_{\CCC^0(X^0)}   
\leq  \exp \bigl( -\Re(s)S_m\phi  \bigl(  (f^m|_{X^m} )^{-1}(y)  \bigr) \bigr)  \exp\left(\Re(s)C^\alpha C_1\right)  
\end{equation}
for $s\in\C$, $m\in\N$, $X^0\in\X^0$, and $X^m\in\X^m$.

By (\ref{eqDefLc}),
\begin{equation} \label{eqPfpropTelescoping_LDiscontHseminormBW}
\RR_{\minus s\phi,X^0_\c, X^m_{\c'}}^{(m)} \bigl( \mathbbm{1}_{X^m_{\c'}} \bigr) (x_\c) - \RR_{\minus s\phi,X^0_\c,X^m_{\c'} }^{(m)} \bigl(\mathbbm{1}_{X^m_{\c'}} \bigr) (x'_\c) = 0.
\end{equation}

By (\ref{eqDefLc}) and Lemma~\ref{lmExpToLiniear} with $T\coloneqq  2 s_0  \Hseminorm{\alpha,\, (S^2,d)}{\phi}$,
\begin{align*}
&                      \Abs{1-  \frac{\RR_{\minus s\phi,X^0_\c,X^m_\c}^{(m)}  (\mathbbm{1}_{X^m_\c} ) (x_\c) }{\RR_{\minus s\phi,X^0_\c,X^m_\c}^{(m)}  (\mathbbm{1}_{X^m_\c} ) (x'_\c) }  }    \\
&\qquad   =    \Abs{1-   \exp  \bigl( -s \bigl(S_m \phi  \bigl(  (f^m|_{X^m_\c} )^{-1}(x_\c)  \bigr) \bigr) -    S_m   \phi  \bigl(  (f^m|_{X^m_\c} )^{-1}(x'_\c)  \bigr) \bigr)   \bigr) \bigr)  } \\
&\qquad \leq  C_{10} \Hseminorm{\alpha,\, (S^2,d)}{s\phi} d(x_\c,x'_\c)^\alpha 
                   =     C_{10} \abs{s} \Hseminorm{\alpha,\, (S^2,d)}{\phi} d(x_\c,x'_\c)^\alpha,
\end{align*}
where the constant $C_{10}=C_{10}(f,\CC,d,\alpha,T)>1$ depends only on $f$, $\CC$, $d$, $\alpha$, and $\phi$ in our context.

Thus by (\ref{eqPfpropTelescoping_LDiscontSupNormBB}),
\begin{align*}
&           \Absbig{ \RR_{\minus s\phi,X^0_\c,X^m_\c}^{(m)}  (\mathbbm{1}_{X^m_\c} ) (x_\c) - \RR_{\minus s\phi,X^0_\c,X^m_\c}^{(m)}  (\mathbbm{1}_{X^m_\c} ) (x'_\c)  }  \\
&\qquad\leq  \Abs{1-  \frac{\RR_{\minus s\phi,X^0_\c,X^m_\c}^{(m)}  (\mathbbm{1}_{X^m_\c} ) (x_\c) }{\RR_{\minus s\phi,X^0_\c,X^m_\c}^{(m)}  (\mathbbm{1}_{X^m_\c} ) (x'_\c) }  }
                         \Absbig{\RR_{\minus s\phi,X^0_\c,X^m_\c}^{(m)}  (\mathbbm{1}_{X^m_\c} ) (x'_\c) }  \\
&\qquad\leq  4^{-1} C_{11} \abs{s} d(x_\c,x'_\c)^\alpha  \exp\bigl( -\Re(s)S_m\phi \bigl(  (f^m|_{X^m_\c} )^{-1}(y)\bigr)  \bigr)   ,
\end{align*}
where we define the constant 
\begin{equation}  \label{eqConst_propTelescoping1}
C_{11} \coloneqq   \max \bigl\{ 2 , \, 4   C_{10}  \Hseminorm{\alpha,\, (S^2,d)}{\phi}  \bigr\}   \exp\left( 2 s_0  C^\alpha C_1\right) 
\end{equation}
depending only on $f$, $\CC$, $d$, $\alpha$, and $\phi$.

So we get
\begin{equation} \label{eqPfpropTelescoping_LDiscontHseminormBB}
      \Hseminormbig{\alpha,\, (X^0_\c,d )}{\RR_{\minus s\phi,X^0_\c,X^m_\c}^{(m)} ( \mathbbm{1}_{X^m_\c} )}  
 \leq  4^{-1} C_{11} \abs{s} \exp\bigl( -\Re(s)S_m\phi \bigl(  (f^m|_{X^m_\c} )^{-1}(y)\bigr)  \bigr)   .  
\end{equation}

Thus by (\ref{eqPfpropTelescoping_LDiscontHseminormBW}) and (\ref{eqPfpropTelescoping_LDiscontHseminormBB}), we have
\begin{equation} \label{eqPfpropTelescoping_LDiscontHseminorm}
      \Hseminormbig{\alpha,\, (X^0,d )}{\RR_{\minus s\phi,X^0,X^m}^{(m)}  ( \mathbbm{1}_{X^m} )}  
 \leq  4^{-1} C_{11} \abs{s} \exp\bigl( -\Re(s)S_m\phi \bigl(  (f^m|_{X^m} )^{-1}(y)\bigr)  \bigr)  ,  
\end{equation}
for $s\in\C$ with $\abs{\Re(s)-s_0} \leq  s_0$, $m\in\N$, $X^0\in\X^0$, and $X^m\in\X^m$.

Hence by (\ref{eqPfpropTelescoping_LDiscontSupNorm}) and (\ref{eqPfpropTelescoping_LDiscontHseminorm}), for all $m\in\N$, $X^m\in\X^m$, $s\in\C$, and $Y^0\in\X^0$ satisfying 
\begin{equation*}
\abs{\Im(s)} \geq 2 s_0 + 1  \text{ and } \abs{\Re(s)-s_0} \leq  s_0,
\end{equation*} 
we have
\begin{equation}  \label{eqPfpropTelescoping_LDiscontHNorm}
\Hnormbig{\alpha}{\RR_{\minus s\phi,Y^0,X^m}^{(m)} (\mathbbm{1}_{X^m})}{(Y^0,d)}    \leq C_{11} \abs{\Im(s)} \exp\bigl(-\Re(s) S_m\phi \bigl(  (f^m|_{X^m} )^{-1}(y)\bigr)  \bigr).
\end{equation}

So by (\ref{eqPfpropTelescoping_LDiscontHNorm}) and the fact that $y\in \CC$, we get
\begin{align}   \label{eqPfpropTelescoping_SumHnormD=L}
&         \sum\limits_{X^m\in\X^m}   \max\limits_{Y^0\in\X^0} \Hnormbig{\alpha}{\RR_{\minus s\phi,Y^0,X^m}^{(m)} (\mathbbm{1}_{X^m})}{(Y^0,d)}  \notag \\
&\qquad=   C_{11} \Im(s)   \sum\limits_{X^m\in\X^m} \exp\bigl( -\Re(s)S_m\phi \bigl(  (f^m|_{X^m} )^{-1}(y)\bigr)  \bigr)  \\
&\qquad=   2C_{11} \Im(s)  \RR_{\minus \Re(s)\phi}^m  ( \mathbbm{1}_{S^2}  ) (y)   \notag.
\end{align}

We construct a sequence of continuous functions $p_m \: \R\rightarrow\R$, $m\in\N$, as
\begin{equation}   \label{eqDefp_m}
p_m(a) \coloneqq   \bigl( \RR_{\minus a\phi}^m  ( \mathbbm{1}_{S^2}  ) (y) \bigr)^{\frac{1}{m}}.
\end{equation}
By Corollary~\ref{corRR^nConvToTopPressureUniform}, the function $a \mapsto p_m(a) - e^{P(f, -a\phi)}$ converges to $0$ as $m$ tends to $+\infty$, uniformly in $a\in \{c\in\R \,|\, \abs{c-s_0} \leq  s_0 \}$. Recall that $a\mapsto P(f, - a\phi)$ is continuous in $a\in\R$ (see for example, \cite[Theorem~3.6.1]{PrU10}). Thus by (\ref{eqPfpropTelescoping_SumHnormD=L}), there exists a constant $C_{12}>0$ depending only on $f$, $\CC$, $d$, $\alpha$, $\phi$, and $\delta$ such that for all $m\in\N$ and $s\in\C$ with $\abs{\Im(s)}\geq 2 s_0 +1$ and $\abs{\Re(s)-s_0}\leq s_0$, 
\begin{align}
&                       \sum\limits_{X^m\in\X^m} \max\limits_{Y^0\in\X^0} \Hnormbig{\alpha}{\RR_{\minus s\phi,Y^0,X^m}^{(m)} (\mathbbm{1}_{X^m})}{(Y^0,d)}    \label{eqPfpropTelescoping_SumHnormSplitRuelleOn1} \\
&\qquad  =       2C_{11} \Im(s)  (p_m(\Re(s)))^m \leq   C_{12} \abs{\Im(s)}  e^{m(\delta + P(f,-\Re(s)\phi))} .   \notag 
\end{align}

Combining (\ref{eqPfpropTelescoping_BoundZn-L}) with the above inequality, we get for all $s\in\C$ with $\abs{\Im(s)}\geq 2 s_0 +1$ and $\abs{\Re(s)-s_0}\leq  s_0$, 
\begin{align*}
&                    \Absbigg{  Z_{\sigma_{A_{\ti}},\,\minus \phi \circsmall \pi_{\ti}}^{(n)} (s)
                                         - \sum\limits_{X^0\in\X^0} \sum\limits_{\substack{X^1\in\X^1\\X^1 \subseteq X^0}} \RR_{\minus s\phi,X^0, X^1}^{(n)} (\mathbbm{1}_{X^1}) (x_{X^1}) }\\
&\qquad \leq   D_\delta \abs{\Im(s)} \sum\limits_{m=2}^{n} \OpHnormD{\alpha}{\RRR_{\minus s\phi}^{n-m}} \Bigl(\frac{1}{\Lambda^\alpha} \exp(\delta + P(f, -\Re(s)\phi))  \Bigr)^m,    \notag
\end{align*}
where $D_\delta \coloneqq C^\alpha C_{12} \Lambda^\alpha > C_{12} >0$ is a constant depending only on $f$, $\CC$, $d$, $\phi$, $\alpha$, and $\delta$.

Inequality (\ref{eqSumHnormSplitRuelleOn1}) now follows from (\ref{eqPfpropTelescoping_SumHnormSplitRuelleOn1}) and $D_\delta  \Lambda^{-\alpha} \geq C_{12}$.
\end{proof}

\subsection{Operator norm}    \label{subsctOperatorNorm}

The following theorem is one of the main estimates we need to prove in this paper. 

\begin{theorem}  \label{thmL2Shrinking}
Let $f\: S^2 \rightarrow S^2$ be an expanding Thurston map with a Jordan curve $\CC\subseteq S^2$ satisfying $f(\CC) \subseteq \CC$ and $\post f\subseteq \CC$. Let $d$ be a visual metric on $S^2$ for $f$ with expansion factor $\Lambda>1$, and $\phi\in \Holder{\alpha}(S^2,d)$ be an eventually positive real-valued H\"{o}lder continuous function with an exponent $\alpha\in(0,1]$ that satisfies the $\alpha$-strong non-integrability condition. Let $s_0\in\R$ be the unique positive real number satisfying $P(f, -s_0\phi)=0$. 

Then there exist constants $\iota \in \N$, $a_0 \in (0, s_0 ]$, $b_0 \in [ 2 s_0 +1,+\infty)$, and $\rho \in (0,1)$ such that for each $\c\in\{\b,\w\}$, each $n\in\N$, each $s\in\C$ with $\abs{\Re(s)-s_0} \leq a_0$ and $\abs{\Im(s)} \geq b_0$, and each pair of functions $u_\b \in \Holder{\alpha}\bigl(\bigl(X^0_\b,d\bigr),\C\bigr)$ and $u_\w \in \Holder{\alpha}\bigl(\bigl(X^0_\w,d\bigr),\C\bigr)$ satisfying $\NHnorm{\alpha}{\Im(s)}{u_\b}{(X^0_\b,d)} \leq 1$ and $\NHnorm{\alpha}{\Im(s)}{u_\w}{(X^0_\w,d)} \leq 1$, we have 
\begin{equation}
\int_{X^0_\c} \! \AbsBig{ \RR_{\wt{\minus s\phi}, X^0_\c, X^0_\b}^{(n \iota)} (u_\b) + \RR_{\wt{\minus s\phi}, X^0_\c, X^0_\w}^{(n\iota)} (u_\w)  }^2   \,\mathrm{d}\mu_{\minus s_0\phi} \leq \rho^n.
\end{equation}
Here $\mu_{\minus s_0\phi}$ denotes the unique equilibrium state for the map $f$ and the potential $-s_0\phi$.
\end{theorem}

We will prove the above theorem at the end of Section~\ref{sctDolgopyat}. Assuming Theorem~\ref{thmL2Shrinking}, we can establish the following theorem.

\begin{theorem}   \label{thmOpHolderNorm}
Let $f\: S^2 \rightarrow S^2$ be an expanding Thurston map with a Jordan curve $\CC\subseteq S^2$ satisfying $f(\CC) \subseteq \CC$ and $\post f\subseteq \CC$. Let $d$ be a visual metric on $S^2$ for $f$ with expansion factor $\Lambda>1$, and $\phi\in \Holder{\alpha}(S^2,d)$ be an eventually positive real-valued H\"{o}lder continuous function with an exponent $\alpha\in(0,1]$ that satisfies the $\alpha$-strong non-integrability condition. Let $s_0\in\R$ be the unique positive real number satisfying $P(f,-s_0\phi)=0$. 

Then there exists a constant $D' = D'(f,\CC,d,\alpha,\phi) >0$ such that for each $\epsilon>0$, there exist constants $\delta_\epsilon \in (0, s_0)$, $\wt{b}_\epsilon\geq 2s_0 + 1$, and $\rho_\epsilon \in (0,1)$ with the following property:

For each $n\in\N$ and all $s\in\C$ satisfying $\abs{\Re(s) - s_0} < \delta_\epsilon$ and $\abs{\Im(s)} \geq \wt{b}_\epsilon$, we have
\begin{equation}  \label{eqOpHolderNorm}
\OpHnormD{\alpha}{\RRR_{\minus s\phi}^n} \leq D'  \abs{\Im(s)}^{1+\epsilon} \rho_\epsilon^n.
\end{equation}
\end{theorem}

\begin{proof}
Fix an arbitrary number $s=a+\I b\in\C$ with $a,b\in\R$ satisfying $\abs{a-s_0}\leq a_0$ and $\abs{b}\geq b_0$, and fix an arbitrary pair of complex-valued H\"{o}lder continuous functions $u_\b \in \Holder{\alpha} \bigl( \bigl(X^0_\b,d \bigr),\C \bigr)$ and $u_\w \in \Holder{\alpha} \bigl( \bigl(X^0_\w,d \bigr),\C \bigr)$ satisfying $\NHnorm{\alpha}{b}{u_\b}{(X^0_\b,d)} \leq 1$ and $\NHnorm{\alpha}{b}{u_\w}{(X^0_\w,d)} \leq 1$. Here $a_0 \in (0, s_0 ]$ and $b_0 \in [ 2 s_0 +1,+\infty)$ are constants from Theorem~\ref{thmL2Shrinking} depending only on $f$, $\CC$, $d$, $\alpha$, and $\phi$.

We choose $\iota_0\in\N$ to be the smallest integer satisfying $\frac{1}{2\iota_0} < \epsilon$, $\iota_0\geq 2$, and $\frac{\iota_0}{\iota}\in\N$, where $\iota\in\N$ is a constant from Theorem~\ref{thmL2Shrinking} depending only on $f$, $\CC$, $d$, $\alpha$, and $\phi$.

We also choose $m\in\N$ to be the smallest integer satisfying
\begin{equation}   \label{eqPfthmOpHolderNorm_m}
m \iota_0 \bigl(-\log \max\bigl\{ \rho^{\frac{\iota_0}{\iota}},  \, \rho_1, \, \Lambda^{-\alpha}  \bigr\}  \bigr)  
\geq 2\log\abs{b} \geq 0,
\end{equation}
where $\rho \in (0,1)$ is a constant from Theorem~\ref{thmL2Shrinking}, and 
$
\rho_1 \coloneqq \rho_1 \bigl(f,\CC,d,\alpha,  H \bigr)\in (0,1),
$
with $H \coloneqq \bigl\{ \wt{- t\phi} \,\big|\, t \in \R, \, \abs{t-s_0}\leq a_0 \bigr\}$ a bounded subset of $\Holder{\alpha}(S^2,d)$, is a constant from Theorem~\ref{thmSpectralGap} depending only on $f$, $\CC$, $d$, and $\alpha$ in our context here.

We first note that by (\ref{eqDefLc}), the Cauchy--Schwartz inequality, Lemma~\ref{lmRtildeNorm=1}, (\ref{eqSpectralGapAll}) in Theorem~\ref{thmSpectralGap}, Theorem~\ref{thmL2Shrinking}, and (\ref{eqSplitRuelleUnifBoundNormalizedNorm_m}) in Lemma~\ref{lmSplitRuelleUnifBoundNormalizedNorm}, we have for each $X^0\in\X^0$ and each $x\in X^0$,
\begin{align*}
     & \Biggl(  \sum\limits_{X^0_1\in \X^0}  \RR_{\wt{\minus a\phi}, X^0, X^0_1}^{(m\iota_0)}
         \Bigl( \AbsBig{ \RR_{\wt{\minus s\phi}, X^0_1, X^0_\b}^{(m\iota_0)} (u_\b) 
                        +\RR_{\wt{\minus s\phi}, X^0_1, X^0_\w}^{(m\iota_0)} (u_\w)}\Bigr)(x)   \Biggr)^2\\
\leq & \Biggl(   \sum\limits_{X^0_1\in\X^0}  \RR_{\wt{\minus a\phi}, X^0, X^0_1}^{(m\iota_0)} 
         \biggl( \AbsBig{ \RR_{\wt{\minus s\phi}, X^0_1, X^0_\b}^{(m\iota_0)} (u_\b) 
                        +\RR_{\wt{\minus s\phi}, X^0_1, X^0_\w}^{(m\iota_0)} (u_\w)}^2\biggr)(x)   \Biggr)
       \Biggl(   \sum\limits_{X^0_2\in\X^0}  \RR_{\wt{\minus a\phi}, X^0, X^0_2}^{(m\iota_0)} 
         \bigl(  \mathbbm{1}_{X^0_2} \bigr)(x)   \Biggr)\\
=    &     \sum\limits_{X^0_1\in\X^0}  \RR_{\wt{\minus a\phi}, X^0, X^0_1}^{(m\iota_0)} 
         \biggl( \AbsBig{ \RR_{\wt{\minus s\phi}, X^0_1, X^0_\b}^{(m\iota_0)} (u_\b) 
                        +\RR_{\wt{\minus s\phi}, X^0_1, X^0_\w}^{(m\iota_0)} (u_\w)}^2\biggr)(x)   \\
\leq & \sum\limits_{X^0_1\in\X^0}  \int_{X^0_1} \AbsBig{ 
              \RR_{\wt{\minus s\phi}, X^0_1, X^0_\b}^{(m\iota_0)} (u_\b) 
             +\RR_{\wt{\minus s\phi}, X^0_1, X^0_\w}^{(m\iota_0)} (u_\w)}^2  \,\mathrm{d}\mu_{\minus s_0 \phi} \\
     &   +  6 \rho_1^{m\iota_0} \max \biggl\{ \Hnormbigg{\alpha}{ \AbsBig{ 
              \RR_{\wt{\minus s\phi}, X^0_2, X^0_\b}^{(m\iota_0)} (u_\b) 
             +\RR_{\wt{\minus s\phi}, X^0_2, X^0_\w}^{(m\iota_0)} (u_\w)}^2 }{(X^0_2, d)}
        \,\bigg|\, X^0_2\in\X^0   \biggr\} \\
\leq & 2 \rho^{\frac{\iota_0}{\iota} m} + 42 A_0 \rho_1^{m\iota_0} \abs{b},
\end{align*}
where $A_0=A_0\bigl(f,\CC,d,\Hseminorm{\alpha,\, (S^2,d)}{\phi}, \alpha\bigr)>2$ is a constant from Lemma~\ref{lmBasicIneq} depending only on $f$, $\CC$, $d$, $\Hseminorm{\alpha,\, (S^2,d)}{\phi}$, and $\alpha$. Combining with (\ref{eqPfthmOpHolderNorm_m}) and the fact that $\iota_0\geq 2$ and $A_0>2$, we get
\begin{align}    \label{eqPfthmOpHolderNorm_CommonSup}
&               \Normbigg{  \sum\limits_{X^0_1\in \X^0}  \RR_{\wt{\minus a\phi}, X^0, X^0_1}^{(m\iota_0)}
                     \Bigl( \AbsBig{ \RR_{\wt{\minus s\phi}, X^0_1, X^0_\b}^{(m\iota_0)} (u_\b) 
                                             +\RR_{\wt{\minus s\phi}, X^0_1, X^0_\w}^{(m\iota_0)} (u_\w)}\Bigr) }_{\CCC^0(X^0)}\notag \\
&\qquad  \leq  \Bigl( 2\abs{b}^{-\frac{2}{\iota_0}}  + 42 A_0 \abs{b}^{-2+1} \Bigr)^{\frac12} 
                  \leq 7A_0 \abs{b}^{-\frac{1}{\iota_0}},                        
\end{align} 
for each $X^0\in\X^0$.

Thus by (\ref{eqDefProjections}), (\ref{eqSplitRuelleCoordinateFormula}), (\ref{eqLDiscontProperties_Split}), and (\ref{eqPfthmOpHolderNorm_CommonSup}), we get that for each $\c\in\{\b,\w\}$,
\begin{align}    \label{eqPfthmOpHolderNorm_Sup}
&            \NormBig{  \pi_\c \Bigl(  \RRR_{\wt{\minus s\phi}}^{2m\iota_0}  (u_\b,u_\w)  \Bigr) }_{\CCC^0(X^0_\c)} \notag\\
&\qquad =    \Normbigg{\sum\limits_{X^0_1\in \X^0}  \RR_{\wt{\minus s\phi}, X^0_\c, X^0_1}^{(m\iota_0)}
         \Bigl(  \RR_{\wt{\minus s\phi}, X^0_1, X^0_\b}^{(m\iota_0)} (u_\b) 
                        +\RR_{\wt{\minus s\phi}, X^0_1, X^0_\w}^{(m\iota_0)} (u_\w) \Bigr)}_{\CCC^0(X^0_\c)}  \notag \\  
&\qquad\leq \Normbigg{\sum\limits_{X^0_1\in \X^0}  \RR_{\wt{\minus a\phi}, X^0_\c, X^0_1}^{(m\iota_0)}
         \Bigl( \AbsBig{ \RR_{\wt{\minus s\phi}, X^0_1, X^0_\b}^{(m\iota_0)} (u_\b) 
                        +\RR_{\wt{\minus s\phi}, X^0_1, X^0_\w}^{(m\iota_0)} (u_\w)}\Bigr)}_{\CCC^0(X^0_\c)} \\
&\qquad\leq  7A_0 \abs{b}^{-\frac{1}{\iota_0}}.   \notag
\end{align}

By (\ref{eqDefProjections}), (\ref{eqSplitRuelleCoordinateFormula}), (\ref{eqBasicIneqN1}) in Lemma~\ref{lmBasicIneq}, Lemma~\ref{lmSplitRuelleUnifBoundNormalizedNorm}, (\ref{eqPfthmOpHolderNorm_CommonSup}), and (\ref{eqPfthmOpHolderNorm_m}), we have for each $\c\in\{\b,\w\}$,
\begin{align}   \label{eqPfthmOpHolderNorm_HolderSeminorm}
&            \frac{1}{\abs{b}}  \HseminormBig{\alpha,\, (X^0_\c,d)}{  \pi_\c  \Bigl( \RRR_{\wt{\minus s\phi}}^{2m\iota_0}  (u_\b,u_\w)  \Bigr) }   \notag\\
&\qquad =     \frac{1}{\abs{b}}  \Hseminormbigg{\alpha,\, (X^0_\c,d)}{ \sum\limits_{X^0_1\in\X^0}  \RR_{\wt{\minus s\phi},X^0_\c,X^0_1}^{(m\iota_0)}
             \Bigl( \RR_{\wt{\minus s\phi},X^0_1,X^0_\b}^{(m\iota_0)}(u_\b)
                   +\RR_{\wt{\minus s\phi},X^0_1,X^0_\w}^{(m\iota_0)}(u_\w)   \Bigr)  }  \notag \\
&\qquad\leq  \sum\limits_{X^0_1\in\X^0}   \frac{C_0}{\Lambda^{\alpha m \iota_0}} \NHnormBig{\alpha}{b}{\RR_{\wt{\minus s\phi},X^0_1,X^0_\b}^{(m \iota_0)} (u_\b)
+\RR_{\wt{\minus s\phi},X^0_1,X^0_\w}^{(m \iota_0)} (u_\w) }{(X^0_1,d)}   \\
&\qquad\quad   + \sum\limits_{X^0_1\in\X^0}    A_0\NormBig{ \RR_{\wt{\minus a\phi}, X^0_\c,X^0_1}^{(m \iota_0)}  \Bigl(   \AbsBig{ \RR_{\wt{\minus s\phi},X^0_1,X^0_\b}^{(m \iota_0)} (u_\b)
            +\RR_{\wt{\minus s\phi},X^0_1,X^0_\w}^{(m \iota_0)} (u_\w)}    \Bigr)}_{\CCC^0(X^0_\c)}   \notag\\
&\qquad\leq  8 A_0 \frac{C_0}{\Lambda^{\alpha m \iota_0}}   + A_0 \Bigl(7A_0 \abs{b}^{-\frac{1}{\iota_0}} \Bigr)   
                \leq   7A_0^2 \abs{b}^{-2} + 7 A_0^2 \abs{b}^{-\frac{1}{\iota_0}} 
                \leq  14A_0^2 \abs{b}^{-\frac{1}{\iota_0}}, \notag
\end{align}
where $C_0 > 1$ is a constant depending only on $f$, $\CC$, and $d$ from Lemma~\ref{lmMetricDistortion}, and $A_0 \geq C_0 >2$ (see Lemma~\ref{lmBasicIneq}).

Hence for each $n\in\N$, by choosing $k\in\N_0$ and $r\in\{0,1,\dots,2 m \iota_0 - 1\}$ with $n=2 m \iota_0 k +r$, we get from (\ref{eqPfthmOpHolderNorm_Sup}), (\ref{eqPfthmOpHolderNorm_HolderSeminorm}), Definition~\ref{defOpHNormDiscont}, and (\ref{eqSplitRuelleUnifBoundNormalizedNorm}) in Lemma~\ref{lmSplitRuelleUnifBoundNormalizedNorm} that if $\abs{b} \geq \wt{b}_\epsilon$, where
\begin{equation}    \label{eqDefwtB_eps}
\wt{b}_\epsilon \coloneqq \max \bigl\{  (21 A_0^2 )^{2\iota_0},  \,  2s_0 + 1 \bigr\} >1,
\end{equation}
then
\begin{align}    \label{eqSplitRuelleTildeOpHnormBound}
      \OpHnormD{\alpha}{ \RRR_{\wt{\minus s\phi}}^n}  
\leq  &  \abs{b} \NOpHnormD{\alpha}{b}{  \RRR_{\wt{\minus s\phi}}^{2 m \iota_0 k + r} } 
\leq     \abs{b} \biggl( \NOpHnormD{\alpha}{b}{  \RRR_{\wt{\minus s\phi}}^{2 m \iota_0} } \biggr)^k 
           \, \NOpHnormD{\alpha}{b}{  \RRR_{\wt{\minus s\phi}}^{r} }  \notag \\
\leq  & 4A_0 \abs{b} \Bigl(  7A_0 \abs{b}^{- \frac{1}{\iota_0}} +  14A_0^2 \abs{b}^{-\frac{1}{\iota_0}} \Bigr)^k  
\leq      4A_0  \abs{b}^{1-\frac{k}{2\iota_0}} \\
\leq  & 4A_0  \abs{b}^{1+ \frac{1}{2\iota_0} - \frac{2 m \iota_0 k + r}{2\iota_0}\frac{1}{2 m \iota_0} }  
\leq      4A_0  \abs{b}^{1+ \frac{1}{2\iota_0}} \abs{b}^{-\frac{n}{4 m\iota_0^2} } \notag \\
\leq  & 4A_0 \abs{b}^{1+\epsilon} \wt{b}_\epsilon^{-\frac{n}{4m \iota_0^2}}.  \notag
\end{align}

\smallskip

We now turn the upper bound for $\OpHnormD{\alpha}{ \RRR_{\wt{\minus s\phi}}^n}$ in (\ref{eqSplitRuelleTildeOpHnormBound}) into a bound for $\OpHnormD{\alpha}{ \RRR_{ \minus s\phi }^n}$.

Since $t\mapsto P(f,t\phi)$ is continuous in $t\in\R$ (see for example, \cite[Theorem~3.6.1]{PrU10}) and $P(f, - s_0 \phi)=0$, we can choose $\delta_\epsilon \in (0,  s_0 )$ sufficiently small so that if $a\in (s_0-\delta_\epsilon, s_0 + \delta_\epsilon)$, then
\begin{equation}  \label{eqPfthmOpHolderNorm_PressureBound}
\abs{P(f,-a\phi)} \leq \frac{1}{2m \iota_0} \log 21.
\end{equation}
By (\ref{eqSplitRuelleOpNormQuot}), (\ref{eqSplitRuelleTilde_NoTilde}), (\ref{eqBound_Uaphi_UpperLower}) in Lemma~\ref{lmBound_aPhi}, and Corollary~\ref{corProductInverseHolderNorm}, we get
\begin{align*}  
              \OpHnormD{\alpha}{\RRR_{\minus s \phi}^n}  
 =     & \sup \Biggl\{  \frac{   \Hnormbig{\alpha}{   \sum_{\c\in\{\b,\w\}} \RR_{ \minus s \phi, X^0, X^0_\c}^{(n)} (v_\c)     }{(X^0,d)}      }
                                           {  \max\bigl\{     \Hnorm{\alpha}{v_\c}{(X^0_\c,d)} \,\big|\, \c\in\{\b,\w\}  \bigr\}}   \Biggr\}  \\
\leq & e^{ n P(f, - a\phi) }     \Hnorm{\alpha}{u_{\minus a\phi}}{(S^2,d)}                          
            \sup \Biggl\{  \frac{   \Hnormbig{\alpha}{    \sum_{\c\in\{\b,\w\}}  \RR_{ \wt{\minus s \phi}, X^0, X^0_\c}^{(n)} \bigl( \frac{v_\c}{u_{\minus a\phi}} \bigr)   }{(X^0,d)}      }
                                           {  \max\bigl\{     \Hnorm{\alpha}{v_\c}{(X^0_\c,d)} \,\big|\, \c\in\{\b,\w\}  \bigr\}}   \Biggr\}  \\
\leq & e^{ n P(f, - a\phi) }     \Hnorm{\alpha}{u_{\minus a\phi}}{(S^2,d)}     \OpHnormD{\alpha}{\RRR_{\wt{\minus s \phi}}^n}  
            \sup \Biggl\{  \frac{ \max\bigl\{     \Hnormbig{\alpha}{ \frac{v_\c}{u_{\minus a\phi}}   }{(X^0_\c,d)} \,\big|\, \c\in\{\b,\w\}  \bigr\}   }
                                            { \max\bigl\{     \Hnorm{\alpha}{v_\c}{(X^0_\c,d)} \,\big|\, \c\in\{\b,\w\}  \bigr\}   }   \Biggr\}  \\    
\leq & e^{ n P(f, - a\phi) }     \Hnorm{\alpha}{u_{\minus a\phi}}{(S^2,d)}     \OpHnormD{\alpha}{\RRR_{\wt{\minus s \phi}}^n}  
                                                      \HnormBig{\alpha}{ \frac{ 1 }{u_{\minus a\phi}}   }{(S^2,d)}    \\                                              
\leq & e^{ n P(f, - a\phi) }     \Hnorm{\alpha}{u_{\minus a\phi}}{(S^2,d)}        \OpHnormD{\alpha}{\RRR_{\wt{\minus s \phi}}^n}
            \frac{1}{\exp(- 2 C_{15}) }    \bigl(   1+ \Hnorm{\alpha}{u_{\minus a\phi}}{(S^2,d)} \bigr)                  \\
\leq &  \OpHnormD{\alpha}{\RRR_{\wt{\minus s \phi}}^n}       e^{2 C_{15}}    
            \bigl(   1+ \Hnorm{\alpha}{u_{\minus a\phi}}{(S^2,d)} \bigr)^2    \exp( n P(f, - a\phi)  ),                        
\end{align*}
where the suprema are taken over all $v_\b\in\Holder{\alpha} \bigl( \bigl(X^0_\b,d \bigr),\C \bigr)$, $v_\w\in\Holder{\alpha} \bigl( \bigl(X^0_\w,d \bigr),\C \bigr)$, and $X^0\in \X^0$ with $\norm{v_\b}_{\CCC^0(X^0_\b)} \norm{v_\w}_{\CCC^0(X^0_\w)} \neq 0$. Here the constant $C_{15} = C_{15}(f,\CC,d,\alpha,T,K)$, with $T\coloneqq 2 s_0$ and $K\coloneqq \Hseminorm{\alpha,\, (S^2,d)}{\phi} > 0$, is defined in (\ref{eqConst_lmBound_aPhi1}) in Lemma~\ref{lmBound_aPhi} and depends only on $f$, $\CC$, $d$, $\alpha$, and $\Hseminorm{\alpha,\, (S^2,d)}{\phi}$ in our context.

Combining the above inequality with (\ref{eqSplitRuelleTildeOpHnormBound}), (\ref{eqDefwtB_eps}), (\ref{eqPfthmOpHolderNorm_PressureBound}), and (\ref{eqBound_Uaphi_Hnorm}) in Lemma~\ref{lmBound_aPhi}, we get that if $a\in (s_0 - \delta_\epsilon, s_0 + \delta_\epsilon)$ and $\abs{b} \geq \wt{b}_\epsilon$, then
\begin{equation*}
        \OpHnormD{\alpha}{\RRR_{\minus s\phi}^n} 
\leq   4A_0 \abs{b}^{1+\epsilon} (21 A_0^2)^{-\frac{n}{2m \iota_0}}  21^{\frac{n}{2m \iota_0}}  e^{2 C_{15}} 
      \bigl(   1+ \Hnorm{\alpha}{u_{\minus a\phi}}{(S^2,d)} \bigr)^2  
\leq D'   \abs{b}^{1+\epsilon}    \rho_\epsilon^n,
\end{equation*}
where $\rho_\epsilon \coloneqq A_0^{\frac{1}{m \iota_0}}  >1$ and
$
D' \coloneqq 4 A_0 e^{2 C_{15}}  \Bigl( 8 \frac{ s_0  \Hseminorm{\alpha,\, (S^2,d)}{\phi} C_0 }{1-\Lambda^{-\alpha}} L   +2 \Bigr)^2 \bigl(e^{2C_{15}}\bigr)^2>1.
$
Here $D'$ depends only on $f$, $\CC$, $d$, $\alpha$, and $\phi$.
\end{proof}

\subsection{Proof of Theorem~\ref{thmZetaAnalExt_SFT}}   \label{subsctProofthmZetaAnalExt_SFT}

Using Proposition~\ref{propTelescoping} and Theorem~\ref{thmOpHolderNorm}, we can get the following bound for the zeta function $\zeta_{\sigma_{A_{\ti}}, \, \minus \phi \circsmall \pi_{\ti} } $ (c.f.\ (\ref{eqDefZetaFn})).

\begin{prop}   \label{propHighFreq}
Let $f$, $\CC$, $d$, $\Lambda$, $\alpha$, $\phi$, $s_0$ satisfy the Assumptions. We assume in addition that $\phi$ satisfies the $\alpha$-strong non-integrability condition, and that $f(\CC) \subseteq \CC$ and no $1$-tile in $\X^1(f,\CC)$ joins opposite sides of $\CC$. Then for each $\epsilon>0$ there exist constants $\wt{C}_\epsilon > 0$ and $\wt{a}_\epsilon \in (0, s_0)$ such that
\begin{equation}   \label{eqLogZetaFnBdd}
     \Absbigg{  \sum\limits_{n=1}^{+\infty} \frac{1}{n}   Z_{\sigma_{A_{\ti}},\,\minus \phi \circsmall \pi_{\ti}}^{(n)} (s)  }
\leq  \wt{C}_\epsilon  \abs{ \Im(s) }^{2+\epsilon}
\end{equation}
for all $s\in\C$ with $\abs{\Re(s) - s_0}  < \wt{a}_\epsilon$ and $\abs{\Im(s)} \geq \wt{b}_\epsilon$, where $\wt{b}_\epsilon \geq 2s_0 + 1$ is a constant depending only on $f$, $\CC$, $d$, $\alpha$, $\phi$, and $\epsilon$ defined in Theorem~\ref{thmOpHolderNorm}.
\end{prop}

Recall $Z_{\sigma_{A_{\ti}},\,\minus \phi \circsmall \pi_{\ti}}^{(n)} (s)$ defined in (\ref{eqDefZn}).

\begin{proof}
Let $\delta \coloneqq \frac{1}{3} \log (\Lambda^\alpha) > 0.$ 

Since $t\mapsto P(f, -t\phi)$ is continuous on $\R$ (see for example, \cite[Theorem~3.6.1]{PrU10}), we fix $\wt{a}_\epsilon \in (0,\delta_\epsilon) \subseteq (0, s_0)$ such that $\abs{P(f, -t\phi)} < \frac{1}{3} \log(\Lambda^\alpha)$ for each $t\in\R$ with $\abs{t-s_0} < \wt{a}_\epsilon$, where $\delta_\epsilon \in (0, s_0)$ is a constant defined in Theorem~\ref{thmOpHolderNorm} depending only on $f$, $\CC$, $d$, $\alpha$, $\phi$, and $\epsilon$.

Fix an arbitrary point $x_{X^1} \in \inte (X^1)$ for each $X^1 \in \X^1$. By Lemma~\ref{lmLDiscontProperties}, Lemma~\ref{lmSplitRuelleCoordinateFormula}, and (\ref{eqSumHnormSplitRuelleOn1}) in Proposition~\ref{propTelescoping}, for each $n\geq 2$ and each $s\in \C$ with $\abs{ \Re(s) - s_0 } < \wt{a}_\epsilon$,  we have
\begin{align}   \label{eqPfpropHighFreq_SumIneq}
&                          \Absbigg{  \sum\limits_{X^0 \in \X^0}  \sum\limits_{\substack{X^1\in\X^1\\X^1 \subseteq X^0}}     \RR_{\minus s\phi,X^0,X^1}^{(n)} (\mathbbm{1}_{X^1}) (x_{X^1})  }  \notag \\
&\qquad   \leq    \sum\limits_{X^0 \in \X^0}  \sum\limits_{\substack{X^1\in\X^1\\X^1 \subseteq X^0}}    
                                 \Absbigg{   \sum\limits_{Y^0 \in \X^0}   \RR_{\minus s\phi,X^0,Y^0}^{(n-1)} \bigl( \RR_{\minus s\phi,Y^0,X^1}^{(1)}  (\mathbbm{1}_{X^1} ) \bigr) (x_{X^1})  } \\
&\qquad   \leq    \OpHnormD{\alpha}{\RRR_{\minus s \phi}^{n-1}}    
                              \sum\limits_{X^0 \in \X^0}  \sum\limits_{\substack{X^1\in\X^1\\X^1 \subseteq X^0}}  
                                       \max\limits_{Y^0\in\X^0} \Hnormbig{\alpha}{\RR_{\minus s\phi,Y^0,X^1}^{(1)} (\mathbbm{1}_{X^1})}{(Y^0,d)}         \notag      \\
&\qquad   \leq    \OpHnormD{\alpha}{\RRR_{\minus s \phi}^{n-1}}    D_\delta  \abs{\Im(s)}  \Lambda^{-\alpha} \exp( \delta + P(f, - \Re(s) \phi) ),                        \notag                            
\end{align}
where $D_\delta>0$ is a constant depending only on  $f$, $\CC$, $d$, $\alpha$, $\phi$, and $\delta$ from Proposition~\ref{propTelescoping}.

Hence by (\ref{eqDefZn}), Proposition~\ref{propTelescoping}, (\ref{eqPfpropHighFreq_SumIneq}), Theorem~\ref{thmOpHolderNorm}, and the choices of $\delta$ and $\wt{a}_\epsilon$ above, we get that for each $s\in\C$ with $\abs{ \Re(s) - s_0 } < \wt{a}_\epsilon$ and $\abs{ \Im(s) } \geq \wt{b}_\epsilon$, 
\begin{align*}
                       \sum\limits_{n=2}^{+\infty} \frac{1}{n} \AbsBig{  Z_{\sigma_{A_{\ti}},\,\minus \phi \circsmall \pi_{\ti}}^{(n)} (s)   }  
\leq & \sum\limits_{n=2}^{+\infty} \frac{1}{n} \biggl(                  
                                                             \Absbigg{               \sum\limits_{X^0 \in \X^0}  \sum\limits_{\substack{X^1\in\X^1\\X^1 \subseteq X^0}}     \RR_{\minus s\phi,X^0,X^1}^{(n)} (\mathbbm{1}_{X^1}) (x_{X^1})  }  \\
        &\qquad \quad    + \Absbigg{ Z_{\sigma_{A_{\ti}},\,\minus \phi \circsmall \pi_{\ti}}^{(n)} (s)
                                               - \sum\limits_{X^0 \in \X^0}  \sum\limits_{\substack{X^1\in\X^1\\X^1 \subseteq X^0}}     \RR_{\minus s\phi,X^0,X^1}^{(n)} (\mathbbm{1}_{X^1}) (x_{X^1})  }  \biggr) \\
\leq &  \sum\limits_{n=2}^{+\infty} \frac{1}{n} \biggl(     \OpHnormD{\alpha}{\RRR_{\minus s \phi}^{n-1}}    D_\delta \abs{ \Im(s) } \Lambda^{-\frac{\alpha}{3}} 
                                                                                                        +   D_\delta \abs{\Im(s)} \sum\limits_{m=2}^{n} \OpHnormD{\alpha}{\RRR_{\minus s\phi}^{n-m}} \Lambda^{-\frac{m\alpha}{3}}   \biggr)   \\
\leq & \abs{\Im(s)}^{2+\epsilon} \sum\limits_{n=2}^{+\infty} \frac{D'}{n}  D_\delta \sum\limits_{m=1}^{n} \rho_\epsilon^{n-m} \Lambda^{-\frac{m\alpha}{3}}     \\
\leq & D' D_\delta \abs{\Im(s)}^{2+\epsilon} \sum\limits_{n=2}^{+\infty}  \wt{\rho}_\epsilon^n 
                 \leq         \frac{    D' D_\delta  } { 1 -  \wt{\rho}_\epsilon }                 \abs{\Im(s)}^{2+\epsilon},                                                                                                                                                  
\end{align*}
where the constant $\wt{\rho}_\epsilon \coloneqq \max \bigl\{ \rho_\epsilon, \, \Lambda^{-\frac{1}{3}\alpha}  \bigr\} < 1$ depends only on $f$, $\CC$, $d$, $\alpha$, $\phi$, and $\epsilon$. Here constants $D'\in (0,s_0)$ and $\rho_\epsilon\in(0,1)$ are from Theorem~\ref{thmOpHolderNorm} depending only on $f$, $\CC$, $d$, $\alpha$, $\phi$, and $\epsilon$.

Therefore, by (\ref{eqDefZn}) and Proposition~\ref{propSFT}~(i), we have
\begin{equation*}
          \Absbigg{  \sum\limits_{n=1}^{+\infty} \frac{1}{n}   Z_{\sigma_{A_{\ti}},\,\minus \phi \circsmall \pi_{\ti}}^{(n)} (s)      }        
\leq  \AbsBig{  Z_{\sigma_{A_{\ti}},\,\minus \phi \circsmall \pi_{\ti}}^{(1)} (s)  }    +   \sum\limits_{n=2}^{+\infty} \frac{1}{n} \AbsBig{  Z_{\sigma_{A_{\ti}},\,\minus \phi \circsmall \pi_{\ti}}^{(n)} (s)   }    
\leq   \wt{C}_\epsilon  \abs{ \Im(s) }^{2+\epsilon}
\end{equation*}
for all $s\in\C$ with $\abs{\Re(s) - s_0}  < \wt{a}_\epsilon$ and $\abs{\Im(s)} \geq \wt{b}_\epsilon$, where the constant
\begin{equation*}
\wt{C}_\epsilon \coloneqq  D' D_\delta ( 1 -  \wt{\rho}_\epsilon )^{-1}    + 2 \deg f \exp \bigl(2 s_0 \norm{\phi}_{\CCC^0(S^2)} \bigr)
\end{equation*}
depends only on $f$, $\CC$, $d$, $\alpha$, $\phi$, and $\epsilon$.
\end{proof}

It follows immediately from the above proposition that $\zeta_{\sigma_{A_{\ti}}, \, \minus \phi \circsmall \pi_{\ti} } (s)$ has a non-vanishing holomorphic extension across the vertical line $\Re(s) = s_0$ for high frequency. In order to verify Theorem~\ref{thmZetaAnalExt_SFT}, we just need to establish the holomorphic extension for low frequency.


\begin{proof}[Proof of Theorem~\ref{thmZetaAnalExt_SFT}]
In this proof, for $s\in\C$ and $r\in\R$, $B(s,r)$ denotes the open ball in $\C$.

For an arbitrary number $t\in\R$, by Proposition~\ref{propZetaFnConv_s0}~(i), we have 
\begin{equation}  \label{eqPfthmZetaAnalExt_SFT_Unique0}
P (\sigma_{A_{\ti}}, - t \phi \circ \pi_{\ti} ) = 0 \quad \text{ if and only if } \quad   t = s_0.
\end{equation}

Fix an arbitrary number $\theta \in (0,1)$. By \cite[Theorem~4.5, Propositions~4.6, 4.7, and 4.8]{PP90} and the discussion preceding them in \cite{PP90}, the exponential of the topological pressure $\exp  (P( \sigma_{A_{\ti}}, \cdot ))$ as a function on $\Holder{1} \bigl( \Sigma_{A_{\ti}}^+, d_\theta \bigr)$ can be extended to a new function (still denoted by $\exp( P( \sigma_{A_{\ti}}, \cdot))$) with the following properties:
\begin{enumerate}
\smallskip
\item[(1)] The domain $\operatorname{dom} (\exp ( P ( \sigma_{A_{\ti}}, \cdot )))$ of $\exp (P ( \sigma_{A_{\ti}}, \cdot ))$ is an nonempty open subset of $\Holder{1} \bigl( \bigl( \Sigma_{A_{\ti}}^+, d_\theta \bigr), \C\bigr)$.

\smallskip
\item[(2)] The function $s \mapsto \exp( P  ( \sigma_{A_{\ti}}, - s \phi \circ \pi_{\ti}  ))$ is a holomorphic map from an open neighborhood $U \subseteq \C$ of $s$ to $\C$ if  $- s \phi \circ \pi_{\ti}  \in \operatorname{dom} (\exp (P ( \sigma_{A_{\ti}}, \cdot )))$ .

\smallskip
\item[(3)] If $\psi \in \operatorname{dom} ( \exp (P ( \sigma_{A_{\ti}}, \cdot )))$ and $\eta \coloneqq \psi + c + 2\pi\I M + u - u\circ \sigma_{A_{\ti}}$ for some $c\in\C$, $M \in \CCC \bigl( \Sigma_{A_{\ti}}^+, \Z \bigr)$, and $u\in \Holder{1} \bigl( \bigl( \Sigma_{A_{\ti}}^+, d_\theta \bigr), \C\bigr)$, then $\eta \in \operatorname{dom} (\exp ( P ( \sigma_{A_{\ti}}, \cdot )))$ and $\exp (P(  \sigma_{A_{\ti}}, \eta)) = e^c \exp (P(  \sigma_{A_{\ti}}, \psi) )$.
\end{enumerate}

\smallskip

We first show that $s_0$ is not an accumulation point of zeros of the function $s\mapsto 1-\exp ( P ( \sigma_{A_{\ti}}, -s \phi \circ \pi_{\ti} ))$. We argue by contradiction and assume otherwise. Then by Property~(2) above, $\exp ( P ( \sigma_{A_{\ti}}, -s \phi \circ \pi_{\ti} )) = 1$ for all $s$ in a neighborhood of $s_0$. This contradicts with (\ref{eqPfthmZetaAnalExt_SFT_Unique0}).

Thus by \cite[Theorem~5.5~(ii) and Theorem~5.6~(b), (c)]{PP90}, we can choose $\vartheta_0 >0$ small enough such that  $\zeta_{\sigma_{A_{\ti}}, \, \minus \phi \circsmall \pi_{\ti} } (s)$ has a non-vanishing holomorphic extension 
\begin{equation}  \label{eqPfthmZetaAnalExt_SFT_Extension}
\zeta_{\sigma_{A_{\ti}}, \, \minus \phi \circsmall \pi_{\ti} } (s) 
=   \frac{ \exp \biggl( \sum\limits_{n=1}^{+\infty}  \frac1n   \sum\limits_{ \underline{x}  \in P_{1,\sigma_{A_{\ti}}^n}}  \bigl(   \exp \bigl( -s S_n^{\sigma_{A_{\ti}}} (\phi \circ \pi_{\ti} )( \underline{x} ) \bigr) 
                                                                                                                                                                                            - \exp( n P( \sigma_{A_{\ti}}, - s \phi \circ \pi_{\ti} ) )  \bigr)    \biggr) }
             { 1 - \exp ( P( \sigma_{A_{\ti}}, - s \phi \circ \pi_{\ti} ) )}     
\end{equation}
to $B(s_0, \vartheta_0 ) \setminus \{s_0\}$, and $\zeta_{\sigma_{A_{\ti}}, \, \minus \phi \circsmall \pi_{\ti} } (s)$ has a pole at $s=s_0$. Moreover, the numerator on the right-hand side of (\ref{eqPfthmZetaAnalExt_SFT_Extension}) is a non-vanishing holomorphic function on $B(s_0, \vartheta_0 )$.

Next, we show that $\zeta_{\sigma_{A_{\ti}}, \, \minus \phi \circsmall \pi_{\ti} } (s)$ has a simple pole at $s=s_0$. It suffices to show that $1 - \exp ( P( \sigma_{A_{\ti}}, -s \phi \circ \pi_{\ti} ) )$ has a simple zero at $s=s_0$. Indeed, since $\phi$ is eventually positive, we fix $m\in\N$ such that $S_m^f \phi $ is strictly positive on $S^2$ (see Definition~\ref{defEventuallyPositive}). By Proposition~\ref{propZetaFnConv_s0}~(i), Theorem~\ref{thmEquilibriumState}~(ii), and the fact that the equilibrium state $\mu_{\minus t\phi}$ for $f$ and $-t\phi$ is an $f$-invariant probability measure (see Theorem~\ref{thmEquilibriumState}~(i) and Subsection~\ref{subsctThermodynFormalism}), we have for $t\in \R$,
\begin{align}   \label{eqPfthmZetaAnalExt_SFT_Simple0}
        \frac{\mathrm{d}}{\mathrm{d}t} ( 1 - \exp ( P( \sigma_{A_{\ti}}, -t \phi \circ \pi_{\ti} ) )   ) 
= &  \frac{\mathrm{d}}{\mathrm{d}t} \bigl( 1 - e^{ P(f, -t \phi  )  } \bigr)    
=     - e^{ P(f, -t \phi  )  }   \frac{\mathrm{d}}{\mathrm{d}t} P(f,- t\phi)  \notag \\
= &e^{ P(f, -t \phi  )  } \int \!\phi \,\mathrm{d}\mu_{\minus t\phi} 
=  \frac{e^{ P(f, -t \phi  )  }}{m} \int \! S_m^f \phi \,\mathrm{d}\mu_{\minus t\phi} 
> 0.
\end{align}  
Hence by (\ref{eqPfthmZetaAnalExt_SFT_Simple0}) and Property~(2) above, we get that $\zeta_{\sigma_{A_{\ti}}, \, \minus \phi \circsmall \pi_{\ti} } (s)$ has a simple pole at $s=s_0$.

We now show that for each $b\in \R \setminus \{ 0 \}$, there exists $\vartheta_b>0$ such that  $\zeta_{\sigma_{A_{\ti}}, \, \minus \phi \circsmall \pi_{\ti} } (s)$ has a non-vanishing holomorphic extension to $B (s_0+ \I b, \vartheta_b)$.

By \cite[Theorem~5.5~(ii) and Theorem~5.6]{PP90}, and the fact that $\operatorname{dom} ( \exp ( P ( \sigma_{A_{\ti}}, \cdot )))$ is open and $\exp ( P( \sigma_{A_{\ti}}, \cdot ))$ is continuous on $\operatorname{dom} (\exp (P ( \sigma_{A_{\ti}}, \cdot )))$ (see Properties~(2) and (3) above), we get that for each $b\in \bigl[ - \wt{b}_\epsilon, \wt{b}_\epsilon \bigr] \setminus \{ 0 \}$, we can always choose  $\vartheta_b>0$ such that $\zeta_{\sigma_{A_{\ti}}, \, \minus \phi \circsmall \pi_{\ti} } (s)$ has a non-vanishing holomorphic extension to $B (s_0+ \I b, \vartheta_b)$ unless the following two conditions are both satisfied:
\begin{enumerate}
\smallskip
\item[(i)] $- (s_0 + \I b) \phi \circ \pi_{\ti} = - s_0 \phi \circ \pi_{\ti} + \I c + 2\pi\I M + u - u\circ \sigma_{A_{\ti}} \in \operatorname{dom} ( \exp ( P ( \sigma_{A_{\ti}}, \cdot )))$ for some $c\in \C$, $M \in \CCC \bigl(\Sigma_{A_{\ti}}^+, \Z \bigr)$ and  $u\in \Holder{1} \bigl( \bigl( \Sigma_{A_{\ti}}^+, d_\theta \bigr), \C \bigr)$.

\smallskip
\item[(ii)] $1- \exp ( P( \sigma_{A_{\ti}}, - (s_0 + \I b) \phi \circ \pi_{\ti} ) ) = 0$.
\end{enumerate}

We will show that conditions~(i) and (ii) cannot be both satisfied. We argue by contradiction and assume that conditions~(i) and (ii) are both satisfied. Then by Property~(3) above, $c \equiv 0 \pmod{2\pi}$. Thus by taking the imaginary part of both sides of the identity in condition~(i), we get that $\phi\circ \pi_{\ti} = K M + \tau - \tau \circ \sigma_{A_{\ti}}$ for some $K\in\R$, $M \in \CCC \bigl( \Sigma_{A_{\ti}}^+ , \Z \bigr)$, and $\tau \in \Holder{1} \bigl( \bigl( \Sigma_{A_{\ti}}^+, d_\theta \bigr), \C \bigr)$. Then by Theorem~\ref{thmNLI}, $\phi$ is co-homologous to a constant in $\CCC( S^2, \R)$, a contradiction, establishing statement~(i) in Theorem~\ref{thmZetaAnalExt_SFT}.

To verify statement~(ii) in Theorem~\ref{thmZetaAnalExt_SFT}, we assume in addition that $\phi$ satisfies the $\alpha$-strong non-integrability condition.

Fix an arbitrary $\epsilon > 0$. Let $\wt{C}_\epsilon > 0$ and $\wt{a}_\epsilon \in (0,s_0)$ be constants from Proposition~\ref{propHighFreq}, and $\wt{b}_\epsilon \geq 2 s_0 + 1$ be a constant from Theorem~\ref{thmOpHolderNorm}, all of which depend only on $f$, $\CC$, $d$, $\alpha$, $\phi$, and $\epsilon$. The inequality (\ref{eqZetaBound_SFT}) follows immediately from (\ref{eqLogZetaFnBdd}) in Proposition~\ref{propHighFreq}.

Therefore, by the compactness of $\bigl[ - \wt{b}_\epsilon, \wt{b}_\epsilon \bigr]$, we can choose $\wt{\epsilon}_0 \in (0, \wt{a}_\epsilon) \subseteq(0, s_0)$ small enough such that $\zeta_{\sigma_{A_{\ti}}, \, \minus \phi \circsmall \pi_{\ti} } (s)$ extends to a non-vanishing holomorphic function on the closed half-plane $\{s\in \C \,|\, \Re(s) \geq s_0 - \wt{\epsilon}_0 \}$ except for a simple pole at $s= s_0$.
\end{proof}

\subsection{Proof of Theorem~\ref{thmLogDerivative}}    \label{subsctProofthmLogDerivative}
Following the ideas from \cite{PoSh98}, we convert the bounds of the zeta function for an expanding Thurston map from Theorem~\ref{thmZetaAnalExt_InvC} to a bound of its logarithmic derivative.

We first record a standard result from complex analysis (see \cite[Theorem~4.2]{EE85}) as in \cite[Section~2]{PoSh98}.

\begin{lemma}    \label{lmLogDerivFromEE}
Given $z\in\C$, $R>0$, and $\delta>0$. Let $F\: \Delta \rightarrow \C$ is a holomorphic function on the closed disk $\Delta\coloneqq \bigl\{ s\in\C \,\big|\, \abs{s-z} \leq R(1+\delta)^3 \bigr\}$. Assume that $F$ satisfies the following two conditions:
\begin{enumerate}
\smallskip
\item[(i)] $F(s)$ has no zeros on the subset 
\begin{equation*}
\bigl\{ s\in \C \,\big|\,  \abs{s-z} \leq R (1+\delta)^2, \Re(s) > \Re(z) - R(1+\delta) \bigr\} \subseteq \Delta.
\end{equation*}
\item[(ii)] There exists a constant $U\geq 0$ depending only on $z$, $R$, $\delta$, and $F$ such that
\begin{equation*}
\log \abs{F(s)} \leq U + \log \abs{F(z)}
\end{equation*}
for all $s\in\Delta$ with $\abs{s-z} \leq R (1+\delta)^3$.
\end{enumerate}
\smallskip
Then for each $s\in\Delta$ with $\abs{s-z} \leq R$, we have
\begin{equation*}
           \Absbigg{ \frac{ F'(s) }{ F(s) } }
\leq  \frac{ 2 + \delta }{ \delta }   \biggl(  \Absbigg{ \frac{ F'(z) }{ F(z) } }   +  \frac{ \bigl(2 + (1+\delta)^{-2} \bigr) (1+\delta) }{  R \delta^2 } U  \biggr).
\end{equation*}
\end{lemma}

We will also need a version of the well-known Phragm\'en--Lindel\"of theorem recorded below. See \cite[Section~5.65]{Ti39} for the statement and proof of this theorem.

\begin{theorem}[The Phragm\'en--Lindel\"of Theorem]    \label{thmPhragmenLindelof}
Given real numbers $\delta_1<\delta_2$. Let $h(s)$ be a holomorphic function on the strip $\{s\in\C \,|\, \delta_1\leq \Re(s) \leq \delta_2 \}$. Assume that the following conditions are satisfied:
\begin{enumerate}
\smallskip
\item[(i)] For each $\sigma > 0$, there exist real numbers $C_\sigma >0$ and $T_\sigma >0$ such that
\begin{equation*}
\abs{ h(\delta + \I t) } \leq C_{\sigma} e^{\sigma\abs{t}}
\end{equation*}
for all $\delta, t\in\R$ with $\delta_1 \leq \delta \leq \delta_2$ and $\abs{t} \geq T_\sigma$.

\smallskip
\item[(ii)] There exist real numbers $C_0>0$, $T_0 >0$, and $k_1,k_2\in \R$ such that
\begin{equation*}
\abs{ h(\delta_1 + \I t) } \leq C_0 \abs{t}^{k_1}       \qquad \text{and} \qquad 
\abs{ h(\delta_2 + \I t) } \leq C_0 \abs{t}^{k_2} 
\end{equation*}
for all $t\in\R$ with $\abs{t} \geq T_0$.
\end{enumerate}
\smallskip
Then there exist real numbers $D>0$ and $T>0$ such that
\begin{equation*}
           \abs{ h(\delta + \I t) }   \leq C\abs{t}^{k(\delta)}
\end{equation*}
for all $\delta, t\in\R$ with $\delta_1 \leq \delta \leq \delta_2$ and $\abs{t} \geq T$, where $k(\delta)$ is the linear function of $\delta$ which takes values $k_1$, $k_2$ for $\delta=\delta_1$, $\delta_2$, respectively.
\end{theorem}

Assuming Theorem~\ref{thmZetaAnalExt_InvC}, we establish Theorem~\ref{thmLogDerivative} as follows.

\begin{proof}[Proof of Theorem~\ref{thmLogDerivative}]
We choose $N_f\in\N$ as in Remark~\ref{rmNf}. Note that $P \bigl( f^i, - s_0 S_i^f \phi \bigr) = i P(f, - s_0 \phi) = 0$ for each $i\in\N$ (see for example, \cite[Theorem~9.8]{Wal82}). We observe that by Lemma~\ref{lmCexistsL}, it suffices to prove the case $n=N_f = 1$. In this case, $F=f$, $\Phi=\phi$, and there exists an Jordan curve $\CC\subseteq S^2$ satisfying $f(\CC)\subseteq \CC$, $\post f\subseteq \CC$, and no $1$-tile in $\X^1(f,\CC)$ joins opposite sides of $\CC$.

Let  $C_\epsilon$, $a_\epsilon\in(0,s_0)$, and $b_\epsilon\geq 2s_0 +1$ be constants from Theorem~\ref{thmZetaAnalExt_InvC} depending only on $f$, $\CC$, $d$, $\alpha$, $\phi$, and $\epsilon$. We fix $\epsilon\coloneqq 1$ throughout this proof.

Define $R\coloneqq \frac{ a_\epsilon }{ 3 }$,  $\beta\coloneqq b_\epsilon + \frac{ a_\epsilon }{2}$, and $\delta \coloneqq  \bigl(\frac{3}{2} \bigr)^{\frac{1}{3}} - 1$. Note that $R(1+\delta)^3 = \frac{ a_\epsilon }{2}$.

Fix an arbitrary $z\in \C$ with $\Re(z) = s_0 + \frac{a_\epsilon}{4}$ and $\abs{ \Im(z) } \geq \beta$. The closed disk 
\begin{equation*}
\Delta \coloneqq \bigl\{ s\in\C \,\big|\, \abs{s-z} \leq R(1+\delta)^3 \bigr\}  =  \Bigl\{ s\in\C \,\Big|\, \abs{s-z} \leq \frac{ a_\epsilon }{ 2 } \Bigr\}
\end{equation*}
is a subset of $\{ s\in\C \,|\, \abs{ \Re(s) - s_0 } < a_\epsilon, \, \abs{ \Im(s) } \geq b_\epsilon \}$. Thus by Theorem~\ref{thmZetaAnalExt_InvC}, inequality~(\ref{eqZetaBound}) holds for all $s\in\Delta$, and the zeta function $\zeta_{f,\,\minus\phi}$ has no zeros in $\Delta$.

For each $s\in \Delta$, by (\ref{eqZetaBound}) in Theorem~\ref{thmZetaAnalExt_InvC} and the fact that $\abs{ \Im(z) } \geq \beta = b_\epsilon + \frac{a_\epsilon}{2}$,
\begin{equation*}
           \Absbig{  \log \Absbig{ \zeta_{f,\,\minus\phi} (s) } -  \log \Absbig{ \zeta_{f,\,\minus\phi} (z) }  } 
 \leq  2 C_\epsilon \Bigl( \abs{ \Im(z) } + \frac{a_\epsilon}{2}  \Bigr)^3
\leq 2^4 C_\epsilon \abs{ \Im(z) }^3
\eqqcolon U.
\end{equation*}

\smallskip

\emph{Claim.} For each $a\in\R$ with $a> s_0$, there exists a real number $\mathcal{K}(a)>0$ depending only on $f$, $\CC$, $d$, $\phi$, and $a$ such that
\begin{equation*}
\Absbigg{  \frac{ \zeta'_{f,\,\minus \phi} ( a+ \I t ) } {    \zeta_{f,\,\minus \phi} ( a+ \I t ) } }     \leq \mathcal{K}(a)
\end{equation*}
for all $t\in\R$.

\smallskip

To establish the claim, we first fix an arbitrary $a\in \R$ with $a> s_0$. By Corollary~\ref{corS0unique}, the topological pressure $P( f, -a \phi) < 0$. It follows from Proposition~\ref{propTopPressureDefPeriodicPts} that there exist numbers $N_a\in\N$ and $\tau_a \in (0,1)$ such that for each integer $n\in\N$ with $n\geq N_a$,
\begin{equation*}
\sum\limits_{ x \in P_{1,f^n} } \exp ( - a S_n\phi (x) ) \leq \tau_a^n.
\end{equation*}
Since the zeta function $\zeta_{f,\,\minus\phi}$ converges uniformly and absolutely to a non-vanishing holomorphic function on $\bigl\{ s\in\C \,\big|\, \Re(s) \geq \frac{ a+ s_0}{2}  \bigr\}$ (see Proposition~\ref{propZetaFnConv_s0}), we get from (\ref{eqDefZetaFn}), Theorem~\ref{thmETMBasicProperties}~(ii), and (\ref{eqDegreeProduct}) that
\begin{align*}
           \Absbigg{  \frac{ \zeta'_{f,\,\minus \phi} ( a+ \I t ) } {    \zeta_{f,\,\minus \phi} ( a+ \I t ) } }   
&  =    \Absbigg{ \sum\limits_{n=1}^{+\infty}     \frac{1}{n}    \sum\limits_{ x \in P_{1,f^n} }  ( S_n\phi (x)  )    \exp ( - (a + \I t) S_n\phi (x) )  }  \\
&\leq \norm{\phi}_{\CCC^0(S^2)}    \sum\limits_{n=1}^{+\infty}     \sum\limits_{ x \in P_{1,f^n} }  \exp ( - a S_n\phi (x) )  \\
&\leq \norm{\phi}_{\CCC^0(S^2)}   \biggl(    \sum\limits_{n=N_a + 1}^{ + \infty }  \tau_a^n    +     \sum\limits_{n=1}^{N_a} \card P_{1,f^n}     \biggr)  \\
&\leq  \mathcal{K}(a),
\end{align*}
for all $t\in\R$, where $\mathcal{K}(a) \coloneqq \norm{\phi}_{\CCC^0(S^2)}   \bigl(    \frac{1}{1-\tau_a} + N_a + \sum_{n=1}^{N_a} (\deg f)^n   \bigr)$ is a constant depending only on $f$, $\CC$, $d$, $\phi$, and $a$. This establishes the claim.

\smallskip

Hence by Lemma~\ref{lmLogDerivFromEE}, the claim with $a\coloneqq s_0 + \frac{a_\epsilon}{4}$, and the choices of $U$, $R$, and $\delta$ above, we get that for all $s\in\Delta$ with $\Im(s) = \Im(z)$ and $\Absbig{ \Re(s) -  \bigl( s_0 + \frac{a_\epsilon}{4} \bigr) }   \leq R= \frac{a_\epsilon}{3}$, we have
\begin{equation}   \label{eqPfthmLogDerivative_1}
          \Absbigg{  \frac{ \zeta'_{f,\,\minus \phi} ( s ) } {    \zeta_{f,\,\minus \phi} ( s ) } } 
\leq   \frac{ 2 + \delta }{ \delta }   \biggl(   \mathcal{K}\Bigl(s_0 + \frac{a_\epsilon}{4}\Bigr)    +  \frac{ 2^4 C_\epsilon \bigl(2 + (1+\delta)^{-2} \bigr) (1+\delta) }{  R \delta^2 }  \abs{ \Im(z) }^3  \biggr)
\leq  C_{19} \abs{ \Im(s) }^3,
\end{equation}
where
\begin{equation*}
C_{19}  \coloneqq \frac{ 2 + \delta }{ \delta }   \biggl(   \mathcal{K}\Bigl(s_0 + \frac{a_\epsilon}{4}\Bigr)    +  \frac{ 2^4 C_\epsilon \bigl(2 + (1+\delta)^{-2} \bigr) (1+\delta) }{  R \delta^2 }   \biggr)
\end{equation*}
is a constant depending only on $f$, $\CC$, $d$, $\alpha$, and $\phi$. Recall that the only restriction on $\Im(z)$ is that $\abs{\Im(z)} \geq \beta$. Thus (\ref{eqPfthmLogDerivative_1}) holds for all $s\in\C$ with $\Absbig{ \Re(s) -  \bigl( s_0 + \frac{a_\epsilon}{4} \bigr) }   \leq  \frac{a_\epsilon}{3}$ and $\abs{\Im(s)} \geq \beta$.

By Theorem~\ref{thmZetaAnalExt_InvC}, $h(s)  \coloneqq   \frac{ \zeta'_{f,\,\minus \phi} ( s ) } {    \zeta_{f,\,\minus \phi} ( s ) }  + \frac{ 1 } { s - s_0 }$ is holomorphic on $\{ s\in\C \,|\, \abs{ \Re(s) - s_0 } < a_\epsilon \}$. Applying the Phragm\'en--Lindel\"of theorem (Theorem~\ref{thmPhragmenLindelof}) to $h(s)$ on the strip $\{ s\in\C \,|\, \delta_1 \leq \Re(s) \leq \delta_2  \}$ with $\delta_1 \coloneqq  s_0 - \frac{a_\epsilon}{12}$ and $\delta_2  \coloneqq  s_0 + \frac{a_\epsilon}{200}$. It follows from (\ref{eqPfthmLogDerivative_1}) that condition~(i) of Theorem~\ref{thmPhragmenLindelof} holds. On the other hand, (\ref{eqPfthmLogDerivative_1}) and the claim above guarantees condition~(ii) of Theorem~\ref{thmPhragmenLindelof} with $k_1 \coloneqq 3$ and $k_2 \coloneqq 0$. Hence by Theorem~\ref{thmPhragmenLindelof}, there exist constants $\wt{D} >0$ and $b\geq 2 s_0 +1$ depending only on $f$, $\CC$, $d$, $\alpha$, and $\phi$ such that
\begin{equation*}
\abs{h(s)}   \leq \wt{D} \abs{ \Im(s) }^{ \frac{1}{2} }
\end{equation*}
for all $s\in\C$ with $\abs{ \Re(s) - s_0 } \leq \frac{ a_\epsilon }{ 200 }$ and $\abs{ \Im(s) } \geq b$.

Therefore inequality (\ref{eqLogDerivative}) holds for all $s\in\C$ with $\abs{ \Re(s) - s_0 } \leq \frac{ a_\epsilon }{ 200 } \eqqcolon a$ and $\abs{ \Im(s) } \geq b$, where $a\in(0,s_0)$, $b\geq 2s_0 + 1$, and $D\coloneqq \wt{D} + 1$ are constants depending only on $f$, $\CC$, $d$, $\alpha$, and $\phi$.
\end{proof}

\section{The Dolgopyat cancellation estimate}   \label{sctDolgopyat}

We adapt the arguments of D.~Dolgopyat \cite{Dol98} in our metric-topological setting aiming to prove Theorem~\ref{thmL2Shrinking} at the end of this section. In Subsection~\ref{subsctSNI}, we first give a formulation of the \emph{$\alpha$-strong non-integrability condition}, $\alpha\in(0,1]$, for our setting and then show its independence on the choice of the Jordan curve $\CC$. In Subsection~\ref{subsctDolgopyatOperator}, a consequence of the $\alpha$-strong non-integrability condition that we will use in the remaining  part of this section is formulated in Proposition~\ref{propSNI}. We remark that it is crucial for the arguments in Subsection~\ref{subsctCancellation} to have the same exponent $\alpha\in(0,1]$ in both the lower bound and the upper bound in (\ref{eqSNIBounds}). The definition of the Dolgopyat operator $\MM_{J,s,\phi}$ in our context is given in Definition~\ref{defDolgopyatOperator} after important constants in the construction are carefully chosen. In Subsection~\ref{subsctCancellation}, we adapt the cancellation arguments of D.~Dolgopyat to establish the $l^2$-bound in Theorem~\ref{thmL2Shrinking}.

\subsection{Strong non-integrability}    \label{subsctSNI}

\begin{definition}[Strong non-integrability condition]   \label{defStrongNonIntegrability}
Let $f\: S^2 \rightarrow S^2$ be an expanding Thurston map and $d$ be a visual metric on $S^2$ for $f$. Given $\alpha \in (0,1]$. Let $\phi\in \Holder{\alpha}(S^2,d)$ be a real-valued H\"{o}lder continuous function with an exponent $\alpha$. 

\begin{enumerate}

\smallskip
\item[(1)] We say that $\phi$ satisfies the \defn{$(\CC, \alpha)$-strong non-integrability condition} (with respect to $f$ and $d$), for a Jordan curve $\CC\subseteq S^2$ with $\post f \subseteq \CC$, if there exist numbers $N_0, M_0\in\N$, $\varepsilon \in (0,1)$, and $M_0$-tiles $Y^{M_0}_\b\in \X^{M_0}_\b(f,\CC)$,  $Y^{M_0}_\w\in \X^{M_0}_\w(f,\CC)$ such that for each $\c\in\{\b,\w\}$, each integer $M \geq M_0$, and each $M$-tile $X\in \X^M(f,\CC)$ with $X \subseteq Y^{M_0}_\c$, there exist two points $x_1(X), x_2(X) \in X$ with the following properties:
\begin{enumerate}
\smallskip
\item[(i)] $\min \{ d(x_1(X), S^2 \setminus X), \, d(x_2(X), S^2 \setminus X), d(x_1(X), x_2(X)) \} \geq \varepsilon \diam_d(X)$, and

\smallskip
\item[(ii)] for each integer $N \geq N_0$, there exist two  $(N+M_0)$-tiles $X^{N+M_0}_{\c,1}, X^{N+M_0}_{\c,2} \in \X^{N+M_0}(f,\CC)$ such that $Y^{M_0}_\c = f^N\bigl(  X^{N+M_0}_{\c,1}  \bigr) =   f^N\bigl(  X^{N+M_0}_{\c,2}  \bigr)$, and that
\begin{equation}   \label{eqSNIBoundsDefn}
\frac{ \Abs{  S_{N }\phi ( \varsigma_1 (x_1(X)) ) -  S_{N }\phi ( \varsigma_2 (x_1(X)) )  -S_{N }\phi ( \varsigma_1 (x_2(X)) ) +  S_{N }\phi ( \varsigma_2 (x_2(X)) )   }  } {d(x_1(X),x_2(X))^\alpha}
\geq \varepsilon,
\end{equation}
where we write $\varsigma_1 \coloneqq \Bigl(f^{N}\big|_{X^{N+M_0}_{\c,1}} \Bigr)^{-1}$ and $\varsigma_2 \coloneqq \Bigl(f^{N}\big|_{X^{N+M_0}_{\c,2}} \Bigr)^{-1}$. 
\end{enumerate}

\smallskip
\item[(2)] We say that $\phi$ satisfies the \defn{$\alpha$-strong non-integrability condition} (with respect to $f$ and $d$) if $\phi$ satisfies the $(\CC, \alpha)$-strong non-integrability condition with respect to $f$ and $d$ for some  Jordan curve $\CC\subseteq S^2$ with $\post f \subseteq \CC$.

\smallskip
\item[(3)] We say that $\phi$ satisfies the \defn{strong non-integrability condition} (with respect to $f$ and $d$) if $\phi$ satisfies the $(\alpha')$-strong non-integrability condition with respect to $f$ and $d$ for some  $\alpha' \in (0, \alpha]$.
\end{enumerate}
\end{definition}

For given $f$, $d$, and $\alpha$ as in Definition~\ref{defStrongNonIntegrability}, if $\phi\in \Holder{\alpha} (S^2, d)$ satisfies the $(\CC,\alpha)$-strong non-integrability condition for some Jordan curve $\CC\subseteq S^2$ with $\post f \subseteq \CC$, then we fix the choices of $N_0$, $M_0$, $\varepsilon$, $Y^{M_0}_\b$, $Y^{M_0}_\w$, $x_1(X)$, $x_2(X)$, $X^{N+M_0}_{\b,1}$, $X^{N+M_0}_{\w,1}$ as in Definition~\ref{defStrongNonIntegrability}, and say that something depends only on $f$, $d$, $\alpha$, and $phi$ even if it also depends on some of these choices.

We will see in the next lemma that the strong non-integrability condition is independent of the Jordan curve $\CC$.

\begin{lemma}   \label{lmSNIwoC}
Let $f$, $d$, $\alpha$ satisfies the Assumptions. Let $\CC$ and $\widehat\CC$ be Jordan curves on $S^2$ with $\post f \subseteq \CC \cap \widehat\CC$. Let $\phi \in \Holder{\alpha}(S^2,d)$ be a real-valued H\"{o}lder continuous function with an exponent $\alpha$. Given arbitrary integers $n, \widehat{n} \in \N$. Let $F\coloneqq f^n$ and $\widehat{F} \coloneqq f^{\widehat{n}}$ be iterates of $f$. Then $\Phi \coloneqq S_n^f \phi$ satisfies the $(\CC, \alpha)$-strong non-integrability condition with respect to $F$ and $d$ if and only if $\widehat\Phi \coloneqq S_{\widehat{n}}^f \phi$ satisfies the $( \widehat\CC, \alpha )$-strong non-integrability condition with respect to $\widehat{F}$ and $d$.

In particular, if $\phi$ satisfies the $\alpha$-strong non-integrability condition with respect to $f$ and $d$, then it satisfies the $(\CC, \alpha)$-strong non-integrability condition with respect to $f$ and $d$.
\end{lemma}

\begin{proof}
Let $\Lambda>1$ be the expansion factor of the visual metric $d$ for $f$. Note that $\post f = \post F = \post \widehat{F}$, and that it follows immediately from Lemma~\ref{lmCellBoundsBM} that $d$ is a visual metric for both $F$ and $\widehat{F}$.

By Lemma~\ref{lmCellBoundsBM}~(ii) and (v), there exist numbers $C_{20} \in (0,1)$ and $l\in \N$ such that for each $\widehat{m}\in\N_0$, each $\widehat{X} \in \X^{\widehat{m}}  ( \widehat{F}, \widehat\CC  )$, there exists $X\in \X^{  \lceil \frac{ \widehat{m} \widehat{n}} {n} \rceil + l}  (F, \CC)$ such that $X\subseteq \widehat{X}$ and $\diam_d (X) \geq C_{20}  \diam_d  ( \widehat{X}  )$.

By symmetry, it suffices to show the forward implication in the first statement of Lemma~\ref{lmSNIwoC}.

We assume that $\Phi$ satisfies the $(\CC,\alpha)$-strong non-integrability condition with respect to $F$ and $d$. We use the choices of numbers $N_0$, $M_0$, $\varepsilon$, tiles $Y^{M_0}_\b\in \X^{M_0}_\b(F,\CC)$,  $Y^{M_0}_\w\in \X^{M_0}_\w(F,\CC)$, $X^{N+M_0}_{\c,1}, X^{N+M_0}_{\c,2} \in \X^{N+M_0}(F,\CC)$, points $x_1(X)$, $x_2(X)$, and functions $\varsigma_1$, $\varsigma_2$ as in Definition~\ref{defStrongNonIntegrability} (with $f$ and $\phi$ replaced by $F$ and $\Phi$, respectively).

It follows from Lemma~\ref{lmCellBoundsBM}~(ii) and (v) again that we can choose an integer $\widehat{M}_0 \in \N$ large enough such that the following statements hold:
\begin{enumerate}
\smallskip
\item[(1)] $\bigl\lceil \frac{ \widehat{M}_0 \widehat{n} } {n} \bigr\rceil + l \geq M_0$.

\smallskip
\item[(2)] There exist $\widehat{M}_0$-tiles $\widehat{Y}^{\widehat{M}_0}_\b\in \X^{\widehat{M}_0}_\b( \widehat{F}, \widehat\CC )$ and $\widehat{Y}^{\widehat{M}_0}_\w\in \X^{\widehat{M}_0}_\w( \widehat{F}, \widehat\CC )$ such that $\widehat{Y}^{\widehat{M}_0}_\b \subseteq \inte \bigl( Y^{M_0}_\b \bigr)$ and $\widehat{Y}^{\widehat{M}_0}_\w \subseteq \inte \bigl( Y^{M_0}_\w \bigr)$.
\end{enumerate}

We define the following constants:
\begin{align}
\widehat{N}_0           \coloneqq &   \biggl\lceil \frac{1}{ \alpha \widehat{n} }  \log_\Lambda   \frac{  2 \Hseminorm{\alpha,\, (S^2,d)}{\phi} C_0 C^{2\alpha} }{ ( 1 - \Lambda^{-\alpha} )  \varepsilon^{1+\alpha}  (1 - C_{20})  }   \biggr\rceil.  \label{eqPflmSNIwoC_wt_N0} \\
\widehat\varepsilon \coloneqq & \varepsilon C_{20}  \in (0, \varepsilon) .   \label{eqPflmSNIwoC_wt_varepsilon}
\end{align}

For each $\c\in\{\b,\w\}$, each integer $\widehat{M} \geq \widehat{M}_0$, and each $\widehat{M}$-tile $\widehat{X} \in \X^{\widehat{M}} ( \widehat{F}, \widehat\CC )$ with $\widehat{X} \subseteq \widehat{Y}^{\widehat{M}_0}_\c$, we denote $M\coloneqq \bigl\lceil \frac{ \widehat{M} \widehat{n} } {n} \bigr\rceil + l \geq M_0$, and choose an $M$-tile $X\in\X^M(F,\CC)$ with
\begin{equation}  \label{eqPflmSNIwoC_X}
X \subseteq \widehat{X}  \qquad\text{and}\qquad \diam_d(X) \geq C_{20} \diam_d ( \widehat{X} ).
\end{equation}
Define, for each $i\in\{1,2\}$,
\begin{equation}  \label{eqPflmSNIwoC_wt_x_i}
\widehat{x}_i ( \widehat{X} ) \coloneqq x_i(X).
\end{equation}

We need to verify Properties~(i) and (ii) in Definition~\ref{defStrongNonIntegrability} for the $(\widehat\CC, \alpha)$-strong non-integrability condition of $\widehat\Phi$ with respect to $\widehat{F}$ and $d$.

Fix arbitrary $\c\in\{\b,\w\}$, $\widehat{M}\in\N$, and $\widehat{X} \in \X^{\widehat{M}} ( \widehat{F}, \widehat\CC )$ with $\widehat{M} \geq \widehat{M}_0$ and $\widehat{X} \subseteq \widehat{Y}^{\widehat{M}_0}_\c$.

\smallskip

\emph{Property~(i).} By (\ref{eqPflmSNIwoC_X}), (\ref{eqPflmSNIwoC_wt_x_i}), (\ref{eqPflmSNIwoC_wt_varepsilon}), and Property~(i) for the $(\CC, \alpha)$-strong non-integrability condition of $\Phi$ with respect to $F$ and $d$, we get
\begin{equation*}
             \frac{  d( \widehat{x}_1 ( \widehat{X} ),  \widehat{x}_2 ( \widehat{X} ) )  } { \diam_d ( \widehat{X} )  }  
\geq     \frac{  d(         x_1 (         X  ),  x_2         (         X  ) )   } { C_{20}^{-1}  \diam_d ( X )  }  
\geq   \varepsilon C_{20} 
=        \widehat\varepsilon,
\end{equation*}
and for each $i\in\{1,2\}$,
\begin{equation*}
             \frac{  d( \widehat{x}_i ( \widehat{X} ),  S^2 \setminus \widehat{X} )  } { \diam_d ( \widehat{X} )  }  
\geq     \frac{  d(         x_i (         X  ),  S^2 \setminus         X )   } { C_{20}^{-1}  \diam_d ( X )  }  
\geq   \varepsilon C_{20} 
=        \widehat\varepsilon.
\end{equation*}

\smallskip

\emph{Property~(ii).}  Fix an arbitrary integer $\widehat{N} \geq \widehat{N}_0$. Choose an integer $N \geq N_0$ large enough so that 
$
Nn > \widehat{N} \widehat{n}.
$

By Proposition~\ref{propCellDecomp}~(i) and (vii), for each $i\in\{1,2\}$, since $F^N$ maps $X^{N+M_0}_{\c,i}$ injectively onto $Y^{M_0}_\c$ and $\widehat{Y}^{\widehat{M}_0}_\c \subseteq \inte \bigl( Y^{M_0}_\c \bigr)$, we have
\begin{equation*}
\varsigma_i \bigl(  \widehat{Y}^{\widehat{M}_0}_\c   \bigr) \in \X^{ \widehat{M}_0 \widehat{n} + Nn } (f, \widehat\CC),
\end{equation*}
where $\varsigma_i = \Bigl(F^{N}\big|_{X^{N+M_0}_{\c,i}} \Bigr)^{-1}$. Define, for each $i\in\{1,2\}$,
\begin{equation*}
\widehat{X}^{ \widehat{N} + \widehat{M}_0 }_{\c,i} \coloneqq f^{Nn - \widehat{N} \widehat{n} }  \bigl( \varsigma_i \bigl(  \widehat{Y}^{\widehat{M}_0}_\c   \bigr) \bigr) 
\in \X^{ \widehat{N} \widehat{n} + \widehat{M}_0 \widehat{n} } (f, \widehat\CC) 
   = \X^{ \widehat{N} + \widehat{M}_0 } (\widehat{F}, \widehat\CC),
\end{equation*}
and write $\widehat\varsigma_i = \Bigl( \widehat{F}^{\widehat{N}}\big|_{\widehat{X}^{\widehat{N}+\widehat{M}_0}_{\c,i}} \Bigr)^{-1}   =   \Bigl( f^{\widehat{N} \widehat{n} }\big|_{\widehat{X}^{\widehat{N}+\widehat{M}_0}_{\c,i}} \Bigr)^{-1} $. Note that $f^{ Nn - \widehat{N} \widehat{n} } \circ \varsigma_i = \widehat\varsigma_i$.

By (\ref{eqPflmSNIwoC_X}), (\ref{eqPflmSNIwoC_wt_x_i}), Properties~(i) and (ii) for the $(\CC, \alpha)$-strong non-integrability condition of $\Phi$ with respect to $F$ and $d$, Lemma~\ref{lmMetricDistortion}, Lemma~\ref{lmCellBoundsBM}~(ii), (\ref{eqPflmSNIwoC_wt_N0}), and (\ref{eqPflmSNIwoC_wt_varepsilon}), we have
\begin{align*}
&                        \frac{ \AbsBig{    S_{\widehat{N}}^{\widehat{F}} \widehat\Phi ( \widehat\varsigma_1 ( \widehat{x}_1( \widehat{X} )) )  -   S_{\widehat{N}}^{\widehat{F}} \widehat\Phi ( \widehat\varsigma_2 ( \widehat{x}_1( \widehat{X} )) ) 
                                                     -  S_{\widehat{N}}^{\widehat{F}} \widehat\Phi ( \widehat\varsigma_1 ( \widehat{x}_2( \widehat{X} )) ) +   S_{\widehat{N}}^{\widehat{F}} \widehat\Phi ( \widehat\varsigma_2 ( \widehat{x}_2( \widehat{X} )) )  }  }   
                                   {d(  \widehat{x}_1( \widehat{X} ), \widehat{x}_2( \widehat{X} ))^\alpha}      \\
&\qquad =        \frac{ \AbsBig{    S_{\widehat{N} \widehat{n}}^{f} \phi ( \widehat\varsigma_1 ( x_1( X )) )  -   S_{\widehat{N} \widehat{n}}^{f} \phi ( \widehat\varsigma_2 ( x_1( X )) )
                                                     -  S_{\widehat{N} \widehat{n}}^{f} \phi ( \widehat\varsigma_1 ( x_2( X )) ) +   S_{\widehat{N} \widehat{n}}^{f} \phi ( \widehat\varsigma_2 ( x_2( X )) )  }  }  
                                  {d(  x_1( X ), x_2( X ))^\alpha}    \\  
&\qquad \geq  \frac{ \AbsBig{    S_{Nn}^{f} \phi ( \varsigma_1 ( x_1( X )) )  -   S_{Nn}^{f} \phi ( \varsigma_2 ( x_1( X )) )
                                                     -  S_{Nn}^{f} \phi ( \varsigma_1 ( x_2( X )) ) +   S_{Nn}^{f} \phi ( \varsigma_2 ( x_2( X )) )  }  }  
                                  {d(  x_1( X ), x_2( X ))^\alpha}    \\
&\qquad\qquad -  \sum\limits_{i\in\{1,2\}}         \frac{ \AbsBig{    S_{ Nn - \widehat{N} \widehat{n} }^{f} \phi ( \varsigma_i ( x_1( X )) )  -   S_{ Nn - \widehat{N} \widehat{n} }^{f} \phi ( \varsigma_i ( x_2( X )) )  }  }      {d(  x_1( X ), x_2( X ))^\alpha} \\
&\qquad \geq  \frac{ \AbsBig{    S_{N}^{F} \Phi ( \varsigma_1 ( x_1( X )) )  -   S_{N}^{F} \Phi ( \varsigma_2 ( x_1( X )) )
                                                     -  S_{N}^{F} \Phi ( \varsigma_1 ( x_2( X )) ) +   S_{N}^{F} \Phi ( \varsigma_2 ( x_2( X )) )  }  }  
                                  {d(  x_1( X ), x_2( X ))^\alpha}    \\
&\qquad\qquad -  \sum\limits_{i\in\{1,2\}}       \frac{ \Hseminorm{\alpha,\, (S^2,d)}{\phi} C_0 }{ 1 - \Lambda^{-\alpha}  } 
                                                                       \cdot  \frac{    d \bigl(  \bigl( f^{ Nn - \widehat{N} \widehat{n} } \circ \varsigma_i \bigr) ( x_1( X ) )  ,  \bigl( f^{ Nn - \widehat{N} \widehat{n} } \circ \varsigma_i \bigr) ( x_2( X ) ) \bigr)    } 
                                                                                          { \varepsilon^\alpha ( \diam_d (X) )^\alpha }   \\
&\qquad \geq      \varepsilon  -   \sum\limits_{i\in\{1,2\}}       \frac{ \Hseminorm{\alpha,\, (S^2,d)}{\phi} C_0 }{ 1 - \Lambda^{-\alpha}  } 
                                                                                               \cdot  \frac{    \diam_d \bigl(  \bigl( f^{ Nn - \widehat{N} \widehat{n} } \circ \varsigma_i \bigr) ( X )   \bigr)    }   { \varepsilon^\alpha ( \diam_d (X) )^\alpha }   \\           
&\qquad \geq      \varepsilon  -    \frac{  2 \Hseminorm{\alpha,\, (S^2,d)}{\phi} C_0 }{ 1 - \Lambda^{-\alpha}  }       
                            \cdot  \frac{   C^\alpha \Lambda^{ - \alpha ( Mn+Nn - ( Nn - \widehat{N}\widehat{n} ) ) }  }   { \varepsilon^\alpha  C^{-\alpha} \Lambda^{ - \alpha  Mn } }   \\
&\qquad \geq      \varepsilon  -    \frac{  2 \Hseminorm{\alpha,\, (S^2,d)}{\phi} C_0 C^{2 \alpha} }{ ( 1 - \Lambda^{-\alpha} )  \varepsilon^\alpha  }    \Lambda^{ - \alpha\widehat{N}_0\widehat{n} }   
                \geq      \varepsilon -  \varepsilon (1 - C_{20}) 
                 =        \widehat\varepsilon,                       
\end{align*}
where $C\geq 1$ is a constant from Lemma~\ref{lmCellBoundsBM} and $C_0 > 1$ is a constant from Lemma~\ref{lmMetricDistortion}, both of which depend only on $f$, $\CC$, and $d$.

The first statement of Lemma~\ref{lmSNIwoC} is now established. The second statement is a special case of the first statement. 
\end{proof}

\begin{prop}  \label{propSNI2NLI}
Let $f$, $d$, $\alpha$ satisfy the Assumptions. Let $\phi \in \Holder{\alpha} (S^2,d)$ be a real-valued H\"older continuous function with an exponent $\alpha$. If $\phi$ satisfies the $\alpha$-strong non-integrability condition (in the sense of Definition~\ref{defStrongNonIntegrability}), then $\phi$ is non-locally integrable (in the sense of Definition~\ref{defLI}). 
\end{prop}

\begin{proof}
We argue by contradiction and assume that $\phi$ is locally integrable and satisfies the $\alpha$-strong non-integrability condition.

Let $\Lambda>1$ be the expansion factor of $d$ for $f$. We first fix a Jordan curve $\CC \subseteq S^2$ containing $\post f$. Then we fix $N_0$, $M_0$, $Y^{M_0}_\b$, and $Y^{M_0}_\w$ as in Definition~\ref{defStrongNonIntegrability}. We choose $M\coloneqq M_0$ and consider an arbitrary $M$-tile $X\in \X^m(f,\CC)$ with $X\subseteq Y^{M_0}_\b$. We fix $x_1(X), x_2(X)\in X$ satisfying Properties~(i) and (ii) in Definition~\ref{defStrongNonIntegrability}. By Theorem~\ref{thmNLI}, $\phi = K + \beta \circ f - \beta$ for some constant $K\in\C$ and some H\"older continuous function $\beta \in \Holder{\alpha}((S^2,d),\C)$.

Then by Property~(ii) in Definition~\ref{defStrongNonIntegrability}, for each $N\geq N_0$,
\begin{equation*}
\frac{ \Abs{   \beta ( \varsigma_1 (x_1(X)) ) -  \beta ( \varsigma_2 (x_1(X)) )  -  \beta ( \varsigma_1 (x_2(X)) ) +  \beta ( \varsigma_2 (x_2(X)) )   }  } {d(x_1(X),x_2(X))^\alpha}
\geq \varepsilon > 0,
\end{equation*}
where $\varsigma_1 \coloneqq \Bigl(f^{N}\big|_{X^{N+M_0}_{\c,1}} \Bigr)^{-1}$ and $\varsigma_2 \coloneqq \Bigl(f^{N}\big|_{X^{N+M_0}_{\c,2}} \Bigr)^{-1}$. Combining the above with Property~(i) in Definition~\ref{defStrongNonIntegrability} and Proposition~\ref{propCellDecomp}~(i), we get
\begin{equation*}
\frac{ 2 \Hseminorm{\alpha,\, (S^2,d)}{\beta}  \bigl( \max \bigl\{ \diam_d \bigl(Y^{N+M_0} \bigr) \,\big|\, Y^{N+M_0} \in \X^{N+M_0}(f,\CC) \bigr\} \bigr)^\alpha } { \varepsilon^\alpha (\diam_d(X))^\alpha }  \geq \varepsilon > 0.
\end{equation*}
Thus by Lemma~\ref{lmCellBoundsBM}~(ii), $2 \Hseminorm{\alpha,\, (S^2,d)}{\beta}  \frac{ C^\alpha \Lambda^{-\alpha N-\alpha M_0 } } { C^{-\alpha} \Lambda^{-\alpha N} }  \geq \varepsilon^{1+\alpha} > 0$, where $C\geq 1$ is a constant from Lemma~\ref{lmCellBoundsBM} depending only on $f$, $\CC$, and $d$. This is a contradiction.
\end{proof}

\subsection{Dolgopyat operator}    \label{subsctDolgopyatOperator}

We now fix an expanding Thurston map $f\: S^2 \rightarrow S^2$, a visual metric $d$ on $S^2$ for $f$ with expansion factor $\Lambda>1$, a Jordan curve $\CC\subseteq S^2$ with $f(\CC)\subseteq \CC$ and $\post f\subseteq \CC$, and an eventually positive real-valued H\"{o}lder continuous function $\phi\in \Holder{\alpha}(S^2,d)$ that satisfies the $(\CC, \alpha)$-strong non-integrability condition. We use the notations from Definition~\ref{defStrongNonIntegrability} below.

We set the following constants that will be repeatedly used in this section. We will see that all these constants defined below from (\ref{eqDef_m0}) to (\ref{eqDef_eta}) depend only on $f$, $\CC$, $d$, $\alpha$, and $\phi$.

\begin{align}
  m_0       & \coloneqq      \max \biggl\{    \biggl\lceil \frac{1}{\alpha}  \log_\Lambda  \bigl( 8C_1 \varepsilon^{\alpha - 1}  \bigr) \biggr\rceil   ,\,
                                                              \bigl\lceil \log_\Lambda \bigl( 10 \varepsilon^{-1} C^2 \bigr) \bigr\rceil   \biggr\}                    \geq 1.   \label{eqDef_m0}  \\
\delta_0  & \coloneqq     \min \Bigl\{ \frac{1}{2C_1},\, \frac{\varepsilon^2}{20C^2} \Bigr\} \in (0,1).   \label{eqDef_delta0}  \\  
  b_0        & \coloneqq \max \Biggl\{ 2 s_0  +1,\, 
                             \frac{C_0 T_0}{1-\Lambda^{-\alpha}},\,
                             \frac{2 A_0 \Hseminormbig{\alpha,\, (S^2,d)}{\wt{- s_0\phi}} }{1-\Lambda^{-\alpha}}   \Biggr\}.    \label{eqDef_b0} \\
   A         & \coloneqq \max \{ 3C_{10}T_0,\, 4 \}.   \label{eqDef_A} \\
\epsilon_1 & \coloneqq \min \Bigl\{  \frac{\pi}{16} \delta_0,\, \frac{1}{ 4A } \Lambda^{-M_0}  \Bigr\} \in(0,1).  \label{eqDef_epsilon1}  \\
  N_1      & \coloneqq \max \biggl\{   N_0,\,  \biggl\lceil \frac{1}{\alpha} \log_\Lambda \biggl( \max \biggl\{ 2^{10} A, \, 
                                                                \frac{1280 A \Lambda C^2}{\delta_0},\,4A_0, \, 4C_{10} \biggr\} \biggr) \biggr\rceil  \biggr\} .   \label{eqDef_N1} \\
 \eta       & \coloneqq \min \Biggl\{ 2^{-12},\, \biggl( \frac{ \delta_0 \epsilon_1 }{ 1280 \Lambda C^2 } \biggr)^2, \,
                             \frac{A\epsilon_1 \varepsilon^\alpha }{240 C_{10} C^2} \Lambda^{-2\alpha m_0 - 1} \bigl(\LIP_d(f)\bigr)^{-\alpha N_1} \Biggr\}. \label{eqDef_eta}
\end{align}   
Here the constants $M_0\in \N$, $N_0\in\N$, and $\varepsilon\in (0,1)$ depending only on $f$, $d$, $\CC$, and $\phi$ are from Definition~\ref{defStrongNonIntegrability}; the constant $s_0$ is the unique positive real number satisfying $P(f,-s_0\phi)=0$; the constant $C\geq 1$ depending only on $f$, $d$, and $\CC$ is from Lemma~\ref{lmCellBoundsBM}; the constant $C_0>1$ depending only on $f$, $d$, and $\CC$ is from Lemma~\ref{lmMetricDistortion}; the constant $C_1>0$ depending only on $f$, $d$, $\CC$, $\phi$, and $\alpha$ is from Lemma~\ref{lmSnPhiBound}; the constant $A_0>2$ depending only on $f$, $\CC$, $d$, $\Hseminorm{\alpha,\, (S^2,d)}{\phi}$, and $\alpha$ is from Lemma~\ref{lmBasicIneq}; the constant $C_{10}=C_{10}(f,\CC,d,\alpha,T_0)>1$ depending only on $f$, $\CC$, $d$, $\alpha$, and $\phi$ is defined in (\ref{eqDefC10}) from Lemma~\ref{lmExpToLiniear}; and the constant $T_0>0$ depending only on $f$, $\CC$, $d$, $\phi$, and $\alpha$ is defined in (\ref{eqDefT0}), and according to Lemma~\ref{lmBound_aPhi} satisfies
\begin{equation}  \label{eqT0Bound_sctDolgopyat}
\sup\bigl\{ \Hseminormbig{\alpha,\, (S^2,d)}{\wt{a\phi}} \,\big|\, a\in\R, \abs{a}\leq 2 s_0     \bigr\} \leq T_0.
\end{equation}

We denote for each $b\in\R$ with $\abs{b}\geq 1$,
\begin{equation}   \label{eqDefCb}
\mathfrak{C}_b \coloneqq \bigl\{  X\in \X^{m(b)}(f,\CC) \,\big|\, X\subseteq Y^{M_0}_\b \cup Y^{M_0}_\w \bigr\},
\end{equation}
where we write 
\begin{equation}   \label{eqDefmb}
m(b) \coloneqq  \biggl\lceil \frac{1}{\alpha}  \log_\Lambda \biggl( \frac{C\abs{b}}{\epsilon_1} \biggr) \biggr\rceil    
\end{equation}
with the constant $C\geq 1$ depending only on $f$, $\CC$, and $d$ from Lemma~\ref{lmCellBoundsBM}.

Note that by (\ref{eqDef_epsilon1}),
\begin{equation*}
m(b) \geq \log_{\Lambda}\frac{1}{\epsilon_1} \geq M_0,
\end{equation*}
and if $X\in \mathfrak{C}_b$, then $\diam_d(X) \leq \bigl(  \frac{\epsilon_1}{\abs{b}}  \bigr)^{\frac{1}{\alpha}}$ by Lemma~\ref{lmCellBoundsBM}~(ii).

For each $X\in\mathfrak{C}_b$, we now fix choices of tiles $\mathfrak{X}_1 (X), \mathfrak{X}_2 (X) \in \X^{m(b)+m_0}(f,\CC)$ and $\mathfrak{X}'_1 (X), \mathfrak{X}'_2 (X) \in \X^{m(b)+ 2 m_0}(f,\CC)$ in such a way that for each $i\in\{1,2\}$,
\begin{equation} \label{eqXXXsubsetXX}
 x_i(X) \in \mathfrak{X}'_i (X)       \subseteq \mathfrak{X}_i(X).
\end{equation}
By Property~(i) in Definition~\ref{defStrongNonIntegrability}, (\ref{eqDef_m0}), and Lemma~\ref{lmCellBoundsBM}~(ii) and (v), it is easy to see that the constant $m_0$ we defined in (\ref{eqDef_m0}) is large enough so that the following inequalities hold:
\begin{align}
  d( \mathfrak{X}_i (X) , S^2\setminus X ) & \geq  \frac{\varepsilon}{10} C^{-1} \Lambda^{-m(b)}, \label{eqXXtoCX_LB}\\
 \diam_d (\mathfrak{X}_i (X))                     & \leq  \frac{\varepsilon}{10} C^{-1} \Lambda^{-m(b)}, \label{eqDiamXX_UB} \\
 d( \mathfrak{X}'_i(X), S^2 \setminus \mathfrak{X}_i(X) )  & \geq  \frac{\varepsilon}{10} C^{-1} \Lambda^{-m(b)-m_0}, \label{eqXXXtoCXX_LB}\\
 \diam_d (\mathfrak{X}'_i (X))                    & \leq  \frac{\varepsilon}{10} C^{-1} \Lambda^{-m(b)-m_0}  \label{eqDiamXXX_UB}
\end{align}
for $i\in\{1,2\}$, and that
\begin{equation} \label{eqXX1toXX2_LB}
     d(\mathfrak{X}_1 (X),\mathfrak{X}_2 (X)  )  \geq \frac{\varepsilon}{10} C^{-1} \Lambda^{-m(b)}.  
\end{equation}

For each $X\in \mathfrak{C}_b$ and each $i\in\{1,2\}$, we define a function $\psi_{i,X} \: S^2\rightarrow \R$ by
\begin{equation}    \label{eqDefCharactFn}
\psi_{i,X}(x) \coloneqq \frac{  d(x, S^2 \setminus \mathfrak{X}_i(X))^\alpha  }
                     {  d(x,\mathfrak{X}'_i(X))^\alpha + d(x, S^2 \setminus \mathfrak{X}_i(X))^\alpha  }
\end{equation}
for $x\in S^2$. Note that 
\begin{equation}  \label{eqPsi_iX}
\psi_{i,X}(x)=1  \text{ if }x\in \mathfrak{X}'_i(X), \quad \text{ and } \quad  \psi_{i,X}(x)=0 \text{ if } x\notin \mathfrak{X}_i(X).
\end{equation} 

\begin{definition}    \label{defFull}
We say that a subset $J\subseteq \{1,2\} \times \{1,2\} \times \mathfrak{C}_b$ has \defn{full projection} if $\pi_3(J) = \mathfrak{C}_b$, where $\pi_3\: \{1,2\} \times \{1,2\} \times \mathfrak{C}_b\rightarrow \mathfrak{C}_b$ is the projection $\pi_3(j,i,X)=X$. We write $\mathcal{F}$ for the collection of all subsets of $\{1,2\} \times \{1,2\} \times \mathfrak{C}_b$ that have full projections.
\end{definition}

For a subset $J \subseteq \{1,2\} \times \{1,2\} \times \mathfrak{C}_b$, we define a function $\beta_J \: S^2 \rightarrow \R$ as
\begin{equation}    \label{eqDefBetaJ}
\beta_J(x) \coloneqq  \begin{cases} 1-\eta \sum\limits_{i\in\{1,2\}} \sum\limits_{\substack{X\in\mathfrak{C}_b\\ \{1,i,X\}\in J}}  \psi_{i,X} \bigl(f^{N_1}(x)\bigr) & \text{if } x\in \inte\bigl(X_{\b,1}^{N_1+M_0}\bigr) \cup \inte\bigl( X_{\w,1}^{N_1+M_0} \bigr), \\  
       1-\eta \sum\limits_{i\in\{1,2\}} \sum\limits_{\substack{X\in\mathfrak{C}_b\\ \{2,i,X\}\in J}}  \psi_{i,X} \bigl(f^{N_1}(x)\bigr) & \text{if } x\in \inte\bigl(X_{\b,2}^{N_1+M_0}\bigr) \cup \inte\bigl(X_{\w,2}^{N_1+M_0}\bigr), \\
       1  & \text{otherwise}, \end{cases}
\end{equation}
for $x\in S^2$.

The only properties of potentials that satisfy $\alpha$-strong non-integrability used in this section are summarized in the following proposition.

\begin{prop}   \label{propSNI}
Let $f$, $\CC$, $d$, $\alpha$, $\phi$ satisfy the Assumptions. We assume in addition that $f(\CC)\subseteq\CC$ and that $\phi$ satisfies the $\alpha$-strong non-integrability condition. Let $b\in\R$ with $\abs{b}\geq 1$. Using the notation above, the following statement holds:

\smallskip

For each $\c\in\{\b,\w\}$, each $X \in \mathfrak{C}_b$, each $x\in \mathfrak{X}'_1(X)$, and each $y\in \mathfrak{X}'_2(X)$,
\begin{equation}   \label{eqSNIBounds}
     \delta_0 d(x,y)^\alpha
\leq \Abs{  S_{N_1 }\phi ( \tau_1 (x) ) -  S_{N_1 }\phi ( \tau_2 (x) )  -S_{N_1 }\phi ( \tau_1 (y) ) +  S_{N_1 }\phi ( \tau_2 (y) )   } 
\leq \delta_0^{-1} d(x,y)^\alpha,
\end{equation}
where we write $\tau_1 \coloneqq \Bigl(f^{N_1}\big|_{X^{N_1+M_0}_{\c,1}} \Bigr)^{-1}$ and $\tau_2 \coloneqq \Bigl(f^{N_1}\big|_{X^{N_1+M_0}_{\c,2}} \Bigr)^{-1}$.
\end{prop}

\begin{proof}
We first observe that the second inequality in (\ref{eqSNIBounds}) follows immediately from the triangle inequality, Lemma~\ref{lmSnPhiBound}, and (\ref{eqDef_delta0}).

It suffices to prove the first inequality in (\ref{eqSNIBounds}). Fix arbitrary $\c\in\{\b,\w\}$, $X \in \mathfrak{C}_b$, $x\in \mathfrak{X}'_1(X)$, and $y\in \mathfrak{X}'_2(X)$. By (\ref{eqXXXsubsetXX}), (\ref{eqXX1toXX2_LB}), Lemma~\ref{lmCellBoundsBM}~(ii), Lemma~\ref{lmSnPhiBound}, (\ref{eqXXXtoCXX_LB}), we get
\begin{align*}
&                         \frac{  \Abs{  S_{N_1 }\phi ( \tau_1 (x) ) -  S_{N_1 }\phi ( \tau_2 (x) )  -S_{N_1 }\phi ( \tau_1 (y) ) +  S_{N_1 }\phi ( \tau_2 (y) )   }   } { d(x,y)^\alpha } \\
&\qquad\geq    \frac{  \Abs{  S_{N_1 }\phi ( \tau_1 (x) ) -  S_{N_1 }\phi ( \tau_2 (x) )  -S_{N_1 }\phi ( \tau_1 (y) ) +  S_{N_1 }\phi ( \tau_2 (y) )   }   } { d(x_1(X),x_2(X))^\alpha }   \cdot
                            \frac{ d( \mathfrak{X}_1(X),\mathfrak{X}_2(X))^\alpha }{ (\diam_d(X))^\alpha}  \\
&\qquad\geq   \biggl( \varepsilon  -  \frac{ 2 C_1 ( \diam_d ( \mathfrak{X}'_1(X))^\alpha + 2 C_1 ( \diam_d ( \mathfrak{X}'_2(X))^\alpha }  {(\diam_d(X))^\alpha}    \biggr) 
                           \frac{ 10^{-\alpha} \varepsilon^\alpha C^{-\alpha} \Lambda^{-\alpha m(b)} }{ (\diam_d(X))^\alpha}  \\
&\qquad\geq   \biggl( \varepsilon  -  \frac{ 4  C_1 10^{-\alpha}  \varepsilon^\alpha C^{-\alpha} \Lambda^{-\alpha m(b) - \alpha m_0}   }  { C^{-\alpha} \Lambda^{-\alpha m(b)} }    \biggr)        
                           \frac{ 10^{-\alpha} \varepsilon^\alpha C^{-\alpha} \Lambda^{-\alpha m(b)} }{ C^{\alpha} \Lambda^{-\alpha m(b)} }  \\
&\qquad\geq  \frac{\varepsilon^{1+\alpha}} {2C^{2\alpha} 10^\alpha } \geq \delta_0,
\end{align*}
where the last two inequalities follow from (\ref{eqDef_m0}) and (\ref{eqDef_delta0}).
\end{proof}

\begin{lemma}   \label{lmBetaJ}
Let $f$, $\CC$, $d$, $\Lambda$, $\alpha$, $\phi$, $s_0$ satisfy the Assumptions. We assume in addition that $f(\CC)\subseteq \CC$ and that $\phi$ satisfies the $\alpha$-strong non-integrability condition. We use the notation in this section.

Given $b\in\R$ with $\abs{b}\geq 2 s_0 +1$. Then for each $X\in \mathfrak{C}_b$ and each $i\in\{1,2\}$, the function $\psi_{i,X} \: S^2\rightarrow \R$ defined in (\ref{eqDefCharactFn}) is H\"{o}lder with an exponent $\alpha$ and
\begin{equation}  \label{eqCharactFnHolderConst}
\Hseminorm{\alpha,\, (S^2,d)}{\psi_{i,X}}  \leq 20 \varepsilon^{-\alpha} C \Lambda^{\alpha ( m(b) + 2m_0 ) }.
\end{equation} 
Moreover, for each subset $J\subseteq \{1,2\}\times\{1,2\}\times\mathfrak{C}_b$, the function $\beta_J\:S^2\rightarrow\R$ defined in (\ref{eqDefBetaJ}) satisfies
\begin{equation} \label{eqBetaJBounds}
1\geq \beta_J(x) \geq 1-\eta > \frac{1}{2}
\end{equation}
for $x\in S^2$. In addition, $\beta_J \in \Holder{\alpha}(S^2,d)$ with $\Hseminorm{\alpha,\, (S^2,d)}{\beta_J} \leq L_\beta$, where
\begin{equation}   \label{eqDefLipBoundBetaJ}
L_{\beta} \coloneqq 40 \varepsilon^{-\alpha} C \Lambda^{ \alpha ( m(b)+2m_0 ) }   (\LIP_d(f))^{\alpha N_1} \eta
\end{equation}
is a constant depending only on $f$, $\CC$, $d$, $\alpha$, $\phi$, and $b$. Here $C\geq 1$ is a constant from Lemma~\ref{lmCellBoundsBM} depending only on $f$, $\CC$, and $d$.
\end{lemma}

\begin{proof}
We will first establish (\ref{eqCharactFnHolderConst}). Consider distinct points $x,y\in S^2$.

If $x,y\in S^2 \setminus \mathfrak{X}_i(X)$, then $\frac{ \psi_{i,X}(x)-\psi_{i,X}(y) }{ d(x,y)^\alpha } = 0$.

If $x \in S^2 \setminus \mathfrak{X}_i(X)$ and $y \in \mathfrak{X}_i(X)$, then by (\ref{eqXXXtoCXX_LB}),
\begin{align*}
        \frac{ \abs{ \psi_{i,X}(x)-\psi_{i,X}(y) } }{ d(x,y)^\alpha }  
 =   &  \frac{ d(y, S^2\setminus \mathfrak{X}_i(X) ) }{ d(x,y)^\alpha }  
        \frac{ 1 }{ d(y,\mathfrak{X}'_i(X) )^\alpha + d(y, S^2\setminus \mathfrak{X}_i(X) )^\alpha } \\
\leq &  \frac{ 1 }{ d(\mathfrak{X}'_i(X), S^2\setminus \mathfrak{X}_i(X) )^\alpha } 
\leq    \frac{10^\alpha}{\varepsilon^\alpha} C^\alpha \Lambda^{ \alpha ( m(b)+m_0 ) }  
\leq \frac{20}{\varepsilon^\alpha} C \Lambda^{ \alpha ( m(b)+2m_0 ) }.
\end{align*}

Similarly, if $y \in S^2 \setminus \mathfrak{X}_i(X)$ and $x \in \mathfrak{X}_i(X)$, then $\frac{ \abs{ \psi_{i,X}(x)-\psi_{i,X}(y) } }{ d(x,y)^\alpha }  \leq 20 \varepsilon^{-\alpha} C \Lambda^{ \alpha ( m(b)+2m_0 ) }$.

If $x,y\in \mathfrak{X}_i(X)$, then by (\ref{eqDiamXX_UB}), (\ref{eqXXXsubsetXX}), and (\ref{eqXXXtoCXX_LB}),
\begin{align*}    
     & \frac{ \abs{ \psi_{i,X}(x)-\psi_{i,X}(y) } }{ d(x,y)^\alpha }  \\
\leq & \frac{ d(x, S^2 \setminus \mathfrak{X}_i(X))^\alpha \abs{d(x,\mathfrak{X}'_i(X))^\alpha - d(y,\mathfrak{X}'_i(X))^\alpha}   }
            { d(x,y)^\alpha  (d(x,\mathfrak{X}'_i(X))^\alpha + d(x, S^2 \setminus \mathfrak{X}_i(X))^\alpha ) (d(y,\mathfrak{X}'_i(X))^\alpha + d(y, S^2 \setminus \mathfrak{X}_i(X))^\alpha )}    \\
     &+\frac{\abs{d(x, S^2 \setminus \mathfrak{X}_i(X))^\alpha - d(y, S^2 \setminus \mathfrak{X}_i(X))^\alpha} d(x,\mathfrak{X}'_i(X))^\alpha}
     {  d(x,y)^\alpha   (d(x,\mathfrak{X}'_i(X))^\alpha + d(x, S^2 \setminus \mathfrak{X}_i(X))^\alpha ) (d(y,\mathfrak{X}'_i(X))^\alpha + d(y, S^2 \setminus \mathfrak{X}_i(X))^\alpha )}   \\ 
\leq & \frac{ d(x, S^2 \setminus \mathfrak{X}_i(X))^\alpha d(x,y)^\alpha  + d(x,y)^\alpha d(x,\mathfrak{X}'_i(X))^\alpha  }
            {  d(x,y)^\alpha  d( \mathfrak{X}'_i(X), S^2 \setminus \mathfrak{X}_i(X)) ^{2\alpha}} \\
\leq & \frac{    \frac{\varepsilon^\alpha}{10^\alpha}C^{-\alpha}\Lambda^{-\alpha m(b)}
               + \frac{\varepsilon^\alpha}{10^\alpha}C^{-\alpha}\Lambda^{-\alpha m(b)}  }
            { \bigl( \frac{\varepsilon}{10}C^{-1}\Lambda^{-m(b) - m_0 }  \bigr)^{2 \alpha}  } 
\leq  20 \varepsilon^{-\alpha} C \Lambda^{ \alpha ( m(b)+2m_0) } . 
\end{align*}

Hence $\Hseminorm{\alpha,\, (S^2,d)}{\psi_{i,X}} \leq 20 \varepsilon^{-\alpha} C \Lambda^{ \alpha ( m(b)+2m_0) }$, establishing (\ref{eqCharactFnHolderConst}).

In order to establish (\ref{eqBetaJBounds}), we only need to observe that for each $j\in\{1,2\}$, and each $X\in \inte\bigl(X_{\b,j}^{N_1+M_0}\bigr) \cup \inte\bigl( X_{\w,j}^{N_1+M_0} \bigr)$, at most one term in the summations in (\ref{eqDefBetaJ}) is nonzero. Indeed, we note that for each pair of distinct tiles $X_1,X_2\in\mathfrak{C}_b$, $\mathfrak{X}_{i_1}(X_1)\cap\mathfrak{X}_{i_2}(X_2) = \emptyset$ for all $i_1,i_2\in\{1,2\}$ by (\ref{eqXXtoCX_LB}), and $\mathfrak{X}_1(X_1) \cap \mathfrak{X}_2(X_1) = \emptyset$ by (\ref{eqXX1toXX2_LB}). Hence by (\ref{eqPsi_iX}), at most one term in the summations in (\ref{eqDefBetaJ}) is nonzero, and (\ref{eqBetaJBounds}) follows from (\ref{eqDef_eta}).

We now show the continuity of $\beta_J$. Note that for each $i\in \{1,2\}$ and each $X\in\mathfrak{C}_b$, by (\ref{eqXXtoCX_LB}), (\ref{eqPsi_iX}), and the continuity of $\psi_{i,X}$, we have 
\begin{equation*}
\psi_{i,X} \bigl( f^{N_1} \bigl( \partial X_{\c,j}^{N_1+M_0} \bigr) \bigr) = \psi_{i,X}\bigl( Y_\c^{M_0} \bigr) = \{0\}
\end{equation*}
for $\c\in\{\b,\w\}$ and $j\in\{1,2\}$. It follows immediately from (\ref{eqDefBetaJ}) that $\beta_J$ is continuous.

Finally, for arbitrary $x,y\in S^2$ with $x\neq y$, we will establish $\frac{\abs{\beta_J(x) - \beta_J(y)}}{ d(x,y)^\alpha } \leq L_\beta$ by considering the following two cases.

\smallskip

\emph{Case 1.} $x,y\in X^{N_1 + m(b)}$ for some $X^{N_1+m(b)} \in \X^{N_1+m(b)}$. If 
\begin{equation*}
X^{N_1+m(b)}  \nsubseteq \bigcup \bigl\{ X_{\c,j}^{N_1+M_0} \,\big|\, \c\in\{\b,\w\},\, j\in\{1,2\} \bigr\},
\end{equation*}
then $\beta_J(x)-\beta_J(y) = 1-1 = 0$. If
\begin{equation*}
X^{N_1+m(b)}  \subseteq \bigcup \bigl\{ X_{\c,j}^{N_1+M_0} \,\big|\, \c\in\{\b,\w\},\, j\in\{1,2\} \bigr\},
\end{equation*}
then by (\ref{eqPsi_iX}), 
\begin{align*}
          \frac{\abs{\beta_J(x) - \beta_J(y)}}{ d(x,y)^\alpha }
   =   &  \frac{ \Bigl( 1- \eta \sum\limits_{i\in\{1,2\}} \psi_{i,X_*} \bigl( f^{N_1} (x) \bigr) \Bigr)
                      -\Bigl( 1- \eta \sum\limits_{i\in\{1,2\}} \psi_{i,X_*} \bigl( f^{N_1} (y) \bigr) \Bigr) } { d(x,y)^\alpha }\\
\leq   &  2 \eta \Hseminorm{\alpha,\, (S^2,d)}{\psi_{i,X_*}}  \bigl( \LIP_d(f) \bigr)^{\alpha N_1}  \leq L_\beta,                    
\end{align*}
where we denote $X_* \coloneqq  f^{N_1} \bigl(X^{N_1+m(b)} \bigr)$.

\smallskip
\emph{Case 2.} $\card \bigl( \{x,y\} \cap X^{N_1+m(b)} \bigr) \leq 1$ for all $X^{N_1+m(b)} \in \X^{N_1+m(b)}$. We assume, without loss of generality, that $\beta_J(x)-\beta_J(y) \neq 0$. Then by (\ref{eqPsi_iX}) and (\ref{eqXXtoCX_LB}), $d\bigl( f^{N_1}(x),f^{N_1}(y)\bigr) \geq \frac{\varepsilon}{10} C^{-1}\Lambda^{-m(b)}$. Thus $d(x,y) \geq \frac{\varepsilon}{10} C^{-1}\Lambda^{-m(b)}  ( \LIP_d(f) )^{-N_1}$. Hence by (\ref{eqBetaJBounds}), 
$
\frac{\abs{\beta_J(x) - \beta_J(y)}}{ d(x,y)^\alpha } \leq  10 \varepsilon^{-\alpha} C \Lambda^{\alpha m(b)} ( \LIP_d(f) )^{\alpha N_1} \eta \leq L_\beta.
$ 
\end{proof}

\begin{definition}   \label{defDolgopyatOperator}
Let $f$, $\CC$, $d$, $\alpha$, $\phi$ satisfy the Assumptions. We assume in addition that $f(\CC)\subseteq \CC$ and that $\phi$ satisfies the $\alpha$-strong non-integrability condition. Let $a,b\in\R$ satisfy $\abs{b}\geq 1$. Denote $s \coloneqq a+\I b$. For each subset $J\subseteq \{1,2\}\times\{1,2\}\times \mathfrak{C}_b$, the \defn{Dolgopyat operator} $\MM_{J,s,\phi}$ on $\Holder{\alpha}\bigl(\bigl(X^0_\b,d\bigr),\C\bigr) \times \Holder{\alpha}\bigl(\bigl(X^0_\w,d \bigr),\C\bigr)$ is defined by
\begin{equation}   \label{eqDefDolgopyatOperator}
\MM_{J,s,\phi} (u_\b,u_\w) = \RRR_{ \wt{a\phi}}^{N_1+M_0}  \bigl( u_\b \beta_J|_{X^0_\b} ,  u_\w  \beta_J|_{X^0_\w} \bigr)
\end{equation}
for $u_\b \in \Holder{\alpha}\bigl(\bigl(X^0_\b,d\bigr),\C\bigr)$ and $u_\w\in\Holder{\alpha}\bigl(\bigl(X^0_\w,d\bigr),\C\bigr)$.
\end{definition}

Here $\mathfrak{C}_b$ is defined in (\ref{eqDefCb}), $\beta_J$ is defined in (\ref{eqDefBetaJ}), $M_0\in\N$ is a constant from Definition~\ref{defStrongNonIntegrability}, and $N_1$ is given in (\ref{eqDef_N1}). Note that in (\ref{eqDefDolgopyatOperator}), since $\beta_J \in \Holder{\alpha}(S^2,d)$ (see Lemma~\ref{lmBetaJ}), we have $u_\c \beta_J|_{X^0_\c}  \in \Holder{\alpha}\bigl(\bigl(X^0_\c,d\bigr),\C\bigr)$ for $\c\in\{\b,\w\}$.

\subsection{Cancellation argument}    \label{subsctCancellation}

\begin{lemma}    \label{lmGibbsDoubling}
Let $f$, $\CC$, $d$ satisfy the Assumptions. Let $\varphi\in \Holder{\alpha}(S^2,d)$ be a real-valued H\"{o}lder continuous function with an exponent $\alpha\in(0,1]$. Then there exists a constant $C_{\mu_\varphi}\geq 1$ depending only on $f$, $d$, and $\varphi$ such that for all integers $m,n\in\N_0$, and tiles $X^n\in \X^n(f,\CC)$, $X^{m+n}\in \X^{m+n}(f,\CC)$ satisfying $X^{m+n} \subseteq X^n$, we have
\begin{equation}    \label{eqGibbsDoubling}
\frac{  \mu_\varphi(X^n) }{ \mu_\varphi(X^{m+n}) } \leq C_{\mu_\varphi}^2 \exp  ( m  ( \norm{\varphi}_{\CCC^0(S^2) } +P(f,\varphi) )  ),
\end{equation}
where $\mu_\varphi$ is the unique equilibrium state for the map $f$ and the potential $\varphi$, and $P(f,\varphi)$ denotes the topological pressure for $f$ and $\varphi$.
\end{lemma}

\begin{proof}
By Theorem~5.16, Corollary~5.18, and Theorem~1.1 in \cite{Li18}, the unique equilibrium state $\mu_\varphi$ is a \emph{Gibbs measure} with respect to $f$, $\CC$, and $\varphi$ as defined in Definition~5.3 in \cite{Li18}. More precisely, there exist constants $P_{\mu_\varphi} \in \R$ and $C_{\mu_\varphi} \geq 1$ such that for each $n\in\N_0$, each $n$-tile $X^n \in\X^n$, and each $x\in X^n$, we have
\begin{equation*}
           \frac{1}{C_{\mu_\varphi}}  
\leq    \frac{ \mu_\varphi(X^n)}  { \exp \bigl( S_n\varphi (x) - n P_{\mu_\varphi}  \bigr) }
\leq   C_{\mu_\varphi}.
\end{equation*}

We fix arbitrary integers $m,n\in\N_0$, and tiles $X^n\in \X^n$, $X^{m+n}\in \X^{m+n}$ satisfying $X^{m+n} \subseteq X^n$. Choose an arbitrary point $x \in X^{m+n}$. Then
\begin{equation*}
                 \frac{  \mu_\varphi(X^n) }{ \mu_\varphi(X^{m+n}) } 
\leq   C_{\mu_\varphi}^2  \frac{  \exp \bigl( S_n\varphi (x) - n P_{\mu_\varphi}  \bigr) }{ \exp \bigl( S_{n+m}\varphi (x) - (n+m) P_{\mu_\varphi}  \bigr) }
\leq   C_{\mu_\varphi}^2 \exp  ( m  ( \norm{\varphi}_{\CCC^0(S^2) } +P(f,\varphi)  )   ).
\end{equation*}

Inequality~(\ref{eqGibbsDoubling}) follows immediately from the fact that $P_{\mu_\varphi} = P(f,\varphi)$ (see \cite[Theorem~5.16 and Proposition~5.17]{Li18}).
\end{proof}

\begin{lemma}   \label{lmArgumentIneqalities}
For all $z_1,z_2\in \C\setminus\{0\}$, the following inequalities hold:
\begin{equation}  \label{eqArgumentTriangleIneq}
\abs{\Arg (z_1 z_2)} \leq \abs{\Arg(z_1)} + \abs{\Arg(z_2)},
\end{equation}
\begin{equation}    \label{eqTriangleIneqWithAngle}
\abs{z_1 + z_2}  \leq \abs{z_1} + \abs{z_2} -  \frac{1}{16} \Bigl(\Arg\Bigl(\frac{z_1}{z_2}\Bigr)\Bigr)^2  \min\{\abs{z_1}, \, \abs{z_2}\},
\end{equation}
\begin{equation}   \label{eqArgDist}
\AbsBig{ \Arg\Bigl( \frac{z_1}{z_2} \Bigr) } \leq \frac{2 \abs{z_1 - z_2}}{\min\{\abs{z_1}, \, \abs{z_2} \} }  .
\end{equation}
\end{lemma}

\begin{proof}
Inequality (\ref{eqArgumentTriangleIneq}) follows immediately from the definition of $\Arg$ (see Section~\ref{sctNotation}).

We then verify (\ref{eqTriangleIneqWithAngle}). Without loss of generality, we assume that $\abs{z_1}\leq\abs{z_2}$ and $\theta \coloneqq \Arg\bigl(\frac{z_1}{z_2}\bigr) \geq 0$. Using the labeling in Figure~\ref{figTriangleIneqWithAngle}, we let $\overrightarrow{OQ} = z_2$ and $\overrightarrow{QC} = z_1$. Then
\begin{align*}
\abs{z_1+z_2}  =  & \abs{OA} + \abs{AC} \leq \abs{z_2} + \abs{BC} \leq \abs{z_2} + \abs{z_1}\cos\biggl( \frac{\theta}{2} \biggr) \\
             \leq & \abs{z_2} + \abs{z_1} \biggl( 1- \frac{\theta^2}{8} + \frac{\theta^4}{4! 2^4} \biggr) 
             \leq   \abs{z_2} + \biggl( 1- \frac{\theta^2}{16} \biggr) \abs{z_1} .
\end{align*}

\begin{figure}
    \centering
    \begin{overpic}
    [width=8cm, 
    tics=20]{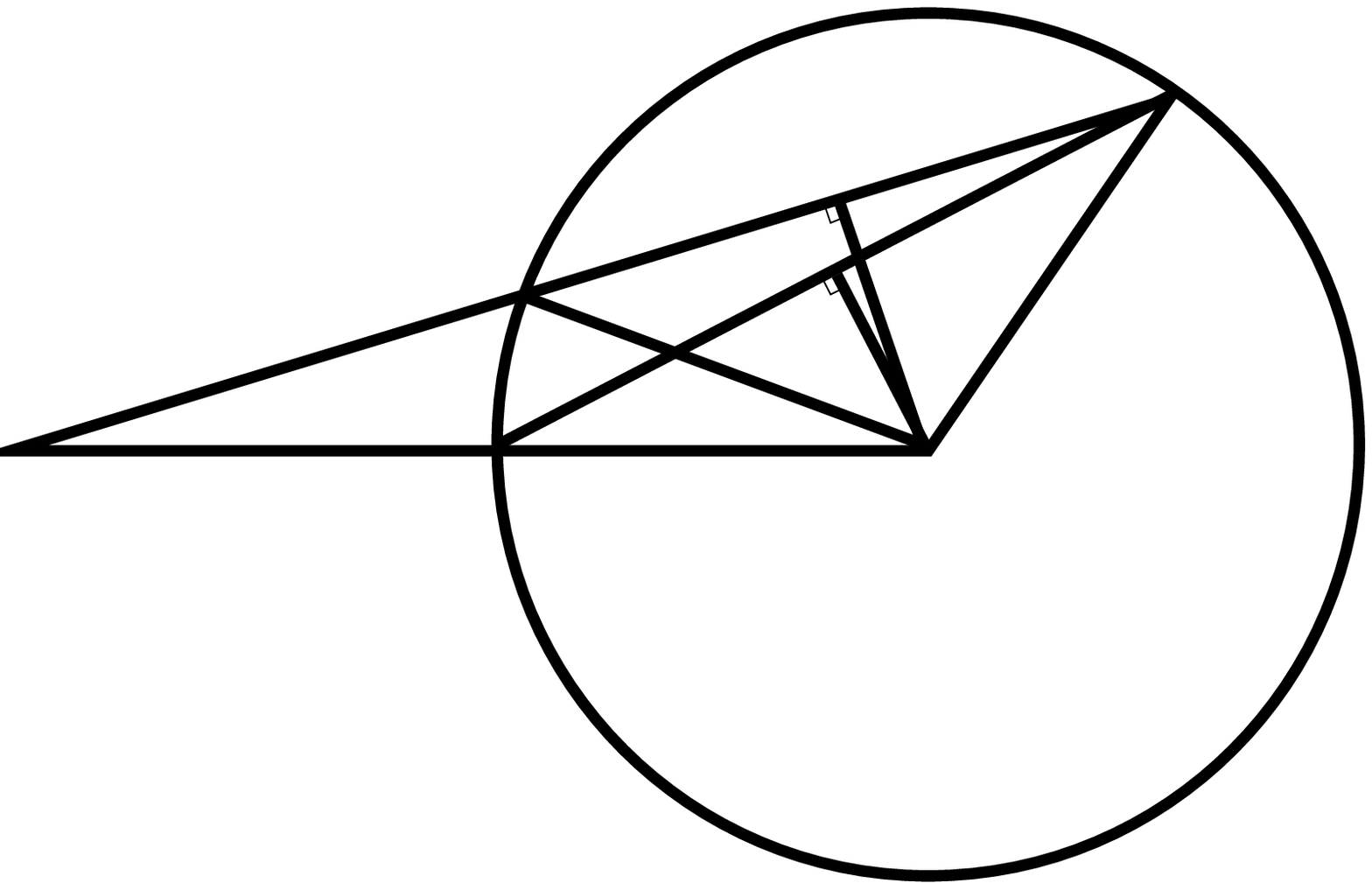}
    \put(-13,67){$O$}
    \put(77,100){$D$}
    \put(154,61){$Q$}
    \put(136,118){$A$}
    \put(131,88){$B$}
    \put(197,133){$C$}    
    \put(175,93){$z_1$}   
    \put(110,63){$z_2$}     
    \end{overpic}
    \caption{Proof of (\ref{eqTriangleIneqWithAngle}) of Lemma~\ref{lmArgumentIneqalities}.}
    \label{figTriangleIneqWithAngle}
\end{figure}

\begin{figure}
    \centering
    \begin{overpic}
    [width=5cm, 
    tics=20]{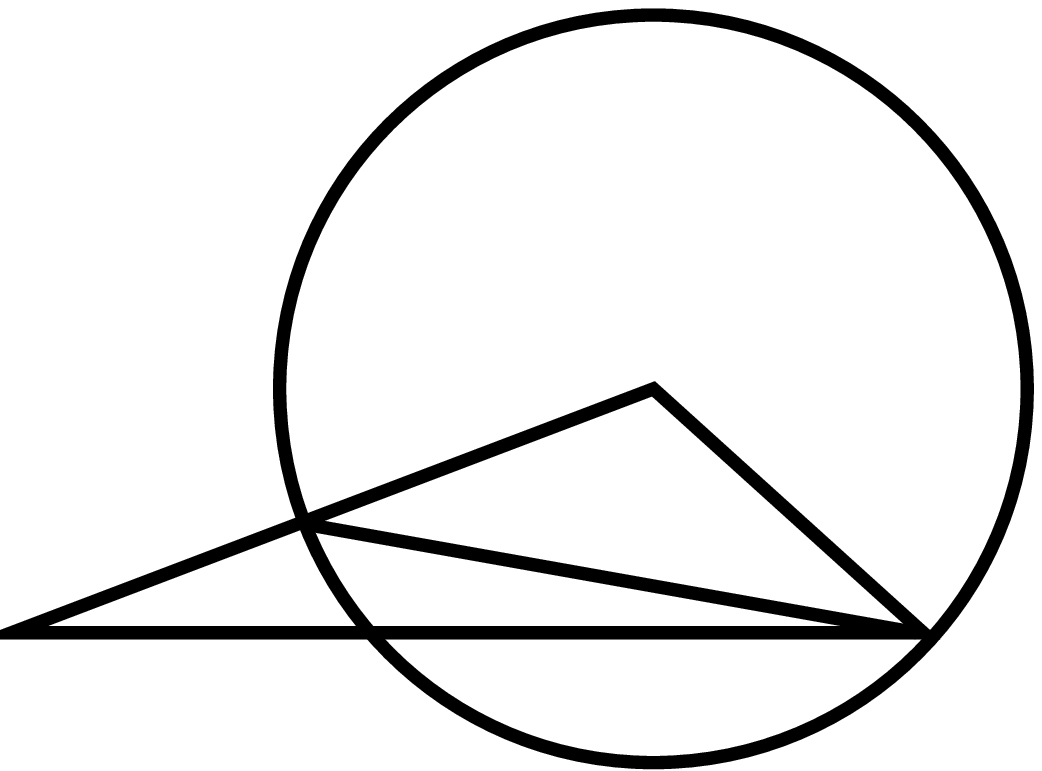}
    \put(-13,15){$B$}
    \put(25,35){$C$}
    \put(87,57){$O$}    
    \put(130,12){$A$}
    \end{overpic}
    \caption{Proof of (\ref{eqArgDist}) of Lemma~\ref{lmArgumentIneqalities}.}
    \label{figArgDist}
\end{figure}

Inequality (\ref{eqArgDist}) follows immediately from the following observation in elementary Euclidean plane geometry. As seen in Figure~\ref{figArgDist}, assume $A=z_1$ and $B=z_2$. Then $\abs{z_1-z_2}=\abs{AB}\geq \abs{AC} \geq \frac12 \abs{OA} \measuredangle AOC = \frac12 \abs{z_1} \Absbig{ \Arg\bigl( \frac{z_1}{z_2} \bigr) }$.    
\end{proof}

\begin{lemma}  \label{lmSupLeq2Inf}
Let $f$, $\CC$, $d$, $\alpha$, $\phi$, $s_0$ satisfy the Assumptions. We assume in addition that $f(\CC)\subseteq\CC$ and that $\phi$ satisfies the $\alpha$-strong non-integrability condition. Given $b\in\R$ with $\abs{b} \geq 2 s_0 +1$. Given $\c\in\{\b,\w\}$ and $h_\c\in K_{A\abs{b}} \bigl(X^0_\c,d\bigr)$. For each $m\geq  m(b) - M_0$ and each $m$-tile $X^m\in\X^m(f,\CC)$ with $X^m\subseteq X^0_\c$, we have
\begin{equation*}
 \sup \{ h_\c(x) \,|\, x\in X^m \}  \leq 2  \inf \{ h_\c(x) \,|\, x\in X^m \}.
\end{equation*}
\end{lemma}

Recall that the cone $K_{A\abs{b}} \bigl(X^0_\c,d\bigr)$ is defined in Definition~\ref{defCone}.

\begin{proof}
Consider arbitrary $x,x' \in X^m$. By Definition~\ref{defCone}, Lemma~\ref{lmCellBoundsBM}~(ii), (\ref{eqDefmb}), and (\ref{eqDef_epsilon1}),
\begin{align*}
\abs{h_\c(x) - h_\c(x')}   \leq & A\abs{b} (h_\c(x) + h_\c(x')) d(x,x')^\alpha  \\
                           \leq & A\abs{b} (h_\c(x) + h_\c(x'))  ( \diam_d(X^m)  )^\alpha \\
                           \leq & A\abs{b} (h_\c(x) + h_\c(x')) C\Lambda^{\alpha M_0 - \alpha m(b)} \\
                           \leq & A\abs{b} \frac{\epsilon_1}{\abs{b}}  \Lambda^{ \alpha M_0}(h_\c(x) + h_\c(x')) \\
                           \leq & \frac{1}{4}  (h_\c(x) + h_\c(x')),
\end{align*}
where $C\geq 1$ is a constant from Lemma~\ref{lmCellBoundsBM} depending only on $f$, $\CC$, and $d$. The lemma follows immediately.
\end{proof}

\begin{lemma}   \label{lmUHDichotomy}
Let $f$, $\CC$, $d$, $\alpha$, $\phi$, $s_0$ satisfy the Assumptions. We assume in addition that $f(\CC)\subseteq \CC$ and that $\phi$ satisfies the $\alpha$-strong non-integrability condition. Given $b\in\R$, $m\in\N$, $\c\in\{\b,\w\}$, $u_\c\in \Holder{\alpha}\bigl( \bigl( X^0_\c, d \bigr),\C \bigr)$, and $h_\c \in K_{A\abs{b}} \bigl( X^0_\c, d \bigr)$ such that $\abs{b} \geq 2 s_0 +1$, $m\geq N_1 + m(b)$, $\abs{u_\c(y)} \leq h_\c(y)$, and $\abs{u_\c(y) - u_\c(y')} \leq A\abs{b} (h_\c(y) + h_\c(y')) d(y,y')^\alpha$ whenever $y,y'\in X^0_\c$. Then for each $X^m \in \X^m(f,\CC)$ with $X^m\subseteq X^0_\c$, at least one of the following statements holds:

\begin{enumerate}
\smallskip
\item[(1)] $\abs{u_\c(x)} \leq \frac{3}{4} h_\c(x)$ for all $x\in X^m$.

\smallskip
\item[(2)] $\abs{u_\c(x)} \geq \frac{1}{4} h_\c(x)$ for all $x\in X^m$. 
\end{enumerate}
\end{lemma}

\begin{proof}
Assume that $\abs{u_\c(x_0)} \leq \frac{1}{4} h_\c(x_0)$ for some $x_0 \in X^m$. Then by Lemma~\ref{lmCellBoundsBM}~(ii), Lemma~\ref{lmSupLeq2Inf}, and (\ref{eqDefmb}), for each $x\in X^m$, 
\begin{align*}
\abs{u_\c(x)}  \leq & \abs{u_\c(x) - u_\c(x_0)} + \frac{1}{4} h_\c(x_0)
               \leq  A\abs{b} (h_\c(x) + h_\c(x_0))  ( \diam_d(X^m)  )^\alpha  + \frac{1}{4} h_\c(x_0) \\
               \leq & \Bigl(2 A\abs{b} C \Lambda^{-\alpha N_1 - \alpha m(b)}   +  \frac{1}{4} \Bigr)   \sup \{ h_\c(y) \,|\, y\in X^m \} \\
               \leq &  \Bigl(4 A \epsilon_1 \Lambda^{-\alpha N_1} + \frac{1}{2} \Bigr) h_\c(x) 
               \leq \frac{3}{4} h_\c(x),
\end{align*}
where $C\geq 1$ is a constant from Lemma~\ref{lmCellBoundsBM}. The last inequality follows from (\ref{eqDef_N1}) and the fact that $\epsilon_1\in (0,1)$ (see (\ref{eqDef_epsilon1})).
\end{proof}

\begin{lemma}   \label{lmUHDichotomySum}
Let $f$, $\CC$, $d$, $\alpha$, $\phi$, $s_0$ satisfy the Assumptions. We assume in addition that $f(\CC)\subseteq \CC$ and that $\phi$ satisfies the $\alpha$-strong non-integrability condition. Fix arbitrary $s \coloneqq a+\I b$ with $a,b\in \R$ satisfying $\abs{a-s_0}\leq s_0$ and $\abs{b}\geq b_0$. Given arbitrary $h_\b\in K_{A\abs{b}}\bigl( X^0_\b,d\bigr)$,  $h_\w\in K_{A\abs{b}}\bigl( X^0_\w,d\bigr)$, $u_\b\in\Holder{\alpha} \bigl( \bigl(X^0_\b,d\bigr), \C\bigr)$, and $u_\w\in\Holder{\alpha} \bigl( \bigl(X^0_\w,d\bigr), \C \bigr)$ satisfying the property that for each $\c\in\{\b,\w\}$, we have $\abs{u_\c (y)}\leq h_\c (y)$ and $\abs{ u_\c(y) - u_\c(y')} \leq A\abs{b} (h_\c(y) + h_\c(y')) d(y,y')^\alpha$ whenever $y,y'\in X^0_\c$.

Define the functions $Q_{\c,j}\: Y_\c^{M_0} \rightarrow \R$ for $j\in\{1,2\}$ and $\c\in\{\b,\w\}$ by
\begin{equation*}
    Q_{\c,j}(x) 
\coloneqq \frac{ \AbsBig{   \sum\limits_{k\in\{1,2\}} u_{ \varsigma (\c,k)} (\tau_k(x)) e^{S_{N_1} \wt{\minus s\phi} (\tau_k(x))} }   }
       { - 2\eta h_{\varsigma(\c,j)} (\tau_j(x)) e^{S_{N_1} \wt{\minus a\phi} (\tau_j(x))}
         + \sum\limits_{k\in\{1,2\}} h_{\varsigma(\c,k)} (\tau_k(x)) e^{S_{N_1} \wt{\minus a\phi} (\tau_k(x))}    }  ,
\end{equation*}
for $x\in Y_\c^{M_0}$, where we write $\tau_k \coloneqq \Bigl( f^{N_1} \big|_{X^{ N_1+M_0}_{\c,k} } \Bigr)^{-1}$ for $k\in\{1,2\}$, and we set $\varsigma(\c,j)\in\{\b,\w\}$ in such a way that $\tau_j \bigl( Y_{\c}^{M_0} \bigr) \subseteq X^0_{\varsigma(\c,j)}$ for $j\in \{1,2\}$.

Then for each $\c\in\{\b,\w\}$ and each $X\in \mathfrak{C}_b$ with $X\subseteq Y_\c^{M_0}$, we have
\begin{equation*}
\min\Bigl\{     \Norm{Q_{\c,j}}_{\CCC^0(\mathfrak{X}_i(X))}  \,\Big|\,  i,j\in\{1,2\} \Bigr\}   \leq 1.
\end{equation*}
\end{lemma}

\begin{proof}
Fix arbitrary $\c\in\{ \b, \w \}$ and $X\in \mathfrak{C}_b$ with $X\subseteq Y_\c^{M_0}$. For typographic reason, we denote in this proof
\begin{equation}  \label{eqPflmUHDichotomySum_uhe}
u(i,x) \coloneqq u_{\varsigma(\c,i)} ( \tau_i(x) ), \qquad 
h(i,x) \coloneqq h_{\varsigma(\c,i)} ( \tau_i(x) ), \qquad
e(i,x) \coloneqq e^{S_{N_1} \wt{\minus s\phi} (\tau_i(x))}
\end{equation}
for $i\in\{1,2\}$ and $x \in X$.

If $\abs{ u(j,\cdot) } \leq \frac{3}{4} h(j,\cdot)$ on $X$, for some $j\in\{1,2\}$, then $ \Norm{Q_{\c,j}}_{\CCC^0(\mathfrak{X}_i(X))} \leq 1$ for all $i\in\{1,2\}$. Thus, by Lemma~\ref{lmUHDichotomy}, we can assume that 
\begin{equation}  \label{eqPflmUHDichotomySum_U14h}
\abs{ u(k,x) } \geq \frac{1}{4} h(k,x)  \qquad \text{for all } x\in  X \text{ and }k\in\{1,2\}.
\end{equation}

We define a function $\Theta\: X \rightarrow (-\pi,\pi]$ by setting
\begin{equation}  \label{eqPflmUHDichotomySum_Theta}
\Theta(x) \coloneqq \Arg  \biggl(   \frac{  u(1,x)  e(1,x) }{  u(2,x)  e(2,x)  }   \biggr)
\end{equation}
for $x\in X$.

\smallskip

We first claim that for all $x,y\in X$, we have
\begin{equation}   \label{eqPflmUHDichotomySum_ArgUxy}
\Absbigg{    \Arg \biggl(  \frac{   u(1,x)  /  u(2,x)    }{    u(1,y)  /  u(2,y)  }  \biggr) } 
\leq 16 A \epsilon_1  \Lambda^{-\alpha N_1}
\leq \frac{\pi}{16}
\end{equation}
and
\begin{equation}   \label{eqPflmUHDichotomySum_ArgEXPxy}
  \abs{b} \abs{ -S_{N_1} \phi (\tau_1(x)) + S_{N_1} \phi (\tau_2(x)) + S_{N_1} \phi (\tau_1(y)) - S_{N_1} \phi (\tau_2(y))  } 
\leq \frac{\pi}{16}.
\end{equation}

Indeed, by (\ref{eqArgumentTriangleIneq}) and (\ref{eqArgDist}) in Lemma~\ref{lmArgumentIneqalities}, (\ref{eqPflmUHDichotomySum_uhe}), (\ref{eqPflmUHDichotomySum_U14h}), Lemma~\ref{lmCellBoundsBM}~(ii), Lemma~\ref{lmSupLeq2Inf}, (\ref{eqDefCb}), and (\ref{eqDefmb}),
\begin{align*}
             \Absbigg{    \Arg \biggl(  \frac{   u(1,x)  /  u(2,x)    }{    u(1,y)  /  u(2,y)  }  \biggr) }   
& \leq  \Absbigg{ \Arg \biggl(   \frac{ u(1,x) }{ u(1,y) } \biggr)    }   +\Absbigg{ \Arg \biggl(   \frac{ u(2,x) }{ u(2,y) } \biggr)    }  \\
& \leq  \sum\limits_{j\in\{1,2\}} \frac{ 2 \abs{ u(j,x)  - u(j,y)   } }   { \inf \{ \abs{ u(j,z) }  \,|\, z\in X \}  }  \\
& \leq  \sum\limits_{j\in\{1,2\}} \frac{ 2 A\abs{b}  ( h(j,x)  + h(j,y) )  } { \inf \{  h(j,z)    \,|\, z\in X \}  }     d(\tau_j(x),\tau_j(y))^\alpha  \\ 
& \leq  4 A\abs{b} \sum\limits_{j\in\{1,2\}} \frac{ \sup \{  h(j,z)    \,|\, z\in X \}  } { \inf \{  h(j,z)   \,|\, z\in X \}  }   C\Lambda^{-\alpha N_1 - \alpha m(b)}  \\
& \leq   16 A\abs{b} \frac{\epsilon_1}{\abs{b}} \Lambda^{-\alpha N_1}
\leq    \frac{\pi}{16},
\end{align*} 
where $C\geq 1$ is a constant from Lemma~\ref{lmCellBoundsBM}. The last inequality follows from the fact that $N_1\geq \bigl\lceil \frac{1}{\alpha} \log_\Lambda \bigl(   2^8 A  \bigr) \bigr\rceil $ (see (\ref{eqDef_N1})) and the fact that $\epsilon_1\in (0,1)$ (see (\ref{eqDef_epsilon1})). We have now verified (\ref{eqPflmUHDichotomySum_ArgUxy}). To show (\ref{eqPflmUHDichotomySum_ArgEXPxy}), we note that by Lemma~\ref{lmCellBoundsBM}~(ii), (\ref{eqDefCb}), (\ref{eqDefmb}), and (\ref{eqDef_epsilon1}),
\begin{align*}
     & \abs{b} \abs{ -S_{N_1} \phi (\tau_1(x)) + S_{N_1} \phi (\tau_2(x)) + S_{N_1} \phi (\tau_1(y)) - S_{N_1} \phi (\tau_2(y))  }  \\
&\qquad \leq  \abs{b} \delta_0^{-1} d(x,y)^\alpha 
\leq   \abs{b} \delta_0^{-1} ( \diam_d(X) )^\alpha
\leq   \abs{b} \delta_0^{-1}  C^\alpha \Lambda^{-\alpha m(b)}    
\leq   \delta_0^{-1} \epsilon_1
\leq   \frac{\pi}{16}.
\end{align*}

The claim is now verified.

\smallskip

We will choose $i_0\in\{1,2\}$, by separate discussions in the following two cases, in such a way that
\begin{equation} \label{eqPflmUHDichotomySum_ThetaLB}
\abs{\Theta(x)} \geq 16 \eta^{\frac12} \qquad \text{for all } x\in \mathfrak{X}_{i_0}(X).
\end{equation}

\smallskip

\emph{Case 1.} $\abs{\Theta(y)} \geq \frac{\pi}{4}$ for some $y\in X$. Then by (\ref{eqArgumentTriangleIneq}) in Lemma~\ref{lmArgumentIneqalities}, (\ref{eqPflmUHDichotomySum_uhe}), (\ref{eqPflmUHDichotomySum_Theta}), (\ref{eqPflmUHDichotomySum_ArgUxy}), (\ref{eqPflmUHDichotomySum_ArgEXPxy}), and the fact that $\eta \in \bigl( 0, 2^{-12} \bigr)$ (see (\ref{eqDef_eta})), for each $x\in X$,
\begin{align*}
\abs{\Theta(x)}   \geq &   \abs{\Theta(y)}    -  \Absbigg{ \Arg \biggl(  \frac{ ( u(1,y) e(1,y) ) / ( u(2,y) e(2,y) ) }{    ( u(1,x) e(1,x) ) / ( u(2,x) e(2,x) ) }  \biggr)} \\
                  \geq &  \frac{\pi}{4}  -  \Absbigg{ \Arg \biggl(  \frac{   u(1,y)  /  u(2,y)  }{    u(1,x)   /   u(2,x)  }  \biggr)} 
                                                      -  \Absbigg{ \Arg \biggl(  \frac{   e(1,y)  /  e(2,y)  }{    e(1,x)   /   e(2,x)  }  \biggr)}  \\
                  \geq &  \frac{\pi}{4} - \frac{\pi}{16} - \frac{\pi}{16}
                  \geq    \frac{\pi}{8}
                  \geq    16\eta^{\frac12}.
\end{align*}
We can choose $i_0=1$ in this case.

\smallskip

\emph{Case 2.} $\abs{\Theta(z)} < \frac{\pi}{4}$ for all $z\in X$. Then by (\ref{eqArgumentTriangleIneq}) in Lemma~\ref{lmArgumentIneqalities}, (\ref{eqPflmUHDichotomySum_uhe}), (\ref{eqPflmUHDichotomySum_Theta}), (\ref{eqPflmUHDichotomySum_ArgUxy}), (\ref{eqPflmUHDichotomySum_ArgEXPxy}), $\abs{b} \geq b_0 \geq 1$ (see (\ref{eqDef_b0})), (\ref{eqSNIBounds}), (\ref{eqDefCb}), and (\ref{eqDefmb}), for each $x\in \mathfrak{X}_1(X)$ and each $y\in \mathfrak{X}_2(X)$,
\begin{align*}
&        \abs{ \Theta (x) - \Theta(y) } \\
&\qquad    =  \Absbigg{ \Arg \biggl(  \frac{ ( u(1,x) e(1,x) ) / ( u(2,x) e(2,x) ) }{    ( u(1,y) e(1,y) ) / ( u(2,y) e(2,y) ) }  \biggr)}\\
&\qquad\geq   \Absbigg{ \Arg \biggl(  \frac{   e(1,x)  /  e(2,x)  }{    e(1,y)   /   e(2,y)  }  \biggr)}  -  \Absbigg{ \Arg \biggl(  \frac{   u(2,y)  /  u(1,y)  }{    u(2,x)   /   u(1,x)  }  \biggr)}  \\ 
&\qquad\geq      \abs{b} \abs{ - S_{N_1} \phi (\tau_1(x)) + S_{N_1} \phi (\tau_2(x)) + S_{N_1} \phi (\tau_1(y)) - S_{N_1} \phi (\tau_2(y))  } 
       -   16 A \epsilon_1  \Lambda^{-\alpha N_1}   \\
&\qquad\geq      \abs{b} \delta_0 d(x,y)^\alpha -  16  A \epsilon_1  \Lambda^{-\alpha N_1} 
\geq       \abs{b} \delta_0 \bigl(10^{-1} C^{-1} \Lambda^{-m(b)} \bigr)^\alpha  -   16 A \epsilon_1  \Lambda^{-\alpha N_1}   \\
&\qquad\geq       \delta_0 \frac{1}{10 \Lambda}  C^{-2} \epsilon_1  -   16 A \epsilon_1  \Lambda^{-\alpha N_1}   
\geq    \frac{\delta_0 \epsilon_1}{20\Lambda C^2},
\end{align*}
where the last inequality follows from the observation that $ 16 A \Lambda^{-\alpha N_1} \leq \frac{\delta_0 }{20\Lambda C^2} $ since $N_1 \geq \bigl\lceil \frac{1}{\alpha} \log_\Lambda \bigl(  \frac{320 A \Lambda C^2}{\delta_0} \bigr) \bigr\rceil$ (see (\ref{eqDef_N1})).

\smallskip

We now claim that at least one of the following statements holds:
\begin{enumerate}
\smallskip
\item[(1)] $\abs{\Theta(x)} \geq \frac{\delta_0 \epsilon_1}{80\Lambda C^2}$ for all $x\in \mathfrak{X}_1(X)$.

\smallskip
\item[(2)] $\abs{\Theta(y)} \geq \frac{\delta_0 \epsilon_1}{80\Lambda C^2}$ for all $y\in \mathfrak{X}_2(X)$.
\end{enumerate}

Indeed, assume that statement (1) fails, then there exists $x_0\in\mathfrak{X}_1(X)$ such that $ \abs{\Theta(x_0)} \leq \frac{\delta_0 \epsilon_1}{80\Lambda C^2}$. Hence for all $y\in \mathfrak{X}_2(X)$,
\begin{equation*}
     \abs{\Theta(y)}  
\geq \abs{\Theta(y)-\Theta(x_0)} - \abs{\Theta(x_0)} 
\geq \frac{\delta_0 \epsilon_1}{20\Lambda C^2} - \frac{\delta_0 \epsilon_1}{80\Lambda C^2}
\geq \frac{\delta_0 \epsilon_1}{80\Lambda C^2}.
\end{equation*}
The claim is now verified.

\smallskip

Thus we can fix $i_0\in\{1,2\}$ such that $\abs{\Theta(x)} \geq \frac{\delta_0 \epsilon_1}{80\Lambda C^2} \geq 16\eta^{\frac{1}{2}}$ (see (\ref{eqDef_eta})) for all $x\in \mathfrak{X}_{i_0} (X)$ in this case.

\smallskip

By Lemma~\ref{lmSnPhiBound}, (\ref{eqPflmUHDichotomySum_uhe}), Lemma~\ref{lmSupLeq2Inf}, Lemma~\ref{lmCellBoundsBM}~(ii), (\ref{eqDefCb}), and (\ref{eqDefmb}), for arbitrary $x,y\in\mathfrak{X}_{i_0}(X)$ and $j\in\{1,2\}$,
\begin{align}   \label{eqPflmUHDichotomySum_QuotHBound}
               \AbsBigg{ \frac{   h(j,x)  \exp\bigl( S_{N_1} \wt{- a\phi} \bigl(\tau_j(x)\bigr)  \bigr) }
                                         {   h(j,y)  \exp\bigl( S_{N_1} \wt{- a\phi} \bigl(\tau_j(y)\bigr)  \bigr) }   }  
& \leq  \Absbigg{ \frac{   h(j,x)  } {   h(j,y) }  }  e^{  \abs{ S_{N_1} \wt{ \minus a\phi}  (\tau_j(x) ) - S_{N_1} \wt{ \minus a\phi}  (\tau_j(y) ) } }  \notag  \\
& \leq  2 \exp \biggl(  C_0 \Hseminormbig{\alpha,\, (S^2,d)}{\wt{- a\phi}}  \frac{ d(x,y)^\alpha  }  {  1- \Lambda^{-\alpha} }    \biggr)  \notag \\
& \leq  2 \exp \biggl(  C_0 \Hseminormbig{\alpha,\, (S^2,d)}{\wt{- a\phi}}  \frac{ C^\alpha \Lambda^{-\alpha m(b)}  }  {  1- \Lambda^{-\alpha} } \biggr) \\
& \leq  2 \exp \biggl(  \frac{\epsilon_1 }{\abs{b}}  C_0 \Hseminormbig{\alpha,\, (S^2,d)}{\wt{- a\phi}}  \frac{ 1 }  {  1- \Lambda^{-\alpha} }  \biggr) 
\leq   8,   \notag
\end{align}
where the last inequality follows from (\ref{eqDef_b0}), (\ref{eqT0Bound_sctDolgopyat}), the condition that $\abs{b} \geq b_0$, and the fact that $\epsilon_1\in(0,1)$ (see (\ref{eqDef_epsilon1})).

We fix $k_0\in\{1,2\}$ such that
\begin{equation}  \label{eqPflmUHDichotomySum_k0}
  \inf \{  h(j,x) e(j,x)               \,|\, x\in\mathfrak{X}_{i_0}(X),\, j\in\{1,2\} \} 
=\inf \{  h(k_0,x) e(k_0,x)     \,|\, x\in\mathfrak{X}_{i_0}(X) \} .  
\end{equation}   

Hence by (\ref{eqTriangleIneqWithAngle}) in Lemma~\ref{lmArgumentIneqalities}, (\ref{eqPflmUHDichotomySum_Theta}), (\ref{eqPflmUHDichotomySum_uhe}), (\ref{eqPflmUHDichotomySum_ThetaLB}), (\ref{eqPflmUHDichotomySum_k0}), (\ref{eqDefWideTildePsiComplex}), and (\ref{eqPflmUHDichotomySum_QuotHBound}), for each $x\in\mathfrak{X}_{i_0} (X)$, we have
\begin{align*}
&                     \abs{  u(1,x) e(1,x)  +  u(2,x) e(2,x)  }  \\
&\qquad \leq   - \frac{\Theta^2(x)}{16}  \min\limits_{k\in\{1,2\}}   \{  \abs{  u(k,x) e(k,x) } \}    + \sum\limits_{j\in\{1,2\}}   \abs{  u(j,x) e(j,x) }     \\
&\qquad\leq   - \frac{\Theta^2(x)}{16}  \min\limits_{k\in\{1,2\}}  \Bigl\{  h(k,x)   e^{ S_{N_1} \wt{\minus a\phi}  (\tau_k(x) )   } \Bigr\}
                          + \sum\limits_{j\in\{1,2\}}          h(j,x)   e^{ S_{N_1} \wt{\minus a\phi}  (\tau_j(x) )   }     \\     
&\qquad\leq  -16\eta   \inf  \Bigl\{  h(k_0,y)   e^{ S_{N_1} \wt{\minus a\phi}  (\tau_{k_0}(y) ) }  \,\Big|\,y\in\mathfrak{X}_{i_0}(X) \Bigr\}
                          + \sum\limits_{j\in\{1,2\}}          h(j,x)  e^{ S_{N_1} \wt{\minus a\phi}  (\tau_j(x) )   }     \\   
&\qquad\leq  -2 \eta   h(k_0,x)   e^{ S_{N_1} \wt{\minus a\phi}  (\tau_{k_0}(x) ) }  + \sum\limits_{j\in\{1,2\}}  h(j,x)  e^{ S_{N_1} \wt{\minus a\phi}  (\tau_j(x) )   }    .     
\end{align*}
Therefore, we conclude that $\norm{Q_{\c,k_0}}_{\CCC^0  ( \mathfrak{X}_{i_0} (X)  ) }  \leq 1$. 
\end{proof}

\begin{prop}    \label{propDolgopyatOperator}
Let $f$, $\CC$, $d$, $\alpha$, $\phi$, $s_0$ satisfy the Assumptions. We assume in addition that $f(\CC)\subseteq \CC$ and that $\phi$ satisfies the $\alpha$-strong non-integrability condition. We use the notation in this section.

There exist numbers $a_0\in (0, s_0 )$ and $\rho\in (0,1)$ such that for all $s \coloneqq a+\I b$ with $a,b\in\R$ satisfying $\abs{a-s_0}\leq a_0$ and $\abs{b}\geq b_0$, there exists a subset $\mathcal{E}_s\subseteq \mathcal{F}$ of the set $\mathcal{F}$ of all subsets of $\{1,2\}\times\{1,2\}\times\mathfrak{C}_b$ with full projection such that the following statements are satisfied:

\begin{enumerate}
\smallskip
\item[(i)] The cone $K_{A\abs{b}}\bigl(X^0_\b,d\bigr)\times K_{A\abs{b}}\bigl(X^0_\w,d\bigr)$ is invariant under $\MM_{J,\minus s,\phi}$, i.e.,
\begin{equation}   \label{eqDolgopyatOpStableCone}
\MM_{J,\minus s,\phi} \bigl( K_{A\abs{b}}\bigl(X^0_\b,d\bigr)\times K_{A\abs{b}}\bigl(X^0_\w,d\bigr) \bigr) 
           \subseteq K_{A\abs{b}}\bigl(X^0_\b,d\bigr)\times K_{A\abs{b}}\bigl(X^0_\w,d\bigr),
\end{equation}
for all $J\in \mathcal{F}$.

\smallskip
\item[(ii)] For all $J\in\mathcal{F}$, $h_\b \in K_{A\abs{b}}\bigl(X^0_\b,d\bigr)$, and $h_\w \in  K_{A\abs{b}}\bigl(X^0_\w,d\bigr)$, we have
\begin{equation}  \label{eqDolgopyatOpL2Shrinking}
                   \sum\limits_{ \c\in\{\b,\w\} }    \int_{X^0_\c} \! \abs{ \pi_\c ( \MM_{J,\minus s,\phi} (h_\b,h_\w) ) }^2 \,\mathrm{d}\mu_{\minus s_0\phi}   
\leq   \rho \sum\limits_{ \c\in\{\b,\w\} }    \int_{X^0_\c} \! \abs{ h_\c }^2 \,\mathrm{d}\mu_{\minus s_0\phi}.
\end{equation}

\smallskip
\item[(iii)] Given arbitrary $h_\b\in K_{A\abs{b}}\bigl( X^0_\b,d\bigr)$,  $h_\w\in K_{A\abs{b}}\bigl( X^0_\w,d\bigr)$, $u_\b\in\Holder{\alpha} \bigl( \bigl(X^0_\b,d\bigr), \C\bigr)$, and $u_\w\in\Holder{\alpha} \bigl( \bigl(X^0_\w,d\bigr), \C \bigr)$ satisfying the property that for each $\c\in\{\b,\w\}$, we have $\abs{u_\c (y)}\leq h_\c (y)$ and $\abs{ u_\c(y) - u_\c(y')} \leq A\abs{b} (h_\c(y) + h_\c(y')) d(y,y')^\alpha$ whenever $y,y'\in X^0_\c$. Then the following statement is true:

\smallskip

There exists $J\in\mathcal{E}_s$ such that
\begin{equation}   \label{eqDolgopyatOpPtwiseBound}
\AbsBig{ \pi_\c \Bigl(\RRR_{\wt{\minus s\phi}}^{N_1+M_0} (u_\b,u_\w) \Bigr)  (x) }    \leq \pi_\c (\MM_{J,\minus s,\phi} (h_\b,h_\w) ) (x)
\end{equation}
and
\begin{align}   \label{eqDolgopyatOpLipBound}
      & \AbsBig{   \pi_\c \Bigl(\RRR_{\wt{\minus s\phi}}^{N_1+M_0} (u_\b,u_\w) \Bigr)  (x)
                 - \pi_\c \Bigl(\RRR_{\wt{\minus s\phi}}^{N_1+M_0} (u_\b,u_\w) \Bigr)  (x') }   \\
\leq &  A\abs{b} (  \pi_\c (\MM_{J,\minus s,\phi} (h_\b,h_\w) ) (x) + \pi_\c (\MM_{J,\minus s,\phi} (h_\b,h_\w) ) (x') ) d(x,x')^\alpha\notag 
\end{align}
for each $\c\in\{\b,\w\}$ and all $x,x'\in X^0_\c$.
\end{enumerate}

\end{prop}

\begin{proof}
For typographical convenience, we write $\iota \coloneqq N_1+M_0$ in this proof. 

We fix an arbitrary number $s=a+\I b$ with $a,b\in\R$ satisfying $\abs{a-s_0}\leq  s_0$ and $\abs{b}\geq b_0$.

\smallskip

(i) Without loss of generality, it suffices to show that for each $J\in \mathcal{F}$,
\begin{equation*}
\pi_\b \bigl( \MM_{J,\minus s,\phi} \bigl( K_{A\abs{b}}\bigl(X^0_\b,d\bigr)\times K_{A\abs{b}}\bigl(X^0_\w,d\bigr) \bigr)\bigr) 
           \subseteq K_{A\abs{b}}\bigl(X^0_\b,d\bigr).
\end{equation*}

Fix $J\in \mathcal{F}$, functions $h_\b \in K_{A\abs{b}}\bigl(X^0_\b,d\bigr)$, $h_\w \in K_{A\abs{b}}\bigl(X^0_\w,d\bigr)$, and points $x,x'\in X^0_\b$ with $x\neq x'$. For each $X^\iota \in \X^\iota_\b$, denote $y_{X^\iota} \coloneqq (f^\iota|_{X^\iota})^{-1}(x)$ and $y'_{X^\iota} \coloneqq (f^\iota|_{X^\iota})^{-1}(x')$.

Then by Definition~\ref{defDolgopyatOperator}, (\ref{eqSplitRuelleCoordinateFormula}) in Lemma~\ref{lmSplitRuelleCoordinateFormula}, Definition~\ref{defSplitRuelle}, and (\ref{eqDefWideTildePsiComplex}),
\begin{align*}
&             \abs{ \pi_\b ( \MM_{J,\minus s,\phi} (h_\b,h_\w) ) (x) - \pi_\b ( \MM_{J,\minus s,\phi} (h_\b,h_\w) ) (x') } \\
&\qquad =     \Absbigg{  \RR_{\wt{\minus a\phi},X^0_\b,X^0_\b}^{(\iota)} \bigl(h_\b \beta_J|_{X^0_\b} \bigr) (x) 
                       + \RR_{\wt{\minus a\phi},X^0_\b,X^0_\w}^{(\iota)} \bigl(h_\w \beta_J|_{X^0_\w} \bigr) (x) \\
&\qquad\quad           - \RR_{\wt{\minus a\phi},X^0_\b,X^0_\b}^{(\iota)} \bigl(h_\b \beta_J|_{X^0_\b} \bigr) (x')
                       - \RR_{\wt{\minus a\phi},X^0_\b,X^0_\w}^{(\iota)} \bigl(h_\w \beta_J|_{X^0_\w} \bigr) (x')  } \\
&\qquad\leq   \sum\limits_{\c\in\{\b,\w\}} \sum\limits_{\substack{X^\iota\in\X^\iota_\b\\X^\iota\subseteq X^0_\c}}
              \AbsBig{   h_\c( y_{X^\iota} )  \beta_J ( y_{X^\iota} )  e^{S_\iota\wt{\minus a\phi}(y_{X^\iota})}  
                      -  h_\c( y'_{X^\iota} ) \beta_J ( y'_{X^\iota} ) e^{S_\iota\wt{\minus a\phi}(y'_{X^\iota})}   } \\
&\qquad\leq  \sum\limits_{\c\in\{\b,\w\}} \sum\limits_{\substack{X^\iota\in\X^\iota_\b\\X^\iota\subseteq X^0_\c}}
                 \Absbig{  h_\c( y_{X^\iota} )   \beta_J ( y_{X^\iota} ) - h_\c( y'_{X^\iota} )  \beta_J ( y'_{X^\iota} ) }
                  e^{S_\iota\wt{\minus a\phi}(y'_{X^\iota})}  \\
&\qquad\quad + \sum\limits_{\c\in\{\b,\w\}} \sum\limits_{\substack{X^\iota\in\X^\iota_\b\\X^\iota\subseteq X^0_\c}}
                 h_\c( y_{X^\iota} )   \beta_J ( y_{X^\iota} )  
                 \AbsBig{ e^{S_\iota\wt{\minus a\phi}(y_{X^\iota})} - e^{S_\iota\wt{\minus a\phi}(y'_{X^\iota})} }   \\
&\qquad\leq    \sum\limits_{\c\in\{\b,\w\}} \sum\limits_{\substack{X^\iota\in\X^\iota_\b\\X^\iota\subseteq X^0_\c}}
                 h_\c( y_{X^\iota} ) \Absbig{    \beta_J ( y_{X^\iota} ) - \beta_J ( y'_{X^\iota} ) }
                 e^{S_\iota\wt{\minus a\phi}(y_{X^\iota})} e^{\Absbig{ S_\iota\wt{\minus a\phi}(y'_{X^\iota}) - S_\iota\wt{\minus a\phi}(y_{X^\iota})}}\\
&\qquad\quad    + \sum\limits_{\c\in\{\b,\w\}} \sum\limits_{\substack{X^\iota\in\X^\iota_\b\\X^\iota\subseteq X^0_\c}}  
                  \Absbig{  h_\c( y_{X^\iota} ) - h_\c( y'_{X^\iota} ) }   \beta_J ( y'_{X^\iota} ) 
                  e^{S_\iota\wt{\minus a\phi}(y'_{X^\iota})}\\
&\qquad\quad    + \sum\limits_{\c\in\{\b,\w\}} \sum\limits_{\substack{X^\iota\in\X^\iota_\b\\X^\iota\subseteq X^0_\c}}
                  h_\c( y_{X^\iota} )  \beta_J ( y'_{X^\iota} )  e^{S_\iota\wt{\minus a\phi}(y_{X^\iota})} 
                  \AbsBig{1 - e^{ S_\iota\wt{\minus a\phi}(y'_{X^\iota}) - S_\iota\wt{\minus a\phi}(y_{X^\iota}) } }  .
\end{align*}

By Lemma~\ref{lmSnPhiBound}, Lemma~\ref{lmBetaJ}, Lemma~\ref{lmMetricDistortion}, and Lemma~\ref{lmExpToLiniear}, the right-hand side of the last inequality is

\begin{align*}
\leq  &  \exp \biggl( \frac{ T_0 C_0 \bigl(\diam_d(S^2)\bigr)^\alpha}{1-\Lambda^{-\alpha}}  \biggr)  
         \Biggl(  \sum\limits_{\c\in\{\b,\w\}} \sum\limits_{\substack{X^\iota\in\X^\iota_\b\\X^\iota\subseteq X^0_\c}}
                      h_\c(y_{X^\iota}) L_\beta C_0^\alpha \Lambda^{-\iota \alpha} d(x,x')^\alpha  e^{S_\iota\wt{\minus a\phi}(y_{X^\iota})}   \\
      &         + \sum\limits_{\c\in\{\b,\w\}} \sum\limits_{\substack{X^\iota\in\X^\iota_\b\\X^\iota\subseteq X^0_\c}}
                      A\abs{b}  \Bigl(    h_\c(y_{X^\iota})   e^{S_\iota\wt{\minus a\phi}(y_{X^\iota})}  
                                       +  h_\c(y'_{X^\iota})  e^{S_\iota\wt{\minus a\phi}(y'_{X^\iota})}   \Bigr) 
           C_0^\alpha \Lambda^{-\alpha\iota} d(x,x')^\alpha    \Biggr)  \\
      & +C_{10} T_0 d(x,x')^\alpha   \sum\limits_{\c\in\{\b,\w\}} 
           \RR_{\wt{\minus a\phi},X^0_\b,X^0_\c}^{(\iota)}  \bigl(h_\c \beta_J|_{X^0_\c} \bigr) (x),
\end{align*}
where $C_0>1$ is a constant from Lemma~\ref{lmMetricDistortion} depending only on $f$, $\CC$, and $d$; $L_\beta$ is a constant defined in (\ref{eqDefLipBoundBetaJ}) in Lemma~\ref{lmBetaJ}; $T_0>0$ is a constant defined in (\ref{eqDefT0}) giving an upper bound of $\Hseminormbig{\alpha,\, (S^2,d)}{\wt{ - a\phi}}$ by Lemma~\ref{lmBound_aPhi} (c.f.\ (\ref{eqT0Bound})); and $C_{10} \coloneqq C_{10}(f,\CC,d,\alpha,T_0)>1$ is a constant defined in (\ref{eqDefC10}) in Lemma~\ref{lmExpToLiniear}. Both $T_0$ and $C_{10}$ depend only on $f$, $\CC$, $d$, $\phi$, and $\alpha$. Thus by (\ref{eqDefC10}), (\ref{eqBetaJBounds}) and (\ref{eqDefLipBoundBetaJ}) in Lemma~\ref{lmBetaJ}, Definition~\ref{defDolgopyatOperator}, (\ref{eqDefmb}), and the calculation above, we get
\begin{align*}
&              \frac{ \abs{ \pi_\b ( \MM_{J,\minus s,\phi} (h_\b,h_\w) ) (x) - \pi_\b ( \MM_{J,\minus s,\phi} (h_\b,h_\w) ) (x') }   }
                     { A\abs{b}  ( \pi_\b ( \MM_{J,\minus s,\phi} (h_\b,h_\w) ) (x) + \pi_\b ( \MM_{J,\minus s,\phi} (h_\b,h_\w) ) (x') )  
                       d(x,x')^\alpha  } \\
&\qquad\leq    \frac{C_{10}}{A\abs{b} (1-\eta)} ( L_\beta   + A\abs{b} ) \Lambda^{-\alpha\iota}
                 + \frac{C_{10} T_0}{ A\abs{b} }  \\
&\qquad\leq    \frac{C_{10}}{1-\eta} \biggl(
                   \frac{40 \varepsilon^{-\alpha} }{A\abs{b}} C  \Lambda^{2\alpha m_0 +1} \frac{ C \abs{b} }{\epsilon_1}
                           (\LIP_d(f))^{\alpha N_1} \eta   + 1 \biggr) \Lambda^{-\alpha(N_1+M_0)} + \frac{C_{10} T_0}{ A\abs{b} } \\
&\qquad\leq    1.      
\end{align*}
The last inequality follows from the observations that $\frac{C_{10}T_0}{A}\leq \frac13$ (see (\ref{eqDef_A})), that 
\begin{equation*}
\Lambda^{-\alpha (N_1 + M_0)} \leq \frac{1}{4 C_{10}} \leq \frac{1}{3}  \frac{1-\eta}{C_{10}}
\end{equation*} 
(see (\ref{eqDef_N1}) and (\ref{eqDef_eta})), and that by (\ref{eqDef_eta}) ,
\begin{equation*}  
 \frac{ 40 \varepsilon^{-\alpha} C_{10} C^2 \eta}{A \epsilon_1 (1-\eta)} \Lambda^{-\alpha N_1 - \alpha M_0 + 2 \alpha m_0 +1} (\LIP_d(f))^{\alpha N_1} \leq \frac13.
\end{equation*}

\smallskip

(ii) Fix $J\in \mathcal{F}$ and two functions $h_\b \in K_{A\abs{b}} \bigl(X^0_\b, d \bigr)$, $h_\w \in K_{A\abs{b}} \bigl(X^0_\w, d \bigr)$.

We first establish that 
\begin{equation}   \label{eqPfpropDolgopyatOperator_HolderIneq}
      \bigl( \pi_\c \bigl( \MM_{J,\minus s,\phi} (h_\b,h_\w) \bigr) (x) \bigr)^2
\leq         \pi_\c \Bigl( \RRR^\iota_{\wt{\minus a\phi}} \bigl( h_\b^2,h_\w^2\bigr) \Bigr) (x) 
     \cdot   \pi_\c \Bigl( \RRR^\iota_{\wt{\minus a\phi}} \Bigl( \bigl( \beta_J|_{X^0_\b} \bigr)^2, \bigl( \beta_J|_{X^0_\w} \bigr)^2 \Bigr) \Bigr) (x) 
\end{equation} 
for $\c \in \{ \b, \w \}$ and $x\in X^0_\c$. Indeed, it suffices to verify the case $\c=\b$ as shown below. Fix an arbitrary point $x\in X^0_\b$. For each $X^\iota \in \X_\b^\iota$, denote $y_{X^\iota} \coloneqq (f^\iota|_{X^\iota})^{-1}(x)$. Then by Definition~\ref{defDolgopyatOperator}, (\ref{eqSplitRuelleCoordinateFormula}) in Lemma~\ref{lmSplitRuelleCoordinateFormula}, and the Cauchy--Schwartz inequality, we have
\begin{align*}
&             \bigl( \pi_\b \bigl( \MM_{J,\minus s,\phi} (h_\b,h_\w) \bigr) (x) \bigr)^2 \\
&\qquad =     \biggl(   \sum\limits_{\c\in\{\b,\w\}} \RR_{\wt{\minus a\phi},X^0_\b,X^0_\c}^{(\iota)} \bigl(h_\c \beta_J|_{X^0_\c} \bigr)(x) \biggr)^2 \\
&\qquad =     \biggl(   \sum\limits_{\c\in\{\b,\w\}} \sum\limits_{\substack{X^\iota\in\X^\iota_\b\\X^\iota\subseteq X^0_\c}}
                      \bigl( h_\c \beta_J \exp\bigl( S_\iota \wt{- a\phi} \bigr)  \bigr) ( y_{X^\iota} )  \biggr)^2  \\
&\qquad \leq  \biggl(   \sum\limits_{\c\in\{\b,\w\}} \sum\limits_{\substack{X^\iota\in\X^\iota_\b\\X^\iota\subseteq X^0_\c}}
                      \bigl( h_\c^2    \exp\bigl( S_\iota \wt{- a\phi} \bigr)  \bigr) ( y_{X^\iota} )  \biggr)
        \biggl(   \sum\limits_{\c\in\{\b,\w\}} \sum\limits_{\substack{X^\iota\in\X^\iota_\b\\X^\iota\subseteq X^0_\c}}
                      \bigl( \beta_J^2 \exp\bigl( S_\iota \wt{- a\phi} \bigr)  \bigr) ( y_{X^\iota} )  \biggr)  \\
&\qquad =          \pi_\c \Bigl( \RRR^\iota_{\wt{\minus a\phi}} \bigl( h_\b^2,h_\w^2\bigr) \Bigr) (x) 
       \cdot \pi_\c \Bigl( \RRR^\iota_{\wt{\minus a\phi}} \Bigl( \bigl( \beta_J|_{X^0_\b} \bigr)^2, \bigl( \beta_J|_{X^0_\w} \bigr)^2 \Bigr) \Bigr) (x).                   
\end{align*}

\smallskip

We will focus on the case when the potential is $\wt{ - s_0\phi}$ for now, and only consider the general case at the end of the proof of statement~(ii).

Next, we define a set
\begin{equation}     \label{eqDefWJ}
W_J \coloneqq \bigcup\limits_{(j,i,X)\in J}  f^{M_0} ( \mathfrak{X}'_i(X) ).
\end{equation}
We claim that for each $\c\in\{\b,\w\}$ and each $x\in W_J \cap X^0_\c$, we have
\begin{equation}   \label{eqPfpropDolgopyatOperator_1-eta}
      \pi_\c \Bigl( \RRR^\iota_{\wt{\minus s_0\phi}} \Bigl( \bigl( \beta_J|_{X^0_\b} \bigr)^2, \bigl( \beta_J|_{X^0_\w} \bigr)^2 \Bigr) \Bigr) (x) 
\leq  1 - \eta \exp \Bigl( - \iota \Normbig{ \wt{ - s_0 \phi} }_{\CCC^0(S^2)} \Bigr).
\end{equation}
Indeed, it suffices to verify the case when $\c=\b$ as shown below. For each $x\in W_J \cap X^0_\b$, by (\ref{eqSplitRuelleCoordinateFormula}) in Lemma~\ref{lmSplitRuelleCoordinateFormula}, Definition~\ref{defSplitRuelle}, and (\ref{eqDefBetaJ}),
\begin{align*}
&            \pi_\b \Bigl( \RRR^\iota_{\wt{\minus s_0\phi}} \Bigl( \bigl( \beta_J|_{X^0_\b} \bigr)^2, \bigl( \beta_J|_{X^0_\w} \bigr)^2 \Bigr) \Bigr) (x) \\  
&\qquad =    \sum\limits_{\c\in\{\b,\w\}} \RR_{\wt{\minus s_0\phi},X^0_\b,X^0_\c}^{(\iota)} \Bigl( \bigl( \beta_J|_{X^0_\c} \bigr)^2 \Bigr)(x) \\
&\qquad =    \sum\limits_{\c\in\{\b,\w\}} \sum\limits_{\substack{X^\iota\in\X^\iota_\b\\X^\iota\subseteq X^0_\c}}
                 \beta_J^2 ( y_{X^\iota}) \exp \bigl( S_\iota \wt{ - s_0\phi} (y_{X^\iota}) \bigr)   \\
&\qquad\leq  \sum\limits_{\c\in\{\b,\w\}} \RR_{\wt{\minus s_0\phi},X^0_\b,X^0_\c}^{(\iota)} \bigl(\mathbbm{1}_{X^0_\c} \bigr) (x) 
            - \eta \psi_{i_X,X} \bigl( f^{N_1}  ( y_*   )  \bigr) \exp \bigl( S_\iota \wt{ - s_0\phi} (y_*) \bigr)\\
&\qquad\leq  1 - \eta \exp \Bigl( - \iota \Normbig{ \wt{ - s_0 \phi} }_{\CCC^0(S^2)} \Bigr),
\end{align*}
where $i_X,j_X\in\{1,2\}$ are chosen in such a way that $(j_X,i_X,X)\in J$ (due to the fact that $J\in\mathcal{F}$ has full projection (c.f.\ Definition~\ref{defFull})), and we denote $y_{X^\iota} \coloneqq (f^\iota|_{X^\iota})^{-1}(x)$ for $X^\iota \in \X^\iota_\b$, and write $y_* \coloneqq y_{X_{\b,j_X}^{N_1+M_0}}$. The last inequality follows from (\ref{eqSplitRuelleTildeSupDecreasing}) in Lemma~\ref{lmRtildeNorm=1}, (\ref{eqPsi_iX}), and (\ref{eqDefWJ}). The claim is now verified.

\smallskip

Next, we claim that for each $\c\in\{\b,\w\}$,
\begin{equation}   \label{eqPfpropDolgopyatOperator_RuelleH2inCone}
        \pi_\c \Bigl( \RRR^\iota_{\wt{\minus s_0\phi}} \bigl(h_\b^2,h_\w^2 \bigr) \Bigr) \in K_{A\abs{b}} (X^0_\c,d).
\end{equation}
Indeed, it suffices to verify the case $\c=\b$ as shown below. By (\ref{eqSplitRuelleCoordinateFormula}) in Lemma~\ref{lmSplitRuelleCoordinateFormula}, Lemma~\ref{lmConePower}, and Lemma~\ref{lmBasicIneq}~(i), for all $x,y\in X^0_\b$,
\begin{align*}
&            \AbsBig{   \pi_\b \Bigl( \RRR^\iota_{\wt{\minus s_0\phi}} \bigl( h_\b^2,h_\w^2 \bigr) \Bigr) (x) 
                              -  \pi_\b \Bigl( \RRR^\iota_{\wt{\minus s_0\phi}} \bigl( h_\b^2,h_\w^2 \bigr) \Bigr) (y)   } \\
&\qquad\leq  \sum\limits_{\c\in\{\b,\w\}}  \AbsBig{   \RR_{\wt{\minus s_0\phi},X^0_\b,X^0_\c}^{(\iota)} \bigl( h_\c^2 \bigr)(x)
                                                                                            - \RR_{\wt{\minus s_0\phi},X^0_\b,X^0_\c}^{(\iota)} \bigl( h_\c^2 \bigr)(y)  }\\
&\qquad\leq  A_0 \biggl(  \frac{ 2A\abs{b} }{ \Lambda^{\alpha\iota} }
                        + \frac{ \Hseminormbig{\alpha,\,(S^2,d)}{\wt{ - s_0\phi}} }{ 1 - \Lambda^{-\alpha} } \biggr)  d(x,y)^\alpha 
             \sum\limits_{\c\in\{\b,\w\}}   \sum\limits_{z\in\{x,y\}}   \RR_{\wt{\minus s_0\phi},X^0_\b,X^0_\c}^{(\iota)} \bigl( h_\c^2 \bigr)(z)  \\
&\qquad\leq  A\abs{b}  d(x,y)^\alpha   \sum\limits_{z\in\{x,y\}}   \pi_\b \Bigl( \RRR^\iota_{\wt{\minus s_0\phi}} \bigl( h_\b^2,h_\w^2 \bigr) \Bigr) (z) ,   
\end{align*}
where $A_0=A_0\bigl(f,\CC,d,\Hseminorm{\alpha,\,(S^2,d)}{\phi}, \alpha\bigr)>2$ is a constant from Lemma~\ref{lmBasicIneq} depending only on $f$, $\CC$, $d$, $\Hseminorm{\alpha,\,(S^2,d)}{\phi}$, and $\alpha$; and $C\geq 1$ is a constant from Lemma~\ref{lmCellBoundsBM} depending only on $f$, $\CC$, and $d$. The last inequality follows from $\frac{A_0}{\Lambda^{\alpha(N_1 + M_0)}} \leq \frac14$ (see (\ref{eqDef_N1})) and 
$ 
\frac{A_0 \Hseminorm{\alpha,\,(S^2,d)}{\wt{ \minus s_0\phi}} }{ 1 - \Lambda^{-\alpha} } \leq \frac12  b_0 \leq \frac12 A b_0 \leq \frac12 A \abs{b}
$
(see (\ref{eqDef_b0}) and (\ref{eqDef_A})). The claim now follows immediately.

\smallskip

We now combine (\ref{eqPfpropDolgopyatOperator_RuelleH2inCone}), Lemma~\ref{lmSupLeq2Inf}, Lemma~\ref{lmGibbsDoubling}, (\ref{eqDefWJ}), and $\abs{b} \geq b_0 \geq 2 s_0 + 1$ (see (\ref{eqDef_b0})) to deduce that for each $\c\in\{\b,\w\}$, we have
\begin{align}    \label{eqPfpropDolgopyatOperator_BoundIntW_J}
&             \int_{X^0_\c} \!  \pi_\c \Bigl(  \RRR_{\wt{\minus s_0\phi}}^\iota \bigl(h_\b^2,h_\w^2\bigr) \Bigr) \,\mathrm{d}\mu_{\minus s_0\phi}  \\
&\qquad\leq   \sum\limits_{\substack{X \in\mathfrak{C}_b \\X \subseteq Y^{M_0}_\c}} 
                  \int_{f^{M_0}(X)} \, \pi_\c \Bigl(  \RRR_{\wt{\minus s_0\phi}}^\iota \bigl(h_\b^2,h_\w^2\bigr) \Bigr) \,\mathrm{d}\mu_{\minus s_0\phi}    \notag\\
&\qquad\leq   \sum\limits_{\substack{X \in\mathfrak{C}_b \\X \subseteq Y^{M_0}_\c}}        \mu_{\minus s_0\phi} \bigl( f^{M_0}(X) \bigr)
                  \sup\limits_{ x\in f^{M_0}(X) }  \Bigl\{ \pi_\c \Bigl(  \RRR_{\wt{\minus s_0\phi}}^\iota \bigl(h_\b^2,h_\w^2\bigr) \Bigr) (x)  \Bigr\}    \notag \\
&\qquad\leq   \sum\limits_{\substack{X \in\mathfrak{C}_b \\X \subseteq Y^{M_0}_\c}}   \mu_{\minus s_0\phi} \bigl( f^{M_0}(X) \bigr) \cdot
                   2 \inf\limits_{  x\in f^{M_0}(X) }   \Bigl\{ \pi_\c \Bigl(  \RRR_{\wt{\minus s_0\phi}}^\iota \bigl(h_\b^2,h_\w^2\bigr) \Bigr) (x)   \Bigr\}      \notag  \\  
&\qquad\leq   C_{18} \sum\limits_{\substack{X \in\mathfrak{C}_b \\X \subseteq Y^{M_0}_\c}}    \mu_{\minus s_0\phi} \bigl(  f^{M_0} \bigl( \mathfrak{X}'_{i_{J,X}}(X) \bigr) \bigr)
                   \inf\limits_{ x\in f^{M_0} \bigl( \mathfrak{X}'_{i_{J,X}}(X) \bigr) }   \Bigl\{ \pi_\c \Bigl(  \RRR_{\wt{\minus s_0\phi}}^\iota \bigl(h_\b^2,h_\w^2\bigr) \Bigr) (x)   \Bigr\}  \notag \\     
&\qquad\leq   C_{18} \sum\limits_{\substack{X \in\mathfrak{C}_b \\X \subseteq Y^{M_0}_\c}} 
                  \int_{f^{M_0}\bigl(\mathfrak{X}'_{i_{J,X}} (X) \bigr)} \, \pi_\c \Bigl(  \RRR_{\wt{\minus s_0\phi}}^\iota \bigl(h_\b^2,h_\w^2\bigr) \Bigr) \,\mathrm{d}\mu_{\minus s_0\phi}    \notag \\    
&\qquad\leq     C_{18} \int_{W_J \cap X^0_\c} \!  \pi_\c \Bigl(  \RRR_{\wt{\minus s_0\phi}}^\iota \bigl(h_\b^2,h_\w^2\bigr) \Bigr) \,\mathrm{d}\mu_{\minus s_0\phi},      \notag                              
\end{align}
where $i_{J,X}\in \{1,2\}$ can be set in such a way that either $(1,i_{J,X},X)\in J$ or $(2,i_{J,X},X)\in J$ due to the assumption that $J\in\mathcal{F}$ has full projection, and the constant $C_{18}$ can be chosen as
\begin{equation}   \label{eqDefC18}
C_{18} \coloneqq 2 C_{\mu_{\minus s_0\phi}}^2 \exp \bigl( 2 m_0 \bigl( \norm{ {-}s_0\phi}_{\CCC^0(S^2)} + P(f,-s_0\phi) \bigr) \bigr) >1,
\end{equation}
which depends only on $f$, $\CC$, $d$, and $\phi$. Here the constant $C_{\mu_{\minus s_0\phi}}\geq 1$ is from Lemma~\ref{lmGibbsDoubling}, depending only on $f$, $d$, and $\phi$.

We now observe that by (\ref{eqSplitRuelleCoordinateFormula}) in Lemma~\ref{lmSplitRuelleCoordinateFormula} and Lemma~\ref{lmRDiscontTildeDual},
\begin{equation}   \label{eqPfpropDolgopyatOperator_IterateSplitRuelle}
   \sum\limits_{\c\in\{\b,\w\}}   
       \int_{X^0_\c} \!  \pi_\c \Bigl(  \RRR_{\wt{\minus s_0\phi}}^\iota \bigl(h_\b^2,h_\w^2\bigr) \Bigr) \,\mathrm{d}\mu_{\minus s_0\phi}
=  \sum\limits_{\c\in\{\b,\w\}}   \int_{X^0_\c} \! h_\c^2     \,\mathrm{d}\mu_{\minus s_0\phi}.
\end{equation}

Combining (\ref{eqPfpropDolgopyatOperator_IterateSplitRuelle}), (\ref{eqPfpropDolgopyatOperator_HolderIneq}), Lemma~\ref{lmRtildeNorm=1}, (\ref{eqBetaJBounds}) in Lemma~\ref{lmBetaJ}, (\ref{eqPfpropDolgopyatOperator_1-eta}), and (\ref{eqPfpropDolgopyatOperator_BoundIntW_J}), we get
\begin{align}  \label{eqPfpropDolgopyatOperator_Difference}
&               \sum\limits_{\c\in\{\b,\w\}}   \int_{X^0_\c} \! h_\c^2     \,\mathrm{d}\mu_{\minus s_0\phi}
             -  \sum\limits_{\c\in\{\b,\w\}}  \int_{X^0_\c} \Absbig{ \pi_\c \bigl( \MM_{J,\minus s_0,\phi} (h_\b,h_\w) \bigr) }^2 \, \mathrm{d}\mu_{\minus s_0\phi} \\
&\qquad =        \sum\limits_{\c\in\{\b,\w\}}   
                     \int_{X^0_\c} \!  \pi_\c \Bigl(  \RRR_{\wt{\minus s_0\phi}}^\iota \bigl(h_\b^2,h_\w^2\bigr) \Bigr) \,\mathrm{d}\mu_{\minus s_0\phi}
             -  \sum\limits_{\c\in\{\b,\w\}}  \int_{X^0_\c} \Absbig{ \pi_\c \bigl( \MM_{J,\minus s_0,\phi} (h_\b,h_\w) \bigr) }^2 \, \mathrm{d}\mu_{\minus s_0\phi}  \notag \\
&\qquad\geq     \sum\limits_{\c\in\{\b,\w\}}   
                 \int_{X^0_\c} \!  \pi_\c \Bigl(  \RRR_{\wt{\minus s_0\phi}}^\iota \bigl(h_\b^2,h_\w^2\bigr) \Bigr) \cdot
                 \Bigl( 1 - \pi_\c \Bigl( \RRR^\iota_{\wt{\minus s_0\phi}} \Bigl( \bigl( \beta_J|_{X^0_\b} \bigr)^2, \bigl( \beta_J|_{X^0_\w} \bigr)^2 \Bigr)\Bigr)\Bigr) \, \mathrm{d}\mu_{\minus s_0\phi}  \notag \\
&\qquad\geq     \sum\limits_{\c\in\{\b,\w\}}   
                 \int_{W_J \cap X^0_\c} \!  \pi_\c \Bigl(  \RRR_{\wt{\minus s_0\phi}}^\iota \bigl(h_\b^2,h_\w^2\bigr) \Bigr) \cdot
                 \Bigl( 1 - \pi_\c \Bigl( \RRR^\iota_{\wt{\minus s_0\phi}} \Bigl( \bigl( \beta_J|_{X^0_\b} \bigr)^2, \bigl( \beta_J|_{X^0_\w} \bigr)^2 \Bigr)\Bigr)\Bigr) \, \mathrm{d}\mu_{\minus s_0\phi}  \notag \\ 
&\qquad\geq   \eta \exp \Bigl( -\iota \Normbig{\wt{ - s_0\phi}}_{\CCC^0(S^2)} \Bigr)  \sum\limits_{\c\in\{\b,\w\}}   
                     \int_{W_J \cap X^0_\c} \!  \pi_\c \Bigl(  \RRR_{\wt{\minus s_0\phi}}^\iota \bigl(h_\b^2,h_\w^2\bigr) \Bigr) \,\mathrm{d}\mu_{\minus s_0\phi}  \notag \\
&\qquad\geq   \frac{\eta}{C_{18}}    \exp \Bigl( -\iota \Normbig{\wt{ - s_0\phi}}_{\CCC^0(S^2)} \Bigr)  \sum\limits_{\c\in\{\b,\w\}}   
                     \int_{X^0_\c} \!  \pi_\c \Bigl(  \RRR_{\wt{\minus s_0\phi}}^\iota \bigl(h_\b^2,h_\w^2\bigr) \Bigr) \,\mathrm{d}\mu_{\minus s_0\phi}  \notag \\
&\qquad\geq   \frac{\eta}{C_{18}}    \exp \Bigl( -\iota \Normbig{\wt{ - s_0\phi}}_{\CCC^0(S^2)} \Bigr)     
              \sum\limits_{\c\in\{\b,\w\}}   \int_{X^0_\c} \! h_\c^2     \,\mathrm{d}\mu_{\minus s_0\phi}.    \notag  
\end{align}

\smallskip

We now consider the general case where the potential is $\wt{- s\phi}$. Fix $\c'\in\{\b,\w\}$ and an arbitrary point $x\in X^0_{\c'}$. For each $X^\iota \in \X^\iota_{\c'}$, denote $y_{X^\iota} \coloneqq (f^\iota|_{X^\iota})^{-1}(x)$. Then by Definition~\ref{defDolgopyatOperator} and (\ref{eqSplitRuelleCoordinateFormula}) in Lemma~\ref{lmSplitRuelleCoordinateFormula},
\begin{align*}
     & \pi_{\c'} \bigl( \MM_{J,\minus s,\phi} (h_\b,h_\w) \bigr) (x) \\
 =   & \sum\limits_{\c\in\{\b,\w\}} \sum\limits_{\substack{X^\iota\in\X^\iota_{\c'}\\X^\iota\subseteq X^0_\c}}
                      h_\c ( y_{X^\iota} ) \beta_J ( y_{X^\iota} ) \exp\bigl( S_\iota \wt{ - a\phi} ( y_{X^\iota} ) \bigr)   \\
\leq & \sum\limits_{\c\in\{\b,\w\}} \sum\limits_{\substack{X^\iota\in\X^\iota_{\c'}\\X^\iota\subseteq X^0_\c}} 
                      h_\c ( y_{X^\iota} ) \beta_J ( y_{X^\iota} ) \exp\bigl( S_\iota \wt{ - s_0\phi} ( y_{X^\iota} ) \bigr)
                      \exp \bigl(  \Absbig{ S_\iota \wt{ - a\phi} ( y_{X^\iota} ) - S_\iota \wt{ - s_0\phi} ( y_{X^\iota} ) }  \bigr)  \\
\leq & \pi_{\c'} \bigl( \MM_{J,\minus s_0,\phi} (h_\b,h_\w) \bigr) (x) 
       e^{ \iota \bigl(   \abs{a-s_0} \norm{\phi}_{\CCC^0(S^2)} + \abs{P(f,-a\phi) - P(f,-s_0\phi)}
                                      + 2 \norm{ \log u_{\minus a\phi} - \log u_{\minus s_0\phi}}_{\CCC^0(S^2)}   \bigr)   } .
\end{align*}

Since the function $t\mapsto P(f,t\phi)$ is continuous (see for example, \cite[Theorem~3.6.1]{PrU10}) and the map $t\mapsto u_{t\phi}$ is continuous in $\Holder{\alpha}(S^2,d)$ equipped with the uniform norm $\norm{\cdot}_{\CCC^0(S^2)}$ by Corollary~\ref{corUphiCountinuous}, we can choose $a_0\in(0, s_0 )$ small enough, depending only on $f$, $\CC$, $d$, $\alpha$, and $\phi$ such that if $s=a+ \I b$ with $a,b\in\R$ satisfies $\abs{a-s_0} \leq a_0$ and $\abs{b} \geq 2 s_0 + 1$, then
\begin{align*}
     &  e^{ \iota \bigl(   \abs{a-s_0} \norm{\phi}_{\CCC^0(S^2)} + \abs{P(f,-a\phi) - P(f,-s_0\phi)}
                                        + 2 \norm{ \log u_{\minus a\phi} - \log u_{\minus s_0\phi}}_{\CCC^0(S^2)}   \bigr)   } \\
\leq & \Biggl( 1+ \frac{ \eta \exp\bigl( -\iota \Normbig{\wt{ - s_0\phi}}_{\CCC^0(S^2)} \bigr) } { C_{18} }   \Biggr)^{\frac{1}{2}} ,
\end{align*}
and consequently
\begin{equation}  \label{eqPfpropDolgopyatOperator_s_s0}
      \pi_{\c'} \bigl( \MM_{J,\minus s,\phi} (h_\b,h_\w) \bigr) (x)  
\leq  \Biggl( 1+ \frac{ \exp\bigl( -\iota \Normbig{\wt{ - s_0\phi}}_{\CCC^0(S^2)} \bigr) } { C_{18} }  \Biggr)^{\frac{1}{2}}
      \pi_{\c'}  ( \MM_{J,\minus s_0,\phi} (h_\b,h_\w)) (x).
\end{equation}

Therefore, if $s=a+\I b$ with $a,b\in\R$ satisfies $\abs{a-s_0} \leq a_0$ and $\abs{b} \geq b_0 \geq 2 s_0 + 1$ (see (\ref{eqDef_b0})), we get from (\ref{eqPfpropDolgopyatOperator_s_s0}) and (\ref{eqPfpropDolgopyatOperator_Difference}) that
\begin{align*}
&            \sum\limits_{\c\in\{\b,\w\}}   \int_{X^0_\c} \! \abs{ \pi_\c ( \MM_{J,\minus s,\phi} (h_\b,h_\w) ) }^2 \,\mathrm{d}\mu_{\minus s_0\phi}   \\
&\qquad\leq  \Biggl( 1 + \frac{ \eta \exp\bigl( -\iota \Normbig{\wt{ - s_0\phi}}_{\CCC^0(S^2)} \bigr) } { C_{18} }  \Biggr) 
             \sum\limits_{\c\in\{\b,\w\}}   \int_{X^0_\c} \! \abs{ \pi_\c ( \MM_{J,\minus s_0,\phi} (h_\b,h_\w) ) }^2 \,\mathrm{d}\mu_{\minus s_0\phi}   \\
&\qquad\leq  \Biggl( 1 - \frac{ \eta^2 \exp\bigl( -2\iota \Normbig{\wt{ - s_0\phi}}_{\CCC^0(S^2)} \bigr) } { C_{18}^2 }  \Biggr) 
             \sum\limits_{\c\in\{\b,\w\}}  \int_{X^0_\c} \! \abs{h_\c}^2 \,\mathrm{d}\mu_{\minus s_0\phi}.   
\end{align*}
We finish the proof of (ii) by choosing 
\begin{equation*} 
\rho \coloneqq 1 - \frac{ \eta^2 \exp\bigl( -2\iota \Normbig{\wt{ - s_0\phi}}_{\CCC^0(S^2)} \bigr) } { C_{18}^2 } \in (0,1),
\end{equation*}
which depends only on $f$, $\CC$, $d$, $\alpha$, and $\phi$.

\smallskip

(iii) Given arbitrary $h_\b$, $h_\w$, $u_\b$, and $u_\w$ satisfying the hypotheses in (iii), we construct a subset $J\subseteq \{1,2\}\times \{1,2\} \times \mathfrak{C}_b$ as follows:

For each $X\in\mathfrak{C}_b$, 
\begin{enumerate}
\smallskip
\item[(1)] if $\Normbig{Q_{\c_X, 1}}_{\CCC^0(\mathfrak{X}_1(X))} \leq 1$, then include $(1,1,X)$ in $J$, otherwise

\smallskip
\item[(2)] if $\Normbig{Q_{\c_X, 2}}_{\CCC^0(\mathfrak{X}_1(X))} \leq 1$, then include $(2,1,X)$ in $J$, otherwise

\smallskip
\item[(3)] if $\Normbig{Q_{\c_X, 1}}_{\CCC^0(\mathfrak{X}_2(X))} \leq 1$, then include $(1,2,X)$ in $J$, otherwise

\smallskip
\item[(4)] if $\Normbig{Q_{\c_X, 2}}_{\CCC^0(\mathfrak{X}_2(X))} \leq 1$, then include $(2,2,X)$ in $J$,
\end{enumerate}
where we denote $\c_X\in\{\b,\w\}$ with the property that $X\subseteq Y_{\c_X}^{M_0}$. Here functions $Q_{\c,j} \: Y_\c^{M_0} \rightarrow \R$, $\c\in\{\b,\w\}$ and $j\in\{1,2\}$, are defined in Lemma~\ref{lmUHDichotomySum}.

By Lemma~\ref{lmUHDichotomySum}, at least one of the four cases above occurs for each $X\in \mathfrak{C}_b$. Thus the set $J$ constructed above has full projection (c.f.\ Definition~\ref{defFull}). 

We finally set $\mathcal{E}_s \coloneqq \bigcup \{ J \}$, where the union ranges over all $h_\b$, $h_\w$, $u_\b$, and $u_\w$ satisfying the hypotheses in (iii).

We now fix such $h_\b$, $h_\w$, $u_\b$, $u_\w$, and the corresponding $J$ constructed above. Then for each $\c\in\{\b,\w\}$ and each $x\in X_\c^0$, we will establish (\ref{eqDolgopyatOpPtwiseBound}) as follows:

\begin{enumerate}
\smallskip
\item[(1)] If $x \notin \bigcup\limits_{X\in\mathfrak{C}_b}  f^{M_0} (\mathfrak{X}_1(X) \cup \mathfrak{X}_2(X))$, then by (\ref{eqPsi_iX}) and (\ref{eqDefBetaJ}), $\beta_J(y)=1$ for all $y\in f^{-(N_1+M_0)}(x)$. Thus (\ref{eqDolgopyatOpPtwiseBound}) holds for $x$ by Definition~\ref{defDolgopyatOperator}, (\ref{eqSplitRuelleCoordinateFormula}) in Lemma~\ref{lmSplitRuelleCoordinateFormula}, and Definition~\ref{defSplitRuelle}.

\smallskip
\item[(2)] If $x\in f^{M_0} (\mathfrak{X}_i(X))$ for some $X\in\mathfrak{C}_b$ and $i\in\{1,2\}$, then one of the following two cases occurs:

\begin{enumerate}
\smallskip
\item[(a)] $(1,i,X)\notin J$ and $(2,i,X)\notin J$. Then by (\ref{eqDefBetaJ}), $\beta_J(y)=1$ for all $y\in f^{-(N_1+M_0)}(x)$. Thus (\ref{eqDolgopyatOpPtwiseBound}) holds for $x$  by Definition~\ref{defDolgopyatOperator}, (\ref{eqSplitRuelleCoordinateFormula}) in Lemma~\ref{lmSplitRuelleCoordinateFormula}, and Definition~\ref{defSplitRuelle}.

\smallskip
\item[(b)] $(j,i,X)\in J$ for some $j\in\{1,2\}$. Then by the construction of $J$, we have $(j',i',X)\in J$ if and only if $(j',i')=(j,i)$. We denote the inverse branches $\tau_k \coloneqq \Bigl( f^{N_1} \big|_{X^{ N_1+M_0}_{\c,k} } \Bigr)^{-1}$ for $k\in\{1,2\}$. Write $z \coloneqq \Bigl( f^{N_1+M_0} \big|_{X^{ N_1+M_0}_{\c,j} } \Bigr)^{-1} (x)$. Then $\beta_J(y)=1$ for each $y \in f^{-(N_1+M_0)}(x) \setminus \tau_j(\mathfrak{X}_i(X)) = f^{-(N_1+M_0)}(x) \setminus \{z\}$. In particular, $\beta_J\bigl(\tau_{j_*}\bigl(f^{N_1}(z) \bigr)\bigr) = 1$, where $j_*\in\{1,2\}$ and $j_*\neq j$. By Lemma~\ref{lmUHDichotomySum}, we get $Q_{\c,j} \bigl( f^{N_1}(z) \bigr) \leq 1$, i.e.,
\begin{align*}
&              \Absbigg{ \sum\limits_{k\in\{1,2\}}  \Bigl( u_{\varsigma(\c,k)} e^{S_{N_1} \wt{-s\phi} } \Bigr) \bigl( \tau_k \bigl( f^{N_1} (z) \bigr)\bigr)  }  \\
&\qquad \leq   -2 \eta h_{\varsigma(\c,j)} (z) e^{S_{N_1} \wt{ \minus a\phi} (z) }
                     +  \sum\limits_{k\in\{1,2\}}  \Bigl( h_{\varsigma(\c,k)} e^{S_{N_1} \wt{ \minus a\phi} } \Bigr) \bigl( \tau_k \bigl( f^{N_1} (z) \bigr)\bigr)   \\
&\qquad \leq    \Bigl( \beta_J h_{\varsigma(\c,j  )} e^{S_{N_1} \wt{ \minus a\phi} } \Bigr) (z) 
              + \Bigl( \beta_J h_{\varsigma(\c,j_*)} e^{S_{N_1} \wt{ \minus a\phi} } \Bigr) \bigl( \tau_{j_*} \bigl( f^{N_1} (z) \bigr)\bigr),
\end{align*}
where $\varsigma(\c,k)$ is defined as in the statement of Lemma~\ref{lmUHDichotomySum}. Hence (\ref{eqDolgopyatOpPtwiseBound}) holds for $x$ by Definition~\ref{defDolgopyatOperator}, (\ref{eqSplitRuelleCoordinateFormula}) in Lemma~\ref{lmSplitRuelleCoordinateFormula}, and Definition~\ref{defSplitRuelle}.
\end{enumerate}
\end{enumerate}

We are going to establish (\ref{eqDolgopyatOpLipBound}) in the case $\c=\b$ now. The other case is similar. By (\ref{eqSplitRuelleCoordinateFormula}) in Lemma~\ref{lmSplitRuelleCoordinateFormula}, (\ref{eqBasicIneqC}) in Lemma~\ref{lmBasicIneq}, Definition~\ref{defSplitRuelle}, and (\ref{eqBetaJBounds}), for $x,x'\in X^0_\b$ with $x\neq x'$,
\begin{align*}
&  \frac{1}{ d(x,x')^\alpha}        \AbsBig{   \pi_\b \Bigl(\RRR_{\wt{\minus s\phi}}^{N_1+M_0} (u_\b,u_\w) \Bigr)  (x)
                                                                     - \pi_\b \Bigl(\RRR_{\wt{\minus s\phi}}^{N_1+M_0} (u_\b,u_\w) \Bigr)  (x') }   \\
&\qquad\leq   \frac{1}{ d(x,x')^\alpha}  \sum\limits_{\c\in\{\b,\w\}}  \AbsBig{   \RR_{\wt{\minus s\phi},X^0_\b,X^0_\c}^{(\iota)} (u_\c)   (x)
                                                                                                                                       - \RR_{\wt{\minus s\phi},X^0_\b,X^0_\c}^{(\iota)} (u_\c)   (x') }   \\
&\qquad\leq  A_0  \sum\limits_{\c\in\{\b,\w\}} \biggl( \biggl(     
                       \frac{A\abs{b}}{\Lambda^{\alpha\iota}}    \sum\limits_{z\in\{x,x'\}}   \RR_{\wt{\minus a\phi},X^0_\b,X^0_\c}^{(\iota)} (h_\c)   (z)    \biggr) 
                                  + \abs{b}  \RR_{\wt{\minus a\phi},X^0_\b,X^0_\c}^{(\iota)} (h_\c)   (x) \biggr) \\
&\qquad\leq  \biggl( \frac{ A_0 A}{\Lambda^{\alpha\iota}}  + 1 \biggr)  \abs{b}   \sum\limits_{\c\in\{\b,\w\}}  \sum\limits_{z\in\{x,x'\}}
                                \RR_{\wt{\minus a\phi},X^0_\b,X^0_\c}^{(\iota)} \bigl(2 h_\c \beta_J|_{X^0_\c}\bigr)   (z)  \\
&\qquad\leq   \biggl( \frac{ 2 A_0 A}{\Lambda^{\alpha\iota}}  + 2 \biggr)
               \abs{b}  \sum\limits_{z\in\{x,x'\}}  \pi_\b (\MM_{J,\minus s,\phi} (h_\b,h_\w) ) (z) \\
&\qquad\leq   A\abs{b}  \sum\limits_{z\in\{x,x'\}}  \pi_\b (\MM_{J,\minus s,\phi} (h_\b,h_\w) ) (z) ,
\end{align*}
where the last inequality follows from $\frac{ 2 A_0}{\Lambda^{\alpha\iota}} \leq \frac12$ (see (\ref{eqDef_N1})) and  $A\geq 4$ (see (\ref{eqDef_A})).
\end{proof}

\begin{proof}[Proof of Theorem~\ref{thmL2Shrinking}]
We set $\iota \coloneqq N_1 + M_0$, where $N_1\in\Z$ is defined in (\ref{eqDef_N1}) and $M_0\in\N$ is a constant from Definition~\ref{defStrongNonIntegrability}. We take the constants $a_0\in (0,  s_0 )$ and $\rho\in(0,1)$ from Proposition~\ref{propDolgopyatOperator}, and $b_0$ as defined in (\ref{eqDef_b0}).

Fix arbitrary $s \coloneqq a+\I b$ with $a,b\in\R$ satisfying $\abs{a-s_0} \leq a_0$ and $\abs{b}\geq b_0$. Fix arbitrary $u_\b\in\Holder{\alpha}\bigl(\bigl(X^0_\b,d\bigr),\C\bigr)$ and $u_\w\in\Holder{\alpha}\bigl(\bigl(X^0_\w,d\bigr),\C\bigr)$ satisfying 
\begin{equation}  \label{eqPfthmL2Shrinking_UNormBound}
\NHnorm{\alpha}{\Im(s)}{u_\b}{(X^0_\b,d)} \leq 1 \qquad \text{and}\qquad  \NHnorm{\alpha}{\Im(s)}{u_\w}{(X^0_\w,d)} \leq 1.
\end{equation}
We recall the constant $A\in\R$ defined in (\ref{eqDef_A}) and the subset $\mathcal{E}_s \subseteq \mathcal{F}$ constructed in Proposition~\ref{propDolgopyatOperator}. 

We will construct sequences $\{h_{\b,k}\}_{k=-1}^{+\infty}$ in $K_{A\abs{b}}\bigl( X^0_\b,d \bigr)$, $\{h_{\w,k}\}_{k=-1}^{+\infty}$ in $K_{A\abs{b}}\bigl( X^0_\w,d \bigr)$, $\{u_{\b,k}\}_{k=0}^{+\infty}$ in $\Holder{\alpha}\bigl(\bigl( X^0_\b,d \bigr),\C \bigr)$, $\{u_{\w,k}\}_{k=0}^{+\infty}$ in $\Holder{\alpha}\bigl(\bigl( X^0_\w,d \bigr),\C \bigr)$, and $\{ J_k \}_{k=0}^{+\infty}$ in $\mathcal{E}_s$ recursively so that the following properties are satisfied for each $k\in\N_0$, each $\c\in\{\b,\w\}$, and all $x,x'\in X^0_\c$:

\begin{enumerate}
\smallskip
\item[(1)] $u_{\c,k} = \pi_\c \Bigl(  \RRR_{\wt{\minus s\phi}}^{k\iota} (u_\b,u_\w) \Bigr)$.

\smallskip
\item[(2)] $\abs{u_{\c,k}(x)} \leq h_{\c,k}(x)$ and $\abs{u_{\c,k}(x) - u_{\c,k}(x')} \leq A\abs{b} (h_{\c,k}(x) + h_{\c,k}(x')) d(x,x')^\alpha$.

\smallskip
\item[(3)] $\sum\limits_{ \c'\in\{\b,\w\} } \int_{ X^0_{\c'} } \!   h_{\c',k}^2 \,\mathrm{d}\mu_{\minus s_0\phi} 
    \leq \rho \sum\limits_{ \c'\in\{\b,\w\} }    \int_{ X^0_{\c'} } \!   h_{\c',k-1}^2 \,\mathrm{d}\mu_{\minus s_0\phi}$.
            
\smallskip
\item[(4)] $\pi_\c \Bigl(  \RRR_{\wt{\minus s\phi}}^{\iota} (u_{\b,k},u_{\w,k}) \Bigr) (x) \leq \pi_\c \bigl(  \MM_{J_k,\minus s,\phi} (h_{\b,k},h_{\w,k}) \bigr) (x)$ and
\begin{align*}
     & \AbsBig{    \pi_\c \Bigl(  \RRR_{\wt{\minus s\phi}}^{\iota} (u_{\b,k},u_{\w,k}) \Bigr) (x) - \pi_\c \Bigl(  \RRR_{\wt{\minus s\phi}}^{\iota} (u_{\b,k},u_{\w,k}) \Bigr) (x')  }\\ 
\leq & A\abs{b} \bigl( \pi_\c \bigl(  \MM_{J_k,\minus s,\phi} (h_{\b,k},h_{\w,k}) \bigr) (x) + \pi_\c \bigl(  \MM_{J_k,\minus s,\phi} (h_{\b,k},h_{\w,k}) \bigr) (x')  \bigr) d(x,x')^\alpha.
\end{align*}
\end{enumerate}

We first set $h_{\c,-1}  \coloneqq \frac{1}{\rho}$, $h_{\c,0} \coloneqq \NHnormD{\Holder{\alpha}(X^0_\c,d)}{b}{u_\c}\in [0,1]$, and $u_{\c,0} \coloneqq u_\c$ for each $\c\in\{\b,\w\}$. Then clearly Properties~(1), (2), and (3) are satisfied for $k=0$. By Property~(2) for $k=0$, we can choose $j_0\in\mathcal{E}_s$ according to Proposition~\ref{propDolgopyatOperator}~(iii) such that Property~(4) holds for $k=0$.

We continue our construction recursively as follows. Assume that we have chosen $u_{\b,i} \in \Holder{\alpha}\bigl( \bigl(X^0_\b,d\bigr),\C \bigr)$, $u_{\w,i} \in \Holder{\alpha}\bigl( \bigl(X^0_\w,d\bigr),\C \bigr)$, $h_{\b,i} \in K_{A\abs{b}} \bigl( X^0_\b,d\bigr)$, $h_{\w,i} \in K_{A\abs{b}} \bigl( X^0_\w,d\bigr)$, and $J_i\in\mathcal{E}_s$ for some $i\in\N_0$. Then we define, for each $\c\in\{\b,\w\}$,
\begin{equation*}
u_{\c,i+1}   \coloneqq   \pi_\c \Bigl(\RRR_{\wt{\minus s\phi}}^\iota(u_{\b,i}, u_{\w,i}) \Bigr)  \qquad \text{and} \qquad 
h_{\c,i+1}   \coloneqq   \pi_\c  (\MM_{J_i,\minus s,\phi}(h_{\b,i}, h_{\w,i})  ).
\end{equation*}
Then for each $\c\in\{\b,\w\}$, by (\ref{eqSplitRuelleRestrictHolder}) we get $u_{\c,i+1} \in \Holder{\alpha}\bigl( \bigl(X^0_\c,d \bigr), \C \bigr)$, and by (\ref{eqDolgopyatOpStableCone}) in Proposition~\ref{propDolgopyatOperator} we have $h_{\c,i+1} \in K_{A\abs{b}} \bigl(X^0_\c,d \bigr)$. Property~(1) for $k=i+1$ follows from Property~(1) for $k=i$. Property~(2) for $k=i+1$ follows from Property~(4) for $k=i$. Property~(3) for $k=i+1$ follows from Proposition~\ref{propDolgopyatOperator}~(ii). By Property~(2) for $k=i+1$ and Proposition~\ref{propDolgopyatOperator}~(iii), we can choose $J_{i+1}\in\mathcal{E}_s$ such that Property~(4) for $k=i+1$ holds. This completes the recursive construction and the verification of Properties~(1) through (4) for all $k\in\N$.

By (\ref{eqSplitRuelleCoordinateFormula}) in Lemma~\ref{lmSplitRuelleCoordinateFormula}, Properties~(1), (2), (3), and Theorem~\ref{thmEquilibriumState}~(iii), we have
\begin{align*}
&            \int_{X^0_\c}\! \AbsBig{  \RR_{\wt{\minus s\phi}, X^0_\c, X^0_\b}^{(n\iota)} (u_\b) 
                                     + \RR_{\wt{\minus s\phi}, X^0_\c, X^0_\w}^{(n\iota)} (u_\w) }^2 \,\mathrm{d}\mu_{\minus s_0\phi} \\
&\qquad =    \int_{X^0_\c}\! \AbsBig{ \pi_\c \Bigl( \RRR_{\wt{\minus s\phi}}^{n\iota} (u_\b, u_\w) \Bigr) }^2 \, \mathrm{d}\mu_{\minus s_0\phi} 
        =    \int_{X^0_\c}\! \abs{ u_{\c,n} }^2 \, \mathrm{d}\mu_{\minus s_0\phi} \\
&\qquad\leq  \int_{X^0_\c}\! h_{\c,n} ^2 \, \mathrm{d}\mu_{\minus s_0\phi}
       \leq  \rho^n \biggl( \int_{X^0_\b} \!   h_{\b,0}^2 \,\mathrm{d}\mu_{\minus s_0\phi} + \int_{X^0_\w} \!  h_{\w,0}^2 \,\mathrm{d}\mu_{\minus s_0\phi} \biggr)\\
&\qquad\leq  \rho^n,
\end{align*}
for all $\c\in\{\b,\w\}$ and $n\in\N$.
\end{proof}

\section{Examples and genericity of strongly non-integrable potentials}    \label{sctExamples}
In this section, we try to discuss on how general the strong non-integrability condition is. We show in Subsection~\ref{subsctLattes} that for the Latt\`{e}s maps, in the class of continuously differentiable real-valued potentials, the weaker condition of non-local integrability implies the (stronger) $1$-strong non-integrability for some visual metric $d$ for $f$. This leads to a characterization of the Prime Orbit Theorems in this context, i.e., Theorem~\ref{thmLattesPOT}, proved at the end of Subsection~\ref{subsctLattes}. The proof relies on the geometric properties of various metrics in this context, and does not generalize to other rational expanding Thurston maps. However, we are able to show the genericity of the $\alpha$-strong non-integrability condition in the set $\Holder{\alpha}(S^2,d)$ of real-valued H\"{o}lder continuous functions with an exponent $\alpha$ on $S^2$ equipped with a visual metric $d$. See Theorem~\ref{thmSNIGeneric} for the precise statement. A constructive proof of Theorem~\ref{thmSNIGeneric} is given at the end of Subsection~\ref{subsctGeneric}, relying on Theorem~\ref{thmPerturbToStrongNonIntegrable} that gives a construction of a potential $\phi$ that satisfies the $\alpha$-strong non-integrability condition arbitrarily close to a given $\psi \in \Holder{\alpha}(S^2,d)$.

\subsection{Examples for Latt\`{e}s maps}   \label{subsctLattes}
In order to carry out the cancellation argument in Section~\ref{sctDolgopyat}, it is crucial to have both the lower bound and the upper bound in (\ref{eqSNIBounds}). As seen in the proof of Proposition~\ref{propSNI}, the upper bound in (\ref{eqSNIBounds}) is guaranteed automatically by the H\"{o}lder continuity of the potential $\phi$ with the right exponent $\alpha$. If we could assume in addition that the identity map on $S^2$ is a bi-Lipschitz equivalence (or more generally, snowflake equivalence) from a visual metric $d$ to the Euclidean metric on $S^2$, and the temporal distance $\phi^{f,\,\CC}_{\xi,\,\xi'}$ is nonconstant and continuously differentiable, then we could expect a lower bound with the same exponent as that in the upper bound in (\ref{eqSNIBounds}) near the same point. 

However, for a rational expanding Thurston map $f\: \widehat{\C}  \rightarrow \widehat{\C}$ defined on the Riemann sphere $\widehat{\C}$, the chordal metric $\sigma$ (see Remark~\ref{rmChordalVisualQSEquiv} for the definition), which is bi-Lipschitz equivalent to the Euclidean metric away from the infinity, is never a visual metric for $f$ (see \cite[Lemma~8.12]{BM17}). In fact, $(S^2,d)$ is snowflake equivalent to $\bigl(\widehat{\C}, \sigma \bigr)$ if and only if $f$ is topologically conjugate to a Latt\`{e}s map (see \cite[Theorem~18.1~(iii)]{BM17} and Definition~\ref{defLattesMap} below).

Recall that we call two metric spaces $(X_1,d_1)$ and $(X_2,d_2)$ are \defn{bi-Lipschitz}, \defn{snowflake}, or \defn{quasisymmetrically equivalent} if there exists a homeomorphism from $(X_1,d_1)$ to $(X_2,d_2)$ with the corresponding property.

We recall a version of the definition of Latt\`{e}s maps.

\begin{definition}   \label{defLattesMap}
Let $f\: \widehat{\C}  \rightarrow \widehat{\C}$ be a rational Thurston map on the Riemann sphere $\widehat{\C}$. If $f$ has a parabolic orbifold and is expanding, then it is called a \defn{Latt\`{e}s map}.
\end{definition}

See \cite[Chapter~3]{BM17} and \cite{Mil06} for other equivalent definitions and more properties of Latt\`{e}s maps.

The special phenomenon mentioned above is not common in the study of Prime Orbit Theorems for smooth dynamical systems, as we are endeavoring out of Riemannian setting into general self-similar metric spaces. We content ourselves with the smooth examples of strongly non-integrable potentals for Latt\`{e}s maps in Proposition~\ref{propLattes} below.

\begin{rem}  \label{rmCanonicalOrbifoldMetric}
For a Latt\`{e}s map $f\: \widehat{\C} \rightarrow \widehat{\C}$, the universal orbifold covering map $\Theta \: \C \rightarrow \widehat{\C}$ of the orbifold $\mathcal{O}_f = \bigl( \widehat{\C}, \alpha_f \bigr)$ associated to $f$ is holomorphic (see \cite[Theorem~A.26, Definition~A.27, and Corollary~A.29]{BM17}). Let $d_0$ be the Euclidean metric on $\C$. Then the \defn{canonical orbifold metric} $\omega_f$ of $f$ is the pushforward of $d_0$ by $\Theta$, more presicely,
\begin{equation*}
\omega_f(p,q)  \coloneqq \inf \bigl\{ d_0(z,w) \,\big|\, z\in \Theta^{-1}(p), \, w\in \Theta^{-1}(q)  \bigr\} 
\end{equation*}
for $p,q\in\widehat{\C}$ (see Section~2.5 and Appendices~A.9 and A.10 in \cite{BM17} for more details on the canonical orbifold metric). Let $\sigma$ be the chordal metric on $\C$ as recalled in Remark~\ref{rmChordalVisualQSEquiv}. By \cite[Proposition~8.5]{BM17}, $\omega_f$ is a visual metric for $f$. By \cite[Lemma~A.34]{BM17}, $\bigl( \widehat{\C}, \omega_f \bigr)$ and $\bigl( \widehat{\C}, \sigma \bigr)$ are bi-Lipschitz equivalent, i.e., there exists a bi-Lipschitz homeomorphism $h \:  \widehat{\C}  \rightarrow  \widehat{\C}$ from $\bigl( \widehat{\C}, \omega_f \bigr)$ to $\bigl( \widehat{\C}, \sigma \bigr)$. Moreover, by the discussion in \cite[Appendix~A.10]{BM17}, $h$ cannot be the identity map.
\end{rem}

\begin{prop}   \label{propLattes}
Let $f\: \widehat{\C} \rightarrow \widehat{\C}$ be a Latt\`{e}s map, and $d\coloneqq \omega_f$ be the canonical orbifold metric of $f$ on $\widehat{\C}$ (as recalled in Remark~\ref{rmCanonicalOrbifoldMetric}). Let $\phi \: \widehat{\C} \rightarrow \R$ be a continuously differentiable real-valued function on the Riemann sphere $\widehat{\C}$. Then $\phi \in \Holder{1} \bigl( \widehat{\C}, d \bigr)$, and the following statements are equivalent:
\begin{enumerate}
\smallskip
\item[(i)] $\phi$ is not co-homologous to a constant in $\CCC\bigl(\widehat{\C},\C \bigr)$, i.e., there are no constant $K\in\C$ and function $\beta \in \CCC \bigl( \widehat{\C},\C \bigr)$ with $\phi = K + \beta \circ f - \beta$.

\smallskip
\item[(ii)] $\phi$ is non-locally integrable with respect to $f$ and $d$ (in the sense of Definition~\ref{defLI}).

\smallskip
\item[(iii)] $\phi$ satisfies the $1$-strong non-integrability condition with respect to $f$ and $d$ (in the sense of Definition~\ref{defStrongNonIntegrability}).
\end{enumerate} 
\end{prop}

\begin{proof}
We denote the Euclidean metric on $\C$ by $d_0$. Let $\sigma$ be the chordal metric on $\C$ as recalled in Remark~\ref{rmChordalVisualQSEquiv}. By \cite[Proposition~8.5]{BM17}, the canonical orbifold metric $d=\omega_f$ is a visual metric for $f$. Let $\Lambda>1$ be the expansion factor of $d$ for $f$. 

Let $\mathcal{O}_f=(S^2,\alpha_f)$ be the orbifold associated to $f$ (see Subsection~\ref{subsctNLI_Orbifold}). Since $f$ has no periodic critical points, $\alpha_f(z) < +\infty$ for all $z\in\widehat{\C}$ (see Definition~\ref{defRamificationFn}).

By inequality (A.43) in \cite[Appendix~A.10]{BM17}, 
\begin{equation}   \label{eqPfrmCanonicalOrbifoldMetric_Chordal<Orbifold}
\sup \bigg\{ \frac{\sigma (z_1,z_2)}{d (z_1,z_2)} \,\bigg|\, z_1,z_2 \in \widehat{\C}, \, z_1\neq z_2 \biggr\} < +\infty.
\end{equation}
By (\ref{eqPfrmCanonicalOrbifoldMetric_Chordal<Orbifold}) and the assumption that $\phi$ is continuously differentiable, we get $\phi \in \Holder{1} \bigl( \widehat{\C}, \sigma \bigr) \subseteq \Holder{1} \bigl( \widehat{\C}, d \bigr)$.

\smallskip

We establish the equivalence of statements~(i) through (iii) as follows.

\smallskip

(i) $\Longleftrightarrow$ (ii). The equivalence follows immediately from Theorem~\ref{thmNLI}.

\smallskip

(ii) $\Longleftrightarrow$ (iii). The backward implication follows from Proposition~\ref{propSNI2NLI}. To show the forward implication, we assume that $\phi$ is non-locally integrable. We observe from Lemma~\ref{lmCexistsL}, Theorem~\ref{thmNLI}, and Lemma~\ref{lmSNIwoC} that by replacing $f$ with an iterate of $f$ if necessary, we can assume without loss of generality that there exists a Jordan curve $\CC \subseteq S^2$ such that $\post f \subseteq \CC$, $f(\CC) \subseteq \CC$, and that there exist $\xi=\{ \xi_{\minus i} \}_{i\in\N_0} \in  \Sigma_{f,\,\CC}^-$ and $\eta=\{ \eta_{\minus i} \}_{i\in\N_0} \in  \Sigma_{f,\,\CC}^-$, $X^1\in\X^1 (f,\CC)$, and $u_0,v_0 \in X^1$ with $X^1 \subseteq f (\xi_0) = f(\eta_0)$, and
\begin{equation}   \label{eqPfrmCanonicalOrbifoldMetric_NLI}
\phi^{f,\,\CC}_{\xi,\,\eta} (u_0, v_0) \neq 0.
\end{equation}
By the continuity of $\phi^{f,\,\CC}_{\xi,\,\eta}$ (see Lemma~\ref{lmDeltaHolder} and Definition~\ref{defTemporalDist}), we can assume that $u_0,v_0\in \inte(X^1)$. Without loss of generality, we can assume that $\infty \notin  X^1$. We use the usual coordinate $z = (x,y) \in \R^2$ on $X^1$. We fix a constant $C_{22} \geq 1$ depending only on $f$ and $\CC$ such that 
\begin{equation}   \label{eqPfrmCanonicalOrbifoldMetric_ChordalEuclLocalBiLip}
 C_{22}^{-1} \sigma(z_1,z_2)   \leq d_0 (z_1,z_2)  \leq C_{22} d(z_1,z_2) \qquad \text{ for all } z_1,z_2 \in X^1.
\end{equation}

Recall that $\alpha_f = 1$ for all $z\in\widehat{\C}\setminus \post f$ (see Definition~\ref{defRamificationFn}). By Proposition~A.33 and the discussion proceeding it in \cite[Appendix~A.10]{BM17}, the following statements hold:
\begin{enumerate}
\smallskip
\item[(1)] The canonical orbifold metric $d$ is a singular conformal metric with a conformal factor $\rho$ that is continuous everywhere except at the points in $\supp(\alpha_f) \subseteq \post f$.

\smallskip
\item[(2)] $d(z_1,z_2) = \inf\limits_{\gamma} \int_\gamma \! \rho \,\mathrm{d} \sigma$, where the infimum is taken over all $\sigma$-rectifiable paths $\gamma$ in $\widehat{\C}$ joining $z_1$ and $z_2$.  

\smallskip
\item[(3)] For each $z\in\widehat{\C} \setminus \supp(\alpha_f)$, there exists a neighborhood $U_z \subseteq \widehat{\C}$ containing $z$ and a constant $C_z \geq 1$ such that $C_z^{-1} \leq \rho(u) \leq C_z$ for all $u \in U_z$. 
\end{enumerate}

Choose connected open sets $V$ and $U$ such that $u_0,v_0 \in V \subseteq \overline{V} \subseteq U \subseteq \overline{U} \subseteq \inte(X^1)$. By compactness and statement~(3) above, there exists a constant $C_{23} \geq 1$ such that
\begin{equation} \label{eqPfrmCanonicalOrbifoldMetric_ConformalDensity}
C_{23}^{-1} \leq \rho(z) \leq C_{23} \qquad \text{ for all } z\in\overline{U}. 
\end{equation}
Thus by (\ref{eqPfrmCanonicalOrbifoldMetric_ChordalEuclLocalBiLip}), (\ref{eqPfrmCanonicalOrbifoldMetric_Chordal<Orbifold}), and a simple covering argument using statement~(2) above, inequality~(\ref{eqPfrmCanonicalOrbifoldMetric_ConformalDensity}), and the fact that $\overline{V}\subseteq U$, there exists a constant $C_{24} \geq 1$ depending only on $f$, $\CC$, $d$, $\phi$, and the choices of $U$ and $V$ such that
\begin{equation} \label{eqPfrmCanonicalOrbifoldMetric_BiLipMetrics}
 C_{24}^{-1} d(z_1,z_2)   \leq d_0 (z_1,z_2)  \leq C_{24} d(z_1,z_2)  \qquad \text{ for all } z_1,z_2\in \overline{V}.
\end{equation}

We denote, for each $i\in\N$,
\begin{equation}  \label{eqPfrmCanonicalOrbifoldMetric_tau}
\tau_i \coloneqq  ( f|_{\xi_{1-i}}  )^{-1} \circ \cdots \circ ( f|_{\xi_{ \minus 1}} )^{-1} \circ ( f|_{\xi_{0}} )^{-1} \text{ and }
\tau'_i \coloneqq  ( f|_{\eta_{1-i}}  )^{-1} \circ \cdots \circ ( f|_{\eta_{ \minus 1}} )^{-1} \circ ( f|_{\eta_{0}}  )^{-1}.
\end{equation}
We define a function $\Phi \: X^1 \rightarrow \R$ by $\Phi(z)\coloneqq \phi^{f,\,\CC}_{\xi,\,\eta} (u_0, z)$ for $z\in X^1$ (see Definition~\ref{defTemporalDist} and Lemma~\ref{lmDeltaHolder}).

\smallskip

\emph{Claim.} $\Phi$ is continuously differentiable on $V$.

\smallskip

By Definition~\ref{defTemporalDist}, it suffices to show that the function $\mathcal{D}(\cdot) \coloneqq \Delta^{f,\,\CC}_{\phi,\, \xi} (u_0, \cdot)$ is continuously differentiable on $V$. By Lemma~\ref{lmDeltaHolder}, the function $\mathcal{D}(z) = \sum\limits_{i=0}^{+\infty} ( (\phi \circ \tau_i) (u_0) - (\phi \circ \tau_i) (z) )$ is the uniform limit of a series of continuous functions on $V$. Since $V \subseteq \inte(X^1)$, by (\ref{eqPfrmCanonicalOrbifoldMetric_tau}) and Proposition~\ref{propCellDecomp}~(i), the function $\phi \circ \tau_i$ is differentiable on $V$ for each $i\in\N$.

We fix an arbitrary integer $i\in\N$. For each pair of distinct points $z_1,z_2 \in \inte(X^1)$, we choose the maximal integer $m\in\N$ with the property that there exist two $m$-tiles $X^m_1,X^m_2 \in \X^m(f,\CC)$ such that $z_1 \in X^m_1$, $z_2 \in X^m_2$, and $X^m_1 \cap X^m_2 \neq \emptyset$. Then by Proposition~\ref{propCellDecomp}~(i) and Lemma~\ref{lmCellBoundsBM}~(i) and (ii),
\begin{align*}
            \frac{  \abs{ (\phi\circ \tau_i) (z_1) - (\phi\circ \tau_i) (z_2)  }  }  { d(z_1,z_2) }
\leq & \frac{  \Hnorm{1}{\phi}{ ( \widehat{\C}, d ) }  \diam_d ( \tau_i ( X^m_1 \cup X^m_2 ) ) }  { C^{-1} \Lambda^{-(m+1)} }  \\
\leq & \Hnorm{1}{\phi}{ ( \widehat{\C}, d ) }   \frac{ 2 C \Lambda^{- (m+i)} }  { C^{-1} \Lambda^{-(m+1)} }  \\
\leq & 2 C^2\Hnorm{1}{\phi}{ ( \widehat{\C}, d ) }   \Lambda^{1-i},
\end{align*}
where $C\geq 1$ is a constant from Lemma~\ref{lmCellBoundsBM} depending only on $f$, $\CC$, and $d$. Thus by (\ref{eqPfrmCanonicalOrbifoldMetric_BiLipMetrics}),
\begin{align*}
&                       \sup \biggl\{    \Absbigg{  \frac{\partial}{\partial x}  (\phi\circ\tau_i) (z)  }    \,\bigg|\, z\in V  \biggr\} \\
&\qquad \leq  \sup \biggl\{                       \frac{  \abs{ (\phi\circ \tau_i) (z_1) - (\phi\circ \tau_i) (z_2)  }  }  { d_0(z_1,z_2) }      \,\bigg|\, z_1, z_2\in V ,\, z_1 \neq z_2  \biggr\} \\    
&\qquad \leq  C_{24} \sup \biggl\{          \frac{  \abs{ (\phi\circ \tau_i) (z_1) - (\phi\circ \tau_i) (z_2)  }  }  { d     (z_1,z_2) }      \,\bigg|\, z_1, z_2\in V ,\, z_1 \neq z_2  \biggr\} \\                     
&\qquad \leq   2 C_{24} C^2\Hnorm{1}{\phi}{ ( \widehat{\C}, d ) }   \Lambda^{1-i}.
\end{align*}

Hence $\frac{\partial}{\partial x }  \mathcal{D}$ exists and is continuous on $V$. Similarly, $\frac{\partial}{\partial y }  \mathcal{D}$ exists and is continuous on $V$. Therefore $\mathcal{D}$ is continuously differentiable on $V$, establishing the claim.

\smallskip

By the claim, (\ref{eqPfrmCanonicalOrbifoldMetric_NLI}), and the simple observation that $\phi^{f,\,\CC}_{\xi,\,\eta} (u_0, u_0) =0$, there exist numbers $M_0 \in \N$, $\varepsilon \in (0,1)$, and $C_{25} > 1$, and $M_0$-tiles $Y^{M_0}_\b \in \X^{M_0}_\b(f,\CC)$ and $Y^{M_0}_\w \in \X^{M_0}_\w(f,\CC)$ such that $C_{25} \geq C_{24}$, $Y^{M_0}_\b \cup Y^{M_0}_\w \subseteq V \subseteq \inte(X^1)$, and at least one of the following two inequalities hold:
\begin{enumerate}
\smallskip
\item[(a)] $\inf \bigl\{ \Absbig{\frac{\partial}{\partial x}  \Phi(z) }   \,\big|\,  z\in h^{-1} \bigl( Y^{M_0}_\b \cup Y^{M_0}_\w \bigr)   \bigr\}  \geq 2 C_{25} \varepsilon$, or

\smallskip
\item[(b)] $\inf \bigl\{ \Absbig{\frac{\partial}{\partial y}  \Phi(z) }   \,\big|\,  z\in h^{-1} \bigl( Y^{M_0}_\b \cup Y^{M_0}_\w \bigr)   \bigr\}  \geq 2 C_{25} \varepsilon$.
\end{enumerate}

We assume now that inequality (a) holds, and remark that the proof in the other case is similar.

Without loss of generality, we can assume that $\varepsilon \in \bigl(0,(2C_{25}C)^{-2} \bigr)$.

Then by Lemma~\ref{lmCellBoundsBM}~(v), for each $\c\in\{\b,\w\}$, each integer $M \geq M_0$, and each $M$-tile $X\in \X^M(f,\CC)$ with $X\subseteq  Y^{M_0}_\c$, there exists a point $u_1(X) = (x_1(X), y_0(X)) \in X$ such that $B_d \bigl( u_1(X), C^{-1}\Lambda^{-M} \bigr) \subseteq X$. We choose $x_2(X)\in\R$ such that $\abs{x_1(X) - x_2(X) } = (4C_{25}C)^{-1} \Lambda^{-M}$. Then by (\ref{eqPfrmCanonicalOrbifoldMetric_BiLipMetrics}) and $C_{25} \geq C_{24}$, we get
\begin{align}   \label{eqPfrmCanonicalOrbifoldMetric_u1u2Def}
u_2(X) \coloneqq (x_2(X), y_0(X)) \in & B_{d_0} \bigl(u_1(X), (2C_{25}C)^{-1} \Lambda^{-M} \bigr)  \notag \\
                                                                 &     \subseteq B_d  \bigl(u_1(X), (  2    C)^{-1} \Lambda^{-M} \bigr)  \\
                                                                 &     \subseteq B_d  \bigl(u_1(X),  C^{-1} \Lambda^{-M} \bigr)   \subseteq X.  \notag
\end{align}
In particular, the entire horizontal line segment connecting $u_1(X)$ and $u_2(X)$ is contained in $\inte (X)$. By (\ref{eqPfrmCanonicalOrbifoldMetric_u1u2Def}), Lemma~\ref{lmCellBoundsBM}~(ii), (\ref{eqPfrmCanonicalOrbifoldMetric_BiLipMetrics}), and $C_{25} \geq C_{24}$, we get
\begin{align}  \label{eqPfrmCanonicalOrbifoldMetric_u1u2}
&                       \min \bigl\{  d \bigl( u_1(X), \widehat{\C} \setminus X \bigr), \, d \bigl( u_2(X),  \widehat{\C} \setminus X \bigr), \, d( u_1(X), u_2(X) ) \bigr\}  \\
&\qquad \geq \min \bigl\{ (  2    C)^{-1} \Lambda^{-M} ,  \,  C_{25}^{-1} (4C_{25}C)^{-1} \Lambda^{-M}   \bigr\}
                 \geq \varepsilon \diam_d(X).   \notag
\end{align}
On the other hand, by (\ref{eqPfrmCanonicalOrbifoldMetric_BiLipMetrics}), $C_{25} \geq C_{24}$, Definition~\ref{defTemporalDist}, inequality~(a) above, and the mean value theorem, 
\begin{equation*}
           \frac{ \Absbig{  \phi^{f,\,\CC}_{\xi,\,\eta} ( u_1(X), u_2(X) ) } } { d ( u_1(X), u_2(X) )}
\geq   \frac{ \Absbig{  \phi^{f,\,\CC}_{\xi,\,\eta} ( u_1(X), u_2(X) ) } } { C_{25} d_0( u_1(X), u_2(X) )}
 =        \frac{ \abs{ \Phi( u_1(X) ) - \Phi( u_2(X) )}  }   {  C_{25}  \abs{ x_1(X) - x_2(X) }} 
 \geq   2 \varepsilon.
\end{equation*}

We choose
\begin{equation}  \label{eqPfrmCanonicalOrbifoldMetric_N0}  
N_0    \coloneqq  \biggl\lceil  \log_\Lambda  \frac{ 2 C^2 \varepsilon^{-2}\Hseminorm{1, (\widehat{\C},d)}{\phi} C_0}{1-\Lambda^{-1}}     \biggr\rceil.    
\end{equation}
where $C_0 > 1$ is a constant depending only on $f$, $\CC$, and $d$ from Lemma~\ref{lmMetricDistortion}.

Fix arbitrary $N\geq N_0$. Define $X^{N+M_0}_{\c,1} \coloneqq \tau_{N} \bigl( Y^{M_0}_\c \bigr)$ and $X^{N+M_0}_{\c,2} \coloneqq \tau'_{N} \bigl( Y^{M_0}_\c \bigr)$ (c.f.\ (\ref{eqPfrmCanonicalOrbifoldMetric_tau})). Note that $\varsigma_1 = \tau_{N}|_{ Y^{M_0}_\c }$ and $\varsigma_2 = \tau'_{N}|_{ Y^{M_0}_\c }$.

Then by Definition~\ref{defTemporalDist}, Lemma~\ref{lmDeltaHolder}, (\ref{eqPfrmCanonicalOrbifoldMetric_u1u2}),  Lemma~\ref{lmSnPhiBound}, Proposition~\ref{propCellDecomp}~(i), and Lemma~\ref{lmCellBoundsBM}~(i) and (ii),
\begin{align*}
&                       \frac{ \abs{  S_{N }\phi ( \varsigma_1 (u_1(X)) ) -  S_{N }\phi ( \varsigma_2 (u_1(X)) )  -S_{N }\phi ( \varsigma_1 (u_2(X)) ) +  S_{N }\phi ( \varsigma_2 (u_2(X)) )   }  } 
                                  {  d(u_1(X),u_2(X))  }  \\
&\qquad \geq  \frac{ \Absbig{ \phi^{f,\,\CC}_{\xi,\,\eta} (u_1(X)) , u_2(X))  }   } { d(u_1(X),u_2(X)) }
                                 -  \limsup\limits_{n\to+\infty}       \frac{ \abs{  S_{n - N }\phi ( \tau_n (u_1(X)) ) -  S_{n - N }\phi ( \tau_n  (u_2(X)) )  }  }   {  \varepsilon \diam_d (X) }  \\
&\qquad\qquad    -  \limsup\limits_{n\to+\infty}       \frac{ \abs{  S_{n - N }\phi ( \tau'_n (u_1(X)) ) -  S_{n - N }\phi ( \tau'_n (u_2(X)) )  }  }   {  \varepsilon \diam_d (X) }              \\
&\qquad \geq 2 \varepsilon -  \frac{\Hseminorm{1, (\widehat{\C},d)}{\phi} C_0}{1-\Lambda^{-1}} \cdot
                                                     \frac{    d ( \tau_{N} (u_1(X)) , \tau_{N} (u_2(X)))  +  d ( \tau'_{N} (u_1(X)) , \tau'_{N} (u_2(X))) }   {  \varepsilon   \diam_d (X) }   \\
&\qquad \geq 2 \varepsilon -  \frac{\Hseminorm{1, (\widehat{\C},d)}{\phi} C_0}{1-\Lambda^{-1}} \cdot
                                                     \frac{  \diam_d ( \tau_{N} (X) )   +    \diam_d ( \tau'_{N} (X) ) }   {  \varepsilon   \diam_d (X) }           \\
&\qquad \geq 2 \varepsilon -   \frac{\Hseminorm{1, (\widehat{\C},d)}{\phi} C_0}{1-\Lambda^{-1}} \cdot
                                                     \frac{  2 C  \Lambda^{ -  (M  + N )}  }   {  \varepsilon  C^{-1} \Lambda^{ - M} }           
                \geq 2 \varepsilon -   \frac{ 2 C^2 \varepsilon^{-1}\Hseminorm{1, (\widehat{\C},d)}{\phi} C_0}{1-\Lambda^{-1}}     \Lambda^{ - N_0}
                \geq \varepsilon.                                                                                                 
\end{align*}
where the last inequality follows from (\ref{eqPfrmCanonicalOrbifoldMetric_N0}). 

Therefore $\phi$ satisfies the $1$-strong non-integrability condition with respect to $f$ and $d$.
\end{proof}

\begin{proof}[Proof of Theorem~\ref{thmLattesPOT}]
By Proposition~\ref{propLattes}, $\phi \in \Holder{\alpha}\bigl( \widehat{\C},d \bigr)$. So the existence and uniqueness of $s_0>0$ follows from Corollary~\ref{corS0unique}.

The implication (i)$\implies$(iii) follows from Proposition~\ref{propLattes} and Theorem~\ref{thmPrimeOrbitTheorem}. The implication (iii)$\implies$(ii) is trivial. The implication (ii)$\implies$(i) follows immediately by a contradiction argument using Theorem~\ref{thmNLI} and Corollary~\ref{corNecessary}.
\end{proof}

\subsection{Genericity of strongly non-integrable potentials}   \label{subsctGeneric}

We recall some concepts related to the expansion of expanding Thurston maps from a combinatorial point of view.

Suppose $f\: S^2 \rightarrow S^2$ be a Thurston map and $\CC\subseteq S^2$ is a Jordan curve with $\post f \subseteq \CC$. For each $n\in\N_0$, we denote by $D_n(f,\CC)$ the minimal number of $n$-tiles required to form a connected set joining opposite sides of $\CC$; more precisely,
\begin{align}   \label{eqDefDn}
D_n(f,\CC) \coloneqq \min \biggl\{ N\in\N \,\bigg|\, & \text{there exist } X_1, X_2, \dots, X_N \in \X^n(f,\CC)  \text{ such that } \\
                                                                                            & \bigcup\limits_{j=1}^{N} X_j \text{ is connected and joins opposite sides of } \CC \biggr\}.   \notag
\end{align}
See \cite[Section~5.7]{BM17} for more properties of $D_n(f,\CC)$. M.~Bonk and D.~Meyer showed in \cite[Proposition~16.1]{BM17} that the limit
\begin{equation}  \label{eqDefCombExpansionFactor}
\Lambda_0(f) \coloneqq \lim\limits_{n\to+\infty} D_n(f,\CC)^{1/n}
\end{equation}
exists and is independent of $\CC$. We have $\Lambda_0(f) \in (1,+\infty)$. The constant $\Lambda_0(f)$ is called the \defn{combinatorial expansion factor} of $f$.

The combinatorial expansion factor $\Lambda_0(f)$ serves as a sharp upper bound of the expansion factors of visual metrics of $f$; more precisely, for an expanding Thurston map $f$, the following statements are true:
\begin{enumerate}
\smallskip
\item[(i)] If $\Lambda$ is the expansion factor of a visual metric for $f$, then $\Lambda\in (1, \Lambda_0(f) ]$.

\smallskip
\item[(ii)] Conversely, if $\Lambda \in (1, \Lambda_0(f) )$, then there exists a visual metric for $f$ with expansion factor $\Lambda$.
\end{enumerate}
See \cite[Theorem~16.3]{BM17} for more details.

\begin{lemma}   \label{lmDisjointBackwardOrbits}
Let $f$ and $\CC$ satisfy the Assumptions. We assume in addition that $f(\CC) \subseteq \CC$. Then there exist two sequences of $1$-tiles $\{ \xi_{\minus i} \}_{i\in\N_0}, \{ \xi'_{\minus i'} \}_{i'\in\N_0}  \in \Sigma_{f,\,\CC}^-$ such that $f (\xi_0 ) = f (\xi'_0)$ and $\xi_{\minus i} = \xi_0 \neq \xi'_{\minus i'}$ for all $i,i'\in\N_0$.
\end{lemma}

Recall that $\Sigma_{f,\,\CC}^-$ is defined in (\ref{eqDefSigma-}).

\begin{proof}
We first claim that if the white $0$-tile $X^0_\w \in\X^0$ does not contain a white $1$-tile, then there exists a black $1$-tile $X^1_\b \in \X^1_\b$ such that $X^1_\b = X^0_\w$.

Indeed, note that for each $1$-edge $e^1\in\E^1$, there exists a unique black $1$-tile $X_\b\in \X^1_\b$ and a unique white $1$-tile $X_\w\in \X^1_\w$ such that $X_\b \cap X_\w = e^1$. Suppose that $X^0_\w$ is a union  $X^0_\w = \bigcup\limits_{i=1}^k X_i$ of $k$ distinct black $1$-tiles $X_i \in \X^1_\b$, $i\in\{1,2,\dots,k\}$, then $\bigcup\limits_{i=1}^k \partial X_i \subseteq \partial X^0_\w = \CC$. Since each of $\CC$ and $\partial X_i$, $i\in\{1,2,\dots,k\}$, is a Jordan curve and $\partial X_j \neq \partial X_{j'}$ for $1 \leq j < j' \leq k$, we conclude that $k=1$, establishing the claim.

\smallskip

Similar statement holds if we exchange black and white.

Next, we observe that if the white $0$-tile $X^0_\w$ is also a white $1$-tile or the black $0$-tile $X^0_\b$ is also a black $1$-tile, then $f$ cannot be expanding.

Hence it suffices to construct the sequences $\{ \xi_{\minus i} \}_{i\in\N_0}$ and $\{ \xi'_{\minus i'} \}_{i'\in\N_0}$ in the following two cases:

\smallskip

\emph{Case 1.} Either $X^0_\w = X^1_\b$ for some black $1$-tile $X^1_\b \in \X^1_b$ or $X^0_\b = X^1_\w$ for some white $1$-tile $X^1_\w \in \X^1_w$. Without loss of generality, we assume the former holds. Since $\deg f \geq 2$, we can choose a black $1$-tile $Y^1_\b \in \X^1_\b$ and a white $1$-tile $Y^1_\w \in \X^1_\w$ such that $Y^1_\b \cup Y^1_\w \subseteq X^0_\b$. Then we define $\xi_{\minus i} \coloneqq Y^1_\b$ for all $i\in\N_0$, $\xi'_{\minus i'} \coloneqq X^1_\b$ if $i'\in\N_0$ is even, and $\xi'_{\minus i'} \coloneqq Y^1_\w$ if $i'\in\N_0$ is odd.

\smallskip

\emph{Case 2.} There exist black $1$-tiles $X^1_\b, Y^1_\b \in \X^1_\b$ and white $1$-tiles $X^1_\w, Y^1_\w \in \X^1_\w$ such that $X^1_\b \cup X^1_\w \subseteq X^0_\w$ and $Y^1_\b \cup Y^1_\w \subseteq X^0_\b$. Then we define $\xi_{\minus i} \coloneqq Y^1_\b$ for all $i\in\N_0$, $\xi'_0 \coloneqq X^1_\b$, and $\xi'_{\minus i'} \coloneqq X^1_\w$ for all $i'\in\N$.

\smallskip

It is trivial to check that in both cases, $\{ \xi_{\minus i} \}_{i\in\N_0}, \{ \xi'_{\minus i'} \}_{i'\in\N_0}  \in \Sigma_{f,\,\CC}^-$,  $f  (\xi_0  ) = f  (\xi'_0 )$, and $\xi_{\minus i} = \xi_0 \neq \xi'_{\minus i'}$ for all $i,i'\in\N_0$.
\end{proof}

\begin{theorem}     \label{thmPerturbToStrongNonIntegrable}
Let $f\: S^2 \rightarrow S^2$ be an expanding Thurston map with an Jordan curve $\CC\subseteq S^2$ satisfying $\post f\subseteq \CC$ and $f(\CC)\subseteq \CC$. Let $d$ be a visual metric on $S^2$ for $f$ with expansion factor $\Lambda>1$. Given $\alpha\in(0,1]$. 

We assume in addition that $\Lambda^\alpha < \Lambda_0(f)$. Then there exists a constant $C_{27} > 0$ such that for each $\varepsilon>0$ and each  real-valued H\"{o}lder continuous function $\varphi \in \Holder{\alpha}(S^2,d)$ with an exponent $\alpha$, there exist integers $N_0,M_0\in \N$, $M_0$-tiles $Y^{M_0}_\b \in \X^{M_0}_\b(f,\CC)$, $Y^{M_0}_\w \in \X^{M_0}_\w(f,\CC)$, and a real-valued H\"{o}lder continuous function $\phi \in \Holder{\alpha}(S^2,d)$ such that for each $\c\in\{\b,\w\}$, each integer $M\geq M_0$, and each $M$-tile $X\in\X^M(f,\CC)$ with $X\subseteq Y^{M_0}_\c$, there exist two points $x_1(X), x_2(X) \in X$ with the following properties:
\begin{enumerate}
\smallskip
\item[(i)] $\min \{ d(x_1(X), S^2 \setminus X), \, d(x_2(X), S^2 \setminus X), \, d(x_1(X), x_2(X)) \} \geq \varepsilon \diam_d(X)$.

\smallskip
\item[(ii)] for each integer $N' \geq N_0$, there exist two  $(N'+M_0)$-tiles $X^{N'+M_0}_{\c,1}, X^{N'+M_0}_{\c,2} \in \X^{N'+M_0}(f,\CC)$ such that $Y^{M_0}_\c = f^{N'}\bigl(  X^{N'+M_0}_{\c,1}  \bigr) =   f^{N'}\bigl(  X^{N'+M_0}_{\c,2}  \bigr)$, and that
\begin{equation}   \label{eqSNIBoundsPerturb}
\frac{ \Abs{  S_{N' }\phi ( \varsigma_1 (x_1(X)) ) -  S_{N' }\phi ( \varsigma_2 (x_1(X)) )  -S_{N' }\phi ( \varsigma_1 (x_2(X)) ) +  S_{N' }\phi ( \varsigma_2 (x_2(X)) )   }  } {d(x_1(X),x_2(X))^\alpha}
\geq \varepsilon,
\end{equation}
where we write $\varsigma_1 \coloneqq \Bigl(f^{N'}\big|_{X^{N'+M_0}_{\c,1}} \Bigr)^{-1}$ and $\varsigma_2 \coloneqq \Bigl(f^{N'}\big|_{X^{N'+M_0}_{\c,2}} \Bigr)^{-1}$. 

\smallskip
\item[(iii)] $\Hnorm{\alpha}{\phi - \varphi}{(S^2,d)} \leq C_{27} \varepsilon$.
\end{enumerate}
\end{theorem}

\begin{proof}

Denote
\begin{equation}   \label{eqDefC26}
C_{26} \coloneqq 4 C^\alpha \Lambda^\alpha > 1.
\end{equation}
Here $C\geq 1$ is a constant from Lemma~\ref{lmCellBoundsBM} depending only on $f$, $\CC$, and $d$.

Since $\Lambda^\alpha < \Lambda_0(f) = \lim\limits_{n\to+\infty} D_n(f,\CC)^{1/n}$ (see (\ref{eqDefCombExpansionFactor})), we can fix $N\in\N$ large enough such that the following statements are satisfied:
\begin{itemize}
\smallskip
\item $3< 3  C_{26} C<\Lambda^{\alpha N} < D_N(f,\CC) - 1$.

\smallskip
\item There exist $u^1_\b, u^2_\b, u^1_\w, u^2_\w \in \V^N$ such that for all $\c \in \{ \b, \w \}$,
\begin{align}
\overline{W}^N \bigl( u^1_\c \bigr) \cup  \overline{W}^N \bigl( u^2_\c \bigr) & \subseteq \inte \bigl( X^0_\c \bigr),   \label{eqPfthmPerturbToStrongNonIntegrable_FlowerLocation} \\
\overline{W}^N \bigl( u^1_\c \bigr) \cap  \overline{W}^N \bigl( u^2_\c \bigr) &  = \emptyset.   \label{eqPfthmPerturbToStrongNonIntegrable_FlowersDisjoint}
\end{align}
\end{itemize}

We denote $D_N \coloneqq D_N(f,\CC)$ in the remaining part of this proof.

It suffices to establish the theorem for $\varepsilon >0$ sufficiently small. Fix arbitrary
\begin{equation}   \label{eqPfthmPerturbToStrongNonIntegrable_varepsilon}
\varepsilon \in \bigl(0, C^{-2} \Lambda^{ - 2N } \bigr)  \subseteq (0,1).
\end{equation} 

We define the following constants
\begin{align}
\rho    & \coloneqq \frac{ \Lambda^{\alpha N} } { D_N - 1 } \in (0,1), \label{eqPfthmPerturbToStrongNonIntegrable_rho}  \\
C_{27} & \coloneqq 1 + C_{26}   C \Bigl( 4   (1-\rho)^{-1} + \Lambda^{ \alpha N} \bigl( 1 - \Lambda^{ - \alpha N} \bigr)^{-1} \Bigr),  \label{eqDefC27}  \\
N_0    & \coloneqq  \biggl\lceil \frac{1}{\alpha} \log_\Lambda  \frac{ 2 C^2 \varepsilon^{-1 - \alpha} \bigl(\Hnorm{\alpha}{\varphi}{(S^2,d)} + \varepsilon C_{27} \bigr) C_0}{1-\Lambda^{-\alpha}}   \biggr\rceil.    \label{eqPfthmPerturbToStrongNonIntegrable_N0}  
\end{align}
Here $C_0 > 1$ is a constant depending only on $f$, $\CC$, and $d$ from Lemma~\ref{lmMetricDistortion}.

Choose two sequences of $1$-tiles $\xi \coloneqq \{ \xi_{\minus i} \}_{i\in\N_0}  \in \Sigma_{f,\,\CC}^-$ and $\xi' \coloneqq \{ \xi'_{ \minus i'} \}_{i'\in\N_0}  \in \Sigma_{f,\,\CC}^-$ as in Lemma~\ref{lmDisjointBackwardOrbits} such that $f  ( \xi_0  ) = f  ( \xi'_0 )$ and $\xi_{\minus i} = \xi_0 \neq \xi'_{\minus i'}$ for all $i,i'\in\N_0$. We denote, for each $j\in\N$,
\begin{equation}  \label{eqPfthmPerturbToStrongNonIntegrable_tau}
\tau_j \coloneqq \bigl( f|_{\xi_{1-j}} \bigr)^{-1} \circ \cdots \circ ( f|_{\xi_{ \minus 1}} )^{-1} \circ ( f|_{\xi_{0}} )^{-1} \text{ and }
\tau'_j \coloneqq \bigl( f|_{\xi'_{1-j}} \bigr)^{-1} \circ \cdots \circ \bigl( f|_{\xi'_{ \minus 1}} \bigr)^{-1} \circ \bigl( f|_{\xi'_{0}} \bigr)^{-1}.
\end{equation}

Since $f$ is expanding Thurston map, we have $f  ( \xi_0  ) \supsetneq \xi_0$. Thus we can fix a constant
\begin{equation}  \label{eqPfthmPerturbToStrongNonIntegrable_M0} 
M_0  \geq    \frac{1}{\alpha} \log_\Lambda \frac{2  C_{26}}{ 1 - \Lambda^{-\alpha N}  }  
\end{equation}
large enough such that we can choose $Y^{M_0}_\b \in \X^{M_0}_\b$ and $Y^{M_0}_\w \in \X^{M_0}_\w$ with $Y^{M_0}_\b \cap Y^{M_0}_\w \neq \emptyset$ and
\begin{equation}   \label{eqPfthmPerturbToStrongNonIntegrable_Y_location}
Y^{M_0}_\b \cup Y^{M_0}_\w \subseteq \inte  ( f  ( \xi_0  )  ) \setminus \xi_0.
\end{equation}
We fix such $Y^{M_0}_\b \in \X^{M_0}_\b$ and $Y^{M_0}_\w \in \X^{M_0}_\w$. See Figure~\ref{figPerturb}.

\smallskip

\begin{figure}
    \centering
    \begin{overpic}
    [width=15cm, 
    tics=20]{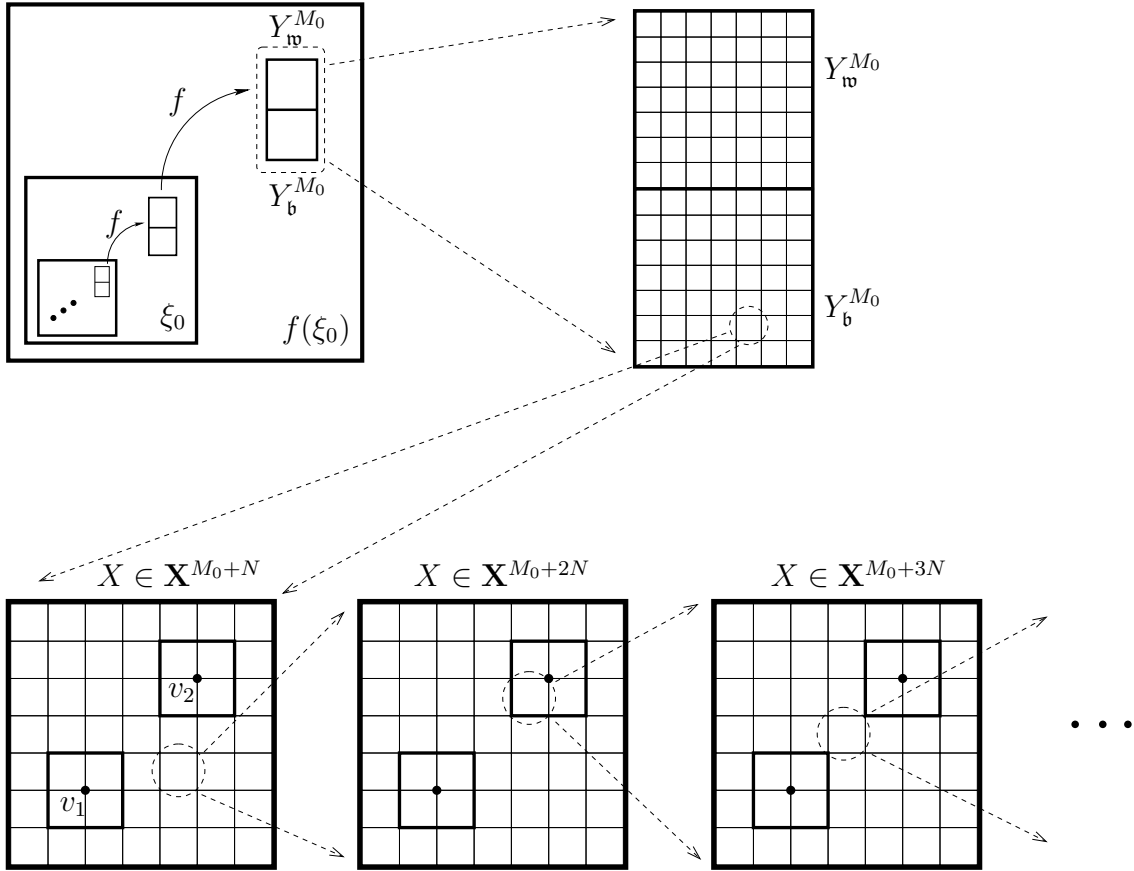}
    \put(100,316){$Y^{M_0}_\w$}
    \put(100,252){$Y^{M_0}_\b$}
    \put(62,288){$f$}
    \put(38,242){$f$}
    \put(105,202){$f(\xi_0)$}
    \put(59,207){$\xi_0$}
    \put(310,300){$Y^{M_0}_\w$}
    \put(310,210){$Y^{M_0}_\b$}
    \put(35,108){$X\in\X^{M_0 + N}$}
    \put(155,108){$X\in\X^{M_0 + 2N}$}
    \put(290,108){$X\in\X^{M_0 + 3N}$}
    \put(21,23){$v_1$}
    \put(62,66){$v_2$}
    \end{overpic}
    \caption{Constructions for the proof of Theorem~\ref{thmPerturbToStrongNonIntegrable}.}
    \label{figPerturb}
\end{figure}

We want to construct, for each $n\in\N_0$ and each $(n+N)$-vertex $v\in \V^{n+N}$, a non-negative bump function $\Upsilon_{v,n} \: S^2 \rightarrow [0,+\infty)$ that satisfies the following properties:
\begin{enumerate}
\smallskip
\item[(a)] $\Upsilon_{v,n} (v) = C_{26}\Lambda^{-\alpha n} \varepsilon$ and $\Upsilon_{v,n} (x) = 0$ if $x\in S^2 \setminus W^{n+N} (v)$.

\smallskip
\item[(b)] $\norm{\Upsilon_{v,n}}_{\CCC^0(S^2)} = C_{26}\Lambda^{-\alpha n} \varepsilon$.

\smallskip
\item[(c)] For each $m\in\N$, each $X\in \X^{n+mN}$, and each pair of points $x,y\in X$,
\begin{equation}  \label{eqPfthmPerturbToStrongNonIntegrable_SingleBumpDiffInEachScale}
\abs{ \Upsilon_{v,n} (x) - \Upsilon_{v,n} (y) } \leq C_{26} \Lambda^{-\alpha n} \varepsilon (D_N - 1)^{-(m-1)}.
\end{equation}
\end{enumerate}

\smallskip

Fix arbitrary $n\in\N_0$ and $v\in \V^{n+N}$.

In order to construct such $\Upsilon_{v,n}$, we first need to construct a collection of sets whose boundaries serve as level sets of $\Upsilon_{v,n}$. More precisely, we will construct a collection of closed subsets $\{U_{\overline{i}} \}_{\overline{i} \in I}$ of $W^{n+N}(v)$ indexed by
\begin{equation}  \label{eqPfthmPerturbToStrongNonIntegrable_I}
I \coloneqq \bigcup\limits_{k\in \N} \{0,1,\dots, D_N-1 \}^k
\end{equation}
that satisfy the following properties:
\begin{enumerate}
\smallskip
\item[(1)] $U_{\overline{i}}$ is either $\{v\}$ or a nonempty union of $(n+(k+1)N)$-tiles if the \emph{length} of $\overline{i}\in I$ is $k\in\N$, i.e., $\overline{i} \in \{ 0, 1, \dots, D_N - 1 \}^k$. Moreover, $U_{\overline{i}} = \{v\}$ if and only if $\overline{i} \eqqcolon (i_1, i_2, \dots, i_k) = (0, 0 , \dots, 0)$.

\smallskip
\item[(2)] $S^2 \setminus U_{\overline{i}}$ is a finite disjoint union of simply connected open sets for each $\overline{i}\in I$.

\smallskip
\item[(3)] $U_{(i_1,i_2,\dots, i_k)} = U_{(i_1,i_2,\dots, i_k,0)}$ for each $k\in\N$ and each $\overline{i} = (i_1,i_2,\dots, i_k) \in I$.

\smallskip
\item[(4)] $U_{\overline{i}} \subseteq \inter U_{\overline{j}}  \subseteq  U_{\overline{j}}  \subseteq  W^{n+N}(v)$ for all $\overline{i}, \overline{j} \in I$ with $\overline{i} < \overline{j}$.
\end{enumerate}
Here we say $\overline{i} < \overline{j}$, for $\overline{i} = (i_1, i_2, \dots, i_k) \in I$ and $\overline{j} = (j_1, j_2, \dots, j_{k'}) \in I$, if one of the following statements is satisfied:
\begin{itemize}
\smallskip
\item $k<k'$, $i_l = j_l$ for all $l \in \N$ with $l\leq k$, and $j_{l'} \neq 0$ for some $l'\in\N$ with $k < l' \leq k'$. 

\smallskip
\item There exists $l'\in\N$ with $l' \leq \min\{k, k'\}$ such that $i_{l'} < j_{l'}$ and $i_l = j_l$ for all $l\in\N$ with $l< l'$.
\end{itemize}

We say $\overline{i} \leq \overline{j}$ for $\overline{i}, \overline{j} \in I$ if either $\overline{i} < \overline{j}$ or $\overline{i} = \overline{j}$.

We denote
\begin{equation}  \label{eqPfthmPerturbToStrongNonIntegrable_Ik}  
I_0 \coloneqq \emptyset, \text{ and } I_l \coloneqq \bigcup\limits_{k=1}^l  \{1, \dots, D_N - 1\}^k \text{ for each } l\in\N.
\end{equation}
\smallskip

We construct $U_{\overline{i}}$ recursively on the length of $\overline{i}\in I$. 

We set $U_{(0)} \coloneqq \{v\}$. For $\overline{i} = (i_1)$, $i_1\in \{1, \dots, D_N - 1\}$, we define a connected closed set 
\begin{align*}
U{(i_1)}  \coloneqq \bigcup \biggl\{  X_{i_1} \,\bigg|\,  & \text{there exist } X_1, X_2, \dots, X_{i_1} \in \X^{n+2N} \\
                                                                                           & \text{ such that }  \bigcup\limits_{m=1}^{i_1} X_m \text{ is connected and } v \in   X_1   \biggr\}.
\end{align*}
Note that $U_{(i_1)} \subseteq W^{n+N}(v)$ for $i_1\in \{1, 2, \dots, D_N - 1\}$ since otherwise there would exist $X_1, X_2, \dots, X_{i_1} \in \X^{n+2N}$ such that the union $\bigcup\limits_{m=1}^{i_1} f^{n+N}(X_m)$ of $N$-tiles $f^{n+N}(X_m)\in \X^N$ (see Proposition~\ref{propCellDecomp}~(i)), $m \in \{1,2,\dots,i_1\}$, is connected and joins opposite sides of $\CC$ which is impossible due to the definition of $D_N$ (see (\ref{eqDefDn})). Then Properties~(1), (2), and (4) hold for $\overline{i},\overline{j} \in \{0, 1, \dots, D_N - 1\}^1$ by our construction.

\begin{figure}
    \centering
    \begin{overpic}
    [width=15cm, 
    tics=20]{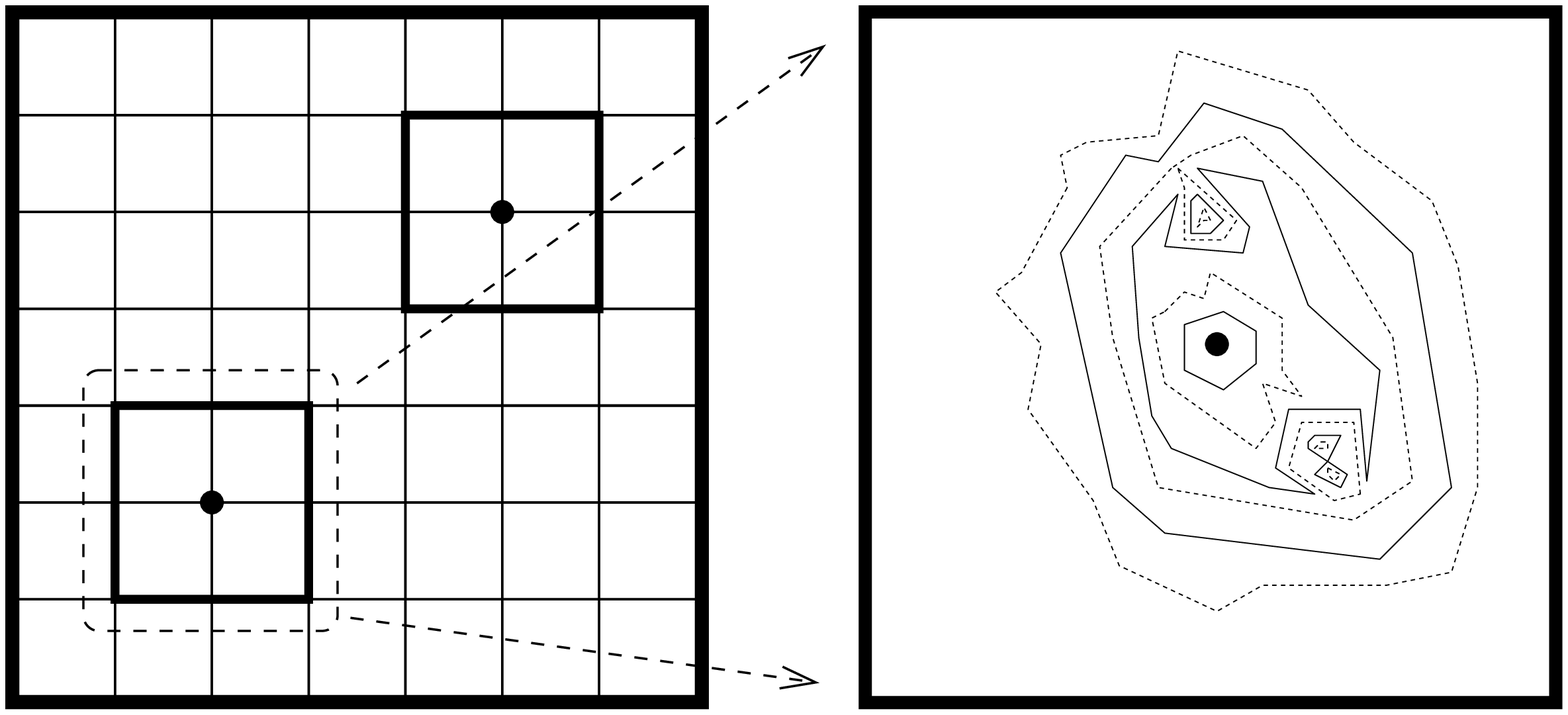}
    \put(47,45){$v_1(X)$}
    \put(126,142){$v_2(X)$}
    \put(5,-10){$X\in\X^{M_0 + N}$}
    \put(239,9){$W^{M_0 + 2 N} (v_1(X))$}
    \end{overpic}
    \caption{Level sets $\partial U_{(i_1)}$, $i_1\in\{1,2,\dots, D_N - 1\}$, of $\Upsilon_{v_1(X),\,M_0 + N}$.}
    \label{figBumpFn1}

    \centering
    \begin{overpic}
    [width=15cm, 
    tics=20]{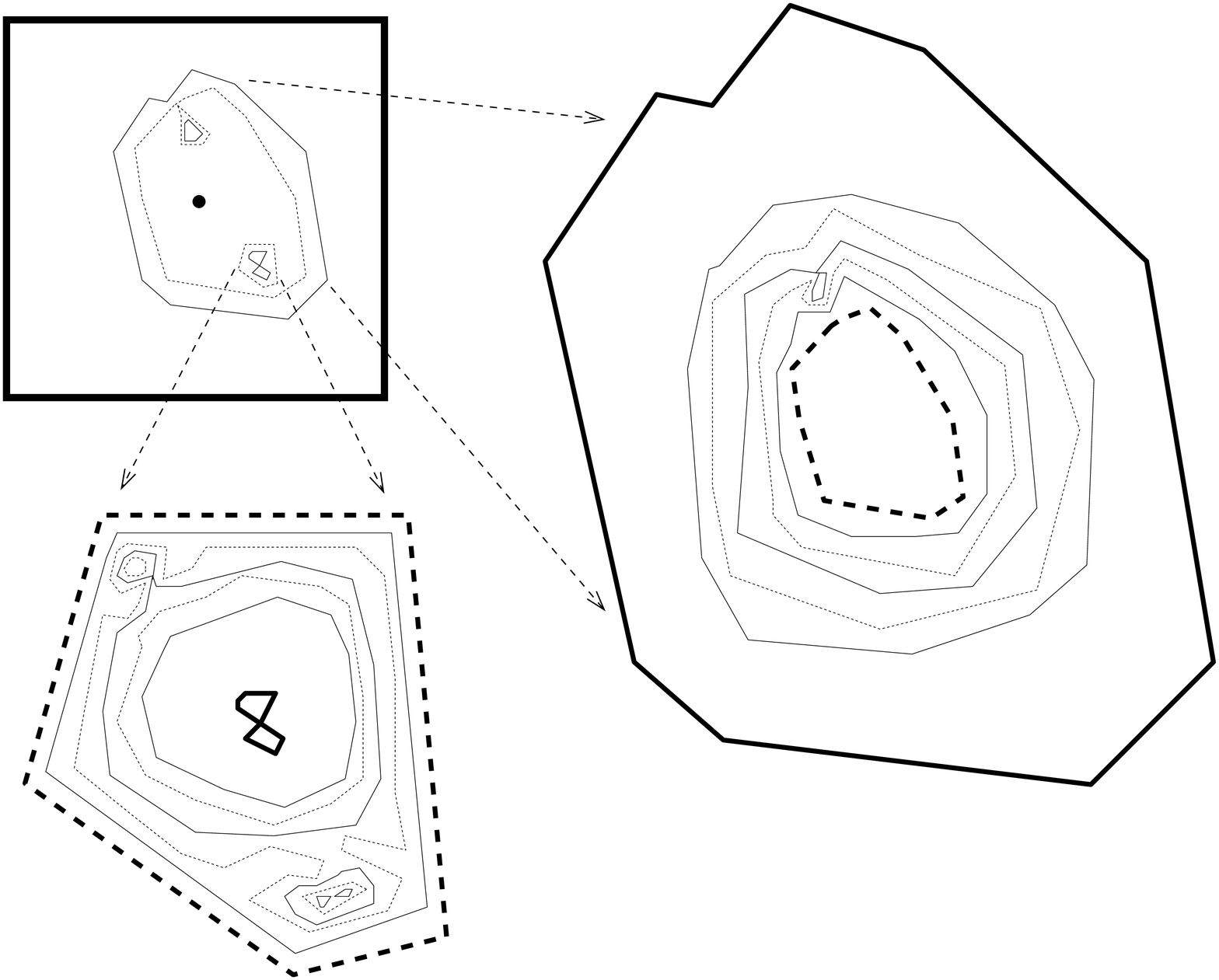}
    \put(60,278){$v_1(X)$}
    \put(5,324){$W^{M_0 + 2 N} (v_1(X))$}
    \end{overpic}
    \caption{Level sets $\partial U_{(4,i_2)}$, $i_2\in\{1,2,\dots, D_N - 1\}$, of $\Upsilon_{v_1(X),\,M_0 + N}$.}
    \label{figBumpFn2}
\end{figure}

Assume that we have constructed $U_{\overline{i}} \subseteq W^{n+N}(v)$ for each $\overline{i} \in I_l$ for some $l\in\N$, that Property~(3) is satisfied for each $\overline{i}\in I_{l-1}$, and that Properties~(1), (2), and (4) are satisfied for all $\overline{i}, \overline{j} \in I_l$.

Fix arbitrary $\overline{i} = (i_1, i_2, \dots, i_l ) \in  \{0, 1, \dots, D_N - 1\}^l$ and $i_{l + 1} \in \{1, 2, \dots, D_N - 1\}$. Denote $\overline{j} \coloneqq (i_1, i_2, \dots, i_l, i_{l + 1})$. We set $U_{(i_1, i_2, \dots, i_l, 0 )} \coloneqq U_{\overline{i}}$. We define a connected closed set
\begin{align*}
U_{\overline{j}}  \coloneqq U_{\overline{i}}  \cup 
                           \bigcup \biggl\{  X_{i_{l + 1}} \,\bigg|\,  & \text{there exist } X_1, X_2, \dots, X_{i_{l + 1}} \in \X^{n+(l + 2)N} \\
                                                                                               & \text{ such that }  \bigcup\limits_{m=1}^{i_{l + 1}} X_m \text{ is connected and } U_{\overline{i}} \cap  X_1 \neq \emptyset   \biggr\}.
\end{align*} 

\smallskip

\emph{Claim~1.} $U_{\overline{j}} \subseteq \inter U_{(i_1, i_2, \dots, i_{l-1}, 1 + i_l)}$ if $i_l \neq D_N - 1$, and $U_{\overline{j}} \subseteq W^{n+N}(v)$ if $i_l = D_N - 1$.

\smallskip

We first establish Claim~1 in the case $i_l \neq D_N - 1$. Denote $\overline{i'} \coloneqq (i_1, i_2, \dots, i_{l-1}, 1 + i_l)$. By Property~(1) of $\{U_{\overline{i}} \}_{\overline{i} \in I_l}$, $U_{\overline{i}}$ and $U_{\overline{i'}}$ are unions of $(n + (l+1)N)$-tiles. By Property~(4) of $\{U_{\overline{i}} \}_{\overline{i} \in I_l}$, $U_{\overline{i}} \subseteq \inter U_{\overline{i'}}$, so $\partial U_{\overline{i}} \cap \partial U_{\overline{i'}} = \emptyset$. We argue by contradiction and assume that $U_{\overline{j}} \nsubseteq \inter U_{\overline{i'}}$. Then there exist $X_1, X_2, \dots, X_{i_{l+1}} \in \X^{n+(l+2)N}$ such that the union $K \coloneqq \bigcup\limits_{m=1}^{i_{l + 1}} X_m$ is a connected set that intersects both $\partial U_{\overline{i}}$ and $\partial U_{\overline{i'}}$ nontrivially. Then $K$ cannot be a subset of a single $(n+(l+1)N)$-flower (of an $(n+(l+1)N)$-vertex).  Since each connected component of the preimage of a $0$-flower under $f^{n+(l+1)N}$ is an $(n+(l+1)N)$-flower, we observe that $f^{n+(l+1)N}(K)$ cannot be a subset of a single $0$-flower (of a $0$-vertex), or equivalently (see \cite[Lemma~5.33]{BM17}), $f^{n+(l+1)N}(K)$ joins opposite sides of $\CC$. Since $f^{n+(l+1)N}(K) = \bigcup\limits_{m=1}^{i_{l + 1}} f^{n+(l+1)N} (X_m)$ is connected, $\bigl\{  f^{n+(l+1)N} (X_m) \,\big|\, m\in\{1,2,\dots,i_{l+1} \}  \bigr\}    \subseteq \X^N$ (see Proposition~\ref{propCellDecomp}~(i))
, and $i_{l+1} \leq D_N - 1$, we get a contradiction to the definition of $D_N$ (see (\ref{eqDefDn})).

Claim~1 is now proved in the case $i_l \neq D_N - 1$. The argument for the proof of the case $i_l = D_N - 1$ is similar and we omit it here.

\smallskip

By Claim~1 and Property~(4) of $\{U_{\overline{i}} \}_{\overline{i} \in I_l}$, we have  $U_{\overline{j}} \subseteq W^{n+N}(v)$.

Then Properties~(1) and (2) hold for each $\overline{i} \in \{0, 1, \dots, D_N - 1\}^{l+1}$, Property~(3) holds for each $\overline{i} \in \{0, 1, \dots, D_N - 1\}^l$. In order to verify Property~(4) of $\{U_{\overline{i}} \}_{\overline{i} \in I_{l+1}}$, it suffices to observe that by Claim~1 and our construction, for all $\overline{j} \in I_l$ and $i_1, i_2, \dots, i_l, i_{l+1}, i'_{l+1} \in \{0, 1, \dots, D_N - 1\}$ with $1 \leq i_{l+1} < i'_{l+1}$ and $\overline{i} \coloneqq (i_1, i_2, \dots, i_l) < \overline{j}$, we have
\begin{equation*}
U_{\overline{i}} \subseteq \inter U_{\overline{i^1}}   \subseteq  U_{\overline{i^1}}  \subseteq \inter U_{\overline{i^2}}  \subseteq U_{\overline{i^2}}  \subseteq \inter U_{\overline{j}},
\end{equation*}
where $\overline{i^1} \coloneqq (i_1, i_2, \dots, i_l, i_{l+1})$ and $\overline{i^2} \coloneqq (i_1, i_2, \dots, i_l, i'_{l+1})$.

The construction of $\{U_{\overline{i}} \}_{\overline{i} \in I}$ and the verification of Properties~(1) through (4) is now complete.

\smallskip

We can now construct the bump function $\Upsilon_{v,\,n} \: S^2 \rightarrow [0, +\infty)$ and verify that it satisfies Properties~(a) through (c) of the bump functions.

We define 
\begin{equation}   \label{eqPfthmPerturbToStrongNonIntegrable_UpsilonDefOnBoundary}
\Upsilon_{v,\,n}(v) \coloneqq C_{26} \Lambda^{-\alpha n} \varepsilon 
\qquad\text{ and } \qquad 
\Upsilon_{v,\,n}(x) \coloneqq 0 \text{ if } x\in S^2 \setminus U_{(D_N-1)}.
\end{equation}
Property~(a) of the bump functions follows from Property~(4) of $\{U_{\overline{i}} \}_{\overline{i} \in I}$.

We denote, for each $k\in\N$,
\begin{equation*}    
I^*_k \coloneqq \{ (i_1, i_2, \dots, i_k) \in I_k \,|\, i_k \neq 0, i_l \neq D_N -1 \text{ for } 1 \leq l < k \}.
\end{equation*}
Define $I^* \coloneqq \bigcup\limits_{k\in\N}  I^*_k$.

For arbitrary $k\in \N$ and $\overline{i} = (i_1, i_2, \dots, i_k) \in I^*_k$, we define a subset $A_{\overline{i}}$ of $W^{n+N} (v)$ by
\begin{equation} \label{eqPfthmPerturbToStrongNonIntegrable_Annulus}
A_{\overline{i}}  \coloneqq U_{(i_1, i_2, \dots, i_{k-1}, i_k)} \setminus U_{(i_1, i_2, \dots, i_{k-1}, i_k - 1, D_N - 1)}.
\end{equation}
In particular $A_{(i_1)}  = U_{(i_1)} \setminus U_{ (i_1 - 1, D_N - 1) }$ for $i_1 \in \{1,2, \dots, D_N - 1 \}$. We note that by Property~(4) of $\{U_{\overline{i}} \}_{\overline{i} \in I}$,
\begin{equation} \label{eqPfthmPerturbToStrongNonIntegrable_DistinctADisjoint}
A_{\overline{i}}  \cap  A_{\overline{j}}  = \emptyset    \qquad\text{for all } \overline{i}, \overline{j}  \in I^* \text{ with } \overline{i} \neq \overline{j}.
\end{equation}
Thus we define, for each $k\in \N$ and each $\overline{i} = (i_1, i_2, \dots, i_k) \in I^*_k$,
\begin{equation}  \label{eqPfthmPerturbToStrongNonIntegrable_UpsilonDefOnAnnulus}
\Upsilon_{v, \, n}(x) \coloneqq C_{26} \Lambda^{-\alpha n} \varepsilon \bigg( 1 - \sum\limits_{j=1}^k \frac{ i_j } { (D_N - 1)^j } \biggr) 
\end{equation}
for each $x \in A_{\overline{i}}$.

With abuse of notation, for each $\overline{i}\in I^*$, we write $\Upsilon_{v,\,n}( A_{\overline{i}} ) \coloneqq \Upsilon_{v,\,n}(x)$ for any $x\in A_{\overline{i}}$.

So far we have defined $\Upsilon_{v,\,n}$ on
\begin{equation}   \label{eqPfthmPerturbToStrongNonIntegrable_UnionA}
\mathfrak{U} \coloneqq  \{v\} \cup \bigl(S^2 \setminus U_{(D_N - 1)}  \bigr) \cup \bigcup\limits_{\overline{i} \in I^*}  A_{\overline{i}} .
\end{equation}

\smallskip

\emph{Claim~2.} The set $\mathfrak{U}$ contains all vertices, i.e., $\bigcup\limits_{k\in\N_0} \V^k  \subseteq  \mathfrak{U}$.

\smallskip

In order to establish Claim~2, it suffices to show that $x \in \mathfrak{U}$ for each $x\in \V^{n + (m+1) N} \cap U_{(D_N - 1)}  \setminus \{v\}$ and each $m\in\N$. We fix an arbitrary integer $m\in\N$ and an arbitrary vertex $x\in \V^{n + (m+1) N} \cap U_{(D_N - 1)} \setminus \{v\}$. We choose a sequence $\{i_k\}_{k\in\N}$ in $\{0, 1, \dots, D_N - 2\}$ recursively as follows:

Let $i_1$ be the largest integer in $\{0,1, \dots, D_N - 2\}$ with $x\notin U_{(i_1)}$. Assume that we have chosen $\{ i_k \}_{k=1}^l$ in $\{0,1, \dots, D_N - 2\}$ for some $l\in\N$ with the property that $x \notin U_{(i_1,i_2,\dots,i_l)}$ and $x \in U_{(i_1,i_2,\dots,i_{l-1}, 1 + i_l)}$, then by Properties~(3) and (4) of  $\{U_{\overline{i}} \}_{\overline{i} \in I}$, we can choose $i_{l+1}$ to be the largest integer in $\{0,1, \dots, D_N - 1\}$ with $x\notin U_{(i_1, i_2, \dots, i_{l+1})}$. Assume that $i_{l+1} = D_N - 1$. Thus $(i_1, i_2, \dots, i_{l-1}, 1 + i_l) \in I^*$ and $x\in U_{(i_1, i_2, \dots, i_{l-1}, 1 + i_l)}  \setminus U_{(i_1, i_2, \dots, i_l, D_N - 1)} = A_{(i_1, i_2, \dots, i_{l-1}, 1 + i_l)}$. 

So we can assume, without loss of generality, that $i_{k} \neq D_N - 1$ for all $k\in\N$, i.e., $\{ i_k \,|\, k\in \N \} \subseteq \{ 0, 1, \dots, D_N - 2\}$ can be constructed above. Then $x \in U_{(i_1, i_2, \dots, i_{m-1}, 1 + i_m)}$. Since both $U_{(i_1, i_2, \dots, i_{m-1}, 1 + i_m)}$ and $U_{(i_1, i_2, \dots,  i_m)}$ are unions of $(n + (m+1) N)$-tiles (see Property~(1) of $\{U_{\overline{i}} \}_{\overline{i} \in I}$), we can see that $x\notin U_{(i_1, i_2, \dots,  i_m, D_N - 1)}$ since otherwise there would exist $X_1, X_2, \dots, X_{D_N - 1} \in \X^{n + (m + 2) N}$ such that the union $K \coloneqq \bigcup\limits_{k=1}^{D_N - 1} X_k$ is connected and have nontrivial intersections with $U_{(i_1, i_2, \dots,i_m)}$ and $\{x\}$, and consequently $K\cap \partial W^{n + (m + 1) N} (x) \neq \emptyset$. This is impossible since $f^{n + (m + 1) N} (K)$, as a union of $N$-tiles $f^{n + (m + 1) N}(X_l)$ (see Proposition~\ref{propCellDecomp}~(i)), $l\in\{1,2,\dots, D_N - 1\}$, cannot join opposite sides of $\CC$ due to the definition of $D_N$ in (\ref{eqDefDn}). Hence $(i_1, i_2, \dots, i_{m-1}, 1 + i_m) \in I^*$ and $x\in U_{(i_1, i_2, \dots, i_{m-1}, 1 + i_m)}  \setminus U_{(i_1, i_2, \dots, i_m, D_N - 1)} = A_{(i_1, i_2, \dots, i_{m-1}, 1 + i_m)}$. Claim~2 is now established.

\smallskip

\emph{Claim~3.} For the function $\Upsilon_{v,\,n}$ defined on $\mathfrak{U}$, inequality~(\ref{eqPfthmPerturbToStrongNonIntegrable_SingleBumpDiffInEachScale}) holds for each $m\in\N$, each $X\in \X^{n+mN}$, and each pair of points $x,y\in X\cap \mathfrak{U}$.

\smallskip

Fix arbitrary $m\in\N$, $X\in\X^{n+mN}$, and $x,y\in X\cap \mathfrak{U}$. Inequality~(\ref{eqPfthmPerturbToStrongNonIntegrable_SingleBumpDiffInEachScale}) holds for $x,y \in X \cap \mathfrak{U}$ trivially if $m=1$ by (\ref{eqPfthmPerturbToStrongNonIntegrable_UpsilonDefOnBoundary}) and (\ref{eqPfthmPerturbToStrongNonIntegrable_UpsilonDefOnAnnulus}). So without loss of generality, we can assume $m\geq 2$. We choose a sequence $\{i_k\}_{k\in\N}$ in $\{0, 1, \dots, D_N - 1\}$ recursively as follows:

Let $i_1$ be the largest integer in $\{0,1, \dots, D_N - 1\}$ with $X \nsubseteq U_{(i_1)}$. Assume that we have chosen $\{ i_k \}_{k=1}^l$ for some $l\in\N$ with the property that $X \nsubseteq U_{(i_1,i_2,\dots,i_l)}$, then by Properties~(3) and (4) of  $\{U_{\overline{i}} \}_{\overline{i} \in I}$, we can choose $i_{l+1}$ to be the largest integer in $\{0,1, \dots, D_N - 1\}$ with $X \nsubseteq U_{(i_1, i_2, \dots, i_{l+1})}$.

We establish Claim~3 by considering the following two cases:

\smallskip

\emph{Case 1.} $i_k = D_N - 1$ for some integer $k\in [1, m-1]$. Without loss of generality, we assume that $k$ is the smallest such integer. Recall that $m\geq 2$. If $k=1$, then by Property~(1) of $\{U_{\overline{i}} \}_{\overline{i} \in I}$, $X\subseteq \bigl(S^2 \setminus \inter U_{(D_N - 1)} \bigr) \subseteq \bigl(S^2 \setminus  U_{(D_N - 1)} \bigr) \cup A_{(D_N - 1)}$, and consequently $\Upsilon(x) = 0 = \Upsilon(y)$ by (\ref{eqPfthmPerturbToStrongNonIntegrable_UpsilonDefOnBoundary}) and (\ref{eqPfthmPerturbToStrongNonIntegrable_UpsilonDefOnAnnulus}). If $k \geq 2$, then $(i_1, i_2, \dots, i_{k-2}, 1 + i_{k-1} ), (i_1, i_2, \dots, i_{k-1}, D_N - 1 )  \in I^*$, and
\begin{align*}
X & \subseteq U_{(i_1, i_2, \dots, i_{k-2}, 1 + i_{k-1} )} \setminus \inter U_{(i_1, i_2, \dots, i_{k-1}, D_N - 1 )}  \\
   &  \subseteq A_{(i_1, i_2, \dots, i_{k-2}, 1 + i_{k-1} )} \cup A_{(i_1, i_2, \dots, i_{k-1}, D_N - 1 )}
\end{align*}
by our choice of $i_{k-1}$, the fact that both $U_{(i_1, i_2, \dots, i_{k-2}, 1 + i_{k-1} )}$ and $U_{(i_1, i_2, \dots, i_{k-1}, D_N - 1 )}$ are unions of $(n + (k+1) N)$-tiles (by Property~(1) of $\{U_{\overline{i}} \}_{\overline{i} \in I}$), and (\ref{eqPfthmPerturbToStrongNonIntegrable_Annulus}).  Hence by (\ref{eqPfthmPerturbToStrongNonIntegrable_UpsilonDefOnAnnulus}), $\Upsilon_{v, \, n}(x) = C_{26} \Lambda^{-\alpha n} \varepsilon \bigl( 1 - \sum_{j=1}^k \frac{ i_j } { (D_N - 1)^j } \bigr) = \Upsilon_{v, \, n}(y)$.

\smallskip

\emph{Case 2.} $i_k \leq D_N - 2$ for all integer $k\in [1, m-1]$. Then by our choice of $i_{m-1}$ and Properties~(1) and (4) of $\{U_{\overline{i}} \}_{\overline{i} \in I}$,
\begin{equation}   \label{eqPfthmPerturbToStrongNonIntegrable_XinAnnulus}
X \subseteq U_{(i_1, i_2, \dots, i_{m-2}, 1 + i_{m-1})}  \setminus \inter U_{(i_1, i_2, \dots,  i_{m-1})}  \subseteq U_{(i_1, i_2, \dots, i_{m-2}, 1 + i_{m-1})}  \setminus U_{\overline{j}}
\end{equation}
for each $\overline{j} \in I$ with $\overline{j} < (i_1, i_2, \dots,  i_{m-1})$.

Note that by (\ref{eqPfthmPerturbToStrongNonIntegrable_Annulus}) and Property~(4) of $\{U_{\overline{i}} \}_{\overline{i} \in I}$,
\begin{equation}  \label{eqPfthmPerturbToStrongNonIntegrable_AnnulusInU}
A_{\overline{i}} \subseteq U_{\overline{j}}  \text{ for all } \overline{i} \in I^* \text{ and } \overline{j}\in I \text{ with } \overline{i} \leq \overline{j}.
\end{equation}
By (\ref{eqPfthmPerturbToStrongNonIntegrable_UpsilonDefOnBoundary}) and (\ref{eqPfthmPerturbToStrongNonIntegrable_UpsilonDefOnAnnulus}),
\begin{equation}  \label{eqPfthmPerturbToStrongNonIntegrable_UpsilonMonotonicity}
\Upsilon_{v,\,n}(A_{\overline{i}})  \geq \Upsilon_{v,\,n} \bigl( A_{\overline{j}} \bigr)  \text{ for all } \overline{i}, \overline{j} \in I^* \text{ with } \overline{i} \leq \overline{j}.
\end{equation}

Thus by (\ref{eqPfthmPerturbToStrongNonIntegrable_XinAnnulus}), (\ref{eqPfthmPerturbToStrongNonIntegrable_AnnulusInU}), and (\ref{eqPfthmPerturbToStrongNonIntegrable_UpsilonMonotonicity}),
\begin{align*}
              \abs{ \Upsilon_{v,\,n} (x) - \Upsilon_{v,\,n} (y) } 
\leq  &   \inf  \bigl\{  \Upsilon_{v,\,n} ( A_{\overline{i}} )   \,\big|\,   \overline{j} \in I, \, \overline{i} \in I^*, \overline{i} \leq \overline{j} < (i_1, i_2, \dots, i_{m-1}) \bigr\}    \\
         & - \inf  \{  \Upsilon_{v,\,n} ( A_{\overline{i}} )   \,|\,   \overline{i} \in I^*, \overline{i} \leq  (i_1, i_2, \dots, i_{m-2}, 1 + i_{m-1})  \}   \\
\leq  & C_{26} \Lambda^{-\alpha n} \varepsilon (D_N - 1)^{-(m-1)},
\end{align*}
where the last identity follows easily from (\ref{eqPfthmPerturbToStrongNonIntegrable_UpsilonDefOnAnnulus}) and the definition of $I^*$ by separate explicit calculations depending on $i_{m-1} = 0$ or not.

Claim~3 is now established.

\smallskip

\emph{Claim~4.} The function $\Upsilon_{v,\,n}$ is continuous on $\mathfrak{U}$.

\smallskip

Fix arbitrary $x,y\in \mathfrak{U}$ and $m\in\N$ with $x\neq y$ and $y\in U^{n + m N}(x)$ (c.f.\ (\ref{defU^n})). Then there exist $X_1,X_2 \in \X^{n + m N}$ such that $x\in X_1$, $y\in X_2$, and $X_1 \cap X_2 \neq \emptyset$. It follows immediately from Definition~\ref{defcelldecomp}~(iii) that there exists an $( n + m N )$-vertex $z$ in $X_1 \cap X_2$. Then by Claim~2 and Claim~3, 
\begin{align*}
             \abs{ \Upsilon_{v,\,n}(x)  -  \Upsilon_{v,\,n}(y) } 
\leq &  \abs{ \Upsilon_{v,\,n}(x)  -  \Upsilon_{v,\,n}(z) }   +  \abs{ \Upsilon_{v,\,n}(z)  -  \Upsilon_{v,\,n}(y) }   \\
\leq &  2 C_{26} \Lambda^{-\alpha n} \varepsilon (D_N - 1)^{-(m-1)}. 
\end{align*}
Hence Claim~4 follows from Lemma~\ref{lmCellBoundsBM}~(iv) and the fact that $D_N - 1 > 1$.

\smallskip

Since we have defined $\Upsilon_{v,n}$ continuously on a dense subset $\mathfrak{U}$ of $S^2$ by Claim~2 and Claim~4, we can now extend $\Upsilon_{v,n}$ continuously to $S^2$. Property~(b) of the bump functions follows immediately from (\ref{eqPfthmPerturbToStrongNonIntegrable_UpsilonDefOnBoundary}) and (\ref{eqPfthmPerturbToStrongNonIntegrable_UpsilonDefOnAnnulus}). Property~(c) of the bump functions follows from Claim~3.

\smallskip

Recall $u^1_\b, u^2_\b, u^1_\w, u^2_\w \in \V^N$ defined above.

For each $n\in\N_0$, each $n$-tile $X\in\X^n$, and each $i\in\{1,2\}$, we define a point
\begin{equation}   \label{eqPfthmPerturbToStrongNonIntegrable_vi}
v_i(X) \coloneqq  \begin{cases} (f^n|_X)^{-1} \bigl( u^i_\b \bigr) & \text{if } X\in\X^n_\b, \\ (f^n|_X)^{-1}  \bigl( u^i_\w  \bigr)   & \text{if } X\in\X^n_\w.  \end{cases}
\end{equation}

Fix an arbitrary real-valued H\"{o}lder continuous function $\varphi \in \Holder{\alpha}(S^2,d)$ with an exponent $\alpha$. 

We are going to construct $\phi \in \Holder{\alpha}(S^2,d)$ for the given $\varphi$ by defining their difference $\Upsilon \in \Holder{\alpha}(S^2,d)$ supported on the (disjoint) backward orbits of $Y^{M_0}_\b \cup Y^{M_0}_\w$ along $\{ \xi_{\minus i} \}_{i\in\N_0}$ as the sum of a collection of non-negative bump functions constructed above.

We construct $\varphi_m \in  \Holder{\alpha}(S^2,d)$ recursively on $m\in\N_0$.

Set $\varphi_0 \coloneqq \varphi$.

Assume that $\varphi_i \in \Holder{\alpha}(S^2,d)$ has been constructed for some $i\in\N_0$, we define a number $\delta_X\in\{0,1\}$, each $(M_0 + (i+1) N)$-tile $X\in\X^{M_0 + (i+1) N}$ with $X\subseteq Y^{M_0}_\b \cup Y^{M_0}_\w$, by
\begin{equation} \label{eqPfthmPerturbToStrongNonIntegrable_deltaX}
\delta_X \coloneqq  \begin{cases} 1 & \text{if } \Absbig{ (\varphi_i)^{f,\,\CC}_{\xi,\,\xi'}   ( v_1(X), v_2(X) ) } < 2 \varepsilon d( v_1(X), v_2(X) )^\alpha,  \\ 0    &  \text{otherwise}.  \end{cases}
\end{equation} 
We define
\begin{equation}  \label{eqPfthmPerturbToStrongNonIntegrable_varphi_i}
\varphi_{i+1} \coloneqq \varphi_i  +  \sum\limits_{j\in\N}        \sum\limits_{\substack{X \in\X^{M_0 + (i+1) N} \\X \subseteq Y^{M_0}_\b \cup Y^{M_0}_\w}}  
                                                                      \delta_X \Upsilon_{v_1 ( \tau_j(X) ),  \, M_0 + (i+1) N + j}  ,
\end{equation}
and finally define the non-negative bump function $\Upsilon\: S^2 \rightarrow [0,1)$ by
\begin{equation}  \label{eqPfthmPerturbToStrongNonIntegrable_Upsilon}
\Upsilon   \coloneqq   \sum\limits_{j\in\N}    \sum\limits_{m\in\N}     \sum\limits_{\substack{X \in\X^{M_0 + m N} \\X \subseteq Y^{M_0}_\b \cup Y^{M_0}_\w}}  
                                                                      \delta_X \Upsilon_{v_1 ( \tau_j(X) ),  \, M_0 + m N + j}  .
\end{equation}
Here the function $\tau_j$ is defined in (\ref{eqPfthmPerturbToStrongNonIntegrable_tau}). It follows immediately from Property~(b) of the bump functions that the series in (\ref{eqPfthmPerturbToStrongNonIntegrable_varphi_i}) and (\ref{eqPfthmPerturbToStrongNonIntegrable_Upsilon}) converge uniformly and absolutely.

We set $\phi \coloneqq \varphi + \Upsilon$.

For each $\c\in\{\b,\w\}$, each integer $M\geq M_0$, and each $M$-tile $X\in \X^M$ with $X\subseteq Y^{M_0}_\c$, we choose an arbitrary $\bigl(M_0 + \bigl\lceil \frac{ M - M_0 }{N} \bigr\rceil N \bigr)$-tile $X'$ with $X' \subseteq X$ and define $x_i(X) \coloneqq v_i(X')$ for each $i\in\{1,2\}$.

Now we discuss some properties of the supports of the terms in the series defining $\Upsilon$ in (\ref{eqPfthmPerturbToStrongNonIntegrable_Upsilon}).  See Figure~\ref{figPerturb}.

Fix arbitrary integers $m,j\in\N$, by Property~(a) of the bump functions, (\ref{eqPfthmPerturbToStrongNonIntegrable_vi}), and properties of $u^1_\b, u^1_\w \in \V^N$, we have
\begin{align}   
                        \supp \Upsilon_{ v_1 ( \tau_j(X) ), \, M_0 + m N + j }   
\subseteq &  \overline{W}^{M_0 + (m+1) N + j} \bigl( v_1 \bigl( \tau_j(X) \bigr) \bigr)    \label{eqPfthmPerturbToStrongNonIntegrable_SuppInTile} \\
\subseteq & \inte \bigl( \tau_j(X) \bigr)
\subseteq \tau_j \bigl( Y^{M_0}_\b \cup Y^{M_0}_\w  \bigr),   \notag
\end{align}
for each $(M_0 + m N)$-tile $X\in \X^{M_0 + m N}$ with $X\subseteq Y^{M_0}_\b \cup Y^{M_0}_\w$. Consequently, by (\ref{eqPfthmPerturbToStrongNonIntegrable_SuppInTile}) and the fact that $\tau_{j_1} \bigl( Y^{M_0}_\b \cup Y^{M_0}_\w \bigr)$ and $\tau_{j_2} \bigl( Y^{M_0}_\b \cup Y^{M_0}_\w \bigr)$ are disjoint for distinct $j_1,j_2\in\N$ (c.f.\ Figure~\ref{figPerturb}), we have
\begin{equation}  \label{eqPfthmPerturbToStrongNonIntegrable_SuppDisjoint}
 \supp \Upsilon_{ v_1 ( \tau_{j_1} (X_1) ), \, M_0 + m N + j_1 }  \cap  \supp \Upsilon_{ v_1 ( \tau_{j_2} (X_2) ), \, M_0 + m N + j_2 }  = \emptyset
\end{equation}
for each pair of integers $j_1, j_2 \in \N$ and each pair of  $(M_0 + m N)$-tiles $X_1, X_2 \in\X^{M_0 + m N}$ with $X_1 \cup X_2 \subseteq Y^{M_0}_{\b} \cup Y^{M_0}_{\w}$ and $(j_1, X_1) \neq (j_2, X_2)$.

We are now ready to verify Property~(iii) in Theorem~\ref{thmPerturbToStrongNonIntegrable}.

\smallskip

\emph{Property~(iii).} By (\ref{eqPfthmPerturbToStrongNonIntegrable_SuppDisjoint}), Property~(b) of the bump functions, and (\ref{eqPfthmPerturbToStrongNonIntegrable_M0}),
\begin{align*}
                \norm{\Upsilon}_{\CCC^0(S^2)}  
\leq &    \sum\limits_{ m\in\N}  \sup \bigl\{  \norm{   \Upsilon_{ v_1 ( \tau_j(X) ), \, M_0 + m N + j }   }_{\CCC^0(S^2)}  \,\big|\, j\in\N, \, X\in \X^{M_0 + m N}, \, X\subseteq   Y^{M_0}_\b \cup Y^{M_0}_\w  \bigr\}  \\
\leq &    \sum\limits_{ m\in\N}  C_{26} \Lambda ^{ - \alpha (M_0 + m N)}  \varepsilon
\leq        \frac{ C_{26} } { 1-\Lambda^{-\alpha N} }  \Lambda^{-\alpha M_0} \varepsilon
\leq        \frac{\varepsilon}{2}.
\end{align*}

Fix $x,y\in S^2$ with $x\neq y$. 

Note that $\supp \Upsilon \subseteq \bigcup\limits_{j\in\N} \tau_j \bigl( Y^{M_0}_\b \cup Y^{M_0}_\w \bigr)$ and that this union is a disjoint union. We bound $\frac{ \abs{\Upsilon(x) - \Upsilon(y)} } { d(x,y)^\alpha }$ by considering the following cases:

\smallskip

\emph{Case 1.} $x \notin \supp \Upsilon$ and $y\notin \supp \Upsilon$. Then $\Upsilon(x) - \Upsilon(y)=0$.

\smallskip

\emph{Case 2.} $\{x, y\} \cap \tau_j \bigl( Y^{M_0}_\b \cup Y^{M_0}_\w \bigr) \neq \emptyset$ and $\{x, y\} \nsubseteq \tau_j  ( f(\xi_0) \setminus \xi_0  )$ for some $j\in\N$. Without loss of generality, we can assume that $j$ is the smallest such integer. Then by (\ref{eqPfthmPerturbToStrongNonIntegrable_Y_location}), Lemma~\ref{lmCellBoundsBM}~(i), and Property~(b) of the bump functions,
\begin{align*}
&                      \frac{ 1 } { d(x,y)^\alpha }  \abs{\Upsilon(x) - \Upsilon(y)}\\
&\qquad  \leq \frac{  \sum\limits_{ m\in\N}  \sup \bigl\{  \norm{   \Upsilon_{ v_1 ( \tau_j(X) ), \, M_0 + m N + j }   }_{\CCC^0(S^2)}  \,\big|\, X\in \X^{M_0 + m N}, \, X\subseteq   Y^{M_0}_\b \cup Y^{M_0}_\w \bigr\}  }  
                                                            { C^{-\alpha} \Lambda^{-\alpha (M_0 + j) } }  \\
&\qquad  \leq  C^{\alpha} \Lambda^{\alpha (M_0 + j) }      \sum\limits_{ m\in\N}  C_{26} \Lambda ^{ - \alpha (M_0 + m N + j)}  \varepsilon      \\
&\qquad  \leq  C \frac{ C_{26} \Lambda^{ - \alpha N }    } { 1 -  \Lambda^{ - \alpha N }     }     \varepsilon 
\leq     C \frac{ C_{26} (3C C_{26})^{-1}   } { 1 -  3^{-1}     }     \varepsilon
  =       \frac{ \varepsilon } {2}.
\end{align*} 
The last inequality follows from our choice of $N$ at the beginning of this proof.

\smallskip

\emph{Case 3.} $\{x, y\} \cap \tau_j \bigl( Y^{M_0}_\b \cup Y^{M_0}_\w \bigr) \neq \emptyset$ and $\{x, y\} \subseteq \tau_j  ( f(\xi_0) \setminus \xi_0  )$ for some $j\in\N$.  Note that such $j$ is unique. Then by (\ref{eqPfthmPerturbToStrongNonIntegrable_Upsilon}) and our constructions of $Y^{M_0}_\b, Y^{M_0}_\w \in \X^{M_0}$ and $\xi \in \Sigma_{f,\,\CC}^-$, we get that for each $z\in\{x,y\}$,
\begin{equation}  \label{eqPfthmPerturbToStrongNonIntegrable_PropertyIIICase3}
\Upsilon (z)  =    \sum\limits_{m\in\N}     \sum\limits_{\substack{X \in\X^{M_0 + m N} \\X \subseteq Y^{M_0}_\b \cup Y^{M_0}_\w}}  
                                                                      \delta_X \Upsilon_{v_1 ( \tau_j(X) ),  \, M_0 + m N + j}  (z).
\end{equation}

Since $f$ is an expanding Thurston map, we can define an integer
\begin{align*}
m_1 \coloneqq \max \bigl\{  k \in \Z   \,\big|\, \text{there exist } &  X_1, X_2 \in \X^{M_0 + k N + j}  \text{ such that } \\
                                                                                        &x\in X_1, \, y\in X_2, \text{ and } X_1 \cap X_2 \neq \emptyset   \bigr\}.
\end{align*}

If $m_1 \leq 0$, then by (\ref{eqPfthmPerturbToStrongNonIntegrable_PropertyIIICase3}), (\ref{eqPfthmPerturbToStrongNonIntegrable_SuppInTile}), Property~(b) of the bump functions, Lemma~\ref{lmCellBoundsBM}~(i), and (\ref{eqDefC27}), we have
\begin{align*}
&                       \frac{ 1 } { d(x,y)^\alpha }  \abs{\Upsilon(x) - \Upsilon(y)}\\
&\qquad  \leq \sum\limits_{m\in\N}  \frac{ \sup \bigl\{  \norm{   \Upsilon_{ v_1 ( \tau_j(X) ), \, M_0 + m N + j }   }_{\CCC^0(S^2)}  \,\big|\, X\in \X^{M_0 + m N}, \, X\subseteq   Y^{M_0}_\b \cup Y^{M_0}_\w \bigr\}    } 
                                                                          { d(x,y)^\alpha }                    \\
&\qquad  \leq \bigl( C^{-1}  \Lambda^{ -  (M_0 + N + j )}  \bigr)^{-\alpha} \sum\limits_{m\in\N}  C_{26} \Lambda^{ - \alpha (M_0 + m N + j) } \varepsilon         \\
&\qquad  \leq   C_{26}   C \bigl(  1 - \Lambda^{ - \alpha N}   \bigr)^{-1}  \varepsilon     
                  \leq    (C_{27} - 1 ) \varepsilon.                                                                 
\end{align*}

If $m_1 \geq 1$, then $y\in U^{M_0 + m_1 N + j} (x)$ and $y \notin U^{M_0 + (m_1 + 1) N + j} (x)$ (c.f.\ (\ref{defU^n})). Choose $X_1,X_2 \in \X^{M_0 + m_1 N + j}$ such that $x\in X_1$, $y\in X_2$, and $X_1 \cap X_2 = \emptyset$. For each $i\in\{1,2\}$ and each $m\in\N$ with $1\leq m \leq m_1$, we denote the unique $(M_0 + m N + j)$-tile containing $X_i$ by $Y^i_m$. Then by (\ref{eqPfthmPerturbToStrongNonIntegrable_PropertyIIICase3}), (\ref{eqPfthmPerturbToStrongNonIntegrable_SuppInTile}), Properties~(b) and (c) of the bump functions, Lemma~\ref{lmCellBoundsBM}~(i), (\ref{eqPfthmPerturbToStrongNonIntegrable_rho}), and (\ref{eqDefC27}),

\begin{align*}
&                       \frac{ 1 } { d(x,y)^\alpha }  \abs{\Upsilon(x) - \Upsilon(y)}\\
&\qquad  \leq  \sum\limits_{m\in\N}    \sum\limits_{\substack{X \in\X^{M_0 + m N} \\X \subseteq Y^{M_0}_\b \cup Y^{M_0}_\w}}  
                           \frac{ \delta_X  \abs{ \Upsilon_{v_1 ( \tau_j(X) ),  \, M_0 + m N + j} (x) -  \Upsilon_{v_1 ( \tau_j(X) ),  \, M_0 + m N + j} (y) } } { d(x,y)^\alpha }                                                  \\
&\qquad  \leq    \sum\limits_{m=m_1}^{+\infty} 
                                 \frac{ \sup \bigl\{  \norm{   \Upsilon_{ v_1 ( \tau_j(X) ), \, M_0 + m N + j }   }_{\CCC^0(S^2)}  \,\big|\, X\in \X^{M_0 + m N}, \, X\subseteq   Y^{M_0}_\b \cup Y^{M_0}_\w \bigr\} } { d(x,y)^\alpha }   \\
&\qquad\qquad + \sum\limits_{m=1}^{m_1 - 1}      \sum\limits_{i\in\{1,2\}}
                                \frac{   \abs{ \Upsilon_{v_1 ( Y^i_m ),  \, M_0 + m N + j} (x) -  \Upsilon_{v_1 ( Y^i_m ),  \, M_0 + m N + j} (y) }  } { d(x,y)^\alpha }                                         \\
&\qquad  \leq \frac{  \sum\limits_{m=m_1}^{+\infty}  C_{26} \Lambda^{ - \alpha (M_0 + m N + j) } \varepsilon 
                                   +  \sum\limits_{m=1}^{m_1 - 1}   4 C_{26} \Lambda^{ - \alpha (M_0 + m N + j) } \varepsilon (D_N - 1)^{ - (m_1 - m - 1)  }  } 
                                   {  C^{-\alpha}  \Lambda^{ - \alpha (M_0 + (m_1 + 1) N + j )}  }   \\
&\qquad  \leq   C_{26}   C \Bigl( \Lambda^{ \alpha N} \bigl( 1 - \Lambda^{ - \alpha N} \bigr)^{-1} + 4   (1-\rho)^{-1} \Bigr)  \varepsilon     
                  =    (C_{27} - 1 ) \varepsilon.                                                                        
\end{align*}

\smallskip

To summarize, we have shown that $\Hnorm{\alpha}{\phi - \varphi}{(S^2,d)} \leq \bigl(\frac{1}{2} +\frac{1}{2} + C_{27} - 1 \bigr) \varepsilon  = C_{27} \varepsilon$, establishing Property~(iii) in Theorem~\ref{thmPerturbToStrongNonIntegrable}.

\smallskip

Finally, we are going to verify Properties~(i) and (ii) in Theorem~\ref{thmPerturbToStrongNonIntegrable}.

Fix arbitrary $\c\in\{\b,\w\}$, $M\in\N$ with $M\geq M_0$, and $X_0 \in \X^M$ with $X_0 \subseteq Y^{M_0}_\c$. Denote $m_0 \coloneqq \bigl\lceil \frac{ M - M_0 }{N} \bigr\rceil$, $M' \coloneqq M_0 + m_0 N \in [M, M+N)$, and fix $X' \in \X^{M'}$ with $x_1(X_0) = v_1(X') \in \V^{M' + N}$ and $x_2(X_0) = v_2(X') \in \V^{M' + N}$.

\smallskip

\emph{Property~(i).} Fix arbitrary $i\in\{1,2\}$. Since $\overline{W}^{M' + N} (x_i(X_0)) \subseteq \inte(X') \subseteq \inte(X_0)$ and $\overline{W}^{M' + N} (x_1(X_0)) \cap \overline{W}^{M' + N} (x_2(X_0)) = \emptyset$ (which follows from (\ref{eqPfthmPerturbToStrongNonIntegrable_FlowerLocation}) and (\ref{eqPfthmPerturbToStrongNonIntegrable_FlowersDisjoint})), we get from Lemma~\ref{lmCellBoundsBM}~(i) and (ii) that
\begin{equation*}
d \bigl( x_i(X_0), S^2\setminus X_0 \bigr) 
\geq C^{-1} \Lambda^{ - (M'+N) } 
\geq C^{-1} \Lambda^{ - M - 2N } 
\geq C^{-2} \Lambda^{ - 2N }  \diam_d( X_0 ),
\end{equation*}
and similarly,
\begin{equation*}
d  ( x_1(X_0),  x_2(X_0)  ) 
\geq C^{-1} \Lambda^{ - (M'+N) } 
\geq C^{-1} \Lambda^{ - M - 2N } 
\geq C^{-2} \Lambda^{ - 2N }  \diam_d( X_0 ).
\end{equation*}

Property~(i) in Theorem~\ref{thmPerturbToStrongNonIntegrable} now follows from (\ref{eqPfthmPerturbToStrongNonIntegrable_varepsilon}).

\smallskip

\emph{Property~(ii).} We first show
\begin{equation}   \label{eqPfthmPerturbToStrongNonIntegrable_PrePropertyII}
 \AbsBig{ \phi^{f,\,\CC}_{\xi,\,\xi'} ( x_1(X_0), x_2(X_0) ) } \geq   2 \varepsilon d( x_1(X_0), x_2(X_0) )^\alpha.
\end{equation}

Indeed, observe that by our construction and (\ref{eqPfthmPerturbToStrongNonIntegrable_SuppInTile}), for each integer $m>m_0$, the sets
\begin{equation*}
                     \bigcup\limits_{j\in\N}     \bigcup\limits_{\substack{X \in\X^{M_0 + m N} \\X \subseteq Y^{M_0}_\b \cup Y^{M_0}_\w}}  \supp \Upsilon_{ v_1 ( \tau_j(X) ), \, M_0 + m N + j }
\subseteq   \bigcup\limits_{j\in\N}     \bigcup\limits_{\substack{X \in\X^{M_0 + m N} \\X \subseteq Y^{M_0}_\b \cup Y^{M_0}_\w}}    \inte( \tau_j(X) )
\end{equation*}
are disjoint from the backward orbits of $v_1(X') \in \V^{M_0 + (m_0 +1) N}$ and $v_2(X') \in \V^{M_0 + (m_0 +1) N}$ under $\xi$ and $\xi'$. Thus by (\ref{eqPfthmPerturbToStrongNonIntegrable_varphi_i}),
\begin{align*}
&                   \AbsBig{ \phi^{f,\,\CC}_{\xi,\,\xi'} ( x_1(X_0), x_2(X_0) ) }  \\
&\qquad =   \AbsBig{ \phi^{f,\,\CC}_{\xi,\,\xi'} ( v_1(X'),    v_2(X') ) }    \\
&\qquad =   \Absbigg{ \Bigl( \varphi_{m_0} +  \sum\limits_{j\in\N}    \sum\limits_{m=m_0 + 1}^{+\infty}     \sum\limits_{\substack{X \in\X^{M_0 + m N} \\X \subseteq Y^{M_0}_\b \cup Y^{M_0}_\w}}  
                                                                      \delta_X \Upsilon_{v_1 ( \tau_j(X) ),  \, M_0 + m N + j}  \Bigr)^{f,\,\CC}_{\xi,\,\xi'} ( v_1(X'), v_2(X') ) }  \\
&\qquad =   \AbsBig{ ( \varphi_{m_0} )^{f,\,\CC}_{\xi,\,\xi'} ( v_1(X'),    v_2(X') ) }                                                                  .
\end{align*}

We observe that for each $j\in\N$, the sets
\begin{equation*}
                   \bigcup\limits_{j\in\N}    \bigcup\limits_{\substack{X \in\X^{M_0 + m_0 N}  \setminus \{ X' \} \\ X \subseteq  Y^{M_0}_\b \cup Y^{M_0}_\w   }}  \supp \Upsilon_{ v_1 ( \tau_j(X) ), \, M_0 + m_0 N + j }
\subseteq  \bigcup\limits_{j\in\N}    \bigcup\limits_{\substack{X \in\X^{M_0 + m_0 N}  \setminus \{ X' \} \\ X \subseteq  Y^{M_0}_\b \cup Y^{M_0}_\w   }}   \inte( \tau_j(X) )
\end{equation*}
are disjoint from the backward orbits of $v_1(X')$ and $v_2(X')$ under $\xi$ and $\xi'$ (by (\ref{eqPfthmPerturbToStrongNonIntegrable_SuppInTile}) and our choices of $\xi$ and $\xi'$ from Lemma~\ref{lmDisjointBackwardOrbits}). See Figure~\ref{figPerturb}. Thus for each $X\in \X^{M_0 + m_0 N}$ with $X \subseteq  Y^{M_0}_\b \cup Y^{M_0}_\w$ and $X \neq X'$,  we have
\begin{equation}    \label{eqPfthmPerturbToStrongNonIntegrable_UpsilonTempDistr0}
( \Upsilon_{v_1 ( \tau_j(X) ),  \, M_0 + m_0 N + j}  )^{f,\,\CC}_{\xi,\,\xi'} ( v_1(X'), v_2(X') )  = 0.
\end{equation}

By our construction in (\ref{eqPfthmPerturbToStrongNonIntegrable_deltaX}) and (\ref{eqPfthmPerturbToStrongNonIntegrable_varphi_i}), if
\begin{equation*}
\AbsBig{ ( \varphi_{m_0 - 1} )^{f,\,\CC}_{\xi,\,\xi'} ( v_1(X'),    v_2(X') ) }  \geq 2 \varepsilon d( v_1(X'), v_2(X') )^\alpha,
\end{equation*}
then $\delta_{X'} = 0$, and consequently, by (\ref{eqPfthmPerturbToStrongNonIntegrable_varphi_i}) and (\ref{eqPfthmPerturbToStrongNonIntegrable_UpsilonTempDistr0}), we have
\begin{equation*}
         \AbsBig{ ( \varphi_{m_0} )^{f,\,\CC}_{\xi,\,\xi'} ( v_1(X'),    v_2(X') ) }  
=       \AbsBig{ ( \varphi_{m_0 - 1} )^{f,\,\CC}_{\xi,\,\xi'} ( v_1(X'),    v_2(X') ) }  
\geq  2 \varepsilon d( v_1(X'), v_2(X') )^\alpha.
\end{equation*}
On the other hand, if
\begin{equation*}
\AbsBig{ ( \varphi_{m_0 - 1} )^{f,\,\CC}_{\xi,\,\xi'} ( v_1(X'),    v_2(X') ) }  < 2 \varepsilon d( v_1(X'), v_2(X') )^\alpha,
\end{equation*}
then $\delta_{X'} = 1$ (see (\ref{eqPfthmPerturbToStrongNonIntegrable_deltaX})), and consequently, by (\ref{eqPfthmPerturbToStrongNonIntegrable_varphi_i}), (\ref{eqPfthmPerturbToStrongNonIntegrable_UpsilonTempDistr0}), Property~(a) of the bump functions, Lemma~\ref{lmCellBoundsBM}~(ii), and (\ref{eqDefC26}), we get
\begin{align*}
&                            \AbsBig{ ( \varphi_{m_0} )^{f,\,\CC}_{\xi,\,\xi'} ( v_1(X'),    v_2(X') ) }    \\
&\qquad \geq      \AbsBig{ \sum\limits_{j\in\N}   ( \Upsilon_{v_1(\tau_j(X')), \, M' + j}   )^{f,\,\CC}_{\xi,\,\xi'}  ( v_1(X'), v_2(X') ) } 
                            - \AbsBig{ ( \varphi_{m_0 - 1} )^{f,\,\CC}_{\xi,\,\xi'} ( v_1(X'),    v_2(X') ) }    \\
&\qquad \geq     \AbsBig{ \sum\limits_{j\in\N}    \Upsilon_{v_1(\tau_j(X')), \, M' + j}  ( v_1 ( \tau_j (X') ) )  } - 2 \varepsilon d( v_1(X'), v_2(X') )^\alpha    \\
&\qquad  =           \sum\limits_{j\in\N} C_{26} \Lambda^{ - \alpha (M'+j) } \varepsilon   -  2 \varepsilon d( v_1(X'), v_2(X') )^\alpha   \\
&\qquad \geq     \frac{ \Lambda^{-\alpha} \varepsilon } { 1- \Lambda^{-\alpha} } C_{26} C^{-\alpha}   ( \diam_d(X')  )^\alpha  -  2 \varepsilon d( v_1(X'), v_2(X') )^\alpha  \\
&\qquad \geq     2 \varepsilon d( v_1(X'), v_2(X') )^\alpha.
\end{align*}

Hence we have proved (\ref{eqPfthmPerturbToStrongNonIntegrable_PrePropertyII}). Now we are going to establish (\ref{eqSNIBoundsPerturb}).

Fix arbitrary $N'\geq N_0$. Define $X^{N'+M_0}_{\c,1} \coloneqq \tau_{N'} \bigl( Y^{M_0}_\c \bigr)$ and $X^{N'+M_0}_{\c,2} \coloneqq \tau'_{N'} \bigl( Y^{M_0}_\c \bigr)$ (c.f.\ (\ref{eqPfthmPerturbToStrongNonIntegrable_tau})). Note that $\varsigma_1 = \tau_{N'}|_{ Y^{M_0}_\c }$ and $\varsigma_2 = \tau'_{N'}|_{ Y^{M_0}_\c }$.

Then by Lemma~\ref{lmSnPhiBound}, Lemma~\ref{lmCellBoundsBM}~(i) and (ii), Proposition~\ref{propCellDecomp}~(i), and Properties~(i) and (iii) in Theorem~\ref{thmPerturbToStrongNonIntegrable},
\begin{align*}
&                       \frac{ \abs{  S_{N' }\phi ( \varsigma_1 (x_1(X_0)) ) -  S_{N' }\phi ( \varsigma_2 (x_1(X_0)) )  -S_{N' }\phi ( \varsigma_1 (x_2(X_0)) ) +  S_{N' }\phi ( \varsigma_2 (x_2(X_0)) )   }  } 
                                  {  d(x_1(X_0),x_2(X_0))^\alpha  }  \\
&\qquad \geq  \frac{ \Absbig{ \phi^{f,\,\CC}_{\xi,\,\xi'} (x_1(X_0)) , x_2(X_0))  }   } { d(x_1(X_0),x_2(X_0))^\alpha }
                                 -  \limsup\limits_{n\to+\infty}       \frac{ \abs{  S_{n - N' }\phi ( \tau_n (v_1(X')) ) -  S_{n - N' }\phi ( \tau_n (v_2(X')) )  }  }   {  \varepsilon^\alpha ( \diam_d (X_0))^\alpha  }  \\
&\qquad\qquad    -  \limsup\limits_{n\to+\infty}       \frac{ \abs{  S_{n - N' }\phi ( \tau'_n (v_1(X')) ) -  S_{n - N' }\phi ( \tau'_n (v_2(X')) )  }  }   {  \varepsilon^\alpha ( \diam_d (X_0))^\alpha  }              \\
&\qquad \geq 2 \varepsilon -  \frac{\Hseminorm{\alpha,\, (S^2,d)}{\phi} C_0}{1-\Lambda^{-\alpha}} \cdot
                                                     \frac{    d ( \tau_{N'} (v_1(X')) , \tau_{N'} (v_2(X')))^\alpha  +  d ( \tau'_{N'} (v_1(X')) , \tau'_{N'} (v_2(X')))^\alpha  }   {  \varepsilon^\alpha ( \diam_d (X_0))^\alpha  }   \\
&\qquad \geq 2 \varepsilon -  \frac{\Hseminorm{\alpha,\, (S^2,d)}{\phi} C_0}{1-\Lambda^{-\alpha}} \cdot
                                                     \frac{   ( \diam_d ( \tau_{N'} (X') ))^\alpha  +  ( \diam_d ( \tau'_{N'} (X') ))^\alpha  }   {  \varepsilon^\alpha ( \diam_d (X_0))^\alpha  }          \\
&\qquad \geq 2 \varepsilon -  \frac{ \bigl(\Hnorm{\alpha}{\varphi}{(S^2,d)} + \varepsilon C_{27} \bigr) C_0}{1-\Lambda^{-\alpha}} \cdot
                                                     \frac{  2 C^\alpha \Lambda^{ - \alpha (M_0 + m_0  N + N' )}  }   {  \varepsilon^\alpha C^{-\alpha} \Lambda^{ - \alpha (M_0 + m_0 N )} }          \\
&\qquad \geq 2 \varepsilon -  \frac{ 2 C^2 \varepsilon^{-\alpha} \bigl(\Hnorm{\alpha}{\varphi}{(S^2,d)} + \varepsilon C_{27} \bigr) C_0}{1-\Lambda^{-\alpha}}   \Lambda^{ - \alpha N_0}
        \geq \varepsilon.                                                                                                 
\end{align*}
The last inequality follows from (\ref{eqPfthmPerturbToStrongNonIntegrable_N0}). Property~(ii) in Theorem~\ref{thmPerturbToStrongNonIntegrable} is now established.

\smallskip

The proof of Theorem~\ref{thmPerturbToStrongNonIntegrable} is now complete.
\end{proof}

\begin{proof}[Proof of Theorem~\ref{thmSNIGeneric}]
Note that for each $n\in\N$, the map $F\coloneqq f^n$ is an expanding Thurston map with $\post F = \post f$ and with the combinatorial expansion factor $\Lambda_0(F) = (\Lambda_0(f))^n$ (by (\ref{eqDefCombExpansionFactor}) and Lemma~\ref{lmCellBoundsBM}~(vii)), and  $d$ is a visual metric for $F$ with expansion factor $\Lambda^n$ (by Lemma~\ref{lmCellBoundsBM}). Thus by \cite[Theorem~15.1]{BM17} (see also Lemma~\ref{lmCexistsL}) and Lemma~\ref{lmSNIwoC}, it suffices to prove Theorem~\ref{thmSNIGeneric} under the additional assumption of the existence of a Jordan curve $\CC \subseteq S^2$ satisfying $\post \subseteq \CC$ and $f(\CC) \subseteq \CC$. We fix such a curve $\CC$ and consider the cell decomposition induced by the pair $(f,\CC)$ in this proof.

We first show that $\mathcal{S}^\alpha$ is an open subset of $\Holder{\alpha}(S^2,d)$, for each $\alpha\in(0,1]$.

Fix $\alpha\in (0,1]$ and $\phi\in \mathcal{S}^\alpha$ with associated constants $N_0, M_0 \in \N$, $\varepsilon \in (0,1)$, and $M_0$-tiles $Y^{M_0}_\b \in \X^{M_0}_\b$ and $Y^{M_0}_\w \in \X^{M_0}_\w$ as in Definition~\ref{defStrongNonIntegrability}. For each $\c\in\{\b,\w\}$, each integer $M\geq M_0$, and each $X\in \X^M$ with $X\subseteq Y^{M_0}_\c$, we choose two points $x_1(X), x_2(X) \in X$ associated to $\phi$ as in Definition~\ref{defStrongNonIntegrability}.

Recall $C_0 > 1$ is a constant depending only on $f$, $\CC$, and $d$ from Lemma~\ref{lmMetricDistortion}.

\smallskip

\emph{Claim.} Fix an arbitrary $\psi \in \Holder{\alpha}(S^2,d)$ with
\begin{equation}   \label{eqPfthmSNIGeneric}
\Hnorm{\alpha}{\phi - \psi}{(S^2,d)} \leq \frac{ 1- \Lambda^{-\alpha}}{4C_0} \varepsilon.
\end{equation}
Then $\psi$ satisfies Properties~(i) and (ii) in Definition~\ref{defStrongNonIntegrability} with the constant $\varepsilon$ for $\phi$ replaced by $\frac{\varepsilon}{2}$ for $\psi$, and with the same constants $N_0, M_0\in\N$, $M_0$-tiles $Y^{M_0}_\b$, $Y^{M_0}_\w$, and points $x_1(X)$, $x_2(X)$ as those for $\phi$.

\smallskip

Property~(i) in Definition~\ref{defStrongNonIntegrability} for $\psi$ follows trivially from that for $\phi$. To establish Property~(ii) for $\psi$, we fix arbitrary integer $N\geq N_0$, and $(N+M_0)$-tiles $X^{N+M_0}_{\c,1}, X^{N+M_0}_{\c,2} \in \X^{N+M_0}$ that satisfies (\ref{eqSNIBoundsDefn}) and $Y^{M_0}_\c = f^N\bigl(  X^{N+M_0}_{\c,1}  \bigr) =   f^N\bigl(  X^{N+M_0}_{\c,2}  \bigr)$. Then by (\ref{eqSNIBoundsDefn}), Lemma~\ref{lmSnPhiBound}, and (\ref{eqPfthmSNIGeneric}),
\begin{align*}
&                              \frac{ \abs{  S_{N }\psi ( \varsigma_1 (x_1(X)) ) -  S_{N }\psi ( \varsigma_2 (x_1(X)) )  -S_{N }\psi ( \varsigma_1 (x_2(X)) ) +  S_{N }\psi ( \varsigma_2 (x_2(X)) )   }  } 
                                         {  d(x_1(X),x_2(X))^\alpha  }  \\
&\qquad       \geq  \frac{ \abs{  S_{N }\phi ( \varsigma_1 (x_1(X)) ) -  S_{N }\phi ( \varsigma_2 (x_1(X)) )  -S_{N }\phi ( \varsigma_1 (x_2(X)) ) +  S_{N }\phi ( \varsigma_2 (x_2(X)) )   }  } 
                                         {  d(x_1(X),x_2(X))^\alpha  }  \\
&\qquad\qquad   -  \sum\limits_{i\in\{1,2\}}    \frac{  \abs{  S_{N } (\psi - \phi) ( \varsigma_i (x_1(X)) )  -S_{N } (\psi - \phi) ( \varsigma_i (x_2(X)) )  }}    {  d(x_1(X),x_2(X))^\alpha  }    \\
&\qquad       \geq  \varepsilon -  \frac{2  \Hseminorm{\alpha,\, (S^2,d)}{\psi - \phi} C_0}{1-\Lambda^{-\alpha}}  
                       \geq \frac{\varepsilon}{2}.
\end{align*}
The claim is now established.

Hence $\mathcal{S}^\alpha$ is open in $\Holder{\alpha}(S^2,d)$.

Finally, recall that $1<\Lambda \leq \Lambda_0(f)$ (see \cite[Theorem~16.3]{BM17}). Thus if either $\alpha \in (0,1)$ or $\Lambda \neq \Lambda_0(f)$, then $\Lambda^{\alpha} < \Lambda_0(f)$, and the density of $\mathcal{S}^\alpha$ in $\Holder{\alpha}(S^2,d)$ follows immediately from Theorem~\ref{thmPerturbToStrongNonIntegrable}.
\end{proof}


\begin{thebibliography}{99}

\bibitem[An00a]{An00a}
{\sc Anantharaman,~N.},
G\'eod\'esiques ferm\'ees d'une surface sous contraintes homologiques. Th\`{e}se de doctorat, Universit\'e Paris 6, 2000.

\bibitem[An00b]{An00b}
{\sc Anantharaman,~N.},
Precise counting results for closed orbits of Anosov flows.
\textit{Ann.\ Sci.\ \'Ec.\ Norm.\ Sup\'er.\ (4)} 33 (2000), 33--56.


\bibitem[AM65]{AM65}
{\sc Artin,~M.} and {\sc Mazur,~B.},
On periodic points.
\textit{Ann.\ of Math.\ (2)} 81 (1965), 82--99.



\bibitem[AGY06]{AGY06}
{\sc Avila,~A.}, {\sc Gou\"ezel,~S.}, and {\sc Yoccoz,~J.C.},
Exponential mixing for the Teichm\"uller flow.
\textit{Publ.\ Math.\ Inst.\ Hautes \'Etudes Sci.} 104 (2006), 143--211.





\bibitem[BabLe98]{BabLe98}
{\sc Babillot,~M.} and {\sc Ledrappier,~F.},
Lalley's theorem on periodic orbits of hyperbolic flows.
\textit{Ergod.\ Th.\ \& Dynam.\ Sys.} 18 (1998), 17--39.



\bibitem[Bai04]{Bai04}
{\sc Baillif,~M.},
Kneading operators, sharp determinants and weighted Lefschetz zeta functions in higher dimension.
\textit{Duke Math.\ J.} 124 (2004), 145--175.


\bibitem[BB05]{BB05}
{\sc Baillif,~M.} and {\sc Baladi,~V.},
Kneading determinants and spectra of transfer operators in higher dimensions: the isotropic case.
\textit{Ergod.\ Th.\ \& Dynam.\ Sys.} 25 (2005), 1437--1470.



\bibitem[Bal00]{Bal00}
{\sc Baladi,~V.},
{\it Positive transfer operators and decay of correlations}, volume~16 of {\it Adv.\ Ser.\ Nonlinear Dynam.}
World Sci.\ Publ., Singapore, 2000.

\bibitem[Bal18]{Bal18}
{\sc Baladi,~V.},
{\it Dynamical zeta functions and dynamical determinants for hyperbolic maps: a functional approach}, volume~68 of {\it Ergeb.\ Math.\ Grenzgeb.\ (3)}.
Springer, Berlin, 2018.



\bibitem[BalLiv12]{BalLiv12}
{\sc Baladi,~V.} and {\sc Liverani,~C.},
Exponential decay of correlations for piecewise cone hyperbolic contact flows.
\textit{Comm.\ Math.\ Phys.} 314 (2012), 689--773.


\bibitem[BDL18]{BDL18}
{\sc Baladi,~V.}, {\sc Demers,~M.}, and {\sc Liverani,~C.},
Exponential decay of correlations for finite horizon Sinai billiard flows.
\textit{Invent.\ Math.} 211 (2018), 39--177.


\bibitem[BJR02]{BJR02}
{\sc Baladi,~V.}, {\sc Jiang,~Y.}, and {\sc Rugh,~H.H.},
Dynamical determinants via dynamical conjugacies for postcritically finite polynomials.
\textit{J.\ Stat.\ Phys.} 108 (2002), 973--993.

\bibitem[BKRS97]{BKRS97}
{\sc Baladi,~V.}, {\sc Kitaev,~A.}, {\sc Ruelle,~D.}, and {\sc Semmes,~S.}
Sharp determinants and kneading operators for holomorphic maps. 
\textit{Tr.\ Mat.\ Inst.\ Steklova} 216, Din.\ Sist.\ i Smezhnye Vopr., (1997) 193--235;
translation in \textit{Proc.\ Steklov Inst.\ Math.} 216 (1997), 186--228.



\bibitem[BR96]{BR96}
{\sc Baladi,~V.} and {\sc Ruelle,~D.},
Sharp determinants. 
\textit{Invent.\ Math.} 123 (1996), 553--574.









\bibitem[Bon06]{Bon06}
{\sc Bonk,~M.},
Quasiconformal geometry of fractals. In {\it Proc.\ Internat.\ Congr.\ Math.\ (Madrid 2006)}, Volume~II,
Eur.\ Math.\ Soc.\ Z\"{u}rich, 2006, pp.\ 1349--1373.


\bibitem[BM10]{BM10} 
{\sc Bonk,~M.} and {\sc Meyer,~D.},
Expanding Thurston maps. Preprint, (arXiv:1009.3647v1), 2010.




\bibitem[BM17]{BM17} 
\textsc{Bonk,~M.} and \textsc{Meyer,~D.},
\textit{Expanding Thurston maps}, volume 225 of \textit{Math.\ Surveys Monogr.}, Amer.\ Math.\ Soc., Providence, RI, 2017.


\bibitem[BCRW08]{BCRW08} 
\textsc{Borwein,~P.}, \textsc{Choi,~St.}, \textsc{Rooney,~B.}, and \textsc{Weirathmueller,~A.},
\textit{The Riemann Hypothesis: a resource for the afficionado and virtuoso alike}, 
Springer, New York, 2008.


\bibitem[BD17]{BD17} 
\textsc{Bourgain,~J.} and \textsc{Dyatlov,~S.},
Fourier dimension and spectral gaps for hyperbolic surfaces.
\textit{Geom.\ Funct.\ Anal.} 27 (2017), 744--771.


\bibitem[BGS11]{BGS11} 
\textsc{Bourgain,~J.}, \textsc{Gamburd,~A.}, and \textsc{Sarnak,~P.},
Generalization of Selberg’s $\frac{3}{16}$ theorem and affine sieve.
\textit{Acta Math.} 207 (2011), 255--290.







\bibitem[Bow72]{Bow72}
{\sc Bowen,~R.},
Entropy-expansive maps. 
\textit{Trans.\ Amer.\ Math.\ Soc.} 164 (1972), 323--333.





         
   

   
   
   










\bibitem[Ca94]{Ca94}
\textsc{Cannon,~J.W.},
The combinatorial Riemann mapping theorem.
\textit{Acta Math.} 173 (1994), 155--234.



\bibitem[CFP07]{CFP07}
{\sc Cannon,~J.W.}, {\sc Floyd,~W.J.}, and {\sc Parry,~W.R.},
Constructing subdivision rules from rational maps.
\textit{Conform.\ Geom.\ Dyn.} 11 (2007), 128--136.













\bibitem[Dol98]{Dol98}
{\sc Dolgopyat,~D.},
On decay of correlations of Anosov flows.  
\textit{Ann.\ of Math.\ (2)} 147 (1998), 357--390. 


\bibitem[DH93]{DH93}
{\sc Douady,~A.} and {\sc Hubbard,~J.H.},
A proof of Thurston's topological characterization of rational functions.  
\textit{Acta Math.} 171 (1993), 263--297. 


\bibitem[DPU96]{DPU96}
{\sc Denker,~M.}, {\sc Przytycki,~F.}, and {\sc Urba\'nski,~M.},
On the transfer operator for rational functions on the Riemann sphere.  
\textit{Ergod.\ Th.\ \& Dynam.\ Sys.} 16 (1996), 255--266.












\bibitem[DZ16]{DZ16} 
\textsc{Dyatlov,~S.} and \textsc{Zahl,~J.},
Spectral gaps, additive energy, and a fractal uncertainty principle.
\textit{Geom.\ Funct.\ Anal.} 26 (2016), 1011--1094.




\bibitem[EE85]{EE85}
{\sc Ellison,~W.J.} and {\sc Ellison,~F.},
{\it Prime numbers},
Hermann, Paris, 1985.



\bibitem[Fo81]{Fo81}
{\sc Forster,~O.},
{\it Lectures on Riemann surfaces},
Springer, New York, 1981.





               
            
             









\bibitem[GLP13]{GLP13}
{\sc Giulietti,~P.}, {\sc Liverani,~C.}, and {\sc Pollicott,~M.},
Anosov flows and dynamical zeta functions.
\textit{Ann.\ of Math.\ (2)} 178 (2013), 687--773.


\bibitem[Gu86]{Gu86}
{\sc Guillop\'e,~L.},
Sur la distribution des longueurs des g\'eode\'esiques ferm\'ees d'une surface compacte \`a bord totalement g\'eode\'esique.
\textit{Duke Math.\ J.} 53 (1986), 827--848.


\bibitem[HP09]{HP09}
{\sc Ha\"{\i}ssinsky,~P.} and {\sc Pilgrim,~K.M.},
Coarse expanding conformal dynamics.
\textit{Ast\'{e}risque} 325 (2009).





\bibitem[Ha02]{Ha02}
\textsc{Hatcher,~A.},
\textit{Algebraic topology},
Cambridge Univ.\ Press, Cambridge, 2002.






\bibitem[He01]{He01}
{\sc Heinonen,~J.}
\textit{Lectures on analysis on metric spaces},
Springer, New York, 2001.




\bibitem[Hu61]{Hu61}
{\sc Huber,~H.},
Zur analytischen Theorie hyperbolischer Raumformen und Bewegungsgruppen. II.
\textit{Math.\ Ann.} 142 (1961), 385--398.







\bibitem[KH95]{KH95}
{\sc Katok,~A.} and {\sc Hasselblatt,~B.},
{\it Introduction to the modern theory of dynamical systems},
Cambridge Univ.\ Press, Cambridge, 1995.



\bibitem[KS90]{KS90}
{\sc Katsuda,~A.} and {\sc  Sunada,~T.},
Closed orbits in homology classes.
\textit{Publ.\ Math.\ Inst.\ Hautes \'Etudes Sci.} 71 (1990), 5--32.






\bibitem[Ki98]{Ki98}
{\sc Kitchens,~B.P.},
{\it Symbolic dynamics: one-sided, two-sided, and countable state Markov shifts},
Springer, Berlin, 1998.


\bibitem[La89]{La89} 
{\sc Lalley,~S.P.},
Renewal theorems in symbolic dynamics, with applications to geodesic flows, noneuclidean tessellations and their fractal limits.
\textit{Acta Math.} 163 (1989), 1--55.


\bibitem[Li15]{Li15} 
{\sc Li,~Z.},
Weak expansion properties and large deviation principles for expanding Thurston maps.
\textit{Adv.\ Math.} 285 (2015), 515--567.


\bibitem[Li16]{Li16} 
{\sc Li,~Z.},
Periodic points and the measure of maximal entropy of an expanding Thurston map.
\textit{Trans.\ Amer.\ Math.\ Soc.} 368 (2016), 8955--8999.

\bibitem[Li17]{Li17} 
\textsc{Li,~Z.},
\textit{Ergodic theory of expanding Thurston maps,} volume~4 of \textit{Atlantis Stud.\ Dyn.\ Syst.}, Atlantis Press, 2017.



\bibitem[Li18]{Li18} 
\textsc{Li,~Z.},
Equilibrium states for expanding Thurston maps. 
\textit{Comm.\ Math.\ Phys.} 357 (2018), 811--872.

\bibitem[Liv04]{Liv04} 
\textsc{Liverani,~C.},
On contact Anosov flows.
\textit{Ann.\ of Math.\ (2)} 159 (2004), 1275--1312.



\bibitem[LM97]{LM97}
{\sc Lyubich,~M.Yu.} and {\sc Minsky,~Y.},
Laminations in holomorphic dynamics.
\textit{J.\ Differential Geom.} 47 (1997), 17--94.



\bibitem[Mar69]{Mar69}
{\sc Margulis,~G.A.},
On some applications of ergodic theory to the study of manifolds on negative curvature.
\textit{Fun.\ Anal.\ Appl.} 3 (1969), 89--90.



\bibitem[Mar04]{Mar04}
{\sc Margulis,~G.A.},
\textit{On some aspects of theory of Anosov systems},
Springer, Berlin, 2004.

\bibitem[MMO14]{MMO14}
{\sc Margulis,~G.A.}, {\sc Mohammadi,~A.}, and {\sc Oh,~H.}, 
Closed geodesics and holonomies for Kleinian manifolds.
\textit{Geom.\ Funct.\ Anal.} 24 (2014), 1608--1636.



 
\bibitem[Mey12]{Mey12}
{\sc Meyer,~D.},
Expanding Thurston maps as quotients.
Preprint, (arXiv:0910.2003v1), 2012.
 

\bibitem[Mey13]{Mey13}
{\sc Meyer,~D.},
Invariant Peano curves of expanding Thurston maps.
\textit{Acta Math.} 210 (2013), 95--171.







\bibitem[Mil06]{Mil06}
{\sc Milnor,~J.},
On Latt\`{e}s maps. 
In {\it Dynamics on the Riemann sphere},
Eur.\ Math.\ Soc., Z\"urich, 2006, pp.9--43.


\bibitem[MT88]{MT88}
{\sc Milnor,~J.} and {\sc Thurston,~W.P.},
Iterated maps of the interval.
In {\sc Alexander,~J.C.} (Eds.), {\it Dynamical systems (Maryland 1986--87)}, 
volume~1342 of {\it Lecture Notes in Math.}, pp.\ 465--563, Springer, Berlin, 1988. 


\bibitem[Mis73]{Mis73}
{\sc Misiurewicz,~M.},
Diffeomorphisms without any measure with maximal entropy.
\textit{Bull.\ Acad.\ Pol.\ Sci.} 21 (1973), 903--910.



\bibitem[Mis76]{Mis76}
{\sc Misiurewicz,~M.},
Topological conditional entropy.
\textit{Studia Math.} 55 (1976), 175--200.





\bibitem[Na05]{Na05}
{\sc Naud,~F.},
Expanding maps on Cantor sets and analytic continuation of zeta functions.
\textit{Ann.\ Sci.\ \'Ec.\ Norm.\ Sup\'er.\ (4)} 38 (2005), 116--153.

\bibitem[Na14]{Na14}
{\sc Naud,~F.},
Density and location of resonances for convex co-compact hyperbolic surfaces.
\textit{Invent.\ Math.} 195 (2014), 723--750.


\bibitem[OP18]{OP18}
{\sc Oh,~H.} and {\sc Pan,~W.},
Local mixing and invariant measures for horospherical subgroups on abelian covers.
\textit{Int.\ Math.\ Res.\ Not.\ IMRN}  (2018).

\bibitem[OW16]{OW16}
{\sc Oh,~H.} and {\sc Winter,~D.},
Uniform exponential mixing and resonance free regions for convex cocompact congruence subgroups of $\rm{SL}_2(\Z)$.
\textit{J.\ Amer.\ Math.\ Soc.} 29 (2016), 1069--1115.



\bibitem[OW17]{OW17}
{\sc Oh,~H.} and {\sc Winter,~D.},
Prime number theorems and holonomies for hyperbolic rational maps.
\textit{Invent.\ Math.} 208 (2017), 401--440.














\bibitem[PP90]{PP90}
{\sc Parry,~W.} and {\sc Pollicott,~M.},
Zeta functions and the periodic orbit structure of hyperbolic dynamics.
\textit{Ast\'{e}risque} 187--188 (1990), 1--268.


\bibitem[PhSa87]{PhSa87}
{\sc Phillips,~R.} and {\sc Sarnak,~P.},
Geodesics in homology classes.
\textit{Duke Math. J.} 55 (1987), 287--297.


\bibitem[Po91]{Po91}
 {\sc Pollicott,~M.},
Homology and closed geodesics in a compact negatively curved surface.
\textit{Amer.\ J.\ Math.} 113 (1991), 379--385.

\bibitem[PoSh98]{PoSh98}
 {\sc Pollicott,~M.} and {\sc Sharp,~R.},
Exponential error terms for growth functions on negatively curved surfaces.
\textit{Amer.\ J.\ Math.} 120 (1998), 1019--1042.







\bibitem[PoU17]{PoU17}
{\sc Pollicott,~M.} and {\sc Urba\'{n}ski,~M.},
Asymptotic counting in conformal dynamical systems.
Preprint, (arXiv:1704.06896v2), 2017.

\bibitem[PrU10]{PrU10}
{\sc Przytycki,~F.} and {\sc Urba\'{n}ski,~M.},
{\it Conformal fractals: ergodic theory methods},
Cambridge Univ.\ Press, Cambridge, 2010.

\bibitem[Ro03]{Ro03}
{\sc Roblin,~T.},
Ergodicit\'{e} et \'{e}quidistribution en courbure n\'{e}gative.
\textit{M\'{e}m.\ Soc.\ Math.\ Fr.\ (N.S.)} 95 (2003), vi+96.






\bibitem[Rue76a]{Rue76a}
{\sc Ruelle,~D.},
Zeta functions and statistical mechanics.
\textit{Ast\'{e}risque} 40 (1976), 167--176.

\bibitem[Rue76b]{Rue76b}
{\sc Ruelle,~D.},
Generalized zeta-functions for axiom A basic sets.
\textit{Bull.\ Amer.\ Math.\ Soc.} 80 (1976), 153--156.

\bibitem[Rue76c]{Rue76c}
{\sc Ruelle,~D.},
Zeta-functions for expanding maps and Anosov flows.
\textit{Invent.\ Math.} 34 (1976), 231--242.


\bibitem[Rue89]{Rue89}
{\sc Ruelle,~D.},
The thermodynamical formalism for expanding maps.
\textit{Comm.\ Math.\ Phys.} 125 (1989), 239--262.


\bibitem[Rue90]{Rue90}
{\sc Ruelle,~D.},
An extension of the theory of Fredholm determinants.
\textit{Publ.\ Math.\ Inst.\ Hautes \'Etudes Sci.} 72 (1990), 175--193.


\bibitem[Rug16]{Rug16}
{\sc Rugh,~H.H.},
The Milnor--Thurston determinant and the Ruelle transfer operator.
\textit{Comm.\ Math.\ Phys.} 342 (2016), 603--614.





\bibitem[Sa80]{Sa80}
{\sc Sarnak,~P.},
Prime geodesic theorems. PhD thesis, Stanford University, 1980.



\bibitem[Se56]{Se56}
{\sc Selberg,~A.},
Harmonic analysis and discontinuous groups in weakly symmetric Riemannian spaces with applications to Dirichlet series.
\textit{J.\ Indian Math.\ Soc.} 20 (1956), 47--87.


\bibitem[Sh93]{Sh93}
{\sc Sharp,~R.},
Closed orbits in homology classes for Anosov flows.
\textit{Ergod.\ Th.\ \& Dynam.\ Sys.} 13 (1993), 387--408.





\bibitem[Sm67]{Sm67}
{\sc Smale,~S.},
Differentiable dynamical systems.
\textit{Bull.\ Amer.\ Math.\ Soc.} 73 (1967), 747--817.


\bibitem[St01]{St01}
{\sc Stoyanov,~L.N.},
Spectrum of the Ruelle operator and exponential decay of correlations for open billiard flows.
\textit{Amer.\ J.\ Math.} 123 (2001), 715--759.

\bibitem[St11]{St11}
{\sc Stoyanov,~L.N.},
Spectra of Ruelle transfer operators for Axiom A flows.
\textit{Nonlinearity} 24 (2011), 1089--1120.

\bibitem[Su83]{Su83}
{\sc Sullivan,~D.P.},
Conformal dynamical systems. In {\it Geometric dynamics}, volume~1007 of {\it Lecture Notes in Math.}, pp.\ 725--752.
Springer, Berlin, 1983.

\bibitem[Su85]{Su85}
{\sc Sullivan,~D.P.},
Quasiconformal homeomorphisms and dynamics I. Solution of the Fatou--Julia problem on wandering domains.
\textit{Ann.\ of Math.\ (2)} 122 (1985), 401--418.


\bibitem[Ti39]{Ti39}
{\sc Titchmarsh,~E.C.},
{\it The theory of functions},
Oxford Univ.\ Press, Oxford, 1939.



\bibitem[Th80]{Th80}
{\sc Thurston,~W.P.},
{\it The geometry and topology of three-manifolds},
lecture notes from 1980, available at \url{http://library.msri.org/books/gt3m/}.






\bibitem[vK01]{vK01}
{\sc von~Koch,~H.},
Sur la distribution des nombres premiers.
\textit{Acta Math.} 24 (1901), 159--182.


\bibitem[Wad97]{Wad97}
{\sc Waddington,~S.},
Zeta functions and asymptotic formulae for preperiodic orbits of hyperbolic rational maps.
\textit{Math.\ Nachr.} 186 (1997), 259--284.



\bibitem[Wal82]{Wal82}
{\sc Walters,~P.},
{\it An introduction to ergodic theory},
Springer, New York, 1982.




\bibitem[Wi16]{Wi16} 
{\sc Winter,~D.},
Exponential mixing of frame flows for convex cocompact hyperbolic manifolds. 
Preprint, (arXiv:1612.00909v1), 2016.

\bibitem[Yi15]{Yi15}
{\sc Yin,~Q.},
Thurston maps and asymptotic upper curvature.
\textit{Geom.\ Dedicata} 176 (2015), 271--293.















\end{thebibliography}
\end{document}